\date{} 
\title{Liouville quantum gravity as a mating of trees}
\author{Bertrand Duplantier, Jason Miller and Scott Sheffield}
\documentclass[11pt,naturalnames]{article}
\usepackage{amsmath}
\usepackage{amssymb}
\usepackage{amsthm}
\usepackage{amsfonts}
\usepackage{mathrsfs}
\usepackage{mathalfa}
\usepackage{subfloat}
\usepackage{morefloats}
\usepackage{graphicx}
\usepackage{color}
\usepackage{tabularx}
\usepackage[table]{xcolor}
\usepackage{enumerate}
\usepackage{subfig}
\usepackage[normalem]{ulem}
\usepackage{microtype}
\usepackage{comment}
\usepackage{todonotes}
\usepackage{multicol}

\def\@rst #1 #2other{#1}
\newcommand\MR[1]{\relax\ifhmode\unskip\spacefactor3000 \space\fi
  \MRhref{\expandafter\@rst #1 other}{#1}}
\newcommand{\MRhref}[2]{\href{http://www.ams.org/mathscinet-getitem?mr=#1}{MR#2}}

\usepackage[margin=1in]{geometry}

\newcommand{\CA}{{\mathcal A}}
\newcommand{\CM}{{\mathcal M}}
\newcommand{\CN}{{\mathcal N}}
\newcommand{\CU}{{\mathcal U}}
\newcommand{\CT}{{\mathcal T}}
\newcommand{\CW}{{\mathcal W}}
\newcommand{\CF}{{\mathcal F}}
\newcommand{\CS}{{\mathcal S}}
\newcommand{\CH}{{\mathcal H}}
\newcommand{\CP}{{\mathcal P}}

\newcommand{\dirichlet}{{\mathrm D}}
\newcommand{\free}{{\mathrm F}}

\newcommand{\qwedge}{\CW}

\definecolor{antiquefuchsia}{rgb}{0.57, 0.36, 0.51}

\allowdisplaybreaks

\newif\ifhyper\IfFileExists{hyperref.sty}{\hypertrue}{\hyperfalse}

\ifhyper\usepackage{hyperref}\fi

\newif\ifdraft
\drafttrue

\long\def\comment#1{}
\numberwithin{equation}{section}
\numberwithin{figure}{section}
\numberwithin{table}{section}

\newtheorem{theorem}{Theorem}
\numberwithin{theorem}{section}
\newtheorem{corollary}[theorem]{Corollary}
\newtheorem{lemma}[theorem]{Lemma}

\newtheorem{proposition}[theorem]{Proposition}

\newtheorem{problem}[theorem]{Question}

\theoremstyle{remark}\newtheorem{definition}[theorem]{Definition}
\theoremstyle{remark}\newtheorem{remark}[theorem]{Remark}

\newcommand{\R}{\mathbf{R}}
\newcommand{\C}{\mathbf{C}}
\newcommand{\D}{\mathbf{D}}
\newcommand{\Z}{\mathbf{Z}}
\newcommand{\T}{\mathcal{T}}
\newcommand{\N}{\mathbf{N}}
\newcommand{\W}{\mathbf{W}}
\newcommand{\HH}{\mathbf{H}}
\newcommand{\h}{\HH}

\definecolor{purple}{rgb}{0.7,0,0.7}
\definecolor{gray}{rgb}{0.6,0.6,0.6}
\definecolor{dgreen}{rgb}{0.0,0.4,0.0}
\definecolor{dblue}{rgb}{0.0,0.0,0.5}

\newcommand{\re}{{\mathrm {Re}}}
\newcommand{\im}{{\mathrm {Im}}}
\newcommand{\Fh}{{\mathfrak h}}
\newcommand{\Fg}{{\mathfrak g}}
\newcommand{\CC}{{\mathcal C}}
\newcommand{\ol}{\overline}
\newcommand{\ul}{\underline}
\newcommand{\wh}{\widehat}
\newcommand{\wt}{\widetilde}

\newcommand{\cov}{\mathrm{cov}}
\newcommand{\var}{\mathrm{var}}

\newcommand{\Q}{{\mathbf Q}}

\newcommand{\CG}{{\mathcal G}}
\newcommand{\CX}{{\mathcal X}}
\newcommand{\CZ}{{\mathcal Z}}
\newcommand{\CB}{{\mathcal B}}
\newcommand{\CV}{{\mathcal V}}
\newcommand{\CD}{{\mathcal D}}
\newcommand{\CL}{{\mathcal L}}
\newcommand{\CE}{{\mathcal E}}
\newcommand{\s}{{\mathbf S}}
\newcommand{\giv}{\,|\,}
\newcommand{\bes}{\mathrm{BES}}
\newcommand{\besq}{\mathrm{BESQ}}

\newcommand{\fan}{{\mathbf F}}

\newcommand{\cyl}{\mathscr{C}}
\newcommand{\strip}{{\mathscr{S}}}
\newcommand{\bubble}{\CB}
\newcommand{\one}{{\mathbf 1}}

\def\H{\mathbf{H}}

\def\Var{\mathop{\mathrm{Var}}}
\def\diam{\mathop{\mathrm{diam}}}

\def\dist{\mathop{\mathrm{dist}}}

\def\Im{{\rm Im}\,}
\def\Re{{\rm Re}\,}

\newcommand{\QLE}{{\rm QLE}}
\newcommand{\SLE}{{\rm SLE}}
\newcommand{\CLE}{{\rm CLE}}

\newcommand{\confrad}{{\mathrm {CR}}}

\newcommand{\qnt}{{\mathfrak q}} 
\newcommand{\qlt}{{\mathfrak l}}

\newcommand{\diskmeasure}{\CM}
\newcommand{\spheremeasure}{\CN}

\newcommand{\ppp}{p.p.p.}

\def \eps {\varepsilon}
\def \P {{\bf P}}

\def \p {{\P}}
\def \E {{\bf E}}

\begin{document} \maketitle

\begin{abstract}
There is a simple way to ``glue together'' a coupled pair of continuum random trees (CRTs) to produce a topological sphere.  The sphere comes equipped with a measure and a space-filling curve (which describes the ``interface'' between the trees).   We present an explicit and canonical way to embed the sphere in $\C \cup \{ \infty \}$.  In this embedding, the measure is a form of Liouville quantum gravity (LQG) with parameter $\gamma \in (0,2)$, and the curve is space-filling SLE$_{\kappa'}$ with $\kappa' = 16/\gamma^2$.

Achieving this requires us to develop an extensive suite of tools for working with LQG surfaces.  We explain how to conformally weld so-called ``quantum wedges'' to obtain new quantum wedges of different weights.  We construct finite-volume quantum disks and spheres of various types, and give a Poissonian description of the set of quantum disks cut off by a boundary-intersecting SLE$_{\kappa}(\rho)$ process with $\kappa \in (0,4)$.  We also establish a {\em L\'evy tree} description of the set of quantum disks to the left (or right) of an  SLE$_{\kappa'}$ with $\kappa' \in (4,8)$.  We show that given two such trees, sampled independently, there is a.s.\ a canonical way to ``zip them together'' and recover the SLE$_{\kappa'}$. 

The law of the CRT pair we study was shown in an earlier paper to be the scaling limit of the discrete tree/dual-tree pair associated to an FK-decorated random planar map (RPM).  Together, these results imply that FK-decorated RPM scales to CLE-decorated LQG in a certain ``tree structure'' topology.
\end{abstract}

\newpage
\setlength{\parskip}{0.01cm plus1mm minus1mm}
\tableofcontents
\newpage

\parindent 0 pt
\setlength{\parskip}{0.25cm plus1mm minus1mm}

\medbreak {\noindent\bf Acknowledgments.}  We have benefited from discussions with many individuals, including (and not limited to) Nathanael Berestycki, Dmitry Chelkak, Nicolas Curien, Hugo Duminil-Copin, Christophe Garban, Richard Kenyon, Greg Lawler, Jean-Fran\c{c}ois Le Gall, Gr\'egory Miermont, Asaf Nachmias, R\'emi Rhodes, Steffen Rohde, Oded Schramm, Stanislav Smirnov, Xin Sun, Vincent Vargas, Sam Watson, Wendelin Werner, and Hao Wu.  We also thank Ewain Gwynne for helpful comments on an earlier version of this article and Ken Stephenson for helping us use his circle packing software. We also thank participants of a 2017 Oberwolfach seminar about this paper who gave us additional feedback on the manuscript, including (among others) Juhan Aru, Ellen Powell, Lukas Schoug, and Avelio Sepulveda.  B.D.\ was partially supported by the CNRS grant PICS06769 and by the ANR grant GRAAL ANR-14-CE25-0014.  J.M.\ was partially supported by the NSF grant DMS-1204894. S.S.\ was partially supported by NSF grants DMS-1712862 and DMS-1209044 and Simons Fellowship with award number 306120.

We also thank several anonymous referees for feedback which led to substantial improvements to the exposition.

\section{Introduction}
\label{sec::introduction}

This paper studies connections between the Schramm-Loewner evolution ($\SLE$) \cite{S0} and the ``random surfaces'' associated with Liouville quantum gravity (LQG).  We derive an extensive collection of results about weldings and decompositions of surfaces of various types (spheres, disks, wedges, chains of disks, trees of disks, etc.)\ along with fundamental structural results about LQG.  These results comprise a sort of ``calculus of random surfaces'' that combines the conformal welding theory of \cite{she2010zipper} with the imaginary geometry theory of \cite{ms2012imag1,ms2012imag2,ms2012imag3,ms2013imag4}.

The LQG surfaces we consider involve a parameter $\gamma \in (0,2)$, while the SLE curves we consider involve parameters $\kappa<4$ and $\kappa'>4$ which are related to $\gamma$ via $\kappa = \gamma^2$ and $\kappa'=16/\kappa$. In Section~\ref{subsec::surfaces}, we will recall the definitions of certain random $\gamma$-LQG surfaces known as quantum wedges, quantum cones, quantum disks, and quantum spheres. Quantum wedges and cones are in some sense the most natural infinite volume LQG surfaces, while quantum disks and spheres are finite volume variants. Quantum wedges and cones are both marked by two special points, an origin and an infinity point, denoted $\infty$. Bounded neighborhoods of the origin contain a finite amount of $\gamma$-LQG mass while neighborhoods of $\infty$ contain an infinite amount of mass.  A quantum wedge additionally comes with ``left'' and ``right'' boundary arcs, each of which has the origin and $\infty$ as endpoints, and each of which comes endowed with a boundary length measure. The law of a quantum wedge depends on a \emph{weight} parameter which indicates in some sense how ``thick'' it is. Our three main types of results can be stated roughly as follows.

\begin{enumerate}
\item Consider a $\gamma$-LQG quantum wedge $\CW$ of weight $W$ decorated by an independent $\SLE_{\kappa}(\rho_1; \rho_2)$ curve from the origin point to the $\infty$ point, so that $\eta$ divides $\CW$ into two surfaces $\CW_1$ and $\CW_2$ (a left side and a right side). Then (as long as $\rho_1$ and $\rho_2$ satisfy a certain constraint) $\CW_1$ and $\CW_2$ are {\em independent} quantum wedges of some weights $W_1$ and $W_2$ satisfying $W_1+W_2 =W$. Moreover, given just $\CW_1, \CW_2$, it is a.s.\ possible to reconstruct the original path-decorated wedge $(\CW,\eta)$.  In fact, using standard complex analysis terminology, as discussed in~\cite{she2010zipper} or in Section~\ref{subsec::conformalmating}, $\CW$ is a {\em conformal welding} of $\CW_1$ and $\CW_2$, where the boundaries of $\CW_1, \CW_2$ are identified to each other in a boundary-length-preserving way. See Theorems~\ref{thm::welding}, \ref{thm::paths_determined} and~\ref{thm::zip_up_wedge_rough_statement}.

\item Consider a $\gamma$-LQG quantum cone $\CC$ decorated by an appropriate {\em space-filling} $\SLE_{\kappa'}$ curve $\eta'$, where $\eta'$ is understood as tracing the interface between a continuum tree $\CT_1$ and a continuum dual tree $\CT_2$. Then each of $\CT_1$, $\CT_2$ inherits a metric structure and a measure, and the joint law of $\CT_1,\CT_2$ is that of a {\em correlated} pair of continuum random trees (CRTs).  Moreover, given just $\CT_1, \CT_2$, it is a.s.\ possible to reconstruct the original path-decorated surface $(\CC,\eta')$.  Borrowing terminology from the complex dynamics literature, see Section~\ref{subsec::conformalmating}, we say that $\CC$ is a {\em conformal mating} of the metric trees, where again the boundaries of $\CT_1,\CT_2$ are glued to each other in a certain boundary-length-measure-preserving way.  See Theorems~\ref{thm::quantum_cone_bm_rough_statement} and~\ref{thm::trees_determine_embedding}.
\item Consider a $\gamma$-LQG quantum wedge $\CW$ decorated by a {\em non-space-filling} $\SLE_{\kappa'}$ (with $\kappa' \in (4,8)$) curve $\eta'$.  Then $\eta'$ cuts $\CW$ into a left side and a right side, which have the law of two {\em independent} ``L\'evy trees of quantum disks.'' Moreover, given the two trees, it is a.s.\ possible to recover the original path-decorated surface $(\CW,\eta')$. Again borrowing terminology from complex dynamics, we say that $\CW$ is a {\em conformal mating} of the two L\'evy trees of quantum disks, where the boundaries of these L\'evy trees are glued together in a boundary-measure-preserving way. See Theorems~\ref{thm::sle_kp_on_wedge} and~\ref{thm::quantum_cone_sle_kp}.

\end{enumerate}
In all three settings, one uses an SLE curve (which is respectively simple, space-filling, or non-simple-non-space-filling) to ``cut'' a type of random surface into two ``pieces.'' And in all three settings, we identify the joint law of the pair of pieces, and show that given the pieces, we can a.s.\ recover the original surface along with the SLE curve. The process of ``gluing the two pieces together'' is what we call {\em conformal welding} in the case of a simple curve, or {\em conformal mating} more generally.

The first item on the above list of results is a generalization of the main result of \cite{she2010zipper}, which concerns the welding of two independent quantum wedges of weight $2$. In this paper, we treat quantum wedges of {\em arbitrary} positive weights, including (most interestingly) small-weight ``thin wedges'' that have cut points, and are defined for the first time in this paper. We will show that if one glues together multiple wedges (or glues a wedge to itself to produce a cone) then the {\em seams} between the wedges become $\SLE_\kappa(\rho)$ curves coupled with each other in such a way that their joint law agrees with the law of certain rays in the imaginary geometry of the Gaussian free field (GFF) as developed in \cite{ms2012imag1, ms2012imag2, ms2012imag3, ms2013imag4}.

The second and third items listed above concern the ``matings of trees'' that the title of this paper refers to. These are the biggest results of this paper and are completely different from those that appear in \cite{she2010zipper} (except that they also involve LQG surfaces decorated by SLE curves). The results in this paper build on each other in a sequential fashion. After recalling the basic result of \cite{she2010zipper}, we will use it to derive general-weight generalizations of  \cite{she2010zipper}, and we will use these generalizations --- along with their imaginary geometry interpretations --- in the proofs of our main tree-mating theorems.

We remark that the second and third items have ``finite volume'' analogs that are discussed in a follow up paper \cite{quantum_spheres}, where a finite-volume quantum sphere is a produced as a mating of two finite continuum random trees. Although the bulk of this paper concerns the infinite volume setting, the finite-volume quantum sphere itself is {\em constructed} in this paper in Section~\ref{subsec::disks_and_spheres} and Appendix~\ref{app::disks_spheres}, and that construction is used in  \cite{quantum_spheres}. The construction of the quantum sphere presented in Section~\ref{subsec::disks_and_spheres} and Appendix~\ref{app::disks_spheres} is equivalent to a construction given by David, Kupiainen, Rhodes and Vargas in \cite{lqg_sphere} (which was posted to the arXiv about the same time as this paper) although the equivalence was far from obvious when the paper was first posted, and was only established later in a work by Aru, Huang and Sun \cite{twoperspectives}.  The construction in \cite{lqg_sphere} follows more closely the constructions given in the physics literature, beginning with Polyakov's seminal 1981 paper \cite{MR623209}, and aims to present a mathematically rigorous version of that approach. The main difference between the construction in \cite{lqg_sphere} and the one in the present paper concerns how one handles marked points: in \cite{lqg_sphere} one parameterizes the quantum sphere by $\C$ and positions three marked points fixed at $0$, $1$ and $\infty$, while in this paper we parameterize the quantum sphere by a cylinder and position two marked points at its two ends.  A ``one marked point'' quantum sphere construction involving limits was also suggested in \cite{she2010zipper} and is proved in Appendix~\ref{app::disks_spheres} to be equivalent to the approaches mentioned above.

The remainder of this section is structured as follows. In Section~\ref{subsec::surfaces}, we give an overview of the different types of quantum surfaces that appear in this article.  (The reader who wishes to get to the main results quickly can skip or skim this section; in particular, the charts relating wedge and cone parameters are given for reference and do not need to be internalized on a first read.) Section~\ref{subsec::welding} gives our main results about weldings of quantum wedges (the first item listed above).  As a warmup, Section~\ref{subsec::easy} briefly describes what it means to {\em topologically mate} random trees to produce a {\em topological} sphere, and Section~\ref{subsec::matingsandloops} then describes our main results about {\em conformal matings} of such trees (the second and third items listed above).  The conformal matings described in Section~\ref{subsec::matingsandloops} can be understood as a way to endow the topological sphere described in Section~\ref{subsec::easy} with a conformal structure (i.e., to parameterize the sphere by $\C \cup \{\infty \}$ in a way that is canonical up to  M\"obius transformation; in this paper, we focus on the infinite volume version of this construction when we discuss conformal matings).  Finally, Section~\ref{subsec::outline} contains an outline of the remainder of the article.  In order to keep the introduction focused on the main results, we defer some discussion of background and motivation to Section~\ref{sec::motivation}.  We will also assume familiarity with the GFF in the introduction and refer the reader to Section~\ref{subsec::gffs} for a more in depth review.

We refer the reader to the survey paper by Gwynne, Holden and Sun \cite{gwynne2019mating} (see also \cite{gwynne2019random}) for an overview of more recent developments in this subject.

\subsection{Quantum wedges, cones, disks, and spheres}
\label{subsec::surfaces}

This section describes the construction of the fundamental random surfaces we study: quantum wedges, cones, disks, and spheres. The reader who wants to get to the main results quickly can skim Section~\ref{subsec::surfaces} on a first read.
\subsubsection{Basic definitions}

A {\bf quantum surface}, as in \cite{ds2011kpz, she2010zipper}, is formally represented by a pair $(D,h)$ where $D$ is planar domain and $h$ is an instance of (some form of) the GFF on $D$.  Intuitively, the quantum surface is the manifold conformally parameterized by $D$, with a metric tensor given by $e^{\gamma h(z)}$ times the Euclidean one, where $\gamma \in [0,2)$ is a fixed constant.  Since $h$ is a distribution and not a function, this notion requires interpretation.

We can use a regularization procedure to define an area measure on $D$:
\begin{equation}
\label{e.mudef}
\mu = \mu_h := \lim_{\eps \to 0} \eps^{\gamma^2/2} e^{\gamma h_\eps(z)}dz,
\end{equation}
where $dz$ is Lebesgue measure on $D$, $h_\eps(z)$ is the mean value of $h$ on the circle~$\partial B(z,\epsilon)$ and the limit represents weak convergence in the space of measures on~$D$.  (The limit is shown to exist a.s.\ when $\eps$ is restricted to powers of two in \cite{ds2011kpz}, and without the power of two restriction in \cite{sheffieldwang}.)  We interpret $\mu_h$ as the area measure of a random surface conformally parameterized by $D$.  When $x \in \partial D$, we let $h_\eps(x)$ be the mean value of $h$ on $D \cap \partial B(x,\epsilon)$.  Similarly, on a linear segment of $\partial D$, where $h$ has free boundary conditions, we may define a boundary length measure by
\begin{equation}
\label{e.nudef}
\nu = \nu_h := \lim_{\eps \to 0} \eps^{\gamma^2/4}e^{\gamma h_\eps(x)/2}dx,
\end{equation}
where $dx$ is Lebesgue measure on $\partial D$ \cite{ds2011kpz}.

One can parameterize the same quantum surface with a different domain $\wt D$, and our regularization procedure implies a simple rule for changing coordinates.  Suppose that $\psi$ is a conformal map from a domain $\wt D$ to $D$ and write $\wt h$ for the
distribution on $\wt D$ given by $h \circ \psi + Q \log |\psi'|$ where
\[ Q := \frac{2}{\gamma} + \frac{\gamma}{2},\]
as in Figure~\ref{fig:qmapfig}.  By \cite[Proposition~2.1]{ds2011kpz}, $\mu_h$ is a.s.\ the image under $\psi$ of the measure $\mu_{\wt h}$. That is, $\mu_{\wt h}(A) = \mu_h(\psi(A))$ for $A \subseteq \wt D$.  (In fact, it is shown further in \cite{sheffieldwang} that if circle averages are replaced with bump function averages then one obtains the same measure a.s.\ and that furthermore the analogous relationship a.s.\ holds for {\em all} conformal maps $\psi$ simultaneously, rather than just holding a.s.\ for a fixed $\psi$.) Similarly, $\nu_h$ is a.s.\ the image under $\psi$ of the measure $\nu_{\wt h}$ in the case that $\nu_h$ and $\nu_{\wt h}$ are both defined and $\psi$ extends to be a homeomorphism on the closure of the domain (in the Riemann sphere).  The invariance of $\nu_h$ under~\eqref{eqn::Qmap} actually yields a definition of the quantum boundary length measure $\nu_h$ when the boundary of $D$ is not piecewise linear --- i.e., in this case, one simply maps to the upper half plane (or any other domain with a piecewise linear boundary) and computes the length there.

Following \cite{ds2011kpz} (motivated by the change of coordinates formula for the LQG measure), we have the following definition.

\begin{definition}
\label{def::q_surface}
A {\bf quantum surface} is an equivalence class of pairs $(D,h)$ under the equivalence relation (see Figure~\ref{fig:qmapfig})
\begin{equation}
\label{eqn::Qmap}
(D,h) \to \psi^{-1}(D,h) := (\psi^{-1}(D), h \circ \psi + Q \log |\psi'|) = (\wt D, \wt h),
\end{equation}
An {\bf embedding} of a quantum surface is a choice of representative $(D,h)$ from the equivalence class, and a transformation of the kind described in \eqref{eqn::Qmap} is called a {\bf coordinate change}.

A {\bf quantum surface with $k$ marked points} is an equivalence class of elements of the form $(D,h,x_1,x_2, \ldots, x_k)$, with each $x_i \in \ol{D}$, under maps of the form~\eqref{eqn::Qmap}, where~\eqref{eqn::Qmap} is understood to map the $x_i\in \overline D$ to $\psi^{-1}(x_i)$.

\end{definition}

In order to specify a quantum surface, one only needs to specify the form of $h$ for a specific embedding.  We emphasize that the same quantum surface can be represented by different distributions $h$.  It can be convenient to make different choices for the embedding of a given quantum surface and we will see several examples of that in this article.  

We also emphasize that the definition of equivalence for quantum surfaces does not require that $D$ be simply connected or even connected.

\begin {figure}[htbp]
\begin {center}
\includegraphics [scale=0.85,page=1]{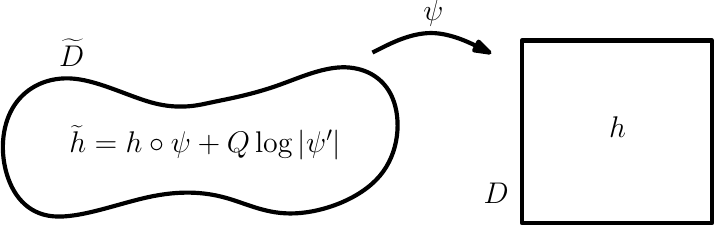}
\caption {\label{fig:qmapfig} A quantum surface coordinate change.}
\end {center}
\end {figure}

\begin {figure}[htbp]
\begin {center}
\includegraphics [scale=0.75,page=1]{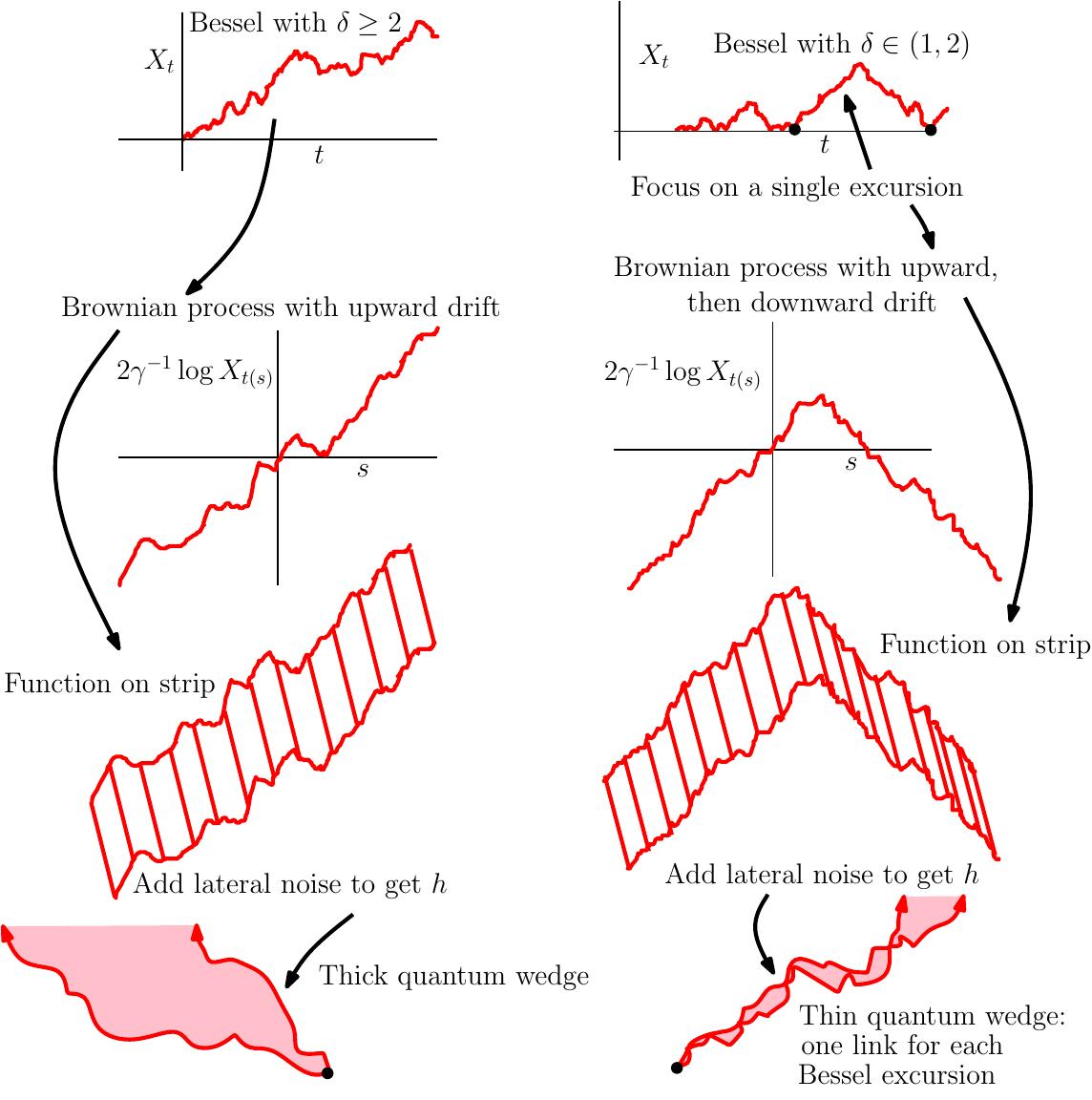}
\caption {\label{fig::besseltowedge} If $X_t$ is a Bessel process with $\delta \geq 2$, then $2 \gamma^{-1} \log X_t$ can be reparameterized by a parameter $s = s(t)$ so that it has quadratic variation $2ds$ and hence evolves as $B_{2s} + as$ where $B$ is a standard Brownian motion and $a$ is a constant.  The reparameterized graph is determined up to horizontal translation. (To define the quantum wedge, it suffices to define a distribution $h$ on the strip modulo horizontal translation --- i.e., modulo the equivalence in which $h(\cdot)$ and $h(a + \cdot)$ are equivalent for real $a$ --- since a quantum surface is defined modulo the equivalence~\eqref{eqn::Qmap} and for any translation $\psi(z) = z+a$ we have $Q\log|\psi'| = 0$.)  When $\delta \in (1,2)$, the restriction of $2 \gamma^{-1}\log X_t$ to a single Bessel excursion can be similarly reparameterized to obtain a Brownian process with a single maximum and a positive/negative drift on the left/right of that maximum (conditioned not to exceed the maximum); we obtain a map from the infinite strip $\strip = \R \times [0, \pi]$ to $\R$ by applying this function to the first coordinate.  After adding ``lateral noise'' obtained by $(\cdot,\cdot)_\nabla$-orthogonally projecting the free boundary GFF on $\strip$ onto the subspace of functions with mean zero on all vertical segments, one obtains a field $h$.  If $\delta \geq 2$ then $(\strip, h)$ is a {\bf thick quantum wedge}.  If $\delta \in (1,2)$, one generates a quantum disk this way for each Bessel excursion; the set of excursions, indexed by zero local time of the Bessel process, forms a Poisson point process, so this procedure generates a Poisson point process of disks. The ``concatenation'' of these disks, as shown in the figure, is called a {\bf thin quantum wedge}.}
\end {center}
\end {figure}

\subsubsection{Quantum wedges}
\label{subsubsec::intro_qwedges}

The quantum wedge is the most natural kind of infinite-volume quantum surface with two marked points. Let $D$ be the infinite wedge $\W_\vartheta = \{z \in \C : \arg(z) \in [0,\vartheta]\}$ for some $\vartheta > 0$. (If $\vartheta \geq 2\pi$, then we view $\W_\vartheta$ as a Riemann surface in the natural way.) The reader first hearing the term ``quantum wedge'' might assume that it refers to the quantum surface $(D, h)$ where $h$ is an instance of the free boundary GFF on $D$.  The trouble is that an instance $h$ of the free boundary GFF is only defined modulo additive constant.  Replacing $(D,h)$ by $(D, h + C)$ corresponds to multiplying the area measure on $D$ by the factor $e^{\gamma C}$.  This means that if $h$ is only defined modulo additive constant then the area measure associated with $(D,h)$ is only defined modulo multiplicative constant. We refer to $(D,h)$ as an {\bf unscaled quantum wedge}. Two equivalent ways to define the unscaled quantum wedge are as follows:

\begin{enumerate}
\item Parameterize instead by the upper half plane $\h$.  If we conformally map $\h$ to $\W_\vartheta$ via the map $\psi_\vartheta(\wt z) = \wt{z}^{\vartheta/\pi}$, then the coordinate change rule~\eqref{eqn::Qmap} for quantum surfaces implies that we can represent this as $\wt{D} = \h$ and let $\wt{h}$ be an instance of the free boundary GFF on $\h$ {\em plus} the deterministic function $Q \log|\psi_\vartheta'|$, which up to additive constant is $Q \log |\wt{z}^{\vartheta/\pi - 1}| = Q(\tfrac{\vartheta}{\pi} - 1) \log|\wt z| = - \alpha \log|\wt z|$ for $\alpha := Q(1 - {\vartheta}/{\pi})$.
\item Parameterize instead by the infinite strip $\strip = \R \times [0, \pi]$.  This can be mapped to $\W_\vartheta$ via the map $\phi_\vartheta(\wt z) = e^{(\vartheta/\pi) \wt{z}}$, and then the coordinate change rule~\eqref{eqn::Qmap} for quantum surfaces implies that we can represent the quantum surface by $\strip$ using an instance of the free boundary GFF on $\strip$ plus the deterministic function $Q \log |\phi_\vartheta'| = Q  (\vartheta/\pi) \Re(\wt{z})=(Q-\alpha)\Re (\wt{z})$ (modulo additive constant).
\end{enumerate}

The strip approach is convenient because it turns out that an instance of the free boundary GFF on $\strip$ can be represented as $B_{2\Re{z}}$ (where $B$ is standard Brownian motion defined up to additive constant --- this function is constant on vertical line segments of the strip) plus a ``lateral noise'' given by the $(\cdot,\cdot)_\nabla$-orthogonal projection ($(f,g)_\nabla = (2\pi)^{-1} \int \nabla f(x) \cdot \nabla g(x) dx$ is the usual Dirichlet inner product) of the free boundary GFF on $\strip$ onto the subspace of functions on $\strip$ that have mean zero on each vertical segment (see Section~\ref{subsubsec::stripandcylinder}).   One obtains an unscaled quantum wedge by replacing $B_{2t}$ with $B_{2t} + at$ where $a = Q-\alpha \geq 0$ (again defined modulo additive constant).  We remark that the parameter space $\W_\vartheta$ is actually less convenient to work with than either the strip or the half plane, and it will not be used outside of this subsection.  

Now that we have defined an {\em unscaled} quantum wedge, we would like to go further and define a quantum wedge, which is defined as an actual random quantum surface --- as opposed to merely being defined modulo scaling. We will do this by defining a random pair $(D, h)$ where $D$ is the horizontal strip $\strip$ defined above.  We emphasize from the start that it suffices to define $h$ modulo horizontal translation since a quantum surface is defined modulo the equivalence~\eqref{eqn::Qmap} and for any translation map $\psi(z) = z+a$ we have that $Q\log|\psi'| = 0$.

The definition of a quantum wedge on $\strip$ is the same as the definition of an unscaled quantum wedge on $\strip$, as described above, except that $B_{2t}+at$ is replaced by a closely related process $A_t$ whose graph is defined modulo horizontal translation, instead of modulo vertical translation.

What is the most natural way to construct a process $A_t$ whose graph looks like the graph of a Brownian motion with positive drift, except that it is defined modulo horizontal translation instead of modulo vertical translation? It turns out that there is only one answer that makes sense in our context, but it takes a bit of thought to explain why, so before we proceed to that, let us consider another question.  Why can we not simply take $A_t = B_{2t} + at$, and consider its graph defined modulo horizontal translation?  (This would be equivalent to taking an unscaled quantum wedge parameterized by $\strip$ and then fixing the additive constant by requiring the mean value of $h$ on the vertical line through zero to be zero.) The basic reason is that if we defined $A_t$ this way, then the corresponding random surface would not be scale invariant.

Furthermore, for reasons that will become clear later, we would like $A_t$ to have the property that if we condition on the process $A_t$ up until the first time it reaches some constant value $C$ (since we are working modulo horizontal translation, we can imagine that we translate the graph horizontally so that this occurs at time zero) then the conditional law of the remainder of $A_t$ is that of $B_{2t} + at$ starting from $C$. (It will turn out that the 180-degree rotated process $-A_{-t}$, also understood modulo horizontal translation, has the same law as $A_t$, and hence also has this property.) 

One approach is to start by defining a process $\wt A^C_t$ for $t\in[0,\infty)$ by setting $\wt A^C_0 = C$ for some very negative value of $C$, and then letting $\wt A^C_t$ evolve over time $t>0$ as ordinary Brownian motion with positive drift. We may horizontally translate the graph of each $\wt A^C_t$ instance to the left in order to obtain the graph of a process that reaches $0$ for the first time at $t=0$, and we may define the law of $A_t$ to be the limit of the laws of these translated processes (in the topology of convergence in law w.r.t.\ the $L^\infty$ metric on compact intervals, say). The $A_t$ thus defined is a process with the property that $A_0 = 0$, and $A_t$ evolves as $B_{2t} + at$ for time $t>0$, and $A_{-t}$ evolves (for $t>0$) as Brownian motion with negative drift $-a t$, in some sense {\em conditioned} on the probability zero event that it never returns to zero. Since we only care about the process modulo horizontal translation, the choice of horizontal translation is irrelevant.

Another way to construct a process $A_t$ with the properties we desire is to note that readers familiar with Bessel processes are {\em already} familiar with such a process: namely, the $\log$ of a Bessel process parameterized by quadratic variation.  More precisely, the reader may recall (and we review this in Section~\ref{subsec::bessel_processes}) that if $X_t$ is a Bessel process started at $X_0 = x$ for some $x>0$, then $2 \gamma^{-1} \log X_t$ evolves as a time change of a Brownian motion with drift. To be precise, if the Bessel dimension is $\delta =2+\frac{2}{\gamma}a$, with $\delta\geq 2$, and the process $2 \gamma^{-1} \log X_t$ is monotonically reparameterized by a parameter $u=u(t)$ chosen so that the process has quadratic variation $2du$ and $u(0) = 0$, then it becomes equivalent in law to the process $B_{2u} + au$ where $B$ is a standard Brownian motion starting at $B_0 = 2\gamma^{-1} \log x$. It is also possible to make sense of a Bessel process started at $X_0=0$, but in this case the amount of quadratic variation time elapsed while $t \in [0, \epsilon]$ is a.s.\ infinite if $\epsilon > 0$, and hence $\log X_t$ --- reparameterized by quadratic variation ---  can be understood as a Brownian motion with drift parameterized by {\em all} of $\R$, whose $-\infty$ limit is $-\infty$ and whose $+\infty$ limit is $+\infty$ (as explained in Proposition~\ref{prop::bessel_exponential_bm}) but the choice of how to translate the graph horizontally is somewhat arbitrary (and was chosen in Figure~\ref{fig::besseltowedge} to make the process reach $0$ for the first time at time $0$).

One reason to use Bessel processes is that doing so leads to an intriguing generalization: namely, we can now also consider Bessel processes of dimension $\delta \in(1,2)$, and define a distinct quantum surface for {\em each} excursion of the Bessel process away from zero, as illustrated in Figure~\ref{fig::besseltowedge} --- this idea will be used in the definition of ``thin quantum wedge'' given just below.

The quantum surfaces obtained for $\delta \geq 2$, or equivalently, $\alpha\leq Q$, are called {\bf thick quantum wedges}.  A thick quantum wedge is a random quantum surface with two marked points (called the origin and $\infty$) whose law is invariant under the operation of multiplying its area measure by a constant.  Each quantum wedge --- when parameterized by $\strip$ as discussed above --- has an infinite amount of quantum mass a.s.\ but only a finite amount corresponding to any particular bounded subset of $\strip$.  (In particular, the law of a quantum wedge is not symmetric under reversing the two marked points, since every neighborhood of its second point has infinite mass, and this is not true of the first point.)  A {\bf thin quantum wedge}, obtained for $\delta\in (1,2)$,  or equivalently, $\alpha\in (Q,Q+\frac{\gamma}{2})$, is (informally) an infinite Poissonian ``chain'' (concatenation) of finite volume quantum surfaces, each with two marked boundary points; there is one quantum surface for each excursion of a Bessel process from $0$, as Figure~\ref{fig::besseltowedge} illustrates. In other words, an instance of a thin quantum wedge is a countable collection of doubly marked disk-homeomorphic quantum surfaces (one for each excursion) whose boundary is homeomorphic to the circle together with a total ordering on the set of such surfaces.

To be a bit more formal, we recall (see Chapters XI and XII of \cite{ry99martingales} for more details) that the excursions of a Bessel process, for $\delta \in (1,2)$, form a Poisson point process, indexed by the local time of the zero set of the Bessel process. More precisely, if the Bessel process $X_t$ has an excursion over the interval $(a,b)$ (i.e., $X_a =X_b = 0$ but $X_t > 0$ for $t \in (a,b)$) then there is a {\em translated (starting at time zero) excursion} given by $X_{t+a}$ for $t \in [0, b-a]$. We then can define a point process by $E = \sum_i \delta_{(s_i, e_i)}$ in which the $e_i$ range over the countable set of excursions of the Bessel process (each translated so that it starts at time zero) and the quantity $s_i$ describes the zero local time accumulated by the process $X_t$ during the corresponding $[0,a]$ interval (i.e., before the start of the excursion). It is well known (see e.g.\ Chapters XI and XII of \cite{ry99martingales}) that the $E$ thus defined is a Poisson point process with intensity $ds \otimes N$, where $ds$ is Lebesgue measure on $[0,\infty)$ and $N$ is the so-called It\^o excursion measure of the Bessel process, and moreover that $E$ a.s.\ uniquely determines the process $X_t$ (and vice versa). Then we can define a thin quantum wedge to be a Poisson point process taking values in $\R \otimes U \otimes H$ (where $U$ is the space of Bessel excursions and $H$ is the space of distributions defined on a strip) as $\sum_i \delta_{(t_i, e_i, h_i)}$ where given $E$, the $h_i$ are independent with a law depending only on $e_i$ as follows: to get from a Bessel excursion $e_i$ to $h_i$, one takes the process $A_t$ indexed by $t \in \R$ --- defined to be the log of the Bessel excursion, parameterized by quadratic variation --- and adds lateral noise as described above; for concreteness, one can fix the horizontal translation by requiring that $A_t$ achieve its maximum at $t=0$. As we explain below, we will also need to endow a thin quantum wedge with some additional structure (a topology, and left and right boundary lengths).

We define a {\bf beaded quantum surface} to be a pair $(D,h)$, modulo the equivalence relation described after \eqref{e.mudef}, except that $D$ is now a closed set (not necessarily homeomorphic to a disk) such that each component of its interior together with its prime-end boundary is homeomorphic to the closed disk, and $h$ is only defined as a distribution on each of these components, and $\psi$ is allowed to be any homeomorphism from $D$ to another closed set $\wt D$ that is conformal on each component of the interior of $D$.

There is a natural way to interpret a thin quantum wedge as a beaded quantum surface. Let $D$ to be the region bounded between the positive real axis and the Bessel process $X_t$ described above, i.e., the region $$S = \{(t,y) : t \geq 0 \, ,\,\,\, 0 \leq y \leq X_t \},$$ endowed with the ordinary Euclidean topology. The interior of $S$ a.s.\ has countably many components (there is one such component for each excursion of the Bessel process away from zero) each of which is topologically homeomorphic to a disk and has a boundary homeomorphic to a circle. Each of these topological disks also comes with two special boundary points, corresponding to points where the real axis and the graph of $X_t$ intersect.  Within any component of $S$ corresponding to an excursion $e_i$, we can define $h$ via $h = h_i \circ \psi (x) + Q \log |\psi'|$ where $\psi$ is any conformal map from that component to the infinite strip, taking the left and right endpoints to the left and right endpoints of the strip. (Again, recall Figure~\ref{fig::besseltowedge}.) In particular, this induces an area measure and a boundary length measure on each component of $S$. The two marked points on a disk-homeomorphic region separate its boundary into a ``left boundary arc'' and a ``right boundary arc.''  It is not hard to see that for any fixed $t > 0$, the sum of the quantum boundary lengths of all components corresponding to excursions that occur before time $t$ is a.s.\ finite.  Thus the ``left boundary'' of $S$ (corresponding to the graph of $X_t$) and the ``right boundary'' (corresponding to the positive real axis) can each a.s.\ be continuously parameterized by this boundary length measure.

Note that although we can use the set $S$ above to parameterize a thin quantum wedge, it might be more natural to change coordinates and replace $S$ by (for example) a set $\wt S$ which looks like $S$ except that each disk homeomorphic component is replaced by an actual disk (centered on the real line, and covering the same real interval as the corresponding component of $S$).  We also remark that it is not actually crucial to insist on parameterizing a beaded quantum surface by a subset of the plane --- in general we can let the parameterizing space $D$ be any closed topological space as long as each component of its interior is homeomorphic to the disk and comes with a ``conformal structure'' (i.e., a distinguished homeomorphism to the ordinary unit disk defined up to M\"obius itransformation).

The thick-quantum wedge condition $\alpha\leq Q$ corresponds in the physics literature on Liouville quantum theory to the so-called  {\it Seiberg bound} \cite{MR1182173}.\footnote{\label{fn::seiberg}For an introduction, see, e.g., \cite{Ginsparg-Moore,MR1320471,MR2073993}.  The thin wedges correspond to values of $\alpha$ above this bound; however, each of the concatenated finite volume quantum surfaces which form the thin wedge locally looks like an $\wt{\alpha}$-wedge with the reflected value $\wt{\alpha} = 2Q-\alpha$ near each of its endpoints, i.e.,\ a thick wedge with $\wt{\alpha} \in(\frac{2}{\gamma},Q)$.  See Definition~\ref{def::skinny_wedge_bessel}.}

\begin{table}[ht!]
\rowcolors{1}{lightgray}{white}
\begin{center}
{\small
\renewcommand{\tabcolsep}{0.15cm}
\begin{tabular}{c|cccccc}
				&$\alpha$		&$W$	 &$\theta$		&$\delta$		 &$a$		& $\Delta$	\\
\hline
 $\alpha$	& --- & $\tfrac{\gamma}{2} +Q - \tfrac{1}{\gamma}W$  & $\tfrac{\gamma}{2}+Q-\chi \tfrac{\theta}{\pi}$ & $Q + \tfrac{\gamma}{2}(2-\delta)$ & $Q-a$ & $\gamma(1-\Delta)$\\
 $W$ & $\gamma(\tfrac{\gamma}{2}+Q-\alpha)$ & --- & $ \chi\gamma\tfrac{\theta}{\pi}$ & $\tfrac{\gamma^2}{2}(\delta-1)$ & $\gamma a + \tfrac{\gamma^2}{2}$ & $2+\gamma^2 \Delta$ \\
 $\theta$ & $\tfrac{\pi}{\chi} (\frac{\gamma}{2}+Q-\alpha)$ & $\tfrac{\pi}{\chi\gamma}W$ & --- & $\tfrac{\pi\gamma}{2\chi} (\delta-1)
 $ &  $\tfrac{\pi}{\chi} (a+\frac{\gamma}{2}) $ & $\tfrac{\pi}{\chi}(\tfrac{2}{\gamma} + \gamma \Delta)$ \\
 $\delta$ & $2 + \tfrac{2}{\gamma}(Q-\alpha)$ & $1+\tfrac{2}{\gamma^2}W$ & $1+\chi\tfrac{2}{\gamma}\tfrac{\theta}{\pi}$ & --- & $2+\tfrac{2}{\gamma}a$ & $2(\Delta + \tfrac{1}{\gamma}Q)$ \\
 $a$ & $Q-\alpha$ & $\tfrac{1}{\gamma}W - \tfrac{\gamma}{2}$ & $\chi\tfrac{\theta}{\pi}-\tfrac{\gamma}{2}$ & $\tfrac{\gamma}{2}(\delta-2)$ & --- & $\chi +  \gamma\Delta$ \\
 $\Delta$ & $1-\tfrac{\alpha}{\gamma}$ & $\tfrac{1}{\gamma^2}(W-2)$ & $\tfrac{\chi}{\gamma}\tfrac{\theta}{\pi} - \tfrac{2}{\gamma^2}$ & $\tfrac{1}{2}\delta - \tfrac{1}{\gamma}Q$ & $\tfrac{1}{\gamma}(a-\chi)$ & ---
\end{tabular}
}
\end{center}
\caption{\label{tab::wedge_parameterization} We can parameterize the space of {\bf quantum wedges} using a number of different variables that are affine transformations of each other once $\gamma$ is fixed.  Shown are six such possibilities which are important for this article: multiple ($\alpha$) of $-\log|\cdot|$ used when parameterizing by $\h$, weight ($W$), angle gap in imaginary geometry ($\theta$), dimension of Bessel process ($\delta$), drift for the corresponding Brownian motion ($a$) when parameterizing by the strip $\strip = \R \times [0,\pi]$, and the quantum scaling exponent ($\Delta$) of a fractal on the boundary of a quantum surface such that the wedge represents the local behavior at a quantum-typical point on that fractal.  ({\bf Important:} The $\theta$ here is {\em not the same} as $\vartheta$ used to define $\W_\vartheta$ above.) Each element of the table gives the variable corresponding to the row as a function of the variable which corresponds to the column.  Here, $Q=2/\gamma+\gamma/2$ and $\chi = 2/\gamma-\gamma/2$ is a constant from imaginary geometry \cite{ms2012imag1,ms2012imag2,ms2012imag3,ms2013imag4}.  Recall that Theorem~\ref{thm::welding} gives the additivity of weights under the welding operation; by combining this with the table above, we can see how the other ways of parameterizing a wedge transform under the welding operation.  In particular, the only other parameterization given above which is additive is $\theta$.
}
\end{table}

\begin{table}[ht!]
\rowcolors{1}{lightgray}{white}
\begin{center}
{\small
\renewcommand{\tabcolsep}{0.15cm}
\begin{tabular}{c|cccccc}
				&$\alpha$		& $W$ 		&$\theta$	 &$\delta$		&$a$			& $\Delta$		\\
\hline
 $\alpha$	& ---		& $Q- \tfrac{1}{2\gamma} W$    & $Q - \chi \tfrac{\theta}{2\pi}$& $Q - \tfrac{\gamma}{4}(\delta-2)$ &	$Q-a$			&	 $\gamma(1-\Delta)$\\
 $W$			&	$2\gamma(Q-\alpha)$	&	---				&	$\gamma \chi \tfrac{\theta}{\pi}$	&	$\tfrac{\gamma^2}{2}(\delta-2)$&	$2 a \gamma$	&		$\gamma^2(2\Delta-1)\!+\!4$			\\
 $\theta$	&$\tfrac{2\pi}{\chi}(Q-\alpha)$ & $\tfrac{\pi}{\gamma \chi} W$  &---				& $\gamma \tfrac{\pi}{2 \chi}(\delta-2)$			&	$\tfrac{2\pi}{\chi}a$		&	$2\pi(1+\tfrac{\gamma}{\chi}\Delta)$\\
$\delta$	&$2+\tfrac{4}{\gamma}(Q-\alpha)$ & $2 + \tfrac{2}{\gamma^2}W$ &	$2 + \chi\tfrac{2}{\gamma}\tfrac{\theta}{\pi}$				&---				&$2+\tfrac{4}{\gamma}a$ & $4\Delta + \tfrac{8}{\gamma^2}$	\\
 $a$			&$Q-\alpha$& $\tfrac{1}{2\gamma} W$ 	&	$\chi\tfrac{\theta}{2\pi}$			&	$\tfrac{\gamma}{4}(\delta-2)$ &---			& $\chi + \gamma \Delta$	\\
$\Delta$ & $1-\tfrac{\alpha}{\gamma}$& $\tfrac{1}{2}\!+\!\tfrac{1}{2\gamma^2}(W-4)$ & $\tfrac{\chi}{\gamma}(\tfrac{\theta}{2\pi}-1)$ & $\tfrac{1}{4} \delta- \tfrac{2}{\gamma^2}$ & $\tfrac{1}{\gamma}(a-\chi)$  & ---
\end{tabular}
}
\end{center}
\caption{\label{tab::cone_parameterization} As in the case of quantum wedges (Table~\ref{tab::wedge_parameterization}), we can parameterize the space of {\bf quantum cones} using a number of different variables.  Shown are six such possibilities which are important for this article: multiple of $-\log|\cdot|$ ($\alpha$) when parameterizing by~$\C$, weight~($W$), ``space of angles'' in imaginary geometry~($\theta$), dimension of Bessel process ($\delta$), and drift for the corresponding Brownian motion~($a$) when parameterizing by the cylinder $\cyl = \R \times [0,2\pi]$ (with $\R \times \{0\}$ and $\R \times \{2\pi\}$ identified), and the quantum scaling exponent ($\Delta$) of a fractal on the interior of a quantum surface such that the cone represents the local behavior at a quantum-typical point on that fractal.  Each element of the table gives the variable corresponding to the row as a function of the variable which corresponds to the column.  Here, $Q=2/\gamma+\gamma/2$ and $\chi = 2/\gamma-\gamma/2$ is a constant from imaginary geometry \cite{ms2012imag1,ms2012imag2,ms2012imag3,ms2013imag4}.}
\end{table}

We will assign a ``weight'' parameter to each quantum wedge and show that these weights are additive under gluing operations.  In particular, a wedge of weight~$W$ can be produced by welding together~$n$ independent wedges of weight $W/n$.  Taking the $n \to \infty$ limit, we will obtain a way to construct a wedge by gluing together countably many quantum surfaces that can in some sense be interpreted as wedges of weight zero.  Explicitly, the {\bf weight} of an $\alpha$-quantum wedge is the number defined from $\alpha$ as follows:
\begin{equation}
\label{eqn::wedge_weight}
W := \gamma \left(\gamma + \frac{2}{\gamma} - \alpha\right) = \gamma\left(\frac{\gamma}{2} + Q - \alpha\right).
\end{equation}
We will usually describe an $\alpha$-quantum wedge in terms of either $\alpha$ or its weight $W$ depending on the context.  Table~\ref{tab::wedge_parameterization} summarizes the relationships between several ways of parameterizing the space of wedges (all equivalent up to affine transformation, once $\gamma$ is fixed).  Two of these will be introduced at later points in this paper.   We have already introduced $\alpha$, $W$, $\delta$, and $a$.  The value $\Delta$ is a ``quantum scaling exponent.''  It turns out that at a point conditioned to intersect a random fractal of quantum scaling exponent $\Delta$ (boundary or interior) in the KPZ framework, the surface typically has a logarithmic singularity of magnitude $\alpha = \gamma -
\gamma \Delta$, see \cite[Equation (63)]{ds2011kpz} and Appendix~\ref{subapp::kpz}.  The value $\theta$ is an imaginary geometry angle, see Proposition~\ref{prop::slice_wedge_many_times}.

Taking $\alpha \leq Q$ corresponds to taking $W \geq \tfrac{\gamma^2}{2}$, which thus corresponds to a thick wedge.  If we extend the formula relating $W$ to $\alpha$ beyond this range, we find that taking $\alpha \in (Q,Q+\tfrac{\gamma}{2})$ formally corresponds to taking $W \in (0,\tfrac{\gamma^2}{2})$, which thus corresponds to a thin wedge.

For the reader who has read \cite{she2010zipper}, we briefly mention a few of the special wedge types discussed there:
\begin{enumerate}
\item Wedge obtained by zooming in at a boundary-measure-typical point: $\alpha = \gamma$, $W=2$.  (See Lemma~\ref{lem::weighted_quantum_surface_boundary} for a calculation of this type; recall that the Green's function for the GFF with free boundary conditions behaves like $-2\log|z-w|$ for $z$ in the boundary).
\item Wedge obtained by zooming in near the origin of a capacity invariant $\SLE_{\kappa}$ quantum zipper: $\alpha = -2/\gamma$, $W=4 + \gamma^2$.
\item Wedge obtained by zooming in near the origin of a quantum-length invariant $\SLE_\kappa$ quantum zipper: $\alpha = -2/\gamma + \gamma$, $W=4$. ($W=4$ arises as the weight is additive under the welding operation and the interface between two independent $W=2$ quantum wedges is an $\SLE_\kappa$.)
\end{enumerate}

We will indicate a quantum wedge $\CW$ parameterized by $D$ and described by a distribution $h$ with the notation $(D,h)$.  If we also wish to emphasize the distinguished origin (call it $z_1$) and~$\infty$ point (call it $z_2$) we will denote the wedge by $(D,h,z_1,z_2)$.

We remark that the weight $W$ of a quantum wedge has a geometric meaning in that $\pi/(\gamma \chi) W$ (where $\chi = 2/\gamma - \gamma/2$ is the constant from imaginary geometry \cite{ms2012imag1,ms2012imag2,ms2012imag3,ms2013imag4}) gives the angle gap between GFF flow lines (of $e^{i h/\chi}$) necessary to cut a wedge with weight $W' > W$ to get a wedge of weight $W$.  Indeed, this will follow from Theorem~\ref{thm::welding}.  Using weight rather than angle is often more convenient because it is more directly related to the $\rho$-values associated with the GFF flow lines.

\subsubsection{Quantum cones}

Roughly, an $\alpha$-{\bf quantum cone} is a quantum surface obtained by taking $e^{\gamma h(z)} dz$ where $\gamma \in (0,2)$ is a fixed constant and $h$ is an instance of the GFF on the infinite cone with opening angle $\vartheta$, i.e.\ the surface that arises by starting with $\W_\vartheta$ and then identifying its left and right sides according to Lebesgue measure.  Analogous to the case of a quantum wedge, by performing a change of coordinates via the map $\psi_\vartheta(\wt{z}) = \wt{z}^{\vartheta/2\pi}$ and applying~\eqref{eqn::Qmap}, we can represent an unscaled quantum cone in terms of the sum of an instance $\wt{h}$ of the whole-plane GFF plus the deterministic function $-\alpha \log|\wt{z}|$ where $\alpha = Q(1-\tfrac{\vartheta}{2\pi})$.  A quantum cone can be defined formally and precisely using the analog of Figure~\ref{fig::besseltowedge} in which the strip is replaced by the cylinder, see Section~\ref{sec::surfaces}.  As in the case of quantum wedges, there are a number of different ways of parameterizing the space of quantum cones; see Table~\ref{tab::cone_parameterization}.

We define the {\bf weight} $W$ of an $\alpha$-quantum cone to be the number
\begin{equation}
\label{eqn::quantum_cone_weight}
W := 2\gamma(Q-\alpha).
\end{equation}
Theorem~\ref{thm::zip_up_wedge_rough_statement} below will show that one can ``glue'' the two sides of a weight $W$ quantum wedge to obtain a weight $W$ quantum cone.

As in the case of wedges, we will indicate a quantum cone~$\CC$ parameterized by a domain~$D$ and described by a distribution~$h$ with the notation $(D,h)$.  If we also wish to emphasize the distinguished origin (say $z_1$) and~$\infty$ (say $z_2$) we will denote the cone by $(D,h,z_1,z_2)$.

Let us mention two of the special types of quantum cones:
\begin{enumerate}
\item Quantum cone obtained by zooming in at a quantum typical point: $\alpha=\gamma$, $W = 4-\gamma^2$.  (See Lemma~\ref{lem::weighted_quantum_surface} for a calculation of this type.)
\item Quantum cone obtained by welding together the positive and negative rays of a weight $W=2$ quantum wedge: $\alpha=\gamma/2 + 1/\gamma$, $W=2$.
\end{enumerate}
As we mentioned just above, we more generally have that the quantum cone obtained by welding together the positive and negative rays of a weight $W$ quantum wedge has weight $W$.

\subsubsection{Quantum disks and spheres}
\label{subsubsec:int_disk_sphere}

We conclude this subsection by pointing out that the Bessel process construction as described in Figure~\ref{fig::besseltowedge} gives us a way to define a family of natural {\it infinite} measures on the space of {\it finite}-volume quantum disks (or spheres).  Recall that for any $\delta < 2$ (which corresponds to a slope $a < 0$) one can define the {\bf Bessel excursion measure} $\nu_\delta^\bes$ (i.e., the It\^o excursion measure associated with the excursions that a Bessel process makes from $0$), which is an infinite measure on the space of continuous processes $X_t$ indexed by $[0,T]$ (for some random $T$) satisfying $X_0 = X_T = 0$ and $X_t > 0$ for $t \in (0,T)$.  The construction of this measure is recalled in Remark~\ref{rem::bessel_ito_excursion}.

When $\delta \in (0,2)$, one can consider a Bessel process $X_t$ and let $\ell_t$ denote the local time the process has spent at $0$ between times $0$ and $t$.  If we parameterize time by the right-continuous inverse of $\ell_t$, then the excursions appear as a Poisson point process (\ppp) on $\R_+ \times \CE$ where $\CE$ is the space of continuous functions $\phi \colon \R_+ \to \R_+$ such that $\phi(0) = 0$ and $\phi(t) = 0$ for all $t \geq T$, some $T$, and it is possible to recover the Bessel process from the \ppp\ from concatenating the countable collection of excursions.  When $\delta \leq 0$, the corresponding \ppp\ is still well-defined, but it a.s.\ assigns, to any finite time interval, a countable collection of excursions whose lengths sum to $\infty$ (so that it is not possible to define the reflecting Bessel process in the same way, essentially because there are ``too many small excursions'').  We will discuss all of these facts in more detail in Section~\ref{subsec::reversing_markov} and Section~\ref{subsec::bessel_processes}.

Given any Bessel excursion measure $\nu_\delta^\bes$, with $\delta < 2$, one can construct a doubly marked quantum surface using the procedure described in Figure~\ref{fig::besseltowedge}.  Observe that this surface looks like a thick quantum wedge with $a$ value given by $-a$ (or $\delta$ value given by $4-\delta$) near each of the two endpoints.

This suggests a way of parameterizing the family of measures on disks (those induced by the measures $\nu_\delta^\bes$ and the procedure from Figure~\ref{fig::besseltowedge}).  We define a {\bf quantum disk of weight $W$} to be a sample from the infinite measure $\diskmeasure_W$ on quantum disks induced by $\nu_\delta^\bes$ with $\delta = 3-\tfrac{2}{\gamma^2}W$.  With this choice, the surface looks like a quantum wedge of weight $W$ near each of its two endpoints.  We define parameters $\alpha$ and $\theta$ similarly.  A similar procedure allows us to define a {\bf quantum sphere of weight $W$} to be a sample from the infinite measure $\spheremeasure_W$ on quantum spheres induced by $\nu_\delta^\bes$ with $\delta$ chosen so that the surface looks like a weight $W$ quantum cone near each of its two endpoints.  We similarly define the parameters $\alpha$ and $\theta$ for quantum disks and spheres to be those of the corresponding wedges/cones.  (We will describe these constructions in more detail in Section~\ref{subsec::surfaces_strips_cylinders}.)

Note that by definition, a thin quantum wedge of weight $W$ (recall that the wedge being thin means $W \in (0,\tfrac{\gamma^2}{2}), \alpha\in (Q,Q+\frac{\gamma}{2})$) is obtained as a Poissonian concatenation of disks of weight $\wt W$, where the $a$ corresponding to $W$ is $-1$ times the $a$ corresponding to $\wt W$.  That is, $\wt W = \gamma^2 - W$, and $\wt  \alpha=2Q-\alpha$ (Table~\ref{tab::wedge_parameterization}). Thus, as mentioned in Footnote~\ref{fn::seiberg}, these marked disks locally look like thick wedges near their marked points and obey the Seiberg bound, $\wt \alpha\leq Q$.

The measure $\diskmeasure_W$ has special significance when $W=2$ so that $\alpha = \gamma$ and $\theta = 2\pi/(2 - \tfrac{\gamma^2}{2})$ and the measure $\spheremeasure_W$ has special significance when $W = 4 - \gamma^2$, so that the cone value for $\alpha$ is $\gamma$ and $\theta = 2\pi$.  In these cases, the two marked points look like ``typical'' points chosen from the boundary or bulk quantum measures (see e.g., Lemma~\ref{lem::weighted_quantum_surface_boundary} and Lemma~\ref{lem::weighted_quantum_surface}).  We use the terms {\bf quantum disk} or {\bf quantum sphere} (without specifying weight) and the symbols $\diskmeasure$ and $\spheremeasure$ to denote the infinite measures with these special weights.  See Remark~\ref{rem::bessel_ito_excursion} at the end of Section~\ref{subsec::bessel_processes} for additional detail on these constructions.

Let us explain a bit further why the definitions of $\diskmeasure$ and $\spheremeasure$ are natural.  Recall that a thick quantum wedge parameterized by $\strip$ can be constructed by taking the average of the field on vertical lines to be given (roughly) by $A_t = B_{2t} + at$ where $B$ is a standard Brownian motion and $a > 0$.  Moreover, the particular case $a = Q-\gamma$ (i.e., $W=2$) corresponds to describing the local behavior of a quantum surface with boundary near a point chosen from the boundary measure.  A quantum disk can be constructed from such a quantum wedge by ``pinching off'' a bubble of mass.  There are various ways of making sense of what this could mean and it is possible to show that they are equivalent (see \cite{quantum_spheres} for some discussion of this in the case of quantum cones and spheres).  One way of doing this is as follows.  We can take the $A_t$ process and then condition on it hitting $-r$ after the first time that it hits $0$.  Assume that the horizontal translation is taken so that $A$ first hits $0$ at time $t=0$.  If we then send $r \to \infty$, then for positive times we will get a Brownian motion but with drift $-a$ (see Lemma~\ref{lem::condition_bm_negative_drift_large} for details).  The law of the process thus obtained can equivalently be sampled from by first sampling a Bessel excursion $Z$ with dimension $3-\tfrac{4}{\gamma^2}$ conditioned on having supremum at least $1$, taking $\tfrac{2}{\gamma} \log Z$, and reparameterizing it to have quadratic variation $2 dt$.  This leads us to the natural infinite measure $\diskmeasure$ on quantum disks with the infinite measure on Bessel excursions as the starting point.  The definition of a quantum sphere is similarly motivated but starting with a weight $4-\gamma^2$ quantum cone in place of a weight $2$ quantum wedge.

One may obtain a {\bf unit boundary length quantum disk} by sampling from $\diskmeasure$ conditioned to have total boundary length $1$.  (This is now a sample from a finite measure, which can be normalized to be a probability measure.)  Similarly, a {\bf unit area quantum sphere} can be defined from $\spheremeasure$.  We define a quantum disk with an arbitrary fixed boundary length (or a quantum sphere with a fixed area) similarly.  Recall that adding a constant $C$ to the field scales quantum boundary lengths (resp.\ quantum areas) by the factor $e^{\gamma C/2}$ (resp.\ $e^{\gamma C}$).  One can therefore define a quantum disk with arbitrary boundary length $\ell > 0$ by taking the unit boundary length quantum disk and then scaling it by adding the constant $\tfrac{2}{\gamma} \log \ell$ to the field to obtain a quantum surface with boundary length $\ell$.  One can similarly define a quantum disk or sphere, respectively, with area $a > 0$ by taking the unit area sphere, respectively, and then adding $\tfrac{1}{\gamma} \log a$ to the field.

\subsection{Welding and cutting quantum wedges and cones}
\label{subsec::welding}

The current subsection and Section~\ref{subsec::matingsandloops} present several theorems involving the ``gluing together'' of infinite volume objects, including quantum wedges, infinite-volume continuum random trees (c.f.\ Section~\ref{subsec::easy}) and infinite-volume trees of disks.  We include the following chart (whose entries will all be explained later) to help the reader keep track of some of these statements.  In this chart, the parameters given for the quantum cones and wedges indicate {\em weight}.

\begin{table}[ht!]
\rowcolors{1}{lightgray}{white}
\begin{center}
{\small
\begin{tabular}{c|c| c }
			{\bf Theorem }	& {\bf Objects to be glued}		& {\bf New object/new interface}\\
\hline
\ref{thm::welding},~\ref{thm::paths_determined} & $W_1$-wedge and $W_2$-wedge & $(W_1\! +\! W_2)$-wedge/$\SLE_{\kappa}(W_1 -2 ; W_2 - 2)$ \\
\ref{thm::zip_up_wedge_rough_statement} & $W$-wedge, self & $W$-cone/whole-plane $\SLE_{\kappa}(W-2)$ \\
\ref{thm::quantum_cone_bm_rough_statement},~\ref{thm::trees_determine_embedding} & coupled pair of CRTs & $2 \gamma \chi$-cone/space-filling $\SLE_{\kappa'}$  \\
\ref{thm::sle_kp_on_wedge} & forested $W_1$-, $W_2$-wedges &  $(W_1 \! +\!  \gamma \chi \!  +\!  W_2)$-wedge/$\SLE_{\kappa'}(\rho_1; \rho_2)$ \\
\ref{thm::quantum_cone_sle_kp} & bi-forested $W$-wedge, self & $(W \! +\!  \gamma \chi)$-cone/whole-plane $\SLE_{\kappa'}(\rho)$
\end{tabular}
}
\end{center}
\end{table}

Note that for cones, the weight $2 \gamma \chi = 4 - \gamma^2$ that appears in the third row corresponds to $\alpha = \gamma$ and $\theta = 2 \pi$.  The value $\gamma \chi$ that appears in the last two rows corresponds to $\theta = \pi$.  The so-called ``forested quantum wedges'' are later introduced and explored through Theorems~\ref{thm::gluingtwoforestedlines}--\ref{thm::quantum_cone_sle_kp} and Corollary~\ref{cor::stablelength}.  These theorems describe the structure of the pair of trees of disks produced by cutting a quantum surface with a form of $\SLE_{\kappa'}$ with $\kappa' \in (4,8)$, which are then related to (an infinite time version of) the pair of L\'evy processes described in Figure~\ref{levytreegluing}.  The values of the $\rho_i$ and $\rho$ in the last two rows of the table appear in the theorem statements.

\begin{figure}[ht!]
\begin{center}
\includegraphics[scale=0.75]{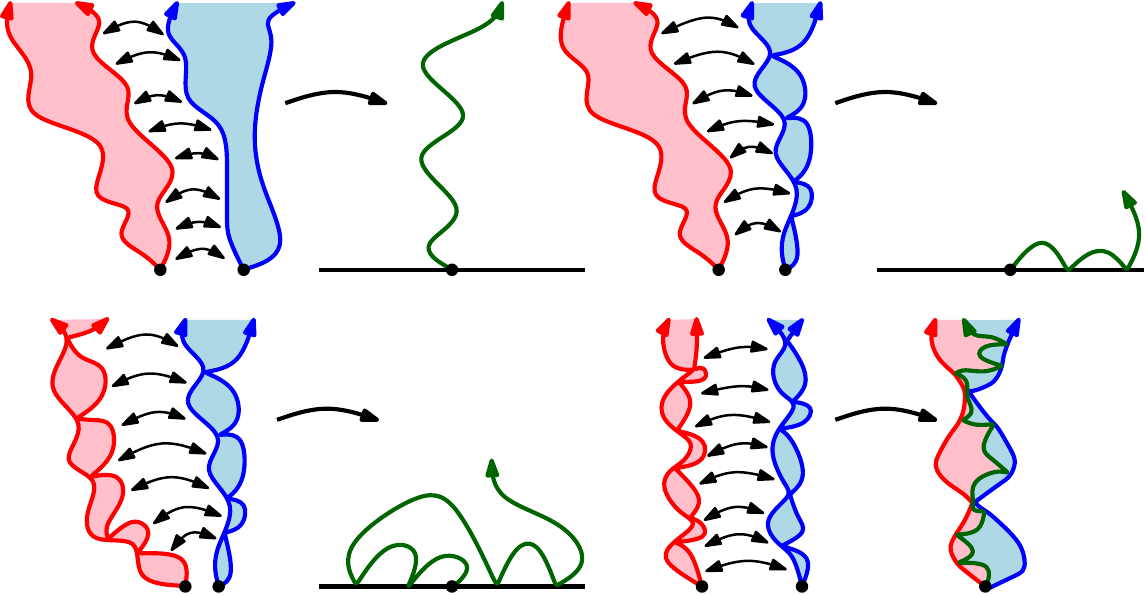}
\end{center}
\caption{\label{fig::wedge_gluing} Suppose that $\CW_1, \CW_2$ are independent quantum wedges with respective weights $W_1,W_2 > 0$.  Recall that if $W_i \geq \tfrac{\gamma^2}{2}$ then $\CW_i$ (together with its prime-end boundary) is homeomorphic to (the closure of) $\h$ and if $W_i \in (0,\tfrac{\gamma^2}{2})$ then $\CW_i$ consists of an ordered, countable sequence of beads each of which is topologically a disk.  Illustrated are the four possible scenarios considered in Theorem~\ref{thm::welding}.  Let $W = W_1+W_2$.  {\bf Top left:} $W_1,W_2 \geq \tfrac{\gamma^2}{2}$ so that both wedges are homeomorphic to $\h$.  Their conformal welding is a wedge $\CW$ of weight $W$ and the interface between them is an $\SLE_\kappa(W_1-2;W_2-2)$ process independent of $\CW$.  {\bf Top right:} The same is true if $W_1 \geq \tfrac{\gamma^2}{2}$, $W_2 \in (0,\tfrac{\gamma^2}{2})$.  Since $W_2 \in (0,\tfrac{\gamma^2}{2})$, the interface intersects the right boundary of $\CW$.  {\bf Bottom left:} The same is true if $W_1,W_2 \in (0,\tfrac{\gamma^2}{2})$ and $W \geq \tfrac{\gamma^2}{2}$.  In this case, the interface intersects both the left and right boundaries of $\CW$.  {\bf Bottom right:} If $W \in (0,\tfrac{\gamma^2}{2})$, the welding of $\CW_1$ and $\CW_2$ is still a wedge $\CW$ of weight $W$.  In this case, $\CW$ is not homeomorphic to $\h$.  Nevertheless, the interface between $\CW_1$ and $\CW_2$ is independently an $\SLE_\kappa(W_1-2;W_2-2)$ in each of the beads of $\CW$.
}
\end{figure}

\begin{figure}[ht!]
\begin{center}
\includegraphics[scale=0.85]{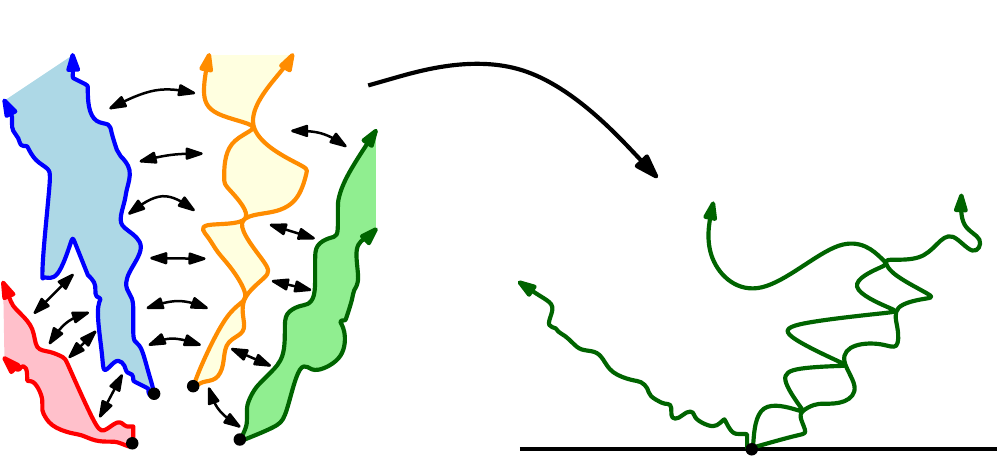}
\caption{\label{fig::wedge_series} Four quantum wedges with respective weights $W_1,\ldots,W_4$ conformally welded along their boundaries and conformally mapped to $\h$.  By Proposition~\ref{prop::slice_wedge_many_times} (the generalization of Theorem~\ref{thm::welding} to the setting of a finite number of independent wedges), the resulting surface is a wedge of weight $W_1 + \cdots W_4$.  In the illustration, $W_1,W_2,W_4 \geq \tfrac{\gamma^2}{2}$ and $W_3 \in (0,\tfrac{\gamma^2}{2})$.  The images of the interfaces are coupled $\SLE_\kappa(\rho_1;\rho_2)$ processes, corresponding to rays in a so-called imaginary geometry \cite{ms2012imag1,ms2012imag2,ms2012imag3,ms2013imag4}.  Since $W_3 \in (0,\tfrac{\gamma^2}{2})$, the middle interface intersects the rightmost interface.}
\end{center}
\end{figure}

Our first main result describes the welding of wedges with general positive weights.  (See Figure~\ref{fig::wedge_gluing} and Figure~\ref{fig::wedge_series} for illustrations.)

\begin{theorem}[Welding and cutting for quantum wedges]
\label{thm::welding}
Fix $\gamma \in (0,2)$ and choose a quantum wedge $\qwedge$ of weight $W > 0$.  Suppose $W = W_1 + W_2 \geq \tfrac{\gamma^2}{2}$ for $W_1,W_2 > 0$, $\qwedge$ is represented by $(\h,h,0,\infty)$, and then independently choose an $\SLE_\kappa(\rho_1; \rho_2)$, for $\rho_i = W_i-2$ and $\kappa = \gamma^2 \in (0,4)$, from~$0$ to~$\infty$ with force points located at $0^-$ and $0^+$.  Let $D_1$ and $D_2$ denote left and right components of $\h \setminus \eta$.  Then the quantum surfaces $\qwedge_1 = (D_1,h, 0, \infty)$ (with $h$ restricted to $D_1$) and $\qwedge_2 = (D_2,h, 0, \infty)$ (with $h$ restricted to $D_2$) are independent.  Moreover, each $\qwedge_i$ has the law of a quantum wedge with weight $W_i$.

In the case that $W \in (0,\tfrac{\gamma^2}{2})$, the same statement holds except we take $\eta$ to be a concatenation of independent $\SLE_\kappa(\rho_1;\rho_2)$ processes (one from the opening point to the closing point of each of the beads of $\qwedge$ with the force points located immediately to the left and right of the opening point) and we take $D_1$ (resp.\ $D_2$) to be the chain of surfaces which are to the left (resp.\ right) of $\eta$.
\end{theorem}

\begin{remark}
\label{rem::linearity_of_wedges}
We remark that this ``linearity of wedge weights under gluing'' (which can be extended to multiple wedges, see Figure~\ref{fig::wedge_series}) is pre-figured by certain results on non-intersection exponents that have appeared in the physics and math literatures.  For example, suppose that on a random infinite planar triangulation of the half plane, one starts $n$ simple random walks at far away locations and conditions on having them all reach the same single boundary point without intersecting each other.  One expects (as we will explain in Appendix~\ref{app::kpz}) that the (infinite volume, fine mesh) scaling limit should consist of $(2n+1)$ independent quantum wedges, with $n$ of them corresponding to a region between the left and right boundaries of a single path, and $n+1$ corresponding to a region between two paths, or between a path and the boundary.  The $n+1$ of the latter type should all have weight $2$ (essentially because this is the weight of a wedge obtained by zooming in at a ``typical boundary measure'' point, as mentioned above --- see Appendix~\ref{app::kpz}).  The $n$ wedges of the former type should all have some weight $W_0$ (which we do not specify for now). Thus the {\em total weight} of the wedge obtained by gluing these individual wedges together should be $W = n W_0 + 2(n+1) = (W_0 + 2) n + 2$.  In particular, this implies that $W-2$ should be a linear function of $n$.  Equivalently, if we take the formula $W=2+\gamma^2\Delta$ (from Table~\ref{tab::wedge_parameterization}) this means that $\Delta$ (the so-called ``boundary quantum scaling exponent'' of the non-intersecting path event, see  Appendix~\ref{app::kpz}) is a linear function of $n$.  The fact that this {\em should} be the case  was predicted and advocated by the first co-author using properties of discrete quantum gravity models and discrete analogs of the wedge weldings discussed above \cite{dup1998rwk,dup2000fractals,dup2004qg,LH2005}.  The KPZ relation expresses that Euclidean exponents are given by a quadratic function of their quantum analogs.  In particular, this suggests that an inverse quadratic function of the analogous Euclidean exponents should be a linear function of $n$; this latter fact was obtained (and termed the ``cascade relation'') in the rigorous work by Lawler and Werner on Brownian intersection exponents \cite{MR1742883,MR1796962} (See Appendix~\ref{app::kpz}.)
\end{remark}

 We also have the following:
\begin{theorem}
\label{thm::paths_determined}
In the construction of Theorem~\ref{thm::welding}, both $\qwedge$ and the interface $\eta$ are uniquely determined by the $\qwedge_i$ and may be obtained by a {\em conformal welding} of the right side of $\qwedge_1$ with the left side of $\qwedge_2$, where each is parameterized by $\gamma$-LQG boundary length.
\end{theorem}

We will not give a separate proof of Theorem~\ref{thm::paths_determined} since it follows from the same argument used to prove \cite[Theorems~1.3, 1.4]{she2010zipper} and the removability results described in Section~\ref{subsec::removability}.  (The proof of \cite[Theorem~1.4]{she2010zipper} is given in \cite[Section~1.4]{she2010zipper}.)

As mentioned briefly above, there is a natural generalization of Theorem~\ref{thm::welding} in which one cuts a quantum wedge~$\CW$ with $n$ $\SLE$ processes (as opposed to cutting with a single path) coupled together as flow lines with varying angles of a GFF \cite{ms2012imag1,ms2012imag2,ms2012imag3,ms2013imag4} which is independent of $\CW$.  In this case, one obtains $n+1$ independent wedges and the sum of their weights is equal to the weight of $\CW$.  This result is illustrated in Figure~\ref{fig::wedge_series} and is stated precisely in Proposition~\ref{prop::slice_wedge_many_times}.  If one lets the number of paths tend to $\infty$ (with the spacing between the angles going to zero), then the collection of paths converges to the so-called {\bf fan} $\fan$.  (See Figures~1.2--1.5 of \cite{ms2012imag1} for simulations of $\fan$.)  This is a random closed set which a.s.\ has zero Lebesgue measure \cite[Proposition~7.33]{ms2012imag1}.  As a consequence of Proposition~\ref{prop::slice_wedge_many_times}, it should be possible to deduce an exact Poissonian structure of the countable collection of surface ``beads'' parameterized by the complement of $\fan$.  In particular, this would make it possible to interpret $\fan$ as describing the interface that arises when one glues together a Poissonian collection of wedges with weight $0$.

\begin{figure}[ht!]
\begin{center}
\includegraphics[scale=0.85]{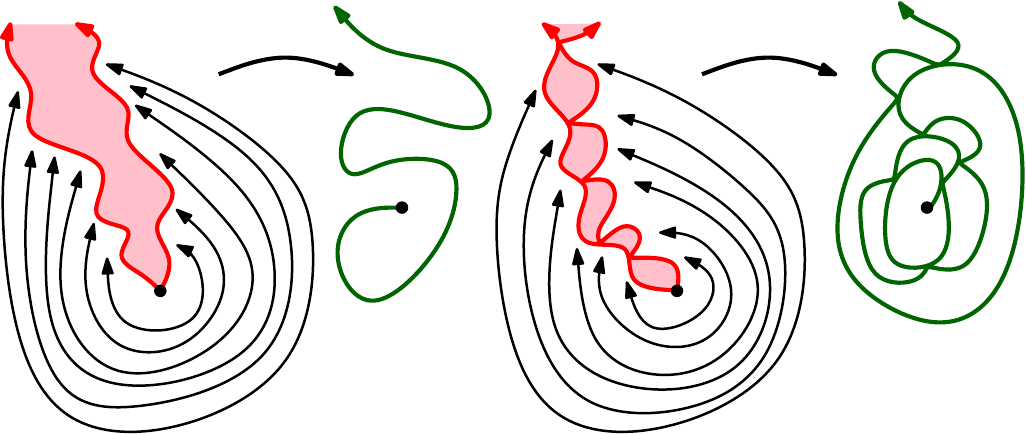}
\end{center}
\caption{\label{fig::fat_wedge_zipped_up} Illustration of Theorem~\ref{thm::zip_up_wedge_rough_statement}.  {\bf Left:} A quantum wedge of weight $W \geq \tfrac{\gamma^2}{2}$ (hence together with its prime-end boundary is homeomorphic to $\ol{\h}$).  If we weld together its left and right sides according to $\gamma$-LQG boundary length, then the resulting surface is a quantum cone of the same weight $W$ (right side) decorated with an independent whole-plane $\SLE_\kappa(W-2)$ process.  {\bf Right:}  The same statement holds in the case that  $W \in (0,\tfrac{\gamma^2}{2})$ so that the wedge is not homeomorphic to~$\h$ but rather is given by a Poissonian sequence of disks; in this case, $W-2 \in (-2,\tfrac{\kappa}{2}-2)$ so that the $\SLE_\kappa(W-2)$ process is self-intersecting.}
\end{figure}

\begin{table}[ht!]
\rowcolors{1}{lightgray}{white}
\begin{center}
{\small
\subfloat[Conversion from wedge to cone parameterizations when zipping up a quantum wedge to obtain a quantum cone.]{
\begin{tabular}{c|c c c c c}
				&$\alpha$		&$W$	 &$\theta$		&$\delta$		 &$a$			\\
\hline
 $\alpha'$	& $\tfrac{1}{2}\alpha + \tfrac{1}{\gamma}$ & $Q - \tfrac{1}{2\gamma} W$  & $Q - \chi \tfrac{\theta}{2\pi}$ & $Q+ \tfrac{\gamma}{4}(1-\delta)$ & $Q-\tfrac{\gamma}{4} - \tfrac{1}{2}a$  \\
 $\theta'$ & $\tfrac{\pi}{\chi} (\frac{\gamma}{2}+Q-\alpha)$ & $\tfrac{\pi}{\chi\gamma}W$ & $\theta$ & $\tfrac{\pi\gamma}{2\chi} (\delta-1)
 $ &  $\tfrac{\pi}{\chi} (a+\frac{\gamma}{2}) $\\
$\delta'$ & $3 + \tfrac{2}{\gamma}(Q-\alpha)$ & $2 + \tfrac{2}{\gamma^2}W$ & $2 +  \chi \tfrac{2}{\gamma}\tfrac{\theta}{\pi}$ & $\delta+1$ & $3+\tfrac{2}{\gamma}a$\\
 $a'$ & $Q-\tfrac{1}{\gamma} - \tfrac{1}{2}\alpha$ & $\tfrac{1}{2\gamma} W$ & $\chi \tfrac{\theta}{2\pi}$ & $ \tfrac{\gamma}{4}(\delta-1)$ &  $\tfrac{\gamma}{4} + \tfrac{1}{2}a$\\
 $\rho'$ & $\gamma^2 - \alpha \gamma$ & $W-2$ & $\chi \gamma \tfrac{\theta}{\pi} -2$ &  $\tfrac{\gamma^2}{2} (\delta-1) -2$ & $\gamma a+\tfrac{\gamma^2}{2}-2$
\end{tabular}}

\vspace{0.025\textheight}

\rowcolors{1}{lightgray}{white}
\subfloat[Conversion from cone to wedge parameterizations when cutting a quantum cone to obtain a quantum wedge.]{\begin{tabular}{c|c c c c c c}
	& $\alpha'$ & $\theta'$ & $\delta'$ & $a'$ & $\rho'$\\
\hline
$\alpha$ & $2\alpha' - \tfrac{2}{\gamma}$ & $Q + \tfrac{\gamma}{2} - \chi\tfrac{\theta'}{\pi}$ & $Q + \tfrac{\gamma}{2}(3-\delta')$ & $Q + \tfrac{\gamma}{2} - 2a'$ & $\gamma - \tfrac{1}{\gamma}\rho'$  \\
$W$  & $2\gamma(Q-\alpha')$ & $\gamma \chi \tfrac{\theta'}{\pi}$ & $\tfrac{\gamma^2}{2} (\delta'-2)$ & $2\gamma a'$ &  $2+\rho'$  \\
$\theta$  & $\tfrac{2\pi}{\chi}(Q-\alpha')$ & $\theta'$ & $\tfrac{\pi \gamma}{2\chi}(\delta'-2)$ & $\tfrac{2\pi}{\chi} a' $ & $\tfrac{\pi}{\chi \gamma}(\rho'+2)$  \\
$\delta$  & $1 + \tfrac{4}{\gamma}(Q-\alpha')$ & $1+\chi\tfrac{2}{\gamma}\tfrac{\theta'}{\pi}$ & $\delta'-1$ & $1 + \tfrac{4}{\gamma} a'$ & $1+\tfrac{2}{\gamma^2}(\rho'+2)$\\
$a$  & $\tfrac{2}{\gamma} + Q - 2\alpha'$ & $\chi \tfrac{\theta'}{\pi} - \tfrac{\gamma}{2}$ & $\tfrac{\gamma}{2}(\delta'-3)$ & $2a' - \tfrac{\gamma}{2}$ & $\tfrac{1}{\gamma}(\rho'+2)-\frac{\gamma}{2}$
\end{tabular}}
}
\end{center}
\caption{\label{tab::cone_wedge_values} If we zip up the left and right sides of a quantum wedge according to $\gamma$-LQG boundary length, then by Theorem~\ref{thm::zip_up_wedge_rough_statement} we get a quantum cone decorated with an independent whole-plane $\SLE_\kappa(\rho)$ process, and, conversely, if we cut a quantum cone with an independent whole-plane $\SLE_\kappa(\rho)$ then we get a quantum wedge.  {\bf a)}  Each of the entries in the first row gives a way of parameterizing the wedge and each entry in the first column gives a way of parameterizing the resulting cone.  The variable $\rho'$ refers to the $\rho$-value for the whole-plane $\SLE_\kappa(\rho)$ process which is the image of $\R$ under the zipping up map.  Each entry of the table gives the type of cone that one gets by zipping up a wedge where the cone is described using the parameterization corresponding to the row of the entry and the wedge is described using the parameterization of the column of the entry.  Note that the row corresponding to $\theta'$ is identical to the row corresponding to $\theta$ in Table~\ref{tab::wedge_parameterization}.  This follows because the ``space of angles'' in the imaginary geometry for a given wedge does not change under the zipping up operation.  {\bf b)} Each entry of the table gives the type of wedge one gets by cutting a cone with an appropriate $\SLE$ process where the cone is described using the parameterization from the column and the wedge is parameterized using the variable from the corresponding row.}
\end{table}

Our next main result implies that a quantum cone can be constructed by identifying the left and right sides of a quantum wedge according to $\gamma$-LQG boundary length.

\begin{theorem}[Welding and cutting for quantum cones]
\label{thm::zip_up_wedge_rough_statement}
Fix $\gamma \in (0,2)$, let $\kappa=\gamma^2\in(0,4)$, and suppose that $\CC = (\C,h,0,\infty)$ is a quantum cone of weight $W > 0$.  Let $\rho = W-2$ and suppose that $\eta$ is a whole-plane $\SLE_\kappa(\rho)$ process independent of $\CC$ starting from~$0$.  Then the quantum surface $\CW$ described by $(\C \setminus \eta,h,0,\infty)$ is a quantum wedge of weight $W$.  Moreover, the pair consisting of $\CC$ and $\eta$ is a.s.\ determined by $\CW$ and can be obtained by conformally welding the left boundary of $\CW$ with its right boundary according to $\gamma$-LQG boundary length. 
\end{theorem}

See Figure~\ref{fig::fat_wedge_zipped_up} for an illustration of Theorem~\ref{thm::zip_up_wedge_rough_statement}.  This result is stated in the case of simple $\SLE_\kappa(\rho)$ processes ($W \geq \tfrac{\gamma^2}{2}$ so that $\rho \geq \tfrac{\kappa}{2}-2$) in Proposition~\ref{prop::cone_divide} and in the case of self-intersecting $\SLE_\kappa(\rho)$ processes ($W \in (0,\tfrac{\gamma^2}{2})$ so that $\rho \in (-2,\tfrac{\kappa}{2}-2)$) in Proposition~\ref{prop::cone_divide_self_intersecting}.  Moreover, in Proposition~\ref{prop::slice_cone_many_times} it is explained that slicing a quantum cone with a collection of whole-plane $\SLE_\kappa(\rho)$ processes coupled together as flow lines of a whole-plane GFF starting from the origin \cite{ms2013imag4} yields a collection of independent wedges.  

If we parameterize our cone in terms of $\alpha$ rather than $W$, then it follows from~\eqref{eqn::quantum_cone_weight} that the value of $\rho$ from Theorem~\ref{thm::zip_up_wedge_rough_statement} is given by
\begin{equation}
\label{eqn::cone_rho}
\rho = 2+\gamma^2-2\alpha \gamma.
\end{equation}
Theorem~\ref{thm::zip_up_wedge_rough_statement} combined with~\eqref{eqn::wedge_weight} and~\eqref{eqn::quantum_cone_weight} tells us that zipping up the left and right sides of an $\alpha'$-quantum wedge yields an $\alpha$-quantum cone with
\begin{equation}
\label{eqn::cone_wedge_value}
\alpha = \frac{\alpha'}{2}+\frac{1}{\gamma}.
\end{equation}
See Table~\ref{tab::cone_wedge_values} for the conversion between several parameterizations for quantum wedges and cones when performing the welding or cutting operation from Theorem~\ref{thm::zip_up_wedge_rough_statement}.

\subsection{Warmup: topological matings of trees}
\label{subsec::easy}

For the reader who has not seen similar constructions before, the idea that one can ``glue together'' or ``mate'' two continuum random trees to produce a sphere may seem surprising.  Such matings appear in the complex dynamics literature, as we recall in Section~\ref{subsec::conformalmating}, where in some sense one ``mates together'' two tree-like Julia sets to obtain a sphere decorated by space-filling curve.  They also appear in the construction of a random surface called the {\em Brownian map} --- see the pioneering works of Cori and Vauquelin \cite{cori1981planar}, Chassaing and Schaeffer \cite{chassaing2004random}, Marckert and Mokkadem \cite{marckert2006limit}, Le Gall \cite{topological2007LeGall}, and Le Gall and Paulin \cite{homeomorphic2008legallpaulin}.  The Brownian map can be understood as a mating of two trees constructed from the {\em Brownian snake}. The two trees in the Brownian map construction (unlike those of the current paper) have very different laws from one another, and can be understood as the geodesic and dual geodesic trees on the resulting surface. Both the complex dynamics and Brownian map literatures (in particular \cite{homeomorphic2008legallpaulin}) make use of a classical result called {\em Moore's theorem} to show that the object obtained in the end is a topological sphere.

The purpose of this brief warmup section is to introduce continuum random trees (CRTs) and explain how Moore's theorem enables us to produce a {\em topological} sphere from a pair of CRTs with relatively little work. The main results of this paper are much stronger than the statements briefly described in this subsection, because the main results endow the mated trees with a {\em conformal structure} (not merely a topological structure) and relate them to Liouville quantum gravity and SLE (although, to be precise, this paper defines the conformal structure for an infinite volume version of the mating of trees; finite volume variants are treated in a follow up paper \cite{quantum_spheres}). The compatibility of the topological mating of trees and the conformal mating of trees is explained in Section~\ref{subsec::conftopcompatibility}, where it is shown that the two constructions are a.s.\ homeomorphic in the obvious way. The reader who is familiar with CRTs and wishes to proceed to the statements of these stronger results may skip or skim this brief subsection.

Let $X_t$ and $Y_t$ be independent Brownian excursions, both indexed by $t \in [0,T]$.  Thus $X_0 = X_T = 0$ and $X_t > 0$ for $t \in (0,T)$ (and similarly for $Y_t$).  Once~$X_t$ and~$Y_t$ are chosen, choose $C$ large enough so that the graphs of $X_t$ and $C-Y_t$ do not intersect.  (The precise value of $C$ does not matter.)  Let $R = [0,T] \times [0,C]$, viewed as a Euclidean metric space.

Let $\cong$ denote the finest equivalence relation (i.e., has the smallest equivalence classes) on $R$ that makes two points equivalent if they lie on the same vertical line segment with endpoints on the graphs of $X_t$ and $C-Y_t$, or they lie on the same horizontal line segment that never goes above the graph of $X_t$ (or never goes below the graph of $C-Y_t$).  Maximal segments of this type are shown in Figure~\ref{fig::lamination}.

\begin {figure}[htbp!]
\begin {center}
\includegraphics[scale=0.85]{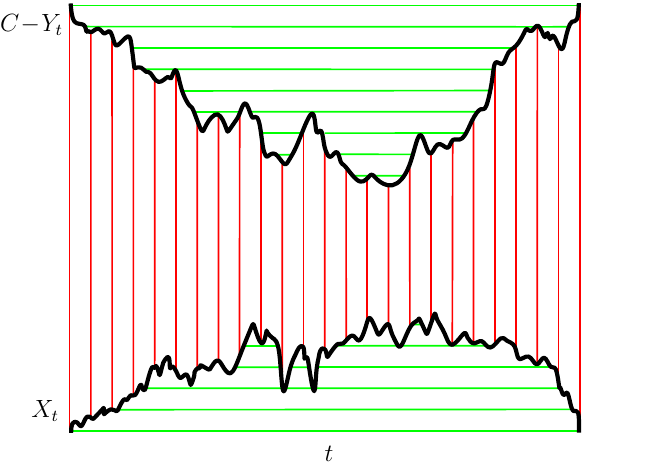}
\caption {\label{fig::lamination}  Points on the same vertical (or horizontal) line segment are equivalent.}
\end{center}
\end{figure}

Observe that the vertical line segment corresponding to time $t \in [0,T]$ shares its lower (resp.\ upper)  endpoint with a horizontal line segment if and only if $X_s > X_t$ (resp.\ $Y_s > Y_t$) for all $s$ in a neighborhood of the form $(t,t+\epsilon)$ or $(t-\epsilon, t)$ for $\epsilon > 0$.  It is easy to see that (aside from the $t=0$ and $t=T$ segments) there a.s.\ exists no vertical line segment that shares both an upper and a lower endpoint with a horizontal line segment.\footnote{It can be shown the set of record minima of a one dimensional Brownian motion agrees in law with the zero set of a one dimensional Brownian motion; the fact that no such minima occur simultaneously is related to the fact that planar Brownian motion a.s.\ does not hit any specific point.  More general statements, related to so-called ``cone times,'' appear in \cite{EVANS_CONE_TIMES} and will be discussed in Section~\ref{sec::brownian_boundary_length}.}  It is also easy to see that each of $X_t$ and $Y_t$ a.s.\ has at most countably many local minima, and that a.s.\ the values obtained by $X_t$ or $Y_t$ at these minima are distinct.  It follows that a.s.\ each of the maximal horizontal segments of Figure~\ref{fig::lamination} intersects the graph of $X_t$ or $C-Y_t$ in two or three places (the two endpoints plus at most one additional point).

Thus the equivalence classes a.s.\ all have one of the following types:
\begin{enumerate}
\item[Type 0:] The outer boundary rectangle $\partial R$.
\item[Type 1:] A single vertical segment that does not share an endpoint with a horizontal segment.
\item[Type 2:] A single maximal horizontal segment beneath the graph of $X_t$ or above the graph of $C-Y_t$, together with the two vertical segments with which it has an endpoint in common.
\item[Type 3:] A single maximal horizontal segment beneath the graph of $X_t$ or above the graph of $C-Y_t$, together with the two vertical segments with which it has an endpoint in common, and one additional vertical segment with an endpoint in the interior of the horizontal segment.
\end{enumerate}
\begin{proposition} \label{prop::topclosed}
The relation $\cong$ is a.s.\ {\em topologically closed}.  That is, if $x_j \to x$, and $y_j \to y$, and $x_j \cong y_j$ for all $j$, then $x \cong y$.
\end{proposition}
\begin{proof}
By passing to a subsequence, one may assume that the type of the equivalence class of $x_j \cong y_j$ is the same for all $j$.  The equivalence class can be described by a finite collection of numbers (the endpoint coordinates for the line segments), and compactness implies that we can find a subsequence along which this collection of coordinates converges to a limit.  We then observe that these limiting numbers describe an equivalence class (or a subset of an equivalence class) comprised of zero or one horizontal lines that lie either below the graph of $X_t$ or above the graph of $C-Y_t$, and one or more incident vertical lines that lie between these graphs.  Since $x$ and $y$ necessarily belong to this limiting equivalence class, we have $x \cong y$. \end{proof}

Let $\wt R$ be the topological quotient $R / \cong$.  Then we have the following:
\begin{proposition} \label{prop::topsphere}
The space $\wt R$ is a.s.\ topologically equivalent to the sphere $\s^2$.  The map $\phi$ that sends $t$ to the equivalence class containing $(t,X_t)$ and $(t,C-Y_t)$ is a continuous surjective map from $[0,T]$ to $\wt R$. 
\end{proposition}
\begin{proof}
Let $R'$ be obtained from $R$ by identifying the outer boundary of $R$ with a single point.  Then $R'$ is topologically a sphere, and $\cong$ induces a topologically closed relation on $R'$.  The fact that $\wt R$ is topologically a sphere is immediate from the properties observed above and Proposition~\ref{prop::moore}, stated just below.  The continuity of $\phi$ is immediate from the fact that the quotient map from $R$ to $\wt R$ is continuous.
\end{proof}

The following was established by R.L.\ Moore in 1925 \cite{moore1925concerning} (the formulation below is lifted from \cite{milnor2004pasting}):

\begin{proposition}
\label{prop::moore}
Let $\cong$ be any topologically closed equivalence relationship on the sphere $\s^2$.  Assume that each equivalence class is connected, but is not the entire sphere.  Then the quotient space $\s^2 / \cong$ is itself homeomorphic to $\s^2$ if and only if no equivalence class separates the sphere into two or more connected components.
\end{proposition}

If $R_1$ is the set of points in $R$ on or under the graph of $X$ and $\cong_1$ is the finest equivalence relation on $R_1$ which identifies points if they lie on a horizontal chord which lies entirely below the graph of $X$, then $\wt R_1 = R_1 / \cong_1$ is a random metric space with distance given by the metric quotient with respect to $\cong_1$ of the internal metric on $R_1$ induced by the Euclidean metric.\footnote{This metric is usually described in terms of the pseudometric $d_1(s,t) = X_s + X_t - 2\inf_{r \in [s,t]} X_r$ defined on $[0,T]$.}  This metric space is called a {\em continuum random tree}, constructed and studied by Aldous in the early 1990's \cite{ald1991crt1,ald1991crt2,ald1993crt3}. See also the general construction of metric trees from excursions in \cite{dlg2002trees_levy}.

Similarly, if $R_2$ is the set of points on or above the graph of $C-Y$ and $\cong_2$ is the finest equivalence relation on $R_2$ which identifies points if they lie on a horizontal chord which lies entirely above the graph of $C-Y$, then $\wt R_2 / \cong_2$ is another continuum random tree.  Thus, $\cong$ induces an equivalence relation on the disjoint union of $\wt R_1$ and $\wt R_2$, hence $\wt R$ can be understood as a topological quotient of this pair of metric trees.  We can define a measure $\nu$ on $\wt R$ by letting $\nu(A)$ be the Lebesgue measure of $\phi^{-1}(A)$.    Observe that for Lebesgue almost all times $t$, the value $\phi^{-1} \bigl(\phi(t)\bigr)$ consists of the single point $t$.   Thus $\nu$ is supported on points that are hit by the space-filling path exactly once.  The set of double points (i.e., points in~$\wt{R}$ hit twice by the path) a.s.\ has $\nu$ measure zero but is uncountable.  The set of triple points (i.e. points in~$\wt{R}$ hit three times by the path) a.s.\ has $\nu$ measure zero and is countable.

Throughout most of this paper, we will actually work on a more general case in which $(X_t,Y_t)$ is an affine transformation of a standard two dimensional Brownian motion such that
\begin{enumerate} \item $X_t+Y_t$ is a standard Brownian motion,
\item $X_t - Y_t$ is an independent {\em constant multiple} of a standard Brownian motion.
\end{enumerate}
The constant in question encodes how {\em correlated} the Brownian processes $X_t$ and $Y_t$ are with one another. We remark that, as one may observe from Figure~\ref{fig::lamination}, replacing $X_t$ and $Y_t$ with $aX_t$ and $bY_t$, where $a$ and $b$ are positive constants, does not actually affect the construction.  Thus, if one assumes that $(X_t,Y_t)$ is an affine transformation of a standard two dimensional Brownian motion, the correlation constant is the only parameter that is really relevant.
Theorem~\ref{thm::quantum_cone_bm_rough_statement} gives a relationship between this constant and the parameter $\kappa'= 16/\gamma^2$, a relationship proved to hold at least in the range $\kappa' \in (4,8]$.  (The parameter is identified for $\kappa' > 8$ in \cite{ghms2015covariance}.)

 Although we do not give any details in this subsection, we mention that one can also give a purely topological description of the mating of ``L\'evy trees of quantum disks'' that this paper develops.  The rough idea is illustrated in Figure~\ref{treegluingvstreeofdiskgluing} and Figure~\ref{levytreegluing}.  (If we cut out the grey filled regions from one of the trees in Figure~\ref{levytreegluing}, we obtain a tree of loops; this tree of loops was also analyzed in the context of scaling limits of random planar maps in \cite{ck2013looptrees}.)

\begin {figure}[h!]
\begin {center}
\includegraphics [scale=1]{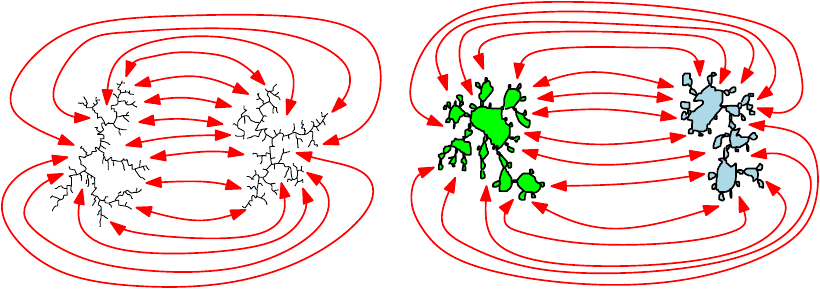}
\caption {\label{treegluingvstreeofdiskgluing} {\bf Left:} If the two trees represented by the upper and lower graphs in Figure~\ref{fig::lamination} are somehow embedded in the plane in a way that preserves the cyclic order around the tree (let us not worry about precisely how) then the vertical lines of Figure~\ref{fig::lamination} correspond to the red lines shown, identifying points along one tree with points along the other.  {\bf Right:} Similar, but the trees are replaced with ``trees of quantum disks'', each somehow embedded in the plane (again, let us not worry about how).  Both left and right identification procedures produce a topological sphere, which we embed canonically in $\C \cup \{\infty \}$.  In both cases the cyclical ordering on the set of red lines describes a loop on the embedded sphere.  In the left scenario, this loop is a space-filling form of $\SLE_{\kappa'}$, with $\kappa' > 4$.   In the right scenario, it is a non-space-filling $\SLE_{\kappa'}$ with $\kappa' \in (4,8)$.}
\end {center}
\end {figure}

\begin {figure}[ht!]
\begin {center}
\includegraphics [scale=.8]{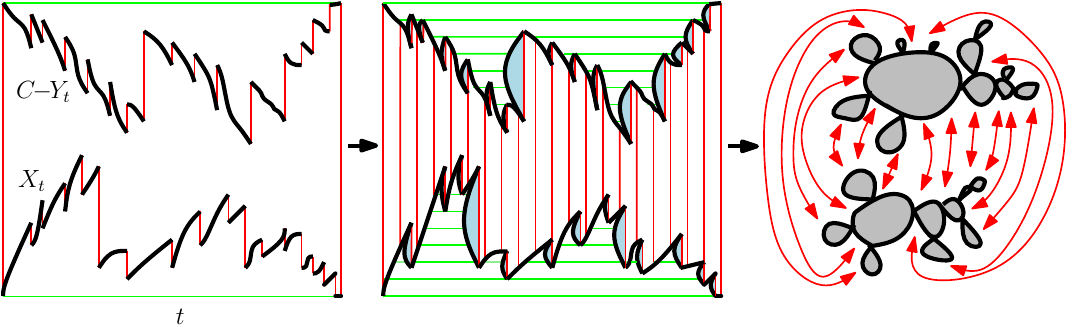}
\caption {\label{levytreegluing} {\bf Left:}  $X_t$ and $Y_t$ are i.i.d.\ L\'evy excursions, each with only negative jumps.  Graphs of $X_t$ and $C-Y_t$ are sketched; red segments indicate jumps. {\bf Middle:} Add a black curve to the left of each jump, connecting its two endpoints; the precise form of the curve does not matter (as we care only about topology for now) but we insist that it intersects each horizontal line at most once and stay below the graph of $X_t$ (or above the graph of $C-Y_t$) except at its endpoints.  We also draw the vertical segments that connect one graph to another, as in Figure~\ref{fig::lamination}, declaring two points equivalent if they lie on the same such segment (or on the same jump segment).  Shaded regions (one for each jump) are topological disks.  {\bf Right:} By collapsing green segments and red jump segments, one obtains two trees of disks with outer boundaries identified, as on the right side of Figure~\ref{treegluingvstreeofdiskgluing}.   }
\end {center}
\end {figure}

\subsection{Conformal matings of trees and trees of disks} 
\label{subsec::matingsandloops}

\subsubsection{Main result on gluing infinite volume CRTs}

In this section, we will see how the Brownian motion pair $(X_t,Y_t)$, as discussed in Section~\ref{subsec::easy}, encodes various objects within LQG and $\SLE$.   Here and throughout much of the rest of the paper, we will use the symbols $(L_t, R_t)$ in place of $(X_t,Y_t)$ in order to highlight the fact that (as we will explain) in many circumstances, $L_t$ and $R_t$ can be interpreted as left and right boundary lengths of the quantum surface parameterized by $\{s : s \leq t \}$, or of the quantum surface parameterized by $\{s : s \geq t \}$.   We include many figures (Figures~\ref{fig::fig1}--\ref{fig::fig10}) which illustrate Theorem~\ref{thm::quantum_cone_bm_rough_statement} and Theorem~\ref{thm::trees_determine_embedding} below.

\begin{figure}[ht!]
\begin{center}
\includegraphics[scale=0.85]{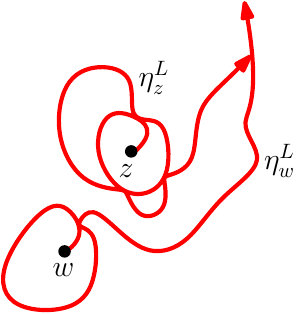}
\end{center}
\caption{\label{fig:space_filling_sle} Shown are flow lines $\eta_z^L,\eta_w^L$ of a whole-plane GFF started at points $z,w$, both with angle $\tfrac{\pi}{2}$.  In the definition of space-filling $\SLE$, we say that $w$ comes before $z$ if $\eta_w^L$ merges into $\eta_z^L$ on its right side.}
\end{figure}

We will now recall from \cite{ms2013imag4} the construction of space-filling $\SLE_{\kappa'}$.  It is a variant of $\SLE_{\kappa'}$ which fills up the components it separates from its target point.  We will begin by focusing on the particular case of {\bf whole-plane space-filling $\SLE_{\kappa'}$ from~$\infty$ to~$\infty$}.  Suppose that~$h$ is a whole-plane GFF with values defined modulo a global additive constant in $2\pi \chi \Z$, where recall that $\chi=2/\sqrt{\kappa} - \sqrt{\kappa}/2$ and $\kappa=16/\kappa'$.  For each $z \in \C$, we let $\eta_z^L$ (resp.\ $\eta_z^R)$ be the flow line of $h$ starting from $z$ with angle $\tfrac{\pi}{2}$ (resp.\ $-\tfrac{\pi}{2}$).  Then~$\eta_z^L$ (resp.\ $\eta_z^R$) has the law of a whole-plane $\SLE_\kappa(2-\kappa)$ process from $z$ to $\infty$ \cite[Theorem~1.1]{ms2013imag4}.  We can use the paths~$\eta_z^L$ (equivalently~$\eta_z^R$) to define an ordering on~$\C$ as follows.  Suppose that $z,w \in \C$.  We note that~$\eta_z^L$ will a.s.\ merge with $\eta_w^L$ \cite[Theorem~1.7]{ms2013imag4}.  We say that~$w$ comes before~$z$ if it is the case that~$\eta_w^L$ merges with~$\eta_z^L$ on its right side.  Equivalently, $w$ comes before $z$ if $\eta_w^R$ merges with $\eta_z^R$ on its left side.  Suppose that $(z_n)$ is a deterministic countable dense set which is ordered in this way.  Space-filling $\SLE_{\kappa'}$ is the unique, continuous, non-self-crossing and non-self-tracing path which visits the $z_n$ according to this ordering \cite{ms2013imag4}.  It is not in fact immediately obvious that such a path exists, but this is proved in \cite{ms2013imag4}.  It is also proved in \cite{ms2013imag4} that the resulting path does not depend on the initial choice of countable, dense set.  Since the $\eta_{z_n}^L$ are a.s.\ determined by $h$, so is $\eta'$.  For each fixed $z \in \C$, we thus have that $\eta_z^L$ (resp.\ $\eta_z^R$) gives the left (resp.\ right) side of the outer boundary of $\eta'$ stopped upon hitting $z$.  Space-filling $\SLE_{\kappa'}$ can be interpreted as tracing (clockwise or counterclockwise) the outside of a (space-filling) tree of GFF flow lines.\footnote{The chordal version of space-filling $\SLE_{\kappa'}$ (i.e., in $\h$ from $0$ to $\infty$) is discussed in detail in \cite{ms2013imag4}.  The whole-plane version from $\infty$ to $\infty$ is not explicitly constructed and shown to be continuous in \cite{ms2013imag4}.  However, it can be easily be constructed and shown to be continuous using chordal space-filling $\SLE_{\kappa'}$.  To accomplish this, one first starts flow and dual flow lines of a whole-plane GFF starting from $0$.  If $\kappa' \geq 8$, these flow lines will partition space into two regions which (with their prime-end boundaries) are each homeomorphic to (the closure of) $\h$.  In this case, a whole-plane space-filling $\SLE_{\kappa'}$ from $\infty$ to $\infty$ can be constructed by splicing together two chordal space-filling $\SLE_{\kappa'}$'s, one for each of the two regions.  The first path is taken to run from $\infty$ to $0$ and the second from $0$ to $\infty$.  If $\kappa' \in (4,8)$, then the flow and dual flow lines started from $0$ will partition space into a countable collection of pockets.  In this case, a whole-plane space-filling $\SLE_{\kappa'}$ from $\infty$ to $\infty$ can be constructed by splicing together a countable collection of chordal space-filling $\SLE_{\kappa'}$'s, one for each of the pockets.}

For fixed $z \in \C$, we can also consider $\eta'$ targeted at $z$.  This means that $\eta'$ does not branch into and fill the components that it separates from $z$.  Equivalently, we can take $\eta'$ and reparameterize it according to capacity as seen from $z$.  Then the resulting path is the (non-space-filling) {\bf counterflow} line of $h$ from $\infty$ to $z$ and has the law of an $\SLE_{\kappa'}(\kappa'-6)$ process.

\begin{figure}[ht!]
\begin{center}
\includegraphics[scale=0.85]{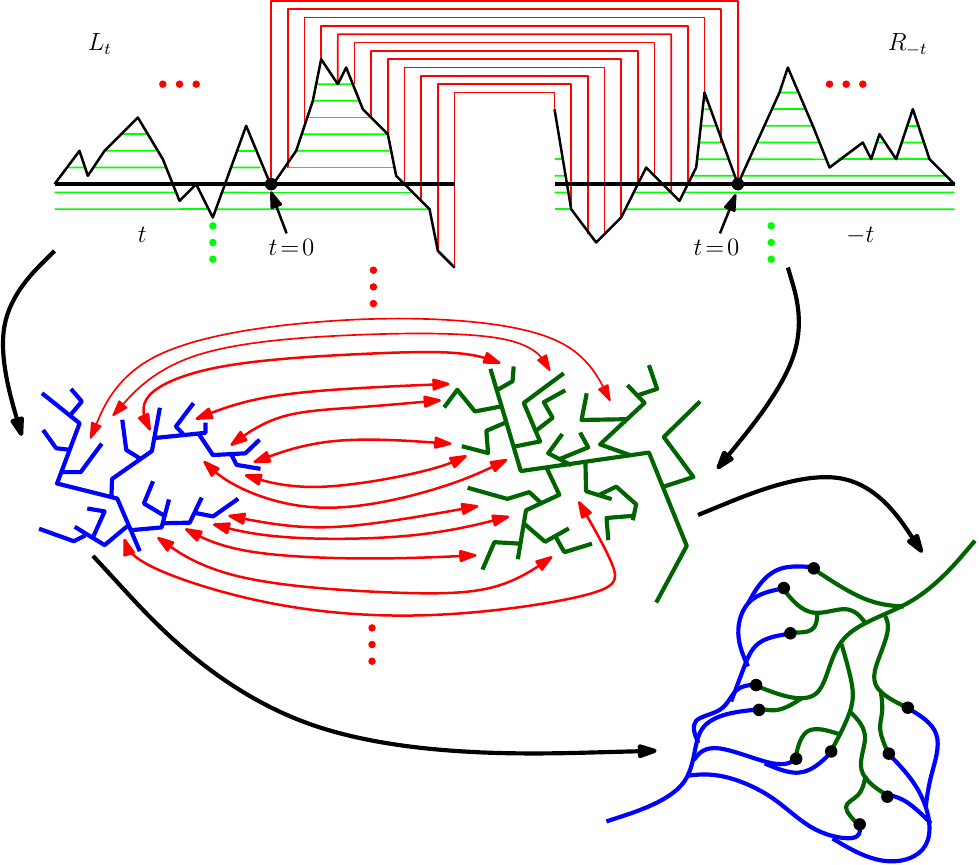}
\end{center}
\caption{\label{fig::fig1}  Gluing together two infinite volume space-filling trees, encoded by correlated Brownian motions $L_t$ and $R_t$, produces a $\gamma$-quantum cone (i.e., $W=4-\gamma^2$ and $\theta=2\pi$).  This cone is decorated by a space-filling $\SLE_{\kappa'}$ from $\infty$ to $\infty$, which in turn encodes the east-going and west-going rays of an imaginary geometry.}
\end{figure}

We will often consider a whole-plane space-filling $\SLE_{\kappa'}$ process $\eta'$ from $\infty$ to $\infty$ on top of an independent $\gamma$-quantum cone $\CC = (\C,h,0,\infty)$.  We take the time parameterization of $\eta'$ so that $\eta'(0) = 0$ and $\mu_h(\eta'([s,t])) = t-s$ for all $s < t$.  In this setting, we can define for each $t \geq 0$ the process $L_t$ (resp.\ $R_t$) which is equal to the change in the length of the left (resp.\ right) boundary of $\eta'$ relative to time $0$.  In other words, recall that the left boundary of $\eta'\bigl((-\infty,t]\bigr)$ and the left boundary of $\eta'\bigl( (-\infty,0]\bigr)$ are both given by flow lines of a common angle that merge, and that come with well defined length measures (recall~\eqref{e.nudef}). Then $L_t$ is the length of the former boundary segment (from $\eta'(t)$ to the merging point) minus the length of the latter (from $\eta'(0)$ to the merging point).  Moreover, $R_t$ is defined in the same manner but with right in place of left.  Our next theorem describes the process $(L,R)$, which we emphasize is defined for all $t \in \R$.

\begin{theorem}
\label{thm::quantum_cone_bm_rough_statement}
Let $\CC = (\C,h,0,\infty)$ be a $\gamma$-quantum cone (which corresponds to $W = 4-\gamma^2$ and $\theta = 2\pi$) together with a space-filling $\SLE_{\kappa'}$ process~$\eta'$ from~$\infty$ to~$\infty$ sampled independently of~$\CC$ and then reparameterized according to~$\gamma$-LQG area.  That is, for $s,t \in \R$ with $s < t$ we have that $\mu_h(\eta'([s,t])) = t-s$.  Let $L_t$ (resp.\ $R_t$) denote the change in the length of the left (resp.\ right) boundary of $\eta'$ relative to time $0$.  Then $(L_t,R_t)_{t \in \R}$ is a correlated two-sided two-dimensional Brownian motion with $L_0 = R_0 = 0$.  If $\kappa' \in (4,8]$, we have (up to a non-random linear reparameterization of time) that
\begin{equation} \label{eqn::crtcovariance} \var(L_t) = |t|, \quad \var(R_t) = |t|,\quad\text{and}\quad \cov(L_t,R_t) = -\cos\left( \frac{4\pi}{\kappa'}\right) |t| \quad\text{for}\quad t \in \R.\end{equation}
Moreover, the joint law of $(h,\eta')$ as a path-decorated quantum surface is invariant under shifting by $t$ units of ($\gamma$-LQG area) time and then recentering.  That is, for each $t \in \R$ we have (as path-decorated quantum surfaces) that
\begin{equation}
\label{eqn::qc_invariance}
 (h,\eta') \stackrel{d}{=} (h(\cdot+\eta'(t)),\eta'(\cdot+t) - \eta'(t)).
\end{equation}
Finally, the quantum surfaces parameterized by $\eta'([0,\infty])$ and $\eta'([-\infty,0])$ are independent quantum wedges, each with parameter $\theta = \pi$, and $W = 2-\tfrac{\gamma^2}{2}$.
\end{theorem}

\begin{remark}In fact, the covariance formula in \eqref{eqn::crtcovariance} also holds for all $\kappa' > 4$, not only for $\kappa' \in (4,8]$. The extension of the formula to the case $\kappa'>8$ will not be presented in this paper, but it has now been established in a follow up paper \cite{ghms2015covariance}.
\end{remark}

\begin{figure}[ht!]
\begin{center}
\subfloat[]{\includegraphics[scale=0.85, page=1]{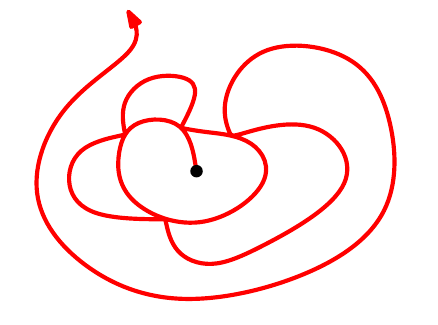}}
\subfloat[]{\includegraphics[scale=0.85, page=2]{figures/cone_skinny_wedge_bm}}
\end{center}
\caption{\label{fig::quantum_cone_bm} Illustration of the key step in the proof of Theorem~\ref{thm::quantum_cone_bm_rough_statement} from Theorem~\ref{thm::welding} and Theorem~\ref{thm::zip_up_wedge_rough_statement}.  {\bf a)} A $\gamma$-quantum cone sliced by an independent whole-plane $\SLE_\kappa(\rho)$ process $\eta_1$ with $\kappa=\gamma^2$ and $\rho=2-\kappa$.  (Whether or not $\eta_1$ is self-intersecting depends on whether $\rho \in  (-2,\tfrac{\kappa}{2}-2)$ or $\rho \geq \tfrac{\kappa}{2}-2$.)  By Theorem~\ref{thm::zip_up_wedge_rough_statement}, the sequence of quantum surfaces corresponding to $\C \setminus \eta_1$ ordered according to when their boundary is first drawn by $\eta_1$ is a wedge of weight $4-\gamma^2$.  {\bf b)}  Conditional on $\eta_1$, we draw in each of the components of $\C \setminus \eta_1$ an independent $\SLE_\kappa(-\tfrac{\kappa}{2};-\tfrac{\kappa}{2})$ process; call their concatenation $\eta_2$.  By the results of \cite{ms2013imag4}, we can view $(\eta_1,\eta_2)$ as flow lines of a common whole-plane GFF with an angle gap of $\pi$ and $(\eta_1,\eta_2)$ give the outer boundary of a space-filling $\SLE_{\kappa'}$, $\kappa'=16/\kappa$, process $\eta'$ stopped upon hitting $0$.  Theorem~\ref{thm::welding} implies that the pair $(\eta_1,\eta_2)$ divides the plane into independent wedges of weight $2-\tfrac{\gamma^2}{2}$.  (These correspond to the regions of $\C$ visited by $\eta'$ before and after it visits $0$ and are respectively colored green and white.)  This is the key observation that leads to the statement that the $\gamma$-LQG length of the left and right boundaries of $\eta'$ (when parameterized by $\gamma$-LQG area) has independent increments and, ultimately, Theorem~\ref{thm::quantum_cone_bm_rough_statement}, which states that they evolve as a certain two-dimensional Brownian motion.}
\end{figure}

\begin{figure}[ht!]
\begin{center}
\includegraphics[scale=0.85]{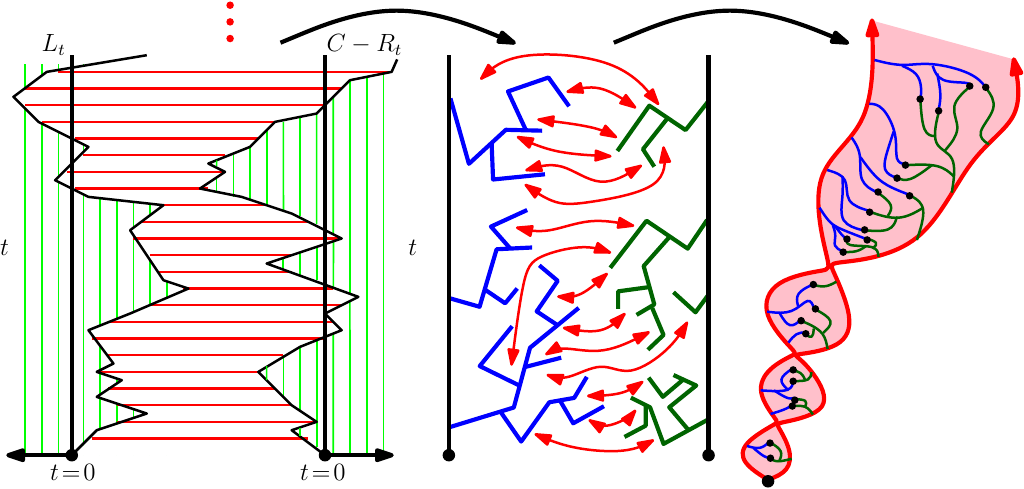}
\end{center}
\caption{\label{fig::fig2}  If we restrict the time in Figure~\ref{fig::fig1} to $t \geq 0$, then the same Brownian motion encodes a $\theta = \pi$ wedge, which corresponds to $W = 2 - \tfrac{\gamma^2}{2}$.  (Note that the vertical and horizontal axes have been swapped from what they were in Figure~\ref{fig::fig1}.  The value of $C$ is chosen so that the two graphs are disjoint up to a fixed, finite time.  To visualize the equivalence relation for all times, one could also make both graphs disjoint by replacing $L_t,R_t$ with $\exp(L_t),\exp(R_t)$.)  The left and right boundaries of the wedge correspond to the record minima of $L_t$ and $R_t$ (see also Figure~\ref{fig::cone_time}, which in turn correspond to the vertical green segments that reach all the way to the bottom $t=0$ line.}
\end{figure}

As we will explain in detail in Section~\ref{sec::brownian_boundary_length}, the main inputs into the proof of Theorem~\ref{thm::quantum_cone_bm_rough_statement} are Theorem~\ref{thm::welding} and Theorem~\ref{thm::zip_up_wedge_rough_statement}.  (See also Figure~\ref{fig::quantum_cone_bm}.)  Indeed, these results imply that drawing a certain pair of whole-plane $\SLE_\kappa(2-\kappa)$ processes coupled together as flow lines of a whole-plane GFF \cite{ms2013imag4} on top of an independent $\gamma$-quantum cone ($W=4-\gamma^2$) yields a pair of independent quantum wedges of weight $2-\tfrac{\gamma^2}{2}$.  These flow lines give the left and right boundaries of~$\eta'$ stopped upon hitting~$0$.  This, combined with the invariance statement~\eqref{eqn::qc_invariance} implies that $(L,R)$ has independent increments and it does not require much additional work to extract from this that $(L,R)$ must be \emph{some} two-dimensional Brownian motion.  In the case that $\kappa' \in (4,8)$, we then compute the a.s.\ Hausdorff dimension of the set of times $t$ for $\eta'$ which are local cut times.  The time parameterization that we take here for $\eta'$ is the $\gamma$-LQG area parameterization (which corresponds to the standard time parameterization so that $(L,R)$ evolves as a Brownian motion).  These local cut times turn out to correspond to so-called ``cone times'' for $(L,R)$, so we are able to determine the covariance matrix for $\kappa' \in (4,8)$ (and for the limiting case $\kappa' = 8$) by matching the dimension that we find with the dimension given in the main result of \cite{EVANS_CONE_TIMES}.  In Section~\ref{sec::brownian_boundary_length}, additional explanation is provided as to how this result relates to the scaling limits of discrete random planar map models \cite{sheffield2011qg_inventory}.

Our next result is that the pair $(L,R)$ from Theorem~\ref{thm::quantum_cone_bm_rough_statement} a.s.\ determines both the space-filling $\SLE_{\kappa'}$ exploration path and the entire LQG surface.

\begin{theorem}
\label{thm::trees_determine_embedding}
In the setting of Theorem~\ref{thm::quantum_cone_bm_rough_statement}, the pair $(L,R)$ a.s.\ determines both $\eta'$ and $h$ (up to a rigid rotation of the complex plane about the origin).
\end{theorem}

\begin{remark}
Theorem~\ref{thm::trees_determine_embedding} (along with  Theorem~\ref{thm::quantum_cone_bm_rough_statement}) is one of the most important results of the paper. The reader may find it reminiscent of the construction of level sets of the GFF \cite{SchrammSheffieldGFF2}. In that setting, one has a coupling of an $\SLE_4$ curve $\wt \eta$ with an instance $\wt h$ of the GFF (with appropriate boundary conditions) and aims to show that $\wt h$ a.s.\ determines $\wt \eta$. This is done by first sampling $\wt h$ and then, given this, sampling two conditionally independent copies of $\wt \eta$ (traversed in opposite directions) and using basic properties of local sets to conclude that these two copies a.s.\ agree. (Similar arguments involving paths traversed in opposite directions, and a range of $\kappa$ values, appear in works of Dub\'edat, for example \cite{dub_dual,dub2009part}.)  When we prove Theorem~\ref{thm::trees_determine_embedding} in Section~\ref{sec::trees_determine_embedding}, we will begin with a coupling between $(L,R)$ and $(h,\eta')$ and aim to show that given the pair $(L,R)$, the conditional law of $(h,\eta')$ is a.s.\ deterministic. But the flavor of the argument is rather different from what appears in the GFF level set story. Roughly speaking, we will begin by conditioning on only {\em part} of the information in $(L,R)$ (namely, the values of $L$ and $R$ at times that are multiples of $1/n$) and then use Efron-Stein based variance bounds, along with various results about what happens when one resamples pieces of the decorated surface described by $(h,\eta')$, to show that in the $n\to \infty$ limit, the conditional law of $(h,\eta')$ becomes deterministic. We remark that if space-filling SLE is drawn on a Euclidean domain (instead of a quantum surface) one can define an Euclidean analog of the pair $(L,R)$, and the fact that one can recover the path from the pair $(L,R)$ in this setting has been recently established in a follow up work by Holden and Sun~\cite{euclideanmatingoftrees}.
\end{remark}

\subsubsection{Non-space-filling counterflow lines}
\label{subsubsec::intro_counterflow}

We will now discuss some further consequences and results related to Theorem~\ref{thm::quantum_cone_bm_rough_statement} in the case that $L_t$ and $R_t$ are positively correlated, which corresponds to $\gamma \in (\sqrt{2}, 2)$.  Recall that this is the range for which the corresponding $\CLE_{\kappa'}$ exist with $\kappa' \in (4,8)$, where $\kappa' = 16/\gamma^2$.  This is also the range in which $\SLE_{\kappa'}$ is itself non-space-filling and hence differs from space-filling $\SLE_{\kappa'}$.

Given $0 < s < t$, we say that $s$ is an {\bf ancestor} of $t$, and we write $s \prec t$, if for all $r \in (s,t]$ we have $L_r >   L_s$ and $R_r > R_s$.  The following facts are obvious from this definition:
\begin{enumerate}
\item $s \prec t$ implies $s < t$.
\item $s \prec t$ and $t \prec u$ implies $s \prec u$.
\item $s \prec t$ implies $s \prec u$ for all $u \in (s,t)$.
\end{enumerate}
If $s$ is an ancestor of some $t > s$, then $s$ is called a {\bf cone time} of the Brownian process.  (If $\gamma \in (0,\sqrt{2}]$ so that $L,R$ are non-positively correlated then $(L,R)$ a.s.\ does not have cone times \cite{EVANS_CONE_TIMES}.)  Figure~\ref{fig::cone_time} illustrates one such cone time.  As the figure illustrates, the set of points that have a cone time $s$ as an ancestor is an open set $(s,s')$ for some~$s'$, and between~$s$ and~$s'$ the Brownian path $(L_t, R_t)$ traces out an excursion into the quadrant $[L_s, \infty) \times [R_s, \infty)$ that begins at the corner and ends on one of the two sides.  A point $t \geq 0$ is called {\bf ancestor free} if there is no $s \in (0,t)$ that is an ancestor of $t$.  The properties above imply that if $s$ is ancestor free then a given $t > s$ is ancestor free if and only if $t$ has no ancestor in $(s,t)$.  This implies that the set of ancestor free times is a regenerative process, and by scale invariance, we may conclude that it has the law of the range of a stable subordinator (which agrees in law with the zero set of a certain Bessel process, and which can be parameterized by a local time). See, e.g., the reference texts \cite{bertoin96levy,bertoin1999subordinators} for a general overview of stable subordinators and Chapter XI of \cite{ry99martingales} for the connection with Bessel processes and the construction of the local time measure for the Bessel process, which then corresponds to a natural measure on the set of ancestor free times.  We recall that an $\alpha$-{\bf stable process} is a L\'evy process $X_t$ such that $t^{-1/\alpha} X_t \stackrel{d}{=} X_1$ for all $t>0$, and a {\bf stable subordinator} is a non-decreasing $\alpha$-stable process.  We also recall that an $\alpha$-stable L\'evy process is called {\bf totally asymmetric} if all of its jumps have the same sign.  In this article, the sign of the jumps will be clear from the context.

Write $t(s)$ for the infimum of times $u \geq 0$ at which the local time measure of the set of ancestor free times in $[0,u]$ exceeds $s$ (noting that $t(s)$ is necessarily an ancestor free time itself).  Then we have the following:

\begin{proposition}
\label{prop::iidlevy}
The processes $L_{t(s)}$ and $R_{t(s)}$ parameterized by time $s$ are independent totally asymmetric $\tfrac{\kappa'}{4}$-stable processes.
\end{proposition}

The statement itself is a straightforward observation (it is clear from Figure~\ref{fig::cone_time} that each jump of the stable subordinator corresponds to a positive jump in precisely one of the processes $L_{t(s)}$ and $R_{t(s)}$, and that the measure on jumps has a power law distribution) but the parameter $\tfrac{\kappa'}{4}$ will not be identified until Section~\ref{sec::duality}.  The processes $L_{t(s)}$ and $R_{t(s)}$ are independent, even though $L_t$ and $R_t$ are not, because they are both L\'evy processes without continuous part which a.s.\ do not have a jump at the same time.  Next, it is natural to ask the following: if we know that $s>0$ is an ancestor of $t>0$ what does that tell us about the space-filling curve $\eta'$? As Figure~\ref{fig::cone_time} illustrates, if $s>0$ is an ancestor of $t>0$ then the set $\eta'([s,\infty))$ must separate $\eta'(t)$ from $\eta'(0)$.  In other words, if we parameterize time in reverse (from $t = \infty$ to $t=0$) then $\eta'$ ``disconnects'' $\eta'(t)$ from $\eta'(0)$ before $\eta'(0)$ is reached.  Recall that this means that $\eta'(t)$ is {\em not} part of the corresponding counterflow line.

Conversely (see also~\ref{fig::ancestor_free}) a time $t>0$ is {\em ancestor free} if and only if $\eta'(t)$ lies on the boundary of the component of $\C \setminus \eta'([t,\infty))$ that contains $\eta'(0)$.  While the time-reversal of $\eta'$ is a space-filling SLE$_\kappa'$ curve, the restriction of that time reversal to the set of ancestor free times $t>0$ turns out to correspond to the continuous but non-space-filling curve that we call the {\em counterflow line} from $\infty$ to $0$. The following list summarizes a few special subsets of the wedge parameterized by $[0,\infty)$, which are obtained by restricting to special times in $[0,\infty)$:

\begin{enumerate}
\item {\bf Counterflow line from $\infty$ to $0$:} parameterized by the set of ancestor free times $t$.
\item {\bf Left (resp.\ right) boundary of entire wedge:} parameterized by times at which $L_t$ (resp.\ $R_t$) attains a record minimum.
\item {\bf Cut points of entire wedge:} times at which $L_t$ and $R_t$ simultaneously achieve record minima.  
See Figure~\ref{fig::ancestor_free}.
\end{enumerate}

The processes described in Proposition~\ref{prop::iidlevy} encode two so-called {\bf L\'evy trees} of disks, after the manner outlined in Figure~\ref{levytreegluing}.  (L\'evy trees are studied in detail in \cite{dlg2002trees_levy}.)  The sets obtained by removing the interior of each of these loops (so that one has a ``tree of circles'') are called {\bf stable looptree} in e.g.\ \cite{ck2013looptrees}. A stable looptree is a random space {\em that comes with a topology and some additional structure} (e.g., each loop comes with a defined boundary length measure, and the looptree is a geodesic metric space) --- we refer the reader to \cite{ck2013looptrees} for a more detailed treatment of stable looptrees.

The procedure for obtaining one of these trees is explained in the top row of Figure~\ref{fig::fig11b}.  (The later rows contain related constructions that will be relevant for Theorem~\ref{thm::quantum_natural_zip_unzip_rough_statement}.)  Note that in a L\'evy tree there are in fact a countably infinite number of small loops along the branch connecting any two given loops (i.e., it is a.s.\ the case that no two loops are adjacent).  Each loop comes with a well-defined boundary length, which is the magnitude of the corresponding jump in the stable L\'evy process.  The outer boundary of the tree of loops also comes with a natural time parameterization, which is the time of the corresponding L\'evy process.  Intuitively, if for some tiny $\epsilon$ one keeps track of the number of loops of size between $\epsilon$ and $2 \epsilon$ that are encountered as one traces the boundary of the tree, then that number (times an appropriate power of $\epsilon$) is a good approximation for this natural time.

Suppose that $\gamma \in (\sqrt{2},2)$ so that $\kappa \in (2,4)$ and $\kappa' \in (4,8)$.  A {\bf forested line} $\CL$ is the beaded quantum surface constructed in the following way. First begin with an $\alpha=\kappa'/4$-stable L\'evy process $X$ with only upward jumps as on the top row of Figure~\ref{fig::fig11b} and create the corresponding stable looptree $\wt{\CL}$ (which we think of as a forest of looptrees attached to a single infinite ray, which we will sometimes refer to as the ``line of $\wt{\CL}$ or $\CL$'').  Each loop has a boundary length $L$ (recall~\eqref{e.nudef}).  Then note that if we condition a quantum disk from $\diskmeasure$ (from the infinite quantum disk measure defined at the end of Section~\ref{subsec::surfaces}) to have boundary length $L$, we obtain a probability measure on quantum disks of boundary length $L$.  So given $\wt{\CL}$, we independently sample one quantum disk of boundary length $L$ corresponding to each loop (i.e., each circle illustrated on the top row of Figure~\ref{fig::fig11b}).  We when topologically identify the boundary of each quantum disk with the corresponding loop of $\wt{\CL}$ in a clockwise length-preserving way, with the rotation chosen uniformly at random (i.e., if we take any point on the boundary of the  $\diskmeasure$ sample, then the location along the length $L$ loop that it is identified with is chosen uniformly). The resulting object $\CL$ is a beaded quantum surface than can be understood as a L\'evy tree of quantum disks. We can extend this definition as follows.

\begin{definition}
\label{def::forested_wedge}	
Suppose that $\CW$ is a quantum wedge of weight $W > 0$ and $\CL$ is a forested line.  A {\bf forested quantum wedge of weight $W$} is the beaded random surface which arises by gluing the line of $\CL$ (with $\alpha = \kappa'/4 \in (1,2)$)  to either the left or the right side of $\CW$.  A {\bf doubly forested quantum wedge of weight $W$} is obtained by letting $\CL_1, \CL_2$ be independent forested lines and then gluing the line of $\CL_1$ (resp.\ $\CL_2$) to the left (resp.\ right) side of $\CW$.  Illustrations of doubly forested wedges appear, e.g., in Figures~\ref{fig::fig8} and~\ref{fig::fig10}.
\end{definition}

We emphasize that the only parameter in the definition of a forested quantum wedge is the weight $W$ because the value $\alpha$ for the stable L\'evy process used to define the forested line is determined by~$\gamma$ and always given by $\alpha = \tfrac{\kappa'}{4} = \tfrac{4}{\gamma^2}$.  The same is likewise true in the case of a doubly forested quantum wedge.

\begin{theorem}
\label{thm::gluingtwoforestedlines}
Consider a quantum wedge $\CW$ of weight $W = 2-\tfrac{\gamma^2}{2}$ (which corresponds to $\theta = \pi$) and a concatenation $\eta'$ of independent $\SLE_{\kappa'}(\kappa'/2-4;\kappa'/2-4)$ processes, one for each bead of $\CW$, from $\infty$ to $0$ (which corresponds to a counterflow line, as depicted in Figure~\ref{fig::ancestor_free} and Figure~\ref{fig::fig6}).  Then $\eta'$ divides $\CW$ into two independent forested lines, whose boundaries are identified with one another according to the natural time parameterization of the outer boundary of the corresponding L\'evy trees.  Moreover, given the two forested lines, it is a.s.\ possible to uniquely recover the quantum wedge and $\eta'$.
\end{theorem}

As explained above, in the statement of Theorem~\ref{thm::gluingtwoforestedlines} we can view $\eta'$ as arising from the set of ancestor free times associated with an independent space-filling $\SLE$ curve drawn on top of a quantum cone.

\begin{figure}[ht!]
\begin{center}
\includegraphics[scale=0.85]{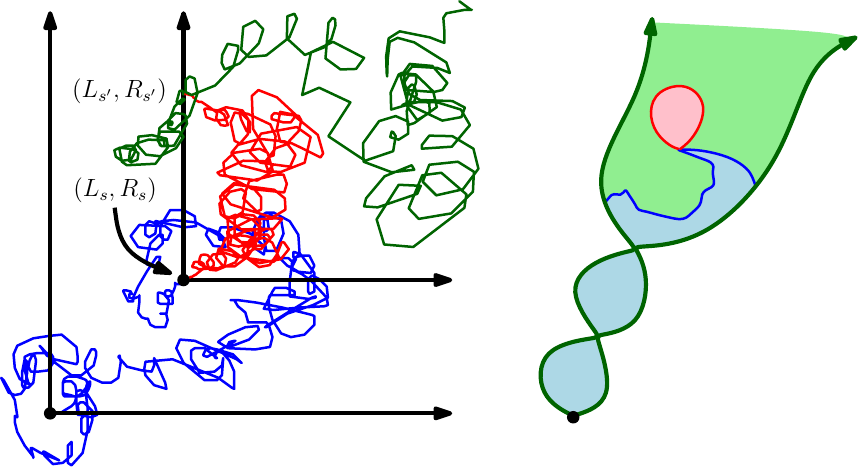}
\end{center}
\caption{\label{fig::cone_time}  When $L_t$ and $R_t$ are positively correlated, we have $\gamma \in (\sqrt{2}, 2)$.  In this case, there a.s.\ exist cone times like the time $s$ illustrated on the left.  As noted on the right, an interval of the type $[s,s']$ (here $s' \in \{t > s : L_t=L_s \, \mathrm{or} \, R_t=R_s \}$) is ``cut off'' by the time-reversal of the space-filling path from $0$ to $\infty$ before it is filled up. (The time-reversal fills first the green region, then the red region, then the blue region.)  At the critical value $\gamma=\sqrt{2}$, $L_t$ and $R_t$ are independent and there a.s.\ do not exist cone times.}
\end{figure}

\begin{figure}[ht!]
\begin{center}
\includegraphics[scale=0.85]{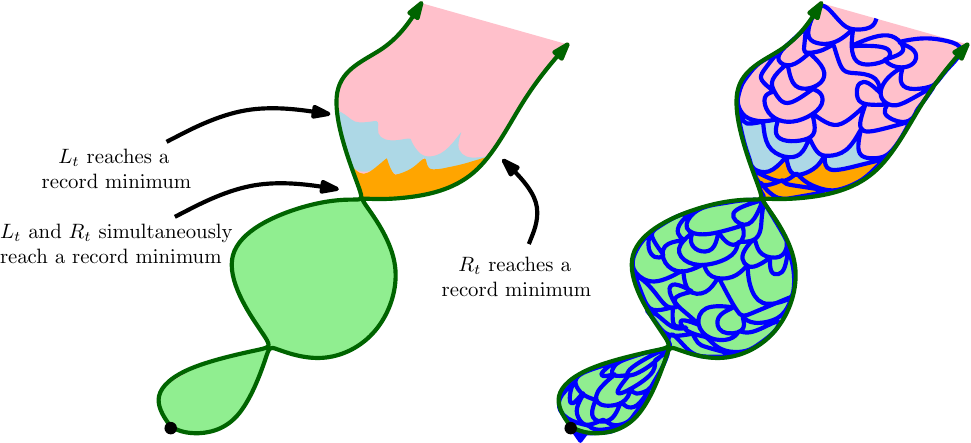}
\end{center}
\caption{\label{fig::ancestor_free}  Shown on the left is a $2-\tfrac{\gamma^2}{2}$ wedge.  When the space-filling $\SLE_{\kappa'}$ process $\eta'$ from $0$ to $\infty$ on this wedge hits a pinch point in the wedge, then $L_t$ and $R_t$ simultaneously hit a record minimum (green).  When $\eta'$ hits the left (resp.\ right) side of the wedge then $L_t$ (resp.\ $R_t$) hits a record minimum.  The points on the blue curve on the right make up the set of ancestor free times which corresponds to the counterflow line from $\infty$ to~$0$.}
\end{figure}

\begin{figure}[ht!]
\begin{center}
\includegraphics[scale=1]{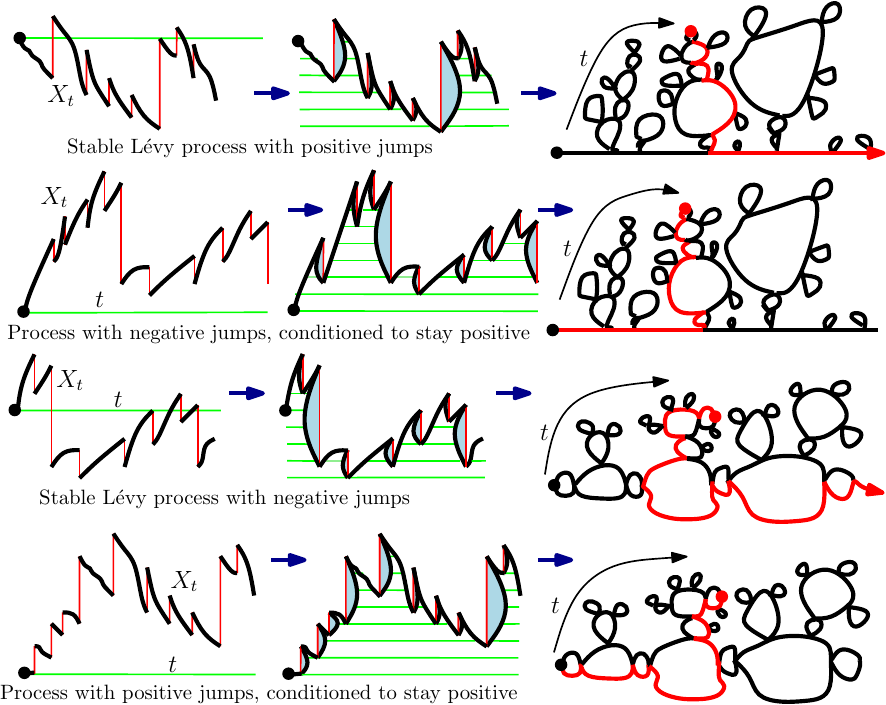}
\end{center}
\caption{\label{fig::fig11b} {\bf First row:} A stable L\'evy process $X_t$ with positive jumps encodes a forested line via the procedure explained in Figure~\ref{levytreegluing} (top right).  Unmatched green segments (record minima of $X_t$) map to points on the line.  As $t$ increases, the red dot on right traces the forest boundary clockwise.  $X_t$ encodes the net change in length of the red path (which traces the right side of the branch of disks containing the point hit at time $t$, and continues right to $\infty$) since time $0$.  A jump occurs when a disk is hit for the first time, with jump size given by the boundary length of the disk.  {\bf Second row:} The same forested line corresponds to a stable L\'evy process with negative jumps conditioned to stay positive.   Unmatched green rays (last hitting times of $X_t$) map to points on the line.  The value of $X_t$ encodes the left boundary length.  A jump occurs when a disk is hit for the last time.  {\bf Third row:} A stable L\'evy process with negative jumps encodes a forested line.   Unmatched green segments map to the underside of the line.  A jump occurs when a disk is hit for the last time.  {\bf Fourth row:} The same forested line corresponds to a positive-jump stable L\'evy process conditioned to stay positive.}
\end{figure}

\begin{figure}[ht!]
\begin{center}
\includegraphics[scale=0.85]{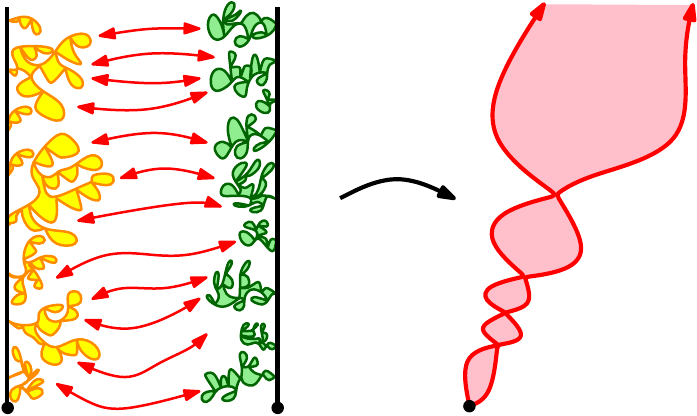}
\end{center}
\caption{\label{fig::fig6}
(Illustration of Theorem~\ref{thm::gluingtwoforestedlines}.) Two forested lines can be welded together according to quantum length to produce a single $\theta = \pi$ quantum wedge.  Recall that $\theta = \pi$ corresponds to $W = 2 - \tfrac{\gamma^2}{2}$.}
\end{figure}

\begin{figure}[ht!]
\begin{center}
\includegraphics[scale=0.85]{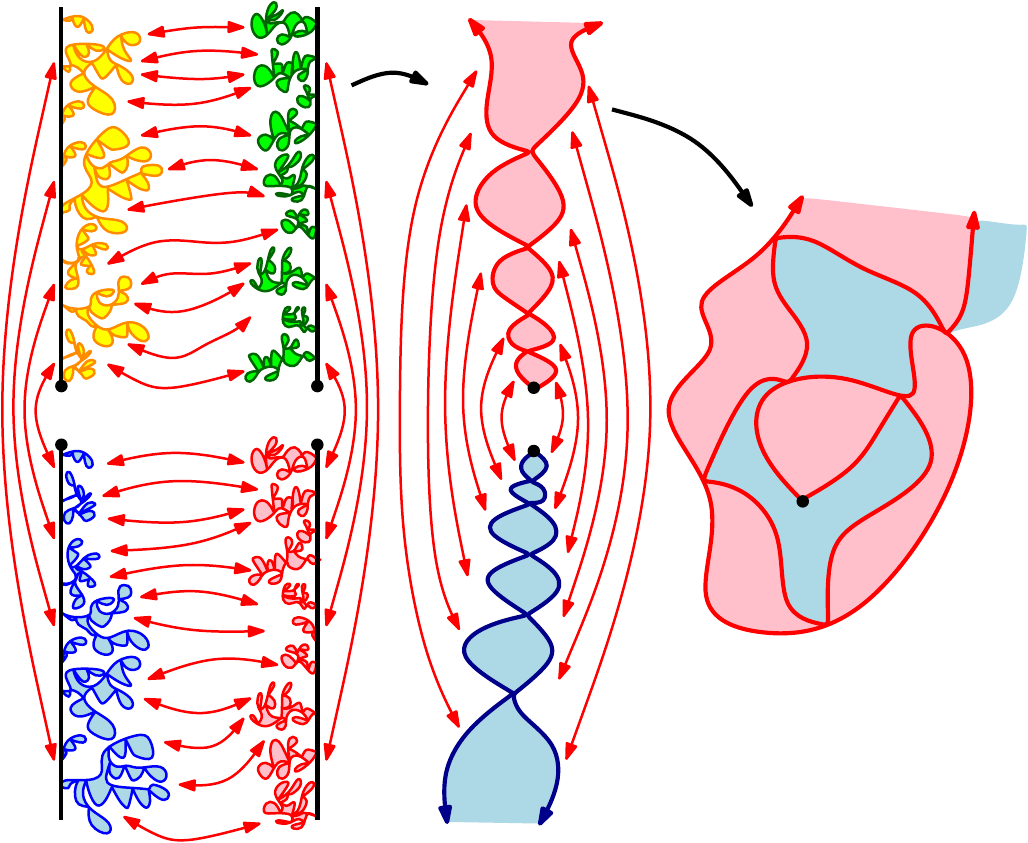}
\end{center}
\caption{\label{fig::fig7}  (Illustration of the combination of Theorem~\ref{thm::zip_up_wedge_rough_statement} and Theorem~\ref{thm::gluingtwoforestedlines}.) Repeating the procedure of Figure~\ref{fig::fig6} for negative time, we can glue together four i.i.d.\ forested wedges to obtain the entire $\theta = 2\pi$ quantum cone.  The interfaces are given by the flow line and dual flow line from $0$ to $\infty$, and by the two counterflow lines coming from $\infty$ to $0$ that have these paths as their boundaries.}
\end{figure}

\begin{figure}[ht!]
\begin{center}
\includegraphics[scale=0.85]{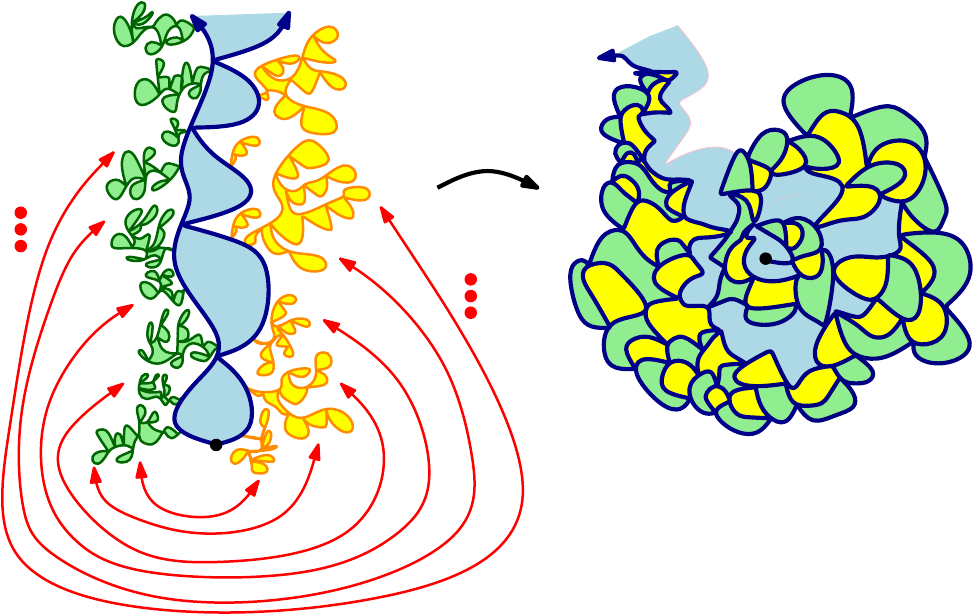}
\end{center}
\caption{\label{fig::fig8}  (Illustration of Theorem~\ref{thm::quantum_cone_sle_kp}.)  Alternatively, one can subdivide only one of the two $\theta = \pi$ wedges within a counterflow line.  In this case, one has a doubly forested $\theta = \pi$ quantum wedge that zips up to become a $\theta = 2\pi$ quantum cone.}
\end{figure}

\begin{figure}[ht!]
\begin{center}
\includegraphics[scale=0.85]{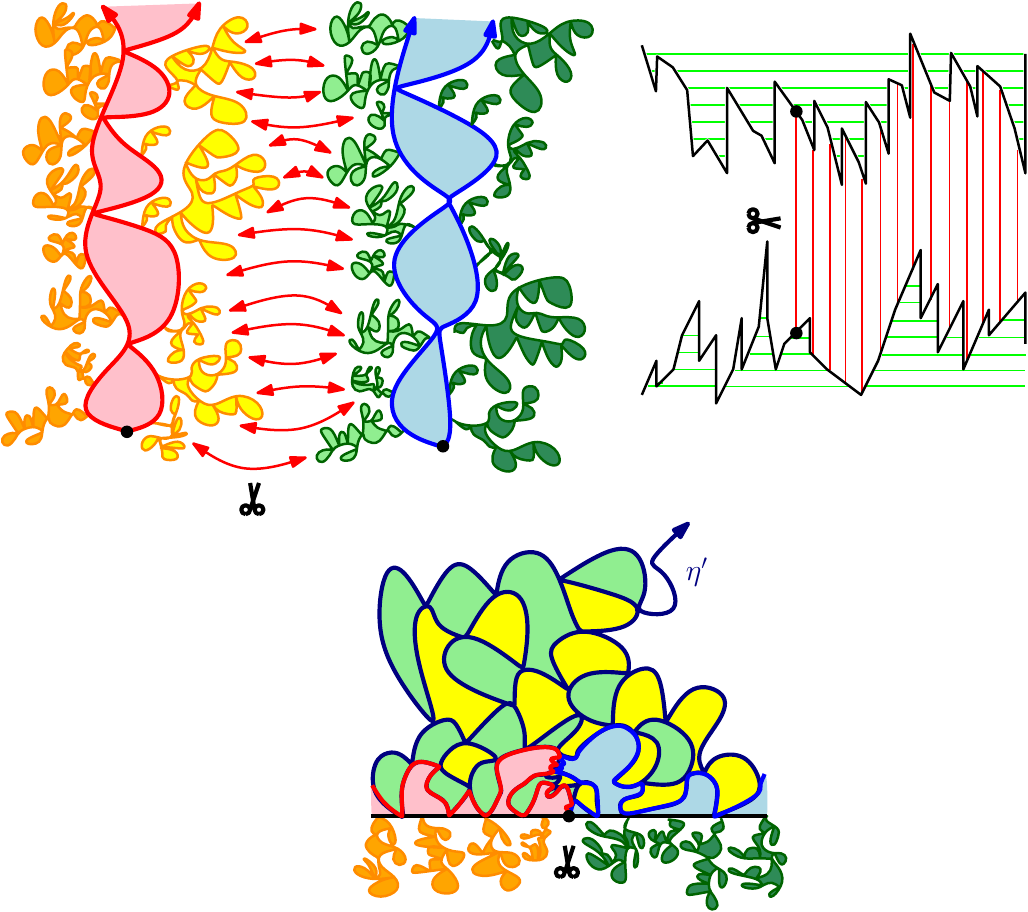}
\end{center}

\caption{\label{fig::fig10}  The quantum length invariant quantum zipper from \cite{she2010zipper} involved two weight $2$ quantum wedges which were zipped together on one side, and could be further zipped or unzipped; in the embedding in the upper half plane, the interface is an $\SLE_{\kappa}$ curve for $\kappa<4$.  In the analog described here for $\kappa' \in (4,8)$ the two weight $2$ quantum wedges are replaced by two doubly forested weight $(\gamma^2-2)$ wedges.  The interface between them is a counterflow line (i.e., $\SLE_{\kappa'}$ curve), and once again the figure is invariant w.r.t.\ zipping or unzipping by a fixed amount of quantum natural time.  Note that the entire collection of loops can be constructed from a pair of stable L\'evy processes with negative jumps indexed by all time $\R$.  The jumps before zero encode the loops in the two forests that have not yet been zipped up.  The jumps after zero encode the loops in the upper half plane.  Those jumps at times $t > 0$ at which $L_t$ (resp.\ $R_t$) achieves a record minimum correspond to the disks that share boundary segments with the left (resp.\ right) real axis.
}
\end{figure}

The following theorem is another fairly easy generalization of Theorem~\ref{thm::gluingtwoforestedlines}, which can be deduced as a consequence of Theorem~\ref{thm::gluingtwoforestedlines} and the additivity of wedges developed in Theorems~\ref{thm::welding} and~\ref{thm::paths_determined}.  It states essentially that one can zip together a right-forested wedge of weight $W_1$ and a left-forested wedge of weight $W_2$ (zipping along the forested sides) in order to obtain a wedge of weight $W_1 + W_2 + 2 - \tfrac{\gamma^2}{2}$ decorated by an appropriate independent $\SLE_{\kappa'}(\rho_1;\rho_2)$ process.  We include this result in order to illustrate that general non-boundary-filling $\SLE_{\kappa'}(\rho_1;\rho_2)$ processes can be obtained by zipping together forested wedges.\footnote{It is also natural to give a description of the structure of the surfaces cut out by an $\SLE_{\kappa'}(\rho_1;\rho_2)$ process, $\kappa' \in (4,8)$, drawn on top of a wedge $\CW = (\h,h,0,\infty)$ of weight $W$ as in~\eqref{eqn::sle_kp_on_wedge_weights} when one or both of $\rho_1,\rho_2$ are in $(-2,\tfrac{\kappa'}{2}-4)$.  (As Theorem~\ref{thm::sle_kp_on_wedge} requires that $\rho_1,\rho_2 \geq \tfrac{\kappa'}{2}-4$, when one or both of $\rho_1,\rho_2$ are in $(-2,\tfrac{\kappa'}{2}-4)$ we are outside of the regime in which Theorem~\ref{thm::sle_kp_on_wedge} applies.)  Recall that if $\rho_1$ (resp.\ $\rho_2$) is in this range then $\eta'$ a.s.\ fills the $\R_-$ (resp.\ $\R_+$).  In this case, the quantum surfaces which are completely surrounded by $\eta'$ and are on its left (resp.\ right) side still have a tree structure.  It is not clear, however, whether these trees are independent of each other.  Describing the law of this pair of trees falls outside of the scope of this article.}

\begin{theorem}
\label{thm::sle_kp_on_wedge}
Fix $\gamma \in (0,2)$ and let $\kappa' = 16/\gamma^2 \in (4,\infty)$.  Fix $\rho_1,\rho_2 \geq \tfrac{\kappa'}{2}-4$.  Let 
\begin{equation}
\label{eqn::sle_kp_on_wedge_weights}
W_i = \gamma^2-2 + \frac{\gamma^2}{4} \rho_i \geq 0 \quad\text{for}\quad i=1,2 \quad\text{and}\quad W = W_1 + W_2 + 2-\tfrac{\gamma^2}{2}.
\end{equation}
In the case that $W \geq \tfrac{\gamma^2}{2}$, let $\CW = (\h,h,0,\infty)$ be a quantum wedge of weight $W$ and let $\eta'$ be an $\SLE_{\kappa'}(\rho_1;\rho_2)$ process in $\h$ from $0$ to $\infty$ with force points located at $0^-,0^+$ which is independent of~$\CW$.  In the case that $W \in (0,\tfrac{\gamma^2}{2})$ so that $\CW$ is not homeomorphic to $\h$, we take $\eta'$ be to be given by a concatenation of independent $\SLE_{\kappa'}(\rho_1;\rho_2)$ processes: one for each of the bubbles of $\CW$ starting at the opening point of the bubble and targeted at the terminal point.  Then the quantum surface $\CW_1$ (resp.\ $\CW_2$) which consists of those components of $\h \setminus \eta'$ which are to the left (resp.\ right) of $\eta'$ is a quantum wedge of weight $W_1$ (resp.\ $W_2$) and the quantum surface $\CW_3$ which is between the left and right boundaries of $\eta'$ is a quantum wedge of weight $2-\tfrac{\gamma^2}{2}$.  (Note that $\CW_1$ is beaded if $\rho_1 \in (\tfrac{\kappa'}{2}-4,\tfrac{\kappa'}{2}-2)$ and likewise $\CW_3$ is beaded if $\rho_2 \in (\tfrac{\kappa'}{2}-4,\tfrac{\kappa'}{2}-2)$.)  Moreover, $\CW_1,\CW_2,\CW_3$ are independent.

Suppose further that $\gamma \in (\sqrt{2},2)$ so that $\kappa' \in (4,8)$.  Then the beaded surface $\wt{\CW}_1$ (resp.\ $\wt{\CW}_2)$ which consists of those components of $\h \setminus \eta'$ whose boundary is drawn by the left (resp.\ right) side of $\eta'$ is a forested wedge of weight $W_1$ (resp.\ $W_2$).  Moreover, $\wt{\CW}_1$ and $\wt{\CW}_2$ are independent and, together, a.s.\ determine both $\CW$ and $\eta'$.
\end{theorem}

Theorem~\ref{thm::quantum_cone_sle_kp} below is another easy consequence of Theorem~\ref{thm::gluingtwoforestedlines}, together with Theorem~\ref{thm::zip_up_wedge_rough_statement}.  It states that we can zip up \emph{both} sides of a doubly forested wedge of general weight, as illustrated in Figure~\ref{fig::fig8}, in order to obtain a quantum cone decorated by an appropriate whole-plane $\SLE_{\kappa'}(\rho)$ process.

But before presenting Theorem~\ref{thm::quantum_cone_sle_kp}, let us say a few words about the geometric picture it describes.  The {\em complement} of the trace of whole plane $\SLE_{\kappa'}(\rho)$ (for $\kappa' \in (4,8)$ and $\rho$ in the range described in the theorem statement) is a random countable collection of bounded open sets.  And these sets are of three different types, as illustrated by the three colors in Figure~\ref{fig::fig8}. One way to describe the distinction is to imagine lifting $\eta'$ to the universal cover of $\C \setminus \{0\}$.  Recall that the universal cover of $\C \setminus \{0\}$ can be parameterized by $\C$ with the covering map given by $z \mapsto e^z$.  A lifting $\wt{\eta}'$ of $\eta'$ is a continuous curve so that $\eta'(t) = \exp(\wt{\eta}'(t))$.  Roughly speaking, whenever $\eta'$ winds around $0$, $\wt{\eta}'$ moves vertically by $2\pi$.  Then the bounded components of the complement of $\wt{\eta}'$ within the universal cover will have only two types --- those whose boundary is on the left side of $\wt{\eta}'$ and those whose boundary is on the right side; projecting back down to $\C \setminus \{0\}$, we can view these sets as subsets of $\C$, and there they correspond to the green and yellow regions indicated in Figure~\ref{fig::fig8}.  The other regions (the blue regions in Figure~\ref{fig::fig8}) together form a quantum wedge (which could be a thin or a thick quantum wedge, depending on $\rho$) and each such region has one boundary arc corresponding to the ``left side'' of $\eta'$ and another complementary arc corresponding to the ``right side.''

\begin{theorem}
\label{thm::quantum_cone_sle_kp}
Fix $\gamma \in (\sqrt{2},2)$ and suppose that $\CC = (\C,h,0,\infty)$ is a quantum cone of weight $W \geq 2-\tfrac{\gamma^2}{2}$ (so that $\theta \geq \pi$).  Let $\rho = \tfrac{4W}{\gamma^2} - 2$ (so that $\rho \geq \frac{4(2-\gamma^2/2)}{\gamma^2} -2 = 8/\gamma^2-4)$ and suppose that $\eta'$ is a whole-plane $\SLE_{\kappa'}(\rho)$ process starting from $0$ independent of $\CC$.  Then the beaded surface $\CW_1$ (resp.\ $\CW_2$) which consists of those components of $\C \setminus \eta'$ which are surrounded by $\eta'$ on its left (resp.\ right) side when viewed as a path in the universal cover of $\C \setminus \{0\}$ has the structure of a forested line and the (possibly beaded) surface~$\CW_3$ which consists of the remaining components of $\C \setminus \eta'$ is a quantum wedge of weight 
\[ W - \left(2 - \frac{\gamma^2}{2}\right) = \gamma^2-2 + \frac{\gamma^2}{4}\rho.\]
Moreover, $\CW_1$, $\CW_2$, and $\CW_3$ are independent and together a.s.\ determine both $\CC$ and $\eta'$.
\end{theorem}

Observe that if we express the $\rho$ in the statement of Theorem~\ref{thm::quantum_cone_sle_kp} in terms of $\kappa'$ and $\alpha$ then we get
\begin{equation}
\label{eqn::cone_rho_2}
\rho = 2 + \kappa' - 2\alpha \sqrt{\kappa'}.
\end{equation}
Note that~\eqref{eqn::cone_rho_2} depends on $\kappa'$ and $\alpha$ in the same way that~\eqref{eqn::cone_rho} (with $\gamma^2=\kappa$) from Theorem~\ref{thm::zip_up_wedge_rough_statement} depends on $\kappa$ and $\alpha$.

Theorem~\ref{thm::quantum_natural_zip_unzip_rough_statement} is one of the more interesting and important consequences of Theorem~\ref{thm::gluingtwoforestedlines}.  As illustrated in Figure~\ref{fig::fig10}, it implies a ``quantum zipper'' invariance principle for $\SLE_{\kappa'}$ analogous to the principle established for $\SLE_{\kappa}$ in  \cite{she2010zipper}.

Combining Theorem~\ref{thm::gluingtwoforestedlines} with Theorem~\ref{thm::sle_kp_on_wedge} leads to another notion of time parameterization for an $\SLE_{\kappa'}$ process $\eta'$, namely the time parameterization associated with the pair of independent stable L\'evy processes which encode the boundary lengths of the bubbles cut off by $\eta'$.  We will refer to this time-parameterization as {\bf quantum natural time}.  It is the quantum version of the so-called ``natural'' parameterization for $\SLE$ \cite{ls2011natural_param,lz2013natural_param,lr2012minkowski,law2015minkowski,gps2013pivotal,lv2016lerw,b2017natural,hls2018natural}, see also \cite{DS3}.  (The quantum analog of the natural parameterization for $\kappa \in (0,4)$ is $\gamma$-LQG length and for $\kappa' \geq 8$ is $\gamma$-LQG area.)  We will write $\qnt_u$ for the (random) function which converts from quantum natural time $u$ to capacity time.  That is, if $\eta'$ is parameterized by capacity time then $\eta'(\qnt_u)$ is parameterized by quantum natural time.

\begin{theorem}
\label{thm::quantum_natural_zip_unzip_rough_statement}
Fix $\gamma \in (\sqrt{2},2)$ and let $\CW = (\h,h,0,\infty)$ be a quantum wedge of weight $\tfrac{3\gamma^2}{2}-2$ and let $\eta'$ be an independent $\SLE_{\kappa'}$ process, $\kappa' = 16/\gamma^2 \in (4,8)$, in~$\h$ from~$0$ to $\infty$.  Then $\eta'$ divides $\CW$ into two independent forested wedges both with weight $\gamma^2-2$.  Moreover, the joint law of $(h,\eta')$ is invariant under the operation of cutting along $\eta'$ until a given quantum natural time and then conformally mapping back and applying~\eqref{eqn::Qmap}.  That is, if $(f_t)$ denotes the centered chordal Loewner flow associated with $\eta'$ (with the capacity time parameterization) and $\qnt_u$ is as above, then we have (as path-decorated quantum surfaces) that
\[ (h,\eta') \stackrel{d}{=} ( h \circ f_{\qnt_u}^{-1} + Q\log| (f_{\qnt_u}^{-1})'|, f_{\qnt_u}(\eta')) \quad\text{for each}\quad u \geq 0.\]

Indeed, the entire image shown in Figure~\ref{fig::fig10} is invariant with respect to the operation of zipping/unzipping according to quantum natural time.
\end{theorem}

If we start with the top row of Figure~\ref{fig::fig11b}, and ``zoom in'' near a typical non-zero time, then we obtain a stable L\'evy process indexed by all of $\R$.  We can then consider a pair of processes of this type.  This is equivalent to drawing an SLE until a typical quantum time, zooming in, and ``unzipping'' up to the distinguished time, which produces a figure like Figure~\ref{fig::fig10}.  Each of the forested lines hanging off the bottom of the image in  Figure~\ref{fig::fig10} is an independent forested line of the sort described by the top row of  Figure~\ref{fig::fig11b}.  If we focus on the quantum wedge parameterized by the upper half plane in  Figure~\ref{fig::fig10} and then cut this wedge along the illustrated counterflow line, then we obtain a pair of forested quantum wedges, each of which corresponds, by construction, to the quantum wedge illustrated in the third line of Figure~\ref{fig::fig11b}.  In particular, this analysis tells us how the $\gamma$-LQG length of $\R_-$ and $\R_+$ changes when we zip and unzip the image shown in Figure~\ref{fig::fig10}.

\begin{corollary}
\label{cor::stablelength}
In the context of Figure~\ref{fig::fig10}, the net change in the $\gamma$-LQG length of~$\R_-$ and~$\R_+$ as one ``zips up'' is given by an independent pair of totally asymmetric $\tfrac{\kappa'}{4}$-stable L\'evy processes with positive jumps, each like the one illustrated in the first row of Figure~\ref{fig::fig11b}.   In the context of Figure~\ref{fig::fig10}, the net change in left and right boundary lengths as one ``unzips'' is given by an independent pair of totally asymmetric $\tfrac{\kappa'}{4}$-stable L\'evy processes with negative jumps, each like the one illustrated in the third row of Figure~\ref{fig::fig11b}.  
\end{corollary}

Recall that the first row of Figure~\ref{fig::fig11b} describes a forested line.  Moreover, Corollary~\ref{cor::stablelength} implies that the third row of Figure~\ref{fig::fig11b} encodes a forested wedge of weight $\gamma^2-2$ corresponding to (by Theorem~\ref{thm::sle_kp_on_wedge}) those bubbles which are either to the right but not surrounded by an $\SLE_{\kappa'}$ process (a $\gamma^2-2$ wedge) or those bubbles which are on the right and completely surrounded by an $\SLE_{\kappa'}$ process (a forested line).  One can glue the former forested line to the latter forested wedge to obtain a doubly forested wedge of weight $\gamma^2-2$.  The law of this doubly forested wedge is invariant under the operation of ``sliding the tip'' a fixed amount of quantum time units along the boundary of the forest.  This simply corresponds to the fact that the $\tfrac{\kappa'}{4}$-stable L\'evy process indexed by all of~$\R$ (which encodes this doubly forested wedge) has a stationary law.  The second and fourth rows of  Figure~\ref{fig::fig11b} describe different ways to encode the same process.  Although we will not need this point here, we remark (and leave it to the reader to check) that as one moves from right to left along the processes in either the second or fourth row, one encounters or completes disks (which correspond to jumps) at a Poisson rate, and the processes shown can be understood as stable L\'evy processes conditioned to stay positive.

We next remark that it is not hard to derive analogs of Theorem~\ref{thm::quantum_natural_zip_unzip_rough_statement} and Corollary~\ref{cor::stablelength} that involve radial and whole-plane $\SLE_{\kappa'}$ processes.  For these variants, one considers a single doubly forested wedge of some weight, zipped up like the doubly forested wedge in Figure~\ref{fig::fig8}, so that the interface becomes the $\SLE_{\kappa'}(\rho)$ curve described in Theorem~\ref{thm::quantum_cone_sle_kp}.   We then keep track of what happens as we begin to cut with scissors along the counterflow line, starting from the origin (always conformally mapping the infinite unexplored region conformally to $\C \setminus \D$).  (This is the usual setup used to construct a whole-plane $\SLE_{\kappa'}(\rho)$ curve.)  The case in which the wedge (along which the two forested lines are added) has weight $\gamma^2 - 2$ turns out to be particularly interesting, because in this case, the disks in the wedge turn out to play (in some sense) the same role as the disks in the forests rooted on that wedge.

\subsubsection{Discrete intuition}
\label{subsubsec::discrete_intuition}

\begin{figure}[ht!]
\begin{center}
\includegraphics[scale=0.85,page=3]{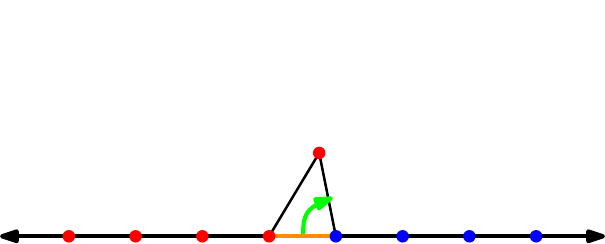}
\end{center}
\caption{\label{fig::half_planar1} The horizontal line represents the (infinite) boundary of a domain Markov half-planar map with a marked boundary edge (orange).  The boundary vertices to the left (resp.\ right) of this marked edge are colored red (resp.\ blue).  The rest of the vertices in the map are colored red or blue i.i.d.\ with probability~$\tfrac{1}{2}$.  Shown is the first step of the percolation exploration starting from the marked edge with red (resp.\ blue) on the left (resp.\ right) side of the path.  Triangles which have an edge on the boundary either have their third vertex on the boundary of the map or in the interior of the map.  The triangle shown is of the latter type.  By the domain Markov property, the conditional law of the map in the unbounded component given the revealed triangle is the same as the law of the original map.}
\end{figure}

\begin{figure}[ht!]
\begin{center}
\includegraphics[scale=0.85,page=2]{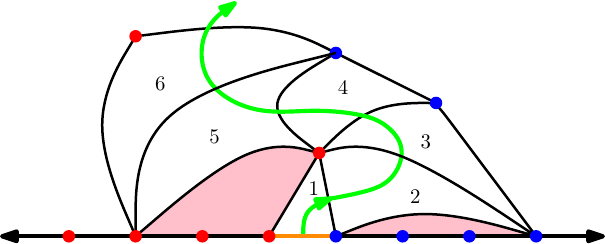}
\end{center}
\caption{\label{fig::half_planar2} (Continuation of Figure~\ref{fig::half_planar1}.)  Shown are five more steps of the percolation exploration (for a total of six steps).  The triangles are numbered according to the order in which they are visited by the exploration path.  When the exploration visits the second and fourth triangles it disconnects regions of the map from $\infty$; these are colored light red in the illustration.  By the domain Markov property, the law of the sequence of these regions is i.i.d.\  One can also keep track of the incremental net changes to the left (resp.\ right) boundary length $L_n$ (resp.\ $R_n$).  Here, $L_0 = R_0 = 0$.  Whenever a triangle is revealed with its third vertex in the interior (as in Figure~\ref{fig::half_planar1}), we have that $L$ (resp.\ $R$) goes up by $1$ if the vertex is red (resp.\ blue) and $R$ (resp.\ $L$) remains unchanged.  Whenever a triangle is revealed with its third vertex on the left (resp.\ right) boundary arc, $L$ (resp.\ $R$) decreases by one less than the boundary length of the region separated from~$\infty$ and~$R$ (resp.\ $L$) remains unchanged.  The domain Markov property implies that the process $(L,R)$ has i.i.d.\ increments.}
\end{figure}

Consider the uniform {\bf half-planar} random planar triangulation as defined in \cite{as2003uipt,angel2003growth,ac2013percolation,ar2013classification}.  This is a simply connected infinite planar map with an infinite boundary and a distinguished ``origin'' edge; it has the {\em domain Markov property}, which means that if we condition on any finite collection of triangles discovered by exploring from the boundary, then the conditional law of the infinite connected component of the unexplored region has the same law as the original map.  (See \cite[Definition~1.1]{ar2013classification} for a more careful definition.)  Suppose that we pick a distinguished edge on the boundary of the map and then color the vertices which are to the left (resp.\ right) of this edge red (resp.\ blue) as illustrated in Figure~\ref{fig::half_planar1}.  We then color the remaining vertices in the map i.i.d.\ red or blue with equal probability $\tfrac{1}{2}$.  Consider the percolation exploration which starts at the marked boundary edge with red vertices on its left side and blue vertices on its right side.  (See Figures~\ref{fig::half_planar1}--\ref{fig::half_planar2}.)  By the domain Markov property, it follows that:
\begin{enumerate}[(i)]
\item\label{it::component_law} The components which are cut off by the exploration form an i.i.d.\ sequence.
\item\label{it::boundary_length_change} The change in the lengths of the left (resp.\ right) boundary of the unbounded component of the map evolve as random walks with independent increments.
\end{enumerate}
For $\kappa'=6$, Theorem~\ref{thm::gluingtwoforestedlines} and Theorem~\ref{thm::sle_kp_on_wedge} give the continuum analog of~\eqref{it::component_law} and Corollary~\ref{cor::stablelength} gives the continuum analog of~\eqref{it::boundary_length_change}.

The stable L\'evy processes described in Proposition~\ref{prop::iidlevy} and Corollary~\ref{cor::stablelength} for $\kappa'=6$ are consistent with a heuristic argument made by Angel just before the statement of \cite[Lemma~3.1]{angel2003growth} for the scaling limit of the boundary length process associated with a percolation exploration on the uniform infinite planar triangulation.  That the scaling limit of these processes are independent $3/2$-stable L\'evy processes in the case of the uniform infinite half-planar triangulation was in fact later proved by Angel and Curien in \cite{ac2013percolation}.  Let us also mention that it is consistent with~\cite[Question~5]{cur2013glimpse}.  The law of the boundary length process associated with the growth of a metric ball on $\sqrt{8/3}$-LQG should be related to the boundary length process associated with a percolation exploration process on $\sqrt{8/3}$-LQG via the following time-change: if $A_t$ denotes the boundary length process associated with the former then the latter should be given by $A_{t(s)}$ where $t(s) = \inf\{ r \geq 0 : \int_0^r A_u du \geq s\}$ (this says that the rate of growth should be proportional to boundary length, as in first passage percolation with i.i.d.\ $\exp(1)$ edge weights).  If~$A$ is a totally asymmetric $\tfrac{3}{2}$-stable process conditioned to be non-negative, then~$A_{t(s)}$ has the law of the time-reversal of a continuous-state branching process with branching mechanism $\psi(u) = u^{3/2}$ \cite{lamp1967csbp}.  This is consistent with a result due to Krikun \cite[Theorem~4]{krikun2005uipq} for the evolution of the length of the boundary of a metric ball as its diameter increases in the setting of the uniform infinite planar quadrangulation.  This is also consistent with a calculation carried out in the continuum for the Brownian plane \cite{cl2012brownianplane} recently announced by Curien and Le Gall in \cite{legall2014icm,cl2014hull}.

The results described in Section~\ref{subsec::matingsandloops} are important ingredients for work by the second two co-authors \cite{qlebm,qle_continuity,qle_determined}  concerning the so-called quantum Loewner evolution ($\QLE$) \cite{ms2013qle}.  In particular, they allow us to define a ``quantum natural time'' version of $\QLE$ \cite{qlebm}.    Since $\QLE(8/3,0)$ is constructed by applying a certain transformation to $\SLE_6$ (so-called ``tip-rerandomization''), Theorem~\ref{thm::quantum_natural_zip_unzip_rough_statement} gives us that the Poissonian structure of the complementary components of a $\QLE(8/3,0)$ exploration and an $\SLE_6$ exploration of $\sqrt{8/3}$-LQG are the same.  Also, Theorem~\ref{thm::sle_exploration_boundary_length} gives us that the evolution of the $\sqrt{8/3}$-LQG length of the outer boundary of a $\QLE(8/3,0)$ is the same as the corresponding boundary length process for a metric ball in the Brownian plane.  Finally, Theorem~\ref{thm::quantum_cone_bm_rough_statement} gives many distributional identities between the Poissonian structure of the complementary components of a $\QLE(8/3,0)$ and a metric ball in the Brownian plane.  As a simple example, it implies that the conditional law of the area of such a component given its boundary length is the same as in the setting of the Brownian plane.

\subsection{Outline}
\label{subsec::outline}

We have mostly organized this paper in a linear fashion (so that results appear in the order they are needed). However, readers who already have some background in this subject may prefer to read the paper in a non-linear order, since the sections after Section~\ref{sec::couplings} are arguably the most interesting. One option for an experienced reader is to jump ahead to either Section~\ref{sec::couplings} or Section~\ref{sec::structure_theorems} and then to refer back to earlier sections for reference as needed. Section~\ref{sec::motivation} contains introductory material that is important for motivation but can also be skipped on a first read. Section~\ref{sec::preliminaries} contains relevant facts about Bessel processes and describes elementary but interesting symmetries enjoyed by forward and reverse $\SLE_\kappa(\rho)$ processes. Section~\ref{sec::surfaces} gives an overview of the different types of GFFs that appear in this article as well as the different types of quantum surfaces (quantum wedges, cones, disks, and spheres). 

Section~\ref{sec::couplings} contains a brief review of the couplings between the GFF and $\SLE$ that are relevant to this article.  We also describe the law of the local behavior of a boundary intersecting $\SLE_\kappa(\rho)$ process at a time when it first cuts off a given boundary point from $\infty$.  In Section~\ref{sec::structure_theorems} we describe the Poissonian structure of the bubbles that one encounters when cutting a $\gamma$-LQG surface along an $\SLE_\kappa(\rho)$ or $\SLE_{\kappa'}$ process; we also introduce quantum natural time.  The main conformal welding results for wedges are established in Section~\ref{sec::quantumwedge}.  We use these results in Section~\ref{sec::brownian_boundary_length} to show that the evolution of the left and right boundary lengths of a space-filling $\SLE_{\kappa'}$ process drawn on top of an LQG surface are indeed given by correlated Brownian motions (matching the continuum limit for random planar maps determined in \cite{sheffield2011qg_inventory}).  Section~\ref{sec::trees_determine_embedding} then shows that the quantum surface and the space-filling $\SLE_{\kappa'}$ can be a.s.\ be recovered from the pair of Brownian motions.  In Section~\ref{sec::duality} we include additional discussion on forested wedges, trees of bubbles defined using L\'evy processes, and quantum surfaces explored by $\SLE_{\kappa'}$ with $\kappa' \in (4,8)$.  These L\'evy trees of disks (and related L\'evy trees of spheres) correspond to objects constructed heuristically in the physics literature, where they are known as $\gamma'$-LQG surfaces with ${\gamma'}^2 \in (4,8)$.  We finish in Section~\ref{sec::questions} with a list of open questions.

Appendix~\ref{app::disks_spheres} provides additional information about finite volume quantum surfaces (spheres and disks) and sets the stage for the follow up paper \cite{quantum_spheres}. Appendix~\ref{app::kpz} makes connections between the results established here and various scaling exponent calculations made by the first co-author via the KPZ formula.

\section{Background and motivation} \label{sec::motivation}
\subsection{Mated CRTs as scaling limits}
\label{subsec::perspective2}
In this section, we say more about how our CRT-mating theorems fit into the context of Liouville quantum gravity, conformal loop ensembles, and random planar maps.  As illustrated in Figure~\ref{papersketch}, this paper can be understood, in some sense, as a culmination of a multi-paper project aimed at establishing a connection between
\begin{enumerate} \item random planar maps ``decorated'' by FK distinguished edge subsets (and the loops forming cluster/dual-cluster interfaces), and
\item Liouville quantum gravity (LQG) random surfaces ``decorated'' by independently sampled conformal loop ensembles (CLE).
\end{enumerate}

\begin {figure}[htbp!]
\begin {center}
\includegraphics [width=5.5in]{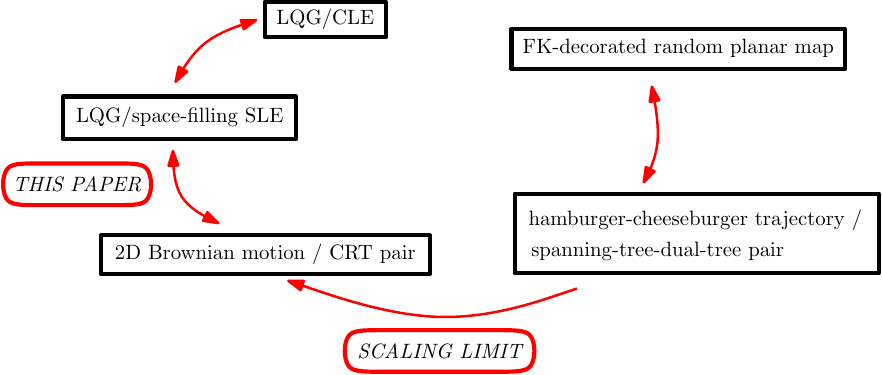}
\caption {\label{papersketch} This paper focuses on the double arrow indicated above.  The LQG/space-filling $\SLE$ box represents a space-filling $\SLE_{\kappa'}$ drawn on top of a $\gamma$-Liouville quantum gravity surface, where $\gamma \in (0,2)$, $\kappa' = 16/\kappa$, and $\kappa = \gamma^2$.  The 2D Brownian motion/CRT pair box represents the coupled pair of CRTs described in Theorems~\ref{thm::quantum_cone_bm_rough_statement} and~\ref{thm::trees_determine_embedding}. }

\end {center}
\end {figure}

This paper will focus narrowly on the double arrow of Figure~\ref{papersketch} as the figure label indicates, and as was discussed in Section~\ref{subsec::easy}.  We will not discuss the other arrows in Figure~\ref{papersketch} outside of a few high level contextual remarks and references, which form the remainder of Section~\ref{subsec::perspective2}.

 \begin{enumerate} \item The double arrow on the right in Figure~\ref{papersketch} is explained in detail in \cite{sheffield2011qg_inventory}.  In that paper, one considers a rooted planar map $M$ with $n$ edges, together with a distinguished subset $E$ of those edges associated to the self-dual FK random cluster model with parameter $q$.  The interfaces between clusters and dual clusters are understood as loops.  Given an $(M,E)$ pair, there is a canonical space-filling path that ``explores'' the loops, and also represents the interface between a special spanning tree $T_1$ and a special dual spanning tree $T_2$.  The structure of these trees (and the way that they are ``glued'' together) is encoded by a walk in $\Z^2$, which is related in  \cite{sheffield2011qg_inventory} to a two-product LIFO (last in first out) inventory management (``hamburger-cheeseburger'') model.    The trees $T_1$ and $T_2$ arising in a computer simulation of this construction (together with the ``space-filling loop'' in between them) appear in Figures~\ref{fig::random_planar_map} and~\ref{fig::circle_packing_animation}.  These figures also illustrate a method of embedding the planar maps in the disk ``conformally'' using circle packing computed with \cite{stephensonCP}.
\item The ``scaling limit'' arrow in Figure~\ref{papersketch} is also explained in \cite{sheffield2011qg_inventory}, where it is shown that the above mentioned walk in $\Z^2$ scales to a two-dimensional Brownian motion, with a diffusion rate depending on $q$.

\item The upper left double arrow in Figure~\ref{papersketch} is explained in more detail in \cite{she2009cle} and \cite{ms2013imag4}.  These works show that there is a space-filling variant of $\SLE_\kappa$ that ``explores'' the entire $\CLE_\kappa$ loop ensemble, and which may be understood as a continuum analog of the loop exploration path discussed in \cite{sheffield2011qg_inventory}.  In this context, the analog of the pair $(T_1, T_2)$ discussed above is a pair $(\T_1, \T_2)$ of space-filling trees, which can be interpreted as flow-line and dual-flow-line trees within an ``imaginary geometry'' based on a GFF.
\end{enumerate}

\begin{figure}[ht!]
\begin{center}
\includegraphics[width=0.48\textwidth]{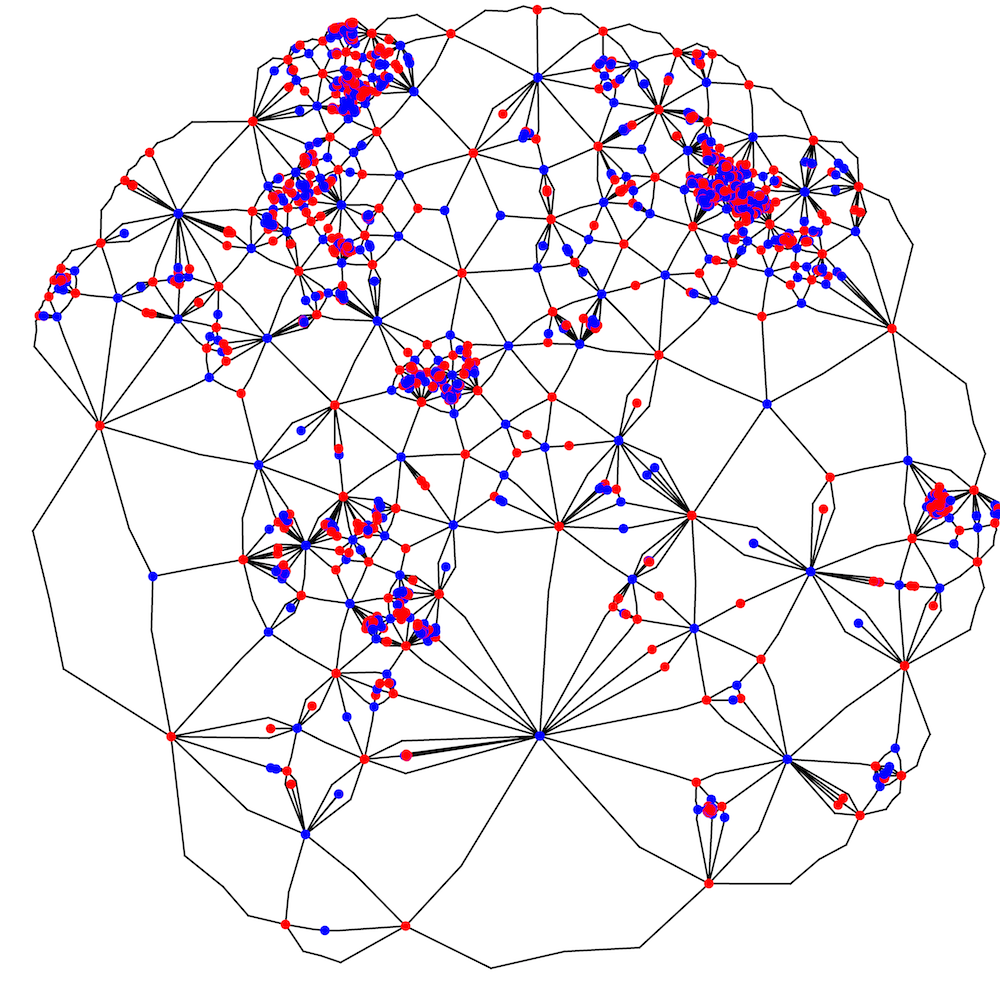}
\hspace{0.02\textwidth}
\includegraphics[width=0.48\textwidth]{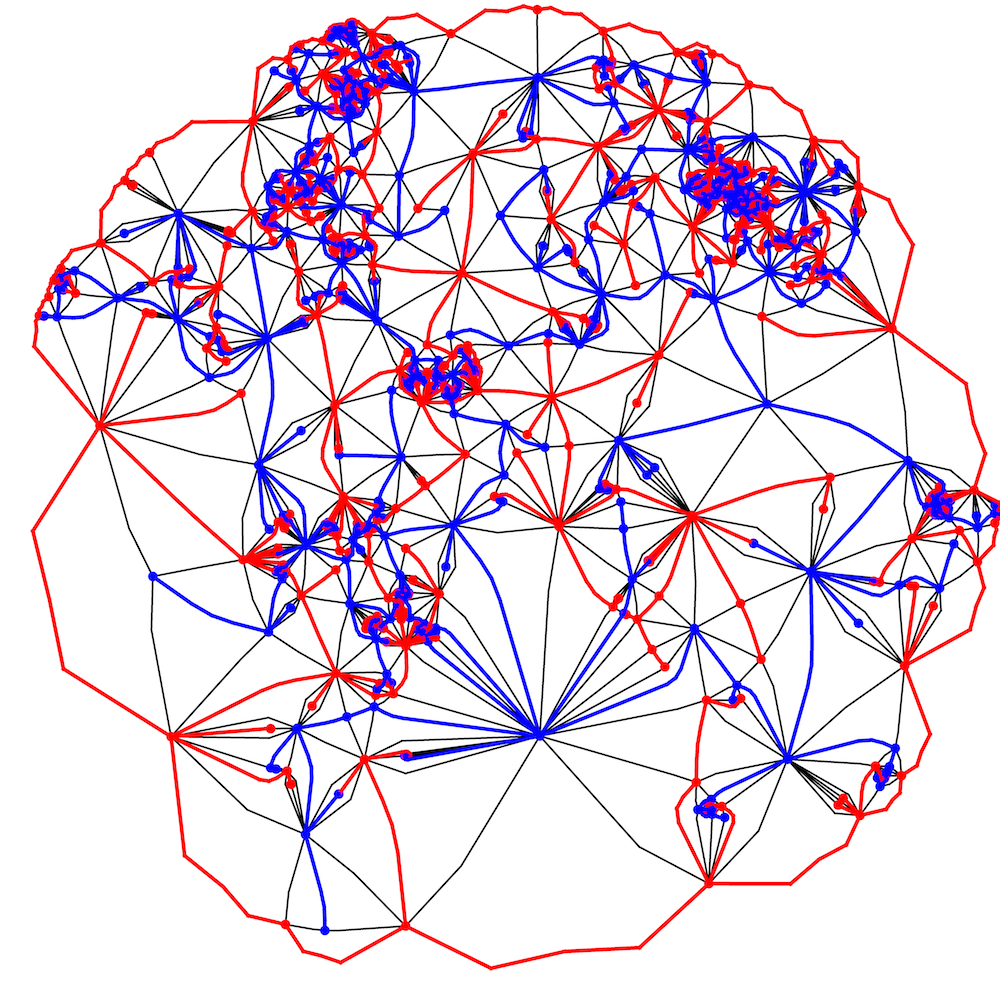}
\vspace{-0.05\textheight}
\includegraphics[width=0.73\textwidth]{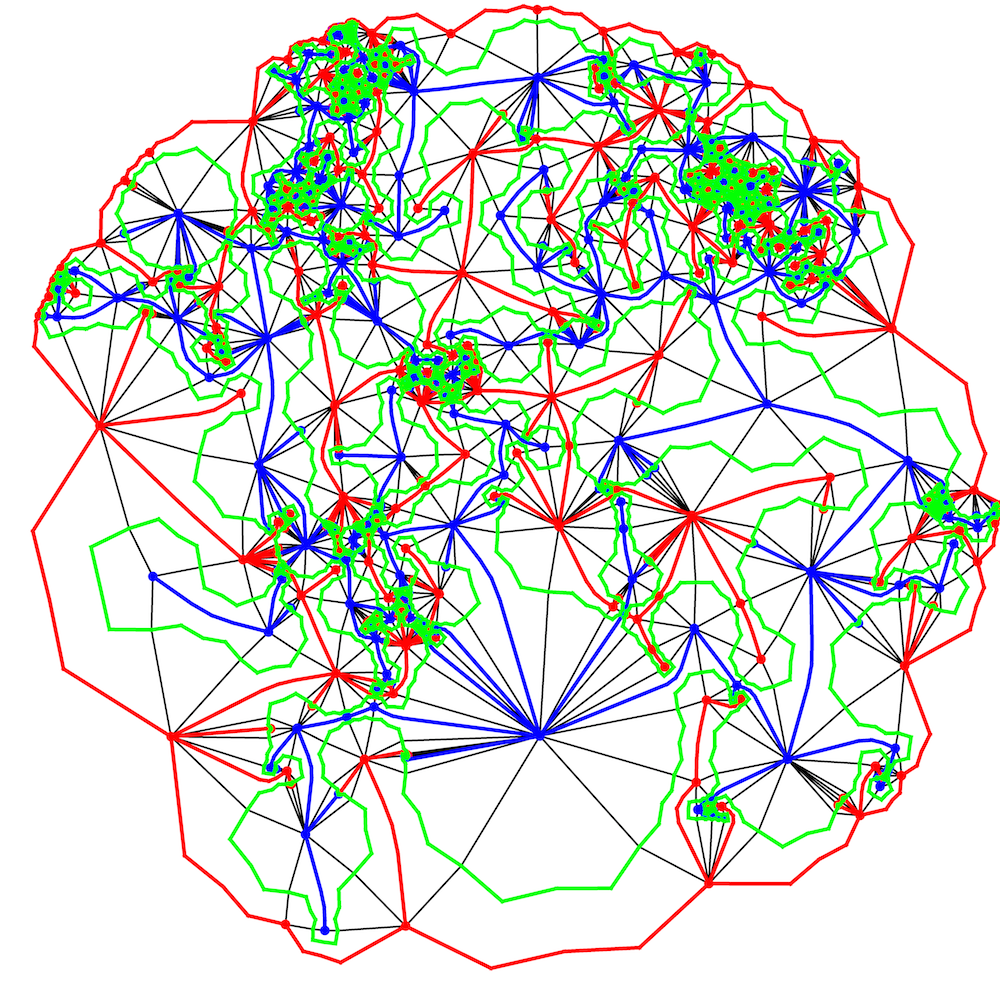}
\end{center}
\caption{\label{fig::random_planar_map} { {\bf Upper left:} random planar map sampled using the bijection from \cite{sheffield2011qg_inventory} with $p=0$ (where $p$ is the parameter from \cite{sheffield2011qg_inventory}) and embedded into $\C$ using \cite{stephensonCP}.  {\bf Upper right:} same map with distinguished tree/dual tree pair.  {\bf Bottom:} map with path which snakes between the trees and visits all of the edges.}}
\end{figure}

\begin{figure}[ht!]
\begin{center}
\subfloat[25\%]{\includegraphics[width=0.48\textwidth]{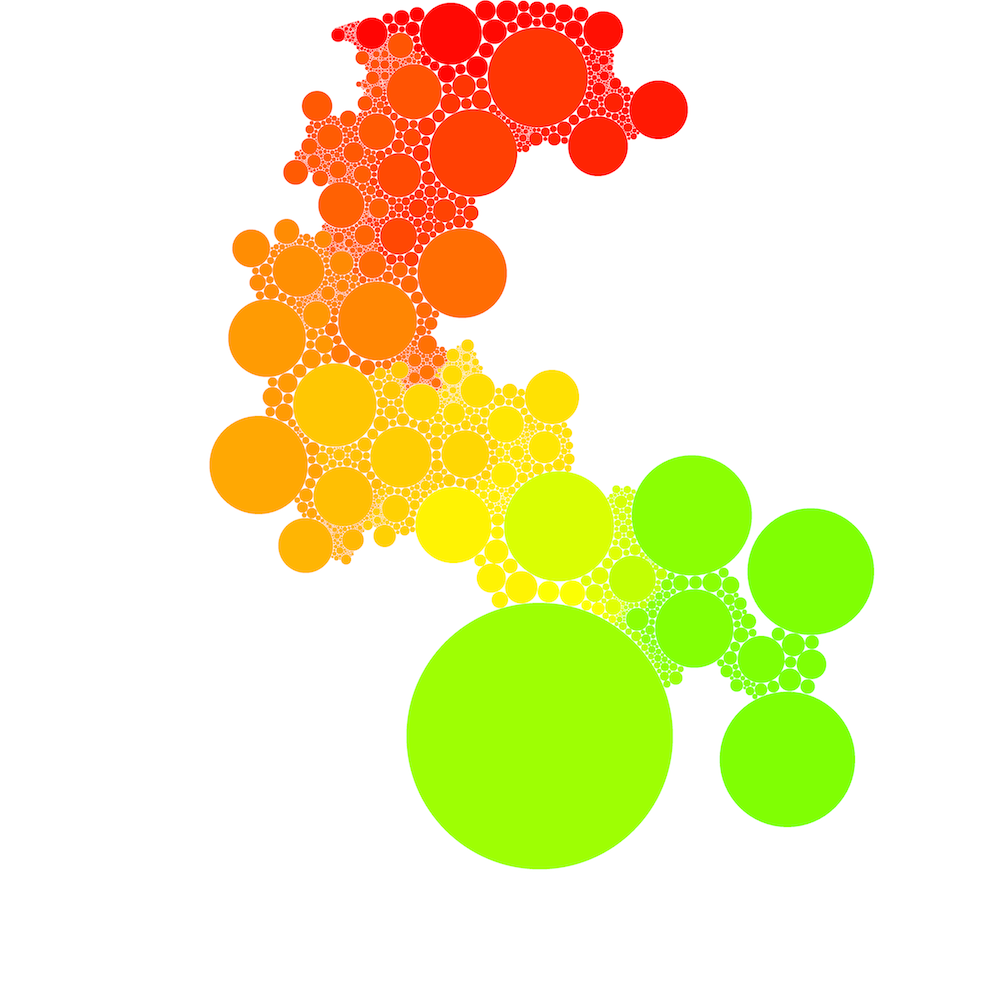}}
\hspace{0.02\textwidth}
\subfloat[50\%]{\includegraphics[width=0.48\textwidth]{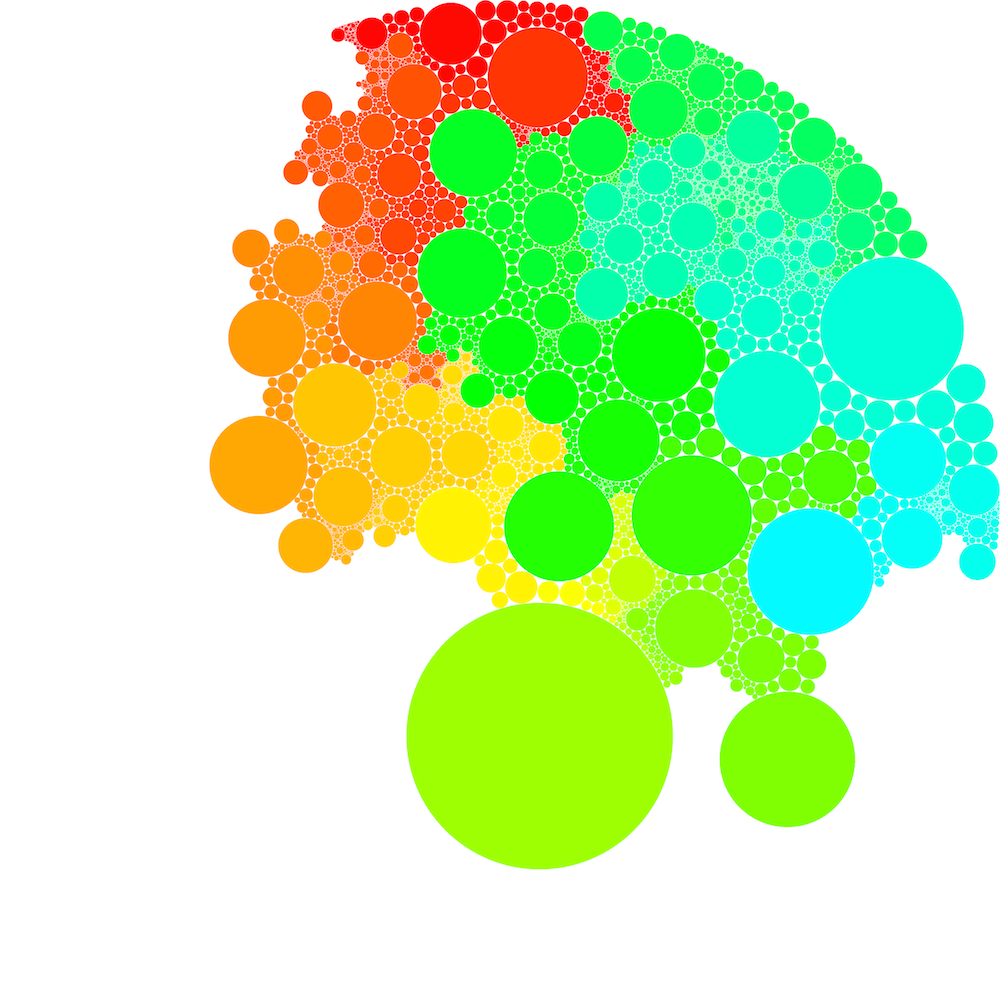}}

\subfloat[75\%]{\includegraphics[width=0.48\textwidth]{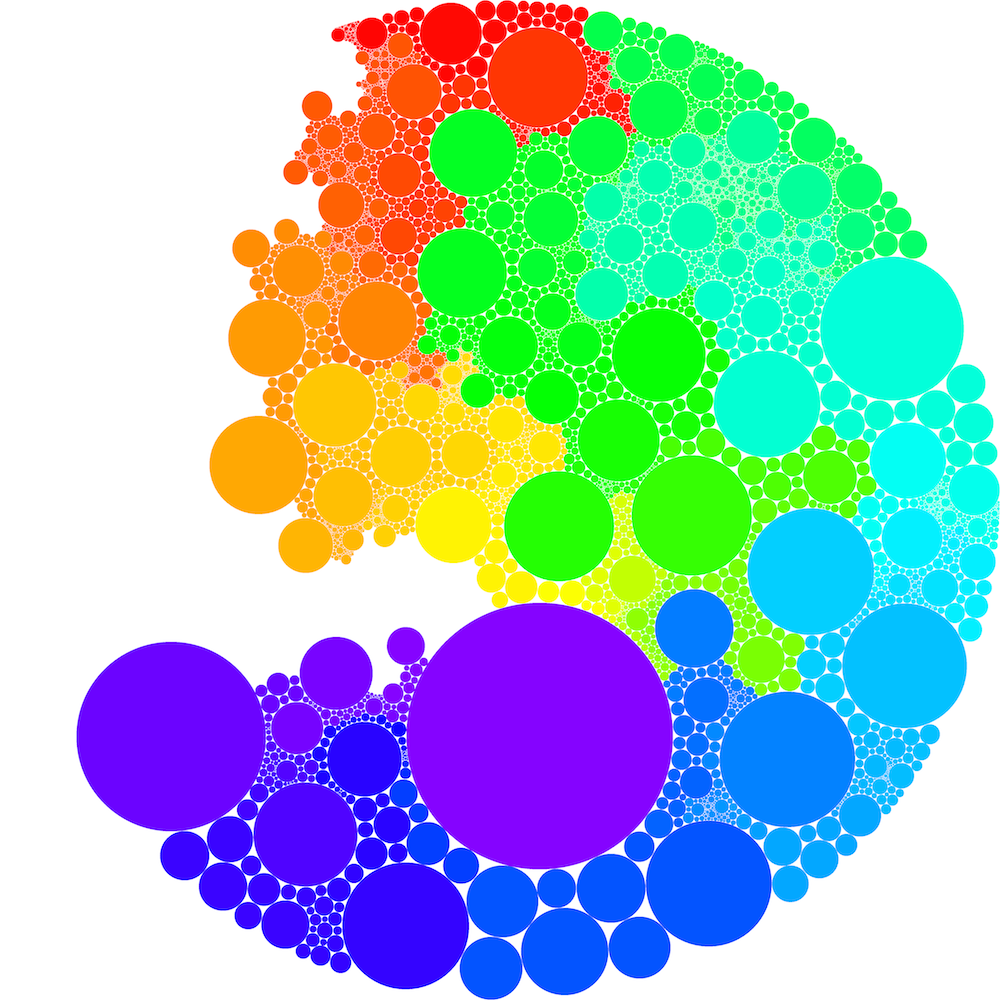}}
\hspace{0.02\textwidth}
\subfloat[\label{fig::circle_packing_big}100\%]{\includegraphics[width=0.48\textwidth]{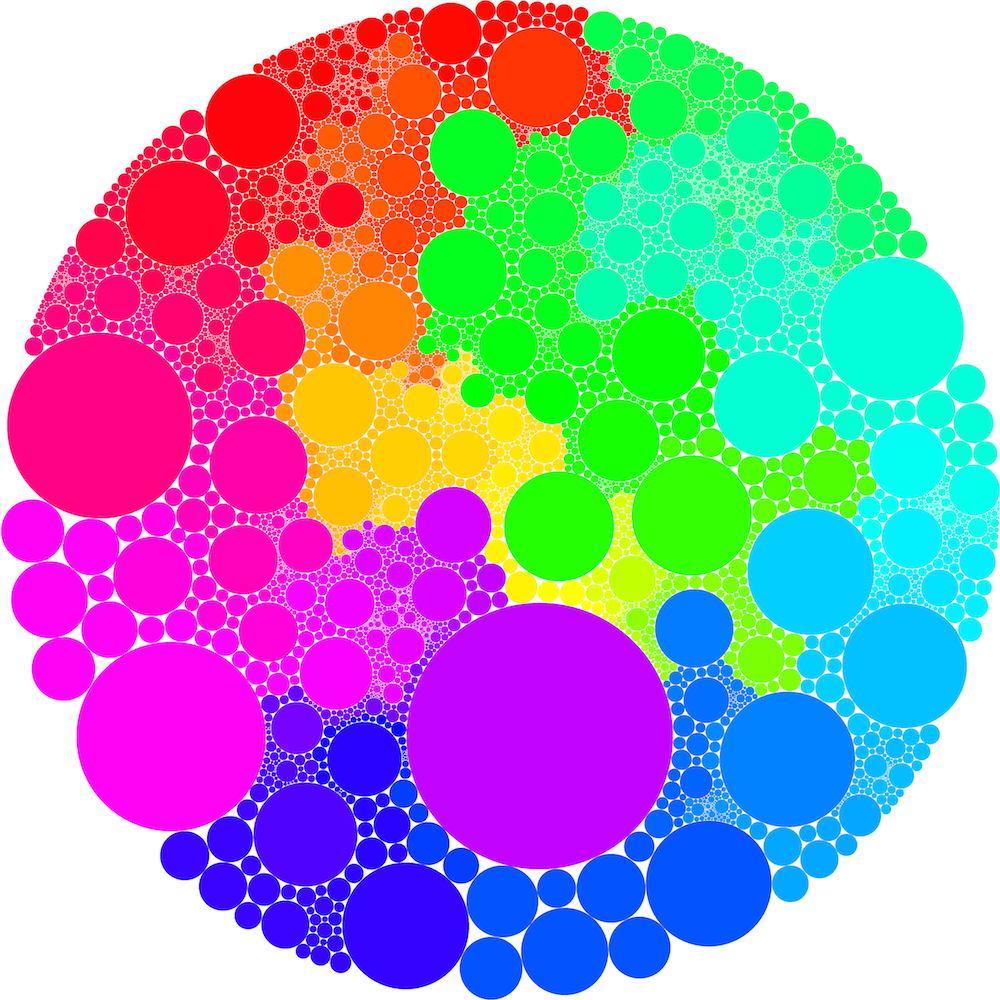}}
\end{center}
\caption{\label{fig::circle_packing_animation} Shown is the circle packing associated with the map from Figure~\ref{fig::random_planar_map}.  The circles are colored according to the order in which they are visited by the space-filling path.  The different panels show the circles visited by the path up to the first time that they cover $25\%$, $50\%$, $75\%$, and $100\%$ of the total area.}
\end{figure}

More references to the physics literature and discussion of the overall project are also found in \cite{sheffield2011qg_inventory, she2009cle, ms2013imag4}.  As discussed in \cite{sheffield2011qg_inventory}, the combination of the results summarized in Figure~\ref{papersketch} implies that certain structures encoded by the discrete models, namely the contour functions of the trees $(T_1, T_2)$, have scaling limits that agree in law with the analogous structures encoded by CLE-decorated LQG. 

\subsubsection{Peanosphere topology}  This brief subsection (which can be skipped on a first read) will explain formally how one can use the equivalences in Figure~\ref{papersketch} to {\em reinterpret} the scaling limit theorem from \cite{sheffield2011qg_inventory} as a theorem about scaling limits of loop-decorated surfaces. Specifically, we will argue that CLE-decorated LQG is the scaling limit of FK-weighted random planar maps in a topology (which we sometimes called the {\em peanosphere topology}, discussed more formally below) where two loop-decorated metric surfaces are considered close when their encoding contour functions are close, or equivalently, when their associated tree/dual-tree pairs are close as measure-endowed metric spaces. (The term ``peanosphere'' was suggested by Richard Kenyon in private communication as an informal name for the sort of surface obtained by mating two continuum random trees---namely, a sphere-homeomorphic surface decorated by space-filling, non-self-crossing ``Peano curve'' that comes with a continuously-varying notion of left and right boundary length.)

In a certain sense this is just a matter of semantics (putting a different spin on Figure~\ref{papersketch}) but we stress this point because it is often the case in this subject that once one has convergence in one topology, it is possible to strengthen the topology of convergence without starting from scratch. For example, Marckert and Mokkadem, building on work by Cori and Vauquelin \cite{cori1981planar} and by Chassaing and Schaeffer \cite{chassaing2004random},
proved a Brownian map convergence result in \cite{marckert2006limit} that is similar in spirit to what we call peanosphere topology convergence; later works by Le Gall and by Miermont used these earlier results to establish stronger forms of convergence in \cite{le2013uniqueness, miermont2013brownian}. Similarly, the recent works \cite{gms2015cone_times,gs2015finite_volume,gs2015finite, gm2016topology} use the convergence results of this paper to establish scaling limit results in topologies that encode the metric structures of loops, among other things. Several more results along these lines are referenced and summarized in the recent survey papers \cite{gwynne2019mating, gwynne2019random}.

Now let us explain the topology of convergence more formally. The quotient construction in Section~\ref{subsec::easy} makes sense if $X_t$ and $Y_t$ are replaced by {\em any} (not necessarily Brownian) continuous excursions. However, the hypotheses of Moore's theorem (Proposition~\ref{prop::moore}) might not be satisfied, so the quotient space might not be a topological sphere.  (For example, if we had $X_t = Y_t$ for all $t$, then the quotient would be a continuum tree --- obtained by ``gluing two identical trees to each other.'') Nonetheless, we would always obtain {\em some} topological space, with a parameterized space filling curve corresponding to traversing the red vertical lines from left to right.  Also, for each $t\in[0,T]$, we would obtain a parameterized path that corresponds to traversing the horizontal green lines (which intersect the line $x=t$) from top to bottom; it has one arc traversed in time $X_t$ and one in time $Y_t$. We interpret $X_t$ and $Y_t$ as left and right ``boundary lengths'' of the ``interface'' between $\eta([0,t])$ and $\eta([t,T])$, even though one cannot expect these to be simple curves in general. The point is that the set of {\em all} pairs $(X_\cdot, Y_\cdot)$ of continuous excursions on $[0,T]$ naturally parameterizes a certain set of topological spaces (each decorated by a parameterized space-filling curve, and certain parameterized ``boundary paths''), and this set comes with a natural topology: namely, the $L^\infty$ topology on maps from $[0,T]$ to $\R^2$. We remark that sphere homeomorphic surfaces (with simple boundary paths) are dense in this topology (since for any given pair of excursions $(L_t, R_t)$ one can always condition on the positive probability event that two Brownian excursions are within $L^\infty$ distance $\epsilon$ of this pair, and we know from Section~\ref{subsec::easy} that the result is a.s.\ homeomorphic to a sphere, with simple boundary paths).

Now, to describe the planar map convergence more formally, we need a topology on a space that includes both probability measures on discrete loop-decorated surfaces {\em and} probability measures on continuum loop-decorated surfaces.  We begin with an infinite volume version of the topology described just above (see \cite{quantum_spheres} for the finite-volume case). Namely, let $C$ be the set of functions from $\R$ to $\R^2$, endowed with the local $L^\infty$ metric. Then let $\mathcal C$ be the corresponding weak topology on the set $S_C$ of {\em probability measures} on $C$. Let $L$ be a set that includes the set of infinite volume discrete loop-decorated planar maps {\em as well as} the set of possible instances of a CLE-decorated $\gamma$-quantum cone. Then let $S_L$ be the set of probability measures on $L$.  In other words, an element of $S_L$ is a weighted average of a measure on discrete loop-decorated maps and a measure whose law is a.s.\ absolutely continuous with respect to the law of a CLE-decorated quantum cone.

Let $g:L \to C$ be the function that takes a discrete (or continuum) loop-decorated infinite volume surface to the corresponding two parameter function that encodes it. (In the discrete case, $g$ takes the FK-decorated map to the corresponding walk described in \cite{sheffield2011qg_inventory}. In the continuum case, it takes an instance of a $\CLE$-decorated $\gamma$-quantum cone to the corresponding pair of Brownian motions (recall Theorem~\ref{thm::quantum_cone_bm_rough_statement}).  In other words, $g$ maps elements of the uppermost two boxes in Figure~\ref{papersketch} to elements of the lowermost two boxes. Note that although we do not define $g$ for {\em every} possible continuum loop decorated surface, $g$ is defined for a.a.\ instances of the appropriate CLE-decorated $\gamma$-quantum cones. Thus $g$ induces a pushforward map $\wt g$ from $S_L$ to $S_C$. Now generally, if we are given a function $f:A\to B$ then for any $B'\subset B$ we write $f^{-1}(B') = \{a \in A : f(a) \in B' \}$.  So $\wt g^{-1}$ can be understood as a map from the collection of subsets of $S_C$ to the collection of subsets of $S_L$. The topology we wish to consider is then  $\wt g^{-1}(\mathcal C)$.  The scaling limit in this topology is by definition equivalent to the main result of \cite{sheffield2011qg_inventory}.

Let us stress that convergence in  $\wt g^{-1}(\mathcal C)$ does {\em not} imply convergence in the weak topology corresponding to the Gromov-Hausdorff topology on metric spaces, or in any topology that encodes the structure of discrete conformal embeddings of planar maps.  In a mathematical sense, $\wt g^{-1}(\mathcal C)$ is neither weaker nor stronger than these other topologies. It simply encodes different information. On the other hand, the type of convergence discussed here (which boils down to proving that functions encoding discrete mated trees converge to correlated Brownian motions) is sometimes the easiest kind of convergence to establish, and it is also extremely useful; see the discussion below.

{\em Update:} As mentioned above this paper suggests that one natural way to study discrete statistical physics models on planar maps is to try to identify a natural pair of trees induced by those models, and then use a combinatorial understanding of the trees to prove a relationship to correlated CRTs, and hence to SLE/LQG. Since the first version of this paper was posted to the arXiv, this approach has been successfully applied by a number of authors to several well-known combinatorial objects that had not previously been studied in the conformal probability context, including bipolar orientations and Schnyder woods as well as certain generalizations of the FK cluster model; see \cite{kmsw2015bipolar,GKMW:active-tree-map} (combined with \cite{ghms2015covariance}) as well as \cite{lsw2017schnyder}.

The new results allow us to add the models described above to the pantheon of canonical SLE/LQG-related models, which already includes loop-erased random walk, percolation, the Ising model, the uniform spanning tree, and the GFF.  In particular, these are the first elements of the pantheon corresponding to $\kappa < 2$ and $\kappa' > 8$. Another recent paper uses the convergence results of this paper to solve open problems about random walks on random planar maps (spectral dimension, return probabilities, and so forth) \cite{gm2017spectral}.  See also \cite{ghs2017distances} for an approach to addressing problems involving graph distances on planar maps using the tools of this paper.  Overall, we expect that these ideas will in time lead to additional bridges between combinatorics and conformal probability --- in part because the mated-CRT convergence results for random planar maps are often more accessible than the analogous results for deterministic lattices, and can be established without a precise understanding of the discrete conformal structure.

\subsection{Continuum implications of tree-mating theory}
When we are given a topological sphere, we use the term ``conformal structure'' to mean a homeomorphism between that topological sphere and the Riemann sphere $\C \cup \{ \infty \}$, defined modulo M\"obius transformation.  A conformal structure on a surface gives a way to define Brownian motion (modulo time change) started from a point on that surface (namely, it is the pullback of Brownian motion on $\C \cup \{\infty \}$). This paper shows that a mated pair of CRTs can a.s.\ be canonically embedded in the plane in a particular way, and that the resulting object has the law of an LQG quantum surface parameterized by $\C \cup \{ \infty \}$ and decorated by a space-filling form of $\SLE$.

One reason that this result is intriguing is that it allows one to convert many questions about $\SLE$ and LQG into questions about the pairs of CRTs that can be asked without reference to conformal structure.  In principle, this could allow future researchers to derive properties of $\SLE$, LQG, and related structures in papers that never mention either Loewner evolution or the GFF.  For example, the major open problem of endowing LQG with a metric space structure for general~$\gamma$ (see the discussion in \cite{ms2013qle}) might turn out to be more easily addressed from the mated-tree perspective than from the conventional LQG perspective.  ({\it Update:} The metric for general $\gamma \in (0,2)$ was constructed in \cite{dddf2019tightness, gm2019local, gm2019confluence, gm2019uniqueness,gm2019conformalcov}.  See also \cite{gp2018components,ghm2015kpz,ghm2016trans} for some examples of this.)

Another such open problem (which we state but will not solve here) is the following.  Consider the graph whose vertices are the components of the complement of a chordal (non-space-filling) $\SLE_{\kappa'}$ trace with $\kappa' \in (4,8)$, where two such components are adjacent if their boundaries intersect; is this graph a.s.\ connected?  ({\it Update:} See \cite{gp2018components} for a solution of this question for $\kappa' \in (4,\kappa_0']$ where $\kappa_0' \approx 5.62$ using tools from the present article.)

Topological properties of $\SLE_{\kappa'}$ (along with some geometric properties like quantum path length) are directly encoded by the CRTs, or by certain stable L\'evy processes derived from these CRTs.  This encoding reduces problems like the one just stated to questions about stable L\'evy processes that can be asked without reference to $\SLE$.

We also stress that, as explained in \cite{she2009cle}, there is a simple procedure for constructing a conformal loop ensemble (CLE) from space-filling $\SLE_{\kappa'}$ when $\kappa' \in (4,8)$, and that the procedure itself is topological, making no use of conformal structure.  Hence the pair of CRTs encodes the CLE loops in a straightforward way, and many properties of CLE (e.g., quantum dimensions of different types of special points, quantum lengths of loops and areas of enclosed regions, properties of the graph whose vertices are loops with two loops adjacent when they intersect, etc.) could also be addressed directly from the mated-CRT construction.

\begin{remark} The procedure for producing a space-filling $\SLE_{\kappa'}$ from a $\CLE_{\kappa'}$ given in \cite{she2009cle} only works when the $\CLE_{\kappa'}$ is non-simple (so that the loops intersect each other), i.e., when $\kappa' > 4$. Although we do not prove this here, it is not too hard to show that as $\kappa' \to 4$ from above, the law of space-filling $\SLE_{\kappa'}$ does not converge to the law of a continuous random path.  (Very roughly speaking, as $\kappa'$ approaches $4$ from above, the path tends to spiral around the boundary of the ``already explored region'' increasingly many times in between each time it ``discovers a macroscopic loop'' --- so any limiting path would have to make infinitely many spirals in between times at which it discovered entire $\CLE$ loops instantaneously.) For this reason, the limiting case $\kappa'=4$ is outside the scope of this paper --- there is simply not a continuous space-filling $\SLE_{\kappa'}$ process when $\kappa'=4$, and thus one does not expect to obtain a {\em continuous} embedding of the pair of CRTs, which become equal to one another in the $\kappa'=4$ limit, since their correlation coefficient tends to $1$. On the other hand, many of the objects in this paper (the LQG measure, the CLEs themselves, the FK model correspondence) {\em do} make sense when $\kappa' = 4$, and it is an interesting open problem to try to construct {\em some} kind of $\kappa'=4$ analog of the results of this paper (even if the corresponding tree embedding is not a continuous one).
\end{remark}

Another reason to study mated CRT maps is that we use the matings mentioned just above, along with several other results from this paper, are used in work about the quantum Loewner evolution (QLE), as introduced in \cite{ms2013qle}, and its relationship to the Brownian map.  

In the time since the first version of this paper was posted to the arXiv, the program to connect LQG surfaces (with $\gamma^2 = 8/3$) to the Brownian map has been completed over a series of papers \cite{quantum_spheres,map_making,qlebm, qle_continuity,qle_determined}.

These papers show that the time-reversals of some of the QLE processes described in \cite{ms2013qle} (the ones constructed using $\SLE_{\kappa'}$ for $\kappa' \in (4,8)$) can be reformulated as ways to construct spheres by gluing trees of disks (more precisely, so-called {\em L\'evy trees} of disks, of the sort described in Figure~\ref{treegluingvstreeofdiskgluing} and~\ref{levytreegluing}) to themselves.  The Brownian map itself is obtained as a sort of ``reshuffling'' of the entire $\SLE_6$ exploration tree.

A final and more speculative reason to study mated CRTs is that developing objects like $\SLE$, $\CLE$ and LQG in a framework that does not explicitly reference conformal structure may help us figure out how to construct interesting analogs of these objects in higher dimensions, where the tools of planar conformal geometry (such as the Riemann mapping theorem) no longer apply.

\subsection{Conformal mating in complex dynamics}
\label{subsec::conformalmating}

The idea of somehow stitching together a pair of trees to obtain a sphere with a conformal structure is not without precedent. Indeed, the term ``conformally mating'' comes from the complex dynamics literature, where in some circumstances it is possible to obtain a conformal sphere by stitching together the Julia set of one polynomial to the Julia set of another polynomial \cite{yampolsky2001mating, milnor2004pasting,aspenberg2009mating,  timorin2010topological}. Milnor presents a particularly clear introduction to the theory in \cite{milnor2004pasting} (see also Milnor's reference text on complex dynamics \cite{milnor2006dynamics}).  In certain cases these Julia sets are ``dendritic'' like the continuum random trees described in Section~\ref{subsec::easy}, or illustrated on the left side of Figure~\ref{treegluingvstreeofdiskgluing}.  In such cases one obtains a space-filling curve and a measure on the sphere, much as in Section~\ref{subsec::easy} \cite{milnor2004pasting}.  In other cases, the Julia set is the boundary of a region with non-empty interior, more like the ``trees of disks'' that the right side of Figure~\ref{treegluingvstreeofdiskgluing} illustrates.\footnote{At the moment, an excellent gallery of related computer animations can be found on Arnaud Ch\'eritat's webpage, illustrating both dendritic and non-dendritic matings.  See also the computer images and algorithmic discussion in \cite{boyd2012medusa}.}  We present a short overview of the theory (which can be skipped on a first read) in the remainder of this subsection.

Let $f$ be a monic complex polynomial of degree $d \geq 2$.  Then the {\em Fatou set} is defined by $\{z : \lim_{n \to \infty}f^{(n)}(z) = \infty \}$, where $f^{(n)}:= f \circ f \circ \ldots \circ f$ is the $n$-fold composition of $f$.  The {\em Julia set} $J(f)$ is the boundary of the Fatou set.  The set $J(f)$ often has an interesting fractal structure.  We consider only cases where $J(f)$ is connected, and we write $K= K(f)$ for the complement of the Fatou set (a.k.a., the {\em filling} of the Julia set).  Then $f$ is a $d$-to-$1$ holomorphic map from $\C \setminus K$ to itself that approximates the power map $z \to z^d$ near infinity.  To make this point more explicit, let $\D \subseteq \C$ be the unit disk; one can show that if $\phi_K : \C \setminus K \to \C \setminus \overline \D$ is the conformal map that approximates the identity near $\infty$, then $f = \phi_K^{-1}\circ g \circ \phi_K$ where $g(z) = z^d$.  We stress that the map $\phi_K^{-1} \circ g \circ \phi_K$ could be defined for any bounded hull $K$, but in general we would not expect it to be a polynomial function (which is smooth on all of $\C$, not just on $\C \setminus K$).  The map $f$ gives a surjective map from $J(f)$ to itself, which is sometimes called the {\em dynamics} of the Julia set $J(f)$.  Also, the map $\phi_K^{-1}$ extends to give a surjective map from $\partial \D$ to $J(f)$, which provides a natural parameterization of (and measure on) $J(f)$.

The $k$-fold composition of $f$ is a $d^k$ to $1$ map from $\C \setminus K$ to itself.  Since $f^{(k)}$ is a polynomial that fixes $K$, it encodes a certain type of self-similarity of the Julia set: a small neighborhood $U$ of a generic point $z \in J(f)$ has $d^k$ pre-images $U_1, \ldots, U_{d^{k}}$ under $f^{(k)}$, and the restrictions $U_i \cap K$ all look like $U \cap K$ (up to distortion by a polynomial map).

Suppose that $J(f_1)$ and $J(f_2)$ are connected Julia sets of polynomials $f_1$ and $f_2$, both of degree $d \geq 2$.  Let $\sigma_i$ be the natural parameterization mapping $\partial \D$ to $J(f_i)$.  We let $\wt J$ be the pair $J_1 \cup J_2$ modulo the equivalence relation that identifies $\sigma_1(t)$ with $\sigma_2(t)$ when $t \in \partial \D$.  Using Moore's theorem, it is often possible to show that $\wt J$ is topologically a sphere.   Note that if each $f_i$ induces a $d$-to-$1$ map from $J(f_i)$ to itself, each of these two maps induces the same map $\psi : \wt J\to \wt J$.  The interesting thing is that, in some cases, there is a unique (up to M\"obius transformation) way to identify $\wt J$ with $\C \cup \{ \infty \}$ in such a way that $\psi$ becomes a rational functional.  This effectively determines the conformal structure of $\wt J$.

In some sense, it is not too surprising that the requirement that $\psi$ be rational, or even just conformal, should determine the conformal structure.  A generic point in $x \in \wt J$ has $d^k$ pre-images under $\psi^{(k)}$; thus, if the conformal structure were fixed in even one small neighborhood of $x$, then the requirement that $\psi$ act conformally would fix the conformal structure in all of the $d^k$ preimages of that neighborhood; and the union of such neighborhoods, taken over all $k$, can be shown to be dense and of full measure in $\wt J$.

We stress, however, that for the random trees described in this paper, we do not have a nice automorphism such as $\psi$.  Although our models are self-similar {\em in law}, an actual instance of a CRT has no {\em exact} self-similarity of the type one encounters for Julia sets, and no natural polynomial or rational dynamics. This paper uses different tools to specify the conformal structure of a glued-together pair of trees or trees of disks. In so doing, this paper also broadens the usage of the term ``conformal mating,'' so that when we say that the gluing of two objects constitutes a {\em conformal mating} we simply mean that the pair of objects a.s.\ uniquely determines {\em in some way} the conformal structure of the surface obtained by gluing the two objects together.

\section{Preliminaries}
\label{sec::preliminaries}

This section assembles a few basic facts about local times, Bessel processes, and $\SLE_\kappa(\rho)$.  The reader who is already well versed with these topics may be able to skim or skip much of this section, referring back as needed for reference. However, one should take note of the forward and reverse symmetries of $\SLE_\kappa(\rho)$ presented in Section~\ref{subsec::forward_reverse_sle_kappa_rho}.  These are relatively straightforward observations, but they are nonetheless interesting, and we have not found them articulated elsewhere in the literature.

We begin with a discussion of the time-reversal of Markov processes using local times in Section~\ref{subsec::reversing_markov}.  We recall that it is true in some generality that the {\em excursions} a Markov process makes away from a fixed value have a Poissonian structure when parameterized by a {\em local time} corresponding to that value, and we formulate a statement of the fact that this remains true when time is reversed.  Next, in Section~\ref{subsec::bessel_processes}, we give a brief overview of some important properties of Bessel processes and derive a few facts about their time reversals.  In Section~\ref{subsec::sle_kappa_rho} we recall the definition of the forward and reverse $\SLE_\kappa(\rho)$ processes in the chordal, radial, and whole-plane settings.  Finally, in Section~\ref{subsec::forward_reverse_sle_kappa_rho} we combine the results of Section~\ref{subsec::reversing_markov} with Section~\ref{subsec::bessel_processes} to establish the interesting forward/reverse symmetries for $\SLE_\kappa(\rho)$ that we mentioned above.

\subsection{Reversing processes using local time}
\label{subsec::reversing_markov}

Recall that a Feller process \cite[Chapter~17]{KAL_FOUND} with semigroup $P^t$ is said to be reversible with respect to a measure $\mu$ if for all bounded and continuous functions $f,g$ we have that
\begin{equation}
\label{eqn::kernel_reversible}
\int f P^t(g) d\mu = \int g P^t(f)d\mu.
\end{equation}
Since we are motivated by the study of radial and chordal $\SLE$ variants, we will be particularly interested in the case that the state space $\CX$ is either the line $\R$ or unit circle $\s^1$.  If this is the case {\em and} $P^t$ admits continuous transition densities $p^t(x,y)$ with respect to Lebesgue measure {\em and} $\mu$ has a continuous density $\pi$ with respect to Lebesgue measure, then we note that~\eqref{eqn::kernel_reversible} is equivalent to
\begin{equation}
\label{eqn::detailed_balance}
\pi(x) p^t(x,y) = \pi(y) p^t(y,x) \quad\text{for all}\quad x,y \in \CX.
\end{equation}

The main result of this section is the following.

\begin{proposition}
\label{prop::markov_local_time_reversal}
Suppose that $X$ is a continuous Feller process on $\CX \in \{\R, \s^1\}$ with a family of continuous transition densities $p^t(x,y)$.  Assume that $X$ is also a semimartingale and is reversible with respect to a measure $\mu$ which has a continuous density with respect to Lebesgue measure on $\CX$.  Fix $a \in \CX$, let $\ell$ be the local time of $X$ at $a$, and assume that $\p[\ell_\infty = \infty] =1$.  For each $u > 0$, let $T_u = \inf\{t > 0 : \ell_t > u\}$ be the right-continuous inverse of $\ell$.  Then we have for each $u > 0$ that
\[  ( X_{T_u - t} : t \in [0,T_u]) \stackrel{d}{=} (X_t : t \in [0,T_u]).\]
\end{proposition}

The reason that we assume that $X$ is a semimartingale is that it implies \cite[Proposition~19.14]{KAL_FOUND} that the structure of the set $\{t : X_t = a\}$ is such that the It\^o excursion decomposition (stated as Theorem~\ref{thm::ito_excursion_decomposition} below) applies.  Proposition~\ref{prop::markov_local_time_reversal} is not stated in maximal generality; the hypotheses above both fit our setting and keep the proof simple.

Assume that we have the same hypotheses as in Proposition~\ref{prop::markov_local_time_reversal}.  For each $\epsilon \geq 0$, let $\CE(\epsilon)$ be the set of continuous functions $\phi \colon [0,T] \to \CX$ with $T \geq \epsilon$ and $\phi(0) = \phi(T) = a$.  Suppose that $X$ is as in Proposition~\ref{prop::markov_local_time_reversal} and let $\nu$ be the It\^o excursion measure associated with the excursions that $X$ makes from $a$.  That is, $\nu$ is the $\sigma$-finite measure on $\CE = \CE(0)$ with the following two properties:
\begin{enumerate}
\item For each $\epsilon > 0$, $\nu(\CE(\epsilon)) < \infty$ and
\item The law of the sequence of excursions $(X^{j,\epsilon})$ that $X$ makes from $a$ with length at least $\epsilon$ (with each such excursion shifted in time so the excursions start at time $0$) is distributed i.i.d.\ according to $\nu(\cdot \giv \CE(\epsilon))$.
\end{enumerate}
(See \cite[Chapter~19]{KAL_FOUND} for additional introduction to the It\^o excursion measure.)  Let $\CE(0;X)$ be the set of all excursions that $X$ makes from $a$ (with time shifted to start from time $0$).  For each excursion $e \in \CE(0;X)$, we let $\ell(e)$ be the value of $\ell$ at the start time of the excursion under $X$.  We now state the It\^o excursion decomposition for $X$.  (See also \cite[Theorem~19.11 and Proposition~19.14]{KAL_FOUND}.)  This is the first ingredient in the proof of Proposition~\ref{prop::markov_local_time_reversal}.  Recall that we will use the abbreviation \ppp\ for Poisson point process.

\begin{theorem}
\label{thm::ito_excursion_decomposition}
Let $du$ be Lebesgue measure on $\R_+$ and let $\nu$ be the It\^o excursion measure of $X$.  Then the process $\{ (\ell(e),e) : e \in \CE(0;X)\}$ is distributed as a \ppp\ on $\R_+ \times \CE$ with intensity measure $du \otimes \nu$.
\end{theorem}

The equivalence of the semimartingale local time and excursion local time established in \cite[Proposition~19.14]{KAL_FOUND} is stated for processes taking values in $\R$, however the proof works verbatim for processes taking values in $\s^1$.

We emphasize that the It\^o excursion decomposition itself does not require that $X$ satisfy all of the assumptions of Proposition~\ref{prop::markov_local_time_reversal}.  These assumptions are only needed for our proof of Proposition~\ref{prop::markov_local_time_reversal}.  The next important ingredient in the proof of Proposition~\ref{prop::markov_local_time_reversal} is the following lemma regarding the invariance of the It\^o excursion measure under time-reversal.

\begin{lemma}
\label{lem::ito_excursion_reversible}
Assume that $X$ is as in the statement of Proposition~\ref{prop::markov_local_time_reversal}.  Let $\nu$ be the It\^o excursion measure of $X$ associated with the excursions that $X$ makes from~$a$.  Then $\nu$ is invariant under time-reversal.  That is, if $\epsilon > 0$ and $Y$ is sampled according to $\nu(\cdot \giv \CE(\epsilon))$ and $T$ is the length of $Y$ then
\begin{equation}
\label{eqn::y_t_reversible}
(Y_t : t \in [0,T]) \stackrel{d}{=} (Y_{T-t} : t \in [0,T]).
\end{equation}
\end{lemma}
\begin{proof}
Let $\pi$ be the density of $\mu$ with respect to Lebesgue measure on $\CX$.  By reversibility, for all $t > 0$ we then have that (recall~\eqref{eqn::detailed_balance})
\begin{equation}
\label{eqn::detailed_balance2}
p^t(x,y) = \frac{\pi(y)}{\pi(x)}p^t(y,x) \quad\text{for all}\quad x,y \in \CX \quad\text{with}\quad \pi(x),\pi(y) > 0.
\end{equation}
Iterating~\eqref{eqn::detailed_balance2}, for any $x_1,\ldots,x_{k+1} \in \CX$ such that $\pi(x_j) > 0$ for each $1 \leq j \leq k+1$ and for any $t_1,\ldots,t_k > 0$ we have that
\begin{align*}
      \prod_{j=1}^k p^{t_j}(x_j,x_{j+1})
&= \prod_{j=1}^k \frac{\pi(x_{j+1})}{\pi(x_j)} p^{t_j}(x_{j+1},x_j)
  = \frac{\pi(x_{k+1})}{\pi(x_1)} \prod_{j=1}^k p^{t_j}(x_{j+1},x_j).
\end{align*}
In particular, if $x_1 = x_{k+1}$, then
\begin{align*}
    \prod_{j=1}^k p^{t_j}(x_j,x_{j+1}) = \prod_{j=1}^k p^{t_j}(x_{j+1},x_j).
\end{align*}
This implies that the conditional law of $X$ given $X_0 = a$ and $X_t = a$ for $t > 0$ fixed is invariant under time-reversal.

Fix $\epsilon > 0$, sample $Y$ from $\nu(\cdot \giv \CE(\epsilon))$, and let $T$ be the length of $Y$.  We will now show that~\eqref{eqn::y_t_reversible} holds.  For each $\delta > 0$, let $\tau_\delta = \inf\{ s \geq 0: |Y_s-a| = \delta\}$ and let $\sigma_\delta = \sup\{s \leq T : |Y_t-a| = \delta\}$.  Let $S_\delta = \sigma_\delta - \tau_\delta$ and $Y_t^\delta = Y_{t+\tau_\delta}$ for $t \in [0,S_\delta]$.  Note that $\tau_\delta \downarrow 0$ and $\sigma_\delta \uparrow T$ as $\delta \to 0$ by continuity and since $Y_t \neq a$ for $t \in (0,T)$.  Consequently, it suffices to show for each $\delta > 0$ that
\begin{equation}
\label{eqn::y_t_delta_reversible}
(Y_t^\delta : t \in [0,S_\delta]) \stackrel{d}{=} (Y_{S_\delta -t}^\delta : t \in [0,S_\delta]).
\end{equation}
Assume for simplicity that $\CX = \R$; the case that $\CX = \s^1$ is analogous.  Given $S_\delta$ and that $Y_t > a$ for $t \in (0,T)$, we note that we can sample from the conditional law of $Y^\delta$ by generating a sample of $X$ with $X_0 = a+\delta$ and $X_{S_\delta} = a+\delta$ conditioned on the positive probability event that $X_t \neq a$ for $t \in [0,S_\delta]$.  Since this event is invariant under time-reversal and, as explained earlier, the law of a bridge of $X$ from $a+\delta$ to $a+\delta$ of length $S_\delta$ is also invariant under time-reversal, we get that~\eqref{eqn::y_t_delta_reversible} holds in this case.  The case that $Y_t < a$ for $t \in (0,T)$ is analogous.  Therefore~\eqref{eqn::y_t_reversible} follows.
\end{proof}

\begin{proof}[Proof of Proposition~\ref{prop::markov_local_time_reversal}]
We note that we can generate a sample from the law of $X_{T_u-t}$ by sampling first the \ppp\ of excursions associated with $X$ as in Theorem~\ref{thm::ito_excursion_decomposition}, and then running the excursions backwards starting at local time $u$ until reaching local time $0$.  The result follows since Lebesgue measure is invariant under reflection and, by Lemma~\ref{lem::ito_excursion_reversible}, the It\^o excursion law for $X$ is invariant under time-reversal.
\end{proof}

\subsection{Bessel processes}
\label{subsec::bessel_processes}

In this section, we will recall a few facts about Bessel processes which will play an important role in this article.  We direct the reader to \cite[Chapter~XI]{ry99martingales} for an in-depth overview of Bessel processes.

Fix $\delta \in \R$ and $x \geq 0$.  The squared $\delta$-dimensional Bessel process starting from $x^2$ is given by the unique strong solution to the SDE
\begin{equation}
\label{eqn::square_bessel_sde}
dZ_t = 2\sqrt{Z_t} dB_t + \delta dt,\quad Z_0 = x^2
\end{equation}
where $B$ is a standard Brownian motion.  Following \cite{ry99martingales}, we will denote this process by $\besq^\delta$ (and suppress the dependency on the starting point).  When $\delta > 0$, the solution to~\eqref{eqn::square_bessel_sde} exists for all $t \geq 0$, while $0$ is an absorbing state for $\delta \leq 0$.  The $\delta$-dimensional Bessel process $X_t$ starting from $x$ is given by $\sqrt{Z_t}$ where $Z_t$ is as in~\eqref{eqn::square_bessel_sde}.  Following \cite{ry99martingales}, we will denote this process by $\bes^\delta$ (and suppress the dependency on the starting point).  An application of It\^o's formula shows that $X_t$ satisfies the SDE
\begin{equation}
\label{eqn::bessel_sde}
dX_t = dB_t + \frac{\delta-1}{2} \cdot \frac{1}{X_t} dt,\quad X_0 = x
\end{equation}
at least at those times when $X_t \neq 0$.

There are three important ranges of~$\delta$ values which determine how~$X$ interacts with~$0$:
\begin{enumerate}
\item If $\delta \geq 2$, then $X_t$ a.s.\ does not hit $0$ except if $X_0 = 0$ in which case $X_t \neq 0$ for all $t > 0$, so $X_t$ in fact solves~\eqref{eqn::bessel_sde} for all times and is a semimartingale.
\item If $\delta \in (1,2)$, then $X_t$ hits $0$ in finite time a.s.  Nevertheless, $X_t$ satisfies~\eqref{eqn::bessel_sde} in the integrated sense for all $t$ and is a semimartingale.
\item If $\delta \in (0,1]$, then $X_t$ hits $0$ in finite time a.s.  In order to view $X_t$ as a solution of~\eqref{eqn::bessel_sde} for those times when $X$ is interacting with $0$ it is necessary to introduce a principle value correction (see \cite[Exercise~1.26, Chapter~XI]{ry99martingales}).  For $\delta =1$, $X_t$ is a semimartingale while for $\delta \in (0,1)$ it is not.
\end{enumerate}

In the case that $\delta \in (0,2)$, $X_t$ is instantaneously reflecting at $0$.  The interaction of a $\bes^\delta$ with $\delta \in (0,1]$ with $0$ will not play a role in this article, so we will not delve into the technicalities associated with this regime.

We now show that a Bessel process can be viewed as a time-changed geometric Brownian motion (see \cite[Exercise~1.28, Chapter~XI]{ry99martingales}).

\begin{proposition}
\label{prop::bessel_exponential_bm}
Let $X_t$ be a $\bes^\delta$ with $X_0 > 0$.  Then $\log X_t$ reparameterized by its quadratic variation (and stopped the first time $X_t$ reaches $0$ or time reaches $\infty$) is equal in distribution to $B_t + \tfrac{\delta-2}{2} t$ where $B$ is a standard Brownian motion with $B_0 = \log X_0$.  Conversely, if $B$ is a standard Brownian motion and $a \in \R$, then $X_t = \exp(B_t + at)$ reparameterized by its quadratic variation evolves as a $\bes^\delta$ with $\delta = 2a+2$. 
\end{proposition}
\begin{proof}
Let $X_t$ be a $\bes^\delta$.  An application of It\^o's formula yields that
\begin{equation}
\label{eqn::log_of_bessel}
d \log X_t = \frac{1}{X_t}dB_t + \frac{\delta-2}{2X_t^2}dt.
\end{equation}
Reparameterizing $\log X_t$ by its quadratic variation $ds = X_t^{-2} dt$ thus yields a standard Brownian motion plus a linear drift term $\frac{\delta-2}{2} s$.  The converse statement follows from an analogous argument.
\end{proof}

It is also natural to consider Proposition~\ref{prop::bessel_exponential_bm} in the context of a $\bes^\delta$ process $X_t$ with $X_0 = 0$.  In this case, one can fix $\epsilon > 0$, take $\tau = \inf\{t \geq 0 : X_t = \epsilon\}$, and then apply Proposition~\ref{prop::bessel_exponential_bm} to the process $X_{t+\tau}$ to obtain a Brownian motion with drift $Y_t = B_t + \tfrac{\delta-2}{2} t$ starting from $\log \epsilon$.  On the event that $\sigma = \inf\{t \geq 0 : Y_t \geq 0\} < \infty$, it is natural to consider the process $Y_{\sigma+t}$ which is defined on the interval $[-\sigma,\infty)$ and takes the value $0$ at $t = 0$.  Taking a limit as $\epsilon \to 0$, one obtains a process which is defined on all of $\R$, takes the value $0$ at $t=0$, and evolves as a Brownian motion with drift $\tfrac{\delta-2}{2} t$ for $t \geq 0$.  The choice of the hitting threshold in the definition of $\sigma$ is natural if $\delta \geq 2$ because then $\sigma < \infty$ a.s., but any other hitting threshold would work equally well.

Proposition~\ref{prop::bessel_exponential_bm} leads to the following important fact regarding the time-reversal of a Bessel process.  This result is originally due to Williams \cite[Theorem~2.5]{w1974pathdecomp}; see also \cite[Chapter~XI, Exercise~1.23]{ry99martingales}.

\begin{proposition}
\label{prop::bessel_time_reversal}
Fix $\delta < 2$ and $x > 0$.  Suppose that $X_t$ is a $\bes^\delta$ started at $x > 0$ and let $\tau = \inf\{t \geq 0 : X_t = 0\}$.  Let $\wt{X}_t$ be a $\bes^{\wt{\delta}}$ started at $0$ with
\begin{equation}
 \label{eqn::wt_delta}
 \wt{\delta} = 4-\delta > 2
  \end{equation}
and let $\wt{\tau} = \sup\{t \geq  0 : \wt{X}_t = x\}$ be the last time that $\wt{X}$ hits $x$ (note $\wt{\tau} < \infty$ a.s.\ since $\wt{\delta} > 2$).  Then $t \mapsto X_{\tau-t}$ for $t \in [0,\tau]$ has the same law as $t \mapsto \wt{X}_t$ for $t \in [0,\wt{\tau}]$.
\end{proposition}
Note that $\delta < 2$ implies $\wt{\delta} > 2$ and that a $\bes^{\wt{\delta}}$ process a.s.\ does not hit~$0$.
\begin{proof}[Proof of Proposition~\ref{prop::bessel_time_reversal}]
Proposition~\ref{prop::bessel_exponential_bm} implies that the time-reversal $\log X_{\tau-t}$ when parameterized by its quadratic variation evolves as a Brownian motion with linear drift $-\tfrac{\delta-2}{2}$.  Thus $X_{\tau-t}$ evolves as a $\bes^{\wt{\delta}}$ where $\wt{\delta}$ satisfies $-\tfrac{\delta-2}{2} = \tfrac{\wt{\delta}-2}{2}$.  Solving for $\wt{\delta}$ yields~\eqref{eqn::wt_delta}, as desired.
\end{proof}

We will explain in Section~\ref{subsec::forward_reverse_sle_kappa_rho} how we can use the results of Section~\ref{subsec::reversing_markov} to time-reverse a $\bes^\delta$ with $\delta \in (0,2)$ stopped when its local time at $0$ first reaches a given value.

As a second application of Proposition~\ref{prop::bessel_exponential_bm}, we will give a description of the law of a standard Brownian motion with negative linear drift conditioned to exceed a given positive value.

\begin{lemma}
\label{lem::condition_bm_negative_drift_large}
Let $B_t$ be a standard Brownian motion starting from $0$, fix $a < 0$, and let $X_t = B_t + at$.  Fix $C > 0$ and let $E_C$ be the event $\{ \sup_{t \geq 0} X_t \geq C\}$.  Then the conditional law of $X$ given $E_C$ can be sampled from as follows.
\begin{enumerate}
\item Sample a standard Brownian motion $X^1$ starting from $0$ with linear drift $-a > 0$ and let $\tau$ be the first time that $X^1$ hits $C$.
\item Sample a standard Brownian motion $X^2$ starting from $C$ with linear drift $a < 0$.
\item Concatenate the processes $X^1|_{[0,\tau]}$ and $X^2(\cdot-\tau)$.
\end{enumerate}
\end{lemma}
\begin{proof} 
There are various ways to prove this, but we find it instructive to use the Bessel process correspondence of Proposition~\ref{prop::bessel_exponential_bm}.  (The connection to Bessel processes is actually our motivation for considering this lemma in the first place.)  Let $Z$ be given by $e^{X_t}$ with time reparameterized by quadratic variation so that $d \langle Z \rangle_t = dt$.  By Proposition~\ref{prop::bessel_exponential_bm}, $Z$ is a $\bes^\delta$ with $\delta = 2a+2$.  Let $\tau_C$ (resp.\ $\tau_0$) be the first time that $Z$ hits $e^C$ (resp.\ $0$).  Then $E_C = \{\tau_C < \tau_0\}$.  Note that $Z_t^{2-\delta}$ is a local martingale for $Z_t$ and that $Z_{\tau_C \wedge \tau_0}^{2-\delta} = e^{(2-\delta)C}\one_{E_C}$.  By the Girsanov theorem \cite{ks91bm,ry99martingales}, weighting the law of $Z_t$ by $Z_t^{2-\delta}$ yields the law of a $\bes^{4-\delta}$; note that $4-\delta=2-2a > 2$.  It thus follows that the law of $Z$ conditional on $E_C$ can be sampled from as follows.
\begin{enumerate}
\item Sample a Bessel process $Z^1$ of dimension $2-2a$ and let $\tau$ be the first time that $Z^1$ hits $e^C$
\item Sample a Bessel process $Z^2$ of dimension $2+2a$ starting from $e^C$
\item Concatenate the processes $Z^1|_{[0,\tau]}$ and $Z^2(\cdot-\tau)$
\end{enumerate}
This proves the result because we can sample from the law of $X$ given $E_C$ by sampling from the law of $Z$ given $E_C$ and then taking $X$ to be $\log Z$ reparameterized according to its quadratic variation.
\end{proof}

\begin{remark}
\label{rem::bessel_ito_excursion}
Fix $\delta \in (0,2)$ and let $\nu_\delta^\bes$ denote the It\^o excursion measure associated with a $\bes^\delta$ process.  We can express $\nu_\delta$ as follows:
\begin{enumerate}
\item Pick a ``lifetime'' $t$ according to the $\sigma$-finite measure
\begin{equation}
\label{eqn::bessel_ito_lifetime}
c_\delta t^{\delta/2-2} dt
\end{equation}
where $c_\delta$ is a constant which depends only on $\delta$ and $dt$ denotes Lebesgue measure on $\R_+$, and then
\item Run a $\bes^\delta$ excursion of length $t$ from $0$ to $0$.
\end{enumerate}
(See \cite{py82bessel} as well as the text just after \cite[Theorem~1]{py96maximum}.)

For each $e \in \CE$, we let $e^*$ be the maximum value taken on by $e$.  Suppose that $\Lambda$ is a \ppp\ chosen from $du \otimes \nu_\delta^\bes$.  Let $\Lambda^* = \{(u,e^*) : (u,e) \in \Lambda\}$.  Then, as explained just after \cite[Theorem~1]{py96maximum}, $\Lambda^*$ is a \ppp\ on $\R_+ \times \R_+$ with intensity measure $du \otimes \nu_\delta^*$ where $\nu_\delta^*(dt) = c_\delta^* t^{\delta-3} dt$ where $dt$ is Lebesgue measure on $\R_+$ and $c_\delta^* > 0$ is a constant.  Moreover, we can use $\nu_\delta^*$ to give a second description of It\^o's excursion law for a $\bes^\delta$ process as follows:
\begin{enumerate}
\item Pick $t$ from the measure $\nu_\delta^*$, and then
\item Join back to back two independent $\bes^{4-\delta}$ processes starting from $0$ and each run until the first time that it hits $t$.
\end{enumerate}

Note that, although $0$ is an absorbing state for a $\bes^\delta$ process with $\delta \leq 0$, we can still make sense of the It\^o excursion measure $\nu_\delta^\bes$ for the ``excursions'' that such a process makes from $0$ by extending the definition described above.  It is just not possible to concatenate the excursions from a \ppp\ $\Lambda$ on $\R_+ \times \CE$ with intensity measure $du \otimes \nu_\delta^\bes$ with $\delta \leq 0$ to form a continuous process as it is a.s.\ the case that for each $u_0 > 0$, the sum of the lengths of the excursions $(u,e) \in \Lambda$ with $u \leq u_0$ is infinite.
\end{remark}

\subsection{$\SLE_\kappa(\rho)$ processes}
\label{subsec::sle_kappa_rho}

\subsubsection{Chordal $\SLE_\kappa(\rho)$}

Chordal $\SLE_\kappa(\rho)$ processes are very natural variants of $\SLE_\kappa$, whose laws depend on the locations of {\bf force points} along the boundary or interior of the domain; such processes arise naturally in many statistical physical settings involving different types of boundary conditions.  We will now review the definitions of forward and reverse $\SLE_\kappa(\rho)$.  In many places in this article, we will be considering simultaneously both forward and reverse Loewner flows.  To keep the direction of time clear, we will typically use a tilde to indicate the latter.

Fix $\rho_1,\ldots,\rho_n \in \R$ and $x_1,\ldots,x_n \in \ol{\H}$.  Recall from \cite{sw2005sle_coordinate_changes} that the SDE which drives a forward $\SLE_\kappa(\ul{\rho})$ process is given by
\begin{align}
\label{eqn::forward_sle_kappa_rho}
\begin{split}
dW_t &= \sum_{i=1}^n \Re\left(\frac{-\rho_i}{f_t(x_i)}\right) dt + \sqrt{\kappa}dB_t,\\
 df_t(x_i) &= \frac{2}{f_t(x_i)} dt - dW_t, \quad f_0(x_i) = x_i \quad\text{for}\quad i=1,\ldots,n.
\end{split}
\end{align}

For $\wt{\rho}_1,\ldots,\wt{\rho}_n \in \R$ and $\wt{x}_1,\ldots,\wt{x}_n \in \ol{\h}$, the SDE for the driving process of a reverse $\SLE_\kappa(\ul{\wt{\rho}})$ is given by
\begin{align}
\label{eqn::reverse_sle_kappa_rho}
\begin{split}
d \wt{W}_t &= \sum_{i=1}^n \Re\left(\frac{-\wt{\rho}_i}{\wt{f}_t(\wt{x}_i)}\right) dt + \sqrt{\kappa}dB_t,\\
 d\wt{f}_t(\wt{x}_i) &= -\frac{2}{\wt{f}_t(\wt{x}_i)} dt - d\wt{W}_t, \quad \wt{f}_0(\wt{x}_i) = \wt{x}_i \quad\text{for}\quad i=1,\ldots,n.
\end{split}
\end{align}
We will often use $V^i$ to denote the processes $f_t(x_i)$ if we do not wish to emphasize the $x_i$.  If there is only one force point, we will often use $V$ to denote the process $f_t(x_1)$.  We will make use of analogous notation in the reverse case.  Note that~\eqref{eqn::forward_sle_kappa_rho} and~\eqref{eqn::reverse_sle_kappa_rho} differ only in that the force points in~\eqref{eqn::forward_sle_kappa_rho} evolve under the forward Loewner flow while in~\eqref{eqn::reverse_sle_kappa_rho} they evolve under the reverse Loewner flow.  In both the forward and reverse evolutions, under our convention, a positive $\rho_i$ corresponds to a force point that pushes $W_t$ away from itself.

The existence and uniqueness of solutions to~\eqref{eqn::forward_sle_kappa_rho} is discussed in detail in \cite[Section~2]{ms2012imag1} in the case that there are only boundary force points.  In particular, a precise notion of a solution to~\eqref{eqn::forward_sle_kappa_rho} is introduced in \cite[Definition~2.1]{ms2012imag1} and it is shown in \cite[Theorem~2.1]{ms2012imag1} that there exists a unique solution to~\eqref{eqn::forward_sle_kappa_rho} up until the {\bf continuation threshold} is hit.  This is the first time $t$ that the sum of the weights of the force points which are immediately to the left of $W_t$ is not more than $-2$ or the sum of the weights of the force points which are immediately to the right of $W_t$ is not more than $-2$.  Combining this result with the Girsanov theorem \cite{ks91bm,ry99martingales}, one gets existence and uniqueness of solutions to~\eqref{eqn::forward_sle_kappa_rho} (with interior force points) until the first time that either the continuation threshold is hit or the imaginary part of one of the interior force points is equal to~$0$.

The arguments given in \cite[Section~2]{ms2012imag1} also lead to existence and uniqueness results for~\eqref{eqn::reverse_sle_kappa_rho} in the case of multiple boundary force points.  In this article, it will be important for us to consider~\eqref{eqn::reverse_sle_kappa_rho} in the case of a force point starting infinitesimally above $0$.  We will establish the existence and uniqueness of such solutions in the following proposition.

\begin{proposition}
\label{prop::reverse_zip_in_solution}
Suppose that $\wt{\rho} < \tfrac{\kappa}{2}+4$.  There exists a unique law on pairs $(\wt{W},\wt{Z})$ of continuous processes such that $\wt{W}_0 = \wt{Z}_0 = 0$, $\p[\im(\wt{Z}_t) >0] =1$ for all $t > 0$, and $(\wt{W},\wt{Z})$ solves~\eqref{eqn::reverse_sle_kappa_rho} for all $t > 0$.
\end{proposition}

In the statement of Proposition~\ref{prop::reverse_zip_in_solution}, $\wt{Z}_t$ plays the role of $\wt{f}_t(x)$ in~\eqref{eqn::reverse_sle_kappa_rho}.

\begin{proof}[Proof of Proposition~\ref{prop::reverse_zip_in_solution}]
Suppose that $z \in \H$.  Then the time evolution $\wt{Z}$ of the force point associated with a centered, reverse $\SLE_\kappa(\wt{\rho})$ with a single force point of weight $\wt{\rho}$ starting from $z \in \ol{\h}$ is given by
\begin{equation}
\label{eqn::singleforce_interior}
d \wt{Z}_t= -\frac{2}{\wt{Z}_t}dt - d \wt{W}_t = -\frac{2}{\wt{Z}_t}dt + \Re \frac{\wt{\rho}}{\wt{Z}_t} dt  - \sqrt{\kappa} dB_t.
\end{equation}
Let $\wt{\theta}_t = \arg \wt{Z}_t$ and $\wt{I}_t = \log \Im \wt{Z}_t$.  Now,
\[ d \Im \wt{Z}_t = -\Im \frac{2}{\wt{Z}_t}dt,\]
so
\[ d \wt{I}_t = d\log \Im \wt{Z}_t = -\frac{1}{\Im \wt{Z}_t} \Im \frac{2}{\wt{Z}_t}dt = \frac{2}{|\wt{Z}_t|^2} dt.\]
Also,
\begin{align*}
 d \log \wt{Z}_t &= - \frac{\sqrt\kappa}{\wt{Z}_t}dB_t + \frac{1}{\wt{Z}_t}\Re\frac{\wt{\rho}}{\wt{Z}_t}dt -\frac{2+\kappa/2}{\wt{Z}_t^2}dt,\\
d \wt{\theta}_t &= - \Im \frac{\sqrt\kappa}{\wt{Z}_t}dB_t + \Im \frac{1}{\wt{Z}_t}\Re\frac{\wt{\rho}}{\wt{Z}_t}dt -\Im \frac{2+\kappa/2}{\wt{Z}_t^2}dt.
\end{align*}

If we reparameterize time by setting $ds = |\wt{Z}_t|^{-2}dt$, then we have $d \wt{I}_{t(s)} = 2ds$ and for a Brownian motion $\wh{B}_s$ we have $d \wh{B}_s := d B_t/|\wt{Z}_t|$.  We then have that,
\begin{align}
d \wt{\theta}_{t(s)} &= \sqrt{\kappa} \sin ( \wt{\theta}_{t(s)} ) d\wh{B}_s + \left(2 + \frac{\kappa}{2} - \frac{\wt{\rho}}{2}\right) \sin(2 \wt{\theta}_{t(s)})ds. \label{eqn::theta_reverse}
\end{align}
It is easy to see from that if we take $z=i\epsilon$ then the law of $(\wt{I}_{t(s)},\wt{\theta}_{t(s)})$ converges weakly with respect to the topology of local uniform convergence as $\epsilon \to 0$.  Indeed, to see this we let $\mu$ be the measure on $[0,\pi]$ given by
\begin{equation}
\label{eqn::theta_reverse_measure} 
 \mu(d\theta) = \sin^a(\theta) d\theta \quad\text{where}\quad a = \frac{8-2\wt{\rho}}{\kappa}
\end{equation}
and $d\theta$ denotes Lebesgue measure.  By noting that the generator for SDE~\eqref{eqn::theta_reverse} is self-adjoint with respect to the space $L^2([0,\pi],\mu)$, it follows that $\mu$ gives a reversible measure for~\eqref{eqn::theta_reverse}.  (The equivalence between reversibility and the generator being self-adjoint is given, for example, in \cite[Proposition~5.3]{LIG_PARTICLES}.)  Note, in particular, that $a > -1$ provided $\wt{\rho} < \tfrac{\kappa}{2}+4$.  In this case, $\mu$ (after normalization) gives the unique stationary distribution to~\eqref{eqn::theta_reverse} to which the law of any solution eventually converges.  Indeed, note that if we set $dr = \kappa \sin^2 \wt{\theta}_{t(s)} ds$, then
\[ d \wt{\theta}_{t(r)} = d \breve{B}_r + \left(\frac{4 + \kappa - \wt{\rho}}{\kappa} \right)\cot(\wt{\theta}_{t(r)}) dr\]
where $\breve{B}$ is a standard Brownian motion.  This process behaves like a Bessel process of dimension $1+(8+2\kappa-2\wt{\rho})/\kappa > 2$ when it is near $0$, hence a.s.\ does not hit $0$.  For the same reason, it a.s.\ does not hit $\pi$.  It therefore follows that in each unit of time, $\wt{\theta}_{t(r)}$ has a positive chance of entering the interval $[\pi/2,\pi)$ uniformly in its starting position in the interval $(0,\pi/2]$.  If $X \sim \bes^\delta$ for $\delta > 2$ with $X_0 = 0$, then it follows from the explicit form of the transition density for the square of a $\bes^\delta$ process starting from $0$ \cite[Chapter~XI, Corollary~1.4]{ry99martingales} that $\int_0^t X_s^{-2} ds$ has a finite expectation.  It therefore follows that in each unit of time, $\wt{\theta}_{t(s)}$ has a positive chance of entering the interval $[\pi/2,\pi)$ uniformly in its starting position in the interval $(0,\pi/2]$.  Consequently, two solutions to~\eqref{eqn::theta_reverse} have a uniformly positive chance of coalescing in each unit of time, regardless of where they start.  This proves the existence component of the proposition since to sample $(\wt{W},\wt{Z})$ we can first sample the infinite time solution $(\wt{I}_{t(s)},\wt{\theta}_{t(s)})$ to the dynamics above and then generate $(\wt{W},\wt{Z})$ from $(\wt{I}_{t(s)},\wt{\theta}_{t(s)})$ by inverting the time change.  Uniqueness follows similarly (we can couple any two solutions $(\wt{W}^1,\wt{Z}^1)$ and $(\wt{W}^2,\wt{Z}^2)$ together onto a common probability space so that in each unit of $t(s)$ time there is a uniformly positive chance that the solutions will coalesce and then stay together). 
\end{proof}

Suppose that $\wt{W}$ is a solution to~\eqref{eqn::reverse_sle_kappa_rho}.  Then the centered reverse chordal $\SLE_\kappa(\ul{\rho})$ process is given by the family of conformal maps $(\wt{f}_t)$ which solve the SDE:
\begin{equation}
\label{eqn::reverse_sle_kappa_rho_bis}
 d \wt{f}_t(z) = -\frac{2}{\wt{f}_t(z)} dt - d \wt{W}_t, \quad \wt{f}_0(z) = z.
\end{equation}
In some cases, it will be convenient to refer to the (uncentered) reverse $\SLE_\kappa(\ul{\rho})$ Loewner flow.  This is given by the family of conformal maps $(\wt{g}_t)$ which solve the ODE
\begin{equation}
\label{eqn::reverse_sle_kappa_rho_loewner}
 d \wt{g}_t(z) = -\frac{2}{\wt{g}_t(z) - \wt{W}_t} dt, \quad \wt{g}_0(z) = z.
\end{equation}
That is, $\wt{g}_t = \wt{f}_t + \wt{W}_t$.  The force points for the uncentered flow have the same relationship with the force points of the centered flow.

We remark that in the centered (resp.\ uncentered) forward chordal $\SLE_\kappa(\ul{\rho})$ evolution is defined in the same way except the minus sign before $2/\wt{f}_t(z)$ (resp.\ $2/(\wt{g}_t(z)-\wt{W}_t)$) in~\eqref{eqn::reverse_sle_kappa_rho_bis} (resp.\ \eqref{eqn::reverse_sle_kappa_rho_loewner}) is not present in the forward case.

Returning to the setting of reverse $\SLE_\kappa(\ul{\rho})$ Loewner flow, in the special case of a single force point $x \in \R$,~\eqref{eqn::reverse_sle_kappa_rho_bis} evaluated at $x$ takes the form
\begin{equation}
\label{eqn::singleforce}
d \wt{f}_t(x) = \frac{-2}{\wt{f}_t(x)}dt - d\wt{W}_t = \frac{-2 + \wt{\rho}}{\wt{f}_t(x)}dt  - \sqrt{\kappa}dB_t.
\end{equation}
Comparing with~\eqref{eqn::bessel_sde}, we see that $\wt{f}_t(x)/\sqrt{\kappa}$ evolves as a $\bes^{\wt{\delta}}$ where $(\wt{\delta}-1)/2 = (\wt{\rho}-2)/\kappa$, i.e.,
\begin{equation}
\label{eqn::besseldim_reverse}
\wt{\delta} = 1 + \frac{2(\wt{\rho}-2)}{\kappa}.
\end{equation}
In particular, if we set $\wt{\tau} = \inf\{t \geq 0 : \wt{f}_t(x) = 0\}$ then $\wt{\tau} < \infty$ a.s.\ if and only if $\wt{\rho} < \tfrac{\wt{\kappa}}{2}+2$ so that $\wt{\delta} < 2$.

In this case, there are two different ways of extending the solution of such a Loewner flow beyond time $\wt{\tau}$ depending on whether we would like the force point to be pushed into $\h$ (i.e., the interior of the domain) or to stay in $\partial \h$ and reflect off $0$.  We will now describe these two possibilities.
\begin{itemize}
\item Suppose that $\wt{\rho} < \tfrac{\wt{\kappa}}{2} - 2$.  Then we can define the solution after time $\wt{\tau}$ to be given by the one constructed in Proposition~\ref{prop::reverse_zip_in_solution}.  That is, if we let $(\wt{h}_t)$ be a reverse Loewner flow as constructed in Proposition~\ref{prop::reverse_zip_in_solution} which is taken to be independent of $(\wt{f}_t)_{t \in [0,\wt{\tau}]}$, for $t > \wt{\tau}$ we set $\wt{f}_t = \wt{h}_{t - \wt{\tau}} \circ \wt{f}_{\wt{\tau}}$.
\item Suppose that $-2-\tfrac{\kappa}{2} < \wt{\rho} < \tfrac{\wt{\kappa}}{2} - 2$ so that $\wt{\delta} \in (0,2)$.  Then we can also define the solution after time $\wt{\tau}$ by requiring that the $\wt{\delta}$-dimensional Bessel process $\wt{f}_t(x)/\sqrt{\kappa}$ evolves as a Bessel process (reflecting off the value $0$ and staying in $\R$) even after time $\wt{\tau}$; in other words, the process $\wt{f}_t(x)$ is constructed from $B_t$ by solving the SDE in~\eqref{eqn::singleforce} in the usual way), which in turn determines $\wt{W}_t$ and hence $\wt{f}_t$ for all $t \geq 0$.
\end{itemize}
We will consider both types of continuations in this article.

Similarly, if $(f_t)$ corresponds to a centered forward Loewner flow with a single force point of weight $\rho$ located at $x$, then $f_t(x)/\sqrt{\kappa}$ evolves as a $\bes^\delta$ where
\begin{equation}
\label{eqn::besseldim_forward}
\delta = 1 + \frac{2(\rho+2)}{\kappa}.
\end{equation}
The reason for the difference from~\eqref{eqn::besseldim_reverse} can be seen by considering the case $\rho = 0$.  In the reverse process, the Loewner drift is pulling $f_t(x)$ toward the origin, while in the forward process the Loewner drift is pushing $f_t(x)$ away from the origin.  In both cases $\rho$ indicates a quantity of additional force pushing $f_t(x)$ away from the origin.

\subsubsection{Radial and whole-plane $\SLE_\kappa(\rho)$}
\label{subsubsec::radial_whole_plane_sle}

Let
\[ \Psi(u,z) = \frac{u+z}{u-z},\quad \Phi(u,z) = z \Psi(u,z), \quad\text{and}\quad \wh{\Phi}(u,z) = \frac{\Phi(u,z) + \Phi(1/\ol{u},z)}{2}.\]
(Note that $\wh{\Phi}(u,z) = \Phi(u,z)$ for $u \in \partial \D$.)  A radial $\SLE_\kappa$ in $\D$ targeted at $0$ is the random growth process $(K_t)$ in $\D$ starting from a point on $\partial \D$ growing towards $0$ which is described by the random family of conformal maps $(g_t)$ which solve the radial Loewner equation:
\begin{equation}
\label{eqn::radial_loewner}
\partial_t g_t(z) = \Phi(U_t,g_t(z)), \quad g_0(z) = z.
\end{equation}
Here, $U_t = e^{i \sqrt{\kappa} B_t}$ where $B_t$ is a standard Brownian motion; $U$ is referred to as the driving function for the radial Loewner evolution.  The set $K_t$ is the complement of the domain of $g_t$ in $\D$ and $g_t$ is the unique conformal transformation $\D \setminus K_t \to \D$ fixing $0$ with $g_t'(0) > 0$.  Time is parameterized by the logarithmic conformal radius as viewed from $0$ so that $\log g_t'(0) = t$ for all $t \geq 0$.

As in the chordal setting, radial $\SLE_\kappa(\rho)$ is a generalization of radial $\SLE$ in which one keeps track of one extra marked point.  We say that a pair of processes $(U,V)$, each of which takes values in~$\s^1$, solves the radial $\SLE_\kappa(\rho)$ equation for $\rho \in \R$ with a single boundary force point of weight~$\rho$ provided that
\begin{equation}
\label{eqn::sle_radial_equation}
\begin{split}
dU_t &= - \frac{\kappa}{2}  U_t dt + i\sqrt{\kappa} U_t dB_t + \frac{\rho}{2} \wh{\Phi}(V_t,U_t) dt\\
dV_t &= \Phi(U_t,V_t) dt.
\end{split}
\end{equation}
A radial $\SLE_\kappa(\rho)$ is the growth process corresponding to the solution $(g_t)$ of~\eqref{eqn::radial_loewner} when $U$ is taken to be as in~\eqref{eqn::sle_radial_equation}.  It is explained in \cite[Section~2.1.2]{ms2013imag4} that~\eqref{eqn::sle_radial_equation} has a unique solution which exists for all time provided $\rho > -2$.

It will often be useful to consider the SDE
\begin{equation}
\label{eqn::theta_equation}
 d\theta_t = \frac{\rho+2}{2} \cot\left(\frac{\theta_t}{2}\right)  dt + \sqrt{\kappa} dB_t
\end{equation}
where $B$ is a standard Brownian motion.  This SDE can be derived formally by taking a solution $(U,V)$ to~\eqref{eqn::sle_radial_equation} and then setting $\theta_t = \arg U_t - \arg V_t$ (see \cite[Equation~4.1]{she2009cle} for the case $\rho=\kappa-6$).

Reverse radial $\SLE_\kappa$ is described in terms of the family of conformal maps $(\wt{g}_t)$ which solve the reverse radial Loewner equation
\begin{equation}
\label{eqn::reverse_radial_loewner}
\partial_t \wt{g}_t(z) = -\Phi(\wt{U}_t,\wt{g}_t(z)), \quad \wt{g}_0(z) = z
\end{equation}
where $\wt{U}_t = e^{i\sqrt{\kappa} B_t}$ and $B$ is a standard Brownian motion.

Reverse radial $\SLE_\kappa(\wt{\rho})$ is a variant of reverse radial $\SLE_\kappa$ in which one keeps track of an extra marked point on $\partial \D$.  It is defined in an analogous way to reverse radial $\SLE_\kappa$ except the driving function $\wt{U}_t$ is taken to be a solution to the SDE:
\begin{align}
\label{eqn::reverse_radial_sle_kappa_rho}
\begin{split}
 d\wt{U}_t &= -\frac{\kappa}{2} \wt{U}_t dt + i \sqrt{\kappa} \wt{U}_t dB_t + \frac{\rho}{2}\wh{\Phi}(\wt{V}_t,\wt{U}_t) dt\\
  d\wt{V}_t &= -\Phi(\wt{U}_t,\wt{V}_t) dt.
  \end{split}
\end{align}
Observe that when $\wt{\rho} = 0$ this is the same as the driving SDE for ordinary reverse radial $\SLE_\kappa$.

Whole-plane $\SLE$ is a variant of $\SLE_\kappa$ which describes a random growth process $K_t$ where, for each $t \in \R$, $K_t \subseteq \C$ is compact with $\C_t = \C \setminus K_t$ simply connected (viewed as a subset of the Riemann sphere).  For each $t$, we let $g_t \colon \C_t \to \C \setminus \D$ be the unique conformal transformation with $g_t(\infty) = \infty$ and $g_t'(\infty) > 0$.  Then $g_t$ solves the whole-plane Loewner equation
\begin{equation}
\label{eqn::whole_plane_loewner}
 \partial_t g_t = \Phi(U_t,g_t(z)), \quad g_0(z) = z.
\end{equation}
Here, $U_t = e^{i\sqrt{\kappa} B_t}$ where $B_t$ is a two-sided standard Brownian motion.  Equivalently, $U$ is given by the time-stationary solution to~\eqref{eqn::sle_radial_equation} with $\rho=0$.  Note that~\eqref{eqn::whole_plane_loewner} is the same as~\eqref{eqn::radial_loewner}.  In fact, for any $s \in \R$, the growth process $1/g_s(K_t \setminus K_s)$ for $t \geq s$ from $\partial \D$ to $0$ is a radial $\SLE_\kappa$ process in $\D$.  Thus, whole-plane $\SLE$ can be thought of as a bi-infinite time version of radial $\SLE$.  Whole-plane $\SLE_\kappa(\rho)$ is the growth process associated with~\eqref{eqn::whole_plane_loewner} where $U$ is taken to be the time-stationary solution of~\eqref{eqn::radial_loewner} (see the discussion in \cite[Section~2]{ms2013imag4} for more on the existence and convergence to the time-stationary solution to~\eqref{eqn::radial_loewner}).

If $(g_t)$ is the Loewner flow associated with a radial $\SLE_\kappa(\ul{\rho})$ or a whole-plane $\SLE_\kappa(\rho)$ with driving function $U_t$, then $f_t = U_t^{-1} g_t$ gives the centered Loewner flow.  The reverse centered radial Loewner flow is defined in the same way but with the reverse Loewner flow in place of the forward Loewner flow.

The continuity of radial and whole-plane $\SLE_\kappa$ for $\kappa \neq 8$ was proved by Rohde and Schramm \cite{rs2005sle} and for $\kappa = 8$ follows by work of Lawler, Schramm, and Werner \cite{lsw2004ust}.  The transience of radial and whole-plane $\SLE_\kappa$ was proved by Lawler \cite{law2013transience}.  The continuity and transience of radial and whole-plane $\SLE_\kappa(\rho)$ for $\rho > -2$ is proved in \cite{ms2013imag4}.

\subsection{Forward/reverse symmetries for $\SLE_\kappa(\rho)$ processes}
\label{subsec::forward_reverse_sle_kappa_rho}

 \begin {figure}[h!]
\begin {center}
\subfloat[\label{fig::forcepoints_origin}Zipping up until the force point reaches the origin.]{\includegraphics [scale=0.85,page=1]{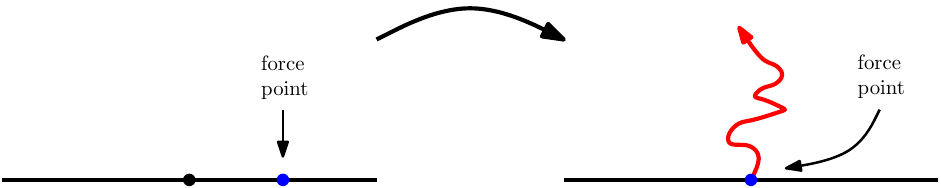}}

\vspace{0.05\textheight}

\subfloat[\label{fig::forcepoints_interior}Zipping a force point at the origin into the interior.]{\includegraphics[scale=0.85,page=2]{figures/forcepoints}}
\caption {\label{fig::forcepoints} Two types of forward/reverse symmetry for $\SLE_\kappa(\rho)$.  In each case, one can generate the path on the right either by starting with the scenario on the left and zipping up (via reverse $\SLE_\kappa(\wt{\rho})$ with the force point as shown) or starting with the scenario on the right and unzipping (i.e., drawing the path on the right as a forward $\SLE_\kappa(\rho)$ with the force point as shown).  In the top setting (see Proposition~\ref{prop::zip_up_to_zero}), we have $\rho = \kappa - \wt{\rho}$.  In the bottom setting (see Proposition~\ref{prop::zip_into_interior}), we have $\rho = \wt{\rho} - 8$.}
\end{center}
\end{figure}

In this section, we are going to establish several types of forward/reverse $\SLE_\kappa(\rho)$ symmetries.  The first of these was also used in \cite{she2010zipper} and justified there using the same argument we present below.  See also Figure~\ref{fig::forcepoints_origin} for an illustration.

\begin{proposition}
\label{prop::zip_up_to_zero}
Suppose that $(\wt{f}_t)$ is the centered reverse Loewner flow associated with an $\SLE_\kappa(\wt{\rho})$ process with a single force point of weight $\wt{\rho} < \tfrac{\kappa}{2}+2$ located at $x > 0$.  Let $\wt{\tau} = \inf\{t \geq 0: \wt{f}_t(x) = 0\}$ and let 
\begin{equation}
\label{eqn::rho_wtrho}
\rho = \kappa-\wt{\rho}.
\end{equation}
Then $\wt{f}_{\wt{\tau}}$ maps $\h$ to $\h \setminus \eta_{\wt{\tau}}$ where $\eta_{\wt{\tau}}$ has the law of the initial segment of a forward $\SLE_\kappa(\rho)$ with a single boundary force point located at $0^+$ of weight $\rho$ stopped at some a.s.\ positive time $\tau$.  In particular, if $\wt{\rho} = \kappa$, then $\eta_{\wt{\tau}}$ has the law of an ordinary $\SLE_\kappa$ stopped at an a.s.\ positive time $\tau$.
\end{proposition}

We note that the time $\tau$ is not a stopping time for the filtration generated by the forward $\SLE_\kappa(\rho)$ process in the statement of Proposition~\ref{prop::zip_up_to_zero}.  As we will see in the proof, it is given by the \emph{last} time that the centered forward Loewner flow sends $0^+$ to $x$.  Note that the assumption $\wt{\rho} < \tfrac{\kappa}{2}+2$ implies that $\p[\wt{\tau} < \infty] = 1$ in the statement of Proposition~\ref{prop::zip_up_to_zero}.  It also implies that $\rho$ from~\eqref{eqn::rho_wtrho} satisfies $\rho > \tfrac{\kappa}{2}-2$ and recall that $\tfrac{\kappa}{2}-2$ is the critical threshold at or above which forward $\SLE_\kappa(\rho)$ a.s.\ does not hit its force point.

\begin{proof}[Proof of Proposition~\ref{prop::zip_up_to_zero}]
Let $(W,V)$ (resp.\ $(\wt{W},\wt{V})$) be the driving process associated with a forward (resp.\ reverse) $\SLE_\kappa(\rho)$ (resp.\ $\SLE_\kappa(\wt{\rho})$) process with a single boundary force point of weight $\rho$ (resp.\ $\wt{\rho}$) located at $0^+$ (resp.\ $x > 0$).  Here, we assume that $V$ (resp.\ $\wt{V}$) is the process which gives the location of the force point for the uncentered Loewner flow.  Then we know that $\kappa^{-1/2}(V-W)$ evolves as a $\bes^\delta$ where $\delta$ is as in~\eqref{eqn::besseldim_forward}.  Similarly, $\kappa^{-1/2}(\wt{V}-\wt{W})$ evolves as a $\bes^{\wt{\delta}}$ where $\wt{\delta}$ is as in~\eqref{eqn::besseldim_reverse}.  Combining~\eqref{eqn::besseldim_reverse} and~\eqref{eqn::besseldim_forward} with~\eqref{eqn::wt_delta} of Proposition~\ref{prop::bessel_time_reversal} gives the relationship between $\rho$ and $\wt{\rho}$.  Namely,
\[ 1 + \frac{2 (\rho+2)}{\kappa} = 4 - \left( 1 + \frac{2(\wt{\rho}-2)}{\kappa} \right),\]
so that $\rho$ and $\wt{\rho}$ satisfy the relationship~\eqref{eqn::rho_wtrho}.  This implies that we can couple together $(W,V)$ and $(\wt{W},\wt{V})$ such that if we let $\wt{\tau} = \inf\{t \geq 0 : \wt{V}_t - \wt{W}_t = 0\}$ then
\begin{equation}
\label{eqn::beq}
V_t - W_t = \wt{V}_{\wt{\tau}-t} - \wt{W}_{\wt{\tau}-t} \quad\text{for all}\quad t \in [0,\wt{\tau}].
\end{equation}

Under this coupling, we have for all $t \in [0,\wt{\tau}]$ that
\begin{align}
      W_t 
&=  V_t + (W_t - V_t) \notag\\
&= \int_0^t \frac{2}{V_s - W_s} ds + (W_t - V_t) \quad\text{(definition of $V$)} \notag\\
&= \int_0^t \frac{2}{\wt{V}_{\wt{\tau}-s} - \wt{W}_{\wt{\tau}-s}} ds + (\wt{W}_{\wt{\tau}-t} - \wt{V}_{\wt{\tau}-t}) \quad\text{(by \eqref{eqn::beq})} \notag\\
&= \int_{\wt{\tau}-t}^{\wt{\tau}} \frac{2}{\wt{V}_s - \wt{W}_s}ds + (\wt{W}_{\wt{\tau}-t} - \wt{V}_{\wt{\tau}-t}) \notag\\
&= \int_0^{\wt{\tau}} \frac{2}{\wt{V}_s - \wt{W}_s} ds - \int_{0}^{\wt{\tau}-t} \frac{2}{\wt{V}_s-\wt{W}_s} ds + (\wt{W}_{\wt{\tau}-t} - \wt{V}_{\wt{\tau}-t}) \notag\\
&= \wt{V}_{\wt{\tau}-t} - \wt{V}_{\wt{\tau}} + (\wt{W}_{\wt{\tau} -t} - \wt{V}_{\wt{\tau}-t}) \quad\text{(definition of $\wt{V}$)} \notag\\
&= \wt{W}_{\wt{\tau}-t} - \wt{V}_{\wt{\tau}}
   = \wt{W}_{\wt{\tau}-t} - \wt{W}_{\wt{\tau}}. \label{eqn::wtw}
\end{align}
In the last step, we used that $\wt{V}_{\wt{\tau}} - \wt{W}_{\wt{\tau}} = 0$.  Since $\wt{W}_0 = 0$, we in particular have that
\begin{equation}
\label{eqn::wt_reversal}
W_{\wt{\tau}} = -\wt{W}_{\wt{\tau}}.
\end{equation}

Let $(g_t)$ (resp.\ $\wt{g}_t$) be the forward (resp.\ reverse) Loewner flow driven by $W$ (resp.\ $\wt{W}$).  We also let $(f_t)$ (resp.\ $(\wt{f}_t)$) be the centered forward (resp.\ reverse) Loewner flow driven by $W$ (resp.\ $\wt{W})$.  Let
\begin{equation}
\label{eqn::vpt_def}
\varphi_t(z) = g_t(\wt{g}_{\wt{\tau}}(z)-\wt{W}_{\wt{\tau}}) + \wt{W}_{\wt{\tau}} = g_t(\wt{f}_{\wt{\tau}}(z)) - W_{\wt{\tau}}.
\end{equation}
(The second equality uses~\eqref{eqn::wt_reversal}.)  Note that
\begin{equation}
\label{eqn::varphi_t_f_t}
\varphi_{\wt{\tau}}(z) = f_{\wt{\tau}}(\wt{f}_{\wt{\tau}}(z)).
\end{equation}
We have that
\begin{align*}
     \varphi_t(z)
&= \wt{g}_{\wt{\tau}}(z)  + \int_0^t \frac{2}{g_s(\wt{f}_{\wt{\tau}}(z)) - W_s} ds \quad\text{(definition of $\varphi_t$ and $g_t$)}\\
&= \wt{g}_{\wt{\tau}}(z) + \int_0^t \frac{2}{\varphi_s(z) - (W_s+\wt{W}_{\wt{\tau}})} ds \quad\text{(definition of $\varphi_t$)}\\
&= \wt{g}_{\wt{\tau}}(z)  + \int_0^t \frac{2}{\varphi_s(z)- \wt{W}_{\wt{\tau}-s}} ds \quad\text{(by \eqref{eqn::wtw})}.
\end{align*}
Since $\wt{g}_{\wt{\tau}-t}$ satisfies the same equation, we must have that
\begin{equation}
\label{eqn::g_g_t_centered}
 f_{\wt{\tau}}(\wt{f}_{\wt{\tau}}(z)) = \varphi_{\wt{\tau}}(z) = \wt{g}_0(z) = z,
\end{equation}
as desired.

Therefore $\wt{f}_{\wt{\tau}}$ maps $\h$ to $\h \setminus \eta_{\wt{\tau}}$ where $\eta_{\wt{\tau}}$ has the law of an initial segment of a forward $\SLE_\kappa(\rho)$, as desired.
\end{proof}

The correspondence between $\rho$ and $\wt{\rho}$ derived in Proposition~\ref{prop::zip_up_to_zero} takes a rather different form if we consider an interior force point instead of a point on $\R$.  See the bottom part of Figure~\ref{fig::forcepoints_interior} for an illustration.

\begin{proposition}
\label{prop::zip_into_interior}
Fix $\wt{\rho} < \tfrac{\kappa}{2}+4$.  Suppose that $(\wt{f}_t)$ is the centered reverse Loewner flow associated with an $\SLE_\kappa(\wt{\rho})$ process with a single interior force point located infinitesimally above~$0$.  Let $\rho = \wt{\rho}-8$ and let $\wt{Z}_t$ denote the evolution of the force point under $\wt{f}_t$.  For each $r > 0$, let $\wt{\tau}_r = \inf\{t > 0 : \im(\wt{Z}_t) = r\}$.  The law of $\h \setminus \wt{f}_{\wt{\tau}_r}(\h) = \eta_{\wt{\tau}_r}$ given $\wt{Z}_{\wt{\tau}_r}$ is that of a forward $\SLE_\kappa(\rho)$ process with a single interior force point of weight $\rho$ located at $\wt{Z}_{\wt{\tau}_r}$.
\end{proposition}
\begin{proof}
Repeating the calculations in the proof of Proposition~\ref{prop::reverse_zip_in_solution} using forward $\SLE_\kappa(\rho)$ instead of reverse $\SLE_\kappa(\wt{\rho})$ yields
\begin{equation}
 d \theta_{t(s)} = \sqrt{\kappa} \sin \theta_{t(s)} d\wh{B}_s + \left(-2 + \frac{\kappa}{2} - \frac{\rho}{2}\right) \sin(2 \theta_{t(s)})ds, \label{eqn::theta_forward}
\end{equation}
and $d I_{t(s)} = -2ds$.  We conclude that if we take $\rho = \wt{\rho} - 8$, then~\eqref{eqn::theta_forward} and~\eqref{eqn::theta_reverse} are the same.  

By the reversibility of~\eqref{eqn::theta_reverse} and~\eqref{eqn::theta_forward}, this implies that the following is true.  Suppose that we fix $r > 0$ and then
\begin{enumerate}
\item Sample $Z_r = X + ir$ so that $\arg(Z_r)$ is given by the stationary measure for~\eqref{eqn::theta_reverse},\eqref{eqn::theta_forward} as described in~\eqref{eqn::theta_reverse_measure} in the proof of Proposition~\ref{prop::reverse_zip_in_solution}.
\item Sample a centered forward $\SLE_\kappa(\rho)$ process $(f_t)$ with a single interior force point of weight $\rho$ located at $Z_r$.
\end{enumerate}
Then the evolution of $f_t(Z_r)$ considered in the time-interval from $0$ to $\inf\{t \geq 0 : \im(f_t(Z_r)) = 0\}$ has the same law as $\wt{Z}_{\wt{\tau}_r-t}$ for $t \in [0,\wt{\tau}_r]$.  Therefore the result follows from the argument given at the end of the proof of Proposition~\ref{prop::zip_up_to_zero}.
\end{proof}

Our next two forward/reverse symmetry results describe how one can sample the range of a boundary-intersecting forward $\SLE_\kappa(\rho)$ using a reverse Loewner flow in both the chordal and radial settings.

\begin{proposition}
\label{prop::zip_up_to_local_time}
Fix $\kappa > 0$, $\rho \in (-2,\tfrac{\kappa}{2}-2)$, and let $\wt{\rho} = \rho+4$.  Let $(W,V)$ (resp.\ $(\wt{W},\wt{V})$) be the driving process associated with a forward (resp.\ reverse) $\SLE_\kappa(\rho)$ (resp.\ $\SLE_\kappa(\wt{\rho})$) process in $\h$ from $0$ to $\infty$ with a single boundary force point of weight $\rho$ (resp.\ $\wt{\rho}$) located at $0^+$.  Let $(f_t)$ (resp.\ $(\wt{f}_t)$) denote the centered forward (resp.\ reverse) Loewner flow driven by $W$ (resp.\ $\wt{W}$).  Let $\ell$ (resp.\ $\wt{\ell}$) denote the local time of the Bessel process $\kappa^{-1/2}(V-W)$ (resp.\ $\kappa^{-1/2}(\wt{V}-\wt{W})$) at $0$ and denote by $T$ (resp.\ $\wt{T}$) the right continuous inverse of $\ell$ (resp.\ $\wt{\ell}$).  For each $u > 0$, we have that $f_{T_u}^{-1} \stackrel{d}{=} \wt{f}_{\wt{T}_u}$.
\end{proposition}

We first note that a $\bes^\delta$ with $\delta > 1$ is a continuous semimartingale, a Feller process, and has continuous transition densities \cite[Chapter~XI]{ry99martingales}.  Such a Bessel process is reversible with respect to the measure $x^{\delta-1}dx$ where $dx$ denotes Lebesgue measure on $\R_+$.  This is a generalization of the observation that Brownian motion is reversible with respect to Lebesgue measure on $\R$.  To see this, one notes that the generator for a Bessel process is self-adjoint with respect to the $L^2$ space induced by the measure $x^{\delta-1}dx$ on $\R_+$.  (Recall \cite[Proposition~5.2]{LIG_PARTICLES} as before.)

\begin{proof}[Proof of Proposition~\ref{prop::zip_up_to_local_time}]
Recall from~\eqref{eqn::besseldim_forward} that $\kappa^{-1/2}(V_t - W_t)$ is a $\bes^\delta$ with $\delta = 1+2(\rho+2)/\kappa$ and from~\eqref{eqn::besseldim_reverse} that $\kappa^{-1/2}(\wt{V}_t - \wt{W}_t)$ is a $\bes^{\wt{\delta}}$ with $\wt{\delta} = 1+2(\wt{\rho}-2)/\kappa$.  By the choice of $\rho$ and $\wt{\rho}$, we have that $\delta = \wt{\delta}$.  Consequently, Proposition~\ref{prop::markov_local_time_reversal} tells us that we can couple $(W,V)$ and $(\wt{W},\wt{V})$ together so that
\begin{equation}
\label{eqn::w_wt_coupling}
T_u = \wt{T}_u \quad\text{and}\quad W_{T_u-t}-V_{T_u-t}  = \wt{W}_t - \wt{V}_t \quad\text{for}\quad t \in [0,T_u].
\end{equation}
The result thus follows from the argument given at the end of Proposition~\ref{prop::zip_up_to_zero}.
\end{proof}

We now give the analog of Proposition~\ref{prop::zip_up_to_local_time} in the setting of radial $\SLE_\kappa(\rho)$.

\begin{proposition}
\label{prop::radial_zip_up_to_local_time}
Fix $\kappa > 0$, $\rho \in (-2,\tfrac{\kappa}{2}-2)$, and let $\wt{\rho} = \rho+4$.  Let $(W,V)$ (resp.\ $(\wt{W},\wt{V})$) be the driving process associated with a forward (resp.\ reverse) radial $\SLE_\kappa(\rho)$ (resp.\ $\SLE_\kappa(\wt{\rho})$) process in $\D$ from $1$ to $0$ with a single boundary force point of weight $\rho$ (resp.\ $\wt{\rho}$) located at $1^+$.  Let $(f_t)$ (resp.\ $(\wt{f}_t)$) denote the centered forward (resp.\ reverse) radial Loewner flow driven by $W$ (resp.\ $\wt{W}$).  Let $\ell$ (resp.\ $\wt{\ell}$) denote the local time of $V_t/W_t$ (resp.\ $\wt{V}_t/\wt{W}_t$) at $1$ and denote by $T$ (resp.\ $\wt{T}$) the right continuous inverse of $\ell$ (resp.\ $\wt{\ell}$).  For each $u > 0$, we have that $f_{T_u}^{-1} \stackrel{d}{=} \wt{f}_{\wt{T}_u}$.
\end{proposition}

The proof of Proposition~\ref{prop::radial_zip_up_to_local_time} is analogous to Proposition~\ref{prop::zip_up_to_local_time}.  The one point that we will explain is why we can apply Proposition~\ref{prop::markov_local_time_reversal} to justify time-reversing the driving process.  First, we write $V_t/W_t = e^{i \theta_t}$ where $\theta_t$ is as in~\eqref{eqn::theta_equation}.  We note that we can view $\theta_t$ as being a time-change of the solution to the SDE
\begin{equation}
d \phi_t = \frac{\delta-1}{4} \cot\left( \frac{\phi_t}{2} \right) dt + dB_t
\end{equation}
with $\delta = 1+2(\rho+2)/\kappa$.  This is the same relationship between $\delta$, $\kappa$, and $\rho$ as in the case of forward chordal $\SLE_\kappa(\rho)$ given in~\eqref{eqn::besseldim_forward}.  Note that $\phi_t$ a semi-martingale for $\delta > 1$ since, by the Girsanov theorem, its law is absolutely continuous with respect to that of a $\bes^\delta$ with $\delta > 1$ when it is interacting with either~$0$ or~$2\pi$.  It is similarly a time-homogeneous Markov process with continuous transition densities.  Moreover, it is reversible with respect to the measure $\sin(x/2)^{\delta-1} dx$ where $dx$ denotes Lebesgue measure on $[0,2\pi]$.  To see this, we note that the generator of $\phi_t$ is self-adjoint with respect to the $L^2$ space associated with the measure $\sin(x/2)^{\delta-1}dx$ on $[0,2\pi]$ (recall \cite[Proposition~5.2]{LIG_PARTICLES} as before).

\subsection{Conformal structure and removability}
\label{subsec::removability}

Given two topological disks with boundary (each endowed with a good area measure in the interior, and a good length measure on the boundary) it is a simple matter to produce a new topological surface by taking a quotient that involves gluing (all or part of) the boundaries to each other in a boundary length preserving way.

The problem of {\em conformally welding} two surfaces is the problem of obtaining a conformal structure on the combined surface, given the conformal structure on the individual surfaces.  (The reader who is completely unfamiliar with conformal weldings may wish to read the introduction of \cite{she2010zipper}, or to see, e.g., \cite{bishop2007conformal} for further discussion and references.)  This is closely related to the problem of defining a Brownian motion (up to monotone reparameterization) on the combined surface, given the definition of a Brownian motion on each of the individual pieces.  We will now briefly (and somewhat informally) describe this connection.  It is well-known that a conformal structure and a choice of initial point determine a Brownian diffusion process $\beta_t$ (up to monotone reparameterization).  Similarly, if one is given the diffusion process (for all starting points), one can recover the conformal structure in a neighborhood of a point as follows: consider any Jordan curve surrounding that point, with three marked points on the curve dividing it into three segments $E_1, E_2, E_3$; then for each $z$ in the region surrounded by the curve, consider the triple $(p_1,p_2,p_3)$ where $p_i$ for $i=1,2,3$ is the probability that the Brownian motion starting from $z$ first hits the curve along $E_i$.  Note that $p_1 + p_2 = p_3 = 1$.  One can then ``conformally map'' this region to a Euclidean triangle by sending each such $z$ to the point in the triangle such that a standard two-dimensional Brownian motion from that point has probability $p_i$ of first hitting the triangle boundary along the $i$th edge.  The image of $\beta_t$ in the triangle will then be two-dimensional Brownian motion (up to a time change).

This recipe can be carried out for any continuous diffusion process $\beta_t$ to produce a map to the triangle; however it is not true that for all diffusion processes the image process in the triangle will be a standard Brownian motion.  For example, it is necessary that the stationary measure for $\beta_t$ (when one allows reflection at the boundary) be an atom-free measure that assigns positive area to open sets, and that $\beta_t$ is reversible with respect to this measure.  The diffusions of {\em Brownian type} --- i.e., diffusions that correspond to time-changed Brownian motions w.r.t.\ {\em some} conformal structure --- are a rather special subset of the set of all of the continuous diffusion processes one might produce on a topological sphere.   ``Conformally welding'' the two conformal surfaces is equivalent to producing a continuous diffusion process of Brownian type that agrees with the Brownian motions on the individual surfaces at times when it is away from the boundary interface.  To make sense of this idea, we will draw from the theory of {\em removability sets}, as explained below.

\begin{definition}
\label{def::conf_removable}
Suppose that $D \subseteq \C$ is a domain.  A compact subset $K$ of $D$ is called (conformally) {\bf removable} if every homeomorphism from $D$ to a subset of $\C$ that is conformal on $D \setminus K$ is also conformal on all of $D$.
\end{definition}

In probabilistic language, the fact that $K$ is removable means that once we are given the Brownian motion behavior off $K$, there is --- among all the possible ways of extending this process to a continuous diffusion on $D$, involving various types of local time pushes along $K$, etc. --- only one of Brownian type.  A Jordan domain $D \subseteq \C$ is said to be a {\bf H\"older domain} if any conformal transformation from~$\D$ to $D$ is H\"older continuous all of the way up to $\partial \D$.  It was shown by Jones and Smirnov \cite{MR1785402} that if~$K \subseteq D$ is the boundary of a H\"older domain, then $K$ is removable; it is also noted there that if a compact set $K$ is removable as a subset of $D$, then it is removable in any domain containing $K$, including all of $\C$.  Thus, at least for compact sets $K$, one can speak of removability without specifying a particular domain $D$.

It was further shown by Rohde and Schramm \cite{rs2005sle} that the following is true.  Suppose that $g_t \colon \h \setminus \eta([0,t]) \to \h$ is the forward Loewner flow associated with an $\SLE_{\kappa}$ curve with $\kappa < 4$.  Then for each $t \geq 0$, the map $g_t^{-1}$ is a.s.\ H\"older continuous. (A more general result, proved in a different way, which implies this is also given in \cite[Section~8.1]{ms2013qle}.)  If $\eta$ is an $\SLE_\kappa$ in $\h$ from $0$ to $1$, then the union of $\eta$ and its reflection across the real axis is the boundary of a bounded H\"older domain, hence is removable by \cite{MR1785402}.  It follows immediately that an $\SLE_{\kappa}$ path $\eta$ from $0$ to $\infty$ is removable in $\h$, and that $\eta([t_1, t_2])$ is removable for any $t_1 < t_2$.  In fact, we can also say the following:

\begin{proposition}
\label{prop::subsegmentremovable}
Suppose that $\eta$ is an $\SLE_\kappa$ curve in $\h$ from $0$ to $1$ with $\kappa \in (0,4)$.  Then it is a.s.\ the case that for a dense set of pairs of times $(t_1, t_2)$, with $0 \leq t_1 < t_2$, the segment $\eta([t_1,t_2])$ is removable in the domain $\h \setminus \eta\bigl( [0,t_1] \cup [t_2, \infty] \bigr)$.
\end{proposition}
\begin{proof}
As explained above, we already have the removability for the overall curve from~$0$ to~$1$.  If we observe the path starting from~$0$ (up to any stopping time) and then observe the path from~$1$ (up to a reverse stopping time) then the conditional law of the remaining path is that of an $\SLE_\kappa$ in the remaining domain \cite{MR2435856,dub_dual,ms2012imag2}, which implies the result.
\end{proof}

Proposition~\ref{prop::subsegmentremovable} also implies the removability of random paths $\eta$ that look locally like $\SLE_\kappa$, such as the $\SLE_{\kappa}(\rho)$ processes with $\rho > -2$ and flow lines of the GFF.  To see why, observe that for any point  $z \in \h$ on such a path, one can find times~$t_1$ and~$t_2$ such that $z \in \eta([t_1,t_2])$ and $\eta([t_1,t_2])$ is removable in $\h \setminus \eta\bigl( [0,t_1] \cup [t_2, \infty] \bigr)$.  It thus follows that any homeomorphism~$\phi$ which is conformal on the complement of~$\eta$ will also be conformal in a neighborhood containing~$\eta((t_1,t_2))$.  Since this holds for any~$z \in \h$, we find that any such map must be conformal in a neighborhood every point  on~$\eta \cap \h$, hence conformal everywhere in $\h$.

To our knowledge, it is not known in general that the union of two (non-disjoint) removable sets is removable.  However, this is not a problem when one of the sets is of $\SLE$ type:

\begin{proposition}
\label{prop::sle_union_removable_is_removable}
Suppose that $\wt{\eta}$ is a random segment of an $\SLE_\kappa$ curve $\eta$ with $\kappa \in (0,4)$. Then it is a.s.\ the case that for all removable sets $K$ (simultaneously) the set $\wt{\eta} \cup K$ is also removable.
\end{proposition}
\begin{proof}
Let $\phi$ be any homeomorphism which is conformal off $\wt{\eta} \cup K$.  Fix $z \in \wt{\eta} \setminus K$.  By Proposition~\ref{prop::subsegmentremovable}, we can find an interval $(t_1,t_2)$ of time such that $z \in \eta((t_1,t_2))$, $\eta([t_1,t_2])$ is removable in $\h \setminus \eta\bigl( [0,t_1] \cup [t_2, \infty] \bigr)$, and $\eta([t_1,t_2])$ is at a positive distance from $K$.  From this one can easily see that $\eta([t_1,t_2])$ is removable in $\h \setminus \left(\eta\bigl( [0,t_1] \cup [t_2, \infty] \bigr) \cup K \right)$ since $K$ has positive distance from $\eta([t_1,t_2])$.  It then follows that $\phi$ must be conformal in a neighborhood of $z$.  Since $z \in \wt{\eta} \setminus K$ was arbitrary, we have that $\phi$ is conformal on the complement of $K$, hence conformal everywhere by the removability of~$K$.
\end{proof}

A similar argument implies the following:

\begin{proposition}
\label{prop::collection_of_sles_removable}
Suppose that $\eta_1, \ldots, \eta_k$ are random curves each of which has the law of an $\SLE_\kappa(\rho)$ with $\kappa \in (0,4)$ and $\rho > -2$ (but with possibly different $\kappa$ and $\rho$ values) defined on the same domain (either~$\H$ or~$\C$).  (The curves may be coupled with each other in an arbitrary way.) Then it is a.s.\ the case that the union of these curves is removable. Furthermore, suppose $\eta_1, \eta_2, \ldots$ is a countably infinite collection of compact $\SLE_\kappa$ segments (coupled in some way) with the property that a.s.\ the set of accumulation points of these segments (i.e., the set of points any neighborhood of which intersects infinitely many segments) is discrete. Then it is again a.s.\ the case that the union of these segments if removable.
\end{proposition}

This implies in particular that flow lines of the GFF (whole plane $\SLE_{\kappa}$ processes started at interior points) are removable.

\begin{remark} The authors in \cite{astala2011random} describe what amounts (in our language) to a conformal welding of a Euclidean disk to a Liouville quantum gravity disk and prove that the resulting curve is removable by studying regularity and symmetry properties of the LQG boundary measure. One might (though this remains far from obvious) be able to adapt the techniques in \cite{astala2011random} to conformally weld two LQG surfaces to each other --- and thereby produce an alternate proof of the welding existence first described in \cite{she2010zipper}. We note that in \cite{she2010zipper} and in the current paper, we obtain more than just the existence of such a welding; we also explicitly establish the law of the interface curve in terms of SLE and give the law of the combined LQG surface. In that sense, the main result in \cite{she2010zipper} is stronger than the main result of  \cite{astala2011random} but the techniques in  \cite{she2010zipper} may be harder to generalize, as they rely on an exact understanding of SLE, LQG, etc.
Finally, let us note that it is currently unknown whether $\SLE_{\kappa'}$ for $\kappa' > 4$  is removable, and that we know no general complex analysis argument (aside from what we do in this paper) that would show that matings of CRTs and/or L\'evy trees of quantum disks are well defined in any canonical sense, though Lin and Rohde have announced some work in progress on the problem of mating a CRT to a Euclidean disk.  {\it Update: See \cite{mmq2018uniqueness} for a proof that $\SLE_{\kappa'}$ for $\kappa' \in (4,8)$ satisfies a weaker version of removability.}
\end{remark}

\section{Quantum surfaces}
\label{sec::surfaces}

The purpose of this section is to give a careful definition of the different types of quantum surfaces which will arise in this article as introduced in Section~\ref{subsec::surfaces}: quantum wedges, cones, disks, and spheres.  We will begin in Section~\ref{subsec::gffs} with a brief introduction to the various types of GFFs which will be important for this article.  We continue in Section~\ref{subsec::thick_wedges} with the definition of a so-called ``thick'' wedge.  This is an infinite volume and infinite boundary length surface which (together with its prime-end boundary) is homeomorphic to~$\ol{\h}$.  We will then give the definition of a quantum cone in Section~\ref{subsec::cones}.  This is an infinite volume surface without boundary which is homeomorphic to~$\C$.  This section will be largely parallel with Section~\ref{subsec::thick_wedges}.  We will then give our first definition of a ``thin'' wedge in Section~\ref{subsec::surfaces_strips_cylinders} (we will give another construction in Section~\ref{sec::quantumwedge}).  This will describe an ordered sequence of surfaces each of which is (together with its prime-end boundary) homeomorphic to $\ol{\D}$.

\subsection{Gaussian free fields}
\label{subsec::gffs}

We are now going to summarize the properties of the GFF which will be important for this article.  We direct the reader to \cite{Sh} for a detailed introduction in addition to \cite[Section~3]{she2010zipper} and \cite[Section~2.2]{ms2013imag4} for some additional discussion of the whole-plane and free boundary GFFs.

\subsubsection{Dirichlet inner product}

Suppose that $D \subseteq \C$ is a domain.  Let $C_0^\infty(D)$ denote the set of $C^\infty$ functions which are compactly supported in~$D$.  The {\bf Dirichlet inner product} is defined by
\begin{equation}
\label{eqn::dirichlet}
(f,g)_\nabla = \frac{1}{2\pi} \int_D \nabla  f(x) \cdot \nabla g(x) dx \quad \text{for}\quad f,g \in C_0^\infty(D).
\end{equation}
More generally,~\eqref{eqn::dirichlet} makes sense for $f,g \in C^\infty$ with $L^2$ gradients.

\subsubsection{Distributions}
\label{subsubsec:gff_spaces}

We view $C_0^\infty(D)$ as a space of test functions and equip it with the topology where a sequence $(\phi_k)$ in $C_0^\infty(D)$ satisfies $\phi_k \to 0$ if and only if there exists a compact set $K \subseteq D$ such that the support of $\phi_k$ is contained in $K$ for every $k \in \N$ and $\phi_k$ as well as all of its derivatives converge uniformly to zero as $k \to \infty$.  A {\bf distribution} on $D$ is a continuous linear functional on $C_0^\infty(D)$ with respect to the aforementioned topology.  A {\bf modulo additive constant distribution} on $D$ is a continuous linear functional which is defined on the subspace of functions $f \in C_0^\infty(D)$ with $\int_D f(x) dx = 0$ with the same topology.

\subsubsection{GFF with Dirichlet boundary conditions}
\label{subsubsec::gff_dirichlet}

Assume that $D \subseteq \C$ is a domain.  Recall that the harmonic measure on $\partial D$ as seen from a point $z \in D$ is the first exit distribution on $\partial D$ defined by a standard Brownian motion starting from $z$.  We shall assume that $D$ has {\bf harmonically non-trivial} boundary, by which we mean that the harmonic measure of $\partial D$ is positive as seen from any point in $D$.  We let $H_0(D)$ be the Hilbert-space closure of $C_0^\infty(D)$ with respect to the Dirichlet inner product~\eqref{eqn::dirichlet}.  The GFF $h$ on $D$ with zero Dirichlet boundary conditions can be expressed as a random linear combination of an $(\cdot,\cdot)_\nabla$-orthonormal basis $(f_n)$ of $H_0(D)$:
\begin{equation}
\label{eqn::gff_series}
h = \sum_n \alpha_n f_n,\quad (\alpha_n) \quad \text{i.i.d.}\quad N(0,1).
\end{equation}
Although this expansion of $h$ does not converge in $H_0(D)$, we claim that it does converge a.s.\ in the space of distributions. That is, the limit 
\begin{equation}
\label{eqn::gff_series2}
(h, f) :=\lim_{N\to\infty} \left(\sum_{n=1}^N \alpha_n f_n, f\right)
\end{equation}
a.s.\ exists for every test function $f \in C_0^\infty(D)$, and the $h$ defined this way is a.s.\ a distribution, where each such $(h,f)$ is a centered Gaussian and \begin{equation} \label{eqn::gff_variance_def} \cov\bigl( (h,f), (h,g) \bigr) = \iint f(x) G(x,y) g(y) dxdy\end{equation} where the Green's function $G(x,y)$ is defined by
\[ G(x,y) := -\log|y-x| - \wt G_x(y)\]
and for each fixed $x$, $\wt G_x(y)$ is the harmonic extension to $D$ of the function on $\partial D$ given by $y \mapsto -\log|y-x|$. (The finiteness of $\wt G_x(y)$ for distinct $x$ and $y$ in $D$ can be obtained from the optional stopping theorem and the fact that if $B_t$ is a Brownian motion in $D$ then $\log|B_t|$ evolves as a local martingale.)

The distributional convergence of \eqref{eqn::gff_series2} and the variance formula \eqref{eqn::gff_variance_def} are justified in \cite{Sh} in the case that $D$ is any bounded domain. By conformal invariance of the Dirichlet inner product and the definition of a distribution, this extends to any domain $D$ which is homeomorphic to a bounded domain; this covers the cases of interest for this paper, including any $D$ whose complement includes a connected path segment.

\begin{remark}
The distributional convergence of \eqref{eqn::gff_series2} and the variance formula \eqref{eqn::gff_variance_def} are \emph{not} justified in \cite{Sh} in the case that, e.g., the complement of $D$ is a harmonically non-trivial Cantor set. Nonetheless, we remark that one may deal with the general harmonically non-trivial boundary case as follows: assume (translating if necessary) that $0 \in D$, let $D_\delta = D\setminus I_\delta$ where $I_\delta$ is the slit $[0,\delta]$ along the real axis, and consider the $\delta \to 0$ limit. It is easy to see that the Green's functions for $D_\delta$ increase monotonically (e.g., by the Brownian motion definition) to $G(x,y)$ in the $\delta \to 0$ limit. For each positive integer $n$, the spaces $H_0(D_{1/n})$ are an increasing sequence whose union is dense in $H_0(D)$. Suppose we let $(f_{1,k})$ be an orthonormal basis for $H_0(D_1)$  and for each $n \in \N$ with $n \geq 2$ we can let $(f_{n,k})$ be an orthonormal basis for the orthogonal complement of $H_0(D_{1/(n-1)})$ in $H_0(D_{1/n})$. Then for each $f$ compactly supported away from $0$, the series $(h,f) = (\sum \alpha_{i,j} f_{i,j}, f)$ converges (for i.i.d.\ Gaussians $(\alpha_{i,j})$) and the random variables defined this way satisfy the covariance formula  \eqref{eqn::gff_variance_def} (and the standard arguments in, e.g., \cite{Sh} imply that this continues to hold with other orthonormal basis choices).  The fact that this a.s.\ holds for {\em all} test functions supported away from $0$ (instead of just for a fixed test function) can be obtained by letting $h_n$ be the projection of $h$ onto $H_0(D_{1/n})$, and noting that the differences $h_{n+1} - h_n$ are random harmonic functions (off $I_{1/n}$) whose suprema on any fixed compact subset of $D$ is bounded (though we will not explain this in detail here).
\end{remark}

It is also shown in  \cite{Sh} that (when $D$ is bounded) one has convergence in the fractional Sobolev space $H^{-\epsilon}(D)$ for each $\epsilon > 0$. 
    If $f,g \in C_0^\infty(D)$ then an integration by parts gives $(f,g)_\nabla = -(2\pi)^{-1} ( f,\Delta g)$.  Using this, we define
\[ (h,f)_\nabla = -\frac{1}{2\pi}(h, \Delta f) \quad\text{for}\quad f \in C_0^\infty(D).\]
Observe that $(h,f)_\nabla$ is a Gaussian random variable with mean zero and variance $(f,f)_\nabla$.  In fact, for any {\em fixed} $f \in H_0(D)$, we can define $(h,f)_\nabla$ via \eqref{eqn::gff_series}, and once this is done the space of random variables of the form $(h,f)_\nabla$ is a {\em Gaussian Hilbert space}, which we denote by $\CG$, and

\[
\cov((h,f)_\nabla,(h,g)_\nabla) = (f,g)_\nabla \quad\text{for all}\quad f,g \in H_0(D).
\]

\subsubsection{GFF with free boundary conditions}

Suppose that $D \subseteq \C$ has harmonically non-trivial boundary.  Let $H(D)$ be the Hilbert space closure of the subspace of functions $f \in C^\infty(D)$ with $\| f \|_\nabla^2 := (f,f)_\nabla < \infty$, defined modulo additive constant, with respect to the Dirichlet inner product~\eqref{eqn::dirichlet}.  (Note that $H(D)$ differs from $H_0(D)$.)  We define the free boundary GFF $h$ on $D$ using the series
\begin{equation}
\label{eqn::freeGFFsum}
h = \sum_n \alpha_n f_n,\quad (\alpha_n) \quad \text{i.i.d.}\quad N(0,1)
\end{equation}
where $(f_n)$ is an orthonormal basis of $H(D)$.  We emphasize that the free boundary GFF is only defined modulo additive constant, but there are various ways of fixing it (e.g., by declaring that the integral against a given test function is equal to $0$). In order to prove that this makes sense, it is sometimes convenient to decompose $H(D)$ into two orthogonal subspaces: $H_0(D)$ and the space of (finite Dirichlet energy) harmonic functions on $D$.  The projection of the free boundary GFF $h$ onto the former is just the zero boundary GFF (after fixing the additive constant so that its average on $\partial \D$ vanishes), while the projection onto the latter is a Gaussian harmonic function (defined modulo additive constant).

Let us first consider the case that $D = \D$ is the unit disk.  Then every harmonic function on $\D$ is the real part of an analytic function, and hence (if we subtract an additive constant to make the constant term zero) has an expansion of the form $\sum_{k=1}^\infty \Re (a_k z^k)$ for a sequence of complex numbers $(a_k)$.  The two components of the gradient of $z^k$ are the real and imaginary parts of the derivative $kz^{k-1}$. Hence the Dirichlet energy of $z^k$ is
\[\frac{1}{2\pi} k^2 \int_D |z|^{2k-2} dz  = k^2 \int_0^1 r^{2k-1} dr = k^2 \frac{ 1}{2k} = \frac{k}{2},\]
and so an explicit orthonormal basis for the space of harmonic functions on $\D$ is given by $\sqrt{2/k} \Re(z^k)$ and $\sqrt{2/k} \Re (iz^k)$ for positive integer $k$.  Thus we can write an instance of the harmonic part of the free boundary GFF as $\sum_{k=1}^\infty \sqrt{2/k} \Re (\alpha_k z^k)$ where $(\alpha_k)$ is a sequence of i.i.d.\ complex Gaussians (i.e., their real and imaginary parts are i.i.d.\ $N(0,1)$).  For any fixed $r \in (0,1)$, it is now obvious that this sum (and all of its derivatives) a.s.\ converge uniformly in $B(0,r)$, and that the limit is a.s.\ a harmonic function. (This is of course much stronger than the convergence we require in order to say that the sum a.s.\ makes sense as a random distribution.)  The convergence of the projection of the GFF to $H_0(\D)$ follows from the discussion in Section~\ref{subsubsec::gff_dirichlet}.  Combined, this proves that the series~\eqref{eqn::freeGFFsum} which defines the free boundary GFF on $\D$ converges in the space of modulo additive constant distributions for a specific choice of orthonormal basis of $H(\D)$: the one given above for the space of harmonic functions on $\D$ together with any fixed orthonormal basis for $H_0(\D)$.  The standard arguments described in, e.g., \cite{Sh} then imply the convergence in the space of modulo additive constant distributions for any fixed choice of orthonormal basis of $H(\D)$.

We will now explain why one has the convergence of~\eqref{eqn::freeGFFsum} in the space of modulo additive constant distributions in the case that $D$ is a general planar domain with harmonically non-trivial boundary.  We may assume without loss of generality (translating and dilating if necessary) that~$D$ contains~$\D$. Let us begin by showing that one has the convergence of~\eqref{eqn::freeGFFsum} restricted to test functions with support in $\D$ (i.e., we will argue that the restriction of the GFF to $\D$ is defined).  We will proceed by showing that one has the convergence with respect to a fixed orthonormal basis and then extend to all orthonormal bases using the standard arguments \cite{Sh}.  Let $H_1$ be the subspace of $H(D)$ which is given by those (modulo additive constant) functions which are constant on the unit circle $\partial \D$ and let $H_2$ be the orthogonal complement of $H_1$. Note that $H_2$ consists of functions that, given their restriction to~$\partial \D$, have minimal Dirichlet energy (such functions are harmonic off~$\partial \D$ and are said to satisfy {\em Neumann} boundary conditions on~$\partial D$).

The projection of the GFF on $D$ to $H_1$ is defined as above in the case that $D=\D$.  When the additive constant is fixed so that its average on $\partial \D$ vanishes, it is a zero boundary GFF.  Let us now describe how the projection onto $H_2$ (restricted to $\D$) is defined.  It will be a random harmonic function on $\D$ (defined modulo additive constant).  Note that each $f \in H_2$ determines (by restriction) an element of $H(\D)$, and the norm squared of an $f \in H_2$ is between $1$ and $2$ times its norm squared as an element of $H(\D)$ (since the total Dirichlet energy {\em outside} of $\D$ --- by conformal symmetry and the fact that this quantity is being minimized subject to values on $\partial \D$ --- is at most the Dirichlet energy inside of $\D$).  In particular, the identity operator mapping $H_2$ (with the Dirichlet inner product computed on all of $D$) to the space of harmonic functions on $H(\D)$ (with the Dirichlet inner product computed on $\D$) is a bounded operator with bounded inverse.

Fix $r \in (0,1)$.  We are now going to explain why one has the modulo additive constant distributional convergence on $B(0,r)$ of the part of the sum~\eqref{eqn::freeGFFsum} consisting of basis elements of $H_2$.  The standard abstract Wiener space theory (as in \cite{Sh}) implies that the harmonic part of the GFF on~$\D$ can be understood as a random element of $L^2(B(0,r))$ (modulo additive constant).  In particular, the harmonic part can be understood as a (modulo additive constant) random distribution on $B(0,r)$ and the defining series a.s.\ converges in this space for any orthonormal basis choice. This is a consequence of the fact that the latter space can be realized as the dual of a space defined by norm of the form $f \to |Lf|$ where $L$ is a so called Hilbert-Schmidt operator.  We will not discuss this theory further here, but we note that (as explained above) the function from $H_2$ to $H(\D)$ is a bounded linear operator, and it is well-known (and immediate from definitions) that the composition of a bounded linear operator (from one Hilbert space to another) and a Hilbert-Schmidt operator is again a Hilbert-Schmidt operator. This implies that the same convergence results in $L^2(B(0,r))$ (modulo additive constant) apply if we use an orthonormal basis for $H_2$ in place of an orthonormal basis for $H(\D)$. Thus if $(f_k)$ is an orthonormal basis for $H_2$ and $(\alpha_k)$ is a sequence of i.i.d.\ standard normal random variables then $\sum \alpha_k f_k$ converges a.s.\ in $L^2(B(0,r))$ (modulo additive constant) to a random harmonic limit.  Combining, we have shown that the sum~\eqref{eqn::freeGFFsum} converges in the modulo additive constant distributional sense on $B(0,r)$ with a fixed choice of orthonormal basis.  As explained before, the standard arguments \cite{Sh} yield that we have the modulo additive constant distributional convergence of~\eqref{eqn::freeGFFsum} with respect to any fixed orthonormal basis (not just one consisting of an orthonormal basis of $H_1$ and an orthonormal basis of $H_2$).

We have thus completed the proof of the modulo additive constant distributional convergence of the sum~\eqref{eqn::freeGFFsum} since any compact set subset of $D$ can be covered by finitely many such disks.

As a consequence of the conformal invariance of the Dirichlet inner product, the GFF with free boundary conditions (as a distribution viewed modulo additive constant) is conformally invariant.  This means that if $D$, $\wt{D}$ are domains with harmonically non-trivial boundary, $\varphi \colon D \to \wt{D}$ is a conformal transformation, and $h$ is a GFF with free boundary conditions on $D$ (defined modulo additive constant) then $\wt{h} = h \circ \varphi^{-1}$ is a GFF with free boundary conditions on $\wt{D}$ (defined modulo additive constant).

The whole-plane GFF is constructed in the same manner except with $D = \C$.  In particular, like the free boundary GFF, the whole-plane GFF is only defined modulo a global additive constant.  See \cite[Section~3]{she2010zipper} and \cite[Section~2.2]{ms2013imag4} for additional discussion of the whole-plane GFF.

More generally, suppose that $D \subseteq \C$ is a domain and $\partial D = \partial^\dirichlet \cup \partial^\free$ where $\partial^\dirichlet \neq \emptyset$ and $\partial^\dirichlet \cap \partial^\free = \emptyset$.  We assume further that the harmonic measure of $\partial^\dirichlet$ is positive as seen from any point $z \in D$.  The GFF on $D$ with Dirichlet (resp.\ free) boundary conditions on $\partial^\dirichlet$ (resp.\ $\partial^\free$) is constructed using a series expansion as in~\eqref{eqn::gff_series} except the space $H_0(D)$ is replaced with the Hilbert space closure with respect to $(\cdot,\cdot)_\nabla$ of the subspace of functions in $C^\infty(D)$ which have an $L^2$ gradient and vanish on~$\partial^\dirichlet$.  We note that the GFF with mixed boundary conditions has well-defined values (and is not just defined modulo additive constant).  As a consequence of the conformal invariance of the Dirichlet inner product, the GFF with mixed boundary conditions is conformally invariant.

\subsubsection{GFF Markov property}

Let us now say a few words about the Markov property of the GFF (with all types of boundary conditions).

\noindent{\it Dirichlet boundary conditions.}  We begin by reminding the reader of the Markov property of the GFF with Dirichlet boundary conditions.  Suppose that $h$ is a GFF on a domain $D \subseteq \C$ with Dirichlet boundary conditions and $U \subseteq D$ is open.  Then the Markov property in this case states that we can write $h = h_1 + h_2$ where $h_1$ (resp.\ $h_2$) is a GFF with zero boundary conditions (resp.\ is harmonic) on $U$ and $h_1,h_2$ are independent.  We recall that this follows by decomposing the Hilbert space $H_0(D)$ into the orthogonal sum consisting of $H_0(U)$ and those functions in $H_0(D)$ which are harmonic on $U$.  If we express the GFF using a series expansion using an orthonormal basis of $H_0(D)$ consisting of an orthonormal basis of $H_0(U)$ and its orthogonal complement then we obtain exactly the claimed decomposition.

\noindent{\it Mixed boundary conditions.}  In the case that $\partial U \cap \partial^\free D = \emptyset$, the statement in the case of a GFF with mixed boundary conditions as described just above is exactly the same as in the case of Dirichlet boundary conditions.  In the case that $\partial U \cap \partial^\free D \neq \emptyset$, we can write $h = h_1 + h_2$ where $h_1$ (resp.\ $h_2$) is a GFF with mixed boundary conditions on $U$ which are free on $\partial U \cap \partial^\free D$ and zero on $\partial U \setminus \partial^\free D$ (resp.\ $h_2$ is harmonic in $U$ with Neumann boundary conditions on $\partial U \cap \partial^\free D$ and boundary conditions given by those of $h$ on $\partial U \setminus \partial^\free D$).  Moreover, $h_1,h_2$ are independent.  If one conditions further on the values of $h$ on all of $\partial U$, then the decomposition becomes $h = h_1+h_2$ where $h_1$ (resp.\ $h_2$) is a GFF with zero boundary conditions (resp.\ $h_2$ is harmonic) on $U$ and $h_1,h_2$ are independent.  These statements are proved in the same way as the Markov property for the GFF with Dirichlet boundary conditions is proved.

\noindent{\it Free boundary conditions.}  If $\partial U \cap \partial D = \emptyset$ and $h$ is a GFF on $D$ with free boundary conditions, then we can write $h = h_1 + h_2$ where $h_1$ (resp.\ $h_2$) is a GFF with zero boundary conditions (resp.\ is harmonic) on $U$ and $h_1,h_2$ are independent.  The difference with the case of Dirichlet boundary conditions is that $h_2$ is only defined modulo additive constant.  This is proved in the same way as the Markov property for the GFF with Dirichlet boundary conditions is proved.  Namely, we can express $H(D)$ as an orthogonal sum consisting of $H_0(U)$ and those functions in $H(D)$ (which we recall are defined modulo additive constant) which are harmonic on $U$.  If $\partial U \cap \partial D \neq \emptyset$, then the decomposition becomes $h = h_1 + h_2$ where $h_1$ (resp.\ $h_2$) is a GFF on $U$ with zero boundary conditions on $\partial U \setminus \partial D$ and free boundary conditions on $\partial U \cap \partial D$ (resp.\ $h_2$ is harmonic on $U$ with Neumann boundary conditions on $\partial U \cap \partial D$ and the boundary conditions given by $h$ on $\partial U \setminus \partial D$) and $h_1,h_2$ are independent.  In this case, $h_2$ is only defined modulo additive constant.  Also, if we condition further on the values of $h$ on $\partial U$, then $h_1$ is simply a zero boundary GFF on $U$.

We can also consider the Markov property of the free boundary GFF in the case that we have fixed the additive constant in some way.  If we fix the additive constant so that the average of $h$ on a set which is \emph{disjoint from $U$} is equal to $0$ or the integral of $h$ against a test function whose support is \emph{disjoint from $U$} is equal to $0$, then the Markov property is exactly the same except now $h_2$ has well-defined values and is no longer only defined modulo additive constant.

\noindent{\it Whole-plane GFF.}  The statements in the case of the whole-plane GFF are the same as in the case of free boundary conditions.

The situation for both the free boundary and whole-plane GFFs are slightly more complicated when one has fixed the additive constant in a way which depends on the field values in $U$, but we will not discuss this possibility further since we will not use the Markov property in this way in this article.

\subsubsection{Circle averages and Brownian motion}

Suppose that $h$ is a GFF on a domain $D \subseteq \C$ (with any type of boundary conditions or a whole-plane GFF if $D = \C$).  For each $\epsilon > 0$, we let $h_\epsilon(z)$ be the average of $h$ on $\partial B(z,\epsilon)$.  As explained in \cite[Proposition~3.1]{ds2011kpz}, the circle average process $(z,\epsilon) \mapsto h_\epsilon(z)$ has a modification which is H\"older continuous jointly in~$z$ and~$\epsilon$.  In particular, $h_\epsilon(z)$ is defined for all $(z,\epsilon)$ pairs simultaneously a.s.  We note that in the case that $h$ has free or mixed boundary conditions, then by the Markov property of the GFF we can write $h = h_1 + h_2$ where $h_1$ is a zero-boundary GFF on $D$ and $h_2$ is harmonic in $D$.  If the additive constant for $h$ has not been fixed, then $h_2$ is only defined modulo additive constant and in this case $h_1,h_2$ are independent.  As the average of $h_2$ on $\partial B(z,\epsilon)$ for any $z \in D$ and $\epsilon > 0$ such that $B(z,\epsilon) \subseteq D$ is equal to $h_2(z)$ as $h_2$ is harmonic, it follows that $(z,\epsilon) \mapsto h_\epsilon(z)$ has the same continuity properties as in the case that $h$ has Dirichlet boundary conditions.  When $h$ is only defined up to a global additive constant, then so is $h_\epsilon$.  In this case, differences of averages such as $h_\epsilon(z) - h_\delta(w)$ for $z,w \in D$ have well-defined values.  For fixed $z \in D$ and $r > 0$ so that $B(z,r) \subseteq D$, $h_{e^{-t}}(z) - h_r(z)$ evolves as a standard Brownian motion for all $t$ such that $e^{-t} \leq r$ \cite[Proposition~3.3]{ds2011kpz}.  (In the case that $h$ has mixed or free boundary conditions, this statement can be reduced to the zero-boundary case by using the Markov property as described above.)  If $h$ has Dirichlet or mixed boundary conditions or has free boundary conditions and the additive constant for $h$ has been fixed in a manner which does not depend on its values in $B(z,r)$, then $h_{e^{-t}}(z)$ for all $t$ such that $e^{-t} \leq r$ evolves as a standard Brownian motion.  If $h$ has free boundary conditions, $L$ is a linear segment of $\partial D$, and $z \in L$, then $h_{e^{-t}}(z) - h_r(z)$ evolves as $\sqrt{2}$ times a standard Brownian motion provided $e^{-t} \leq r$ and $r > 0$ is small enough so that $B(z,r) \cap \partial D = B(z,r) \cap L \neq \emptyset$ (see the discussion in \cite[Section~6]{ds2011kpz} and recall that the Neumann Green's function on $\h$ is given by $G(y,z) = -\log|y-z| - \log|y-\ol{z}|$ so that $G(0,z) = -2\log|z|$).  If the additive constant for $h$ has been fixed in a manner which does not depend on the values of $h$ in $B(z,r) \cap D$, then $h_{e^{-t}}(z)$ evolves as $\sqrt{2}$ times a standard Brownian motion for all $t$ such that $e^{-t} \leq r$.

Due to the following orthogonal decomposition for $H(\h)$ it is often convenient to take $D = \h$.  (As we explain momentarily, this leads to a two-step procedure for sampling a free-boundary GFF which will be important when we give the definition of a quantum wedge.)

\begin{lemma}
\label{lem::h_spaces_orthogonal}
Let $\CH_1(\h)$ be the subspace of $H(\h)$ which consists of functions which are radially symmetric about the origin and let $\CH_2(\h)$ be the subspace of $H(\h)$ which consists of functions which have a common mean about all semi-circles centered at the origin.  Then $\CH_1(\h)$ and $\CH_2(\h)$ give an orthogonal decomposition of $H(\h)$.
\end{lemma}

As we will see in the proof of Lemma~\ref{lem::h_spaces_orthogonal}, the radially symmetric part $f_1$ of $f$ is the function whose common value on $\partial B(0,r) \cap \h$ is given by the average of $f$ on $\partial B(0,r) \cap \h$.  Since $f$ is only defined modulo additive constant, so is $f_1$.  However, we emphasize that $f_2 = f-f_1$ does not depend on the additive constant for $f$, $f_1$ hence is a well-defined function and by construction has mean zero on all semi-circles $\partial B(0,r) \cap \h$.

\begin{proof}[Proof of Lemma~\ref{lem::h_spaces_orthogonal}]
If $f \in H(\h)$ then we can let $f_1$ be the function which on a given circle $\partial B(0,r) \cap \h$ is given by the average of $f$ on $\partial B(0,r) \cap \h$.  (Note that in the case $f \in C^\infty(\h)$ with $\| f\|_\nabla < \infty$ we have that $\|f_1 \|_\nabla \leq \|f \|_\nabla$, so the linear map $f \mapsto f_1$ is continuous hence extends to $H(\h)$.)  Let $f_2 = f-f_1$.  Then $f = f_1 + f_2$, $f_1 \in \CH_1(\h)$, and $f_2 \in \CH_2(\h)$.  Consequently,  $H(\h) = \CH_1(\h) \oplus \CH_2(\h)$.  To see that the sum is orthogonal, for $f_1 \in \CH_1(\h)$ and $f_2 \in \CH_2(\h)$ one computes $(f_1,f_2)_\nabla$ by putting $\nabla f_1$ and $\nabla f_2$ into polar coordinates.  In particular, the normal derivative of $f_1$ at $z \in \h$ depends only on $|z|$ while the mean of the normal derivative of $f_2$ on $\partial B(0,r) \cap \h$ is $0$ for each $r > 0$.  Thus the mean of the product of the normal derivatives of $f_1$ and $f_2$ on $\partial B(0,r) \cap \h$ is equal to $0$ for each $r > 0$.  Similarly, the tangential derivative of $f_1$ on $\partial B(0,r) \cap \h$ is $0$ for each $r > 0$ so the mean of the product of the tangential derivatives of $f_1$ and $f_2$ is equal to $0$.
\end{proof}

Lemma~\ref{lem::h_spaces_orthogonal} implies that we can sample a free boundary GFF on $\h$ with the additive constant fixed so that its average on $\h \cap \partial \D$ is equal to $0$ in two steps:
\begin{enumerate}
\item Sample its projection onto $\CH_1(\h)$ to be given by the function whose common value on $\partial B(0,e^{-t}) \cap \h$ is given by $B_{2t}$ where $B$ is a two-sided standard Brownian motion with $B_0 = 0$.
\item Sample its projection onto $\CH_2(\h)$ independently by taking it to be $\sum_n \alpha_n f_n$ where $(f_n)$ is an orthonormal basis of $\CH_2(\h)$ and $(\alpha_n)$ is a sequence of i.i.d.\ $N(0,1)$ random variables.
\end{enumerate}
We then fix the additive constant for both projections so that their average on $\partial \D \cap \partial \h$ vanishes.

\subsubsection{GFFs on strips and cylinders} \label{subsubsec::stripandcylinder}

In many places in this article it will be convenient to work on the infinite strip $\strip = \R \times [0,\pi]$.  Suppose that $h$ is a free boundary GFF on $\strip$ with additive constant fixed so that its average on $[0, i \pi]$ vanishes.  Let $\psi \colon \strip \to \h$ be the conformal transformation given by $z \mapsto e^{z}$.  Then $\wt{h} = h \circ \psi^{-1}$ is a free boundary GFF on $\h$ with additive constant fixed so that its average on $\partial \D \cap \partial \h$ vanishes.  Let $\wt{h}_\epsilon(z)$ be the circle average process associated with $\wt{h}$.  Note $\psi$ maps the vertical line $u + [0,i\pi]$ to the half-circle $\partial B(0,e^u) \cap \h$.  Since $\wt{h}_{e^{-t}}(0)$ evolves as $\sqrt{2}$ times a standard Brownian motion, it follows that the average of $h$ on $u+[0,i\pi]$ evolves as $\sqrt{2}$ times a standard Brownian motion in $u$.

In analogy with Lemma~\ref{lem::h_spaces_orthogonal}, we have the following orthogonal decomposition for $H(\strip)$:
\begin{lemma}
\label{lem::strip_spaces_orthogonal}
Let $\CH_1(\strip)$ be the subspace of $H(\strip)$ which consists of functions which are constant on vertical lines of the form $u+[0,i\pi]$ for $u \in \R$ and let $\CH_2(\strip)$ be the subspace of $H(\strip)$ which consists of functions which have a common mean on all such vertical lines.  Then $\CH_1(\strip)$ and $\CH_2(\strip)$ give an orthogonal decomposition of $H(\strip)$.
\end{lemma}
\begin{proof}
This follows from the same argument as in the proof of Lemma~\ref{lem::h_spaces_orthogonal}.  Alternatively, the result can be deduced from Lemma~\ref{lem::h_spaces_orthogonal} using the conformal invariance of the Dirichlet inner product and the map $\psi$ defined above.
\end{proof}

In analogy with the case of a free boundary GFF on $\h$, we can use Lemma~\ref{lem::strip_spaces_orthogonal} to sample a free boundary GFF on $\strip$ in two steps by first sampling the Brownian motion which determines its averages on vertical lines and then sampling independently its orthogonal projection onto $\CH_2(\strip)$.  We can fix the additive constant for $h$ by taking the additive constant for its projection onto $\CH_2(\strip)$ to be so that its average on $[0,\pi i]$ vanishes and then fixing the additive constant for its projection onto $\CH_1(\strip)$ in some way to be specified.   This will be useful in Section~\ref{subsec::surfaces_strips_cylinders} below where we will want to sample a Poissonian collection of fields on $\strip$.

\subsection{Thick quantum wedges}
\label{subsec::thick_wedges}

We will now give a precise description of a quantum wedge.  In this section, we will focus on the case of a wedge which (together with its prime-end boundary) is homeomorphic to $\ol{\h}$ and will refer to such a wedge as being ``thick.''  In Section~\ref{subsec::surfaces_strips_cylinders}, we will give the definition of a quantum wedge which is ``thin'' (i.e., not homeomorphic to $\h$).  The contents of this section are similar to that of \cite[Section~1.6]{she2010zipper}, however the presentation here is somewhat different.

Fix a proper simply connected domain $\wt{D} \subseteq \C$.  Suppose that we have a quantum surface represented by $(\wt{D},\wt{h})$ with two distinct marked boundary points.  Then we can apply the change of coordinates formula~\eqref{eqn::Qmap} to represent the surface as $(\h,h,0,\infty)$ for some $h$ and the two marked points are taken to be $0$ and $\infty$.  We will focus on the case that there is a.s.\ a finite amount of both~$\mu_h$ and~$\nu_h$ mass in each bounded neighborhood of the origin and infinite mass in each neighborhood of $\infty$.  We will refer to such a surface as {\bf half-plane-like} since it has a distinguished point of ``infinity'' whose neighborhoods have an infinite amount of area and boundary length and a distinguished ``origin,'' sufficiently small neighborhoods of which have a finite amount of area and boundary length.  In general, we will refer to a surface with two distinguished boundary points as being {\bf doubly marked}.

Note that the $h$ which describes the surface is canonical except that it has one free parameter corresponding to constant rescalings of $\h$.  There are various ways of fixing this parameter.  In practice, we will typically fix this parameter so that if $r = \sup\{s > 0 : h_s(0) + Q\log s = 0\}$ then $r=1$ where $h_s(0)$ denotes the average of $h$ on $\h \cap \partial B(0,s)$.  This choice of scaling yields the so-called {\bf circle-average embedding.}  As we will see, this particular choice is convenient mathematically.

\begin{remark}
\label{rem::bm_drift_conditioned}
In many places in this article (e.g., Definition~\ref{def::quantum_wedge} below) we will consider processes of the form $B_t + a t$ where $B$ is a standard Brownian motion with $B_0 \geq 0$ and $a > 0$ \emph{conditioned} on the event that $B_t + at > 0$ for all $t > 0$.  If $B_0 > 0$, then this is a positive probability event so there is no difficulty in making sense of this process.  In the case that $B_0 = 0$, however, we are conditioning on an event which occurs with zero probability.  Such a continuous process $X_t$ is characterized by the properties that $X_0 = 0$, $X_t > 0$ for all $t > 0$, and for each $\epsilon > 0$ if we let $\tau_\epsilon = \inf\{t \geq 0 : X_t = \epsilon\}$ then $X_{t + \tau_\epsilon}$ is equal in distribution to a standard Brownian motion with linear drift $a t$ starting from $\epsilon$ and conditioned on the (positive probability) event that it does not subsequently hit $0$.  Indeed, suppose that $\wt{X}_t$ is another such process and $\wt{\tau}_\epsilon = \inf\{t \geq 0 : \wt{X}_t = \epsilon\}$.  Then for each $\epsilon > 0$ we can construct a coupling of $X$, $\wt{X}$ onto a common probability space so that $X_{t + \tau_\epsilon} = \wt{X}_{t + \wt{\tau}_\epsilon}$ for all $t > 0$.  Note that $\tau_\epsilon,\wt{\tau}_\epsilon \to 0$ a.s.\ as $\epsilon \to 0$.  Thus it is not difficult to see that in any subsequential limit as $\epsilon \to 0$ of this coupling, one has that $X = \wt{X}$.  One can show that such a process exists by first sampling a Brownian motion $\wt{B}$ with $\wt{B}_0 = 0$, letting $\tau$ be the last time that $\wt{B}_t + at$ hits $0$ (this time is a.s.\ finite due to the positive drift), and then taking $B_t + at = \wt{B}_{t + \tau} + a(t+\tau)$.  When $a=0$, it is a well-known fact that one can make sense of a standard Brownian motion $B_t$ with $B_0  = 0$ conditioned on $B_t > 0$ for all $t > 0$ by taking it to be a $\bes^3$ process.

One can also make sense of $B_t + at$ with $B_0 = 0$ conditioned on $B_t + at > 0$ for all $t > 0$ in a unified manner when $a \geq 0$ using Bessel processes.  Let us first consider the case $a > 0$.  Proposition~\ref{prop::bessel_exponential_bm} implies that if $X_t$ is a $\bes^{2+2a}$ process then $\log X_t$, reparameterized to have quadratic variation $dt$, evolves as a standard Brownian motion with linear drift $at$.  Suppose that $X_0 = 0$ and $\tau$ is the last time that $X$ hits $1$.  Then $\log X_{t+ \tau}$, reparameterized to have quadratic variation $dt$, has the law of $B_t + at$ with $B_0 = 0$ conditioned on $B_t + a t > 0$ for all $t > 0$.

In order to extend this perspective in a unified manner to the case $a \geq 0$, one has to do something that at first sight might seem strange.  Let again $X_t$ be a $\bes^{2+2a}$ with $X_0 = 0$ and now let $\tau$ be the first time that $X_t$ hits $1$.  Then the time-reversal $-\log X_{\tau-t}$ starts from $0$ and does not subsequently hit $0$.  Then when $a > 0$ this process corresponds to $B_t + at$ conditioned on $B_t + at > 0$ for all $t \geq 0$ as above and when $a = 0$ it corresponds to a $\bes^3$ process, also as above.  This process is in fact defined for all times and can be thought of as a Brownian motion defined for all times where $t=0$ is the last time that it hits $0$.  This treatment is useful in the setting of defining quantum wedges since it allows us to give a unified definition for all wedge weights.
\end{remark}

\begin{definition}
\label{def::quantum_wedge}
Fix $\alpha < Q$.  An {\bf $\alpha$-quantum wedge} is the doubly marked quantum surface (when parameterized by $\h$) described by the distribution $h$ on $\h$ whose law can be sampled from as follows.  Let $A_s$ be the process defined for $s \in \R$ such that
\begin{enumerate}
\item For $s > 0$, $A_s = B_{2s} + \alpha s$ where $B$ is a standard Brownian motion with $B_0 = 0$ and
\item For $s < 0$, $A_s = \wh{B}_{-2s}  + \alpha s$ where $\wh{B}$ is a standard Brownian motion independent of $B$ with $\wh{B}_0 = 0$ conditioned so that $\wh{B}_{2t} + (Q-\alpha) t > 0$ for all $t > 0$.
\end{enumerate}
Let $\CH_1(\h)$ and $\CH_2(\h)$ be as in the statement of Lemma~\ref{lem::h_spaces_orthogonal}.  Then $h$ is the field with projection onto $\CH_1(\h)$ given by the function whose common value on $\partial B(0,e^{-s}) \cap \h$ is equal to $A_s$ for each $s \in \R$ and with projection onto $\CH_2(\h)$ given by the corresponding projection of a free boundary GFF on $\h$.  That is, the projection of $h$ onto $\CH_2(\h)$ is given by $\sum_n \alpha_n f_n$ where $(f_n)$ is an orthonormal basis of $\CH_2(\h)$ and $(\alpha_n)$ is an i.i.d.\ $N(0,1)$ sequence.  We fix the additive constant for the projection of $h$ onto $\CH_1(\h)$ (resp.\ $\CH_2(\h)$) so that its average on $\partial \D \cap \h$ vanishes.
\end{definition}

Recall that a quantum surface is only defined modulo the equivalence relation~\eqref{eqn::Qmap}.  Since a quantum wedge has two marked points (the origin point and the infinity point), if we parameterize it by $\h$ with the origin point taken to $0$ and the infinity point to $\infty$, then we are left with one free parameter in the choice of embedding.  The particular embedding of a quantum wedge into $\h$ given in Definition~\ref{def::quantum_wedge} is the \emph{circle average embedding} because the scaling parameter is chosen based on the average $h_{e^{-t}}(0)$ of the field on semicircles $\h \cap \partial B(0,e^{-t})$.  The circle average embedding is the one for which if we set $\tau = \sup\{r \geq 0 : h_r(0) + Q \log r = 0\}$ then $\tau = 1$.  Note that if we perform the change of coordinates $z \mapsto z/r$, then by~\eqref{eqn::Qmap} we must add $Q \log r$ to the field.  Thus the circle average embedding can equivalently be described as the one such that $r = 1$ is the largest value of $r$ so that if we were to perform the rescaling which takes $\h \cap \partial B(0,r)$ to $\h \cap \partial \D$ then the average of the field on $\h \cap \partial \D$ is equal to $0$.  In practice, this is a very convenient choice because the law of the field in $\h \cap \D$ takes a simple form.  In particular, if $(\h,h,0,\infty)$ is an $\alpha$-quantum wedge with $h$ as in Definition~\ref{def::quantum_wedge}, then the law of $h$ restricted to $\D \cap \h$ is the same as the law of the corresponding restriction of $\wt{h} - \alpha \log|\cdot|$ where $\wt{h}$ is a free boundary GFF on $\h$ with the additive constant fixed so that its average on $\h \cap \partial \D$ is equal to $0$.

As mentioned in Table~\ref{tab::wedge_parameterization}, there are many different variables that one can use to parameterize the space of quantum wedges.  In this article, we will mainly use $\alpha$ (multiple of $\log$ singularity at the origin) and weight $W$ (as defined in~\eqref{eqn::wedge_weight}).  Note that the constraint $\alpha < Q$ corresponds to $W > \tfrac{\gamma^2}{2}$.  The critical value $\alpha = Q$ and $W = \tfrac{\gamma^2}{2}$ also describes a quantum surface which (together with its prime-end boundary) is homeomorphic to~$\ol{\h}$.  However, it is convenient to give the definition of the surface in this case using Bessel processes in Section~\ref{subsec::surfaces_strips_cylinders}.

\begin{remark}
\label{rem::fat_quantum_wedge_strip}
As we will see in Section~\ref{subsec::surfaces_strips_cylinders}, it is also natural to parameterize an $\alpha$-quantum wedge by the strip $\strip$.  Note that if we start with a surface parameterized by $\h$ and we change coordinates using $z \mapsto \log(z)$ to parameterize it by $\strip$, then applying~\eqref{eqn::Qmap} involves adding the function $Q \re(z)$.  Suppose that $\CH_1(\strip),\CH_2(\strip)$ are as in Lemma~\ref{lem::strip_spaces_orthogonal}.  It thus follows that we can sample from the law of a distribution~$h$ which represents an $\alpha$-quantum wedge parameterized by $\strip$ by sampling its projection onto $\CH_1(\strip)$ from the law of the function whose common value on $u + [0,i\pi]$ is $\wt{A}_u$ where:
\begin{enumerate}
\item For $u > 0$, $\wt{A}_u = B_{2u} + (Q-\alpha) u$ where $B$ is a standard Brownian motion with $B_0 = 0$ conditioned so that $B_{2u} + (Q-\alpha)u > 0$ for all $u > 0$ and
\item For $u < 0$, $\wt{A}_u = \wh{B}_{-2u} + (Q-\alpha)u$ where $\wh{B}$ is a standard Brownian motion independent of $B$ with $\wh{B}_0 = 0$.
\end{enumerate}
The projection of $h$ onto $\CH_2(\strip)$ can be sampled independently from the law of the corresponding projection of a free boundary GFF on $\strip$.  We fix the additive constant for both projections so that their average on $[0,i \pi]$ vanishes.

As in the case of the parameterization of a quantum wedge by $\h$, when we parameterize it by $\strip$ and we take the origin (resp.\ infinity) to correspond to $-\infty$ (resp.\ $+\infty$), we have the free parameter which corresponds to the horizontal translation when we specify an embedding.  The embedding described above is the one so that the projection of the field onto $\CH_1(\strip)$ last hits the value $0$ at $u=0$.  This is the reason for the asymmetry in the definition of $\wt{A}$.  Another natural choice is to take the horizontal translation to be the one so that the projection first hits the value $0$ at $u=0$.

If we chose instead to have the origin (resp.\ infinity) correspond to $+\infty$ (resp.\ $-\infty$) then we would instead take $\wt{A}_u$ to be:
\begin{enumerate}
\item For $u > 0$, $\wt{A}_u = B_{2u} + (\alpha-Q) u$ where $B$ is a standard Brownian motion with $B_0 = 0$ and
\item For $u < 0$, $\wt{A}_u = \wh{B}_{-2u} + (\alpha-Q)u$ where $\wh{B}$ is a standard Brownian motion independent of $B$ with $\wh{B}_0 = 0$ conditioned so that $\wh{B}_{2t} + (Q-\alpha) t > 0$ for all $t > 0$.
\end{enumerate}
This process and the one above agree up to a reflection and horizontal translation of $\strip$, which means that the corresponding fields differ by a conformal map hence one obtains a quantum surface with the same law. 
\end{remark}

We note that by Proposition~\ref{prop::bessel_exponential_bm}, the construction of a quantum wedge described in Remark~\ref{rem::fat_quantum_wedge_strip} is equivalent to the definition described in Figure~\ref{fig::besseltowedge}.

We are now going to show that the quantum wedge arises as a certain type of infinite volume limit of the local behavior of a free boundary GFF near a point with a certain type of $\log$ singularity.  The particular case of a $\gamma$-quantum wedge arises upon taking an infinite volume limit near a point chosen from the quantum boundary length measure (which should be thought of as the LQG analog of a Benjamini-Schramm limit near a boundary typical point).  As a consequence of this, we will deduce that the law of a quantum wedge (as a quantum surface) is invariant under the operation of multiplying its area by a constant.  We say that a sequence of doubly marked quantum surfaces $(\h,h_n,0,\infty)$ converges to the doubly marked surface $(\h,h,0,\infty)$ if the distribution which describes the circle average embedding of $(\h,h_n,0,\infty)$ converges to the distribution which describes the circle average embedding of $(\h,h,0,\infty)$.

Note that if $(D,h)$ is a quantum surface and $C \in \R$ is a constant then replacing $h$ with $h+C/\gamma$ corresponds to multiplying the $\mu_h$ area of $D$ by $e^C$.

\begin{proposition}
\label{prop::quantum_wedge_properties}
Fix $\alpha < Q$.  Then the following hold:
\begin{enumerate}[(i)]
\item \label{it::quantum_wedge_invariance} The law of an $\alpha$-quantum wedge is invariant under the operation of multiplying its area by a constant.  That is, if $(\h,h,0,\infty)$ is the circle average embedding of an $\alpha$-quantum wedge, we fix $C \in \R$, and then let $(\h,\wt{h},0,\infty)$ be the circle average embedding of $(\h,h+C/\gamma,0,\infty)$, then $(\h,h,0,\infty) \stackrel{d}{=} (\h,\wt{h},0,\infty)$.
\item \label{it::quantum_wedge_limit} Suppose that $\wt{h}$ is a free boundary GFF on $\h$ and let $h = \wt{h} - \alpha \log|\cdot|$.  Assume that the additive constant for $h$ is fixed so that its average on $\h \cap \partial \D$ is equal to $0$.  Then the quantum surfaces $(\h,h + C/\gamma,0,\infty)$ equipped with the circle average embedding converge weakly as $C \to \infty$ in the space of distributions to an $\alpha$-quantum wedge with the circle average embedding.
\end{enumerate}
\end{proposition}

We remark that part~\eqref{it::quantum_wedge_invariance} of Proposition~\ref{prop::quantum_wedge_properties} does not apply to unscaled quantum wedges because an unscaled quantum wedge is only defined modulo additive constant.  If one were to fix the additive constant for an unscaled wedge in some way, for example by setting its average on $\h \cap \partial \D$ to be equal to $0$, then the first part of Proposition~\ref{prop::quantum_wedge_properties} still would not apply.

\begin{proof}[Proof of Proposition~\ref{prop::quantum_wedge_properties}]
The proofs of both parts of the proposition are explained in the text of \cite[Section~1.6]{she2010zipper}.  For the convenience of the reader, we are going to provide a self-contained treatment here.  Let $h = \wt{h} - \alpha \log|\cdot|$ where $\wt{h}$ is a free boundary GFF on $\h$ and $\alpha < Q$.  We fix the additive constant for $h$ as in the statement of the proposition so that its average on $\partial \D \cap \h$ is equal to $0$.  For each $\epsilon > 0$ and $z \in \ol{\h}$, we let $h_\epsilon(z)$ be the average of $h$ on $\partial B(z,\epsilon) \cap \h$.  Due to how we have fixed the additive constant for $h$, we have that $h_1(0) = 0$.  Then $h_{e^{-t}}(0)$ for $t > 0$ evolves as $B_{2t} + \alpha t$ where $B$ is a standard Brownian motion; for $t < 0$, it evolves as $\wh{B}_{-2t} + \alpha t$ where $\wh{B}$ is a standard Brownian motion which is independent of $B$.

For each $C \in \R$ we let $h^C = h + C/\gamma$.  Then $\mu_{h^C} = e^C \mu_h$.  For each $a > 0$, we let $h^{a,C} = h^C(a \cdot) + Q \log a$.  By~\eqref{eqn::Qmap}, considering $h^{a,C}$ in place of $h^C$ corresponds to scaling spatially by the constant factor $a^{-1}$.  As usual, for each $r > 0$ we let $h_r^{a,C}(0)$ be the average of $h^{a,C}$ on $\h \cap \partial B(0,r)$.  Let
\[ \epsilon_0^C = \sup\left\{ \epsilon \in [0,1] : h_1^{\epsilon,C}(0)= 0\right\} = \sup\left\{ \epsilon \in [0,1] : h_\epsilon(0) + Q\log \epsilon + C/\gamma = 0\right\}.\]
That is, $\epsilon_0^C$ gives the largest value of $\epsilon \in [0,1]$ such that if we add $C/\gamma$ to $h$ and then rescale by $\epsilon^{-1}$ so that $B(0,\epsilon)$ is mapped to $B(0,1)$ then the average of the rescaled field on $\h \cap \partial \D$ is equal to $0$.  We note that for each $C > 0$ there a.s.\ exists such an $\epsilon \in [0,1]$ since $h_{e^{-t}}(0) - Q t + C/\gamma$ for $t \geq 0$ evolves as a Brownian motion with negative drift (recall $\alpha < Q$) and starts from a positive value.  Thus we can equivalently let
\[ t_0^C = \inf\{t \geq 0 : h_{e^{-t}}(0) - Q t + C/\gamma = 0\}\]
(which is then obviously finite) and set $\epsilon_0^C = e^{-t_0^C}$.

We claim that:
\begin{enumerate}[(a)]
\item\label{it::brownian_forward} For $t > 0$, $h_{e^{-t}}^{\epsilon_0^C,C}(0)$ evolves as $B_{2t} + \alpha t$ where $B$ is a standard Brownian motion with $B_0 = 0$.
\item\label{it::brownian_middle} As $t$ decreases from $0$ to $\log \epsilon_0$, $h_{e^{-t}}^{\epsilon_0^C,C}(0) - Q t$ evolves as $\wh{B}_{-2t}+(\alpha - Q) t$, $\wh{B}$ a standard Brownian motion with $\wh{B}_0 = 0$, conditioned on being positive in $[\log \epsilon_0,0)$ and run until the last time it hits $C/\gamma$.  Equivalently, $h_{e^{-t}}^{\epsilon_0^C,C}(0)$ evolves as $\wh{B}_{-2t} + \alpha t$ conditioned on being larger than $Q t$ for $t \in [\log \epsilon_0,0)$ and run up until the last time that $h_{e^{-t}}^{\epsilon_0^C,C}(0) - Q t$ hits $C/\gamma$.
\item\label{it::brownian_backward} As $t$ decreases from $\log \epsilon_0$ to $-\infty$, $h_{e^{-t}}^{\epsilon_0^C,C}(0)$ evolves as $\wt{B}_{-2t} + \alpha t$ where $\wt{B}$ is a standard Brownian motion (with no conditioning).
\end{enumerate}
Indeed, \eqref{it::brownian_forward} follows from the strong Markov property for Brownian motion.  The same is also true for~\eqref{it::brownian_backward}.  To see why~\eqref{it::brownian_middle} holds, we note that $h_{e^{-t}}(0) - Qt + C/\gamma$ for $t \in [0,t_0^C]$ evolves as the process $B_{2t} + (\alpha-Q) t$ starting from the value $C/\gamma$ and stopped at the first time it hits $0$.  Therefore its time-reversal evolves as $\wh{B}_{-2t} + (\alpha - Q) t$ for $t \leq 0$ conditioned on being positive and stopped at the last time it hits $C/\gamma$ (Propositions~\ref{prop::bessel_exponential_bm} and~\ref{prop::bessel_time_reversal}, or alternatively \cite[Theorem~2.5]{w1974pathdecomp}).

It is therefore easy to see that if we take a limit as $C \to \infty$, then the law of $r \mapsto h_r^{\epsilon_0^C,C}(0)$ (i.e., the projection of $h$ onto $\CH_1(\h)$ with additive constant fixed to vanish on $\partial \D \cap \h$) converges weakly with respect to the topology of local uniform convergence to the law of $r \mapsto A_{\log r^{-1}}$ where $A_s$ is as in Definition~\ref{def::quantum_wedge}.  On the other hand, this rescaling procedure does not affect the projection of $h$ onto the subspace $\CH_2(\h)$ with additive constant fixed to vanish on $\partial \D \cap \h$.

Part~\eqref{it::quantum_wedge_invariance} also follows from this argument because if we fix $u > 0$ and take $C' = C+u$, the law of the circle average embeddings of the surfaces associated with the field with area rescaled by $e^C$ and the field with area rescaled by $e^{C'}$ will have the same limit as $C \to \infty$ (with $u$ fixed).  (We also note that part~\eqref{it::quantum_wedge_invariance} can be seen directly from the definition of a quantum wedge.)
\end{proof}

We finish with the following characterization of an $\alpha$-quantum wedge which will be useful later in the article.  (This characterization, along with its companion Proposition~\ref{prop::quantum_cone_characterization} for quantum cones, may seem trivial.  However, both results will be useful for showing that certain surfaces are $\alpha$-quantum wedges or cones which at first glance may not appear to be so.)

\begin{proposition}
\label{prop::quantum_wedge_characterization}
Fix $\alpha < Q$ and suppose that $(\h,h,0,\infty)$ has the circle average embedding of a quantum surface parameterized by $\h$ such that the following hold:
\begin{enumerate}[(i)]
\item The law of $(\h,h,0,\infty)$ (as a quantum surface) is invariant under the operation of multiplying its area by a constant.  That is, if we fix $C \in \R$, and then let $(\h,\wt{h},0,\infty)$ be the circle average embedding of $(\h,h+C/\gamma,0,\infty)$, then $(\h,h,0,\infty) \stackrel{d}{=} (\h,\wt{h},0,\infty)$.
\item The total variation distance between the law of $(\h,h,0,\infty)$ restricted to $B(0,r) \cap \h$ and the law of an $\alpha$-quantum wedge $(\h,h_\alpha,0,\infty)$ restricted to $B(0,r) \cap \h$ tends to $0$ as $r \to 0$.
\end{enumerate}
Then $(\h,h,0,\infty)$ has the law of an $\alpha$-quantum wedge.
\end{proposition}
\begin{proof}
Observe that if $R,\epsilon > 0$ are fixed then there exists $C_0 > 0$ such that for all $C \geq C_0$ we can couple together the circle average embedding of $(\h,h+C/\gamma,0,\infty)$ with the circle average embedding of $(\h,h_\alpha+C/\gamma,0,\infty)$ in $B(0,R)$ so that they agree with probability at least $1-\epsilon$.  Thus sending $C \to \infty$ yields an asymptotic coupling where both surfaces agree with probability $1$.  The claim then follows since the laws of both surfaces are invariant under the operation of multiplying their area by a constant.  
\end{proof}

\subsection{Quantum cones}
\label{subsec::cones}

We are now going to give a precise description of a quantum cone.  This will be an infinite volume surface without boundary which is homeomorphic to $\C$ with a marked interior point and will be derived from the whole-plane GFF.  Like quantum wedges, quantum cones are also doubly marked where the marked interior point plays the role of the origin and $\infty$ plays the role of the marked point whose neighborhoods have infinite mass.  The discussion will be largely parallel with that given in Section~\ref{subsec::thick_wedges}.  In order to give the definition of a quantum cone, we first need to record the whole-plane analog of Lemma~\ref{lem::h_spaces_orthogonal}.

\begin{lemma}
\label{lem::c_spaces_orthogonal}
Let $\CH_1(\C)$ be the subspace of $H(\C)$ which consists of functions which are radially symmetric about the origin and let $\CH_2(\C)$ be the subspace of $H(\C)$ which consists of functions which have common mean on all circles centered at the origin.  Then $\CH_1(\C)$ and $\CH_2(\C)$ give an orthogonal decomposition of $H(\C)$.
\end{lemma}
\begin{proof}
This is proved in the same manner as Lemma~\ref{lem::h_spaces_orthogonal}.
\end{proof}

\begin{definition}
\label{def::quantum_cone}
Fix $\alpha < Q$.  An {\bf $\alpha$-quantum cone} is the doubly marked quantum surface (when parameterized by $\C$) described by the distribution $h$ on $\C$ whose law can be sampled from as follows.  Let $A_s$ be as in Definition~\ref{def::quantum_wedge} except with $B_{2s}$ and $\wh{B}_{2s}$ replaced by $B_s$ and $\wh{B}_s$.  Let $\CH_1(\C)$ and $\CH_2(\C)$ be as in the statement of Lemma~\ref{lem::c_spaces_orthogonal}.  Then $h$ is the field with projection onto $\CH_1(\C)$ given by the function whose common value on $\partial B(0,e^{-s})$ is equal to $A_s$ for each $s \in \R$ and with projection onto $\CH_2(\C)$ is given by the corresponding projection of a whole-plane GFF.  That is, the projection of $h$ onto $\CH_2(\C)$ given by $\sum_n \alpha_n f_n$ where $(f_n)$ is an orthonormal basis of $\CH_2(\C)$ and $(\alpha_n)$ is an i.i.d.\ $N(0,1)$ sequence.  We fix the additive constant for the projection of $h$ onto $\CH_1(\C)$ (resp.\ $\CH_2(\C)$) so that its average on $\partial \D$ vanishes.
\end{definition}

\begin{remark}
In principle, one could also define the $\alpha$-quantum cone for the case $\alpha = Q$, which would correspond to $W=0$ (and a Bessel process of dimension $2$), but we will not actually have any use for that case in this paper.  By contrast, the $\alpha$-quantum {\em wedge} with $\alpha = Q$ corresponds to a positive weight ($W = \gamma^2/2$) which happens to be the thin-versus-thick critical point.
\end{remark}

The same remark made after Definition~\ref{def::quantum_wedge} also applies here with respect to different choices of embeddings of a quantum cone.  The particular embedding described in Definition~\ref{def::quantum_cone} is the \emph{circle average embedding}.

Note that if $(\C,h,0,\infty)$ is an $\alpha$-quantum cone with $h$ as in Definition~\ref{def::quantum_cone}, then the law of $h$ restricted to $\D$ is the same as the law of the restriction to $\D$ of $\wt{h} - \alpha \log|\cdot|$ where $\wt{h}$ is a whole-plane GFF with the additive constant fixed so that its average on $\partial \D$ is equal to $0$.

As mentioned in Table~\ref{tab::cone_parameterization}, there are many different variables that one can use to parameterize the space of quantum cones.  In this article, we will mainly use $\alpha$ (multiple of $\log$ singularity at the origin) and weight $W$ (as defined in~\eqref{eqn::quantum_cone_weight}).  Note that the constraint $\alpha < Q$ corresponds to $W > 0$.

\begin{remark}
\label{rem::quantum_cone_cyl}
As we will see in Section~\ref{subsec::surfaces_strips_cylinders}, it is also natural to parameterize an $\alpha$-quantum cone by the infinite cylinder $\cyl$.  Note that if we start with a surface parameterized by $\C$ and we change coordinates using $z \mapsto \log(z)$ to parameterize it by $\cyl$, then applying~\eqref{eqn::Qmap} involves adding the function $Q \re(z)$.  Suppose that $\CH_1(\cyl),\CH_2(\cyl)$ are as in Lemma~\ref{lem::strip_spaces_orthogonal} (with $\cyl$ in place of $\strip$).  It thus follows that we can sample from the law of a distribution $h$ which represents an $\alpha$-quantum cone parameterized by $\cyl$ by sampling its projection onto $\CH_1(\cyl)$ from the law of the function whose common value on $u + [0,2\pi i]$ is $\wt{A}_u$ where:
\begin{enumerate}
\item For $u > 0$, $\wt{A}_u = B_{u} + (Q-\alpha) u$ where $B$ is a standard Brownian motion with $B_0 = 0$ and
\item For $u < 0$, $\wt{A}_u = \wh{B}_{-u} + (Q-\alpha)u$ where $\wh{B}$ is a standard Brownian motion independent of $B$ with $\wh{B}_0 = 0$ conditioned so that $\wh{B}_t + (\alpha-Q) < 0$ for all $t > 0$.
\end{enumerate}
The projection of $h$ onto $\CH_2(\cyl)$ can be sampled from independently from the law of the corresponding projection of a GFF on $\cyl$.  With this choice, the origin (resp.\ infinity) corresponds to $-\infty$ (resp.\ $+\infty$).  We fix the additive constant for the projections of $h$ onto $\CH_1(\cyl)$, $\CH_2(\cyl)$ so that their average on $[0,2\pi i]$ vanishes. Remark~\ref{rem::fat_quantum_wedge_strip} also applies here to explain the asymmetry in the definition of $\wt{A}$.

If we wish to take the origin (resp.\ infinity) to correspond to $+\infty$ (resp.\ $-\infty$), then we take $\wt{A}_u$ to be:
\begin{enumerate}
\item For $u > 0$, $\wt{A}_u = B_{u} + (\alpha-Q) u$ where $B$ is a standard Brownian motion with $B_0 = 0$ and
\item For $u < 0$, $\wt{A}_u = \wh{B}_{-u} + (\alpha-Q)u$ where $\wh{B}$ is a standard Brownian motion independent of $B$ with $\wh{B}_0 = 0$ conditioned so that $\wh{B}_t + (Q-\alpha) t > 0$ for all $t > 0$.
\end{enumerate}
This process and the one above agree up to a reflection and horizontal translation of $\cyl$, which is why one obtains a quantum surface with the same law.
\end{remark}

We next record the analog of Proposition~\ref{prop::quantum_wedge_properties}.

\begin{proposition}
\label{prop::quantum_cone_properties}
Fix $\alpha < Q$.  Then the following hold:
\begin{enumerate}[(i)]
\item \label{it::quantum_cone_invariance} The law of an $\alpha$-quantum cone is invariant under the operation of multiplying its area by a constant.  That is, if $(\C,h,0,\infty)$ is the circle average embedding of an $\alpha$-quantum cone, we fix $C \in \R$, and then let $(\C,\wt{h},0,\infty)$ be the circle average embedding of $(\C,h+C/\gamma,0,\infty)$, then $(\C,h,0,\infty) \stackrel{d}{=} (\C,\wt{h},0,\infty)$.
\item \label{it::quantum_cone_limit} Suppose that $D \subseteq \C$ is a domain with $0 \in D$ and that $\wt{h}$ is a GFF on $D$ with any type of boundary conditions if $D \neq \C$ or a whole-plane GFF if $D = \C$.  Let $h = \wt{h} - \alpha \log|\cdot|$.  If $h$ is only defined modulo additive constant, we fix $r > 0$ so that $B(0,r) \subseteq D$ and then fix the additive constant by taking $h_r(0) = 0$.  Then the quantum surfaces $(\C,h + C/\gamma,0,\infty)$ equipped with the circle average embedding converge weakly in the space of distributions as $C \to \infty$ to an $\alpha$-quantum wedge with the circle average embedding.
\end{enumerate}
\end{proposition}
\begin{proof}
We first consider part~\eqref{it::quantum_cone_limit}.  For each $C \in \R$ we let $h^C = h+C/\gamma$.  For each $a > 0$, we let $h^{a,C} = h^C(a \cdot) + Q \log a$.  For each $r > 0$, we let $h_r^{a,C}(0)$ be the average of $h_r^{a,C}$ on $\partial B(0,r)$.  Let
\[ \epsilon_0^C = \sup\{ \epsilon \in [0,1] : h_1^{\epsilon,C}(0) = 0 \} = \sup\{ \epsilon \in [0,1] : h_\epsilon(0) + Q \log \epsilon + C/\gamma = 0\}.\]
The fact that the projection of $h^{\epsilon_0^C,C}$ onto $\CH_1(\C)$ (with the additive constant fixed) converges weakly with respect to the topology of local uniform convergence to the corresponding projection of an $\alpha$-quantum cone follows from the same argument as given in Proposition~\ref{prop::quantum_wedge_properties}.

To complete the proof of part~\eqref{it::quantum_cone_limit}, we need to explain why the projection of $h^{\epsilon_0^C,C}$ onto $\CH_2(\C)$ converges weakly as $C \to \infty$ in the space of distributions to the corresponding projection of an $\alpha$-quantum cone.  Recall that the latter has the same law as the corresponding projection of a whole-plane GFF onto $\CH_2(\C)$.  We will prove the result in the case that $\wt{h}$ has Dirichlet boundary conditions; the other cases follow from the same argument.   By rescaling $D$ if necessary, we may assume without loss of generality that $\D \subseteq D$.  By the Markov property, we can write $\wt{h} = \wt{h}^0 + \wt{\Fh}$ where $\wt{h}^0$ is a zero boundary GFF on $\D$ and $\wt{\Fh}$ is harmonic on $\D$.  Let $\wh{h}$ be a whole-plane GFF with the additive constant fixed so that its average on $\partial \D$ is equal to $0$.  Then we can write $\wh{h} = \wh{h}^0 + \wh{\Fh}$ where $\wh{h}^0$ is a zero boundary GFF on $\D$ and $\wh{\Fh}$ is harmonic on $\D$.  We assume that $\wt{h}$, $\wh{h}$ are coupled together so that $\wt{h}^0 = \wh{h}^0$.  Fix $\phi \in \CH_2(\C)$.  By the scale-invariance of $\wh{h}$, we have that
\[ (\wh{h},\phi)_\nabla \stackrel{d}{=} (\wh{h}(\epsilon_0^C \cdot),\phi)_\nabla = (\wh{h}^0(\epsilon_0^C \cdot),\phi)_\nabla + (\wh{\Fh}(\epsilon^C \cdot),\phi)_\nabla.\]
Since $\wh{\Fh}$ is a $C^\infty$ function in $\D$, we have that $(\wh{\Fh}(\epsilon_0^C \cdot), \phi)_\nabla \to 0$ as $C \to \infty$.  This implies that the law of the projection of $\wh{h}^0(\epsilon_0^C \cdot)$ onto $\CH_2(\C)$ converges weakly in the space of distributions to that of the projection of $\wh{h}$ onto $\CH_2(\C)$.  Therefore the same is also true for $\wt{h}^0(\epsilon_0^C \cdot)$.  The same argument as above implies that $(\wt{\Fh}(\epsilon_0^C\cdot),\phi)_\nabla \to 0$ as $C \to \infty$.  Altogether, this implies that  the projection of $h^{\epsilon_0^C,C}$ onto $\CH_2(\C)$ converges weakly as $C \to \infty$ in the space of distributions to the corresponding projection of an $\alpha$-quantum cone.

Having proved part~\eqref{it::quantum_cone_limit}, part~\eqref{it::quantum_cone_invariance} follows as explained in the proof of Proposition~\ref{prop::quantum_wedge_properties}.
\end{proof}

We also have the following analog of Proposition~\ref{prop::quantum_wedge_characterization}.

\begin{proposition}
\label{prop::quantum_cone_characterization}
Fix $\alpha < Q$ and suppose that $(\C,h,0,\infty)$ is the circle average embedding of a quantum surface parameterized by $\C$ such that the following hold:
\begin{enumerate}[(i)]
\item The law of $(\C,h,0,\infty)$ (as a quantum surface) is invariant under the operation of multiplying its area by a constant.  That is, if we fix $C \in \R$, and then let $(\C,\wt{h},0,\infty)$ be the circle average embedding of $(\C,h+C/\gamma,0,\infty)$, then $(\C,h,0,\infty) \stackrel{d}{=} (\C,\wt{h},0,\infty)$.
\item The total variation distance between the law of $(\C,h,0,\infty)$ restricted to $B(0,r)$ and the law of an $\alpha$-quantum cone $(\C,h_\alpha,0,\infty)$ restricted to $B(0,r)$ tends to $0$ as $r \to 0$.
\end{enumerate}
Then $(\C,h,0,\infty)$ has the law of an $\alpha$-quantum cone.
\end{proposition}
\begin{proof}
This follows from the same argument given in the proof of Proposition~\ref{prop::quantum_wedge_characterization}.
\end{proof}

\subsection{Encoding surfaces using Bessel processes}
\label{subsec::surfaces_strips_cylinders}

We are now going to explain how to encode a quantum wedge using a Bessel process (as will be clear from the following discussion, one can similarly encode a quantum cone using a Bessel process).  This perspective will allow us to give the definition of a thin quantum wedge (i.e., a wedge not homeomorphic to $\h$).  Recall that if we consider a thick quantum wedge parameterized by $\strip$ then the projection of the field onto $\CH_1(\strip)$ (with additive constant fixed so that the average on $[0, \pi i]$ vanishes) can be expressed as $X_t = B_{2t} + a t$ where $a \in \R$ and $B$ is a standard Brownian motion.  Write $Z_t = \exp(\tfrac{\gamma}{2} X_t)$.  By It\^o's formula,
\[ d \langle Z \rangle_t = \frac{\gamma^2}{2} Z_t^2 dt.\]
By Proposition~\ref{prop::bessel_exponential_bm}, we see that if we reparameterize $Z$ by its quadratic variation then it evolves as a $\bes^\delta$ with
\begin{equation}
\label{eqn::bessel_a_gamma}
   \delta = 2+\frac{2a}{\gamma}.
\end{equation}
Equivalently,
\begin{equation}
\label{eqn::a_gamma_delta}
a = \frac{\gamma}{2}(\delta-2).
\end{equation}
Recalling the phases described in Section~\ref{subsec::bessel_processes}, we note that:
\begin{enumerate}
\item If $a \geq 0$, then $\delta \geq 2$. This implies that $Z$ --- if started at zero --- a.s.\ does not hit zero in the future.
\item If $a \in (-\gamma,0)$, then $\delta \in (0,2)$.  This implies that we can view $Z$ as a process which is defined for all times and is instantaneously reflecting at $0$.
\end{enumerate}
In the former case, we can associate with $Z$ a single quantum wedge where the projection of the field onto $\CH_1(\strip)$ is determined by $Z$ by reparameterizing $2\gamma^{-1} \log Z$ to have quadratic variation given by $2dt$ (with additive constant fixed to agree exactly with this process).  This is how we define the $\alpha=Q$, $W = \tfrac{\gamma^2}{2}$ quantum wedge.  In the latter case, each excursion of $Z$ from $0$ (recall the form of the excursion measure from Remark~\ref{rem::bessel_ito_excursion}) can be associated with a (finite area) quantum surface, which we can view as being parameterized by $\strip$.  Using the change of coordinates formula (with the map $\strip \to \h$ given by $z \mapsto e^z$ 
this corresponds to a quantum wedge with $\alpha = Q-a$).  We will now use this construction to give a definition of a quantum wedge with $\alpha \in (Q,Q+\tfrac{\gamma}{2})$.

\begin{definition}
\label{def::skinny_wedge_bessel}
Fix $\alpha \in (Q,Q + \tfrac{\gamma}{2})$.  An {\bf $\alpha$-quantum wedge} is defined as follows.  Let
\begin{equation}
\label{eqn::wedge_bessel_dimension}
\delta = 2 + \frac{2(Q-\alpha)}{\gamma} = 3 + \frac{4 - 2 \alpha \gamma}{\gamma^2}
\end{equation}
(note that $\delta \in (1,2)$) and let $Y$ be a $\bes^\delta$.  Let $\CH_1(\strip)$, $\CH_2(\strip)$ be as in the statement of Lemma~\ref{lem::strip_spaces_orthogonal}.  Sample a countable collection of distribution valued random variables $h_e$ on $\strip$ indexed by the excursions $e$ of $Y$ from $0$ where the projection of $h_e$ onto $\CH_1(\strip)$ is given by reparameterizing $2\gamma^{-1} \log(e)$ to have quadratic variation $2dt$ (with additive constant fixed to agree exactly with this process) and the projection of $h_e$ onto $\CH_2(\strip)$ is sampled independently according to the law of the corresponding projection of a free boundary GFF on $\strip$ (with additive constant fixed so that its average on $[0, \pi i]$ vanishes).  As explained in Section~\ref{subsubsec::intro_qwedges}, a thin quantum wedge can be viewed as a \ppp and also a.s.\ can be given a topological structure in which the left and right boundaries are each homeomorphic to the positive real axis.
\end{definition}

\begin{remark}
\label{rem::excursion_lengths}
Suppose that $2/\gamma^2 > 1$ (so that Lemma~\ref{lem::moment_bound} below applies; see also Lemma~\ref{lem::negative_drift_first_moment}).  By \cite[Proposition~1.2]{ds2011kpz}, we have that the conditional expectation of the quantum mass of the surface given $X$ (i.e., the projection of the field onto the subspace $\CH_1(\strip)$ of $\CH(\strip)$ with fixed additive constant) is (up to a multiplicative constant) given by
\begin{equation}
\label{eqn::mass_qv}
\int_{-\infty}^\infty \exp\left(\gamma X_t \right) dt = \frac{2}{\gamma^2} \int_{-\infty}^\infty d\langle Z \rangle_t = \frac{2}{\gamma^2}\langle Z \rangle_\infty.
\end{equation}
Therefore the length of a given excursion $e$ of $Y$ from $0$ as in Definition~\ref{def::skinny_wedge_bessel} is equal to a non-random constant times the conditional expectation of the quantum mass of the quantum surface described by $h_e$ given $e$ (recall~\eqref{eqn::mass_qv}).
\end{remark}

We will see in Section~\ref{sec::structure_theorems} and Section~\ref{sec::quantumwedge} that the manner in which we string together a sequence of quantum disks in Definition~\ref{def::skinny_wedge_bessel} using a Bessel process is natural in the context of quantum wedges.  In particular, we will show that the beaded surface which arises by slicing a thick wedge with an independent boundary intersecting $\SLE_\kappa(\rho_1;\rho_2)$ process yields a thin wedge as defined in Definition~\ref{def::skinny_wedge_bessel}.

\begin{remark}
\label{rem::critical_bessel_values}
Some critical values for the Bessel process occur when $\delta \in \{0,1,2\}$.  As explained in Section~\ref{subsec::bessel_processes}, these respectively correspond to when~$0$ is an absorbing state, the critical value below which the process fails to be a semimartingale when it is interacting with~$0$, and the critical value at or above which the process never hits~$0$.  These respectively correspond to
\[ a \in \left\{-\gamma, -\frac{\gamma}{2}, 0 \right\},\]
which in turn correspond to
\[ \alpha \in \left\{\gamma+Q, \frac{\gamma}{2}+Q, Q \right\} = \left\{\frac{3\gamma}{2} + \frac{2}{\gamma}, \gamma + \frac{2}{\gamma}, \frac{\gamma}{2}+\frac{2}{\gamma} \right\},\]
and
\[ W \in \left\{-\frac{\gamma^2}{2}, 0,\frac{\gamma^2}{2} \right\}.\]
\end{remark}

It will also be natural to encode an $\alpha$-quantum wedge using a Bessel process $\wh{Y}$ so that the length of each excursion of $\wh{Y}$ from $0$ corresponds to a (non-random) constant times the $\gamma$-LQG mass of the corresponding bubble (rather than a conditional expectation).   We will now explain why this is possible and why the dimension of the Bessel process is the same as in~\eqref{eqn::wedge_bessel_dimension}.  The rest of this subsection may be skipped on a first reading.

\begin{proposition}
\label{prop::actual_quantum_area}
Fix $\alpha \in (Q,Q+\tfrac{\gamma}{2})$ and suppose that $\CW$ is an $\alpha$-quantum wedge.  The law of the $\gamma$-LQG areas of the bubbles of $\CW$ coincides with the law of the lengths of the excursions from $0$ of a Bessel process of dimension $\delta$ with $\delta$ as in~\eqref{eqn::wedge_bessel_dimension}. 
\end{proposition}

We will need two preparatory lemmas to prove Proposition~\ref{prop::actual_quantum_area}.  We begin with the following basic observation about \ppp's.

\begin{lemma}
\label{lem::poisson_reweight}
Suppose that $\Lambda$ is a \ppp\ on $\R_+ \times \R_+ \times \R_+$ chosen with intensity measure $du \otimes t^\alpha dt \otimes \mu$ where $\mu$ is a probability measure on $\R_+$ and both $du$ and $dt$ denote Lebesgue measure on $\R_+$.  Assume that $c = \E[U^{-\alpha-1}] < \infty$ for $U \sim \mu$.  Then $\{(u, U e) : (u,e,U) \in \Lambda\}$ is a \ppp\ on $\R_+ \times \R_+$ with intensity measure $c (du \otimes t^\alpha dt)$.
\end{lemma}
\begin{proof}
Recall that the image of a \ppp\ under a measurable map is a \ppp\ and that the intensity of the latter is given by the pushforward of the intensity of the former.  Consequently, $\{(u, U e) : (u,e,U) \in \Lambda\}$ is a \ppp\ on $\R_+ \times \R_+$.  Let $\nu$ be the law of $\log U$ for $U \sim \mu$.  Applying a change of variables using the $\log$ function, the density $t^\alpha dt$ becomes $e^{(\alpha+1)t}dt$.  The convolution of this density with $\nu$ is given by:
\[ \int_\R e^{(\alpha+1)(t-s)} d\nu(s) = e^{(\alpha+1) t} \E[ U^{-\alpha-1}],\]
which implies the lemma.
\end{proof}

\begin{lemma}
\label{lem::negative_drift_first_moment}
Fix $a < 0$ and suppose that $h$ is a random distribution on $\strip$ such that
\begin{enumerate}[(i)]
\item Its projection onto $\CH_1(\strip)$ is given by $X_u = B_{2u} + a u$ for $u \geq 0$ and $X_u = \wh{B}_{-2u} - a u$ for $u < 0$ where $B,\wh{B}$ are independent standard Brownian motions with $B_0 = \wh{B}_0=0$ conditioned so that $X_u \leq 0$ for all~$u \in \R$ (with additive constant fixed so that its average on $[0, \pi i]$ vanishes).
\item Its projection onto $\CH_2(\strip)$ is given independently by the corresponding projection of a free boundary GFF on $\strip$ (with additive constant fixed so that its average on $[0, \pi i]$ vanishes).
\end{enumerate}
Then we have for each $p \in (0,4/\gamma^2)$ that $\E[ (\nu_h(\partial \strip))^p] < \infty$ and for each $p \in (0,2/\gamma^2)$ that $\E[(\mu_h(\strip))^p] < \infty$.
\end{lemma}

The distribution $h$ in the statement of Lemma~\ref{lem::negative_drift_first_moment} describes a single bead of a quantum wedge parameterized by $\strip$ with the two special points located at $\pm \infty$ with the horizontal translation taken so that the supremum of the projection onto $\CH_1(\strip)$ (with the additive constant fixed so that its average on $[0, \pi i]$ is the average of $h$ on $[0, \pi i]$) occurs at $u=0$.

\begin{proof}[Proof of Lemma~\ref{lem::negative_drift_first_moment}]
We will give the proof that $\E[ (\mu_h(\strip))^p] < \infty$ for each $p \in (0,2/\gamma^2)$; that $\E[ (\nu_h(\partial \strip))^p] < \infty$ for each $p \in (0,4/\gamma^2)$ follows from the same argument.  Let $\wh{h}$ be the projection of $h$ onto $\CH_2(\strip)$.

Fix $p \in (0,2/\gamma^2)$ and assume that $p \geq 1$.  For each $k \in \Z$, let $X_k^* = \sup_{t \in [k,k+1]} X_t$.  We also let $A_k = [k,k+1) \times [0,\pi]$.  Then we have that
\begin{align*}
   \mu_h(\strip)
&= \sum_{k \in \Z} \mu_h(A_k)
 \leq \sum_{k \in \Z} e^{\gamma X_k^*} \mu_{\wh{h}}(A_k).
\end{align*}
By Jensen's inequality, we have that
\begin{align*}
( \mu_h(\strip))^p \leq \left(\sum_{k \in \Z} e^{\gamma X_k^*} \right)^{p} \sum_{k \in \Z} p_k (\mu_{\wh{h}}(A_k))^p \quad\text{where}\quad p_k = \frac{e^{\gamma X_k^*}}{\sum_{j \in \Z} e^{\gamma X_j^*}}.
\end{align*}
Let $c_0 = \E[ (\mu_{\wh{h}}(A_0)^p]$ and note that $c_0 < \infty$ by Lemma~\ref{lem::moment_bound}.  Taking expectations of both sides over the randomness over $\mu_{\wh{h}}$ and using that $X_k^*$ is independent of $\mu_{\wh{h}}$, we thus see that
\begin{align*}
	\E[ (\mu_h(\strip))^p]
&\leq c_0 \E\left[ \left( \sum_{k \in \Z} e^{\gamma X_k^*} \right)^{p} \right].
\end{align*}
Lemma~\ref{lem:bm_moment_bound} implies that the expectation is finite, which completes the proof in the case of $\mu_h$.

As mentioned earlier, the case of $\nu_h$ follows from the same argument except one uses \cite[Proposition~3.5]{rv2010revisited} in place of Lemma~\ref{lem::moment_bound} to bound the analog of $c_0$ in this case.
\end{proof}

\begin{proof}[Proof of Proposition~\ref{prop::actual_quantum_area}]
We can think of sampling the Bessel process $Y$ in Definition~\ref{def::skinny_wedge_bessel} by first sampling a \ppp\ $\Lambda^*$ according to $du \otimes \nu_\delta^*$ (where $du$ is Lebesgue measure on $\R_+$ and $\nu_\delta^*$ is in Remark~\ref{rem::bessel_ito_excursion}) and then sampling $Y$ by associating with each $(u,e^*) \in \Lambda^*$ a $\bes^\delta$ excursion from~$0$ to~$0$ whose maximum value is given by~$e^*$ (this law can be constructed using $\bes^{4-\delta}$ processes as described in Remark~\ref{rem::bessel_ito_excursion}).  Note that the total quantum mass of the surface associated with a given $(u,e^*)$ is given by~$(e^*)^2 U_e$ where the $U_e$ are i.i.d.\ random variables indexed by the excursions of $Y$ from $0$.  By the discussion just after~\eqref{eqn::bessel_ito_lifetime}, we note that $\{ (u,(e^*)^2) : (u,e^*) \in \Lambda^*\}$ is a \ppp\ on $\R_+ \times \R_+$ with intensity measure proportional to $du \otimes t^{\delta/2-2} dt$ where~$dt$ denotes Lebesgue measure on~$\R_+$.  Moreover, the law of $U_e$ is given by the law of the $\gamma$-LQG mass associated with the law on distributions on~$\strip$ described in the statement of Lemma~\ref{lem::negative_drift_first_moment} with $a = Q-\alpha$ (this follows from Proposition~\ref{prop::bessel_exponential_bm} and Proposition~\ref{prop::bessel_time_reversal}).  In particular, we have that $\E[ (\mu_h(\strip))^{1-\delta/2}] < \infty$ since $\delta \in (1,2)$ hence $1-\delta/2 < 1/2 < 2/\gamma^2$.  The result is thus a consequence of Lemma~\ref{lem::poisson_reweight}.
\end{proof}

\subsection{Quantum disks and spheres}
\label{subsec::disks_and_spheres}

As in Section~\ref{subsec::surfaces}, for each value of $W \geq \tfrac{\gamma^2}{2}$, we let $\diskmeasure_W$ be the infinite measure on doubly marked surfaces $(\strip,h,-\infty,+\infty)$ a ``sample'' from which can be generated using the following steps.
\begin{enumerate}
\item ``Sample'' a Bessel excursion $Z$ with $\delta =  3-\tfrac{2}{\gamma^2} W$.  Take the projection of $h$ onto $\CH_1(\strip)$ to be given by $\tfrac{2}{\gamma} \log Z$ reparameterized to have quadratic variation $2dt$ (with additive constant fixed so as to agree exactly with this process).
\item Take the projection of $h$ onto $\CH_2(\strip)$ to be given by the corresponding projection of a GFF on $\strip$ with free boundary conditions which is independent of $Z$ (with additive constant fixed so that its average on $[0, \pi i]$ vanishes).
\end{enumerate}

Recall from Table~\ref{tab::wedge_parameterization} that the dimension of the Bessel process which encodes a quantum wedge of weight $W$ is given by $1+\tfrac{2}{\gamma^2} W$.  Recall that a Bessel process of dimension $\delta < 2$ conditioned to be non-negative yields a Bessel process of dimension $4-\delta$.  Thus the dimension $3-\tfrac{2}{\gamma^2} W$ corresponds to that which encodes a thick quantum wedge which near its marked point around which bounded neighborhoods have finite mass has the same behavior as a sample from $\diskmeasure_W$ near one of its marked points.  When we do not specify a value of $W$, we mean $W = 2$ and refer to this as simply the quantum disk.  Note that this gives a Bessel process dimension of $3-\tfrac{4}{\gamma^2}$.  We write $\diskmeasure$ to indicate this case.  (See Section~\ref{subsubsec:int_disk_sphere} for some additional motivation for this definition.)

A quantum disk naturally has two marked points at $\pm \infty$.  It turns out that the conditional law of these two points given the quantum surface (i.e., modulo conformal transformation) is that of independent picks from the boundary length measure (see Appendix~\ref{app::disks_spheres}).  Equivalently, the law of $h$ described above (which is defined modulo a horizontal translation of $\strip$) is invariant under the operation of sampling $x,y$ independently from~$\nu_h$ and then replacing~$h$ with $h \circ \varphi + Q \log|\varphi'|$ where~$\varphi$ is a conformal map $\strip \to \strip$ which takes $\pm \infty$ to $x,y$.

For each value of $W \geq \tfrac{\gamma^2}{2}$, we let $\spheremeasure_W$ be the infinite measure on doubly marked surfaces which can be parameterized by $\cyl$ and sampled from as follows.
\begin{enumerate}
\item ``Sample'' a Bessel excursion $Z$ with $\delta = 2-\tfrac{2}{\gamma^2} W$.  Take the projection of $h$ onto $\CH_1(\cyl)$ to be given by $\tfrac{2}{\gamma} \log Z$ reparameterized to have quadratic variation $dt$ (with additive constant fixed so as to agree exactly with this process).
\item Take the projection of $h$ onto $\CH_2(\cyl)$ to be given by the corresponding projection of a GFF on $\cyl$ which is independent of $Z$ (with additive constant fixed so that its average on $[0, 2\pi i]$ vanishes).
\end{enumerate}

In analogy with the case of $\diskmeasure_W$, the local behavior of a sample from $\spheremeasure_W$ near its marked points is the same as that of a weight $W$ quantum cone near its marked point around which finite neighborhoods have finite mass.  We will refer to the case $W=4-\gamma^2$ as the quantum sphere and simply write $\spheremeasure$ to indicate this case.  Note that this gives a Bessel process dimension of $4-\tfrac{8}{\gamma^2}$.  As in the case of a quantum disk, the conditional law of the marked points (corresponding to $\pm \infty$ above) for a quantum sphere given the quantum surface structure is that of independent picks from $\mu_h$ (see Appendix~\ref{app::disks_spheres}).  Equivalently, the law of $h$ described above (which is defined modulo horizontal translation of $\cyl$) is invariant under the operation of sampling $x,y$ from $\mu_h$ and then replacing $h$ with $h \circ \varphi + Q \log |\varphi'|$ where $\varphi$ is a conformal map $\cyl \to \cyl$ which takes $\pm \infty$ to $x,y$.

\begin{definition}
\label{def::finite_volume_surfaces} $\ $
\begin{enumerate}[(i)]
\item The {\bf unit boundary length quantum disk} is the law on quantum surfaces given by $\diskmeasure$ conditional on $\nu_h(\partial \strip) = 1$.
\item The {\bf unit area quantum sphere} is the law on quantum surfaces given by $\spheremeasure$ conditional on $\mu_h(\cyl) = 1$.
\end{enumerate}
\end{definition}

A sample from $\diskmeasure$ naturally comes with two marked points which, when parameterized by $\strip$ as above, correspond to $\pm \infty$.

Let us now explain why the conditional laws in Definition~\ref{def::finite_volume_surfaces} exist.  We first consider the setting of the unit boundary length quantum disk.  Suppose that $0 < l_1 < l_2 < \infty$.  It follows from Lemma~\ref{lem::poisson_reweight} (with $a=4/\gamma^2$) and Lemma~\ref{lem::negative_drift_first_moment} (with $p=4/\gamma^2-1$) that the mass that $\diskmeasure$ assigns to the event $E_{l_1,l_2}$ that $\nu_h(\partial \strip) \in [l_1,l_2]$ is positive and finite.  We can thus make sense of $\diskmeasure$ conditioned on $E_{l_1,l_2}$ as a probability measure.  It therefore follows from the general theory of regular conditional probability distributions that the conditional law of $\diskmeasure$ given $\nu_h(\partial \strip) = l$ exists for a.e.\ $l > 0$.  Let $\diskmeasure( \cdot \giv \nu_h(\partial \strip) = l)$ denote this regular conditional probability distribution.

We will now explain why this conditional law exists for \emph{every} $l > 0$ and is in fact given by the limit of $\diskmeasure$ conditioned on $E_{l,l+\epsilon}$ as $\epsilon \to 0$.  Note that if $h$ has distribution $\diskmeasure$ and $C \in \R$ is a constant then replacing $h$ with $h+C$ has the effect of replacing the Bessel excursion $Z$ in its construction by $e^{\gamma C/2} Z$.  Recalling that $Z$ has dimension $\delta = 3-\tfrac{4}{\gamma^2}$ and the distribution of the maximum of $Z$ is given by $c_\delta t^{\delta-3} dt = c_\delta t^{-4/\gamma^2} dt$, $c_\delta > 0$ a constant, it follows that the distribution of the maximum of $e^{\gamma C/2} Z$ is given by $c_\delta e^{C(Q-\gamma)} t^{-4/\gamma^2} dt$.  Therefore the distribution of $h+C$ is given by $e^{C(Q-\gamma)} \diskmeasure$.  Moreover, replacing $h$ by $h+C$ has the effect of multiplying the boundary length measure by $e^{\gamma C/2}$.  It follows that the distribution of $\nu_h(\partial \strip)$ has density given by $c l^{-4/\gamma^2} dl$, $c > 0$ a constant, where $dl$ denotes Lebesgue measure on $\R_+$ (recall Lemma~\ref{lem::poisson_reweight}).  That is, we can write
\[ \diskmeasure(\cdot) = c \int_{\R_+} \diskmeasure(\cdot \giv \nu_h(\partial \strip) = l) l^{-4/\gamma^2} dl.\]
This representation combined with the Lebesgue differentiation theorem implies that for a.e.\ $l > 0$, $\diskmeasure(\cdot \giv \nu_h(\partial \strip) = l)$ exists and is given by the weak limit as $\epsilon \to 0$ of $\diskmeasure( \cdot \giv E_{l,l+\epsilon})$.  Fix $l_0 > 0$ so that this holds.  If we start with $h$ sampled from $\diskmeasure(\cdot  \giv E_{l_0,l_0+\epsilon})$ and then replace $h$ with $h+C$ we obtain a sample from $\diskmeasure(\cdot \giv E_{e^{\gamma C/2} l_0, e^{\gamma C/2} (l_0+\epsilon)})$.  Since the former converges as $\epsilon \to 0$ to $\diskmeasure(\cdot \giv \nu_h(\partial \strip) = l_0)$, it follows that the latter also converges as $\epsilon \to 0$ and the limit is $\diskmeasure( \cdot \giv \nu_h(\partial \strip) = l_0 e^{\gamma C/2})$.  Since $C \in \R$ was arbitrary, we therefore see that the conditional law $\diskmeasure(\cdot \giv \nu_h(\partial \strip) = l)$ exists for every $l > 0$.  Moreover, if $h$ has law $\diskmeasure(\cdot \giv \nu_h(\partial \strip) = l)$ then $h+C$ has law $\diskmeasure(\cdot \giv \nu_h(\partial \strip) = l e^{\gamma C/2})$.  In particular, the unit boundary length quantum disk is in fact defined.  A similar discussion applies to show that the unit area quantum sphere is also defined.

\begin{remark}
\label{rem::finite_volume_constructions}
In Appendix~\ref{app::disks_spheres}, we will show that each of the laws on finite volume surfaces from Definition~\ref{def::finite_volume_surfaces} can be constructed using a limiting procedure.  We will explain in Section~\ref{subsec::sle_kappa_prime_bubbles} that the law of the unit boundary length quantum disk can be used to describe the structure of the bubbles which are cut out from a weight $\tfrac{3\gamma^2}{2}-2$ quantum wedge by an independent $\SLE_{\kappa'}$ process, $\kappa' \in (4,8)$.  (As we will see later, this is the type of wedge which describes the local behavior of the surface parameterized by the unbounded complementary component of an $\SLE_{\kappa'}$ process drawn up to a quantum typical time.) We will also explain in Remark~\ref{rem::equivalence_of_disk_measures} that the unit boundary length quantum disk can be constructed from a certain excursion measure which is naturally associated with the cone times of a $2$-dimensional Brownian motion.  (Altogether, there will be four different constructions of this object in this article.)  Finally, we show in Section~\ref{sec::duality} that the law of the unit area quantum sphere can be used in conjunction with L\'evy trees to construct a type of $\gamma$-LQG surface with $\gamma^2 \in (4,8)$.
\end{remark}

\section{SLE/GFF couplings}
\label{sec::couplings}

In this section, we will describe reverse (Section~\ref{subsec::reverse_sle_coupling}) and forward (Section~\ref{subsec::forward_sle_coupling}) couplings of $\SLE$ with the GFF.  

For the reverse case, the coupling statement we state and prove here is slightly different from the statement given in \cite{she2010zipper}, mainly because \cite{she2010zipper} deals with weldings of two thick wedges, and in this paper we will also treat weldings involving thin wedges; in particular, we will be interested in cutting a thick quantum wedge (parameterized by the upper half-plane) along a {\em boundary-intersecting} $\SLE_\kappa(\rho)$ curve, instead of along a boundary-avoiding curve. Conceptually, if one is ``cutting'' a thick wedge along a boundary-intersecting $\SLE_\kappa(\rho)$ curve, one ``cuts off pieces'' of the quantum wedge at times when the curve collides with the boundary; analogously, when one is running the appropriate reverse $\SLE_\kappa(\wt{\rho})$, new pieces are getting ``glued in'' in a manner that will be discussed more in the next section.  To handle this, we will require a reverse SLE coupling that applies to reverse SLE involving a force point that bounces off the boundary (rather than getting absorbed into the interior).  Extending the argument of  \cite{she2010zipper} to accommodate bouncing force points of this kind requires a short stochastic calculus exercise, which we will present here. On the other hand, the forward coupling statement we require will be recalled here and stated in its original formulation without proof.

As an application of the forward coupling, in Section~\ref{subsec::sle_kappa_rho_excursion} we will give the conditional law of the excursion of a boundary intersecting $\SLE_\kappa(\rho)$ process in $\h$ with $\kappa \in (0,4)$ which separates a given boundary point $u$ from $\infty$ given the realization of the path before (resp.\ after) the start (resp.\ end) of the excursion.

\subsection{Reverse coupling}
\label{subsec::reverse_sle_coupling}

Throughout, we fix $\kappa > 0$, let $\gamma = \min(\sqrt{\kappa},4/\sqrt{\kappa})$, and $Q = \tfrac{2}{\gamma} + \tfrac{\gamma}{2}$.  Let
\begin{equation}
\label{eqn::neumann_green_half_plane}
G(y,z) = -\log|y-z| - \log|y-\ol{z}|
\end{equation}
denote the Neumann Green's function on $\h$.  For a given centered reverse Loewner flow $(\wt{f}_t)$ (which will be clear from the context), we let $\wt{G}_t(y,z) = G(\wt{f}_t(y),\wt{f}_t(z))$,
\[ \wt{\Fh}_0(z) = \frac{2}{\sqrt{\kappa}} \log|z|, \quad\text{and}\quad \wt{\Fh}_t(z) = \wt{\Fh}_0(\wt{f}_t(z)) + Q\log|\wt{f}_t'(z)|.\]
The following is a version of \cite[Theorem~4.5]{she2010zipper}.

\begin{theorem}
\label{thm::reverse_coupling}
Fix $\kappa > 0$, $\wt{\rho}_1,\ldots,\wt{\rho}_k \in \R$, and $\wt{x}_1,\ldots,\wt{x}_k \in \ol{\h}$.  Suppose that $\wt{W}$ solves the reverse $\SLE_\kappa(\ul{\wt{\rho}})$ SDE~\eqref{eqn::reverse_sle_kappa_rho} with force points located at $\wt{x}_1,\ldots,\wt{x}_k$ of weight $\wt{\rho}_1,\ldots,\wt{\rho}_k$, respectively, and let $(\wt{f}_t)$ be the centered reverse Loewner evolution driven by $\wt{W}$.  For each $t \geq 0$, let
\[ \wh{\Fh}_t(z) = \wt{\Fh}_t(z) + \frac{1}{2\sqrt{\kappa}} \sum_{j=1}^k \wt{\rho}_j \wt{G}_t(\wt{x}_j,z)\]
and let $h$ be an instance of the free boundary GFF on $\h$ independent of $\wt{W}$.  Suppose that $\wt{\tau}$ is an a.s.\ finite stopping time for the filtration generated by the Brownian motion which drives $\wt{W}$ that occurs at or before the continuation threshold for $\wt{W}$ is reached.  Then
\[ \wh{\Fh}_0 + h \stackrel{d}{=} \wh{\Fh}_{\wt{\tau}} + h \circ \wt{f}_{\wt{\tau}}\]
where the left and right sides are viewed as modulo additive constant distributions.
\end{theorem}
The proof of Theorem~\ref{thm::reverse_coupling} is similar to that of \cite[Theorem~4.5]{she2010zipper}, but we will provide it for completeness.  The proof below closely follows the presentation of the radial version of this theorem given in \cite[Section~5]{ms2013qle}.

\begin{lemma}
\label{lem::dg_df}
For each $y,z \in \h$ distinct, we have that
\begin{align}
d\wt{G}_t(y,z) &= -\re \frac{2}{\wt{f}_t(y)} \re \frac{2}{\wt{f}_t(z)} dt.\label{eqn::dg}\\
\intertext{For each $z \in \h$, we have that}
d \wh{\Fh}_t(z) &= -\re \frac{2}{\wt{f}_t(z)} dB_t. \label{eqn::df}
\end{align}
\end{lemma}
\begin{proof}
Applications of It\^o's formula imply that
\begin{align*}
 d \log(\wt{f}_t(y) - \wt{f}_t(z)) =  \frac{2dt}{\wt{f}_t(y) \wt{f}_t(z)} \quad\text{and}\quad
 d \log(\wt{f}_t(y) - \ol{\wt{f}_t(z)}) = \frac{2dt}{\wt{f}_t(y) \ol{ \wt{f}_t(z)}}.
\end{align*}
Adding these up, multiplying by $-1$, and then taking real parts yields~\eqref{eqn::dg}.  See the proof of \cite[Theorem~1.2]{she2010zipper} for a similar calculation (though for the forward rather than reverse coupling).

For~\eqref{eqn::df}, we first note by several applications of It\^o's formula that:
\begin{align*}
  d \wt{f}_t(z) &= \left(-\frac{2}{\wt{f}_t(z)} + \sum_{j=1}^k \re \frac{\wt{\rho}_j}{\wt{f}_t(\wt{x}_j)}\right) dt - \sqrt{\kappa} dB_t,\\
  d\log \wt{f}_t(z) &= \left(-\frac{4+\kappa}{2\wt{f}_t^2(z)} + \frac{1}{\wt{f}_t(z)}\sum_{j=1}^k \re \frac{\wt{\rho}_j}{\wt{f}_t(\wt{x}_j)}\right)dt - \frac{\sqrt{\kappa}}{\wt{f}_t(z)} dB_t,\\
  d\wt{f}_t'(z) &= \frac{2\wt{f}_t'(z)}{\wt{f}_t^2(z)} dt, \quad\text{and}\quad d \log \wt{f}_t'(z) = \frac{2}{\wt{f}_t^2(z)} dt.
\end{align*}
Taking real parts and adding up (and using~\eqref{eqn::dg}) yields~\eqref{eqn::df}.
\end{proof}

\begin{lemma}
\label{lem::ht_dist_mg}
For each $t \geq 0$, $\wh{\Fh}_t + h \circ \wt{f}_t$ a.s.\ takes values in the space of modulo additive constant distributions.  Suppose that $\phi \in C_0^\infty(\h)$ with $\int_\h \phi(z) dz = 0$.  Then both $(\wh{\Fh}_t + h \circ \wt{f}_t,\phi)$ and $(\wh{\Fh}_t,\phi)$ have a.s.\ continuous modifications in $t$ and the latter is a square-integrable martingale. 
\end{lemma}
\begin{proof}
We first note that it is clear that $h \circ \wt{f}_t$ takes values in the space of distributions modulo additive constant and that $(h \circ \wt{f}_t,\phi)$ is a.s.\ continuous in $t$ from how it is defined.  This leaves us to deal with $\wh{\Fh}_t$.  It follows from~\eqref{eqn::df} of Lemma~\ref{lem::dg_df} that for each fixed $z \in \h$ we have
\[ d \langle \wh{\Fh}_t(z) \rangle = \left( \re \frac{2}{\wt{f}_t(z)} \right)^2 dt.\]
Note that it follows from the form of the SDE for the centered reverse Loewner flow that $\im \wt{f}_t(z)$ is non-decreasing in $t$ for any fixed $z \in \h$.  Fix $K \subseteq \h$ compact and $T < \infty$.  It consequently follows that there exists a constant $C > 0$ depending only on $K$ and $T$ such that
\[ \sup_{z \in K} \langle \wh{\Fh}_t(z) - \wh{\Fh}_u(z) \rangle \leq C(t-u) \quad\text{for all}\quad 0 \leq u \leq t \leq T.\]
It therefore follows from the Burkholder-Davis-Gundy inequality that for each $p \geq 1$ there exists $C_p,C_p' > 0$ depending only on $p$, $K$, and $T$ such that
\begin{align}
      \sup_{z \in K} \E\left[ \sup_{u \leq s \leq t} |\wh{\Fh}_s(z) - \wh{\Fh}_u(z)|^p \right]
\leq&C_p \sup_{z \in K} \E\left[ \langle \wh{\Fh}_t(z) - \wh{\Fh}_u(z) \rangle^{p/2} \right]\notag\\
\leq&C_p' (t-u)^{p/2} \quad\text{for all}\quad 0 \leq u \leq t \leq T. \label{eqn::bdg}
\end{align}
It is easy to see from~\eqref{eqn::bdg} with $u =0$ and Fubini's theorem that for each $K \subseteq \h$ compact we have that $\wh{\Fh}_t|_K$ is a.s.\ in $L^p(K)$.  By combining~\eqref{eqn::bdg} and the Kolmogorov-Centsov theorem, it is also easy to see that $t \mapsto (\wh{\Fh}_t,\phi)$ has an a.s.\ continuous modification for any fixed $\phi \in C_0^\infty(\h)$ with $\int_\h \phi(z) dz = 0$.  Lastly, it follows from~\eqref{eqn::df} and~\eqref{eqn::bdg} of Lemma~\ref{lem::dg_df} that $(\wh{\Fh}_t,\phi)$ is a square-integrable martingale.  This completes the proof of both assertions of the lemma.
\end{proof}

For each $\phi \in C_0^\infty(\h)$ with $\int_\h \phi(z) dz = 0$ and $t \geq 0$ we let
\[ \wt{E}_t(\phi) = \iint_{\h^2} \phi(y) \wt{G}_t(y,z) \phi(z) dy dz\]
be the conditional variance of $(h \circ \wt{f}_t,\phi)$ given $\wt{f}_t$.

\begin{lemma}
\label{lem::h_phi_qv}
For each $\phi \in C_0^\infty(\h)$ with $\int_\h \phi(z) dz = 0$ we have that
\[ d \langle (\wh{\Fh}_t,\phi) \rangle = -d\wt{E}_t(\phi).\]
\end{lemma}
\begin{proof}
Since $(\wh{\Fh}_t,\phi)$ is a continuous $L^2$ martingale, the process $\langle (\wh{\Fh}_t,\phi) \rangle$ is characterized by the property that
\[ (\wh{\Fh}_t,\phi)^2 - \langle (\wh{\Fh}_t,\phi) \rangle\]
is a continuous local martingale in $t \geq 0$.  Thus to complete the proof of the lemma, it suffices to show that
\[ (\wh{\Fh}_t,\phi)^2 + \wt{E}_t(\phi)\]
is a continuous local martingale.  It follows from~\eqref{eqn::dg} and~\eqref{eqn::df} of Lemma~\ref{lem::dg_df} that
\[ \wh{\Fh}_t(y) \wh{\Fh}_t(z) + \wt{G}_t(y,z)\]
evolves as the sum of a continuous martingale in $t \geq 0$ plus a drift term which can be expressed as a sum of terms, one of which depends only on $y$ and the other only on $z$.  These drift terms cancel upon integrating against $\phi(y)\phi(z) dy dz$ which in turn implies the desired result.
\end{proof}

\begin{proof}[Proof of Theorem~\ref{thm::reverse_coupling}]
Fix $\phi \in C_0^\infty(\h)$ with $\int_\h \phi(z) dz = 0$.  Let~$\wt{\CF}_t$ be the filtration generated by~$\wt{f}_t$.  Note that~$\wh{\Fh}_t$ is $\wt{\CF}_t$-measurable and that, given $\wt{\CF}_t$, $(h \circ \wt{f}_t,\phi)$ is a Gaussian random variable with mean zero and variance $\wt{E}_t(\phi)$.  Let $\wt{I}_t(\phi) = (h \circ \wt{f}_t + \wh{\Fh}_t,\phi)$.  For $\theta \in \R$ we have that
\begin{align*}
   &\E[ \exp(i \theta \wt{I}_t(\phi))]
= \E[ \E[ \exp(i \theta \wt{I}_t(\phi)) \giv \wt{\CF}_t]]\\
=& \E[ \E[ \exp(i \theta (h \circ \wt{f}_t,\phi)) \giv \wt{\CF}_t] \exp(i \theta (\wh{\Fh}_t,\phi))]\\
=& \E[ \exp(i \theta (\wh{\Fh}_t,\phi) - \tfrac{\theta^2}{2} \wt{E}_t(\phi))]\\
=& \exp(i \theta (\wh{\Fh}_0,\phi) - \tfrac{\theta^2}{2} \wt{E}_0(\phi)).
\end{align*}
Therefore $\wt{I}_t(\phi) \stackrel{d}{=} \wt{I}_0(\phi)$ for each $\phi \in C_0^\infty(\h)$ with $\int_\h \phi(z) dz = 0$.  The result in the case of deterministic times $t$ follows since this holds for all such test functions $\phi$ and $\phi \mapsto \wt{I}_0(\phi)$ has a Gaussian distribution.  The same argument applies in the case of a.s.\ finite stopping times; note that the the optional stopping theorem applies as we are using it in the setting of the bounded martingale $t \mapsto \exp(i \theta (\wh{\Fh}_t,\phi) - \tfrac{\theta^2}{2} \wt{E}_t(\phi))$.
\end{proof}

In analogy with \cite[Theorem~5.1]{ms2013qle} and Theorem~\ref{thm::reverse_coupling}, it is also possible to couple reverse radial $\SLE_\kappa(\wt{\rho})$ with the GFF.  This is proved in \cite[Theorem~5.5]{ms2013qle}; we restate the result here for the convenience of the reader since we will make use of it later in the article.

\begin{theorem}
\label{thm::radial_rho_coupling_existence}
Fix $\kappa > 0$.  Suppose that $h$ is a free boundary GFF on $\D$ and let $(\wt{f}_t)$ be the centered reverse radial $\SLE_\kappa(\wt{\rho})$ Loewner flow which is driven by the solution $\wt{U}$ as in~\eqref{eqn::reverse_radial_sle_kappa_rho} with $\wt{V}_0 = \wt{v}_0 \in \partial \D$ taken to be independent of $h$.  For each $t \geq 0$ and $z \in \D$ we let
\begin{equation}
\label{eqn::fh_definition}
\begin{split}
\wt{\Fh}_t(z) =& \frac{2}{\sqrt{\kappa}} \log|\wt{f}_t(z)-1| - \frac{\kappa+6 - \wt{\rho}}{2 \sqrt{\kappa}} \log|\wt{f}_t(z)| -\\
 &\frac{\wt{\rho}}{\sqrt{\kappa}} \log|\wt{f}_t(z) - \wt{V}_t| + Q \log|\wt{f}_t'(z)|.
 \end{split}
\end{equation}
Let $\wt{\tau}$ be an a.s.\ finite stopping time for the filtration generated by the Brownian motion $B$ which drives $\wt{U}$ which occurs at or before the continuation threshold is reached.  Then
\begin{equation}
\label{eqn::coupling_rho}
\wt{\Fh}_0 + h  \stackrel{d}{=} \wt{\Fh}_{\wt{\tau}} + h \circ \wt{f}_{\wt{\tau}}
\end{equation}
where we view the left and right sides as distributions defined modulo additive constant.
\end{theorem}

\subsection{Forward coupling}
\label{subsec::forward_sle_coupling}

\begin{figure}[ht!]
\begin{center}
\includegraphics[scale=0.85,page=2]{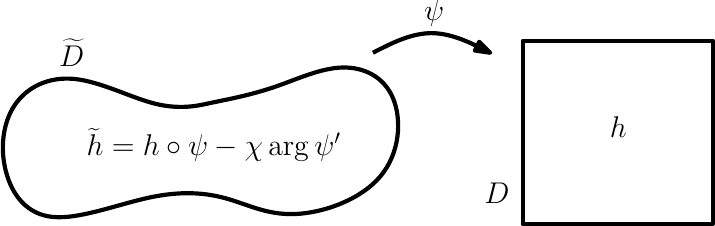}
\end{center}
\caption{\label{fig::coordinatechange} An imaginary surface coordinate change.}
\end{figure}

\begin{figure}[ht!]
\begin{center}
\subfloat{\includegraphics[scale=0.85]{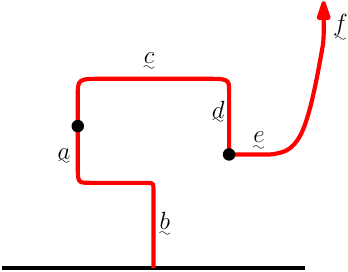}}
\hspace{0.05\textwidth}
\subfloat{\includegraphics[scale=0.85]{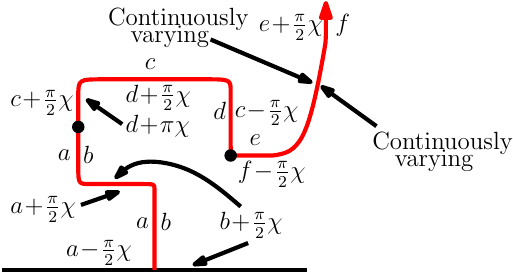}}
\caption{\label{fig::winding}  The notation on the left is a shorthand for the boundary data indicated on the right.  We often use this shorthand to indicate GFF boundary data.  In the figure, we have placed some black dots on the boundary $\partial D$ of a domain $D$.  On each arc $L$ of $\partial D$ that lies between a pair of black dots, we will draw either a horizontal or vertical segment $L_0$ and label it with $\uwave{x}$.  This means that the boundary data on $L_0$ is given by $x$, and that whenever $L$ makes a quarter turn to the right, the height goes down by $\tfrac{\pi}{2} \chi$ and whenever $L$ makes a quarter turn to the left, the height goes up by $\tfrac{\pi}{2} \chi$.  More generally, if $L$ makes a turn which is not necessarily at a right angle, the boundary data is given by $\chi$ times the winding of $L$ relative to $L_0$.  If we just write $x$ next to a horizontal or vertical segment, we mean to indicate the boundary data at that segment and nowhere else.  The right side above has exactly the same meaning as the left side, but the boundary data is spelled out explicitly everywhere.  Even when the curve has a fractal, non-smooth structure, the {\em harmonic extension} of the boundary values still makes sense, since one can transform the figure via the rule in Figure~\ref{fig::coordinatechange} to a half plane with piecewise constant boundary conditions.
}
\end{center}
\end{figure}

\begin{figure}[ht!]
\begin{center}
\includegraphics[scale=0.85]{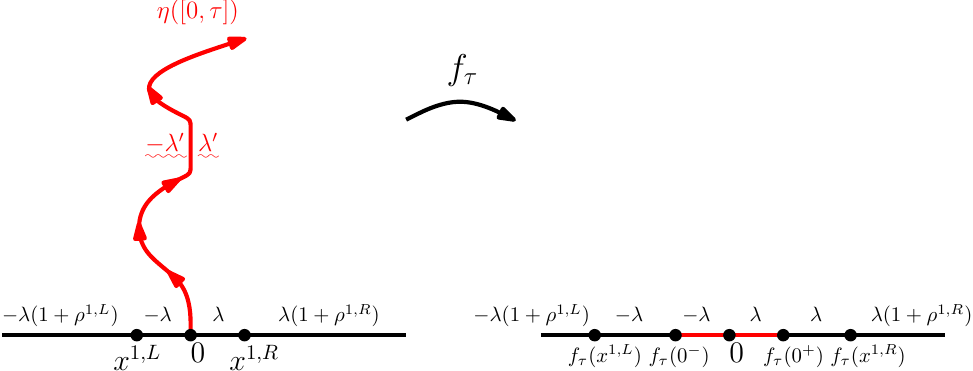}
\caption{\label{fig::conditional_boundary_data}  Suppose that $h$ is a GFF on $\h$ with the boundary data depicted above.  Then the flow line $\eta$ of $h$ starting from $0$ is an $\SLE_\kappa(\ul{\rho}^L;\ul{\rho}^R)$ curve in $\h$ where $|\ul{\rho}^L| = |\ul{\rho}^R| = 1$.  For any $\eta$ stopping time $\tau$, the conditional law of $h$ given $\eta|_{[0,\tau]}$ is equal to that of a GFF on $\h \setminus \eta([0,\tau])$ with the boundary data depicted above (the notation $\uwave{a}$ is explained in Figure~\ref{fig::winding}).  It is also possible to couple $\eta' \sim\SLE_{\kappa'}(\ul{\rho}^L;\ul{\rho}^R)$ for $\kappa' > 4$ with $h$ and the boundary data takes on the same form.  The difference is in the interpretation.  The (a.s.\ self-intersecting) path $\eta'$ is not a flow line of $h$, but for each $\eta'$ stopping time $\tau'$ the left and right {\em boundaries} of $\eta'([0,\tau'])$ are $\SLE_{\kappa}$ flow lines, where $\kappa=16/\kappa'$, angled in opposite directions.  The union of the left boundaries --- over a collection of $\tau'$ values --- is a tree of merging flow lines, while the union of the right boundaries is a corresponding dual tree whose branches do not cross those of the tree.}
\end{center}
\end{figure}

We will now give a brief overview of the theory of GFF flow lines developed in \cite{ms2012imag1,ms2012imag2,ms2012imag3,ms2013imag4}.  We assume throughout that $\kappa \in (0,4)$ so that $\kappa' := 16/\kappa \in (4,\infty)$.  We will often make use of the following definitions and identities:
\begin{align}
 \label{eqn::deflist} &\lambda := \frac{\pi}{\sqrt \kappa}, \,\,\,\,\,\,\,\,\lambda' := \frac{\pi}{\sqrt{16/\kappa}} = \frac{\pi \sqrt{\kappa}}{4} = \frac{\kappa}{4} \lambda < \lambda, \,\,\,\,\,\,\,\, \chi := \frac{2}{\sqrt \kappa} - \frac{\sqrt \kappa}{2}\\
 \label{eqn::fullrevolution} &\quad\quad\quad\quad\quad\quad\quad\quad 2 \pi \chi = 4(\lambda-\lambda'), \,\,\,\,\,\,\,\,\,\,\,\lambda' = \lambda - \frac{\pi}{2} \chi\\
\label{eqn::fullrevolutionrho} &\quad\quad\quad\quad\quad\quad\quad\quad\quad 2 \pi \chi = (4-\kappa)\lambda = (\kappa'-4)\lambda'.
\end{align}

See \cite[Theorem~1.1]{ms2012imag1} for a formal statement of the forward coupling of $\SLE$ with the GFF.  The boundary data one associates with the GFF on $\h$ so that its flow line $\eta$ from $0$ to $\infty$ is an $\SLE_\kappa(\ul{\rho}^L;\ul{\rho}^R)$ process with force points located at $\ul{x} = (\ul{x}^L;\ul{x}^R)$ for $\ul{x}^L = (x^{k,L} < \cdots < x^{1,L} \leq 0 ^-)$ and $\ul{x}^R = (0^+ \leq x^{1,R} < \cdots < x^{\ell,R})$ and with weights $(\ul{\rho}^L;\ul{\rho}^R)$ for $\ul{\rho}^L = (\rho^{1,L},\ldots,\rho^{k,L})$ and $\ul{\rho}^R = (\rho^{1,R},\ldots,\rho^{\ell,R})$ is
\begin{align}
 -&\lambda\left( 1 + \sum_{i=1}^j \rho^{i,L}\right) \quad\text{for}\quad x \in [x^{j+1,L},x^{j,L}) \quad\text{and}\\
 &\lambda\left( 1 + \sum_{i=1}^j \rho^{i,R}\right) \quad\text{for}\quad x \in [x^{j,R},x^{j+1,R}).
\end{align}
This is depicted in Figure~\ref{fig::conditional_boundary_data} in the special case that $|\ul{\rho}^L| = |\ul{\rho}^R| = 1$.  For any $\eta$ stopping time $\tau$, the conditional law of $h$ given $\eta|_{[0,\tau]}$ is that of a GFF on $\h \setminus \eta([0,\tau])$.  The boundary data of the conditional field agrees with that of $h$ on $\partial \h$.  On the right side of $\eta([0,\tau])$, it is $\lambda' + \chi \cdot {\rm winding}$, where the terminology ``winding'' is explained in Figure~\ref{fig::winding}, and to the left it is $-\lambda' + \chi \cdot {\rm winding}$.  This is also depicted in Figure~\ref{fig::conditional_boundary_data}.

A complete description of the manner in which GFF flow lines interact with each other is given in \cite[Theorem 1.5]{ms2012imag1} and \cite[Theorem~1.7]{ms2013imag4}.  See, in particular, \cite[Figure~1.13]{ms2013imag4}.  For more on flow lines started from boundary points, see \cite{ms2012imag1} as well as \cite[Section~2.3]{ms2013imag4}.  The introduction to \cite{ms2013imag4} contains a detailed overview of the behavior of flow lines started from interior points.  See also \cite[Section~2]{mw2013intersections} for an overview of the $\SLE$/GFF coupling.

\subsection{Law of an $\SLE_\kappa(\rho)$ excursion}
\label{subsec::sle_kappa_rho_excursion}

\label{sec::sle_kappa_rho_excursion}

\begin{figure}[ht!]
\begin{center}
\includegraphics[scale=0.85, page=1]{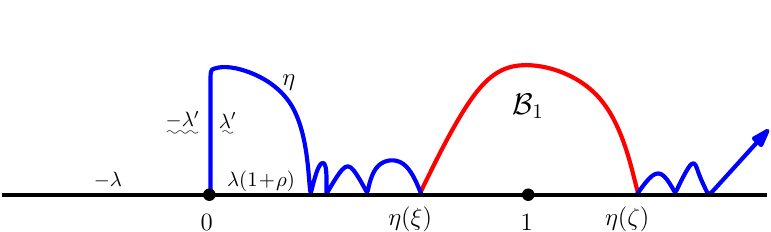}
\end{center}
\caption{\label{fig::excursion_over_1}
Shown is an $\SLE_\kappa(\rho)$ process $\eta$ in $\h$ from $0$ to $\infty$ with $\kappa \in (0,4)$ and $\rho \in (-2,\tfrac{\kappa}{2}-2)$ where the force point is located at~$0^+$ viewed as the flow line starting from $0$ of a GFF on $\h$ with the indicated boundary data.  The red segment represents the excursion of $\eta$ from $\partial \h$ which traces the boundary of the component $\CB_1$ of $\h \setminus \eta$ with $1$ on its boundary.  The time $\xi$ (resp.\ $\zeta$) is the start (resp.\ end) time of this excursion.  In Theorem~\ref{thm::excursion_law}, we show that the conditional law of $\eta|_{[\xi,\zeta]}$ given $\eta|_{[0,\xi]}$ and $\eta|_{[\zeta,\infty)}$ is an $\SLE_\kappa(\wh{\rho})$ process where $\wh{\rho} = \kappa-4-\rho$.}
\end{figure}

In this section, we will determine the law of the excursion of a boundary intersecting $\SLE_\kappa(\rho)$ process in $\h$ with $\kappa \in (0,4)$ which disconnects a given boundary point from $\infty$ (see Figure~\ref{fig::excursion_over_1}).  We will then use this result to determine the local behavior of the path at the starting and ending points of the excursion.

\begin{theorem}
\label{thm::excursion_law}
Fix $\kappa \in (0,4)$, $\rho \in (-2,\tfrac{\kappa}{2}-2)$, and let $\wh{\rho} = \kappa-4-\rho > \tfrac{\kappa}{2}-2$.  Let $\eta$ be an $\SLE_\kappa(\rho)$ process in $\h$ from $0$ to $\infty$ with a single boundary force point of weight $\rho$ located at~$0^+$.  Let $\CB_1$ be the component of $\h \setminus \eta$ which has $1$ on its boundary and let $\xi$ (resp.\ $\zeta$) be the first (resp.\ last) time that $\eta$ visits a point on $\partial \CB_1$.  Then the conditional law of $\eta|_{[\xi,\zeta]}$ given $\eta|_{[0,\xi]}$ and $\eta|_{[\zeta,\infty)}$ is that of an $\SLE_\kappa(\wh{\rho})$ process in the component of $\h \setminus \big( \eta([0,\xi]) \cup \eta([\zeta,\infty))\big)$ with $1$ on its boundary with a single boundary force point of weight $\wh{\rho}$ located at $(\eta(\xi))^+$.
\end{theorem}

Theorem~\ref{thm::excursion_law} gives the conditional law of $\eta|_{[\xi,\zeta]}$ given $\eta|_{[0,\xi]}$ and $\eta|_{[\zeta,\infty)}$.  We also know the conditional law of $\eta|_{[\zeta,\infty)}$ given $\eta|_{[0,\zeta]}$: it is an $\SLE_\kappa(\rho)$ process in the unbounded component of $\h \setminus \eta([0,\zeta])$ from $\eta(\zeta)$ to $\infty$ with a single boundary force point located at $(\eta(\zeta))^+$.  This just follows from the conformal Markov property for $\SLE_\kappa(\rho)$.  By the reversibility of $\SLE_\kappa(\rho)$ \cite[Theorem~1.1]{ms2012imag2}, we also know the conditional law of $\eta|_{[0,\xi]}$ given $\eta|_{[\xi,\infty)}$: it is an $\SLE_\kappa(\rho)$ in the component of $\h \setminus \eta([\xi,\infty))$ with $0$ on its boundary from $0$ to $\eta(\xi)$ with a single boundary force point of weight $\rho$ located at~$0^+$.  These facts will be important for us when we establish Proposition~\ref{prop::zoomed_in_picture} below.

\begin{figure}[ht!]
\begin{center}
\includegraphics[scale=0.85, page=2]{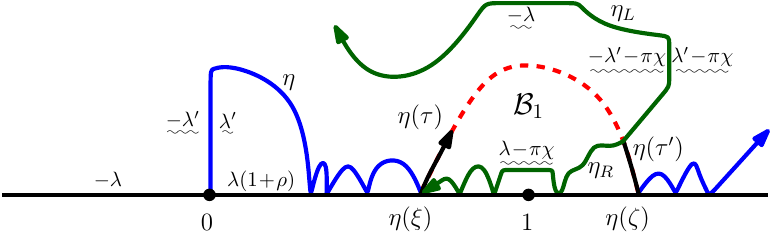}
\end{center}
\caption{\label{fig::excursion_over_1_proof}
Setup for the proof of Theorem~\ref{thm::excursion_law}.  We know the conditional law of $\eta|_{[\tau,\tau']}$ given $\eta|_{[0,\tau]}$, $\eta|_{[\tau',\infty)}$, and the dual flow lines $\eta_L$ and $\eta_R$ (i.e., with angles $\pi$ and $-\pi$).  We similarly know the conditional law of the dual flow lines $\eta_L$ and $\eta_R$ given all of $\eta$.  As explained in the proof of \cite[Theorem~6.1]{ms2012imag2}, these resampling properties determine the conditional law of $\eta|_{[\tau,\tau']}$ given $\eta|_{[0,\tau]}$ and $\eta|_{[\tau',\infty)}$.  The corresponding triple associated with an $\SLE_\kappa(\kappa-4-\rho)$ process satisfies the same resampling properties which gives the desired result.}
\end{figure}

\begin{proof}[Proof of Theorem~\ref{thm::excursion_law}]
The proof is based on resampling arguments which are the similar to those used to prove \cite[Theorem~1.1]{ms2012imag2}.  See Figure~\ref{fig::excursion_over_1_proof} for an illustration of the setup.

As illustrated in Figure~\ref{fig::excursion_over_1}, we view $\eta$ as the flow line starting from $0$ of a GFF $h$ on $\h$ with boundary conditions given by $-\lambda$ on $\R_-$ and $\lambda(1+\rho)$ on $\R_+$.  Let $\tau$ be a forward stopping time for $\eta$ with $\eta(\tau) \in \h$ a.s., let $\sigma$ be the first time after $\tau$ that $\eta(t) \in \partial \h$, and let $\tau'$ be a reverse stopping time for $\eta$ given $\eta|_{[0,\tau]}$ and $\eta|_{[\sigma,\infty)}$ such that $\eta(\tau') \in \h$ a.s.  Assume that we are working on the positive probability event that $\eta|_{[\tau,\tau']}$ does not intersect $\partial \h$.  We also assume that are working on the event that $\xi = \sup\{ t < \tau : \eta(t) \in \partial \h\}$ and $\zeta = \inf\{t \geq \tau : \eta(t) \in \partial \h\}$.  Let $\eta_L$ (resp.\ $\eta_R$) be the flow line of $h$ given $\eta$ starting from $\eta(\tau')$ with angle $\pi$ (resp.\ $-\pi$).  Note that $\eta_R$ a.s.\ terminates at~$\eta(\xi)$.  In particular, $\eta_R \cup \eta([0,\tau]) \cup \eta([\tau',\infty))$ separates $\eta|_{[\tau,\tau']}$ from $\partial \h$.  From the boundary data for the conditional law of $h$ given $\eta$, we can read off the conditional laws of $\eta_L$ and $\eta_R$ given $\eta$.  Moreover, \cite[Lemma~5.13]{ms2012imag2} gives us the conditional law of $\eta|_{[\tau,\tau']}$ given $\eta|_{[0,\tau]}$, $\eta|_{[\tau',\infty)}$, $\eta_L$, and $\eta_R$: it is given by an $\SLE_\kappa(\tfrac{\kappa}{2}-2;\tfrac{\kappa}{2}-2)$ process in the component of $\h \setminus (\eta_L \cup \eta_R \cup \eta([0,\tau]))$ with $\eta(\tau)$ on its boundary from $\eta(\tau)$ to $\eta(\tau')$.

Let $\wt{\eta}$ be an $\SLE_\kappa(\wh{\rho})$ process in the component of $\h \setminus \big(\eta([0,\xi]) \cup \eta([\zeta,\infty)) \big)$ with~$1$ on its boundary from $\eta(\xi)$ to $\eta(\zeta)$ with a single boundary force point located at $(\eta(\xi))^+$.  Let $\wt{\tau}$ be a stopping time for $\wt{\eta}$ and let $\wt{\tau}'$ be a reverse stopping time for $\wt{\eta}$ given $\wt{\eta}|_{[0,\wt{\tau}]}$ and assume that we are working on the event that $\wt{\eta}|_{[0,\wt{\tau}]}$ and $\wt{\eta}|_{[\wt{\tau}',\infty)}$ are disjoint.  Let $A = \eta([0,\xi]) \cup \eta([\zeta,\infty)) \cup \wt{\eta}$ and let $\wt{h}$ be a GFF on $\h \setminus A$ with the same boundary conditions as $h$ on $\partial \h$ and with boundary conditions along $A$ as if it were the flow line of $h$ from $0$ to $\infty$ (with angle $0$).  Let $\wt{\eta}_L$ (resp.\ $\wt{\eta}_R$) be the flow line of $\wt{h}$ starting from $\wt{\eta}(\wt{\tau}')$ with angle $\pi$ (resp.\ $-\pi$).  Then the conditional law of $\wt{\eta}_L$ and $\wt{\eta}_R$ given $A$ is the same as the conditional law of $\eta_L$ and $\eta_R$ given $\eta$ (in the setup described in the previous paragraph).  Moreover, \cite[Proposition~5.12]{ms2012imag2} gives us the conditional law of $\wt{\eta}$ given $\wt{\eta}_L$, $\wt{\eta}_R$, $\eta|_{[0,\xi]}$, $\eta|_{[\zeta,\infty)}$, $\wt{\eta}|_{[0,\wt{\tau}]}$, and $\wt{\eta}|_{[\wt{\tau}',\infty)}$: it is given by an $\SLE_\kappa(\tfrac{\kappa}{2}-2;\tfrac{\kappa}{2}-2)$ process in the component of $\h \setminus (\wt{\eta}_L \cup \wt{\eta}_R \cup \eta([0,\xi]) \cup \wt{\eta}([0,\wt{\tau}]))$ with $\wt{\eta}(\wt{\tau})$ on its boundary from $\wt{\eta}(\wt{\tau})$ to $\wt{\eta}(\wt{\tau}')$.

As explained in the proof of \cite[Theorem~6.1]{ms2012imag2}, this resampling invariance implies that the conditional law of $\eta$ given $\eta|_{[0,\tau]}$ and $\eta|_{[\tau',\infty)}$ (conditioned on the event described in the first paragraph) is the same as the conditional law of $\wt{\eta}$ given $\eta|_{[0,\xi]}$, $\eta|_{[\zeta,\infty)}$, $\wt{\eta}|_{[0,\wt{\tau}]}$, and $\wt{\eta}|_{[\wt{\tau}',\infty)}$.  This completes the proof since it holds for all $\tau$, $\tau'$, $\wt{\tau}$, and $\wt{\tau}'$.
\end{proof}

\begin{figure}[ht!]
\begin{center}
\includegraphics[scale=0.85]{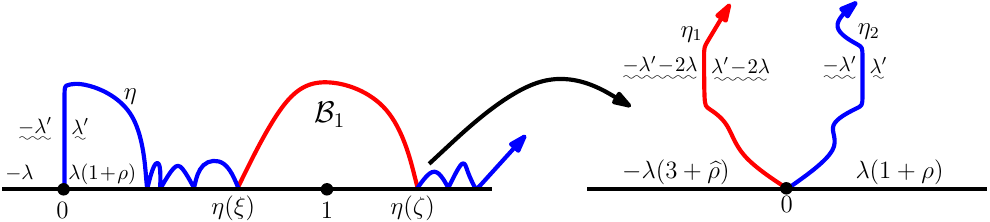}
\end{center}
\caption{\label{fig::excursion_zoom_in}
Illustration of the statement of Proposition~\ref{prop::zoomed_in_picture}.  Shown on the left is an $\SLE_\kappa(\rho)$ process $\eta$ with $\rho \in (-2,\tfrac{\kappa}{2}-2)$ and $\zeta$ is the first time that $\eta$ separates $1$ from $\infty$.  In Proposition~\ref{prop::zoomed_in_picture}, we show that the local picture of the pair of paths $t \mapsto \eta(\zeta-t)$ and $t \mapsto \eta(\zeta+t)$ near $\eta(\zeta)$ is given by a pair of flow lines of a GFF with the boundary data on the right with angles $2\lambda/\chi$ ($\eta_1$) and $0$ ($\eta_2$).
}
\end{figure}

\begin{proposition}
\label{prop::zoomed_in_picture}
Suppose that $\kappa \in (0,4)$ and $\rho \in (-2,\tfrac{\kappa}{2}-2)$.  Let $\eta$ be an $\SLE_\kappa(\rho)$ process in $\h$ with a single boundary force point of weight $\rho$ located at~$0^+$.  Let $\zeta$ be the first time $t$ that $\eta$ separates~$1$ from~$\infty$.  For each $\delta > 0$, we let $\eta_1^\delta = \delta^{-1}(\eta(\zeta-\cdot)-\eta(\zeta))$ and $\eta_2^\delta = \delta^{-1}(\eta(\zeta+\cdot)-\eta(\zeta))$.  The joint law of $(\eta_1^\delta,\eta_2^\delta)$ converges weakly as $\delta \to 0$ with respect to the topology of local uniform convergence modulo reparameterization to the law of a pair of paths $(\eta_1,\eta_2)$ which are constructed as follows.  Let $\wh{\rho}=\kappa-4-\rho$ and let $h$ be a GFF on $\h$ with boundary data given by $-\lambda(3+\wh{\rho})$ on $\R_-$ and $\lambda(1+\rho)$ on $\R_+$.  Then $\eta_1$ (resp.\ $\eta_2$) is the flow line of $h$ starting from~$0$ with angle $2\lambda/\chi$ (resp.\ $0$).
\end{proposition}

See Figure~\ref{fig::excursion_zoom_in} for an illustration of the setup.  Let $\xi$ be as in the statement of Theorem~\ref{thm::excursion_law}.  Note that $\xi$ corresponds to the first time that the time-reversal of $\eta$ separates $1$ from $\infty$.  Consequently, by the reversibility of $\SLE_\kappa(\rho_1;\rho_2)$ \cite[Theorem~1.1]{ms2012imag2}, Proposition~\ref{prop::zoomed_in_picture} can also be used to describe the local behavior of $\eta$ near $\eta(\xi)$.

\begin{figure}[ht!]
\begin{center}
\includegraphics[scale=0.85]{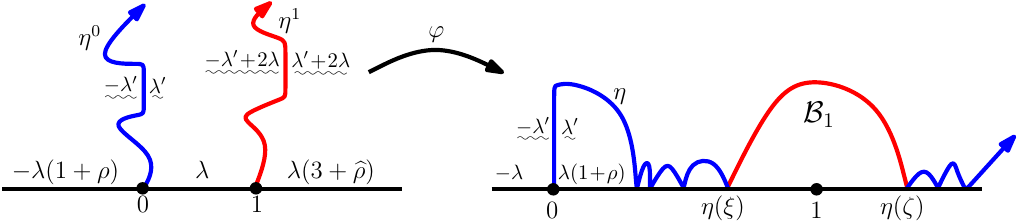}
\end{center}
\caption{\label{fig::bubble_and_continuation}
Setup for the proof of Proposition~\ref{prop::zoomed_in_picture}.  On the left, we suppose that we have a GFF $h$ on $\h$ with the indicated boundary data where $\eta^0$ (resp.\ $\eta^1$) is the flow line of $h$ starting from $0$ (resp.\ $1$) with angle $0$ (resp.\ $-2\lambda/\chi$).  Then $\eta^0 \sim \SLE_\kappa(\rho;2+\wh{\rho})$ and $\eta^1 \sim \SLE_\kappa(\rho+2;\wh{\rho})$.  Moreover, the conditional law of $\eta^0$ given $\eta^1$ is that of an $\SLE_\kappa(\rho)$ process and the conditional law of $\eta^1$ given $\eta^0$ is that of an $\SLE_\kappa(\wh{\rho})$ process.  In particular, the conditional law of $\eta^1$ given $\eta^0$ is the same as $\eta|_{[\xi,\zeta]}$ given $\eta|_{[0,\xi]}$ and $\eta|_{[\zeta,\infty)}$ and the conditional law of $\eta^0$ given $\eta^1$ is the same as the time-reversal of $\eta|_{[\zeta,\infty)}$ given $\eta|_{[0,\xi]}$ and $\eta|_{[\xi,\zeta]}$.  Consequently, \cite[Theorem~4.1]{ms2012imag2} implies that with $\varphi$ the unique conformal transformation from $\h$ to the unbounded component of $\h \setminus \eta([0,\xi])$ with $\varphi(0) = \infty$, $\varphi(1) = \eta(\xi)$, and $\varphi(\infty) = \eta(\zeta)$ we have that $\varphi(\eta^0,\eta^1) \stackrel{d}{=} (\eta|_{[\zeta,\infty)},\eta|_{[\xi,\zeta]})$ given $\eta|_{[0,\xi]}$ (up to time-reversal).
}
\end{figure}

\begin{proof}[Proof of Proposition~\ref{prop::zoomed_in_picture}]
We define the times $\xi,\zeta$ as in the statement of Theorem~\ref{thm::excursion_law}.  By Theorem~\ref{thm::excursion_law}, we know the conditional law of $\eta|_{[\xi,\zeta]}$ given $\eta|_{[0,\xi]}$ and $\eta|_{[\tau,\infty)}$: it an $\SLE_\kappa(\wh{\rho})$ process from $\eta(\xi)$ to $\eta(\zeta)$ in the component of $\h \setminus (\eta([0,\xi]) \cup \eta([\zeta,\infty)))$ with $1$ on its boundary.  Similarly, we know the conditional law of $\eta|_{[\zeta,\infty)}$ given $\eta|_{[0,\xi]}$ and $\eta|_{[\xi,\zeta]}$: it is an $\SLE_\kappa(\rho)$ process from $\eta(\zeta)$ to $\infty$ in the unbounded component of $\h \setminus \eta([0,\zeta])$.  By the reversibility of $\SLE_\kappa(\rho)$ \cite[Theorem~1.1]{ms2012imag2}, we also know that the time-reversal of $\eta|_{[\zeta,\infty)}$ given $\eta|_{[0,\xi]}$ and $\eta|_{[\xi,\zeta]}$ is an $\SLE_\kappa(\rho)$ process.  By \cite[Theorem~4.1]{ms2012imag2}, we have that these conditional laws determine the joint law of the pair $\eta|_{[\xi,\zeta]}$ and the time-reversal of $\eta|_{[\zeta,\infty)}$ conditional on $\eta|_{[0,\xi]}$.

The joint law can in particular be sampled from as follows (see Figure~\ref{fig::bubble_and_continuation} for an illustration).  Let $h$ be a GFF on $\h$ with boundary data $-\lambda(1+\rho)$ on $(-\infty,0]$, $\lambda$ on $(0,1]$, and $\lambda(3+\wh{\rho})$ on $(1,\infty)$.  Let $\eta^0$ (resp.\ $\eta^1$) be the flow line of $h$ starting from $0$ (resp.\ $1$) with angle $0$ (resp.\ $-2\lambda/\chi$).  Finally, let $\varphi$ be the unique conformal transformation from $\h$ to the unbounded component of $\h \setminus \eta([0,\xi])$ taking $0$ to $\infty$, $1$ to $\eta(\xi)$, and $\infty$ to $\eta(\zeta)$.  Then $(\varphi(\eta^0),\varphi(\eta^1),\eta|_{[0,\xi]})$ has the same law as $(\eta|_{[\zeta,\infty)},\eta|_{[\xi,\zeta]},\eta|_{[0,\xi]})$ (up to time-reversal).  The behavior of $\eta^0$ and $\eta^1$ near $\infty$ is asymptotically the same as that of the flow lines $(\wt{\eta}_1,\wt{\eta}_2)$ of a GFF on $\h$ with boundary conditions given by $-\lambda(1+\rho)$ (resp.\ $\lambda(3+\wh{\rho})$) on $\R_-$ (resp.\ $\R_+$) starting from $0$ with angles $0$ and $-2\lambda/\chi$.  (See the discussion after the statement of \cite[Lemma~5.13]{ms2012imag2} which explains how a flow line of a GFF from $0$ to $\infty$ can be realized as the common boundary of a pair of counterflow lines starting from $\infty$.)  This implies the result since the reversibility of $\SLE_\kappa(\rho_1;\rho_2)$ for $\kappa \in (0,4)$ and $\rho_1,\rho_2 > -2$ \cite[Theorem~1.1]{ms2012imag2} implies that the joint law of $(\wt{\eta}_1,\wt{\eta}_2)$ is invariant under applying the map $z \mapsto -1/z$ and then reversing time.
\end{proof}

\begin{proposition}
\label{prop::zoomed_in_picture_general_rho}
Suppose that $\kappa \in (0,4)$.  Fix weights $\rho_1 > -2$ and $\rho_2 \in (-2,\tfrac{\kappa}{2}-2)$.  Let $\eta$ be an $\SLE_\kappa(\rho_1;\rho_2)$ process with boundary force points located $x \leq 0 \leq y < 1$.  Let $\zeta$ be the first time that $\eta$ separates $1$ from $\infty$.  We define paths $\eta_1^\delta,\eta_2^\delta$ as in Proposition~\ref{prop::zoomed_in_picture}.  Then the joint law of $(\eta_1^\delta,\eta_2^\delta)$ converges weakly as $\delta \to 0$ with respect to the topology of local uniform convergence modulo reparameterization to the law of a pair of paths $(\eta_1,\eta_2)$ which are constructed as in Proposition~\ref{prop::zoomed_in_picture} where $\rho = \rho_2$.
\end{proposition}
\begin{proof}
This follows from absolute continuity and Proposition~\ref{prop::zoomed_in_picture}.
\end{proof}

\begin{remark}
\label{rem::sle_kappa_prime_bubble_closing_point}
Using the version of $\SLE$ duality established in \cite{ms2012imag1} and Proposition~\ref{prop::zoomed_in_picture_general_rho}, we can also describe the local behavior of an $\SLE_{\kappa'}(\rho_1';\rho_2')$ process near the point where it first separates $1$ from $\infty$.  Indeed, fix $\kappa' > 4$, $\rho_1' > -2$, and $\rho_2' \in (\tfrac{\kappa'}{2}-4,\tfrac{\kappa'}{2}-2)$.  Suppose that $\eta'$ is an $\SLE_{\kappa'}(\rho_1';\rho_2')$ process in $\h$ from $0$ to $\infty$ with boundary force points located at~$0^-$ and~$0^+$.  (Note that $\rho_2'$ is in the range so that $\eta'$ hits but does not fill $\R_+$.)  Let $\eta$ denote the right boundary of $\eta'$.  Then $\eta$ is an $\SLE_\kappa(\rho_1;\rho_2)$ process in $\h$ from $\infty$ to $0$ with
\[ \rho_1 = \frac{\kappa}{2} - 2 + \frac{\kappa \rho_1'}{4} \quad\text{and}\quad \rho_2 = \frac{\kappa}{2}-4 + \frac{\kappa \rho_2'}{4}\]
where the force points are located immediately to the left and right of $\infty$.  Let $\zeta'$ be the first time that $\eta'$ separates $1$ from $\infty$.  Since the right boundary of $\eta'$ is given by $\eta$ and $\eta \sim \SLE_\kappa(\rho_1;\rho_2)$, we can consequently apply Proposition~\ref{prop::zoomed_in_picture_general_rho} to $\eta$ to get that the law of the right boundary of $\eta'$ near $\eta'(\zeta')$ is described by a certain pair of GFF flow lines (same law as would correspond to $\SLE_\kappa(\rho_2)$).
\end{remark}

We finish this section by giving the radial version of Proposition~\ref{prop::zoomed_in_picture}.

\begin{proposition}
\label{prop::zoomed_in_picture_radial}
Suppose that $\kappa \in (0,4)$ and $\rho \in (-2,\tfrac{\kappa}{2}-2)$.  Let $\eta$ be a radial $\SLE_\kappa(\rho)$ process in $\D$ starting from $1$ and targeted at $0$ with a single boundary force point of weight $\rho$ located at $1^+$.  Fix $\theta \in (0,2\pi)$ and let $\zeta$ be the first time that $\eta$ separates $e^{i \theta}$ from $0$.  On the positive probability event that $\eta(\zeta) \in \partial \D$, we let $\varphi \colon \D \to \h$ be the unique conformal transformation which takes $\eta(\zeta)$ to $0$ and~$0$ to~$i$.  Let $\eta_\delta^1 = \delta^{-1} \varphi(\eta(\zeta-\cdot))$ and $\eta_\delta^2 = \delta^{-1} \varphi(\eta(\zeta+\cdot))$.  Then the joint law of $(\eta_1^\delta,\eta_2^\delta)$ converges weakly as $\delta \to 0$ with respect to the topology of local uniform convergence modulo reparameterization to the law of a pair of paths $(\eta_1,\eta_2)$ which are constructed as in Proposition~\ref{prop::zoomed_in_picture}.
\end{proposition}
\begin{proof}
This follows from absolute continuity and Proposition~\ref{prop::zoomed_in_picture} using the same argument used to prove Proposition~\ref{prop::zoomed_in_picture_general_rho}.
\end{proof}

\section{Structure theorems and quantum natural time}
\label{sec::structure_theorems}

\newcommand{\bsize}{\ol{\epsilon}}

The purpose of this section is to describe the Poissonian structure of the surfaces which are cut out of a $\gamma$-LQG surface by an independent $\SLE$ process.  Sections~\ref{subsec::skinny_wedge},~\ref{subsec::cone_bubbles}, and~\ref{subsec::sle_kappa_prime_bubbles} respectively focus on the cases of boundary-intersecting $\SLE_\kappa(\rho)$ processes, radial and whole-plane $\SLE_\kappa(\rho)$ processes, and $\SLE_{\kappa'}$ processes on top of a free boundary GFF with an appropriate $\log$ singularity and the additive constant normalized in a particular way.  The strategy to derive each of these results is very similar.  We have included all of the details in the case of boundary intersecting $\SLE_\kappa(\rho)$ processes (Theorem~\ref{thm::skinny_wedge_bubble_structure}) and in the other cases we explain the necessary modifications to transfer the argument.  These statements for $\SLE$ on top of a free boundary GFF are in fact intermediate results that will be used to deduce the corresponding result when we instead work on top of a quantum wedge.  This will be deduced in Section~\ref{subsec::quantum_typical_zip_unzip}.

Associated with each of the aforementioned structure theorems is a local time which should be thought of as counting the number of small bubbles which have been cut off by the $\SLE$.  (As $\epsilon$ tends to zero, the number of excised bubbles of quantum mass between $\epsilon$ and $2 \epsilon$ --- times an appropriate power of $\epsilon$ --- converges to this local time.) The right-continuous inverse of this local time leads to an alternative time-parameterization for $\SLE$ which is intrinsic to these bubbles (viewed as quantum surfaces).  We will refer to this notion of time as ``quantum natural time.''  For $\kappa' \in (4,8)$, it is the LQG analog of the natural parameterization for $\SLE$ \cite{ls2011natural_param,lz2013natural_param,lr2012minkowski,law2015minkowski,gps2013pivotal,lv2016lerw,b2017natural,hls2018natural} (for $\kappa \in (0,4)$ the LQG analog of the natural parameterization is $\gamma$-LQG length and for $\kappa' \geq 8$ it is $\gamma$-LQG area).  In Section~\ref{subsec::quantum_typical_zip_unzip}, we will show that the laws of certain types of quantum wedges are invariant under zipping/unzipping according to quantum natural time.  This is analogous to one of the main results of \cite{she2010zipper} (where the notion of time is given by $\gamma$-LQG length).

Throughout, we let $\strip = \R \times [0,\pi]$, $\strip_+ = \R_+ \times [0,\pi]$, $\strip_- = \R_- \times [0,\pi]$, and $\CH_1(\cdot),\CH_2(\cdot)$ be as in Section~\ref{sec::surfaces}.  In what follows, we will be considering both forward and reverse Loewner flows.  We will distinguish between forward and reverse by using a tilde for objects associated with the latter.

\subsection{Strategy overview}
\label{subsec::strategy}

We will now give a general overview of the strategy of this section in the case of a chordal $\SLE_\kappa(\rho)$ process.  The strategy in the case of radial $\SLE_\kappa(\rho)$ processes and $\SLE_{\kappa'}$ processes is analogous.  Assume that $\rho \in (-2,\kappa/2-2)$ so that an $\SLE_\kappa(\rho)$ process is boundary intersecting.
\begin{itemize}
\item We first consider a free boundary GFF on~$\h$ with a~$\log$ singularity at $0$ which is appropriate for the reverse coupling of $\SLE_\kappa(\wt{\rho})$, $\wt{\rho} = \rho+4$ (recall Proposition~\ref{prop::zip_up_to_local_time}), with the GFF (recall Theorem~\ref{thm::reverse_coupling}).  More specifically, $h = \wh{h} - \tfrac{\rho+2}{\gamma} \log|\cdot|$ where~$\wh{h}$ is a free boundary GFF on~$\h$.
\item Let $\eta$ be an $\SLE_\kappa(\rho)$ process in~$\h$ from~$0$ to~$\infty$ with a single boundary force point located at~$0^+$.  We assume that~$\eta$ is sampled independently of $h$ and then we fix the additive constant for $h$ so that the average of $h \circ f_{T_u}^{-1} + Q\log|(f_{T_u}^{-1})'|$ on $\h \cap \partial \D$ is equal to $0$.  Here, $T_u$ is the right-continuous inverse of the local time of $V-W$ at $0$ where $(W,V)$ is the driving pair for $\eta$, $(f_t)$ is its centered forward Loewner evolution, and $u > 0$ is large and fixed.
\item We then show (Theorem~\ref{thm::skinny_wedge_bubble_structure}) that the structure of the bubbles separated by $\eta|_{[0,T_u]}$ from $\infty$ which are to the right of $\eta$ agree with the initial part of those of a weight $\rho+2$ quantum wedge.  Here, the bubbles are ordered from \emph{right to left} starting from $\eta(T_u)$.  This may look like the wrong order in view of the statement of Theorem~\ref{thm::welding}.  It will arise because the proof is based on the reverse coupling of $\SLE$ with the GFF, in which we are performing a reverse Loewner flow.  We emphasize that the ``size'' of this initial segment is defined in terms of $T_u$, which is related to the capacity of $\eta$ since it is defined in terms of the driving pair.  In particular, it is not intrinsic to the weight $\rho+2$ quantum wedge structure because it depends on the entire surface.  (See Figure~\ref{fig::skinny_wedge_illustration} for an illustration.)
\item The weight $\rho+2$ quantum wedge structure of these bubbles gives us a definition of a quantum measure on the set $\eta([0,T_u]) \cap \partial \h$.  We can then understand the local behavior of the field at a typical point chosen from this quantum measure.  This is analogous to the statement that when one chooses a typical point from the boundary measure, the local behavior of the field is described by a $\gamma$-quantum wedge (weight $2$).  Here, we will see a different wedge weight which is a reflection of the fact that the measure we will analyze is supported on a fractal set rather than the whole domain boundary.  (See Lemma~\ref{lem::field_weighted_by_quantum_local_time}.)
\item At the same time, the bubbles which are to the right of $\eta([0,T_u])$ near and to the right of this quantum typical intersection point will look like a weight $\rho+2$ quantum wedge, but this time ordered from left to right (since this amounts to recentering a Poisson point process at a typical time and then using the reversibility of the measure used in its construction).  This will complete the proof of the main structure theorem in this case because we will have identified the law of the bubbles which are to the right of an independent $\SLE_\kappa(\rho)$ process on top of a weight $\rho+4$ quantum wedge.  The weight $\rho+4$ will arise in the proof because in the reverse coupling of $\SLE_\kappa(\rho)$ the singularity at the origin is given by $-\tfrac{2+\rho}{\gamma}\log|\cdot|$.  Near a typical point $w$ chosen from the quantum measure on $\eta([0,T_u]) \cap \partial \h$ we will find that there is an additional singularity given by $-\gamma(2-d)\log|\cdot - w|$ where $d = 1+\tfrac{2}{\gamma^2}(\rho+2)$.  If we apply the change of coordinates from the centered Loewner flow up to the time that $w$ is mapped to $0$, the overall singularity at the origin is given by the sum $\tfrac{2+\rho-\gamma^2}{\gamma} \log| \cdot|$.  Recall that the formula~\eqref{eqn::wedge_weight} which gives the relationship between the weight and the multiple $-\alpha$ of the $\log$ function is given by $W = \gamma(\gamma+2/\gamma-\alpha)$.  Inserting $\alpha = \tfrac{\gamma^2-\rho-2}{\gamma}$ gives $W=\rho+4$.
  The invariance under cutting using this quantum time will be proved at the same time that this final structure theorem is deduced.
\end{itemize}
We emphasize that in the final bullet point above, there are two quantum wedges: the one which consists of the surfaces which are to the right of $\eta$ and the overall surface on top of which $\eta$ is drawn.  The former has weight $\rho+2$ and the latter has weight $\rho+4$.

At the end of this section, we will have not yet proved Theorem~\ref{thm::welding} because we will not have identified the law of the surface which is to the \emph{left} of an $\SLE_\kappa(\rho)$ process (when the force point is on the right), we will have not shown that the surface to the left and the surface to the right of the $\SLE$ are independent of each other, and we will not have proved results for general wedge weights on both the left and right sides.  These steps will be accomplished in Section~\ref{sec::quantumwedge}.

\subsection{Boundary-intersecting $\SLE_\kappa(\rho)$ processes}
\label{subsec::skinny_wedge}

The purpose of this section is to determine the law of the beads which are cut out by a boundary intersecting $\SLE_\kappa(\rho)$ process for $\kappa \in (0,4)$ and $\rho \in (-2,\tfrac{\kappa}{2}-2)$ drawn on top of a free boundary GFF with an appropriate $\log$ singularity and with the additive constant fixed in a certain way.  (Recall the correspondence between forward $\SLE_\kappa(\rho)$ and reverse $\SLE_\kappa(\wt{\rho})$ for $\wt{\rho}=\rho+4$ from Proposition~\ref{prop::zip_up_to_local_time}.)  See Figure~\ref{fig::skinny_wedge_illustration} for an illustration of the setup.

\begin{figure}[ht!]
\begin{center}
\includegraphics[scale=0.85]{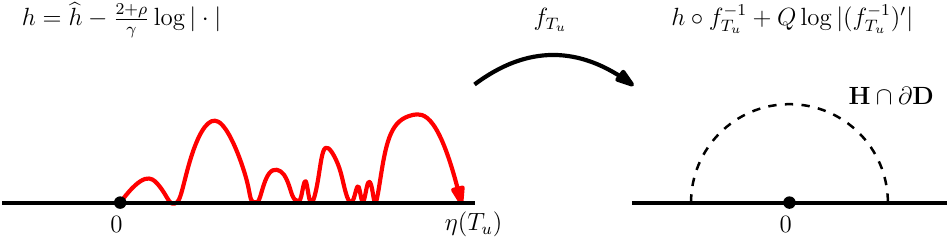}	
\end{center}
\caption{\label{fig::skinny_wedge_illustration} Illustration of the setup for Theorem~\ref{thm::skinny_wedge_bubble_structure}.  Shown on the left is a (forward) $\SLE_\kappa(\rho)$ process $\eta$ drawn up to the first time that the associated Bessel process has accumulated $u$ units of local time at $0$.  It is drawn on top of $h = \wh{h} - (2+\rho)/\gamma \log|\cdot|$ where $\wh{h}$ is a free boundary GFF which is first sampled independently of $\eta$ and with its additive constant then fixed so that the average of $h \circ f_{T_u}^{-1} + Q\log|(f_{T_u}^{-1})'|$ on $\h \cap \partial \D$ is equal to $0$.  Theorem~\ref{thm::skinny_wedge_bubble_structure} implies that the bubbles cut off from $\infty$ by $\eta|_{[0,T_u]}$, viewed as quantum surfaces and ordered from right to left, have the structure of an initial part (depending on $u$) of a weight $\rho+2$ quantum wedge.  We will show later in Section~\ref{subsec::quantum_typical_zip_unzip} that if one cuts along $\eta$ until a ``quantum typical'' intersection point of $\eta$ with $\partial \h$ and then zooms in, one obtains an independent $\SLE_\kappa(\rho)$ process drawn on top of a weight $4+\rho$ quantum wedge and the bubbles which are to the right of the former are a wedge of weight $\rho+2$.}
\end{figure}

\subsubsection{Statement}

\begin{theorem}
\label{thm::skinny_wedge_bubble_structure}
Fix $\kappa \in (0,4)$, $\rho \in (-2,\tfrac{\kappa}{2}-2)$, $\wt{\rho}=\rho+4$, and let $\gamma = \sqrt{\kappa}$.  Let $\wh{h}$ be a free boundary GFF on $\h$ and let
\[ h = \wh{h} + \frac{2-\wt{\rho}}{\gamma} \log| \cdot| = \wh{h} - \frac{2+\rho}{\gamma}\log|\cdot|.\]
Let $\eta$ be an $\SLE_\kappa(\rho)$ process starting from $0$ with a single boundary force point of weight $\rho$ located at $0^+$ and assume that $\eta$ is independent of $h$.  Let $(f_t)$ be the centered forward Loewner flow for $\eta$, $(W,V)$ be its driving process, $\ell_t$ be the local time of $V-W$ at $0$, and let $T_u = \inf\{t > 0 : \ell_t > u\}$ be the right-continuous inverse of~$\ell_t$.  Fix $u > 0$ large and assume that the additive constant for~$h$ is fixed so that the average of $h \circ f_{T_u}^{-1} + Q \log |(f_{T_u}^{-1})'|$ on $\h \cap \partial \D$ is equal to~$0$.  Then the law of the quantum surfaces parameterized by the bounded components of $\h \setminus \eta([0,T_u])$, ordered from right to left, is equal to those of a weight $\rho+2$ wedge up to a time~$s_u$ which tends to $\infty$ in probability as $u \to \infty$.
\end{theorem}

To further clarify the statement of Theorem~\ref{thm::skinny_wedge_bubble_structure}, we recall that the definition of a thin quantum wedge (Definition~\ref{def::skinny_wedge_bessel}) involves a Bessel process~$Y$.  Associated with this Bessel process is its local time $\qlt$ at~$0$.  The time~$s_u$ in the statement of Theorem~\ref{thm::skinny_wedge_bubble_structure} refers to the amount of local time associated with the bubbles (in the sense of the Bessel process which encodes the thin wedge) cut out by~$\eta$ by time~$T_u$.  We emphasize that $s_u$ is not equal to~$T_u$ because~$T_u$ refers to an amount of capacity time and the capacity of the bubbles depends on their embedding into $\h$.  Also, the natural direction of time for $Y$ is opposite to the natural direction of time for $\eta$; this will also manifest itself in the proof.  The motivation for the particular way that we have fixed the additive constant for $h$ will also become clearer in the proof.  There are also other ways of fixing the additive constant so that the same result holds.  In fact, any way of fixing the additive constant which only depends on the field $h \circ f_{T_u}^{-1} + Q \log| (f_{T_u}^{-1})'|$ would suffice.  For example, the statement of Theorem~\ref{thm::skinny_wedge_bubble_structure} is also true if we were to instead normalize $h \circ f_{T_u}^{-1} + Q \log| (f_{T_u}^{-1})'|$ so that the amount of LQG mass it associates with $\D \cap \h$ is equal to $1$.

Recalling~\eqref{eqn::wedge_weight} (as well as Table~\ref{tab::wedge_parameterization}), we note that the dimension of the Bessel process $Y$ associated with a wedge of weight $\rho+2$ is given by:
\begin{equation}
\label{eqn::bessel_wedge_correspondence}
\delta = \delta(\kappa,\rho) = 1 + \frac{2\rho+4}{\kappa} = 1 + \frac{2\rho+4}{\gamma^2}.
\end{equation}
Note that the value of $\delta(\kappa,\rho)$ in~\eqref{eqn::bessel_wedge_correspondence} varies in $(1,2)$ as $\rho$ varies in $(-2,\tfrac{\kappa}{2}-2)$.

\begin {figure}[ht!]
\begin {center}
\includegraphics[scale=0.85]{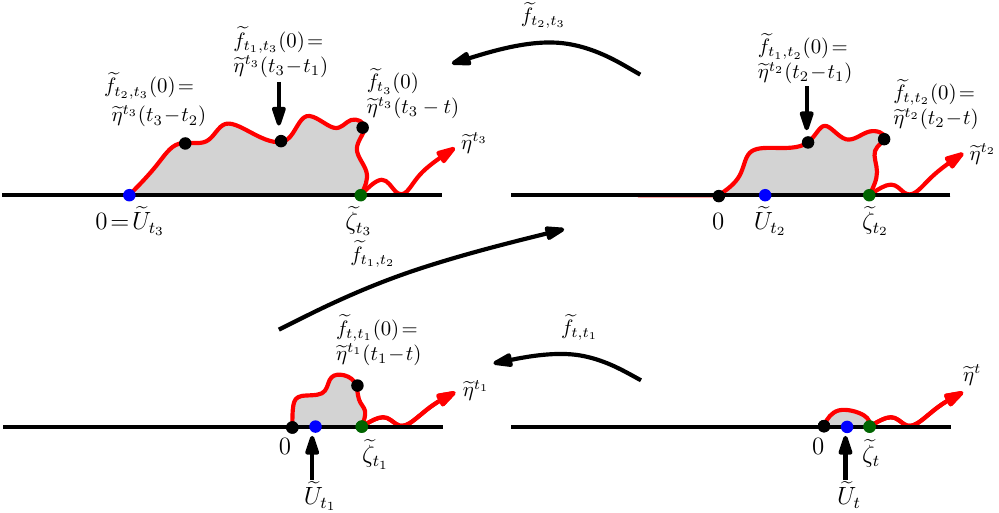}
\caption {\label{fig::cuttingslekr} Illustration of why the bubbles cut out by an $\SLE_\kappa(\rho)$ process have a Poissonian structure.  Shown on the bottom right is the curve $\wt{\eta}^t$ constructed in Lemma~\ref{lem::fields_consistent} for a given value of $t$ at which $\wt{\eta}^t$ is making an excursion away from $\partial \h$ starting from $0$ and terminating at $\wt{\zeta}_t$.  Times $t < t_1 < t_2 < t_3$ are chosen so that $t_3$ is the first time that $\wt{U}$, the location of the force point, collides with $0$ after time $t$.  This corresponds to the first time that the grey region is completely ``zipped in''.  An important observation is that the quantum surface which is parameterized by the grey region does not change as one performs the zipping up procedure.  (In particular, the quantum length of the interval $[\wt{U}_s,\wt{\zeta}_s]$ is the same for all $s \in [t,t_3]$.)  If the grey region is very small, then by the Markov property of the field, conditioning on the field values in this region tells us almost nothing about the rest of the surface.  As these field values determine the surface parameterized by the grey region, we see that observing it tells us almost nothing about the quantum surfaces parameterized by the subsequent excursions that $\wt{\eta}^t$ makes from $0$.}
\end{center}
\end{figure}

\begin{figure}[h!]
\begin{center}
\includegraphics[scale=0.85]{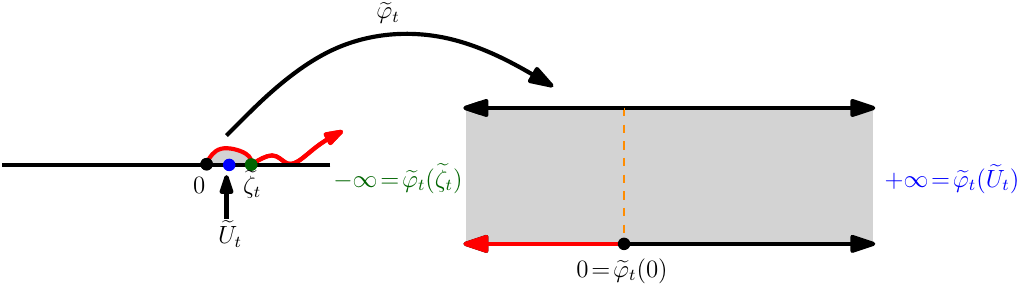}
\caption {\label{fig::smallexcursiontostrip} A conformal map from the bottom right part of Figure~\ref{fig::cuttingslekr} to the infinite strip~$\strip$.  The part of $\wt{\eta}^t$ that forms the upper boundary of the grey region on the left (the red curve) is mapped to $(-\infty,0]$.  The blue vertex on the left is mapped to the endpoint of the strip at $+\infty$, while the green vertex is mapped to the endpoint of the strip at $-\infty$.  Note that the field~$\wh{h}^t$ on the right obtained from the quantum coordinate change does not change as more of the grey region is zipped up (modulo horizontal translation).  Given~$\wt{\CF}_t$ (defined in Section~\ref{subsubsec::filtration}), the conditional law of~$\wh{h}^t$ on~$\strip_+$ (right of the dotted orange line) is that of a GFF on $\strip_+$ (with the field values to the left of the dotted orange line having been already determined), with free boundary conditions on the upper and lower boundaries, and an asymptotic linear growth rate at $+\infty$ depending on the weight of the force point.  This linear growth rate is what determines the structure of the Poisson law because it will determine the type of Bessel process (hence quantum wedge) which describes the bubbles (recall Proposition~\ref{prop::bessel_exponential_bm}).}
\end{center}
\end{figure}

\subsubsection{Intuition and setup}

Throughout, we suppose that $\rho$ and $\wt{\rho}$ are as in the statement of Theorem~\ref{thm::skinny_wedge_bubble_structure}.  Let $(\wt{f}_t)$ denote the centered reverse Loewner flow associated with a reverse $\SLE_\kappa(\wt{\rho})$ process.  Let $(\wt{W},\wt{V})$ be the associated driving process and let $\wt{U} = \wt{V}-\wt{W}$.  Note $\wt{U} \geq 0$ since we have taken the force point to start at $0^+$.  For each $s \leq t$, we also let $\wt{f}_{s,t} = \wt{f}_t \circ \wt{f}_s^{-1}$.  Equivalently, $\wt{f}_{s,t}$ is given by solving the centered reverse Loewner equation using the driving function $(\wt{W},\wt{V})$ in the time-interval $[s,t]$.

In the coupling of reverse $\SLE_\kappa(\wt{\rho})$ with the GFF described in Theorem~\ref{thm::reverse_coupling}, the collision times of $\wt{W}$ and $\wt{V}$ correspond to the instances of time when new quantum surfaces are added to the big surface (i.e., the surface parameterized by $\h$).  This is because when one performs the forward Loewner flow, these times correspond to collision times of the path with the boundary.  Roughly speaking, the reason that the surfaces in $\h \setminus \eta$ which are to the right of~$\eta$ in Theorem~\ref{thm::skinny_wedge_bubble_structure} have a Poissonian law is that the structure of a newly zipped in surface (in an infinitesimal amount of time after it is added to the big surface) is determined by the values of the GFF in an infinitesimal region and the conditional law of the GFF given its values in a very small region will be very close to the law of the unconditioned field (see e.g., Lemma~\ref{lem::sigma_algebra_is_trivial}).  This tells us that the quantum surfaces to the right of $\eta$ should be independent of each other.  See Figure~\ref{fig::cuttingslekr} for an illustration of this and Figure~\ref{fig::smallexcursiontostrip} for an illustration of how to obtain the form of the Poissonian law.

We now proceed to the details of the proof of Theorem~\ref{thm::skinny_wedge_bubble_structure}.  The remainder of this section is structured as follows.  We will first define a collection of stopping times that serve to restrict our attention to ``large'' bubbles.  We will then study the quantum surface structure of these bubbles.  As we will see, it will be convenient to use $\strip$ as the underlying domain on which we will parameterize them.  The first step is Proposition~\ref{prop::line_average_law}, which will determine the form of the projection of the field which describes such a bubble onto $\CH_1(\strip)$ (with fixed additive constant).  This, in turn, will determine the form of the Poisson law (see Figure~\ref{fig::smallexcursiontostrip}).  The second step is to control the dependency structure between the bubbles, which will be carried out primarily in Lemma~\ref{lem::harmonic_converges}.

In the proof, it will be useful to be able to keep track of the fields which arise when zipping and unzipping using a common probability space.  This is the purpose of the following lemma.

\begin{lemma}
\label{lem::fields_consistent}
There exists a family $(\wt{h}^t)$ of modulo additive constant distributions such that for each $t \geq 0$, $\wt{h}^t$ can be expressed as the sum of a free boundary GFF on $\h$ plus $\tfrac{2}{\gamma} \log|\cdot| - \tfrac{\wt{\rho}}{\gamma} \log|\cdot-\wt{U}_t|$ which satisfy
\[ \wt{h}^s = \wt{h}^t \circ \wt{f}_{s,t} + Q \log|\wt{f}_{s,t}'| \quad\text{for each} \quad 0 \leq s \leq t < \infty.\]
There also exists a family of curves $(\wt{\eta}^s)$, each defined on $[0,\infty)$, where $\eta = \wt{\eta}^0$ is an $\SLE_\kappa(\rho)$ process in $\h$ from $0$ to $\infty$ with a single boundary force point of weight $\rho$ located at $0^+$ which is independent of $\wt{h}^0$ and, for each $r,t > 0$, the hull of $\wt{\eta}^t([0,t+r])$ is equal to the hull of $\wt{f}_t(\partial \h \cup \wt{\eta}^0([0,r]))$.
\end{lemma}

The transformation from $(\wt{h}^s,\wt{\eta}^s)$ to $(\wt{h}^t,\wt{\eta}^t)$ for $t > s$ corresponds to zipping up for $t-s$ units of capacity time and the transformation from $(\wt{h}^s,\wt{\eta}^s)$ to $(\wt{h}^t,\wt{\eta}^t)$ for $t < s$ corresponds to unzipping for $s-t$ units of capacity time.

\begin{proof}[Proof of Lemma~\ref{lem::fields_consistent}]
Fix $v \geq 0$.  Suppose that $(\wt{W}^v,\wt{V}^v) \stackrel{d}{=} (\wt{W},\wt{V})$, $\wt{U}^v = \wt{V}^v - \wt{W}^v$, and $(\wt{f}_t^v)$ is the centered reverse Loewner flow driven by $(\wt{W}^v,\wt{V}^v)$.  Let $\wt{\ell}_t^v$ be the local time of $\wt{U}_t^v$ at $0$ and let $\wt{T}_u^v = \inf\{t > 0 : \wt{\ell}_t^v > u\}$ be its right-continuous inverse.  For each $0 \leq t \leq \wt{T}_v^v$, we let $\wt{f}_{t,\wt{T}_v^v}^v = \wt{f}_{\wt{T}_v^v}^v \circ (\wt{f}_{t}^v)^{-1}$.  In other words, $\wt{f}_{t,\wt{T}_v^v}^v$ is the solution of the centered reverse Loewner flow driven by $(\wt{W}^v,\wt{V}^v)$ in $[t,\wt{T}_v^v]$.  We then let $\wt{h}^{\wt{T}_v,v}$ be the sum of a free boundary GFF (viewed modulo additive constant) on $\h$ plus $\tfrac{2-\wt{\rho}}{\gamma} \log| \cdot|$ taken to be independent of $(\wt{W}^v,\wt{V}^v)$.  For each $0 \leq t \leq \wt{T}_v^v$, we set $\wt{h}^{t,v} = \wt{h}^{\wt{T}_v,v} \circ \wt{f}_{t,\wt{T}_v}^v + Q \log| (\wt{f}_{t,\wt{T}_v}^v)'|$.  Theorem~\ref{thm::reverse_coupling} implies that for each $0 \leq t \leq \wt{T}_v^v$ we have that $\wt{h}^t$ (as a modulo additive constant distribution) has the law of the sum of a GFF with free boundary conditions on $\h$ and $\tfrac{2}{\gamma} \log|\cdot| - \tfrac{\wt{\rho}}{\gamma} \log| \cdot - \wt{U}_t^v|$.  Proposition~\ref{prop::zip_up_to_local_time} implies that there exists an $\SLE_\kappa(\rho)$ curve $\wt{\eta}^{\wt{T}_v^v,v}$ so that for each $t \in [0,\wt{T}_v^v]$ its centered (forward) mapping out function at time $t$ is $(\wt{f}_{\wt{T}_v^v-t,\wt{T}_v^v}^v)^{-1}$.  This defines $\wt{\eta}^{\wt{T}_v^v,v}$ in $[0,\wt{T}_v]$.  We extend its definition to $\R_+$ by taking $(\wt{f}_{\wt{T}_v^v}^v)^{-1}(\wt{\eta}^{\wt{T}_v^v,v}(\cdot+\wt{T}_v^v))$ to be independent of everything else.  For each $t \in [0,\wt{T}_v^v]$, we let $\wt{\eta}^{t,v}(\cdot) = (\wt{f}_{\wt{T}_v^v-t,\wt{T}_v^v}^v)^{-1}(\wt{\eta}^{\wt{T}_v^v,v}(\cdot+t))$.  Proposition~\ref{prop::zip_up_to_local_time} implies that $\wt{\eta}^{0,v}$ is an $\SLE_\kappa(\rho)$ process which, by construction, is independent of $\wt{h}^{0,v}$.  More generally, for all $0 \leq w \leq v$, we have that $\wt{\eta}^{\wt{T}_w^v}$ is an $\SLE_\kappa(\rho)$ process in $\h$ from $0$ to $\infty$.  It is determined by $(\wt{W}^v,\wt{V}^v)$ in $[\wt{T}_w^v,\wt{T}_v^v]$ and $\wt{\eta}^{0,v}$ and therefore independent of $\wt{h}^{\wt{T}_w^v,v}$.

 Theorem~\ref{thm::reverse_coupling} implies that if we take $w > v$ and define the fields $\wt{h}^{t,w}$ and paths $\wt{\eta}^{t,w}$ analogously for $0 \leq t \leq \wt{T}_w^w$, then we have that $(\wt{h}^{t,w},\wt{\eta}^{t,w})_{t \in [0,\wt{T}_v^w]} \stackrel{d}{=} (\wt{h}^{t,v},\wt{\eta}^{t,v})_{t \in [0,\wt{T}_v^v]}$.  The Kolmogorov extension theorem thus implies that one can find a family $(\wt{h}^t,\wt{\eta}^t)$ as in the lemma statement which is defined for $t \geq 0$.  We let write $(\wt{W},\wt{V})$ for the corresponding reverse Loewner flow and $\wt{U} = \wt{V} - \wt{W}$.  If we let $\wt{\ell}$ be the local time of $\wt{U}$ at $0$ and $\wt{T}_u = \inf\{t > 0 : \wt{\ell}_t > u\}$ its right-continuous inverse, then the construction immediately implies that $(\wt{h}^{t+T_u},\wt{\eta}^{t+T_u})_{t \geq 0} \stackrel{d}{=} (\wt{h}^t,\wt{\eta}^t)_{t \geq 0}$ for each $u \geq 0$.  Therefore a second application of the Kolmogorov extension theorem implies that we can find such a family which is defined for all $t \in \R$.
\end{proof}

We will assume throughout that $(\wt{h}^t)$ and $(\wt{\eta}^t)$ are as in the statement of Lemma~\ref{lem::fields_consistent}.  By definition, $\eta$ has the law of a (forward) $\SLE_\kappa(\rho)$ process.  For general times $t \geq 0$, however, $\wt{\eta}^t$ does not have the law of a (forward) $\SLE_\kappa(\rho)$ process.  Using the forward/reverse symmetry for $\SLE_\kappa(\rho)$ established in Proposition~\ref{prop::zip_up_to_local_time}, we do have that $\wt{\eta}^t$ has the law of a (forward) $\SLE_\kappa(\rho)$ when we zip up to a given local time for $\wt{U}$ at $0$.  We record this fact in the following lemma (which is implied by the proof of Lemma~\ref{lem::fields_consistent}).

\begin{lemma}
\label{lem::zip_up_sle_kappa_rho_law}
Suppose that we have the setup as described above.  Let $\wt{\ell}_t$ denote the local time of $\wt{U}_t$ at $0$ and, for each $u > 0$, we let $\wt{T}_u = \inf\{t > 0 : \wt{\ell}_t > u\}$ be the right-continuous inverse of $\wt{\ell}_t$.  For each $u >0$, we have that $\wt{\eta}^{\wt{T}_u}$ has the law of a (forward) $\SLE_\kappa(\rho)$ process.  In particular, 
\[ (\wt{h}^{\wt{T}_u},\wt{\eta}^{\wt{T}_u}) \stackrel{d}{=} (\wt{h}^0,\wt{\eta}^0) \quad\text{for each}\quad u \geq 0.\]
\end{lemma}
We emphasize that in the equality of laws in the statement of Lemma~\ref{lem::zip_up_sle_kappa_rho_law}, $\wt{h}^{\wt{T}_u}$ is a distribution defined modulo additive constant.

We fix the additive constant for $\wt{h}^0$ by taking its average on $\h \cap \partial \D$ to be equal to $0$.  Since $\wt{h}^0 = \wt{h}^t \circ \wt{f}_t + Q \log |\wt{f}_t'|$ for each $t \geq 0$, fixing the additive constant for $\wt{h}^0$ (so that it is a well-defined distribution, not just modulo additive constant) serves to fix the additive constant for $\wt{h}^t$ for all $t \geq 0$.  That is, the additive constant for $\wt{h}^t$ is fixed so that the average of $\wt{h}^t \circ \wt{f}_t + Q \log|\wt{f}_t'|$ on $\h \cap \partial \D$ is is equal to $0$.

Figure~\ref{fig::cuttingslekr} illustrates what happens to a given bubble when one performs either the forward or the reverse Loewner flow.  In the forward direction, there are certain special times at which the $\SLE$ completes an entire excursion away from $\partial \h$, and at each such time a positive-but-finite-quantum-area region is ``cut off'' from $\h$.  If we parameterize time in the reverse direction, then at such a time, a positive area region is added all at once.

We are now going to define the objects which are used in the proof of Theorem~\ref{thm::skinny_wedge_bubble_structure}; see Figure~\ref{fig::smallexcursiontostrip} for an illustration of the definitions.  For each $t \geq 0$ such that $\wt{U}_t \neq 0$, we let $\wt{\bubble}_t$ denote the component of $\h \setminus \wt{\eta}^t$ with $\wt{U}_t$ on its boundary.  If $\wt{U}_t = 0$, we take $\wt{\bubble}_t = \emptyset$.  Note that $\wt{\bubble}_t$ corresponds to the grey region in Figure~\ref{fig::cuttingslekr} and Figure~\ref{fig::smallexcursiontostrip}.  It is also a Jordan domain as its boundary consists of part of the curve~$\wt{\eta}^t$ and an interval in $\partial \h$.  Also, $\wt{\bubble}_t$ is the bubble which is being ``zipped in'' (i.e., ``zipped up'') by the reverse Loewner flow at time $t$.  In other words, the part of its boundary corresponding to the interval $[0,\wt{U}_t]$ is being glued to part of $\R_-$ by the reverse Loewner flow at time $t$.  We emphasize that $0$ is always the leftmost point of $\partial \wt{\bubble}_t \cap \partial \h$ when $\wt{\bubble}_t \neq \emptyset$.  The reason for this is that the time-reversal of the centered reverse Loewner flow is the centered forward Loewner flow.  When one considers the centered forward Loewner flow, the curve always starts from $0$.

For $t \geq 0$ such that $\wt{\bubble}_t \neq \emptyset$, we let
\begin{itemize}
\item $\wt{\zeta}_t$ be the closing point of $\partial \wt{\bubble}_t$ (the opening point is $0$).  In other words, $0 = \wt{\eta}^t(\inf\{s : \wt{\eta}^t(s) \in \partial \wt{\bubble}_t\})$ and $\wt{\zeta}_t = \wt{\eta}^t(\sup\{s : \wt{\eta}^t(s) \in \partial \wt{\bubble}_t\})$.
\item $\wt{\varphi}_t \colon \wt{\bubble}_t \to \strip$ be the unique conformal transformation which takes $\wt{U}_t$ to $+\infty$, $0$ to $0$, and $\wt{\zeta}_t$ to $-\infty$.  Note that the choice of the images of $\wt{U}_t$ and $\wt{\zeta}_t$ determines $\wt{\varphi}_t$ up to horizontal translation.  Taking $\wt{\varphi}_t$ so that $\wt{\varphi}_t(0) = 0$ fixes the horizontal translation.
\item $\wh{h}^t = \wt{h}^t \circ \wt{\varphi}_t^{-1} + Q \log| (\wt{\varphi}_t^{-1})'|$.
\item $\wh{X}^t$ be the projection of $\wh{h}^t$ onto $\CH_1(\strip)$ (with fixed additive constant).
\end{itemize}

\subsubsection{Filtration and stopping times}
\label{subsubsec::filtration}

There are different $\sigma$-algebras one can use with the zipping up process.  One option is to keep track only of the path (and not make any additional observations about the field).  Alternatively, one could imagine considering the $\sigma$-algebra $\wt{\CG}_t = \sigma(\wt{h}^t, \wt{\eta}^t)$ which is generated by the entire field and path corresponding to the zipping up process up to time $t$.  We will use an intermediate approach.  For each $t \geq 0$, we let~$\wt{\CF}_t$ be the $\sigma$-algebra in which the following are measurable:
\begin{itemize}
\item Both $\eta = \wt{\eta}^0$ and the driving pair $(\wt{W}_s,\wt{V}_s)$ of the reverse Loewner flow for $s \leq t$.
\item The field $\wt{h}^t$ restricted to $\h \setminus \wt{\bubble}_t$ (i.e., $\wt{h}^t$ restricted to any component of $\h \setminus \wt{\eta}^t$ {\em other} than the one currently being generated).
\item The restriction of $\wh{h}^t$ to $\strip_-$ (i.e., the restriction of $\wt{h}^t$ to the portion of the current bubble whose image under $\wt{\varphi}_t$ lies to the left of the dotted orange line of Figure~\ref{fig::smallexcursiontostrip}).
\end{itemize}

Note that $\wt{\zeta}_t$, $\wt{\varphi}_t$, and the restriction of $\wh{X}^t$ to $\strip_-$ are all $\wt{\CF}_t$-measurable.  Let us now explain why $(\wt{\CF}_t)$ is a filtration.  It is clear that the amount of information associated with $\eta=\wt{\eta}^0$, $(\wt{W}_s,\wt{V}_s)$ for $s \leq t$, and the field $\wt{h}^t$ restricted to $\h \setminus \wt{\bubble}_t$ is increasing in $t$.  We thus need to understand how the fields $\wh{h}^t$ are related to each other for different values of $t$.  Suppose that we have $s < t$ such that there is no $r \in [s,t]$ for which $\wt{U}_r = 0$.  This means that the quantum surface parameterized by $\wt{\bubble}_s$ is the same as the quantum surface parameterized by $\wt{\bubble}_t$.  Consequently, the field $\wh{h}^s$ describes the same quantum surface as the field $\wh{h}^t$.  Let $\wh{\varphi}_{s,t} = \wt{\varphi}_t \circ \wt{f}_{s,t} \circ \wt{\varphi}_s^{-1}$.  Then $\wh{\varphi}_{s,t}$ is a conformal transformation $\strip \to \strip$ which fixes $\pm \infty$.  In particular, it must be given by a horizontal translation.  That is, there exists $a \in \R$ so that $\wh{\varphi}_{s,t}(z) = z+a$, hence $\wh{\varphi}_{s,t}'(z) = 1$ for all $z \in \strip$.  Since we have that $\wh{h}^s = \wh{h}^t \circ \wh{\varphi}_{s,t} + Q \log| \wh{\varphi}_{s,t}'| = \wh{h}^t \circ \wh{\varphi}_{s,t}$, it follows that one can obtain $\wh{h}^t$ from $\wh{h}^s$ by horizontally translating the latter.  We note that although the quantum surface described by $\wh{h}^t$ does not change with $t$, the point which corresponds to $0$ does change.  To finish proving the claim, we just need to explain why $a < 0$.  This can be seen because the quantum length of $\h \cap \partial \wt{\CB}_t$ is increasing in $t$ in any interval of times in which $\wt{\CB}_t \neq \emptyset$.  Indeed, this follows because performing the reverse Loewner flow corresponds to identifying $\R_+$ (from $0$ to $\infty$) to $\R_-$ (from $0$ to $-\infty$) according to quantum length.  In particular, in an interval in which $\wt{\CB}_t$ is being pushed into $\h$ under the reverse Loewner flow, $\partial \wt{\CB}_t \cap \partial \h \subseteq \R_+$ is being identified with points in $\R_-$ according to quantum length.  Therefore the quantum length of $\partial \wt{\CB}_t \cap \h$ is increasing in $t$.    (Note that in fact the quantum length assigned to $\R_- \subseteq \partial \strip$ tends to $0$ as $t$ tends to the time that $\wt{\CB}_t$ first appears and to the quantum length of the clockwise part of $\partial \wt{\CB}_t$ from $\wt{U}_t$ to $\wt{\zeta}_t$ as $t$ tends to the time at which $\wt{\CB}_t$ is fully zipped in.)  This means that the quantum length of $\R_- \subseteq \partial \strip$ is also increasing in $t$, which can only happen if $a < 0$.

Fix $\epsilon>0, \bsize > 0$ and $r \in \R$.  We now define stopping times inductively as follows.  We let $\wt{\tau}_1^{\epsilon,\bsize,r}$ be the first time $t \geq 0$ that:
\begin{itemize}
\item The quantum length of $\partial \wt{\bubble}_t \cap \h$ is at least $\epsilon$ with respect to the $\gamma$-LQG length measure induced by $\wt{h}^t$,
\item $\inf\{ u \in \R : \wh{X}_u^t = r\} = 0$, and
\item The $\gamma$-LQG mass of $\strip_-$ associated with the field $\wh{h}^t$ is at least $\bsize$.
\end{itemize} 
We emphasize that the time $\wt{\tau}_1^{\epsilon,\bsize,r}$ is a.s.\ finite because for a given bubble, the horizontal translation for $\wh{h}^t$ changes continuously in $t$ and tends to the extreme values when one takes a limit as $t$ tends to either the first time that the bubble appears in the reverse Loewner flow or the first time that it is fully zipped in (recall the explanation just above about how the quantum length of $\R_- \subseteq \partial \strip$ changes in $t$).  In particular, if $\sup_{u \in \R} \wh{X}_u^t \geq r$, there will exist $s \in \R$ so that $\inf\{u \in \R : \wh{X}_u^s = r\} = 0$.

We then let $\wt{\sigma}_1^{\epsilon,\bsize,r}$ be the first time $t$ after time $\wt{\tau}_1^{\epsilon,\bsize,r}$ that $\wt{U}_t = 0$.  Given that $\wt{\tau}_j^{\epsilon,\bsize,r},\wt{\sigma}_j^{\epsilon,\bsize,r}$ for $1 \leq j \leq k$ have been defined for some $k \in \N$, we let $\wt{\tau}_{k+1}^{\epsilon,\bsize,r}$ be defined in exactly the same way as $\wt{\tau}_1^{\epsilon,\bsize,r}$ except with $t \geq \wt{\sigma}_k^{\epsilon,\bsize,r}$.  We then let $\wt{\sigma}_{k+1}^{\epsilon,\bsize,r}$ be the first time $t$ after time $\wt{\tau}_{k+1}^{\epsilon,\bsize,r}$ that $\wt{U}_t = 0$.

For each $j$, we let~$\wh{X}^{j,\epsilon,\bsize,r}$ be given by~$\wh{X}^{\wt{\tau}_j^{\epsilon,\bsize,r}}$ and $\wh{h}_j^{\epsilon,\bsize,r}$ be given by $\wh{h}^{\wt{\tau}_j^{\epsilon,\bsize,r}}$.  We emphasize that $\inf\{u \in \R : \wh{X}_u^{j,\epsilon,\bsize,r} = r\} = 0$.

The reason for this particular choice of stopping times is that in the proof of Theorem~\ref{thm::skinny_wedge_bubble_structure} we will analyze the structure of a given bubble when it has only been partially zipped into the big surface and these conditions will be a useful way to restrict our attention to a discrete collection of ``large bubbles.''

\subsubsection{Form of the drift}

The main input into the derivation of the form of the Poisson law in Theorem~\ref{thm::skinny_wedge_bubble_structure} is the following observation regarding the law of the sequences $(\wh{X}^{j,\epsilon,\bsize,r})$ and $(\wh{X}^{j,\bsize,r}) = (\wh{X}^{j,0,\bsize,r})$.  In particular, we will show that $(\wh{X}^{j,\epsilon,\bsize,r})$ and $(\wh{X}^{j,\bsize,r}) = (\wh{X}^{j,0,\bsize,r})$ are given by independent Brownian motions with a common downward linear drift.  Each such Brownian motion will ultimately correspond to an excursion of a Bessel process from $0$ (recall Proposition~\ref{prop::bessel_exponential_bm}).  In particular, the speed of the downward drift is what determines the dimension of the Bessel process that we will consider hence determines the form of the Poisson law.

\begin{proposition}
\label{prop::line_average_law}

Let
\[ a = \frac{\rho+4}{\gamma} - Q = \frac{\rho+2}{\gamma} - \frac{\gamma}{2}.\]
Let $\wh{B}_{2 t}^{j,\epsilon,\bsize,r} = at - \wh{X}_t^{j,\epsilon,\bsize,r}$ for $t \geq 0$.  Given $\wt{\CF}_{\wt{\tau}_j^{r,\epsilon,\bsize}}$, the process $\wh{B}^{j,\epsilon,\bsize,r}$ is a standard Brownian motion with $\wh{B}_0^{j,\epsilon,\bsize,r} = -r$.  
\end{proposition}

We will write $\wh{B}^{j,\bsize,r}$ for $\wh{B}^{j,0,\bsize,r}$.  Before we give the proof of Proposition~\ref{prop::line_average_law}, we will describe the behavior of the maps $\wt{\varphi}_t^{-1} \colon \strip \to \wt{\bubble}_t$ near $+\infty$.

\begin{lemma}
\label{lem::map_asymptotic_growth}
Fix $t \geq 0$ such that $\wt{U}_t \neq 0$.  Let $\wt{\psi}_t \colon \h \to \wt{\bubble}_t$ be the unique conformal transformation with $\wt{\psi}_t(0) = \wt{U}_t$, $\wt{\psi}_t(-1) = 0$, and $\wt{\psi}_t(\infty) = \wt{\zeta}_t$.  Then we have that
\begin{align*}
\lim_{R \to \infty} \sup_{\substack{z \in \strip \\ \re(z) \geq R}} \left| e^z (\wt{\varphi}_t^{-1})'(z) - \wt{\psi}_t'(0) \right| = 0 \quad\text{and}\quad
\lim_{R \to \infty} \sup_{\substack{z \in \strip \\ \re(z) \geq R}} \left| e^z ( \wt{\varphi}_t^{-1}(z) - \wt{U}_t) - \wt{\psi}_t'(0) \right| = 0.
\end{align*}
\end{lemma}
\begin{proof}
We first note that the map $-e^{-z}$ takes $\strip$ to $\h$ with $+\infty$ sent to $0$, $-\infty$ sent to $\infty$, and $0$ sent to $-1$.  The first part of the lemma then follows since $\wt{\varphi}_t^{-1}(z) = \wt{\psi}_t(-e^{-z})$.  To see the second part of the lemma, we note that
\begin{align*}
   \wt{\varphi}_t^{-1}(z) - \wt{U}_t
&= \wt{\psi}_t(-e^{-z}) - \wt{U}_t
 = -\int_z^\infty (\wt{\psi}_t(-e^{-w}))' dw
 = -\int_z^\infty \wt{\psi}_t'(-e^{-w}) e^{-w} dw\\
&= -\int_z^\infty (\wt{\psi}_t'(-e^{-w}) - \wt{\psi}_t'(0)) e^{-w} dw - \int_z^\infty \wt{\psi}_t'(0) e^{-w} dw\\
&= \wt{\psi}_t'(0) e^{-z} + o(|e^{-z}|) \quad\text{as}\quad \re(z) \to \infty.
\end{align*}
\end{proof}

\begin{proof}[Proof of Proposition~\ref{prop::line_average_law}]
Given $\wt{\CF}_{\wt{\tau}_j^{\epsilon,\bsize,r}}$, we note that $\wh{h}_j^{\epsilon,\bsize,r}$ is independent from $\wh{h}_1^{\epsilon,\bsize,r},\ldots,\wh{h}_{j-1}^{\epsilon,\bsize,r}$ (as the latter are in fact determined by $\wt{\CF}_{\wt{\tau}_j^{\epsilon,\bsize,r}}$).  To complete the proof we just need to identify the law of each $\wh{B}^{j,\epsilon,\bsize,r}$.

Recall that we can write $\wt{h}^t = \wh{h} + \tfrac{2}{\gamma}\log|\cdot| - \frac{\wt{\rho}}{\gamma}\log|\cdot - \wt{U}_t|$ where $\wh{h}$ is a free boundary GFF on $\h$.  We thus have that
\[ \wh{h}^t = \wh{h} \circ \wt{\varphi}_t^{-1} + \frac{2}{\gamma}\log|\wt{\varphi}_t^{-1}(\cdot)| - \frac{\wt{\rho}}{\gamma} \log|\wt{\varphi}_t^{-1}(\cdot) - \wt{U}_t| + Q \log| (\wt{\varphi}_t^{-1})'(\cdot)|.\]
Applying this for $t = \wt{\tau}_j^{\epsilon,\bsize,r}$, we see that the conditional law of the restriction of $\wh{h}_j^{\epsilon,\bsize,r}$ to $\strip_+$ given $\wt{\CF}_{\wt{\tau}_j^{\epsilon,\bsize,r}}$ is that of a GFF with the given boundary conditions on $[0,i \pi]$ and free boundary conditions on $\partial \strip_+ \setminus [0,i \pi]$ plus a harmonic function $\wh{\Fh}_j^{\epsilon,\bsize,r}$ on $\strip_+$.  The boundary conditions for $\wh{\Fh}_j^{\epsilon,\bsize,r}$ are equal to $0$ on $[0,i \pi]$, Neumann on $\partial \strip_+ \setminus [0,i \pi]$, and $\wh{\Fh}_j^{\epsilon,\bsize,r}$ has the same behavior at $+\infty$ as the harmonic function $\wh{\Fg}_t(\cdot) = -\tfrac{\wt{\rho}}{\gamma} \log|\wt{\varphi}_t^{-1}(\cdot) - \wt{U}_t| + Q \log| (\wt{\varphi}_t^{-1})'(\cdot)|$.  By this latter statement, we mean that $\wh{\Fh}_j^{\epsilon,\bsize,r}(z) - \wh{\Fg}_t(z)$ converges to a constant as $\re(z) \to \infty$ uniformly in $\im(z)$.  Lemma~\ref{lem::map_asymptotic_growth} implies that both $(\wt{\varphi}_t^{-1})'(z) = \alpha e^{-z} + o(|e^{-z}|)$ and $\wt{\varphi}_t^{-1}(z) - \wt{U}_t = \alpha e^{-z} + o(|e^{-z}|)$ as $\re(z) \to \infty$ uniformly in $\im(z)$ where $\alpha = \wt{\psi}_t'(0)$.  Let $a = \tfrac{\wt{\rho}}{\gamma} - Q$ be as in the statement of the proposition.  We can insert the asymptotic expressions for $(\wt{\varphi}_t^{-1})'(z)$ and $\wt{\varphi}_t^{-1}(z)$ into the definition of $\wh{\Fg}_t$ to obtain that
\begin{align*}
   \wh{\Fg}_t(z)
&= -\frac{\wt{\rho}}{\gamma} \log|\wt{\varphi}_t^{-1}(z) - \wt{U}_t| + Q \log| (\wt{\varphi}_t^{-1})'(z)|\\
&= -\frac{\wt{\rho}}{\gamma} \log|\alpha e^{-z} + o(|e^{-z}|)| + Q \log| \alpha e^{-z} + o(|e^{-z}|)|\\
&= a \re(z) -a\log|\alpha| + o(1) \quad\text{as}\quad \re(z) \to \infty
\end{align*}
uniformly in $\im(z)$.  Combining, we have for a constant $c$ that
\[ \wh{\Fh}_j^{\epsilon,\bsize,r}(z) =  \big( \wh{\Fh}_j^{\epsilon,\bsize,r}(z) - \wh{\Fg}_t(z) \big) + \wh{\Fg}_t(z) = a \re(z) + c + o(1) \quad\text{as}\quad \re(z) \to \infty\]
uniformly in $\im(z)$.  The desired result thus follows by combining everything with Lemma~\ref{lem::harmonic_strip_constant} because the function $\R_+ \to \R$ defined by starting with the function $\wh{\Fh}_j^{\epsilon,\bsize,r}$, which is harmonic on $\strip_+$ with Neumann boundary conditions on $\partial \strip_+ \setminus [0,\pi i]$, and averaging it on vertical lines is linear in $\re(z)$.
\end{proof}

\subsubsection{Controlling the dependency}

Proposition~\ref{prop::line_average_law} proved just above implies that the projections of the $\wh{h}_j^{\epsilon,\bsize,r}$ onto $\CH_1(\strip)$ (with fixed additive constant) and restricted to $\strip_+$ are i.i.d.\ and can be related to the excursions of a Bessel process from $0$ (recall Proposition~\ref{prop::bessel_exponential_bm}).  In order to complete the proof of Theorem~\ref{thm::skinny_wedge_bubble_structure}, we need to control the dependency between the projections of the fields onto $\CH_2(\strip)$.  Recall that conditioning $\wh{h}_j^{\epsilon,\bsize,r}$ on $\wt{\CF}_{\wt{\tau}_j^{\epsilon,\bsize,r}}$ corresponds to conditioning on the values of $\wh{h}_j^{\epsilon,\bsize,r}$ in $\strip_-$ (in addition to some other information which does not involve the values of $\wh{h}_j^{\epsilon,\bsize,r}$ in $\strip_+$).  It thus follows from the Markov property of the GFF that the conditional law of~$\wh{h}_j^{\epsilon,\bsize,r}$ given~$\wt{\CF}_{\wt{\tau}_j^{\epsilon,\bsize,r}}$ is given by that of a GFF in~$\strip_+$ with Dirichlet boundary conditions on $[0,\pi i]$ and free boundary conditions on the rest of $\partial \strip_+$.  In other words, the dependency between $\wh{h}_j^{\epsilon,\bsize,r}$ restricted to $\strip_+$ and $\wh{h}_1^{\epsilon,\bsize,r}, \ldots, \wh{h}_{j-1}^{\epsilon,\bsize,r}$ is encoded by the function which is harmonic in $\strip_+$ with boundary conditions given by the values of $\wh{h}_j^{\epsilon,\bsize,r}$ on $[0,\pi i]$ and Neumann boundary conditions on the rest of $\partial \strip_+$.

In what follows, we will make use of the following notation.  We let $\wh{w}_j^{\epsilon,\bsize,r} \in \R$ be such that the $\gamma$-LQG length associated with the field $\wh{h}_j^{\epsilon,\bsize,r}$ of $(-\infty,\wh{w}_j^{\epsilon,\bsize,r}]$ is equal to $\epsilon$ (note that there always is such a $\wh{w}_j^{\epsilon,\bsize,r}$ by the definition of the stopping times $\wt{\tau}_j^{\epsilon,\bsize,r}$).  We then let $\wh{\psi}_j^{\epsilon,\bsize,r}$ be the function which is harmonic in $[\wh{w}_j^{\epsilon,\bsize,r},\infty) \times [0,\pi]$ with Neumann boundary conditions on the horizontal parts of the strip boundary (and at $+\infty$) and Dirichlet boundary conditions on $\wh{w}_j^{\epsilon,\bsize,r} + [0,i\pi]$ with boundary values given by those of $\wh{h}_j^{\epsilon,\bsize,r}$.  We emphasize that the function $\wh{\psi}_j^{\epsilon,\bsize,r}$ has Neumann boundary conditions at $+\infty$; this rules out the possibility of $\wh{\psi}_j^{\epsilon,\bsize,r}(z)$ having linear growth in $\re(z)$ as $z \to +\infty$.  By the Markov property of the field, given $\wh{\psi}_j^{\epsilon,\bsize,r}$, we have that $\wh{h}_j^{\epsilon,\bsize,r}$ is independent from $\wh{h}_1^{\epsilon,\bsize,r},\ldots,\wh{h}_{j-1}^{\epsilon,\bsize,r}$.  (The Markov property is applied at the first time that the quantum length of the part of the bubble being zipped in at time $\tau_j^{\epsilon,\bsize,r}$ is exactly equal to $\epsilon$.)  The purpose of Lemma~\ref{lem::harmonic_converges} stated and proved just below is to show that the functions $\wh{\psi}_j^{\epsilon,\bsize,r}$ become trivial in the limit as $\epsilon \to 0$.  As we have explained just above, this will ultimately imply that the successive bubbles which are zipped in are independent.

\begin{lemma}
\label{lem::harmonic_converges}
Fix $j \in \N$, $\bsize > 0$, and $r \in \R$.  There exists a sequence $(\epsilon_k)$ of positive numbers decreasing to $0$ such that $\wh{\psi}_j^{\epsilon_k,\bsize,r}$ a.s.\ converges to the $0$ function on $\strip_+$ as $k \to \infty$ with respect to the topology of local uniform convergence modulo a global additive constant.
\end{lemma}

Before we proceed to the proof of Lemma~\ref{lem::harmonic_converges}, let us briefly explain the idea.  If one imagines that the bubble associated with the function $\wh{\psi}_j^{\epsilon,\bsize,r}$ is fully zipped in, then the behavior of $\wh{\psi}_j^{\epsilon,\bsize,r}$ (modulo additive constant) as $\epsilon \to 0$ is determined by the microscopic behavior of the field/path pair near the terminal point of the bubble.  We will use Proposition~\ref{prop::zoomed_in_picture} to deduce that this microscopic behavior cannot be sufficiently pathological to prevent $\wh{\psi}_j^{\epsilon,\bsize,r}$ from becoming constant as $\epsilon \to 0$.

\begin{proof}[Proof of Lemma~\ref{lem::harmonic_converges}]
Fix a value of $u > 0$ large.  For $x > 0$ rational, let $\CV_x$ be the connected component of $\h \setminus \wt{\eta}^{\wt{T_u}}$ with $x$ on its boundary.  (Note that a unique such component exists a.s.\ for all $x  > 0$ rational simultaneously since the probability that $\wt{\eta}^{\wt{T}_u}$ hits any given point in $(0,\infty)$ is zero.)  We also fix $\epsilon > 0$ and assume that we are working on the event that the quantum length of $\partial \CV_x \setminus \partial \h$ is at least $\epsilon$.  Note that the probability of this event tends to $1$ as $\epsilon \to 0$ (for $x$ fixed).  We let $\xi$ (resp.\ $\zeta$) be the first (resp.\ last) time that $\wt{\eta}^{\wt{T}_u}$ hits a point on $\partial \CV_x$ and let $\varphi^\epsilon \colon \CV_x \to \strip$ be the unique conformal transformation which takes $\wt{\eta}^{\wt{T}_u}(\xi)$ to $+\infty$ and $\wt{\eta}^{\wt{T}_u}(\zeta)$ to $-\infty$ where the horizontal translation has been fixed so that the quantum length of $(-\infty,0]$ under the $\gamma$-LQG boundary measure associated with $\wt{h}^{\wt{T}_u} \circ (\varphi^\epsilon)^{-1} + Q\log|((\varphi^\epsilon)^{-1})'|$ is equal to $\epsilon$.  Let $\psi^\epsilon$ be the function which is harmonic in $\strip_+$ with Neumann boundary conditions on $\partial \strip_+ \setminus [0,i\pi]$ and Dirichlet boundary conditions on $[0,i\pi]$ with boundary values given by the values of $\wt{h}^{\wt{T}_u} \circ (\varphi^\epsilon)^{-1} + Q\log|((\varphi^\epsilon)^{-1})'|$ on $[0,i\pi]$.

By Lemma~\ref{lem::harmonic_strip_constant}, it suffices to show that, as $\epsilon \to 0$, the law of $\psi^\epsilon$ converges weakly with respect to the topology of local uniform convergence in $\strip_+$ modulo a global additive constant to a function which is harmonic in $\strip_+$.  Indeed, here we are using that for each fixed $j \in \N$, $\epsilon > 0$, and $r \in \R$, if we ``zip up'' to a large enough local time as in Lemma~\ref{lem::zip_up_sle_kappa_rho_law} (i.e., take $u > 0$ large enough) then with high probability the bubble associated with $\wh{\psi}_j^{\epsilon,\bsize,r}$ will correspond to one of the $\CV_x$ for $\wt{\eta}^{\wt{T}_u}$ and $\wh{\psi}_j^{\epsilon,\bsize,r}$ will correspond to $\psi^\epsilon$ with a large horizontal translation.  Thus if $\psi^\epsilon$ is converging weakly to a limit, then $\wh{\psi}_j^{\epsilon,\bsize,r}$ will be converging to a constant due to the large horizontal translation.

We are going to deduce this from Proposition~\ref{prop::zoomed_in_picture} and the independence of $\wt{\eta}^{\wt{T}_u}$ and $\wt{h}^{\wt{T}_u}$ (viewed modulo additive constant).  For each $\delta > 0$ we let $\eta_1^\delta = \delta^{-1} (\wt{\eta}^{\wt{T}_u}(\zeta-\cdot)-\wt{\eta}^{\wt{T}_u}(\zeta))$ and $\eta_2^\delta = \delta^{-1}(\wt{\eta}^{\wt{T}_u}(\zeta+\cdot)-\wt{\eta}^{\wt{T}_u}(\zeta))$.  Then Proposition~\ref{prop::zoomed_in_picture} gives us that the law of $(\eta_1^\delta,\eta_2^\delta)$ converges as $\delta \to 0$ to the law of a certain pair of GFF flow lines $(\eta_1,\eta_2)$ starting from $0$.  Let $h^\delta$ be the field which arises by precomposing $\wt{h}^{\wt{T}_u}$ with $z \mapsto \delta(z+\wt{\eta}^{\wt{T}_u}(\zeta))$.  Then $h^\delta$ (modulo additive constant) converges as $\delta \to 0$ to a free boundary GFF on $\h$ by the scale and translation invariance of free boundary GFFs and the independence of $\wt{h}^{\wt{T}_u}$ (modulo additive constant) and $\wt{\eta}^{\wt{T}_u}$.  By the independence of $\wt{h}^{\wt{T}_u}$ (modulo additive constant) and $\wt{\eta}^{\wt{T}_u}$, we get that the triple $(\eta_1^\delta,\eta_2^\delta,h^\delta)$ converges in law as $\delta \to 0$ to the law of a triple consisting of a pair of GFF flow lines and an independent free boundary GFF on~$\h$.

For each $t \in (0,\zeta-\xi)$, we let $\wt{\varphi}_t^\epsilon$ be the conformal transformation which takes $\h \setminus \wt{\eta}^{\wt{T}_u}([\zeta-t,\zeta])$ to $\strip$ with $\wt{\eta}^{\wt{T}_u}(\zeta)$ taken to $-\infty$, $\wt{\eta}^{\wt{T}_u}(\zeta-t)$ taken to $+\infty$, and the horizontal shift normalized so that the $\gamma$-LQG boundary measure of $(-\infty,0]$ with respect to $\wt{h}^{\wt{T}_u} \circ (\wt{\varphi}_t^\epsilon)^{-1} + Q \log|((\wt{\varphi}_t^\epsilon)^{-1})'|$ is equal to $\epsilon$.  Then the law of the field $\wt{h}^{\wt{T}_u} \circ (\wt{\varphi}_t^\epsilon)^{-1} + Q \log |((\wt{\varphi}_t^\epsilon)^{-1})'|$ (viewed as a modulo additive constant distribution on $\strip$) has a scaling limit as~$t,\epsilon \to 0$ where $\epsilon \to 0$ at a sufficiently fast rate relative to the rate at which $t \to 0$.  This follows from what we explained in the previous paragraph.  Therefore the law of the function which is harmonic in $\strip_+$ with Neumann boundary conditions on $\partial \strip_+ \setminus [0,i\pi]$ and with Dirichlet boundary conditions on $[0,i\pi]$ given by the values of $\wt{h}^{\wt{T}_u} \circ (\wt{\varphi}_t^\epsilon)^{-1} + Q \log |((\wt{\varphi}_t^\epsilon)^{-1})'|$ on $[0,i\pi]$ converges weakly as $t,\epsilon \to 0$ as before with respect to the topology of local uniform convergence modulo a global additive constant.

The claimed result will follow by showing that $\wt{h}^{\wt{T}_u} \circ (\wt{\varphi}_t^\epsilon)^{-1} + Q \log|((\wt{\varphi}_t^\epsilon)^{-1})'|$ is close to $\wt{h}^{\wt{T}_u} \circ (\varphi^\epsilon)^{-1} + Q\log|((\varphi^\epsilon)^{-1})'|$ for small enough $\epsilon$ relative to $t$.  That this is the case can be seen as follows.  Let $t_\epsilon \in [\xi,\zeta]$ be such that the quantum length of $\wt{\eta}^{\wt{T}_u}([t_\epsilon,\zeta])$ is $\epsilon$ and assume that we are working on the event that $t_\epsilon \in (\zeta-t,\zeta)$.  Note that the probability of this event tends to $1$ as $\epsilon \to 0$ with $t$ fixed.  The image of a point $z \in \h \setminus \wt{\eta}^{\wt{T}_u}([t_\epsilon,\zeta])$ under $\wt{\varphi}_t^\epsilon$ is determined by the harmonic measure as seen from~$z$ of the boundary segments (counterclockwise) from~$\eta(\zeta)$ to~$\eta(t_\epsilon)$, from~$\eta(t_\epsilon)$ to~$\eta(\zeta-t)$, and from~$\eta(\zeta-t)$ to~$\eta(\zeta)$.  Similarly, the image of $z \in \CV_u$ under $\varphi^\epsilon$ is determined by the harmonic measure as seen from $z$ of the boundary segments (counterclockwise) of $\partial \CV_u$ from $\wt{\eta}^{\wt{T}_u}(\zeta)$ to $\wt{\eta}^{\wt{T}_u}(t_\epsilon)$, $\wt{\eta}^{\wt{T}_u}(t_\epsilon)$ to $\wt{\eta}^{\wt{T}_u}(\xi)$, and from $\wt{\eta}^{\wt{T}_u}(\xi)$ to $\wt{\eta}^{\wt{T}_u}(\zeta)$.  From this, it is easy to see that if $K \subseteq \strip$ is any fixed compact set, then $\wt{\varphi}_t^\epsilon \circ (\varphi^\epsilon)^{-1}$ converges uniformly to the identity on $K$ as $\epsilon \to 0$.
\end{proof}

\subsubsection{Proof of Theorem~\ref{thm::skinny_wedge_bubble_structure}}

First of all, suppose that $X_t = B_{2t} + a t$ where $B$ is a standard Brownian motion and $a$ is as in Proposition~\ref{prop::line_average_law}.  By the discussion in the beginning of Section~\ref{subsec::surfaces_strips_cylinders} (see also Proposition~\ref{prop::bessel_exponential_bm}), we know that $Z_t = \exp(\tfrac{\gamma}{2} X_t)$ (reparameterized to have quadratic variation $dt$) evolves as a $\bes^\delta$ with $\delta=\delta(\kappa,\rho)$ where $\delta(\kappa,\rho)$ is as in~\eqref{eqn::bessel_wedge_correspondence}.  Suppose that $Y$ is a $\bes^\delta$.  Fix $u > 0$ large and let $\wt{T}_u$ be as in the statement of Lemma~\ref{lem::zip_up_sle_kappa_rho_law}.  Proposition~\ref{prop::line_average_law} then allows us to construct a coupling between $Y$ and the sequence of bounded components of $\h \setminus \wt{\eta}^{\wt{T}_u}$ which are to the left of $\wt{f}_{\wt{T}_u}(0)$, viewed as quantum surfaces and ordered from right to left, as follows.  For any given $r$, we can couple the excursions that $Y$ makes above $\exp(\tfrac{\gamma}{2} r)$ with the sequence $\exp(\tfrac{\gamma}{2} \wh{X}^{j,\bsize,r})$, where we view both processes as starting from the first time that they hit $\exp(\tfrac{\gamma}{2} r)$, to be the same (after reparameterizing the latter according to quadratic variation) for those $j$ which correspond to bubbles zipped in before time $\wt{T}_u$.  Note that, for each $r$, given $(\wh{X}^{j,\bsize,r})$ the bubbles are otherwise independent as quantum surfaces since Lemma~\ref{lem::harmonic_converges} implies that the projection of $\wh{\psi}_j^{\epsilon_k,\bsize,r}$ onto $\CH_2(\strip)$ tends to $0$ a.s.\ as $k \to \infty$ for a sequence $(\epsilon_k)$ of positive numbers that decreases to $0$ sufficiently quickly.

Note that for each fixed $j \in \N$ and $r \in \R$ there exists $\bsize_0 > 0$ (random) such that $\bsize \in (0,\bsize_0)$ implies that $\wh{X}^{j,r} = \wh{X}^{j,\bsize,r}$.  Indeed, recall that at this point in the proof we know that the bubbles which correspond to the $(\wh{X}^{j,\bsize,r})$ are conditionally independent given the $(\wh{X}^{j,\bsize,r})$.  Moreover, note that the probability that the bubble which corresponds to $\wh{X}^{j,\bsize,r}$ has $\gamma$-LQG mass at least $e^{\gamma r}$ is uniformly positive as $\bsize \to 0$.  Thus if the claim were not true, we would easily be led to a contradiction to the fact that the amount of $\gamma$-LQG mass in compact subsets of the origin is a.s.\ finite.

Sending $\bsize \to 0$, we get an asymptotic coupling where the excursions of~$Y$ from~$0$ which reach level at least $\exp(\tfrac{\gamma}{2} r)$ up until some time $s_u$ each correspond to a bounded component of $\h \setminus \wt{\eta}^{\wt{T}_u}$ which is to the left of $\wt{f}_{\wt{T}_u}(0)$.  Sending $r \to -\infty$, we get an asymptotic coupling where the excursions of~$Y$ from~$0$ up until some time $s_u$ each correspond to a component of $\h \setminus \wt{\eta}^{\wt{T}_u}$ which is to the left of $\wt{f}_{\wt{T}_u}(0)$ and this correspondence is bijective.  Note that~$Y$ encodes these bubbles from right to left. \qed

\subsection{Radial and whole-plane $\SLE_\kappa(\rho)$ processes}
\label{subsec::cone_bubbles}

We are now going to determine the law of the bubbles (viewed as quantum surfaces) which are cut off by radial and whole-plane $\SLE_\kappa(\rho)$ processes.

\begin{theorem}
\label{thm::radial_bubble_form}
Fix $\kappa \in (0,4)$, $\rho \in (-2,\tfrac{\kappa}{2}-2)$, $\wt{\rho}=\rho+4$, and let $\gamma = \sqrt{\kappa}$.  Suppose that
\[ h = \wh{h} - \frac{\gamma^2+6-\wt{\rho}}{2\gamma}\log|\cdot| + \frac{2-\wt{\rho}}{\gamma}\log|\cdot-1| = \wh{h} - \frac{\gamma^2+2-\rho}{2\gamma}\log|\cdot| - \frac{\rho+2}{\gamma}\log|\cdot-1|\]
where $\wh{h}$ is a free boundary GFF on $\D$.  Let $\eta$ be a radial $\SLE_\kappa(\rho)$ process in $\D$ starting from $1$ and targeted at $0$ with a single boundary force point of weight $\rho$ located at $1^+$.  Assume that $\eta$ is independent of $h$.  Let $(f_t)$ be the centered Loewner flow associated with $\eta$, let $(W,V)$ be its Loewner driving pair, let~$\ell_t$ be the local time at~$1$ of $V/W$ and $T_u = \inf\{t > 0 : \ell_t > u\}$ be its right-continuous inverse.   Fix $u > 0$ and assume that the additive constant for $h$ has been fixed so that the average of $h \circ f_{T_u}^{-1} + Q \log |(f_{T_u}^{-1})'|$ on $\partial B(0,1/2)$ is equal to $0$.  Then the law of the quantum surfaces parameterized by the components of $\D \setminus \eta([0,T_u])$ separated from $0$ by $\eta$, in the reverse order in which they were cut off from $0$ by $\eta$, is equal to those of a weight $\rho+2$ wedge up to a time~$s_u$ which tends to~$\infty$ in probability as $u \to \infty$.
\end{theorem}

As in the statement of Theorem~\ref{thm::skinny_wedge_bubble_structure}, in the statement of Theorem~\ref{thm::radial_bubble_form}, the time~$s_u$ is not deterministic because it is not determined by the capacity of the bubbles cut out by $\eta$ up to time $T_u$ which, in turn, depends on both the bubbles in addition to the outside surface.  We also emphasize that the particular way in which the additive constant is fixed in Theorem~\ref{thm::skinny_wedge_bubble_structure} is not important, as long as it is done so in a way which only depends on the field $h \circ f_{T_u}^{-1} + Q \log| (f_{T_u}^{-1})'|$.

We consider the following setup.  We suppose that we have a collection of pairs $(\wt{h}^t,\wt{\eta}^t)$ defined for $t \geq 0$ together with a centered reverse radial $\SLE_\kappa(\rho)$ Loewner flow $(\wt{f}_t)$ with driving process $(\wt{W},\wt{V})$ such that with $\wt{f}_{s,t} = \wt{f}_t \circ \wt{f}_s^{-1}$ we have that
\[  \wt{h}^s = \wt{h}^t \circ \wt{f}_{s,t} + Q \log|\wt{f}_{s,t}'| \quad\text{for each}\quad 0 \leq s \leq t < \infty \]
and $\wt{f}_{s,t}^{-1}(\wt{\eta}^t) = \wt{\eta}^s$.  As in the chordal case, the transformation from $(\wt{h}^s,\wt{\eta}^s)$ to $(\wt{h}^t,\wt{\eta}^t)$ for $t > s$ corresponds to zipping up for $t-s$ units of capacity time and the transformation from $(\wt{h}^s,\wt{\eta}^s)$ to $(\wt{h}^t,\wt{\eta}^t)$ for $t < s$ corresponds to unzipping for $s-t$ units of capacity time.  Letting $U_t = V_t/W_t$, for each time $t$, we have that $\wt{h}^t$ (modulo additive constant) is equal in distribution to 
\[ \wh{h} - \frac{\gamma^2+2-\rho}{2\gamma}\log|\cdot| - \frac{\rho+2}{\gamma} \log|\cdot-\wt{U}_t| \]
where $\wh{h}$ is a free boundary GFF on~$\D$.  If~$\wt{\ell}_t$ denotes the amount of local time that~$\wt{U}_t$ has spent at~$1$ up to time~$t$ and $\wt{T}_u = \inf\{t > 0 : \wt{\ell}_t > u\}$ is the right-continuous inverse of~$\wt{\ell}_t$ then we in fact have that $(\wt{h}^{\wt{T}_u},\wt{\eta}^{\wt{T}_u}) \stackrel{d}{=} (\wt{h}^0,\wt{\eta}^0)$ for all $u \geq 0$ where here we view the $\wt{h}^t$ as modulo additive constant distributions.  We fix the additive constant for $\wt{h}^0$ so that its average on $\partial B(0,1/2)$ is equal to $0$.  This in turn fixes the additive constant for $\wt{h}^t$ for all $t \geq 0$.

\begin{proof}[Proof of Theorem~\ref{thm::radial_bubble_form}]
With the setup described just above, this follows from the same argument used to prove Theorem~\ref{thm::skinny_wedge_bubble_structure}.  The only difference that we need to explain is why the analogs of the functions $\wh{\psi}_j^{\epsilon_k,\bsize,r}$ converge to $0$ as $k \to \infty$ (modulo a global additive constant) where $(\epsilon_k)$ is a sequence of positive numbers which decrease to $0$ sufficiently quickly.  We can apply the same argument as in Theorem~\ref{thm::skinny_wedge_bubble_structure} (using Proposition~\ref{prop::zoomed_in_picture_radial} in place of Proposition~\ref{prop::zoomed_in_picture}) to get the result for the bubbles whose terminal boundary segment of $\gamma$-LQG length $\epsilon$ terminates in $\partial \D$ but we cannot apply this argument directly in the case of a bubble where the terminal boundary segment of length $\epsilon$ terminates in $\D$.  (The latter correspond to self-intersection points of the path.)  To get that all of the $\wh{\psi}_j^{\epsilon_k,\bsize,r}$ converge to $0$ as $k \to \infty$ (modulo a global additive constant), we can apply the above argument to the setting in which we have unzipped the path to a given rational time and use that a.s.\ for each bubble there exists a rational time such that if we unzip to that time then the terminal segment of $\gamma$-LQG length $\epsilon$ of the boundary of the bubble will end in $\partial \D$.
\end{proof}

We will now deduce a statement about whole-plane $\SLE_\kappa(\rho)$ from Theorem~\ref{thm::radial_bubble_form}.  We assume that we have the same setup as described just above except we fix the additive constant for $\wt{h}^0$ (hence $\wt{h}^t$ for all $t \geq 0$) by taking it so that the amount of $\gamma$-LQG area  it assigns to $\D$ is equal to $1$.  Fix $t \geq 0$ and let $(\wh{h}^t,\wh{\eta}^t)$ be given by the pair $(\wt{h}^t,\wt{\eta}^t)$ rescaled by $e^t$.  Sending $t \to \infty$, we have that the law of the pair $(\wh{h}^t,\wh{\eta}^t)$ converges to the law of the pair $(\wh{h},\wh{\eta})$ where $\wh{h}$ (modulo additive constant) is a whole-plane GFF plus $-(\gamma^2+2-\rho)/(2\gamma) \log|\cdot|$ and $\wh{\eta}$ is a whole-plane $\SLE_\kappa(\rho)$ process from $\infty$ to~$0$.  (We assume that $\wh{\eta}$ is parameterized by capacity as seen from $0$.)  Moreover, $\wh{\eta}$ is independent of~$\wh{h}$ (viewed modulo additive constant) and the additive constant of~$\wh{h}$ is fixed so that the component of $\C \setminus \wt{\eta}((-\infty,0])$ containing the origin has LQG mass equal to~$1$.  Theorem~\ref{thm::radial_bubble_form} implies that the bubbles separated by~$\wh{\eta}$ from~$0$ before time~$0$, after time-reversal, have the law of a weight $\rho+2$ quantum wedge.  This implies that the bubbles separated from $\infty$ by the time-reversal of~$\wh{\eta}$, after it has separated $1$ unit of quantum mass from $\infty$, have the law of a weight $\rho+2$ quantum wedge.  Note that the time-reversal of $\wh{\eta}$ is a whole-plane $\SLE_\kappa(\rho)$ process from $0$ to $\infty$ by \cite{ms2013imag4}.

We have therefore obtained the following corollary from Theorem~\ref{thm::radial_bubble_form}.

\begin{corollary}
\label{cor::cone_bubble_form}
Fix $\kappa \in (0,4)$, $\rho \in (-2,\tfrac{\kappa}{2}-2)$, $\wt{\rho} = \rho+4$, and let $\gamma = \sqrt{\kappa}$.  Suppose that
\[ h = \wh{h} - \frac{\gamma^2+6-\wt{\rho}}{2\gamma}\log|\cdot| = \wh{h} - \frac{\gamma^2+2-\rho}{2\gamma}\log|\cdot|,\]
$\wh{h}$ a whole-plane GFF, and $\eta$ is a whole-plane $\SLE_\kappa(\rho)$ process from $0$ to $\infty$.  Assume that we sample $h$ and $\eta$ to be independent and then fix the additive constant for $h$ so that if $\wh{\eta}$ is the time-reversal of $\eta$ then the amount of LQG mass assigned to the component of $\C \setminus \wh{\eta}((-\infty,0])$ containing $0$ is equal to $1$.  Then the sequence of bubbles that $\eta$ separates from $\infty$ after it has separated $1$ unit of mass from $\infty$ has the law of a quantum wedge of weight $\rho + 2$.
\end{corollary}

\subsection{$\SLE_{\kappa'}$ processes with $\kappa' \in (4,8)$}
\label{subsec::sle_kappa_prime_bubbles}

We will now determine the structure of the bubbles (viewed as quantum surfaces) which arise when cutting along an $\SLE_{\kappa'}$ process.  The following theorem gives the analogs of Theorem~\ref{thm::skinny_wedge_bubble_structure}, Theorem~\ref{thm::radial_bubble_form}, and Corollary~\ref{cor::cone_bubble_form} in this case.  Throughout, we will use $\CM$ to denote the infinite measure on quantum disks as defined in Section~\ref{subsec::disks_and_spheres}.

\subsubsection{Statement}

\begin{theorem}
\label{thm::kappa_prime_bubbles}
Fix $\kappa' \in (4,8)$, let $\gamma = 4/\sqrt{\kappa'}$, and suppose that
\[ h = \wh{h} + \frac{2}{\sqrt{\kappa'}}\log|\cdot| = \wh{h} + \frac{\gamma}{2} \log|\cdot|\]
where $\wh{h}$ is a free boundary GFF on $\h$.  Let $\eta'$ be an $\SLE_{\kappa'}$ process in $\h$ starting from $0$ and targeted at~$\infty$ which is independent of~$h$ and let $(f_t)$ be the centered Loewner flow associated with~$\eta'$.  Fix $t > 0$ and assume that the additive constant for~$h$ has been fixed so that the average of $h \circ f_t^{-1} + Q\log|(f_t^{-1})'|$ on $\h \cap \partial \D$ is equal to~$0$.  We order the components of $\h \setminus \eta'([0,t])$ so that a component~$\bubble$ comes before another component~$\bubble'$ if~$\eta'$ finishes drawing all of~$\partial \bubble$ before it finishes drawing all of~$\partial \bubble'$.  The law of the time-reversal of the ordered sequence of components of $\h \setminus \eta'([0,t])$, each viewed as a quantum surface and which were completely cut off by~$\eta'$ before time~$t$, is equal to that of a \ppp\ $\Lambda$ with intensity measure
\[ du \otimes \frac{1}{\nu_h(\partial \strip)} d \diskmeasure(h),\]
where $du$ denotes Lebesgue measure on $\R_+$, up to a time which tends to $\infty$ in probability as $t \to \infty$.  The same also holds if we (with the additive constant fixed as in Theorem~\ref{thm::radial_bubble_form} and Corollary~\ref{cor::cone_bubble_form}):
\begin{itemize}
\item Take $\eta'$ to be a radial $\SLE_{\kappa'}$ process in $\D$ from $1$ to $0$ and
\[ h = \wh{h} + \frac{2}{\sqrt{\kappa'}}\log|\cdot-1| - \frac{\kappa'+6}{2\sqrt{\kappa'}}\log|\cdot|
     = \wh{h} + \frac{\gamma}{2}\log|\cdot-1| - \frac{3\gamma^2+8}{4\gamma}\log|\cdot|\]
where $\wh{h}$ is a free boundary GFF on $\D$ independent of $\eta'$.
\item Take $\eta'$ to be a whole-plane $\SLE_{\kappa'}$ process from~$\infty$ to~$0$ and
\[ h = \wh{h} - \frac{\kappa'+2}{2\sqrt{\kappa'}} \log|\cdot| = \wh{h} - \frac{\gamma^2+8}{4\gamma} \log|\cdot|\]
where $\wh{h}$ is a whole-plane GFF independent of $\eta'$.
\end{itemize}
\end{theorem}

\subsubsection{Proof ideas}

The proof of Theorem~\ref{thm::kappa_prime_bubbles} will proceed along lines which are similar to Theorem~\ref{thm::skinny_wedge_bubble_structure}, Theorem~\ref{thm::radial_bubble_form}, and Corollary~\ref{cor::cone_bubble_form}, though there will be a few key differences which we now highlight.

\begin{itemize}
\item As before, we will deduce the quantum surface structure of a bubble cut out by $\eta'$ by conformally mapping it to $\strip$.  In this case, each bubble which is completely separated from $\partial \h$ by $\eta'$ only comes equipped with a single marked point which is given by the first (equivalently past point) on its boundary which is visited by $\eta'$.  This is in contrast to the bubbles cut out by a boundary intersecting $\SLE_\kappa(\rho)$ curve, which have \emph{two} marked points.  In order to define a map to $\strip$, we will take the intermediate step of considering the law of the surfaces corresponding to the components of $\h \setminus \eta'$ weighted by their $\gamma$-LQG boundary length.  As shown in Lemma~\ref{lem::weighted_quantum_surface_boundary}, this is equivalent to adding a certain $\log$ singularity to the boundary of each component of $\h \setminus \eta'$.  This weighting will introduce some technicalities since we will then need to unweight the law.
\item Since we will be weighting the bubbles according to their $\gamma$-LQG boundary length, we will give a different definition of a large bubble which will require that the diameter is at least some given value.
\item When $\eta'$ is in the process of cutting out a given bubble, it will in the mean time cut out many other bubbles as the set of double points is dense in the range of $\eta'$.  (This is in fact the reason that the bubbles cut out by $\eta'$ have a tree structure.)  This means that for the reverse Loewner flow, in between the time a given bubble first appears and is being zipped in, infinitely many other bubbles will appear.
\end{itemize}

\subsubsection{Setup and stopping times}

The setup for the proof of Theorem~\ref{thm::kappa_prime_bubbles} is similar to that of Theorem~\ref{thm::skinny_wedge_bubble_structure}.  Let $(\wt{f}_t)$ denote the centered reverse Loewner flow associated with a reverse $\SLE_{\kappa'}$ process.  For each $s \leq t$, we also let $\wt{f}_{s,t} = \wt{f}_t \circ \wt{f}_s^{-1}$.  As in the proof of Lemma~\ref{lem::zip_up_sle_kappa_rho_law}, it is not hard to see that we can construct a family of fields $(\wt{h}^t)$ such that, for each $t$, $\wt{h}^t$ can be expressed as the sum of a free boundary GFF on $\h$ plus $\tfrac{\gamma}{2}\log|\cdot|$ which satisfy
\[ \wt{h}^s = \wt{h}^t \circ \wt{f}_{s,t} + Q \log|\wt{f}_{s,t}'| \quad\text{for each} \quad 0 \leq s \leq t < \infty.\]
Let $\eta'$ be an $\SLE_{\kappa'}$ process independent of $h = \wt{h}^0$ and $(\wt{f}_t)$ and, for each $t > 0$, we let $(\wt{\eta}')^t$ be the curve associated with $\wt{f}_t(\partial \h \cup \eta')$.  Note that $(\wt{\eta}')^t$ has the law of an $\SLE_{\kappa'}$ process for each $t \geq 0$.

We fix the additive constant for $\wt{h}^0$, hence $\wt{h}^t$ for all $t \geq 0$, by setting its average on $\h \cap \partial \D$ to be equal to $0$.

Fix $\epsilon > 0$.  We define stopping times inductively as follows.  We let $\wt{\tau}_1^\epsilon$ be the first time $t$ that the diameter of a partially zipped in bubble by the reverse Loewner flow at time $t$ is at least $\epsilon$ (i.e., the connected components of $\h \setminus \eta'$ are not counted).  Let~$\wt{\bubble}_1^\epsilon$ be the corresponding bubble.  Assuming that $\wt{\tau}_j^\epsilon$ has been defined for $1 \leq j \leq k$, some $k \in \N$, we let $\wt{\tau}_{k+1}^\epsilon$ be the first time $t$ after time $\wt{\tau}_k^\epsilon$ that the diameter of a partially zipped in bubble, distinct from those corresponding to $\wt{\tau}_1^\epsilon,\ldots,\wt{\tau}_k^\epsilon$, is at least $\epsilon$.  Let $\wt{\bubble}_{k+1}^\epsilon$ be the corresponding bubble.

For each $\wt{\bubble}$ which is not fully zipped in, we take $x_{\wt{\bubble}}$ to be the endpoint of the interval corresponding to the interior of $\partial \wt{\bubble} \cap \partial \h$ which is closest to $0$.  We let $y_{\wt{\bubble}_j^\epsilon}$ be a point independently picked in the interior of $\partial \wt{\bubble}_j^\epsilon \cap \partial \h$ with density as in Lemma~\ref{lem::weighted_quantum_surface_boundary} with respect to Lebesgue measure.  Let $G_j^\epsilon(u,v)$ be the Green's function on $\wt{\bubble}_j^\epsilon$ with Neumann boundary conditions on $\partial \wt{\bubble}_j^\epsilon \cap \partial \h$ and Dirichlet boundary conditions on $\partial \wt{\bubble}_j^\epsilon \setminus \partial \h$.  By Lemma~\ref{lem::weighted_quantum_surface_boundary}, the law of $(\wt{\bubble}_j^\epsilon, \wt{h}^{\wt{\tau}_j^\epsilon}+\tfrac{\gamma}{2} G_j^\epsilon(y_{\wt{\bubble}_j^\epsilon},\cdot))$ (viewed as a quantum surface) is equal to the law of $(\wt{\bubble}_j^\epsilon,\wt{h}^{\wt{\tau}_j^\epsilon})$ weighted by the $\gamma$-LQG boundary length of the interior of $\partial \wt{\bubble}_j^\epsilon \cap \partial \h$.

Let $G(u,v)$ be the Green's function on $\strip$ with Neumann boundary conditions on $\partial \strip \setminus (-\infty,0]$ and Dirichlet boundary conditions on $(-\infty,0]$ and let $G(u) = \lim_{v \to \infty} G(u,v)$; this limit exists by Lemma~\ref{lem::harmonic_strip_constant}.  For each $j \in \N$ and $\epsilon > 0$, we let
\begin{itemize}
\item $\wt{\varphi}_j^\epsilon \colon \wt{\bubble}_j^\epsilon \to \strip$ be the unique conformal transformation which takes $x_{\wt{\bubble}_j^\epsilon}$ to $-\infty$ and $y_{\wt{\bubble}_j^\epsilon}$ to $+\infty$ with the horizontal translation fixed so that $\partial \wt{\bubble}_j^\epsilon \cap \h$ is mapped to $(-\infty,0]$.
\item $\wh{h}_j^\epsilon = \wt{h}^{\wt{\tau}_j^\epsilon} \circ (\varphi_j^\epsilon)^{-1} + \tfrac{\gamma}{2} G(u)  + Q\log|( (\varphi_j^\epsilon)^{-1})'|$.
\item $\wh{X}^{j,\epsilon}$ be the projection of $\wh{h}_j^\epsilon$ onto $\CH_1(\strip)$ (with fixed additive constant).
\end{itemize}

By Lemma~\ref{lem::weighted_quantum_surface_boundary}, $(\strip,\wh{h}_j^\epsilon)$ viewed as a quantum surface has the law of $(\wt{\bubble}_j^\epsilon,\wt{h}^{\wt{\tau}_j^\epsilon})$ weighted by the $\gamma$-LQG boundary length of the interior of $\partial \wt{\bubble}_j^\epsilon \cap \partial \h$.

For each $j$ and $\wt{\epsilon} > 0$, let $\wh{u}_j^{\epsilon,\wt{\epsilon}} \in \R$ be such that the $\gamma$-LQG length measure associated with $\wh{h}_j^\epsilon$ assigns mass $\wt{\epsilon}$ to $(-\infty,\wh{u}_j^{\epsilon,\wt{\epsilon}}]$; we take $\wh{u}_j^{\epsilon,\wt{\epsilon}} = +\infty$ if the $\gamma$-LQG length of $\R$ is smaller than $\wt{\epsilon}$.  We let $\wh{\psi}_j^{\epsilon,\wt{\epsilon}}$ be the function which is harmonic on $[\wh{u}_j^{\epsilon,\wt{\epsilon}},\infty) \times [0,\pi]$ with Neumann boundary conditions on the horizontal part of the strip boundary (and at $+\infty$) and Dirichlet boundary conditions on $\wh{u}_j^{\epsilon,\wt{\epsilon}} + [0,i\pi]$ given by the values of $\wh{h}_j^\epsilon$.

We similarly let $\wh{\psi}_j^{\epsilon}$ be the function which is harmonic on $\strip_+$ with Neumann boundary conditions on $\partial \strip_+ \setminus [0,i\pi]$ (in particular, also at $+\infty$) and Dirichlet boundary conditions on $[0,i\pi]$ equal to those of $\wh{h}_j^\epsilon$ on $[0,i\pi]$.

We next record the following analog of Lemma~\ref{lem::harmonic_converges} for the present setting.
\begin{lemma}
\label{lem::psi_j_kp_constant}
Fix $j \in \N$ and $\epsilon > 0$.  There exists a sequence $(\wt{\epsilon}_k)$ of positive numbers with $\wt{\epsilon}_k \to 0$ as $k \to \infty$ such that $\wh{\psi}_j^{\epsilon,\wt{\epsilon}_k}$ a.s.\ converges to the $0$ function on $\strip_+$ as $k \to \infty$ with respect to the topology of local uniform convergence modulo a global additive constant.
\end{lemma}
\begin{proof}
This follows from an argument which is very similar to that given in Lemma~\ref{lem::harmonic_converges} (using Remark~\ref{rem::sle_kappa_prime_bubble_closing_point} in place of Proposition~\ref{prop::zoomed_in_picture}), so we will omit the details.
\end{proof}

Fix $\bsize > 0$ and $r \in \R$.  Let $\wh{\varsigma}_j^{\epsilon,r} = \inf\{ u \in \R : \wh{X}_u^{j,\epsilon} = r\}$.  For each $k$, let $j_k$ be the $k$th index $j$ such that both
\begin{itemize}
\item $\wh{\varsigma}_j^{\epsilon,r} \in [0,\infty)$ and
\item the total variation distance between the conditional law of the projection of $\wh{h}_j^{\epsilon}(\cdot+\wh{\varsigma}_j^{\epsilon,r})$ onto $\CH_2(\strip)$ restricted to $\strip_+$ given $\wh{\psi}_j^\epsilon(\cdot + \wh{\varsigma}_j^{\epsilon,r})$ and the corresponding projection of a free boundary GFF on $\strip$ is at most $\bsize$.
\end{itemize}
(We emphasize that Lemma~\ref{lem::psi_j_kp_constant} implies that there exists such times.)  We then let~$\wt{\bubble}_k^{\epsilon,\bsize,r} = \wt{\bubble}_{j_k}^\epsilon$.  We also let~$\wh{X}^{k,\epsilon,\bsize,r}$, $\wh{h}_k^{\epsilon,\bsize,r}$, and $\wh{\psi}_k^{\epsilon,\bsize,r}$ be respectively given by~$\wh{X}^{{j_k},\epsilon}$, $\wh{h}_{j_k}^\epsilon$, and $\wh{\psi}_{j_k}^\epsilon$ with the horizontal translation for each chosen so that $\wh{X}^{k,\epsilon,\bsize,r}$ first reaches $r$ at $u=0$.  We also let $\wh{w}_k^{\epsilon,r} = - \wh{\varsigma}_{j_k}^{\epsilon,r}$.  Then $\partial_k^{\epsilon,r} \strip = \partial \strip \setminus (-\infty,\wh{w}_k^{\epsilon,r}]$ gives the of image of the part of the boundary of the corresponding bubble which lies in $\partial \h$.

Note that Lemma~\ref{lem::psi_j_kp_constant} implies that for every given bubble generated by the reverse Loewner flow and $\bsize > 0$ there exists $\epsilon_0 >0$ and $r_0 \in \R$ such that for all $\epsilon \in (0,\epsilon_0)$ and $r \leq r_0$ we have that the bubble will be part of the sequence $(\wh{h}_k^{\epsilon,\bsize,r})$ (of course, the index $j$ associated with any fixed bubble is changing as $\epsilon \to 0$).

\subsubsection{Form of the drift}

We are now going to identify the law of the sequence $(\wh{X}^{k,\epsilon,\bsize,r})$.

\begin{proposition}
\label{prop::kappa_prime_large_bubble_form}

Let 
\[ a =  \gamma - Q = \frac{\gamma}{2}-\frac{2}{\gamma}.\]
Let $\wh{B}_{2t}^{j,\epsilon,\bsize,r} = \wh{X}_t^{j,\epsilon,\bsize,r} - at$.  Given $\wh{\psi}_j^{\epsilon,\bsize,r}$, $\wh{B}^{j,\epsilon,\bsize,r}$ is a standard Brownian motion with $\wh{B}_0^{j,\epsilon,\bsize,r} = r$.
\end{proposition}
\begin{proof}
The proof is analogous to that of Proposition~\ref{prop::line_average_law}.
\end{proof}

\subsubsection{Controlling the dependency}
\label{subsubsec::kp_dependency}

Throughout the rest of this subsection, we let $\delta = 4 - \tfrac{8}{\gamma^2} = 4 - \tfrac{\kappa'}{2}$ (as in the statement of Theorem~\ref{thm::kappa_prime_bubbles}).  For each $r \in \R$, we let $E^r$ be the set of distributions on $\strip$ whose projection onto $\CH_1(\strip)$ (with fixed additive constant) has supremum which is at least $r$.  Note that $\diskmeasure(E^r) \in (0,\infty)$ (recall Remark~\ref{rem::bessel_ito_excursion}).

\begin{lemma}
\label{lem::total_var_bound}
For each $k \in \N$, the total variation distance between the conditional law of $\wh{h}_k^{\epsilon,\bsize,r}$ given $\wh{\psi}_k^{\epsilon,\bsize,r}$ and the sample from $\diskmeasure$ conditional on $E^r$ (where we take the horizontal translation for samples from the latter law so that the projection onto $\CH_1(\strip)$ (with fixed additive constant) first hits~$r$ at $u=0$) and restricted to~$\strip_+$ is at most~$\bsize$.
\end{lemma}
\begin{proof}
By Proposition~\ref{prop::kappa_prime_large_bubble_form}, we know that the conditional law given $\wh{\psi}_k^{\epsilon,\bsize,r}$ of the projection of $\wh{h}_k^{\epsilon,\bsize,r}$ onto $\CH_1(\strip)$ (with fixed additive constant) is the same as the corresponding projection for a sample chosen from $\diskmeasure$ given $E^r$ (with the horizontal translation chosen as in the statement).  Therefore it is just a matter of showing that the conditional law given $\wh{\psi}_k^{\epsilon,\bsize,r}$ of the projection of $\wh{h}_k^{\epsilon,\bsize,r}$ onto $\CH_2(\strip)$ and restricted to $\strip_+$ has total variation distance at most $\bsize$ from the law of the corresponding projection of sample chosen from $\diskmeasure$ conditional $E^r$.  This, in turn, follows from the definition of $\wh{h}_k^{\epsilon,\bsize,r}$ since the latter is given by the law of the projection of a free boundary GFF on $\strip$ onto $\CH_2(\strip)$ and then restricted to~$\strip_+$.
\end{proof}

For each $c > 0$, we let $E_c^r$ be the set of elements in $E^r$ whose associated $\gamma$-LQG boundary measure assigns mass at least $c$ to $\partial \strip \cap \partial \strip_+$, mass at most $c^{-1} e^r$ to $\partial \strip \setminus \partial \strip_+$, and whose associated $\gamma$-LQG area measure assigns mass at least~$c$ to~$\strip_+$.  In what follows, when we refer to the $\gamma$-LQG boundary length of $\partial \strip_+$ (either in words or using the notation $\nu_h(\partial \strip_+)$) we mean the $\gamma$-LQG boundary length of $\partial \strip \cap \partial \strip_+$.  In particular, we do not want to count the length of $[0,i\pi]$.  Let~$\mu_k^{\epsilon,\bsize,r}$ denote the conditional law of~$\wh{h}_k^{\epsilon,\bsize,r}$ given~$\wh{\psi}_k^{\epsilon,\bsize,r}$.

\begin{lemma}
\label{lem::total_var_bound2}
The total variation distance between $(\nu_h(\partial \strip_+))^{-1}d\diskmeasure(h)$ conditional on $E_c^r$ (with the horizontal translation chosen as in Lemma~\ref{lem::total_var_bound} and restricted to $\strip_+$) and $(\nu_h(\partial \strip_+))^{-1} d\mu_k^{\epsilon,\bsize,r}(h)$ conditional on $E_c^r$ and given $\wh{\psi}_k^{\epsilon,\bsize,r}$ (and restricted to $\strip_+$) is $O(\bsize)$ where the implicit constants depend on $r$ and $c$ but not $\bsize$ or $\epsilon$.
\end{lemma}
\begin{proof}
We begin with three observations.  First, suppose that $\mu_1,\mu_2$ are measures on some space $\CX$ with $\mu_2$ absolutely continuous with respect to $\mu_1$.  Suppose that $f$ is a positive function on $\CX$ and, for $i=1,2$, let $\nu_i$ be the measure whose Radon-Nikodym derivative with respect to $\mu_i$ is $f$.  Using that $d\nu_2/d\nu_1 = d\mu_2/d\mu_1$, we have that
\begin{align*}
   \| \nu_1 - \nu_2 \|_{TV}
&= \int \left| 1 - \frac{d\nu_2}{d\nu_1} \right| d\nu_1
 \leq \left(\sup_{x \in \CX} f(x)\right) \int \left| 1 - \frac{d\mu_2}{d\mu_1}\right| d\mu_1\\
&=  \left(\sup_{x \in \CX} f(x)\right) \| \mu_1 - \mu_2 \|_{TV}.
\end{align*}
For $i=1,2$, let $N_i = \nu_i(\CX)$ be the total amount of $\nu_i$ mass, assume that $N_i \in (0,\infty)$, and let $\ol{\nu}_i = \nu_i/N_i$.  Second, we have that:
\begin{align*}
       \| \ol{\nu}_1 - \ol{\nu}_2 \|_{TV}
&\leq \frac{1}{N_2}\left|N_1-N_2\right| + \frac{1}{N_2} \|\nu_1 - \nu_2\|_{TV}.
\end{align*}
Third, we have that
\begin{align*}
   |N_1 - N_2|
&= \left| \int f(x) d\mu_1(x) - \int f(x) d\mu_2(x) \right|
 \leq \left(\sup_{x \in \CX} f(x) \right) \| \mu_1 - \mu_2 \|_{TV}.
\end{align*}
Combining all three observations, we have
\[ \| \ol{\nu}_1 - \ol{\nu}_2 \|_{TV} \leq \frac{2}{N_2}\left(\sup_{x \in \CX} f(x)\right) \| \mu_1 - \mu_2\|_{TV}.\]
Combining this with Lemma~\ref{lem::total_var_bound} implies the result because the probability that each assigns to $E_c^r$ is uniformly bounded from below as $\bsize \to 0$ (in this context, $\sup_{x \in \CX} f(x) \leq c^{-1}$).
\end{proof}

\begin{lemma}
\label{lem::total_var_bound3}
The total variation distance between $(\nu_h(\partial \strip))^{-1}d\diskmeasure(h)$ conditional on $E_c^r$ (with the horizontal translation chosen as in Lemma~\ref{lem::total_var_bound2} and restricted to $\strip_+$) and $(\nu_h(\partial_k^{\epsilon,r} \strip))^{-1} d\mu_k^{\epsilon,\bsize,r}(h)$ conditional on $E_c^r$ and given $\wh{\psi}_k^{\epsilon,\bsize,r}$ (and restricted to $\strip_+$) is $O(\bsize) + O(c^{-2} e^r)$ where the implicit constants in the first summand depend on $r$ and $c$ but not $\bsize$ or $\epsilon$ and the implicit constants in the second summand are uniform.
\end{lemma}
\begin{proof}
Lemma~\ref{lem::total_var_bound2} gives us that the total variation distance between $(\nu_h(\partial \strip_+))^{-1} d\diskmeasure(h)$ conditional on $E_c^r$ and $(\nu_h(\partial \strip_+))^{-1} d\mu_k^{\epsilon,\bsize,r}(h)$ conditional on $E_c^r$ and given $\wh{\psi}_k^{\epsilon,\bsize,r}$ is $O(\bsize)$.  Thus to prove the result we just have to bound the total variation distance between
\begin{itemize}
\item $(\nu_h(\partial \strip))^{-1} d\diskmeasure(h)$ and $(\nu_h(\partial \strip_+))^{-1} d\diskmeasure(h)$ conditional on $E_c^r$ and
\item $(\nu_h(\partial_\epsilon^{k,r}  \strip))^{-1} d\mu_k^{\epsilon,\bsize,r}(h)$ and $(\nu_h(\partial \strip_+))^{-1} d\mu_k^{\epsilon,\bsize,r}(h)$ conditional on $E_c^r$ and given $\wh{\psi}_k^{\epsilon,\bsize,r}$.
\end{itemize}
We have for $h \in E_c^r$ that
\begin{align*}
 1 \leq \frac{\nu_h(\partial \strip)}{\nu_h(\partial \strip_+)} = 1 + \frac{\nu_h(\partial \strip \setminus \partial \strip_+)}{\nu_h(\partial \strip_+)} = 1 + O( c^{-2} e^r).
\end{align*}
Therefore it is easy to see in the first case that the Radon-Nikodym derivative between the two measures is equal to $1+ O(c^{-2} e^r)$ with uniform constants.  It is similarly easy to see that this holds in the second case, which proves the result.
\end{proof}

\subsubsection{Proof of Theorem~\ref{thm::kappa_prime_bubbles}}

Fix $r \in \R$, $c > 0$, and let $(H_k^{r,c})$ be an i.i.d.\ sequence chosen from $(\nu_h(\partial \strip))^{-1} d \diskmeasure(h)$ conditional on $E_c^r$ with the horizontal translation chosen so that the projection of $H_k^{r,c}$ onto $\CH_1(\strip)$ (with fixed additive constant) first hits $r$ at $u=0$.  Fix $\bsize > 0$.  We also let $(\wh{h}_k^{\epsilon,\bsize,r})$ be a sequence chosen from $\prod_{k=1}^\infty (\nu_{h_k}(\partial_k^{\epsilon,r} \strip))^{-1} d \mu_k^{\epsilon,\bsize,r}(h_k)$.  Note that the $(\wh{h}_k^{\epsilon,\bsize,r})$ have the same law as the subsequence of bubbles which are zipped in (viewed as quantum surfaces) by the reverse Loewner flow such that a certain event occurs.  Let $(\acute{h}_k^{\epsilon,\bsize,r,c})$ be the subsequence of $(\wh{h}_k^{\epsilon,\bsize,r})$ for which $E_c^r$ occurs.  Fix $n \in \N$.  Lemma~\ref{lem::total_var_bound3} implies that the total variation distance between the law of the first $n$ of the $H_k^{r,c}$ and the law of the first $n$ of the $\acute{h}_k^{\epsilon,\bsize,r,c}$ (both restricted to $\strip_+$) is $O(\bsize) \cdot n + O(c^{-2} e^r) \cdot n$ where the $O(\bsize)$ term tends to $0$ as $\bsize \to 0$ with $r$ and $c$ fixed and the constants in the $O(c^{-2} e^r)$ term are uniform.

It is easy to see that the quantum surfaces $(\strip,\acute{h}_k^{\epsilon,\bsize,r,c})$ converge to a limiting family of quantum surfaces $(\strip,\acute{h}_k^c)$ when we take a limit first as $\epsilon \to 0$, then $r \to -\infty$, and then as $\bsize \to 0$.  (One can take the horizontal translation so that the mass assigned to $\strip_+$ is exactly equal to $c/2$.)  Indeed, for otherwise we would get a contradiction to the statement that the amount of $\gamma$-LQG mass in each bounded neighborhood of zero is a.s.\ finite --- recall that the $\gamma$-LQG mass of each bubble is at least $c$.  Moreover, we have that the $(\acute{h}_k^c)$ are i.i.d.\ sampled from $(\nu_h(\partial \strip))^{-1} d \diskmeasure(h)$ conditioned on having both $\gamma$-LQG boundary length and $\gamma$-LQG area at least $c$.  Sending $c \to 0$ proves the result.\qed

\subsubsection{Alternative formulation}

We are now going to give an alternative formulation of the Poisson law described in Theorem~\ref{thm::kappa_prime_bubbles} in terms of the boundary length quantum disk of Definition~\ref{def::finite_volume_surfaces}.

\begin{proposition}
\label{prop::unit_area_bl_construction}
There exists a constant $c > 0$ such that the law on quantum surfaces described in the statement of Theorem~\ref{thm::kappa_prime_bubbles} admits the following equivalent method of sampling.  Sample a \ppp\ $\Lambda$ on $\R_+ \times \R_+$ with intensity measure $c du \otimes t^{-\kappa'/4-1} dt$ where $du$ and $dt$ denote Lebesgue measure on $\R_+$.  Then sample an i.i.d.\ collection of unit boundary length quantum disks $(D_{u,t} : (u,t) \in \Lambda)$ indexed by the elements of~$\Lambda$.  Finally, take the process $((u,\wt{D}_{u,t}) : (u,t) \in \Lambda)$ where $\wt{D}_{u,t}$ for $(u,t) \in \Lambda$ is given by taking $D_{u,t}$ and then scaling so that its boundary length is equal to~$t$.
\end{proposition}
\begin{proof}
Let $\delta = 4-\tfrac{8}{\gamma^2} = 4 - \tfrac{\kappa'}{2} \in (0,2)$ as in the statement of Theorem~\ref{thm::kappa_prime_bubbles}.  For each $t > 0$, let $\CV_{\delta,1}^t$ denote the law on functions on $\strip$ which is obtained by sampling a~$\bes^\delta$ excursion~$e$ from~$0$ given that its supremum is equal to~$t$ (a sample can be produced from this law by joining back to back two $\bes^{4-\delta}$ processes starting from~$0$ and stopped upon hitting $t$ as explained in Remark~\ref{rem::bessel_ito_excursion}), then taking $4\gamma^{-1}\log(e)$, and then finally reparameterizing so that it has constant quadratic variation $2ds$.  We can fix the horizontal translation so that the supremum is reached at $s=0$.  We also let $\CV_2$ be the law of the projection of a free boundary GFF on $\strip$ onto $\CH_2(\strip)$.  Note that the infinite measure $(\nu_h(\partial \strip))^{-1} d \diskmeasure(h)$ used to describe the Poisson law in the statement of Theorem~\ref{thm::kappa_prime_bubbles} is equivalent to:
\[ \frac{1}{\nu_h(\partial \strip)} d\CV_{\delta,1}^t(h_1) d\CV_2(h_2) d\nu_\delta^*(t) = \frac{t^2}{\nu_h(\partial \strip)} d\CV_{\delta,1}^t(h_1) d\CV_2(h_2) d\nu_{\delta-2}^*(t)\]
where $h=h_1+h_2$ and $\nu_\delta^*$ is as in Remark~\ref{rem::bessel_ito_excursion}.  Let
\[ d\CV^t(h) = \frac{t^2}{\nu_h(\partial \strip)} d\CV_{\delta,1}^t(h_1) d\CV_2(h_2).\]
Observe that a sample from $\CV^t$ can be produced by picking $h$ from $\CV^1$ and then adding $4\gamma^{-1}\log(t)$ to $h$.

In view of Remark~\ref{rem::bessel_ito_excursion}, we can describe $\CV^1$ explicitly as follows.  We let $\wt{\CV}^1$ be the law on distributions on~$\strip$ which can be sampled from by:
\begin{enumerate}
\item Taking its projection $X$ onto $\CH_1(\strip)$ to be given by $X_u = B_{2u} + (\gamma-Q)u$ for $u \geq 0$ where $B$ is a standard Brownian motion and to be given by $X_u = \wh{B}_{-2u} - (\gamma-Q)u$ for $u < 0$ where $\wh{B}$ is a standard Brownian motion with $B_0 = \wh{B}_0 = 0$ conditioned so that $X_u \leq 0$ for all $u$.  We assume that the additive constant for the projection is fixed to agree exactly with this process.
\item Independently sampling its projection onto $\CH_2(\strip)$ using the corresponding projection of a free boundary GFF on $\strip$.  We assume that the additive constant is fixed so that the average on $[0, \pi i]$ vanishes.
\end{enumerate}
Then $d\CV^1(h) = (\nu_h(\partial \strip))^{-1} d\wt{\CV}^1(h)$.

It thus follows from the discussion so far that if we let $\wt{\Lambda}$ be a \ppp\ with intensity measure $du \otimes d\CV^1(h) \otimes \nu_{\delta-2}^*$
then $\Lambda = \{(u,h+4\gamma^{-1}\log(t)) : (u,h,t) \in \wt{\Lambda}\}$ is a \ppp\ with the same intensity measure as in Theorem~\ref{thm::kappa_prime_bubbles}.

We want to argue now that $\nu_h(\partial \strip)$ with $h$ sampled from $\CV^1$ has a finite $\tfrac{\kappa'}{4}$-moment.  Note that this moment is equal to $\int (\nu_h(\partial \strip))^{\kappa'/4-1} d\wt{\CV}^1(h)$.  Since $\kappa'/4-1 \in (0,1)$ as $\kappa' \in (4,8)$, it suffices to show that $\int \nu_h(\partial \strip) d\wt{\CV}^1(h) < \infty$.  This, in turn, follows from Lemma~\ref{lem::negative_drift_first_moment}.

It thus follows from Lemma~\ref{lem::poisson_reweight} (as in the proof of Proposition~\ref{prop::actual_quantum_area}), that there exists a constant $c > 0$ such that $\{(u,h+4 \gamma^{-1}\log(t),t^2 \nu_h(\partial \strip)) : (u,h,t) \in \wt{\Lambda}\}$ is a \ppp\ with intensity measure given by $c du \otimes \big( t^{-\kappa'/4-1} d\CU^t dt\big)$ where $\CU^t$ is the law of the unit boundary length quantum disk scaled to have boundary length equal to $t$.  The result thus follows because the $\gamma$-LQG boundary length of $\partial \strip$ associated with the field $h+4\gamma^{-1} \log(t)$ is equal to $t^2 \nu_h(\partial \strip)$.
\end{proof}

\subsection{Zipping according to quantum natural time}
\label{subsec::quantum_typical_zip_unzip}

Suppose that~$(h,\eta)$ is as in the statement of Theorem~\ref{thm::skinny_wedge_bubble_structure}, let~$(W,V)$ be the driving pair for~$\eta$, let~$(f_t)$ be the centered (forward) Loewner flow associated with~$\eta$, let~$\ell_t$ be the local time at~$0$ of the Bessel process~$\kappa^{-1/2}(V-W)$, and let~$T_u$ be the right continuous inverse of~$\ell_t$.  In this setting, in Theorem~\ref{thm::skinny_wedge_bubble_structure} we showed that the beaded surface which consists of the bounded components of~$\h \setminus \eta([0,T_u])$ which are to the right of~$\eta$ can be described in terms of a certain Bessel process~$Y$ (i.e., as in Definition~\ref{def::skinny_wedge_bessel}), where the additive constant is fixed as in the statement of Theorem~\ref{thm::skinny_wedge_bubble_structure}.  That is, we have that the average of $h \circ f_{T_u}^{-1} + Q \log|(f_{T_u}^{-1})'|$ on $\partial \D$ is equal to $0$.  We will now use this to prove the following theorem as well as Theorem~\ref{thm::sle_kp_quantum_local_typical} stated below, which are the main results of this section and are versions of Theorem~\ref{thm::reverse_coupling} in which one has the invariance with respect to shifts associated with the local time associated with $Y$.

\begin{theorem}
\label{thm::bubbles_quantum_local_typical}
Fix $\kappa \in (0,4)$, $\rho \in (-2,\tfrac{\kappa}{2}-2)$, and suppose that $(\h,h,0,\infty)$ is a quantum wedge of weight $\rho+4$.  Let~$\eta$ be an $\SLE_\kappa(\rho)$ process in $\h$ from~$0$ to~$\infty$ with a single boundary force point of weight $\rho$ located at~$0^+$ and assume that~$\eta$ is independent of~$h$.  Then the following hold:
\begin{enumerate}[(i)]
\item\label{it::quantum_typical_bubble_form} The law of the beaded surface consisting of the components of $\h \setminus \eta$ which are to the right of $\eta$ is that of a quantum wedge of weight $\rho+2$.
\item\label{it::quantum_typical_stationary} Let $\qnt_u$ be the first capacity time $t$ for $\eta$ that the local time at $0$ of the Bessel process $Y$ which encodes the weight $\rho+2$ wedge as in Part~\eqref{it::quantum_typical_bubble_form} is equal to $u$.  Also let $(f_t)$ be as described above.  For each $u > 0$, we have (as path-decorated quantum surfaces) that
\begin{equation}
\label{eqn::quantum_local_time_invariance}
(h,\eta) \stackrel{d}{=} \left( h \circ f_{\qnt_u}^{-1} + Q\log|(f_{\qnt_u}^{-1})'|, f_{\qnt_u}(\eta) \right).
\end{equation}
\end{enumerate}
Part~\eqref{it::quantum_typical_bubble_form} also holds if we take $\eta$ to be a whole-plane $\SLE_\kappa(\rho)$ process from $0$ to $\infty$ and replace $(\h,h)$ with a quantum cone $(\C,h)$ of weight $\rho+2$.
\end{theorem}

\begin{definition}
\label{def::qnt_sle_k}
We call $\qnt_u$ the {\bf quantum natural time} parameterization of the bubbles which are cut off by $\eta$ from $\infty$.
\end{definition}

\begin{figure}[ht!]
\begin{center}
\includegraphics[scale=0.85]{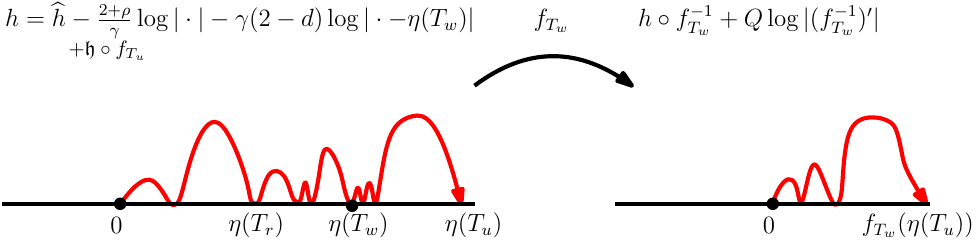}	
\end{center}
\caption{\label{fig::bubbles_typical_illustration} Illustration of the setup for the proof of Theorem~\ref{thm::bubbles_quantum_local_typical}.  It is shown in Lemma~\ref{lem::field_weighted_by_quantum_local_time} that if we weight the law of $h$ as in Theorem~\ref{thm::skinny_wedge_bubble_structure} by the amount of local time that the Bessel process which encodes the weight $\rho+2$ quantum wedge structure accumulates between when the Bessel process which drives the $\SLE_\kappa(\rho)$ accumulates $r$ and $u$ units of local time (capacity times $T_r$ and $T_u$), then the result is a free boundary GFF with an extra $\log$ singularity which is located at $\eta(T_w)$ where $w \in [r,u]$ plus a certain harmonic function.  If one cuts along $\eta$ until the time $T_w$, the $\log$ singularity at $\eta(T_w)$ is moved to the origin and one obtains a free boundary GFF plus $(\rho+2-\gamma^2)/\gamma \log|\cdot|$ (and some harmonic function) decorated by the $\SLE_\kappa(\rho)$ process $f_{T_w}(\eta)$.  If one then zooms in at the origin as in part~\eqref{it::quantum_wedge_limit} of Proposition~\ref{prop::quantum_wedge_properties}, one thus obtains a weight $\rho+4$ quantum wedge.  On the other hand, as the local behavior of the bubbles which are to the right of $\eta(T_w)$ (ordered from left to right) is described by a weight $\rho+2$ quantum wedge (by time-reversing the Poisson point process at a typical time), it follows that the bubbles to the right of $f_{T_w}(\eta)$ after zooming in is described by a weight $\rho+2$ quantum wedge.}
\end{figure}

See Figure~\ref{fig::bubbles_typical_illustration} for an illustration of the setup for the proof of Theorem~\ref{thm::bubbles_quantum_local_typical}.  The proof of Theorem~\ref{thm::bubbles_quantum_local_typical} will follow the strategy used in \cite{she2010zipper} to deduce \cite[Theorem~1.8]{she2010zipper} (the invariance of a weight-$4$ quantum wedge under the operation of zipping/unzipping according to quantum length) from \cite[Theorem~1.2]{she2010zipper} (the invariance of the free boundary GFF plus an appropriate $\log$ singularity under zipping/unzipping according to capacity time).  In particular, Theorem~\ref{thm::skinny_wedge_bubble_structure} gives us a measure (which is defined from the weight $\rho+2$ quantum wedge structure) on the intersection points of an $\SLE_\kappa(\rho)$ curve with the domain boundary.  The main step is to determine the local behavior of the field near a point chosen from this measure, and this will be accomplished in Lemma~\ref{lem::quantum_local_time_change_field} and Lemma~\ref{lem::field_weighted_by_quantum_local_time} below.  We note that the statement that we will obtain to this effect is analogous to the fact that the local behavior of a point chosen from the quantum boundary measure is described by a $\gamma$-quantum wedge, though here we will obtain a different type of quantum wedge and the proof will also proceed along somewhat different lines.

\subsubsection{Local behavior near a typical intersection point}

Fix $d \in (0,2)$ and recall the description of the It\^o excursion measure $\nu_d^\bes$ for the excursions made by a $\bes^d$ process from $0$ given in Remark~\ref{rem::bessel_ito_excursion}.  As in Section~\ref{sec::preliminaries}, for each $\epsilon > 0$, we let $\CE(\epsilon)$ be the set of excursions with length at least~$\epsilon$.  Then~\eqref{eqn::bessel_ito_lifetime} leads to the value of $\nu_d^\bes(\CE(\epsilon))$.  We record this in the following lemma.

\begin{lemma}
\label{lem::bessel_ito_excursion_measure}
Fix $d \in (0,2)$ and let $\nu_d^\bes$ be the It\^o excursion measure associated with a $\bes^d$ process.  There exists a constant $C_d > 0$ depending only on $d$ such that $\nu_d^\bes(\CE(\epsilon)) = C_d \epsilon^{d/2-1}$.
\end{lemma}

Using Lemma~\ref{lem::bessel_ito_excursion_measure}, we can now describe how the local time $\qlt$ of $Y$ at $0$ changes when we perturb the field by adding a smooth function.

\begin{lemma}
\label{lem::quantum_local_time_change_field}
Assume that we have the same setup as in Theorem~\ref{thm::skinny_wedge_bubble_structure} (with the additive constant for $h$ fixed in the same way) and let $\qlt$ be the local time at $0$ associated with the Bessel process $Y$ which encodes the weight $\rho+2$ wedge corresponding to the bounded components of $\h \setminus \eta([0,T_u])$.  For each smooth function $\phi$, let~$Y^\phi$ be the corresponding process with $h$ replaced by~$h+\phi$ (which we assume to be parameterized according to its quadratic variation).  Let~$\ol{\nu}$ (resp.\ $\ol{\nu}^\phi$) denote the empirical measure of the excursions made by~$Y$ (resp.\ $Y^\phi$) from~$0$ which correspond to the bubbles cut off $\infty$ by $\eta([T_r,T_u])$.  Let $d$ be the Bessel dimension of $Y$ as given in~\eqref{eqn::bessel_wedge_correspondence}.  We a.s.\ have (off a common set of measure $0$) for all $r \in [0,u]$ and all such smooth functions $\phi$ that the limit $m^\phi([r,u])$ of $\ol{\nu}^\phi(\CE(\epsilon)) / \nu_d^\bes(\CE(\epsilon))$ as $\epsilon \to 0$ exists and (with $m = m^0$)
\begin{equation}
\label{eqn::quantum_local_time_change_field}
m^\phi([r,u]) = \int_r^u \exp\left( \frac{\gamma(2-d)}{2} \phi(\eta(T_v)) \right) dm(v).
\end{equation}
\end{lemma}

Before we give the proof of Lemma~\ref{lem::quantum_local_time_change_field}, we first make the following remark.

\begin{remark}
\label{rem::h_phi_abs_cts}
\begin{enumerate}[(i)]
\item If $\phi$ is a deterministic smooth function with $\| \phi \|_\nabla < \infty$, then the law of~$h+\phi$ is mutually absolutely continuous with respect to the law of~$h$.  Therefore it follows that the law of~$Y^\phi$ is mutually absolutely continuous with respect to the law of~$Y$.  In particular, the local time~$\qlt^\phi$ of $Y^\phi$ at $0$ exists (and can be defined, for example, as in \cite[Corollary~19.6]{KAL_FOUND}).  Lemma~\ref{lem::quantum_local_time_change_field} in fact implies that this local time can be defined for all smooth $\phi$ simultaneously off a common set of measure $0$.
\item It is natural in~\eqref{eqn::quantum_local_time_change_field} to see~$\tfrac{1}{2}\gamma(2-d)$ rather than~$\gamma(2-d)$ in the exponential because we are dealing with a \emph{boundary} (rather than bulk) measure supported on~$\eta \cap \partial \h$.  We note that in the limit $\rho \downarrow -2$, we have that $d \downarrow 1$ so that the constant term in the exponent of~\eqref{eqn::quantum_local_time_change_field} converges to $\tfrac{\gamma}{2}$, i.e., the exponent which appears in the definition of the boundary measure.
\end{enumerate}
\end{remark}

\begin{proof}[Proof of Lemma~\ref{lem::quantum_local_time_change_field}]
If~$\phi$ is a constant function, then~$Y^\phi$ is given by transforming~$Y$ by scaling spatially by~$\exp(\tfrac{\gamma}{2} \phi)$ and then speeding up time by the factor~$\exp(\gamma \phi)$ (recall that~$Y^\phi$ is parameterized by quadratic variation so that~$d\langle Y^\phi \rangle_t = dt$).  This means that an excursion of length~$l$ made by~$Y$ will correspond to an excursion of length~$e^{-\gamma \phi} l$ made by~$Y^\phi$.  Fix~$\epsilon > 0$.  We then have that
\begin{equation}
\label{eqn::empirical_excursion_add_constant_function}
\ol{\nu}^\phi(\CE(\epsilon)) = \ol{\nu}(\CE(e^{-\gamma \phi} \epsilon)).
\end{equation}
We also let~$\nu_d^\bes$ denote the It\^o excursion measure associated with a~$\bes^d$.  Using Lemma~\ref{lem::bessel_ito_excursion_measure} in the final equality, we then have that
\begin{align*}
   \frac{\ol{\nu}^\phi(\CE(\epsilon))}{\nu_d^\bes(\CE(\epsilon))}
&= \frac{\ol{\nu}(\CE(e^{-\gamma \phi} \epsilon))}{\nu_d^\bes(\CE(\epsilon))}
 = \frac{\nu_d^\bes(\CE(e^{-\gamma \phi}\epsilon))}{\nu_d^\bes(\CE(\epsilon))} \times \frac{\ol{\nu}(\CE(e^{-\gamma \phi} \epsilon))}{\nu_d^\bes(\CE(e^{-\gamma \phi}\epsilon))}\\
&= \exp\left(\frac{\gamma (2-d)}{2} \phi \right) \frac{\ol{\nu}(\CE(e^{-\gamma \phi} \epsilon))}{\nu_d^\bes(\CE(e^{-\gamma \phi}\epsilon))}.
\end{align*}
By \cite[Proposition~19.12]{KAL_FOUND}, as~$\epsilon \to 0$, the right side converges to~$\exp(\tfrac{\gamma(2-d)}{2}\phi) m([r,u])$ and therefore the left side also converges.  This proves~\eqref{eqn::quantum_local_time_change_field} if~$\phi$ is a constant function.

We are now going to generalize to the case that~$\phi$ is a smooth function.  For each~$q < s$, we let~$\ol{\nu}_{q,s}$ (resp.\ $\ol{\nu}_{q,s}^\phi$) be the empirical measure for the excursions from~$0$ made by~$Y$ (resp.\ $Y^\phi$) which correspond to the bubbles cut off by $\eta([T_q,T_s])$.  We let~$b_{q,s}$ (resp.\ $B_{q,s}$) be the infimum (resp.\ supremum) of the values that~$\phi$ takes on in the interval~$[\eta(T_q),\eta(T_s)]$.   Arguing as in~\eqref{eqn::empirical_excursion_add_constant_function}, we have that
\begin{equation}
\label{eqn::empircal_excursion_upper_lower}
\ol{\nu}_{q,s}(\CE(e^{-\gamma b_{q,s}} \epsilon)) \leq \ol{\nu}_{q,s}^\phi(\CE(\epsilon)) \leq \ol{\nu}_{q,s}(\CE(e^{-\gamma B_{q,s}} \epsilon)).
\end{equation}
Pick a partition~$r \leq q_1 < \cdots < q_k \leq u$.  Applying~\eqref{eqn::empircal_excursion_upper_lower} in the inequality, we have that
\begin{align}
   \frac{\ol{\nu}^\phi(\CE(\epsilon))}{\nu_d^\bes(\CE(\epsilon))}
&= \sum_{j=1}^k \frac{\ol{\nu}_{q_{j-1},q_j}^\phi(\CE(\epsilon))}{\nu_d^\bes(\CE(\epsilon))}\notag\\
&\leq \sum_{j=1}^k \frac{\nu_d^\bes(\CE(e^{-\gamma B_{q_{j-1},q_j}}\epsilon))}{\nu_d^\bes(\CE(\epsilon))} \times \frac{\ol{\nu}_{q_{j-1},q_j}(\CE(e^{-\gamma B_{r_{j-1},r_j}}\epsilon))}{\nu_d^\bes(\CE(e^{-\gamma B_{q_{j-1},q_j}}\epsilon))} \notag\\
&=    \sum_{j=1}^k \exp\left(\frac{\gamma(2-d)}{2} B_{q_{j-1},q_j}\right) \frac{\ol{\nu}_{q_{j-1},q_j}(\CE(e^{-\gamma B_{q_{j-1},q_j}}\epsilon))}{\nu_d^\bes(\CE(e^{-\gamma B_{q_{j-1},q_j}}\epsilon))} \label{eqn::local_time_big_b}.
\end{align}
Arguing as in the case that $\phi$ is a constant function, the existence of the limit as $\epsilon \to 0$ of the expression in~\eqref{eqn::local_time_big_b} follows from \cite[Proposition~19.12]{KAL_FOUND}.  (We emphasize that this holds for all smooth functions $\phi$ off a common set of measure $0$.)  Combining, we have that
\begin{align*}
   \limsup_{\epsilon \to 0} \frac{\ol{\nu}^\phi(\CE(\epsilon))}{\nu_d^\bes(\CE(\epsilon))}
\leq \sum_{j=1}^k \exp\left(\frac{\gamma(2-d)}{2} B_{q_{j-1},q_j}\right)  m((q_{j-1},q_j]).
\end{align*}
Since $T_q$ is right continuous and $\phi$ is smooth, we note that $q \mapsto \phi(\eta(T_q))$ is also right continuous.  Thus taking a limit as the mesh size of the partition tends to zero, we get that
\begin{align*}
   \limsup_{\epsilon \to 0} \frac{\ol{\nu}^\phi(\CE(\epsilon))}{\nu_d^\bes(\CE(\epsilon))}
\leq  \int_r^u \exp\left(\frac{\gamma(2-d)}{2} \phi(\eta(T_v))\right) dm(v).
\end{align*}
Arguing in the same manner (using $b_{q,r}$ in place of $B_{q,r}$) implies that
\begin{align*}
   \liminf_{\epsilon \to 0} \frac{\ol{\nu}^\phi(\CE(\epsilon))}{\nu_d^\bes(\CE(\epsilon))}
\geq  \int_r^u \exp\left(\frac{\gamma(2-d)}{2} \phi(\eta(T_v))\right) dm(v),
\end{align*}
which gives the existence of the limit as $\epsilon \to 0$ of $\ol{\nu}^\phi(\CE(\epsilon))/\nu_d^\bes(\CE(\epsilon))$ and establishes~\eqref{eqn::quantum_local_time_change_field}.
\end{proof}

\begin{lemma}
\label{lem::field_weighted_by_quantum_local_time}
Assume that we have the same setup as in Theorem~\ref{thm::skinny_wedge_bubble_structure} (with the additive constant for $h$ fixed in the same way) and recall that
\begin{equation}
\label{eqn::field_weighted_by_quantum_local_time_h}
h = \wh{h} -\frac{2+\rho}{\gamma}\log|\cdot|
\end{equation}
where $\wh{h}$ is a free boundary GFF on $\h$.  For each $0 \leq q < r \leq u$, we let $m_{q,r}$ be the restriction of $m$ as in Lemma~\ref{lem::quantum_local_time_change_field} to subsets of $[q,r]$.  Consider the law on $(w,h,\eta)$ triples given by $\CZ_{q,r}^{-1} d m_{q,r} dh d\eta$ where $dh$ denotes the law as in~\eqref{eqn::field_weighted_by_quantum_local_time_h} and $\CZ_{q,r}^{-1}$ is a normalization constant.  (Note that $m_{q,r}$ depends on $h$ and $\eta$.)
\begin{enumerate}[(i)]
\item \label{it::ql_field_cond_law} Given $w$ and $\eta$, the conditional law of $h$ is equal to the law of 
\[ \wh{h} - \frac{2+\rho}{\gamma} \log|\cdot| - \gamma(2-d)\log|\cdot-\eta(T_w)| + \Fh \circ f_{T_u} \]
where $\wh{h}$ is a free boundary GFF on $\h$, $d$ is the dimension of the Bessel process $Y$ as given in~\eqref{eqn::bessel_wedge_correspondence}, $\Fh$ is a function which harmonic outside of $\h \cap \partial \D$, and the additive constant is fixed in the same manner as for $h$.
\item \label{it::ql_path_cond_law} Given $w$ and $\eta|_{[0,T_w]}$, the conditional law of $\eta|_{[T_w,\infty)}$ is that of an $\SLE_\kappa(\rho)$ process in the unbounded component of $\h \setminus \eta([0,T_w])$ from $\eta(T_w)$ to $\infty$ with a single boundary force point of weight $\rho$ located at $(\eta(T_w))^+$ weighted by the Radon-Nikodym derivative
\[ \CZ^{-1}\exp\left( \frac{\gamma^2(2-d)^2}{8} \left( \iint G(f_{T_u}^{-1}(x),f_{T_u}^{-1}(y)) dx dy - 2 \int G(\eta(T_w), f_{T_u}^{-1}(x)) dx \right) \right)\] where $G$ denotes the Neumann Green's function on $\h$ and $\CZ^{-1}$ is a normalizing constant and the integrals are all over $\partial \D \cap \partial \h$.

\item \label{it::ql_bubbles_independent} If one zooms in near $\eta(T_w)$ as in part~\eqref{it::quantum_wedge_limit} of Proposition~\ref{prop::quantum_wedge_properties}, then the law of the beaded surface which consists of the components of $\h \setminus \eta$ which are to the right of $\eta|_{[T_w,\infty)}$ converges to that of a quantum wedge of weight $\rho+2$.
\end{enumerate}
\end{lemma}

As we mentioned earlier, Lemma~\ref{lem::field_weighted_by_quantum_local_time} serves to describe the local behavior of the field near a ``quantum typical'' intersection point of an $\SLE_\kappa(\rho)$ process with the domain boundary.  The idea of the proof will be to use Lemma~\ref{lem::quantum_local_time_change_field} to relate weighting the law of the field/path pair by the ``quantum mass'' of the intersection of the $\SLE$ with the boundary to introducing a Radon-Nikodym derivative which in turn will correspond to a shift in the mean of the field.

\begin{proof}[Proof of Lemma~\ref{lem::field_weighted_by_quantum_local_time}]
Let $(\phi_n)$ be an orthonormal basis for $H(\h)$ consisting of smooth functions and write $\wh{h} = \sum_n \wh{\alpha}_n \phi_n$ where $(\wh{\alpha}_n)$ are i.i.d.\ $N(0,1)$ random variables.  We fix the additive constant so that the average of $h \circ f_{T_u}^{-1} + Q\log|(f_{T_u}^{-1})'|$ on $\h \cap \partial \D$ is equal to $0$.  Note that we can write the field $h \circ f_{T_u}^{-1} + Q\log|(f_{T_u}^{-1})'|$ as $\sum_n (\wh{\alpha}_n + \beta_n) \phi_n \circ f_{T_u}^{-1}$ for some sequence of coefficients $(\beta_n)$.  Fixing the additive constant so that the average of $h \circ f_{T_u}^{-1} + Q\log|(f_{T_u}^{-1})'|$ on $\h \cap \partial \D$ is equal to $0$ is equivalent to taking $\wh{h} = \sum_n \left( \wh{\alpha}_n \phi_n - (\wh{\alpha}_n + \beta_n) p_n \right)$ where $p_n$ is the average of $\phi_n \circ f_{T_u}^{-1}$ on $\h \cap \partial \D$.

For each $N \in \N$, we let $h_N = \sum_{n=1}^N \left( \wh{\alpha}_n \phi_n - (\wh{\alpha}_n + \beta_n) p_n \right)$ and we let $h^N = h-h_N$.  Since $\eta$ determines $f_{T_u}$ hence $p_n$ for every $n$, it follows that $h_N$ and $h^N$ are conditionally independent given~$\eta$.  Let $m_{q,r}^N$ be the restriction of $m^\phi$ to $[q,r]$ where $\phi = -h_N$.  Let $\qlt_t^N = \qlt_t^\phi$ be the associated local time at~$0$ of~$Y^\phi$.  By Lemma~\ref{lem::quantum_local_time_change_field} we have that
\begin{align}
   dm_{q,r}(w) dh d\eta
&= \exp\left(\frac{\gamma(2-d)}{2} h_N(\eta(T_w)) \right) dm_{q,r}^N(w) dh d\eta. \label{eqn::quantum_local_gff_weighted_expression}
\end{align}
Since $h_N$ is conditionally independent of $(\qlt_t^N,h^N)$ given $\eta$, we can rewrite~\eqref{eqn::quantum_local_gff_weighted_expression} as
\begin{align}
dm_{q,r}(w) dh d\eta = \exp\left( \frac{\gamma(2-d)}{2} h_N(\eta(T_w)) \right) dh_N dm_{q,r}^N(w) dh^N d\eta.
\label{eqn::quantum_local_gff_weighted_expression2}	
\end{align}
Since $h_N(\eta(T_w)) = \sum_{n=1}^N (\wh{\alpha}_n \phi_n(\eta(T_w)) - (\wh{\alpha}_n + \beta_n) p_n)$, it follows that weighting the law of $h_N$ by $\exp( \tfrac{\gamma(2-d)}{2} h_N(\eta(T_w)) )$ has the effect of shifting the mean of $\wh{\alpha}_n$ by
\[ \frac{\gamma(2-d)}{2} \left( \phi_n(\eta(T_w)) - p_n \right)\]
for each $n \in \{1,\ldots,N\}$.  Consequently, under $\CZ_{q,r}^{-1} dm_{q,r} dh d\eta$ the conditional law of $h_N$ given $w$ and $\eta|_{[0,T_r]}$ is equal to the law of
\[ \breve{h}_N + \sum_{n=1}^N \left( \frac{\gamma(2-d)}{2} \left(\phi_n(\eta(T_w)) - p_n\right) \phi_n - \left(\breve{\alpha}_n + \beta_n + \frac{\gamma(2-d)}{2}\left( \phi_n(\eta(T_w)) - p_n \right) \right) p_n\right) \]
where $\breve{h}_N = \sum_{n=1}^N \breve{\alpha}_n \phi_n$ for i.i.d.\ $N(0,1)$ random variables $\breve{\alpha}_1,\ldots,\breve{\alpha}_N$.  We have that
\[ \sum_{n=1}^N \phi_n(\eta(T_w)) \phi_n(\cdot) \to G(\eta(T_w),\cdot) \quad\text{as}\quad N \to \infty\]
where $G$ is the Neumann Green's function on $\h$.  We similarly have that
\[ \sum_{n=1}^N \phi_n \circ f_{T_u}^{-1}(z) \phi_n(\cdot) \to G(f_{T_u}^{-1}(z),\cdot) \quad\text{as}\quad N \to \infty.\]
Note that this function is harmonic for $w \neq f_{T_u}^{-1}(z)$.  If we now average it over $z \in \h \cap \partial \D$ (and recall the definition of $p_n$), then we see that $\sum_{n=1}^N p_n \phi_n(\cdot)$ converges as $N \to \infty$ to a function which is harmonic off $f_{T_u}^{-1}(\h \cap \partial \D)$.  Equivalently, the limit can be written as the composition of a function which is harmonic off $\h \cap \partial \D$ and $f_{T_u}$.  Combining everything, this completes the proof of part~\eqref{it::ql_field_cond_law}.

Let us now explain how to deduce~\eqref{it::ql_path_cond_law}.  Given $\eta$ and $T_w$, the expectation with respect to the law of $h_N$ of $\exp(\tfrac{\gamma(2-d)}{2} h_N(\eta(T_w)))$ is equal to a constant times
\[ \exp\left( \frac{\gamma^2(2-d)^2}{8} \sum_{n=1}^N \big( \phi_n^2(\eta(T_w)) - 2 p_n \phi_n(\eta(T_w)) + p_n^2 \big) \right).  \]
The terms $\phi_n^2(\eta(T_w))$ are determined by $\eta$ and $T_w$ so do not enter into the Radon-Nikodym derivative.  If we take a limit as $N \to \infty$, the other terms in the exponential are given by
\[ \sum_{n=1}^\infty p_n^2 - 2 \sum_{n=1}^\infty \phi_n(\eta(T_w)) p_n = \iint G(f_{T_u}^{-1}(x), f_{T_u}^{-1}(y)) dx dy - 2 \int G(\eta(T_w),f_{T_u}^{-1}(x)) dx,\]
(where the integrals are over $\partial \D \cap \h$) as desired.

It is left to justify part~\eqref{it::ql_bubbles_independent}.  Under the (unweighted) law~$dh d\eta$ (but with the additive constant still fixed as in Theorem~\ref{thm::skinny_wedge_bubble_structure}), it follows from Theorem~\ref{thm::skinny_wedge_bubble_structure} that the beaded surface which consists of the bubbles parameterized by the bounded components of~$\h \setminus \eta([0,T_r])$ (from right to left) is given by that of a quantum wedge of weight $\rho+2$, up to a given amount of local time, from which the claimed result follows.
\end{proof}

\subsubsection{Proof of Theorem~\ref{thm::bubbles_quantum_local_typical}}
We will first give the proof of part~\eqref{it::quantum_typical_bubble_form}.  Let $(f_t)$ be the centered, forward Loewner flow associated with $\eta$.  We first recall that by Lemma~\ref{lem::zip_up_sle_kappa_rho_law} we know that $dh d\eta$ is invariant under applying a change of coordinates using~$f_{T_q}$.  That is, if $(h,\eta)$ are sampled from $dh d\eta$ then $(h \circ f_{T_q}^{-1} + Q\log|(f_{T_q}^{-1})'|,f_{T_q}(\eta)) \stackrel{d}{=} (h,\eta)$ where we take the first coordinate modulo additive constant.  If we fix the additive constant for $h$ so that the average of $h \circ f_{T_u}^{-1} + Q\log| (f_{T_u}^{-1})'|$ on $\h \cap \partial \D$ is equal to $0$, then the additive constant for $h \circ f_{T_q}^{-1} + Q\log|(f_{T_q}^{-1})'|$ is fixed so that if one cuts for another $u-q$ units of local time (here $u > q$) then the resulting field has average on $\h \cap \partial \D$ equal to $0$.

For $0 \leq q < r \leq u$, we let $m_{q,r}$ be as in the statement of Lemma~\ref{lem::field_weighted_by_quantum_local_time} and let $m_r = m_{0,r}$.  We now work under the law $\CZ_u^{-1} dm_u(w) dh d\eta$ where $\CZ_u = \CZ_{0,u}$ is the normalizing constant from Lemma~\ref{lem::field_weighted_by_quantum_local_time}.  We claim that the conditional law of the pair consisting of $f_{T_w}(\eta)$ and $h \circ f_{T_w}^{-1} + Q \log|(f_{T_w}^{-1})'|$ given $w$ is equal to that of an $\SLE_\kappa(\rho)$ process $\breve{\eta}$ in $\h$ with a single boundary force point located at $0^+$ weighted by the Radon-Nikodym derivative from part~\eqref{it::ql_path_cond_law} of Lemma~\ref{lem::field_weighted_by_quantum_local_time} and centered Loewner flow $(\breve{f}_t)$ with associated inverse local times $\breve{T}$ and
\[ \breve{h} = \wh{h} + \frac{\rho+2-\gamma^2}{\gamma} \log|\cdot| + \Fh \circ \breve{f}_{\breve{T}_{u-w}}\]
where $\wh{h}$ is a free boundary GFF and the additive constant is normalized so that the average of the field after applying the change of coordinates with $\breve{f}_{\breve{T}_{u-w}}$ on $\h \cap \partial \D$ is equal to $0$.  Here, $\Fh$ is the harmonic function from Lemma~\ref{lem::quantum_local_time_change_field}.  This will follow from three observations.
\begin{itemize}
\item The conditional law of $(w,h,\eta)$ under $\CZ_u^{-1} dm_u(w) dh d\eta$ given $w \in [q,r]$ is equal to $\CZ_{q,r}^{-1} dm_{q,r}(w) dh d\eta$.
\item If $(w,h,\eta)$ have law $\CZ_{q,r}^{-1} dm_{q,r}(w) dh d\eta$, then the triple consisting of $w-q$, $h \circ f_{T_q}^{-1} + Q\log|(f_{T_q}^{-1})'|$, and $f_{T_q}(\eta)$ has law $\CZ_{r-q}^{-1} dm_{r-q}(w) dh d\eta$.
\item The limit as $s \to 0$ of the law $\CZ_s^{-1} dm_s(w) dh d\eta$ is equal to the law of the triple consisting of $w=0$, an $\SLE_\kappa(\rho)$ process $\breve{\eta}$ on $\h$ from $0$ to $\infty$ with a single boundary force point located at $0^+$ weighted by the Radon-Nikodym derivative from part~\eqref{it::ql_path_cond_law} of Lemma~\ref{lem::field_weighted_by_quantum_local_time} with $w=0$, and
\[ \breve{h} = \wh{h} + \frac{\rho+2-\gamma^2}{\gamma} \log|\cdot| + \Fh \circ \breve{f}_{\breve{T}_u}\]
where $\wh{h}$ is a free boundary GFF on $\h$.  Here, the additive constant for $\breve{h}$ is fixed so that the average of $\breve{h} \circ \breve{f}_{\breve{T}_u}^{-1} + Q\log|(\breve{f}_{\breve{T}_u}^{-1})'|$ on $\h \cap \partial \D$ is equal to $0$ where $(\breve{f}_t)$ and $\breve{T}$ are the centered Loewner flow and inverse local times for $\breve{\eta}$ and $\Fh$ is the harmonic function from Lemma~\ref{lem::field_weighted_by_quantum_local_time}.  Indeed, this follows from the explicit forms of the conditional laws given in part~\eqref{it::ql_field_cond_law} and~\eqref{it::ql_path_cond_law} of Lemma~\ref{lem::field_weighted_by_quantum_local_time}.
\end{itemize}
Combining these three observations proves the claim.

Part~\eqref{it::ql_bubbles_independent} of Lemma~\ref{lem::field_weighted_by_quantum_local_time} implies that the bubbles which are to the right of $f_{T_w}(\eta)$ are locally given by the bubbles in a weight $\rho+2$ quantum wedge.  Thus the first assertion of the theorem follows by zooming in as in part~\eqref{it::quantum_wedge_limit} of Proposition~\ref{prop::quantum_wedge_properties}.

We now turn to prove part~\eqref{it::quantum_typical_stationary}.  We let $u,r > 0$ and consider the law $\CZ_u^{-1} dm_u dh d\eta$ (with the additive constant fixed so that the average of $h \circ f_{T_u}^{-1} + Q \log |(f_{T_u}^{-1})'|$ on $\h \cap \partial \D$ is equal to~$0$).  We will eventually take a limit as $u \to \infty$ but we will leave $r$ fixed.  Let $s_u = m([0,u])$ be the total amount of local time of the bubbles cut out by $\eta([0,T_u])$ as measured by the Bessel process $Y$ which encodes the weight $\rho+2$ quantum wedge.  We note that, at the moment, these bubbles are naturally ordered from right to left (rather than from left to right as in the statement).  Suppose that $w$ has been picked from $m_u$, $t = m_u([0,w])$, and let $t' = t+r$.  Note that we can sample from the law of $t$ by first picking $U$ uniformly in $[0,m_u([0,u])]$ then taking $t = U$.  We can similarly sample from~$t'$ by picking~$U'$ uniformly in $[0,m_u([0,u])]$ and then taking $t' = U'+r$.  Since $m_u([0,u]) \to \infty$ as $u \to \infty$, it follows that the total variation distance between the law of $U$ and the law of $U'+r$ tends to $0$ as $r \to \infty$, which implies that the same is true for the laws of $t$ and $t'$.  For $v \geq 0$, let $\wh{\qnt}_v$ be the capacity time for $\eta$ which corresponds to when $\eta$ has cut off $v$ units of local time as measured by $Y$ from $\infty$.  It follows that, as $u \to \infty$, the total variation distance between the laws of $(h \circ f_{\wh{\qnt}_t}^{-1} + Q\log|(f_{\wh{\qnt}_t}^{-1})'|, f_{\wh{\qnt}_t}(\eta))$ and $(h \circ f_{\wh{\qnt}_{t'}}^{-1} + Q\log|(f_{\wh{\qnt}_{t'}}^{-1})'|, f_{\wh{\qnt}_{t'}}(\eta))$ tends $0$.  By the proof of part~\eqref{it::quantum_typical_bubble_form} given above, we have that the local behavior near $0$ under the law of $(h \circ f_{\wh{\qnt}_t}^{-1} + Q\log|(f_{\wh{\qnt}_t}^{-1})'|, f_{\wh{\qnt}_t}(\eta))$ is equal to the law of the pair on the left side of~\eqref{eqn::quantum_local_time_invariance}.  Note that the right side of~\eqref{eqn::quantum_local_time_invariance} is obtained by shifting time in the left side by a given amount of quantum local time.  Part~\eqref{it::quantum_typical_stationary} follows because $(h \circ f_{\wh{\qnt}_{t'}}^{-1} + Q\log|(f_{\wh{\qnt}_{t'}}^{-1})'|, f_{\wh{\qnt}_{t'}}(\eta))$ is obtained from $(h \circ f_{\wh{\qnt}_t}^{-1}  + Q\log|(f_{\wh{\qnt}_t}^{-1})'|, f_{\wh{\qnt}_t}(\eta))$ in the same manner.

We are now going to finish the proof of the theorem by establishing the claimed result in the setting of a weight $\rho+2$ quantum cone decorated by an independent whole-plane $\SLE_\kappa(\rho)$ process.  First, we recall that Corollary~\ref{cor::cone_bubble_form} implies that the following is true.  Suppose that
\[ h = \wh{h} - \frac{\gamma^2+2-\rho}{2\gamma}\log|\cdot|\]
where $\wh{h}$ is a whole-plane GFF.  Let $\eta$ be an independent whole-plane $\SLE_\kappa(\rho)$ process.  We then assume that we have fixed the additive constant for $h$ as in the statement of Corollary~\ref{cor::cone_bubble_form} so that if $\ol{\eta}$ denotes the time-reversal of $\eta$ then the amount of mass in the component of $\C \setminus \ol{\eta}((-\infty,0])$ containing the origin is equal to $1$.  Corollary~\ref{cor::cone_bubble_form} then tells us that the structure of the bubbles, each viewed as a quantum surface, that $\eta$ separates from $\infty$ after it has separated $1$ unit of quantum mass from $\infty$ is given by that of a wedge of weight $\rho+2$.  We note that if we instead let $\epsilon > 0$ and fix the additive constant for $h$ so that the amount of mass in the component of $\C \setminus \ol{\eta}((-\infty,0])$ containing $0$ is equal to $\epsilon$ then the structure of the bubbles that $\eta$ separates from $\infty$ after it has separated $\epsilon$ units of mass from $\infty$ is also that of a weight $\rho+2$ quantum wedge.  Note that the same remains true if we then scale the pair $(h,\eta)$ so that $h_r(0) + Q\log r$ first hits $0$ at $r=0$.  Moreover, in the limit as $\epsilon \to 0$, the pair $(h,\eta)$ converges to a pair consisting of a weight $\rho+2$ quantum cone and independent whole-plane $\SLE_\kappa(\rho)$ process. \qed

\subsubsection{The case of $\SLE_{\kappa'}$ processes}

We now give the analog of Theorem~\ref{thm::bubbles_quantum_local_typical} for $\SLE_{\kappa'}$ processes.  We will not provide a separate proof since it follows from the same argument used to prove Theorem~\ref{thm::bubbles_quantum_local_typical} (using the representation of the Poissonian structure from Proposition~\ref{prop::unit_area_bl_construction}).

\begin{theorem}
\label{thm::sle_kp_quantum_local_typical}
Fix $\kappa' \in (4,8)$, let $\gamma=4/\sqrt{\kappa'}$, and suppose that $(\H,h)$ is a quantum wedge of weight $\tfrac{3\gamma^2}{2}-2$.  Let $\eta'$ be an $\SLE_{\kappa'}$ process in~$\h$ from~$0$ to~$\infty$ which is independent of~$\eta'$.  Then the following hold:
\begin{enumerate}[(i)]
\item The law of the beaded surface consisting of the components of $\h \setminus \eta'$ whose boundary is contained in $\eta'$ can be sampled from as in Theorem~\ref{thm::kappa_prime_bubbles} (with the $\kappa'$ value matched).
\item Let $(f_t)$ be the centered (forward) Loewner flow for $\eta'$.  Let $\qnt_u$ be the time parameterization of the bubbles in the previous part which comes from the first coordinate in the \ppp\ description.  For each $u \geq 0$, we have (as curve-decorated quantum surfaces) that
\begin{equation}
\label{eqn::sle_kp_typical_invariant}
(h,\eta') \stackrel{d}{=} (h \circ f_{\qnt_u}^{-1} + Q\log| (f_{\qnt_u}^{-1})'|, f_{\qnt_u}(\eta')).
\end{equation}
\end{enumerate}
The first item holds if we take $\eta'$ to be a whole-plane $\SLE_{\kappa'}$ process and replace $(\H,h)$ with a quantum cone $(\C,h)$ of weight $\tfrac{\gamma^2}{2}$.
\end{theorem}

\begin{definition}
\label{def::qnt_sle_kp}
The time parameterization~$\qnt_u$ from Theorem~\ref{thm::sle_kp_quantum_local_typical} is the {\bf quantum natural time} parameterization of the $\SLE_{\kappa'}$ process $\eta'$.
\end{definition}

\begin{remark}
\label{rem::quantum_typical_other_bessel}
In the setting of Theorem~\ref{thm::sle_kp_quantum_local_typical}, the multiple of $\log$ singularity at the starting point of $\eta'$ is given by $-\sqrt{8/3}$ in the case that $\kappa'=6$.  This is natural to expect in the context of \cite{ms2013qle} because $\QLE(8/3,0)$ has the interpretation of corresponding to first passage percolation on a $\sqrt{8/3}$-LQG surface.  In particular, this implies that in the quantum natural time version of $\QLE(8/3,0)$ (as opposed to the capacity time version constructed in \cite{ms2013qle}) we can pick the points from which we draw independent segments of $\SLE_6$ from the $\sqrt{8/3}$-LQG boundary measure.
\end{remark}

\section{Welding quantum wedges}
\label{sec::quantumwedge}

The purpose of this section is to describe the conformal welding of wedges with general weights, thus generalizing the results of \cite{she2010zipper}.  We will prove Theorem~\ref {thm::zip_up_wedge_rough_statement} in pieces:  Theorem~\ref {thm::zip_up_wedge_rough_statement} follows directly from Proposition~\ref{prop::cone_divide}, Proposition~\ref{prop::cone_divide_self_intersecting} and Theorem~\ref{thm::paths_determined}.

The first step is Section~\ref{subsec::adding_to_weight_2}, in which we will explain how to weld together a weight $W \geq \tfrac{\gamma^2}{2}$ wedge with a wedge of weight $2$ and that the resulting object is a wedge of weight $W+2$.  This will follow the strategy of \cite{she2010zipper}, though we will provide most of the details of the proof.  We will then show in Section~\ref{subsec::fat_zipping} that conformally welding the left and right boundaries of a wedge of weight $W \geq \tfrac{\gamma^2}{2}$ yields a quantum cone of weight $W$ with an independent non-self-intersecting whole-plane $\SLE_\kappa(\rho)$ process ($\kappa = \gamma^2$) with $\rho = W-2$ drawn on top of it.  Combining the results of Section~\ref{subsec::adding_to_weight_2} and Section~\ref{subsec::fat_zipping}, we will complete the proof of Theorem~\ref{thm::welding} by showing in Section~\ref{subsec::zipping_general_wedges} that the conformal welding of two wedges of weights $W_1, W_2 > 0$ yields a wedge of weight $W_1+W_2$.  We will then show in Section~\ref{subsec::zipping_thin_wedge} that welding the left and right sides of a wedge of weight $W \in (0,\tfrac{\gamma^2}{2})$ yields a quantum cone of weight $W$ with an independent self-intersecting whole-plane $\SLE_\kappa(\rho)$ with $\rho = W-2$ drawn on top of it.

\subsection{Welding a thick wedge and a weight $2$ wedge}
\label{subsec::adding_to_weight_2}

\begin{figure}[h!]
\begin{center}
\includegraphics[scale=0.85]{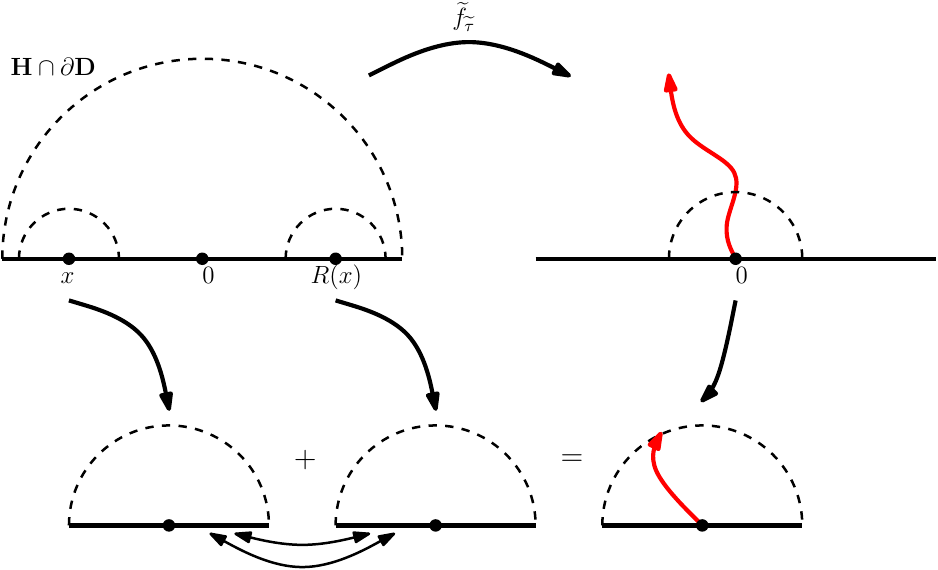}
\caption {\label{fig::triplezoom} Illustration of the proof of Proposition~\ref{prop::high_weight_welding}.  For some fixed $x \in (-1,0)$, write $h = \wh{h} + \frac{2}{\gamma} \log |\cdot| - \alpha \log|x - \cdot|,$ where $\wh{h}$ is a free boundary GFF on $\H$ normalized so that the mean of $h$ on $\h \cap \partial \D$ is equal to $0$.  Suppose that $(\wt{f}_t)$ is the centered reverse Loewner flow associated with a reverse $\SLE_\kappa(\wt{\rho})$ process with a single boundary force point of weight $\wt{\rho} = \alpha \gamma$ located at $x$.  Let $\wt{\tau} = \inf\{t \geq 0 : \wt{f}_t(x) = 0\}$ so that $\wt{f}_{\wt{\tau}}$ zips $[x,0]$ up with $[0,\wt{R}(x)]$.  Consider the three random surfaces obtained by choosing a semi-disk of quantum mass $\eps$ centered at each of $x$ and $\wt{R}(x)$ (on the left side) and $0$ (on the right side), and multiplying areas by $1/\eps$ (zooming in) so that three semi-disks each have unit quantum mass.  In the $\eps \to 0$ limit, part~\eqref{it::quantum_wedge_limit} of Proposition~\ref{prop::quantum_wedge_properties} implies that the left two quantum surfaces respectively converge to an independent $\alpha$-quantum wedge and $\gamma$-quantum wedge, and the right surface converges to the conformal welding of these two.}
\end{center}
\end{figure}

Suppose that $h = \wh{h} - \alpha \log|\cdot|$ where $\wh{h}$ is a free boundary GFF on~$\h$.  Assume that the additive constant for $h$ is fixed so that its average on $\h \cap \partial \D$ is equal to $0$.  Recall from~\eqref{eqn::wedge_weight} that the formula for the weight $W$ of the wedge $\CW$ which arises by rescaling $h$ near~$0$ as in part~\eqref{it::quantum_wedge_limit} of Proposition~\ref{prop::quantum_wedge_properties} is given by
\[ W = \gamma\left(\gamma+\frac{2}{\gamma}-\alpha\right).\]
The purpose of this section is to show that the conformal welding of a wedge with weight $W_1 \geq \tfrac{\gamma^2}{2}$ with a wedge of weight $W_2 = 2$ yields a wedge of weight $W_1 + W_2$.  The main result is the following proposition.

\begin{proposition}
\label{prop::high_weight_welding}
Suppose that $\CW$ is a wedge of weight $W \geq \tfrac{\gamma^2}{2}+2$ represented by some $(D,h,(z_1,z_2))$ with $z_1,z_2 \in \partial D$ distinct.  Let $\eta$ be an $\SLE_\kappa(W-4;0)$ process in $D$ independent of $\CW$ from $z_1$ to $z_2$ with a single boundary force point located at~$z_1^-$.  (Note that $\eta$ a.s.\ does not hit $\partial D$ except at its starting and terminal points since $W-4 \geq \tfrac{\kappa}{2}-2$.)  Let $D_1$ and $D_2$ denote the left and right components of $D \setminus \eta$.  Then the quantum surfaces $\CW_1 = (D_1,h,(z_1,z_2))$ and $\CW_2 = (D_2,h,(z_1,z_2))$ are independent and $\CW_1$ (resp.\ $\CW_2$) has the law of a quantum wedge with weight $W_1 = W-2$ (resp.\ $W_2 = 2$).
\end{proposition}

We are going to prove Proposition~\ref{prop::high_weight_welding} in the case that $W > \tfrac{\gamma^2}{2}+2$; the case that $W = \tfrac{\gamma^2}{2}+2$ then follows by taking a limit as $W \downarrow \tfrac{\gamma^2}{2}+2$.  Indeed, it is easy to see that the law of an $\SLE_\kappa(\rho)$ process is continuous in $\rho$ with respect to the Caratheodory topology because the law of the Bessel process used to define the driving function for the $\SLE_\kappa(\rho)$ is continuous with respect to the weak topology associated with local uniform convergence.  Similarly, the law of a thick quantum wedge is also continuous in $W$ as the law of the Bessel process used to define it is continuous in $W$.  The proof of Proposition~\ref{prop::high_weight_welding} will follow from Lemmas~\ref{lem::high_weight_near_origin}--\ref{lem::high_weight_near_rx}, stated and proved below, and the argument sketched in Figure~\ref{fig::triplezoom}.  In particular, Lemma~\ref{lem::high_weight_near_origin} gives the rightmost downward arrow in Figure~\ref{fig::triplezoom} and Lemma~\ref{lem::high_weight_near_x} and Lemma~\ref{lem::high_weight_near_rx} respectively give the leftmost and middle downward arrows in Figure~\ref{fig::triplezoom}.  We begin with the following elementary lemma for the free boundary GFF.

\begin{lemma}
\label{lem::sigma_algebra_is_trivial}
Suppose that $h$ is a free boundary GFF on $\h$ and $x \in \partial \h$.  For each $\delta > 0$, let $\CF_{x,\delta}$ be the $\sigma$-algebra generated by the restriction of $h$ to $\h \cap B(x,\delta)$.  Then  $\cap_{\delta >0} \CF_{x,\delta}$ is trivial.  
\end{lemma}
\begin{proof}

Let $\phi \in C_0^\infty(\h)$ with $\int \phi(z) dz = 0$.  It suffices to show that the conditional law of $(h,\phi)$ given $\CF_{x,\delta}$ converges as $\delta \to 0$ to the unconditioned law of $(h,\phi)$.  Note that the latter is explicitly given by that of a Gaussian random variable with mean zero and variance $\iint \phi(y) G(y,z) \phi(z) dy dz$ where $G$ is the Neumann Green's function on~$\h$.  For each $\delta > 0$, we let $\Fh_{x,\delta}$ be the harmonic extension of the boundary values of~$h$ from $\h \cap \partial B(x,\delta)$ to $\h \setminus B(x,\delta)$.  We also let $G_{x,\delta}$ be the Green's function on $\h \setminus B(x,\delta)$ with Dirichlet (resp.\ Neumann) boundary conditions on $\h \cap \partial B(x,\delta)$ (resp. $\partial \h \setminus B(x,\delta)$).  By the Markov property for the GFF, the conditional law of $(h,\phi)$ given $\CF_{x,\delta}$ is given by that of a Gaussian random variable with mean $(\Fh_{x,\delta},\phi)$ and variance $\iint \phi(y) G_{x,\delta}(y,z) \phi(z) dy dz$.  It is clear that 
\[ \iint \phi(y) G_{x,\delta}(y,z) \phi(z) dy dz \to \iint \phi(y) G(y,z) \phi(z) dy dz \quad\text{as}\quad \delta \to 0\]
since $G_{x,\delta} - G \to 0$ locally uniformly as $\delta \to 0$.

To complete the proof, it suffices to show that $(\Fh_{x,\delta},\phi) \to 0$ in probability as $\delta \to 0$.  To see this, we note that $(\Fh_{x,\delta},\phi)$ is a Gaussian random variable with mean zero and variance given by $\iint \phi(y) \cov(\Fh_{x,\delta}(y),\Fh_{x,\delta}(z)) \phi(z) dy dz$ (see e.g., \cite[Section~4.2.4]{ms2013qle}).  Letting $\CP_{x,\delta}$ be the Poisson kernel on $\h \setminus B(x,\delta)$, we have that
\[ \cov(\Fh_{x,\delta}(y), \Fh_{x,\delta}(z)) = \iint \CP_{x,\delta}(y,u) \CP_{x,\delta}(z,v) G(u,v) dudv\]
where the integral is over $(\h \cap \partial B(x,\delta))^2$.  This quantity is easily seen to tend to $0$ as $\delta \to 0$ since $\CP_{x,\delta}(y,u) = O(1)$ uniformly over $y$ in the support of $\phi$ and $u \in \h \cap \partial B(x,\delta)$ as $\delta \to 0$.
\end{proof}

\begin{lemma}
\label{lem::high_weight_near_origin}
Fix $\kappa \in (0,4)$, $\alpha < Q$, and let $\wt{\rho} =  \alpha \gamma$.  Suppose that $\wh{h}$ is a free boundary GFF on~$\h$, $x < 0$, and let $h = \wh{h} + \tfrac{2}{\gamma}\log|\cdot| - \alpha\log|\cdot-x|$ with the additive constant fixed so that its average on $\h \cap \partial \D$ is equal to $0$.  Let $(\wt{f}_t)$ be the centered reverse Loewner flow associated with a reverse $\SLE_\kappa(\wt{\rho})$ process with a single boundary force point of weight $\wt{\rho}$ located at $x$.  For each $t \geq 0$, let
\begin{equation}
\label{eqn::high_weight_ht}
\wt{h}^t = h \circ \wt{f}_t^{-1} + Q\log|(\wt{f}_t^{-1})'(\cdot)|.
\end{equation}
Let $\wt{\tau} = \inf\{t \geq 0: \wt{f}_t(x) = 0\}$ and let $\wt{\eta}_{\wt{\tau}} = \h \setminus \wt{f}_{\wt{\tau}}(\h)$.  Then the pair $(\wt{h}^{\wt{\tau}},\wt{\eta}_{\wt{\tau}})$ rescaled as in part~\eqref{it::quantum_wedge_limit} of Proposition~\ref{prop::quantum_wedge_properties} converges in law to $(\CW,\eta)$ where $\CW$ is a wedge of weight $\gamma^2+4-\alpha \gamma$ and $\eta$ is a forward $\SLE_\kappa(\rho)$ process with $\rho = \kappa-\gamma \alpha$ with a single boundary force point located at $0^-$.  Moreover, $\CW$ and $\eta$ are independent.
\end{lemma}
\begin{proof}
The assumption that $\alpha < Q$ implies that $\wt{\rho} =  \alpha \gamma < \tfrac{\kappa}{2}+2$.  This, in turn, implies that the dimension $\delta$ of the Bessel process $\wt{f}_t(x)/\gamma$ (recall~\eqref{eqn::singleforce} and~\eqref{eqn::besseldim_reverse}) satisfies $\delta < 2$.  This implies that $\p[\wt{\tau} < \infty] = 1$ (recall Section~\ref{subsec::bessel_processes}).

By Proposition~\ref{prop::zip_up_to_zero}, we also know that $\wt{\eta}_{\wt{\tau}}$ is given by the initial segment of a forward $\SLE_\kappa(\rho)$ process with a single boundary force point located at $0^-$ of weight $\rho = \kappa-\wt{\rho}$ stopped at an a.s.\ positive time.  Moreover, by Theorem~\ref{thm::reverse_coupling}, we know that $\wt{f}_{\wt{\tau}}$ and $\wt{h}^{\wt{\tau}}$ are independent hence $\wt{\eta}_{\wt{\tau}}$ and $\wt{h}^{\wt{\tau}}$ are also independent.  Since $\wt{h}^{\wt{\tau}}$ can be written as the sum of a free boundary GFF on $\h$ plus $(\tfrac{2}{\gamma}-\alpha)\log|\cdot|$, it follows that the quantum surface described by $\wt{h}^{\wt{\tau}}$ rescaled as in part~\eqref{it::quantum_wedge_limit} of Proposition~\ref{prop::quantum_wedge_properties} converges to the type of wedge indicated in the statement.  Likewise, it is clear that $\wt{\eta}_{\wt{\tau}}$ converges to an independent $\SLE_\kappa(\rho)$ process drawn from $0$ to $\infty$ upon applying the same scaling.  Combining gives the result.
\end{proof}

\begin{lemma}
\label{lem::high_weight_near_x}
Suppose that $\wh{h}$ is a free boundary GFF on $\h$, $x < 0$, and $\alpha < Q$.  Let $\nu_h$ be the $\gamma$-LQG boundary length measure associated with $h=\wh{h}+\tfrac{2}{\gamma}\log|\cdot|-\alpha\log|x-\cdot|$, with the additive constant fixed so that its average on $\h \cap \partial \D$ is equal to~$0$.  Fix $r > 0$.  Conditional on both the values of $h$ restricted to $\h \setminus B(x,r)$ and $\nu_h([x,x+r])$, the law of the field near $x$ rescaled as in part~\eqref{it::quantum_wedge_limit} of Proposition~\ref{prop::quantum_wedge_properties} converges to a wedge of weight $\gamma^2 + 2 - \alpha \gamma$.
\end{lemma}
\begin{proof}
For each $\delta > 0$, we let $\CF_{x,\delta}$ be the $\sigma$-algebra generated by the values of~$h$ restricted to~$\h \cap B(x,\delta)$.  By Lemma~\ref{lem::sigma_algebra_is_trivial}, $\cap_{\delta > 0} \CF_{x,\delta}$ is trivial.  In particular, by the reverse martingale convergence theorem, the conditional law of the pair $(h|_{\h \setminus B(x,r)}, \nu_h([x,x+r]))$ given $\CF_{x,\delta}$ converges to the unconditioned law as $\delta \to 0$.  The result then follows by applying Bayes' rule to reverse the order of the conditioning (i.e., to consider the conditional law of the restriction of $h$ to $\h \cap B(x,\delta)$ given the pair $(h|_{\h \setminus B(x,r)}, \nu_h([x,x+r]))$) and then combining with part~\eqref{it::quantum_wedge_limit} of Proposition~\ref{prop::quantum_wedge_properties}.
\end{proof}

\begin{lemma}
\label{lem::high_weight_near_rx}
Suppose that $\wh{h}$ is a free boundary GFF on $\h$, $x < 0$, and $\alpha < Q$.  Let $h = \wh{h} + \tfrac{2}{\gamma}\log|\cdot| - \alpha\log|x-\cdot|$ with the additive constant fixed so that its average on $\h \cap \partial \D$ is equal to~$0$.  Let $\wt{\rho} = \alpha  \gamma $ and let $(\wt{f}_t)$ be the centered reverse Loewner flow associated with a reverse $\SLE_\kappa(\wt{\rho})$ process with a single boundary force point of weight $\wt{\rho}$ located at $x$.  Assume that $(\wt{f}_t)$ and $h$ are independent.  Let $\wt{R}(x) > 0$ be the unique point which is identified with $x$ under $(\wt{f}_t)$.  Then the joint law on quantum surfaces which arises by rescaling as in part~\eqref{it::quantum_wedge_limit} of Proposition~\ref{prop::quantum_wedge_properties} near both $x$ and $\wt{R}(x)$ is that of independent wedges respectively of weight $\gamma^2 + 2 - \alpha \gamma$ and of weight~$2$.
\end{lemma}

Before we prove Lemma~\ref{lem::high_weight_near_rx}, we first recall \cite[Proposition~5.5]{she2010zipper}.

\begin{proposition}
\label{prop::quantum_typical_zoom_in}
Fix $\gamma \in [0,2)$ and let $D$ be a bounded subdomain of $\h$ for which $\partial D \cap \R$ is a segment of positive length.  Let $\wh{h}$ be a GFF with zero boundary conditions on $\partial D \setminus \R$ and free boundary conditions on $\partial D \cap \R$.  Let $[a,b]$ be any sub-interval of $\partial D \cap \R$ and let $\Fh_0$ be a continuous function on $D$ that extends continuously to the interval $(a,b)$.  Let $dh$ be the law of $\Fh_0 + \wh{h}$, and let $\nu_h([a,b]) dh$ denote the measure whose Radon-Nikodym derivative with respect to $dh$ is $\nu_h([a,b])$.  Now suppose we:
\begin{enumerate}
\item Sample $h$ from $\nu_h([a,b])dh$ (normalized to be a probability measure).
\item Sample $x$ uniformly from $\nu_h$ restricted to $[a,b]$ (normalized to be a probability measure).
\item Let $h^*$ be $h$ translated by $-x$ units horizontally.
\end{enumerate}
Then as $C \to \infty$ the random surfaces $(\h,h^* + \tfrac{C}{\gamma})$ converge in law (with respect to the topology of convergence of doubly-marked quantum surfaces) to a $\gamma$-quantum wedge.  Moreover, the same remains true even if, when we choose $h$ and $x$, we condition on a particular pair of values $L_1 = \nu_h([a,x])$ and $L_2 = \nu_h([x,b])$.
\end{proposition}

We note that the first part of Proposition~\ref{prop::quantum_typical_zoom_in} is a consequence of part~\eqref{it::quantum_wedge_limit} of Proposition~\ref{prop::quantum_wedge_properties}.

\begin{proof}[Proof of Lemma~\ref{lem::high_weight_near_rx}]
We first note that it follows from \cite[Theorems~1.2 and~1.3]{she2010zipper} and the absolute continuity properties of the GFF that $\nu_h([x,0]) = \nu_h([0,\wt{R}(x))])$.  Fix $r > 0$ so that $B(x,r) \cap \partial \h \subseteq \R_-$.  Let $L_1 = \nu_h([x,x+r])$, $L_2 = \nu_h([x+r,0])$, and let $L = L_1 + L_2$.  Fix $y > 0$ large.  Proposition~\ref{prop::quantum_typical_zoom_in} implies that the following is true.  Suppose that we pick $u \in [0,y]$ according to $\nu_h$ and condition on $\nu_h([0,u]) = L$.  Then the conditional law of $h$ translated by $-u$ given $L_1$, $L_2$, and the restriction of $h$ to $\h \cap \partial B(x,r)$ rescaled as in part~\eqref{it::quantum_wedge_limit} of Proposition~\ref{prop::quantum_wedge_properties} is described by that of a $\gamma$-quantum wedge (i.e., of weight $2$).  We also know from Lemma~\ref{lem::high_weight_near_x} that the conditional law of $h$ translated by $-x$ given $L_1$ and $h$ restricted to $\h \cap  \partial B(x,r)$ rescaled as in part~\eqref{it::quantum_wedge_limit} of Proposition~\ref{prop::quantum_wedge_properties} converges to that of a wedge of weight $\gamma^2 + 2 - \alpha \gamma$.  The lemma thus follows because the restriction of $h$ to $\h \cap B(x,r)$ is conditionally independent of the pair consisting of the restriction of $h$ to $\h \setminus B(x,r)$ and $L$ given the restriction of $h$ to $\h \cap \partial B(x,r)$ and $L_1$.
\end{proof}

\begin{proof}[Proof of Proposition~\ref{prop::high_weight_welding}]
Combine Lemmas~\ref{lem::high_weight_near_origin}--\ref{lem::high_weight_near_rx} with the argument described in the caption of Figure~\ref{fig::triplezoom}.
\end{proof}

\subsection{Zipping up a thick wedge}
\label{subsec::fat_zipping}

Fix $\alpha < Q$.  Recall from Section~\ref{subsec::cones} that an $\alpha$-quantum cone is the surface which arises by starting with an instance of the whole-plane GFF, adding to it $-\alpha\log|\cdot|$, and then rescaling as part~\eqref{it::quantum_cone_limit} of Proposition~\ref{prop::quantum_cone_properties} so that the resulting surface is invariant under the operation of multiplying its area by a constant.  Recall that the weight of an $\alpha$-quantum cone is given by
\[ W = 2\gamma(Q-\alpha).\]
The purpose of this section is to explain how zipping up the left and right sides of a quantum wedge with weight $W \geq \tfrac{\gamma^2}{2}$ (so that the wedge (together with its prime-end boundary) is homeomorphic to $\ol{\h}$) according to quantum length yields a quantum cone of weight $W$ where the image of the welding of $\R_-$ with $\R_+$ is given by an independent whole-plane $\SLE_\kappa(\rho)$ process with $\rho = W-2$.  This is equivalent to stating that cutting a quantum cone of weight $W$ along an independent whole-plane $\SLE_\kappa(\rho)$ process with $\rho = W-2$ yields a quantum wedge of weight $W$ and, moreover, that the quantum length of the image of the left side of a segment of the path is the same as that of the image of its right side.  This result is contained in the following proposition.  (See Figure~\ref{fig::fat_wedge_zipped_up} for an illustration; in Section~\ref{subsec::zipping_thin_wedge} we will explain how to generalize this procedure to wedges which are not homeomorphic to $\h$.)

\begin{proposition}
\label{prop::cone_divide}
Fix $W \geq \tfrac{\gamma^2}{2}$ and suppose that $\CC = (\C,h,0,\infty)$ is a quantum cone of weight~$W$.  Suppose that $\eta$ is a whole-plane $\SLE_\kappa(W-2)$ process independent of $\CC$ starting from~$0$.  Then $(\C \setminus \eta,h,0,\infty)$ is a quantum wedge of weight $W$.  Moreover, if $f \colon \h  \to \C \setminus \eta$ is a conformal transformation fixing $0$ and $\infty$, $\wh{h} = h \circ f + Q\log|f'|$, and $x \in \eta$, then the lengths of the intervals connecting $0$ and the two images of $x$ under $f^{-1}$ are the same with respect to the $\gamma$-LQG boundary measure $\nu_{\wh{h}}$ on $\R$.
\end{proposition}

Note that the assumption that $W \geq \tfrac{\gamma^2}{2}$ implies that $\rho = W-2 \geq \tfrac{\kappa}{2}-2$.  This is the critical threshold at or above which the whole-plane $\SLE_\kappa(\rho)$ processes are simple and below which they are self-intersecting (see \cite[Section~2.1]{ms2013imag4}).

\begin{proof}[Proof of Proposition~\ref{prop::cone_divide}]
In the proof, we are going to describe our cones in terms of $\alpha$ values rather than weights.  We are going to prove the result in the case that $\alpha < \tfrac{2}{\gamma} + \tfrac{\gamma}{4}$; the case in which $\alpha = \tfrac{2}{\gamma} + \tfrac{\gamma}{4}$ then follows by taking a limit as $\alpha \uparrow \tfrac{2}{\gamma} + \tfrac{\gamma}{4}$.  Suppose that $\wh{h}$ is a free boundary GFF on~$\h$.  Let $\wt{\rho} = 2\alpha \gamma$ and let $(\wt{f}_t)$ be the centered reverse Loewner flow associated with a centered reverse $\SLE_\kappa(\wt{\rho})$ process with a single interior force point which starts infinitesimally above~$0$.  Note that $\alpha < \tfrac{2}{\gamma} + \tfrac{\gamma}{4}$ implies that $\wt{\rho} < \tfrac{\kappa}{2}+4$ so that the existence of this process is given by Proposition~\ref{prop::reverse_zip_in_solution}.  We assume that $\wh{h}$ and $(\wt{f}_t)$ are independent.  Let $\wt{Z}_t$ denote the evolution of the force point under $\wt{f}_t$.    Let~$G$ be the Neumann Green's function on~$\h$ (recall~\eqref{eqn::neumann_green_half_plane}).  For each $t \geq 0$, we let
\begin{equation}
\label{eqn::new_h_t}
 \wt{h}^t = \wh{h} \circ \wt{f}_t  + \frac{2}{\gamma} \log|\wt{f}_t(\cdot)| + \frac{\wt{\rho}}{2\gamma}G(\wt{f}_t(\cdot),\wt{Z}_t) + Q\log|\wt{f}_t'(\cdot)|.
\end{equation}
By Theorem~\ref{thm::reverse_coupling}, we have that $\wt{h}^t \stackrel{d}{=} \wt{h}^0$ for all $t \geq 0$.

For each $r > 0$, let $\wt{\tau}_r = \inf\{t > 0 : \im(\wt{Z}_t) = r\}$.  By Proposition~\ref{prop::zip_into_interior}, we know that we can also sample from the law of $\wt{f}_{\wt{\tau}_r}$ by first sampling $\wt{Z}_{\wt{\tau}_r}$, then drawing an independent (forward) $\SLE_\kappa(\breve{\rho})$ process $\wt{\eta}_{\wt{\tau}_r}$ from $0$ to $\infty$ in $\h$ with a single interior force point located at $\wt{Z}_{\wt{\tau}_r}$ of weight
\[ \breve{\rho} = \wt{\rho} - 8 = 2\alpha \gamma - 8\]
and then take $\wt{f}_{\wt{\tau}_r}$ to be the unique conformal map $\h \to \h \setminus \wt{\eta}_{\wt{\tau}_r}$ with $\wt{f}_{\wt{\tau}_r}(z) = z(1+o(1))$ as $z \to \infty$.  Note that
\[ \rho = 2+\kappa-2\alpha \gamma = \kappa-6-\breve{\rho}.\]
The above is thus in turn the same as drawing an independent radial $\SLE_\kappa(\rho)$ process~$\wt{\eta}_{\wt{\tau}_r}$ with a single boundary force point of weight~$\rho$ and targeted at~$\wt{Z}_{\wt{\tau}_r}$ (see \cite[Theorem~3]{sw2005sle_coordinate_changes}). 

Let 
\[ \ol{h}^{\wt{\tau}_r} = \wh{h}(\cdot + \wt{Z}_{\wt{\tau}_r}) + \frac{2}{\gamma} \log|\cdot + \wt{Z}_{\wt{\tau}_r}|+ \frac{\wt{\rho}}{2\gamma} G(\cdot+\wt{Z}_{\wt{\tau}_r},\wt{Z}_{\wt{\tau}_r}),\]
$\ol{\eta}_{\wt{\tau}_r} = \wt{\eta}_{\wt{\tau}_r} - \wt{Z}_{\wt{\tau}_r}$, and $\ol{f}_{\wt{\tau}_r} = \wt{f}_{\wt{\tau}_r} - \wt{Z}_{\wt{\tau}_r}$.  Then the law of the triple $(\ol{h}^{\wt{\tau}_r},\ol{\eta}_{\wt{\tau}_r},\ol{f}_{\wt{\tau}_r})$ converges as $r \to \infty$ to the law of the triple $(h,\eta,f)$ where: $h$ is a whole-plane GFF plus $-\alpha \log|\cdot|$, $\eta$ (after reversing time and applying \cite[Theorem~1.20]{ms2013imag4}) is a whole-plane $\SLE_\kappa(\rho)$ process from $0$ to $\infty$ independent of~$h$, and $f \colon \h \to \C \setminus \eta$ is a conformal transformation fixing $0$ and $\infty$.

Let $\acute{h} = h \circ f + Q\log|f'(\cdot)|$.  Then we know that $\acute{h} \stackrel{d}{=} \wt{h}^0$ where $\wt{h}^0$ is as in~\eqref{eqn::new_h_t}.  We fix the additive constant for $h$ (hence also $\acute{h}$) by setting its average on $\partial \D$ to be equal to $0$.  For each $x < 0$, we let $R(x) > 0$ be the unique point identified with~$x$ under~$f$.  Let $\nu_{\acute{h}}$ be the $\gamma$-LQG boundary measure associated with $\acute{h}$.  By applying absolute continuity for $\SLE_\kappa(\rho)$ processes and the GFF and combining with \cite[Theorem~1.3]{she2010zipper}, we know that the quantum lengths of $[x,0]$ and $[0,R(x)]$ under $\nu_{\wh{h}}$ a.s.\ agree.

Rescaling $(\C,h,0,\infty)$ as in part~\eqref{it::quantum_cone_limit} of Proposition~\ref{prop::quantum_cone_properties} yields an $\alpha$-quantum cone.  The law of~$\eta$ is not affected by the rescaling procedure since it is independent of~$h$.  Let $\CC = (\C,\mathring{h},0,\infty)$ be the resulting quantum cone, let $\mathring{\eta}$ be the resulting path, and let $\mathring{f} \colon \h \to \C \setminus \mathring{\eta}$ be the resulting conformal map.  We claim that $\mathring{h} \circ \mathring{f} + Q \log| \mathring{f}'|$ is a $(2\alpha-\tfrac{2}{\gamma})$-quantum wedge.  We assume that $\mathring{h}$ is such that the embedding of~$\CC$ is exactly as in Definition~\ref{def::quantum_cone}.  Then the restriction of~$\mathring{h}$ to~$\D$ is equal in law to that of a whole-plane GFF in $\C$ plus $-\alpha \log|\cdot|$ with the additive constant fixed so that its average is equal to $0$ on $\partial \D$.  From the discussion above, it is clear that the law of $\mathring{h} \circ \mathring{f} + Q \log| \mathring{f}'|$ in a sufficiently small neighborhood of $0$ (of random size) is the same as what one gets by starting with a free boundary GFF on $\h$, adding $-(2\alpha-\tfrac{2}{\gamma}) \log|\cdot|$, then restricting to the same small neighborhood.  Moreover, the law of the quantum surface described by $\mathring{h} \circ \mathring{f} + Q \log| \mathring{f}'|$ is invariant under the operation of multiplying its area by a constant because it is given as the image of a surface which has this property.  Thus the claim follows by applying Proposition~\ref{prop::quantum_wedge_characterization}.
\end{proof}

\subsection{Welding more general wedges: proof of Theorem~\ref{thm::welding}}
\label{subsec::zipping_general_wedges}

In this section, we are going to complete the proof of Theorem~\ref{thm::welding}.  The following are the main steps:
\begin{enumerate}
\item Show that dividing a weight $2+W$ wedge by an independent $\SLE_\kappa(0;W-2)$ process for $W > 0$ yields independent wedges of weight $2$ and $W$.  Proposition~\ref{prop::high_weight_welding} is the case that $W \geq \tfrac{\gamma^2}{2}$ and Lemma~\ref{lem::slice_2_plus_w} generalizes Proposition~\ref{prop::high_weight_welding} to include the case that $W \in (0,\tfrac{\gamma^2}{2})$.
\item Show that dividing a weight $2$ wedge by an independent $\SLE_\kappa(-W;W-2)$ process with $W \in (0,2)$ yields independent wedges of weight~$2-W$ and~$W$ (Lemma~\ref{lem::define_skinny_wedge}).
\item Show how to weld together wedges with weights $W_1,W_2 > 0$ with $W_1 + W_2 \in (0,2]$ and that the result is a wedge with weight $W_1 + W_2$ (Lemma~\ref{lem::slice_weight_2}).
\item Generalize to the setting of welding together wedges with weights $W_1,W_2 > 0$ and that the result is a wedge with weight $W_1 + W_2$.  Upon showing this, the proof of Theorem~\ref{thm::welding} will be complete.
\end{enumerate}

\begin{figure}
\begin{center}
\includegraphics[scale=0.85]{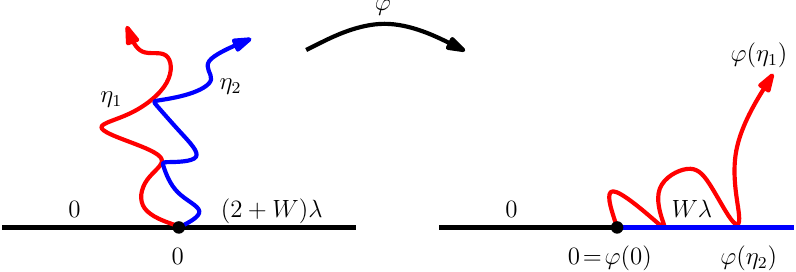}
\end{center}
\caption{\label{fig::slice_2_plus_w} Illustration of the setup of the proof of Lemma~\ref{lem::slice_2_plus_w}.  On the left, we have a wedge $\CW$ of weight $4+W$ and drawn on top of $\CW$ are flow lines of a common GFF on $\h$ which is independent of $\CW$ with the boundary data shown.  The angle of $\eta_1$ (resp.\ $\eta_2$) is $-\lambda/\chi$ (resp.\ $-(1+W)\lambda/\chi$).  The conformal map $\varphi$ takes the component of $\h \setminus \eta_2$ which is to the left of $\eta_2$ to $\h$ with $0$ and $\infty$ fixed.  The resulting surface is, by Proposition~\ref{prop::high_weight_welding}, a weight $2+W$ wedge which is sliced by $\varphi(\eta_1)$, an independent $\SLE_\kappa(0;W-2)$ process.}
\end{figure}

\begin{lemma}
\label{lem::slice_2_plus_w}
Fix $W > 0$ and suppose that $\CW = (\h,h,0,\infty)$ is a wedge of weight $2+W$.  Let~$\eta$ be an~$\SLE_\kappa(0;W-2)$ process in $\h$ from $0$ to $\infty$ with a single boundary force point of weight $W-2$ located at $0^+$ independent of $\CW$.  Then the quantum surfaces parameterized by the regions which are to the left and right of $\eta$ are independent of each other and are respectively given by wedges of weight $2$ and $W$.
\end{lemma}
In the statement of Lemma~\ref{lem::slice_2_plus_w}, the quantum surface to the right of $\eta$ is not connected in the case that $W \in (0,\tfrac{\gamma^2}{2})$ so that the $\rho$ value associated with the force point of $\eta$ located at $0^+$ is in $(0,\tfrac{\kappa}{2}-2)$: it is a countable collection of beads that comes with a natural ordering (the order in which the boundaries are drawn by $\eta$), and a marked pair of boundary points for each set (corresponding to the first and last points on the boundary hit by $\eta$).  For the moment, we will take this as our definition of a wedge of weight $W \in (0,\tfrac{\gamma^2}{2})$ and show momentarily that this definition is equivalent to the one given in Definition~\ref{def::quantum_wedge}.
\begin{proof}[Proof of Lemma~\ref{lem::slice_2_plus_w}]
See Figure~\ref{fig::slice_2_plus_w} for an illustration of the setup.  Proposition~\ref{prop::high_weight_welding} gives the result in the case that $W \geq \tfrac{\gamma^2}{2}$, so we will assume that $W \in (0,\tfrac{\gamma^2}{2})$.  We will prove the result by considering the following modified setup.  Let $\CW = (\h,h,0,\infty)$ be a wedge of weight $4+W$.  Let $\eta_1$ and $\eta_2$ be flow lines of a common GFF on $\h$ with boundary conditions given by $0$ (resp.\ $(2+W)\lambda$) on $\R_-$ (resp.\ $\R_+$) with respective angles given by $-\lambda/\chi$ and $-(1+W)\lambda/\chi$ starting from $0$ and targeted at $\infty$.  Let $\CW_1$ (resp.\ $\CW_3$) be the quantum surface parameterized by the region which is to the left (resp.\ right) of $\eta_1$ (resp.\ $\eta_2$) and let $\CW_2$ be the quantum surface parameterized by the region in between $\eta_1$ and $\eta_2$.  Note that $\eta_1$ (resp.\ $\eta_2$) is marginally an $\SLE_\kappa(0;W)$ (resp.\ $\SLE_\kappa(W;0)$) process independent of $\CW$.  Moreover, the conditional law of $\eta_1$ given $\eta_2$ is that of an $\SLE_\kappa(0;W-2)$ and the conditional law of $\eta_2$ given $\eta_1$ is an $\SLE_\kappa(W-2;0)$ process.  Proposition~\ref{prop::high_weight_welding} implies that $(\CW_1, \CW_2)$ is independent of $\CW_3$ and $\CW_1$ is independent of $(\CW_2, \CW_3)$.  This implies that $\CW_1$, $\CW_2$, and $\CW_3$ are all independent of each other.  Moreover, the quantum surface parameterized by the region which is to the left of $\eta_2$ is a wedge of weight $2+W$ sliced by an $\SLE_\kappa(0;W-2)$ process, which proves the result.
\end{proof}

\begin{lemma}
\label{lem::define_skinny_wedge}
Suppose that $\CW = (\h,h,0,\infty)$ is a wedge of weight $2$, fix $W \in (0,2)$, and suppose that $\eta$ is an $\SLE_\kappa(-W;W-2)$ process in $\h$ from $0$ to $\infty$ with force points located at $0^-,0^+$ and independent of $\CW$.  Then the quantum surfaces parameterized by the regions which are to the left and right of $\eta$ are independent of each other and are respectively given by wedges of weight $2-W$ and $W$.
\end{lemma}
\begin{proof}
We consider the following modified setup.  We suppose that $\CW = (\h,h,0,\infty)$ is a wedge of weight $2+W$.  We let $\eta_1$ and $\eta_2$ be flow lines of a common GFF on $\h$ independent of $\CW$ with boundary conditions given by $0$ on $\R_-$ and $W\lambda$ on $\R_+$ with respective angles given by $(1-W)\lambda/\chi$ and $-\lambda/\chi$.  Then $\eta_1$ is marginally an $\SLE_\kappa(W-2;0)$ process and $\eta_2$ is marginally an $\SLE_\kappa(0;W-2)$ process.  Moreover, the conditional law of $\eta_2$ given $\eta_1$ is that of an $\SLE_\kappa(-W;W-2)$ process and the conditional law of $\eta_1$ given $\eta_2$ is that of an $\SLE_\kappa(W-2;-W)$ process.  Let $\CW_1$ (resp.\ $\CW_3$) be the quantum surface parameterized by the region which is to the left (resp.\ right) of $\eta_1$ (resp.\ $\eta_2$) and $\CW_2$ be the quantum surface parameterized by the region which is in between $\eta_1$ and $\eta_2$.  Lemma~\ref{lem::slice_2_plus_w} implies that $(\CW_1, \CW_2)$ is independent of $\CW_3$ and that $\CW_1$ is independent of $(\CW_2, \CW_3)$.  This implies that $\CW_1$, $\CW_2$, and $\CW_3$ are all independent of each other.  Moreover, the quantum surface parameterized by the region which is to the right of $\eta_1$ is a wedge of weight $2$ sliced by an independent $\SLE_\kappa(-W;W-2)$ process.  This proves the result since the quantum surface parameterized by the region to the right of $\eta_2$ is a wedge of weight $W$ (and symmetry with $W$ replaced by $2-W$ implies that the surface to the left of $\eta$ implies that it is a weight $2-W$ wedge).
\end{proof}

\begin{figure}
\begin{center}
\includegraphics[scale=0.85]{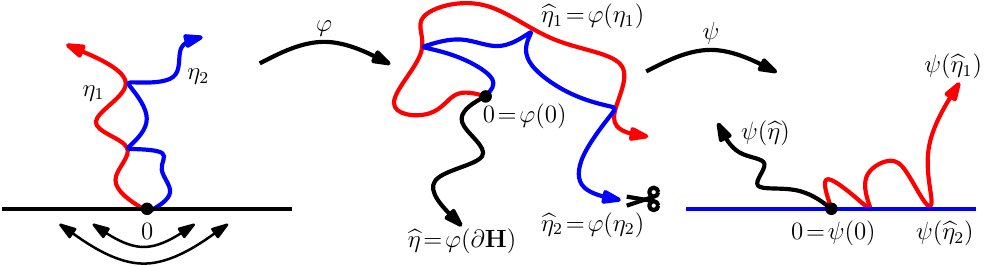}
\end{center}
\caption{\label{fig::zip_unzip} Illustration of the proof of Lemma~\ref{lem::skinny_well_defined}.  On the left, we have a weight~$2$ wedge~$\CW = (\h,h,0,\infty)$ with two paths $\eta_1$, $\eta_2$ (GFF flow lines) drawn on top of it which are independent of $\CW$.  To go from the picture on the left to the picture on the right, we zip up the left and right sides of $\CW$.  The zipping up map is $\varphi$ and the resulting object is, by Proposition~\ref{prop::cone_divide}, a quantum cone decorated with three paths which are independent of the cone and whose joint law can be described in terms of whole-plane GFF flow lines.  We then unzip along $\wh{\eta}_2 = \varphi(\eta_2)$ to yield, again by Proposition~\ref{prop::cone_divide}, a weight~$2$ wedge~$\wh{\CW}$ decorated with two independent paths whose joint law can be described in terms of GFF flow lines.  In particular, $(\psi(\wh{\eta}_1),\wh{\CW})$ is as in the statement of Lemma~\ref{lem::define_skinny_wedge}.}
\end{figure}

The following shows that there is a family of equivalent ways to construct a wedge of weight $W \in (0,\tfrac{\gamma^2}{2})$ as a slice of a weight $2$ wedge.

\begin{lemma}
\label{lem::skinny_well_defined}
Fix $W_1,W_2 > 0$ with $W_1 + W_2 \in (2-\tfrac{\gamma^2}{2},2)$.  Suppose that~$\eta_1$ and~$\eta_2$ are given by the flow lines starting from~$0$ and targeted at~$\infty$ of a GFF on~$\h$ with zero boundary values and with angles respectively given by $(1-W_1)\lambda/\chi$ and $-(1-W_2)\lambda/\chi$.  (Note that the marginal law of $\eta_1$ (resp.\ $\eta_2$) is given by that of an $\SLE_{\kappa}(W_1-2; -W_1)$ (resp.\  $\SLE_{\kappa}(-W_2; W_2-2)$ process in $\h$ from $0$ to $\infty$ with force points located at $0^-,0^+$).  Let $\CW = (\h,h,0,\infty)$ be a wedge of weight $2$ which is independent of $(\eta_1,\eta_2)$.   Then the beaded quantum surface bounded between $\eta_1$ and $\eta_2$ has the law of a wedge of weight $W = 2-(W_1+W_2)$.
\end{lemma}
\begin{proof}
See Figure~\ref{fig::zip_unzip} for an illustration of the proof.  We are first going to apply Proposition~\ref{prop::cone_divide} to zip up $\CW$ to yield a cone of weight $2$ and an independent whole-plane $\SLE_\kappa$ path $\wh{\eta}$.

Let $\varphi \colon \h \to \C \setminus \wh{\eta}$ be the conformal transformation associated with the zipping up process.  By \cite[Theorem~1.4]{ms2013imag4}, we can view~$\wh{\eta}$ as the flow line starting from~$0$ of a whole-plane GFF plus $-(\tfrac{\gamma}{2} - \tfrac{1}{\gamma})\arg(\cdot)$.  Moreover, we can realize~$\wh{\eta}_1 = \varphi(\eta_1)$ and~$\wh{\eta}_2 = \varphi(\eta_2)$ as flow lines of the same GFF starting from $0$ with angles respectively given by $(1-W_1)\lambda/\chi$ and $-(1-W_2)\lambda/\chi$.  This, in turn, gives us the marginal law of~$\wh{\eta}_2$ along with the conditional laws of each of~$\wh{\eta}$ and~$\wh{\eta}_1$ given the other two paths; see \cite[Theorem~1.11]{ms2013imag4}.  In particular, $\wh{\eta}_2$ is a whole-plane $\SLE_\kappa$ process.  Thus we can apply Proposition~\ref{prop::cone_divide} a second time to unzip along~$\wh{\eta}_2$ to yield a weight~$2$ wedge~$\wh{\CW}$ decorated with two independent paths.  Let $\psi \colon \C \setminus \wh{\eta}_2 \to \h$ be the corresponding unzipping map.  Then we know the joint law of $(\psi(\wh{\eta}),\psi(\wh{\eta}_1))$ can be sampled from by taking the paths to be flow lines starting from~$0$ and targeted at~$\infty$ of a GFF on $\h$ with zero boundary values and with respective angles $-(W_2-1)\lambda/\chi$ and $(W-1)\lambda/\chi$.  In particular, $\psi(\wh{\eta}_1)$ is an $\SLE_\kappa(-W;W-2)$ process and the pair $(\psi(\wh{\eta}_1),\wh{\CW})$ is as in the setting of Lemma~\ref{lem::define_skinny_wedge}.  We therefore know that the beaded surface to the right of $\psi(\wh{\eta}_1)$ is a weight $W$ wedge.  This completes the proof since this is the same (as quantum surfaces) as the region of~$\CW$ between~$\eta_1$ and~$\eta_2$.
\end{proof}

In Lemma~\ref{lem::define_skinny_wedge}, we showed that slicing a wedge of weight $2$ with an $\SLE_\kappa(-W;W-2)$ process yields a wedge of weight $W$.  We now generalize this to the case of wedges with weight in $(0,2]$.

\begin{lemma}
\label{lem::slice_weight_2}
Fix $W \in (0,2]$ and let $\CW = (\h,h,0,\infty)$ be a wedge of weight $W$.  Fix $W_1,W_2 > 0$ with $W_1 + W_2 = W$ and let $\eta$ be an $\SLE_\kappa(W_1-2;W_2-2)$ process in~$\h$ from~$0$ to~$\infty$ with force points located at $0^-,0^+$ which is independent of~$\CW$.  (In the case that $W \in (0,\tfrac{\gamma^2}{2})$, we have an independent $\SLE_\kappa(W_1-2;W_2-2)$ process in each of the beads of $\CW$.)  Then the quantum surface $\CW_1$ (resp.\ $\CW_2$) parameterized by the region which is to the left (resp.\ right) of $\eta$ is a weight $W_1$ (resp.\ $W_2$) wedge and $\CW_1, \CW_2$ are independent.
\end{lemma}
\begin{proof}
We consider the following modified setup.  We suppose that $\CW$ is a wedge of weight $2$.  We let $\eta_1$ and $\eta_2$ be flow lines of a GFF on $\h$ with zero boundary conditions with respective angles $(W-1)\lambda/\chi$ and $(W_2-1)\lambda/\chi$.  Then the marginal law of $\eta_1$ is that of an $\SLE_\kappa(-W;W-2)$ process and the marginal law of $\eta_2$ is that of an $\SLE_\kappa(-W_2;W_2-2)$ process, both in $\h$ from $0$ to $\infty$ with force points located at $0^-,0^+$.  Moreover, the conditional law of $\eta_2$ given $\eta_1$ is that of an $\SLE_\kappa(W_1-2;W_2-2)$ process in each of the components of $\h \setminus \eta_1$ which are to the right of $\eta_1$.  Let $\CW_1$ be the quantum surface parameterized by the region which is to the left of $\eta_1$, $\CW_2$ be the quantum surface parameterized by the region between $\eta_1$ and $\eta_2$, and $\CW_3$ be the quantum surface parameterized by the region to the right of $\eta_2$.  We also let~$\wh{\CW}$ be the quantum surface parameterized by the region which is to the right of $\eta_1$.  Lemma~\ref{lem::define_skinny_wedge} implies that~$\CW_1$ is a wedge of weight $2-W$, $\wh{\CW}$ is a wedge of weight $W$, and $\CW_3$ is a wedge of weight~$W_2$.  Lemma~\ref{lem::skinny_well_defined} implies that $\CW_2$ is a wedge of weight~$W_1$.  The same argument as in the end of the proofs of Lemma~\ref{lem::slice_2_plus_w} and Lemma~\ref{lem::define_skinny_wedge} (with Lemma~\ref{lem::define_skinny_wedge} and Lemma~\ref{lem::skinny_well_defined} in place of Proposition~\ref{prop::high_weight_welding}) implies that $\CW_2$ and $\CW_3$ are independent, which implies the result.
\end{proof}

\begin{lemma}
\label{lem::any_to_weight_2j}
Fix $W > 0$, $j \in \N$, and let $\CW = (\h,h,0,\infty)$ be a wedge of weight $2 j+W$.  Let $\eta$ be an $\SLE_\kappa(2j-2;W-2)$ process in $\h$ from $0$ to $\infty$ with force points located at $0^-,0^+$ independent of $\CW$.  Let $\CW_1$ (resp.\ $\CW_2$) be the quantum surface parameterized by the region which is to the left (resp.\ right) of $\eta$.  Then $\CW_1$ (resp.\ $\CW_2$) is a wedge of weight $2j$ (resp.\ $W$) and $\CW_1$ and $\CW_2$ are independent.
\end{lemma}
\begin{proof}
We will prove the result by induction on~$j$.  The case that $j=1$ is given in Lemma~\ref{lem::slice_2_plus_w}.  Suppose that the result holds for some $k \geq 1$; we will now show that the result holds for $j=k+1$.  Note that we can sample~$\eta$ as a flow line starting from~$0$ and targeted at~$\infty$ of a GFF on~$\h$ independent of~$\CW$ with boundary values given by $0$ on $\R_-$ and $(W + 2 k)\lambda$ on~$\R_+$ where the angle of $\eta$ is given by $-(2j-1)\lambda/\chi$.  Let $\wh{\eta}$ be the flow line of the same GFF starting from $0$ and targeted at $\infty$ with angle $-\lambda/\chi$.  Then $\wh{\eta}$ is marginally an $\SLE_\kappa(0;W+2k-2)$ process.  It thus follows from Lemma~\ref{lem::slice_2_plus_w} that the quantum surface parameterized by the region $\wh{\CW}_1$ (resp.\ $\wh{\CW}_2$) which is to the left (resp.\ right) of $\wh{\eta}$ is a wedge of weight $2$ (resp.\ $W+2k$) and $\wh{\CW}_1$, $\wh{\CW}_2$ are independent.  Note that the conditional law of $\eta$ given $\wh{\eta}$ is that of an $\SLE_\kappa(2k-2;W-2)$ process.  Therefore it follows from the induction hypothesis that $\eta$ splits $\wh{\CW}_2$ into independent wedges of weight $2k$ and $W$.  Applying Lemma~\ref{lem::slice_2_plus_w}, gluing the former to $\wh{\CW}_1$ yields a wedge of weight $2j$ which proves the result.
\end{proof}

\begin{proof}[Proof of Theorem~\ref{thm::welding}]
Suppose that we have a wedge $\CW = (\h,h,0,\infty)$ of weight $W \geq \tfrac{\gamma^2}{2}$ and that $W_1, W_2 > 0$ with $W_1 + W_2 = W$.  If $W \in (0,2]$, then we know from Lemma~\ref{lem::slice_weight_2} that drawing an independent $\SLE_\kappa(W_1-2;W_2-2)$ process in $\h$ from $0$ to $\infty$ with force points located at $0^-,0^+$ on top of $\CW$ divides it into independent wedges of weight $W_1$ and $W_2$.  Lemma~\ref{lem::any_to_weight_2j} implies that the result holds if $W_1 = 2j$ for $j \in \N$.

To complete the proof, we thus assume that $W_1 > 0$ is not an even integer with $W_1 +W_2 > 2$.  Let $j = \lfloor W_1/2 \rfloor$ and $k = \lfloor W_2/2 \rfloor$ and let $\wh{\eta}_1$, $\eta$, and $\wh{\eta}_2$ be flow lines of a common GFF on $\h$ independent of $\CW$ with boundary values given by~$0$ on~$\R_-$ and $(W-2)\lambda$ on~$\R_+$ with respective angles $-(2j-1)\lambda/\chi$, $(1-W_1)\lambda/\chi$, and $(1-(W-2k))\lambda/\chi$.  Note that $\wh{\eta}_1$, $\eta$, and $\wh{\eta}_2$ are respectively given by $\SLE_\kappa(2j-2;W-2j-2)$, $\SLE_\kappa(W_1-2;W_2-2)$, and $\SLE_\kappa(W-2k-2;2k-2)$ processes in $\h$ from $0$ to $\infty$ with force points located at $0^-,0^+$.  Lemma~\ref{lem::any_to_weight_2j} implies that the quantum surface $\wh{\CW}_1$ (resp.\ $\wh{\CW}_2$) parameterized by the region which is to the left (resp.\ right) of~$\wh{\eta}_1$ (resp.~$\wh{\eta}_2$) is a wedge of weight $2j$ (resp.\ $2k$) and the quantum surface $\wh{\CW}$ parameterized by the region which is between $\wh{\eta}_1$ and $\wh{\eta}_2$ is a wedge of weight $W-2j-2k$.  Moreover, $\wh{\CW}_1$, $\wh{\CW}_2$, and $\wh{\CW}$ are independent.  It suffices to show that $\eta$ divides $\wh{\CW}$ into independent wedges of weight $W_1-2j$ and $W_2-2k$ because then we can apply Lemma~\ref{lem::any_to_weight_2j} to glue the former to $\wh{\CW}_1$ and the latter to $\wh{\CW}_2$ to yield independent wedges of weight $W_1$ and $W_2$, respectively.  If $W-2j-2k \leq 2$, then the claim follows from Lemma~\ref{lem::slice_weight_2}.  If $W-2j-2k > 2$, then note that $W-2j-2k < 4$ by the definition of $j,k$.  In this case, the claim follows by slicing $\wh{\CW}$ with a flow line of angle $-(2j+1)\lambda/\chi$ to yield independent wedges of weight $2$ and $W-2j-2k-2$ using Lemma~\ref{lem::slice_2_plus_w} and then slicing the resulting wedge containing $\eta$ with $\eta$ and applying Lemma~\ref{lem::slice_weight_2}.

To complete the proof, it is left to show that the notion of a wedge with weight in $(0,\tfrac{\gamma^2}{2})$ is the same as that given in Definition~\ref{def::skinny_wedge_bessel}.  We state and prove this result in the next lemma.
\end{proof}

\begin{lemma}
\label{lem::wedge_definitions_equivalent}
The definition of a wedge of weight $W \in (0,\tfrac{\gamma^2}{2})$ given in Lemma~\ref{lem::slice_2_plus_w} is equivalent to the definition given in Definition~\ref{def::skinny_wedge_bessel}.
\end{lemma}
\begin{proof}
Note that if we apply the rescaling procedure from part~\eqref{it::quantum_wedge_limit} of Proposition~\ref{prop::quantum_wedge_properties} to the setting of Theorem~\ref{thm::bubbles_quantum_local_typical} then the result consists of a wedge of weight $\rho+4$ sliced by an independent $\SLE_\kappa(\rho)$ process.  The result above implies that the ordered sequence of surfaces which are to the right of the path together form a wedge of weight $\rho+2$ as defined earlier in this section.  On the other hand, Theorem~\ref{thm::bubbles_quantum_local_typical} implies that this sequence of surfaces has the law of a wedge of weight $\rho+2$ as defined in Definition~\ref{def::skinny_wedge_bessel}.
\end{proof}

\begin{proposition}
\label{prop::slice_wedge_many_times}
Fix $W \geq \tfrac{\gamma^2}{2}$ and let $\CW = (\h,h,0,\infty)$ be a wedge of weight~$W$.  Consider a GFF on $\h$ independent of $\CW$ with boundary conditions given by~$0$ on $\R_-$ and $(W-2)\lambda$ on $\R_+$.  Fix angles $(1-W)\lambda /\chi < \theta_1 < \cdots < \theta_n < \lambda/\chi$ and, for each~$i$, let~$\eta_i$ be the flow line starting from~$0$ and targeted at~$\infty$ with angle~$\theta_i$.  (Each $\eta_i$ is marginally an $\SLE_\kappa(r_i-2;W-2-r_i)$ process in~$\h$ from~$0$ to~$\infty$ with force points at $0^-,0^+$ where $r_i = 1-\theta_i \chi/\lambda$.  We take the convention that $\eta_0 = \R_+$, $\theta_0 = (1-W)\lambda/\chi$, $r_0 = W$, $\eta_{n+1} = \R_-$, $\theta_{n+1} =\lambda/\chi$, and $r_{n+1} = 0$.)  Then for each $i \in \{1,2,\ldots,n+1\}$, the (possibly beaded) quantum surface bounded between~$\eta_{i-1}$ and~$\eta_i$ (which we denote by $\CW_i$) is a quantum wedge of weight $r_{i-1} - r_i = (\theta_i - \theta_{i-1})\chi/\lambda$ and $\CW_1,\ldots,\CW_{n+1}$ are independent.

The result similarly holds if $W \in (0,\tfrac{\gamma^2}{2})$ except for each bubble $\CB$ of $\CW$ we sample an independent GFF (with boundary data as above) and then view $\eta_j$ as the concatenation of paths from the GFFs according to the natural ordering of the bubbles $\CB$.
\end{proposition}
\begin{proof}
Note that if $n=1$, we have that $\eta_1$ is an $\SLE_\kappa(r_1-2;W-2-r_1)$ process independent of $\CW$.  Since the region which is between $\eta_1$ and $\eta_2 = \R_-$ is the same as the region which is to the left of $\eta_1$ and the region which is between $\eta_1$ and $\eta_0 = \R_+$ is the same as the region which is to the right of $\eta_1$, the result thus follows in this case from Theorem~\ref{thm::welding}.  Then general case follows by iterating.
\end{proof}

\subsection{Zipping up a thin wedge}
\label{subsec::zipping_thin_wedge}

We are now restate part of Theorem~\ref{thm::bubbles_quantum_local_typical}, which extends Proposition~\ref{prop::cone_divide} to show that slicing a quantum cone with an independent self-intersecting whole-plane $\SLE_\kappa(\rho)$ process yields a thin wedge.  Equivalently, conformally welding the left and right sides of a thin wedge yields a certain quantum cone decorated with an independent self-intersecting whole-plane $\SLE_\kappa(\rho)$ process.

\begin{proposition}
\label{prop::cone_divide_self_intersecting}
Fix $W \in (0,\tfrac{\gamma^2}{2})$, suppose that $\CC = (\C,h,0,\infty)$ is a quantum cone of weight $W$.  Suppose that $\eta$ is a whole-plane $\SLE_\kappa(W-2)$ process independent of $\CC$ starting from $0$.  (Note that $\rho \in (-2,\tfrac{\kappa}{2}-2)$ so that $\eta$ is self-intersecting.)  Then the beaded surface $(\C \setminus \eta,h,0,\infty)$ is a quantum wedge of weight $W$.
\end{proposition}
\begin{proof}
This is proved in Theorem~\ref{thm::bubbles_quantum_local_typical}.
\end{proof}

We can also slice a quantum cone with a collection of paths coupled together as flow lines of a GFF to yield a collection of independent quantum wedges.  This is the quantum cone version of Proposition~\ref{prop::slice_wedge_many_times}.

\begin{proposition}
\label{prop::slice_cone_many_times}
Fix $W \geq \tfrac{\gamma^2}{2}$, suppose that $\CC = (\C,h,0,\infty)$ is a quantum cone of weight $W$.  Let $\wh{h}$ be given by a whole-plane GFF minus $\tfrac{1}{2\gamma}(W+\gamma^2-4) \arg(\cdot)$ and assume that $\wh{h}$ is independent of $h$.  Fix angles $0 \leq \theta_1 < \cdots < \theta_n < \tfrac{2\pi}{4-\gamma^2} W = \tfrac{\pi}{\gamma \chi} W$ and, for each $j$, let $\eta_j$ be the flow line of $\wh{h}$ starting from $0$ with angle $\theta_j$.  For each $1 \leq j \leq n$, let $\CW_j$ be the quantum surface parameterized by the region between $\eta_{j-1}$ and $\eta_j$ (in the counterclockwise direction) where we take $\eta_0 = \eta_n$.  Then $\CW_1,\ldots,\CW_n$ are independent and $\CW_j$ is a wedge of weight $(\theta_j-\theta_{j-1})\tfrac{\chi}{\lambda}$ where we take $\theta_0 = \theta_n - \tfrac{2\pi}{4-\gamma^2} W$.
\end{proposition}

If we describe our cone in terms of $\alpha$ rather than weight, then argument singularity for the whole-plane GFF is given by minus $(\gamma-\alpha)\arg(\cdot)$ and the range of angles is given by $2\pi(\chi + \gamma-\alpha)/\chi$.

\begin{proof}[Proof of Proposition~\ref{prop::slice_cone_many_times}]
The case that $n=1$ is given in Proposition~\ref{prop::cone_divide} for $W \geq \tfrac{\gamma^2}{2}$ and Proposition~\ref{prop::cone_divide_self_intersecting} gives the case that $W \in (0,\tfrac{\gamma^2}{2})$.  Now suppose that $n \geq 2$.  Then we can describe the conditional law of $\eta_2,\ldots,\eta_n$ given $\eta_1$ as follows.  First suppose that $W \geq \tfrac{\gamma^2}{2}$ so that $\eta$ is a whole-plane $\SLE_\kappa(\rho)$ process with $\rho = W-2 \geq \tfrac{\kappa}{2}-2$.  That is, $\eta_1$ is non-self-intersecting.  Then $\eta_2,\ldots,\eta_n$ are given by the conformal image of flow lines $\wh{\eta}_2,\ldots,\wh{\eta}_n$ of a GFF on $\h$ with boundary conditions 
\[ \lambda(5-\gamma^2-W) - (\theta_1+2\pi)\chi \quad\text{on}\quad \R_- \quad\text{and}\quad -\lambda-\theta_1 \chi \quad\text{on}\quad \R_+.\]
starting from $0$ and targeted at $\infty$ with respective angles given by $\theta_2,\ldots,\theta_n$.  (See \cite[Theorem~1.11]{ms2013imag4} and \cite[Figure~3.25]{ms2013imag4}.)  By adding a constant to shift the boundary data of the field and adjusting the angles appropriately, we can take our GFF to have boundary conditions given by
\[ 0 \quad\text{on}\quad \R_- \quad\text{and}\quad (W-2)\lambda \quad\text{on}\quad \R_+.\]
This puts us into the setting of Proposition~\ref{prop::slice_wedge_many_times} in the case of slicing a wedge of weight $W$ which, by Proposition~\ref{prop::cone_divide}, is exactly the weight of the wedge corresponding to $\C \setminus \eta_1$.  This completes the proof in this case.  The proof is analogous for $W \in (0,\tfrac{\gamma^2}{2})$ except the above describes the conditional law of $\eta_2,\ldots,\eta_n$ in each of the components of $\C \setminus \eta_1$.  Proposition~\ref{prop::slice_wedge_many_times} still applies in this case, which gives the result.
\end{proof}

\subsection{Gluing infinitesimally thin wedges: the Poissonian formulation}
\label{subsec::fan_wedges}

\begin{remark}
\label{rem::fan_structure}
Using the ideas developed in the previous subsections, it should be possible to give a Poissonian description of the bubbles which arise when cutting a quantum wedge or cone along GFF flow lines starting from a given point with a continuum of angles.  This is the so-called {\bf fan} $\fan$ introduced in \cite{ms2012imag1,ms2013imag4}.  Recall from \cite[Proposition~7.33]{ms2012imag1} that the Lebesgue measure of $\fan$ is a.s.\ equal to zero.  In fact, it is shown in \cite{m2016lightcone_dim} that the Hausdorff dimension of $\fan$ is a.s.\ $1+\kappa/8$, the same as the dimension of a single $\SLE$.

A precise statement of what should be true is the following.  Suppose that $\CW = (\h,h,0,\infty)$ is a wedge of weight $W \geq \tfrac{\gamma^2}{2}$.  Let~$\wh{h}$ be a GFF on~$\h$ independent of~$\CW$ with boundary values given by~$0$ on~$\R_-$ and~$(W-2)\lambda$ on~$\R_+$.  Fix a countable, dense set $\CD$ of $[(1-W)\lambda/\chi,\lambda/\chi]$ and let $\fan$ be the closure of the set of points accessible by flow lines of $\wh{h}$ from $0$ to $\infty$ with angles in $\CD$.  Then the collection of quantum surfaces parameterized by the components of $\h \setminus \fan$ can be sampled from as follows:
\begin{enumerate}
\item Sample a \ppp\ $\Lambda$ on $\R_+ \times [(1-W)\lambda/\chi,\lambda/\chi] \times \CE$ with intensity measure given by $du \times d\theta \times \nu_1^\bes$ where $du$ denotes Lebesgue measure on $\R_+$, $d\theta$ denotes the uniform measure on $[(1-W)\lambda/\chi,\lambda/\chi]$, and $\nu_1^\bes$ is the It\^o excursion measure associated with a $\bes^1$ process.
\item Associate with each $(u,\theta,e) \in \Lambda$ a quantum surface on $\strip$ as in Definition~\ref{def::quantum_wedge}.
\end{enumerate}
If $\CC$ is a quantum cone of weight $W$ and $\fan$ is defined in the analogous manner except in terms of flow lines from $0$ to $\infty$ with angles in a countable dense subset of $[0,\tfrac{\pi}{\gamma \chi} W]$ of a whole-plane GFF minus $(W+\gamma^2-4) / (2\gamma) \arg(\cdot)$ then the quantum surfaces parameterized by the components of $\C \setminus \fan$ admit the same description except we replace the interval $[(1-W)\lambda/\chi,\lambda/\chi]$ of angles with $[                                                                                                                                                                                                                                                                                                                                                                                                                                                                                                                                                                                                                                                                                                                                                                                                                                                                                                                                                                                                                                                                                                                                                                                                                                                                                                                                                                                                                                                                                                                                                                                                                                                                                                                                                                                                                                                                                                                                                                                                                                                                                                                                                                                                                                                                                                                                                                                                                                                                                                                                                                                                                                                                                                                      
0,\tfrac{\pi}{\gamma \chi} W]$.
\end{remark}

\section{Space-filling SLE and Brownian motion}
\label{sec::brownian_boundary_length}

\subsection{Boundary length processes and Brownian motion}

Throughout, we let $\kappa \in [2,4)$ and $\kappa'=16/\kappa \in (4,8]$.  Consider a $\gamma$-quantum cone $\CC = (\C,h,0,\infty)$ (i.e., $W=4-\gamma^2$ and $\theta=2\pi$), together with a space-filling $\SLE_{\kappa'}$ process $\eta'$ from $\infty$ to $\infty$ within $\C$ so that $\eta'(0) = 0$.  We assume that $\eta'$ is first sampled independently of $\CC$ and then reparameterized according to $\gamma$-LQG area so that for $s,t \in \R$ with $s < t$ we have that $\mu_h(\eta'([s,t])) = t-s$.  That is, $\eta'$ traces $t$ units of quantum area in $t$ units of time.  Let $L_t$ (resp.\ $R_t$) denote the change in the left (resp.\ right) $\gamma$-LQG boundary length between times $0$ and $t$.  Note that $L_0 = R_0 = 0$.  The main result of this section is that $(L,R)$ evolves as a certain $2$-dimensional Brownian motion.

\begin{theorem}
\label{thm::bm_statement}
Fix $\kappa \in [2,4)$ and let 
\begin{equation}
\label{eqn::theta_kappa_relationship}
\theta_\kappa = \frac{\pi \kappa}{4}.
\end{equation}
There exists (non-random) $a > 0$ such that $(L,R)$ evolves as a $2$-dimensional Brownian motion with $\var(L_1) = \var(R_1) = a^2$ and $\cov(L_1,R_1) = -a^2 \cos(\theta_\kappa)$.
\end{theorem}

Note that $\cos(\theta_\kappa) < 0$ for $\kappa \in (2,4)$ so that $\cov(L_1,R_1) > 0$.  In the case that $\kappa=2$, $\cos(\theta_\kappa) = 0$ so that $L$ and $R$ are independent.

Before we give the proof of Theorem~\ref{thm::bm_statement}, we will relate $(L,R)$ to the scaling limits of the discrete random planar map models considered in \cite{sheffield2011qg_inventory}.  We first note that it follows from Theorem~\ref{thm::bm_statement} that $L+R$ and $L - R$ are independent Brownian motions.  Moreover, Theorem~\ref{thm::bm_statement} implies that
\begin{align*}
   \var(L_1 + R_1) = 2a^2(1-\cos(\theta_\kappa)) \quad\text{and} \quad \var(L_1 - R_1) = 2a^2(1+\cos(\theta_\kappa)).
\end{align*}
By a (non-random) linear reparameterization of time, we can take $a = (2(1-\cos(\theta_\kappa)))^{-1/2}$ so that $\var(L_1 + R_1) = 1$.  Then, 
\begin{equation}
\label{eqn::l1_r1_var}
\var(L_1 - R_1) = \frac{1+\cos(\theta_\kappa)}{1-\cos(\theta_\kappa)}.
\end{equation} 
Let
\begin{equation}
\label{eqn::p_in_terms_of_theta}
p = -\frac{\cos(\theta_\kappa)}{1-\cos(\theta_\kappa)}
\end{equation} 
and note from~\eqref{eqn::l1_r1_var} that with this choice we have
\begin{equation}
\label{eqn::l1_r1_var_in_p}
\var(L_1 - R_1) = 1-2p.
\end{equation} 

The first co-author \cite{MR1964687}, following work by den Nijs \cite{PhysRevB.27.1674}, Nienhuis \cite{Nien84,DG1987}, his and Saleur \cite{MR947310} (see also Kager and Nienhuis \cite{kg2004guide_to_sle}), has conjectured that the relationship between the parameter $q \in (0,4)$ of the FK representation of the $q$-state Potts model and $\kappa'$ is given by
\begin{equation}
\label{eqn::q_kappa_relationship}
q = 2 + 2 \cos\left(\frac{8\pi}{\kappa'}\right) = 2 + 2 \cos\left(\frac{\pi \kappa}{2}\right).
\end{equation}
Rewriting~\eqref{eqn::q_kappa_relationship} in terms of $\theta_\kappa$ as in~\eqref{eqn::theta_kappa_relationship} yields
\begin{equation}
\label{eqn::q_theta_relationship}
q = 2 + 2 \cos (2 \theta_\kappa) = 4 \cos^2(\theta_\kappa).
\end{equation}
Combining this with~\eqref{eqn::p_in_terms_of_theta} yields
\begin{equation}
\label{eqn::p_q_relationship}
 p = \frac{\sqrt{q}}{2+\sqrt{q}}.
\end{equation}
Summarizing, we have that $(L+R,L-R)$ is a pair of independent Brownian motions with
\[ \var(L_1+R_1) = 1 \quad\text{and}\quad \var(L_1-R_1) = 1 - 2p\]
which matches the scaling limit determined in \cite[Theorem~2.5]{sheffield2011qg_inventory} and the value of $p$ matches the value in the bijection described in \cite[Section~4]{sheffield2011qg_inventory}.

We now turn to the proof of Theorem~\ref{thm::bm_statement}.  The first step (Lemma~\ref{lem::boundary_length_is_brownian_motion}) is to show that $(L,R)$ evolves as \emph{some} $2$-dimensional Brownian motion.  This result will hold for all $\kappa' > 4$.   We will then identify the diffusion matrix in the case that $\kappa \in (2,4)$ so that $\kappa' \in (4,8)$ Lemma~\ref{lem::brownian_covariance} as a consequence of Lemma~\ref{lem::quantum_pinch_times} which gives the a.s.\ Hausdorff dimension of the local cut times for $\eta'$.  We will then identify the diffusion matrix in the case that $\kappa=2$ so that $\kappa'=8$ in Lemma~\ref{lem::diffusionkappa8} using a different method.

\begin{lemma}
\label{lem::boundary_length_is_brownian_motion}
Fix $\kappa \in (0,4)$ so that $\kappa' \in (4,\infty)$.  There exists a pair of independent standard Brownian motions $(X^1, X^2)$ defined for $t \in \R$ with $X_0^1 = X_0^2 = 0$ and a linear transformation $\Lambda$ such that $(L,R) = \Lambda (X^1,X^2)$.
\end{lemma}

Before we prove Lemma~\ref{lem::boundary_length_is_brownian_motion}, we first need the following.

\begin{lemma}
\label{lem::invariant_under_translation}
For each $t \in \R$, the joint law of $(\CC,\eta')$ is invariant under the operation of translating by $t$ units of (quantum area) time and then recentering so that $\eta'(t)$ is translated to the origin.  That is, for each $t \in \R$ we have that
\[ (h,\eta) \stackrel{d}{=} (h(\cdot + \eta'(t)), \eta'(\cdot+t) - \eta'(t))\]
as curve-decorated quantum surfaces.
\end{lemma}
\begin{proof}
Let $\eta'$ be a space-filling $\SLE_{\kappa'}$ process normalized so that $\eta'(0) = 0$.  We assume here (in contrast to the above) that $\eta'$ is parameterized according to Lebesgue measure so that for $s,t \in \R$ with $s < t$ we have that the Lebesgue measure of $\eta'([s,t])$ is equal to $t-s$.  Let $dh$ denote the law of a whole-plane GFF with the additive constant fixed so that its average on $\partial \D$ is equal to~$0$.  For each $t > 0$, let $A_t = \mu_h(\eta'([0,t]))$.  Let $\tau$ be the first time that $\eta'$ exits $\D$.  Note that $A_s$ is a continuous, increasing function of $s$ so that $dA_s$ induces a measure on $[0,\tau]$.  Consider the law $\CZ^{-1} dA_s dh d\eta'$ on $(s,h,\eta')$ triples with $s \in [0,\tau]$ where $\CZ$ is a normalizing constant so that this defines a probability measure.  Under this law, the marginal of $h$ can be sampled from by first sampling $\eta'$ independently as a space-filling $\SLE_{\kappa'}$ normalized so that $\eta'(0) = 0$ and weighted by $\E[ A_\tau \giv \eta']$, then $s$ in $[0,\tau]$ according to Lebesgue measure, and then taking $h = \wh{h} - \gamma\log|\cdot-\eta'(s)| + \gamma \log\max(1,|\cdot|)$ where $\wh{h}$ has the law of a whole-plane GFF independent of $\eta'$ where the additive constant is fixed so that the average of $h$ on $\partial \D$ is equal to~$0$.  Indeed, this follows from the same argument used to prove Lemma~\ref{lem::weighted_quantum_surface}.  The conditional law of $s$ given $h$ and $\eta'$ is given by $A_\tau^{-1} dA_s$.  Equivalently, we first pick $a$ uniformly in $[0,A_\tau]$ and then take $s = \inf\{ u \geq 0: A_u = a\}$.

Now suppose that $s$ is picked in $[0,\tau]$ from $A_\tau^{-1} dA_s$.  Fix $t > 0$ and let $\sigma = \inf\{u \geq s : A_u = A_s + t\}$.  Equivalently, we first pick $a$ uniformly in $[0,A_\tau]$, and then let $\sigma$ be the first time $u$ that $A_u = a+t$.  It follows that the total variation distance between the laws of $a$ and $a+t$ tends to zero as $t \to 0$, hence the same is also true for the laws of $s$ and $\sigma$.  Note that this in particular holds even if we condition on $h$ and $\eta'$ and that $h$ and $\eta'$ together determine $A$.  Therefore, the total variation distance between the laws of $(h(\cdot+\eta'(s)),\eta'(\cdot+s) - \eta'(s))$ and $(h(\cdot+\eta'(\sigma)),\eta'(\cdot+\sigma) - \eta'(\sigma))$ tends to $0$ as $t \to 0$.  It then follows from part~\eqref{it::quantum_cone_limit} of Proposition~\ref{prop::quantum_cone_properties} that the desired invariance holds for quantum cones.
\end{proof}

\begin{proof}[Proof of Lemma~\ref{lem::boundary_length_is_brownian_motion}]
By the definition of space-filling $\SLE_{\kappa'}$ and \cite[Theorem~1.1]{ms2013imag4} we have that the left side $\eta_L$ of the outer boundary of $\eta'([-\infty,0])$ is a whole-plane $\SLE_\kappa(2-\kappa)$ process.  Consequently, it follows from Theorem~\ref{thm::zip_up_wedge_rough_statement} that the surface described by cutting $\CC$ along $\eta_L$ is a quantum wedge of weight $4-\gamma^2$.  Since the relative angle between $\eta_L$ and $\eta_R$ is given by $\pi$, Proposition~\ref{prop::slice_cone_many_times} implies that the (beaded if $\gamma^2 > 2$) quantum surfaces corresponding to $\eta'([0,\infty])$ and $\eta'([-\infty,0])$ are independent quantum wedges, both of weight $2-\tfrac{\gamma^2}{2}$.  Lemma~\ref{lem::invariant_under_translation} states that the joint law of $(\CC,\eta')$ is invariant under the operation of translating time by one unit and recentering at~$\eta'(1)$, hence we also find that $\eta'([1,\infty])$ and $\eta'((-\infty,1])$ parameterize independent quantum wedges of weight $2-\tfrac{\gamma^2}{2}$.

Let $\CS_{a,b}$ denote the (possibly beaded) quantum surface parameterized by $\eta'([a,b])$.  Suppose that $A = \CS_{-\infty, 0}$, $B = \CS_{0,1}$ and $C = \CS_{1,\infty}$.  Since the pair $(A,B)$ is independent of $C$ and the pair $(B,C)$ is independent of $A$, it follows that $A$, $B$, and $C$ are all independent of each other. Repeating this argument, we find that the increments $\CS_{n,n+1}$, for integer $n$, are i.i.d.  A similar argument shows that for any $\epsilon > 0$, the increments $\CS_{n\epsilon, (n+1)\epsilon}$ are i.i.d.

The discussion above implies that both $L_t$ and $R_t$ are infinitely divisible random variables.  Moreover,
\begin{enumerate}
\item by the scale invariance of $\CC$ and the fact that quantum area scales as the square of quantum length, it follows that $L_t$ is equal in law to $\sqrt{t} L_1$, and
\item by the symmetry of the construction, we also have that $L_t$ is equal in law to $L_{-t}$.
\end{enumerate}
It is well-known in the literature on stable processes, that these properties imply that $L_1$ is a Gaussian with mean zero \cite{nolan2003stable}.\footnote{This can be seen directly by observing that equivalence in law of $L_1$ and $-L_1$ imply that the characteristic function $\phi(t) = \E[e^{i L_1}]$ is real; the divisibility and scaling rules imply that $\log \phi(a) = \log \phi(1) a^2$ for rational $a$, and hence $\log \phi$ is a centered parabola, which implies that $L_1$ is Gaussian.}

The same follows for any linear combination of $L_1$ and $R_1$, hence the pair is jointly Gaussian, and by symmetry the joint law is the same if the laws of $L_1$ and $R_1$ are reversed.
\end{proof}

We say that a time $t \in \R$ is a local cut time for $\eta'$ if there exists $\epsilon > 0$ such that $\eta'([t-\epsilon,t))$ is disjoint from $\eta'((t,t+\epsilon])$.  If $t$ is a local cut time for $\eta'$, then we call $\eta'(t)$ a local cut point for $\eta'$.  The following lemma, which gives the a.s.\ Hausdorff dimension of the local cut times of $\eta'$, will be used to identify the diffusion matrix of $(L,R)$.

\begin{lemma}
\label{lem::quantum_pinch_times}
Suppose that $\kappa \in (2,4)$ so that $\kappa' \in (4,8)$.  Then the Hausdorff dimension of the set of local cut times for $\eta'$ (parameterized by quantum area) is a.s.\ $1-\tfrac{2}{\gamma^2}$.
\end{lemma}
\begin{proof}
Let $\eta_L$ (resp.\ $\eta_R$) denote the left (resp.\ right) side of the outer boundary of $\eta'([-\infty,0])$.  Then the points in $\big(\eta_L \cap \eta_R\big) \setminus \{0\}$ are all local cut points for $\eta'$.  More generally, we let $\eta_L^t$ (resp.\ $\eta_R^t$) denote the left (resp.\ right) side of the outer boundary of $\eta'([-\infty,t])$.  Then 
\[ \bigcup_{t \in \Q} \left( \eta_L^t \cap \eta_R^t \right) \setminus \{\eta'(t)\}\]
gives all of the local cut points of $\eta'$.  By the invariance of the setup, the law of the dimension of the amount of time that $\eta'$ spends in $\big(\eta_L^t \cap \eta_R^t\big) \setminus \{\eta'(t)\}$ does not depend on $t$.  Thus to prove the result, it suffices to show that the dimension of the amount of time that $\eta'$ spends in $\big(\eta_L \cap \eta_R\big) \setminus \{0\}$ is a.s.\ $1-\tfrac{2}{\gamma^2}$.

Recall from the proof of Lemma~\ref{lem::boundary_length_is_brownian_motion} that $\eta'([0,\infty))$ parameterizes a wedge of weight $2-\tfrac{\gamma^2}{2}$.  Consequently, the $\gamma$-LQG area of the surface beads in $\eta'([0,\infty))$ between $\eta_L$ and $\eta_R$ can be sampled from by first sampling a Bessel process $\wh{Y}$ of dimension $\tfrac{4}{\gamma^2}$ (recall Proposition~\ref{prop::actual_quantum_area}).  Note that $\{t \geq 0 : \eta'(t) \in (\eta_L \cap \eta_R) \setminus \{0\}\} = \{t \geq 0: \wh{Y}_t = 0\}$.  Let $\wh{\qlt}$ be the local time at $0$ of $\wh{Y}$ and let $\wh{T}$ be the right-continuous inverse of $\wh{\qlt}$.  Note that the latter set is equal to the range of $\wh{T}$.  Recall that the inverse local time of a Bessel process of dimension $\delta$ is a $1-\tfrac{\delta}{2}$ stable subordinator.  By \cite[Chapter~III, Theorem~15]{bertoin96levy}, the a.s.\ Hausdorff dimension of the range of a $1-\tfrac{\delta}{2}$ stable subordinator is given by $1-\tfrac{\delta}{2}$.  Plugging in the value $\delta = \tfrac{4}{\gamma^2}$ gives $1-\tfrac{2}{\gamma^2}$, as desired.
\end{proof}

We are now going to determine the matrix $\Lambda$ from Lemma~\ref{lem::boundary_length_is_brownian_motion}.  Before we do so, we first need to recall the definition of a cone time.  For each $\theta \in [0,2\pi)$, we let $\W_\theta = \{ z \in \C : \arg(z) \in [0,\theta]\}$ be the Euclidean wedge of opening angle $\theta$.  Suppose that $Z$ is a continuous process taking values in $\R^2$.  A time $t_0 \geq 0$ is said to be a {\bf $\theta$-cone time}, $\theta \in (0,2\pi)$, if there exists $\epsilon > 0$ such that $Z([t_0,t_0+\epsilon]) \subseteq \W_\theta + Z_{t_0}$.  We will make use of a result due to Evans \cite{EVANS_CONE_TIMES} which gives that the a.s.\ Hausdorff dimension of the set of $\theta$-cone times for a $2$-dimensional standard Brownian motion (i.e., with covariance matrix given by the identity) is equal to~$0$ if $\theta \in [0,\tfrac{\pi}{2}]$ and given by $1-\tfrac{\pi}{2\theta}$ for $\theta \in (\tfrac{\pi}{2},2\pi)$.  This result will allow us to determine $\Lambda$ from Lemma~\ref{lem::boundary_length_is_brownian_motion}.

\begin{lemma}
\label{lem::brownian_covariance}
For each $\kappa \in (2,4)$ (hence $\kappa' \in (4,8)$) there exists a constant $a > 0$ such that the linear transformation $\Lambda$ of Lemma~\ref{lem::boundary_length_is_brownian_motion} is represented by the matrix
\begin{equation}
\label{eqn::sigma_form}
\Lambda = a \begin{pmatrix} \sin(\theta_\kappa) & -\cos(\theta_\kappa)\\ 0 & 1 \end{pmatrix}.
\end{equation}
In particular, $\cov(L_1,R_1) = -a^2 \cos(\theta_\kappa)$.
\end{lemma}
\begin{proof}
As in the statement of Lemma~\ref{lem::boundary_length_is_brownian_motion}, we write $(L,R) = \Lambda(X^1,X^2)$.  We note that a local cut time as in Lemma~\ref{lem::quantum_pinch_times} corresponds to a $\tfrac{\pi}{2}$-cone time for the process $(L,R)$.  This, in turn, corresponds to a $\theta$-cone time for $(X^1,X^2)$ for some value of~$\theta$.  Applying the main result of \cite{EVANS_CONE_TIMES}, the value of $\theta$ is determined by $1-\tfrac{2}{\kappa} = 1-\tfrac{\pi}{2\theta}$.  That is, $\theta = \theta_\kappa = \tfrac{\pi \kappa}{4}$.  This implies that $\Lambda$ maps the line with angle~$\theta$ (relative to the first coordinate axis) to the second coordinate axis, fixes the first coordinate axis, and applies the same scaling to both of these lines.  Therefore, up to a constant factor, $\Lambda$ is given by
\[ \begin{pmatrix} 1 & \cos(\theta_\kappa)\\ 0 & \sin(\theta_\kappa) \end{pmatrix}^{-1} = \begin{pmatrix} 1 & -\cot(\theta_\kappa)\\ 0 & \csc(\theta_\kappa) \end{pmatrix}.\]
Scaling this matrix by $\sin(\theta_\kappa)$ gives the matrix in the right hand side of~\eqref{eqn::sigma_form}.
\end{proof}

It is left to determine the form of $\Lambda$ from Lemma~\ref{lem::boundary_length_is_brownian_motion} in the case that $\kappa=2$.  We collect the following preliminary lemma in order to do so.

\begin{lemma}
\label{lem::bm_negative_cor_cone}
Suppose that $(X,Y)$ is a Brownian motion with $\var(X_1) = \var(Y_1) = 1$ and $\cov(X_1,Y_1) = \alpha$ for $\alpha < 0$.  For each $k \in \N$, let $E_k$ be the event that $X_k - \inf_{t \in [0,k]} X_t \leq 1$ and $Y_k - \inf_{t \in [0,k]} Y_t \leq 1$.  There exist constants $c > 0$ and $\beta > 1$ such that $\p[E_k] \leq c k^{-\beta}$.  In particular, the number of such $k \in \N$ for which $E_k$ occurs is a.s.\ finite. 
\end{lemma}
\begin{proof}
By a performing a time-reversal, we have that the probability of $E_k$ is at most the probability that $(X,Y)$ starting from $(1,1)$ does not leave $\R_+^2$ in the time interval $[0,k]$.  Fix $\epsilon > 0$ small.  Assume that $X_0 = Y_0 = 1$ and let $F_k$ be the event that $(X,Y)$ hits $\partial B(0,k^{1/2-\epsilon})$ before leaving $\R_+^2$.  It suffices to show that there exist constants $c > 0$ and $\beta > 1$ such that $\p[F_k] \leq c k^{-\beta}$ for each $k \in \N$ since the probability that $(X,Y)$ does not leave $B(0,k^{1/2-\epsilon})$ in $[0,k]$ is at most $c_0 e^{-c_1 k^{2\epsilon}}$ for constants $c_0,c_1 > 0$.  Let
\[ \Sigma = \frac{1}{(1-\alpha^2)^{1/2}} \begin{pmatrix} (1-\alpha^2)^{1/2} & 0 \\ -\alpha & 1 \end{pmatrix}\]
so that $(\wh{X},\wh{Y})^T = \Sigma (X,Y)^T$ is a standard Brownian motion.  Note that $\Sigma$ takes $\R_+^2$ to a Euclidean wedge $W_\theta$ of opening angle $\theta = \arccos(-\alpha) \in (0,\pi/2)$.  Let $\zeta = \pi/\theta > 2$ and fix $a \in \C$ with $|a| = 1$ so that $z \mapsto a z^\zeta$ takes $W_\theta$ to $\h$ and let $(\breve{X},\breve{Y})$ be the image of $(\wh{X},\wh{Y})$ under this map.  Then the probability of $F_k$ is equal to the probability of the event $G_k$ that $(\breve{X},\breve{Y})$ hits $\partial B(0,k^{\zeta(1/2-\epsilon)})$ before leaving $\h$.  This event has probability at most a constant times $k^{-\zeta(1/2-\epsilon)}$, so the first part follows by taking $\epsilon > 0$ sufficiently small so that $\zeta(1/2-\epsilon) > 1$.  The second claim of the lemma follows from the first and the Borel-Cantelli lemma.
\end{proof}

\begin{lemma}
\label{lem::diffusionkappa8}
For $\kappa = 2$ (hence $\kappa' =8$) there exists a constant $a > 0$ such that the linear transformation $\Lambda$ of Lemma~\ref{lem::boundary_length_is_brownian_motion} is represented by $a$ times the identity matrix.  In particular, $\cov(L_1,R_1) = 0$.
\end{lemma}
\begin{proof}
Since $\SLE_8$ does not have local cut points, it must be that $\cov(L_1,R_1) \leq 0$, so it is just a matter of showing that $\cov(L_1,R_1) = 0$.  One can see this because if we take the quantum wedge of weight $2-\tfrac{\gamma^2}{2} = 1$ which is parameterized by $\C \setminus \eta'((-\infty,0])$ and then conformally map it to $\strip$ with $0$ (resp.\ $\infty$) sent to $-\infty$ (resp.\ $+\infty$), then the projection of the corresponding field onto $\CH_1(\strip)$ has zero drift.  In other words, $a = 0$ using the notation of Table~\ref{tab::wedge_parameterization}.  From this perspective, it is clear that $(L,R)$ gets within distance $1$ of hitting a simultaneous running global infimum (relative to time $0$) at arbitrarily large times.  That is, for every $T > 0$ there exists $t \geq T$ such that $L_t - \inf_{s \in [0,t]} L_s \leq 1$ and $R_t - \inf_{s \in [0,t]} R_s \leq 1$.  (We emphasize that it never reaches an actual simultaneous global running infimum relative to time $0$, since this would correspond to a local cut point in the path.)  From this, it follows that $\cov(L_1,R_1) = 0$ as Lemma~\ref{lem::bm_negative_cor_cone} implies that a pair of standard Brownian motions with strictly negative covariance will not get within distance $1$ of hitting a simultaneous running global infimum (relative to time $0$) at arbitrarily large times.
\end{proof}

\begin{proof}[Proof of Theorem~\ref{thm::bm_statement}]
The result follows by combining Lemma~\ref{lem::boundary_length_is_brownian_motion} with Lemma~\ref{lem::brownian_covariance} if $\kappa \in (2,4)$ and with Lemma~\ref{lem::diffusionkappa8} if $\kappa=2$.
\end{proof}

\subsection{Compatibility of topological and conformal matings} \label{subsec::conftopcompatibility}

The construction in Section~\ref{subsec::easy} has an obvious infinite volume analog. To explain this, we first take $L$ and $R$ to be (possibly correlated) real-valued Brownian motions defined for all $t \in \R$ (the additive constant is not important; we may fix it in the standard way by taking $L_0 = 0$ and $R_0 = 0$). Note that the topological space described in Section~\ref{subsec::easy} does not change if we replace $L$ and $R$ by $f(L)$ and $f(R)$, where $f$ is some continuous strictly increasing function from $\R$ to $\R$. In our infinite volume context, we may replace $L$ by $\wt L_t = -e^{-L_t}$ and $R$ by $\wt R_t = -e^{-R_t}$ (note that $x \mapsto -e^{-x}$ is increasing).

The replacement is done because $\wt L$ and $\wt R$ both have finite upper bounds, which ensures that the graphs of $\wt L$ and $C - \wt R$ can be drawn as disjoint subsets of the plane (which we can identify with the complex plane $\C$), which leads to an infinite volume version of Figure~\ref{fig::lamination}. In this section we define $\cong$ the same way as in Section~\ref{subsec::easy} except that we use the whole plane $\C$ instead of a rectangle, and we use $\wt L$ and $\wt R$ in place of $L$ and $R$.  Note that $\C$ becomes a topological sphere when one adds the point at $\infty$. (This point at $\infty$ is the analog of the equivalence class given by the outer boundary of the rectangle in  Section~\ref{subsec::easy}.)

Thus, the same argument as in Section~\ref{subsec::easy} shows that (in the zero correlation case, where $L$ and $R$ are independent) the quotient $\C / \cong$ is a.s\ topologically a sphere if one includes the point at $\infty$, and hence topologically a plane if one does not include $\infty$.  For non-zero values of the correlation coefficient, we have not yet proved that $\C / \cong$ is topologically a plane a.s., but we will deduce this by the end of this subsection.

The previous subsection shows how an SLE-decorated LQG quantum cone determines a coupled pair of Brownian motions $(L, R)$.  Each of these Brownian motions describes an infinite CRT, and the quotient $\C / \cong$ described above can be interpreted as a ``topological mating'' of these trees.  It comes equipped with a space-filling path $\zeta: \R \to \C / \cong$, where $\zeta(t)$ is the equivalence class containing the vertical segment (between the graphs of $\wt L$ and $C-\wt R$) with horizontal coordinate $t$. In this section, we show that there is a map $\Phi : \C / \cong \to \C$ such that $\Phi\bigr( \zeta(t) \bigr) = \eta'(t)$ for all $t$. The existence of such a function is equivalent to the fact that $\zeta(s) = \zeta(t)$ implies $\eta'(s) = \eta'(t)$, which is essentially immediate from the way $L$ and $R$ are defined from $\eta'$ and $h$ (see Lemma~\ref{lem::Phidefined} below).  With a little more work, we will show further that this $\Phi$ is a homeomorphism.  Informally, this means that the conformal mating of trees is homeomorphic to the topological mating of trees in the obvious way. We establish this via several lemmas.

Throughout, we will use the notation $\eta_z^L$ (resp.\ $\eta_z^R$) to denote the flow line with angle $\pi/2$ (resp.\ $-\pi/2$) starting from $z$ of the whole-plane GFF with values modulo $2\pi \chi$ associated with $\eta'$.  Then $\eta_z^L$ (resp.\ $\eta_z^R$) gives the left (resp.\ right) outer boundary of $\eta'$ stopped upon hitting $z$.  We will refer to $\eta_z^L$ in many places in what follows as the flow line starting from $z$ (and omit the angle) and will refer to $\eta_z^R$ as the dual flow line starting from $z$.  Recall from \cite[Theorem~1.7]{ms2013imag4} that flow and dual flow lines cannot cross each other and from \cite[Theorem~1.10]{ms2013imag4} that flow lines starting from distinct points a.s.\ eventually merge with each other and the same is true for dual flow lines starting from distinct points.  Moreover, when two flow (or dual flow) lines have merged with each other they do not subsequently separate.

\begin{lemma}
\label{lem::Phidefined}
In the context discussed just above, there is a.s.\ a unique map $\Phi : \C / \cong \to \C$ such that $\Phi\bigr( \zeta(t) \bigr) = \eta'(t)$ for all $t$. In other words, it is a.s.\ the case that $\zeta(s) = \zeta(t)$ implies $\eta'(s) = \eta'(t)$. 
\end{lemma}
\begin{proof}
It is not hard to see that, starting from the time that $\eta'$ first hits a rational point $z$, all subsequent record minima of $L$ correspond to points along $\eta_z^L$ (which are hit in order).  This is because the change in $L_t$ and $R_t$ corresponds to the change in the length of the boundary of $\eta'\bigl( (-\infty, t] \bigr)$ and this length can only decrease when existing boundary is traversed by $\eta'$ (and hence removed from the boundary).  Indeed, if $s < t$ then the change in $L$ from time $s$ to time $t$ is given by the length of the intersection of the left boundary of $\eta'((-\infty,s])$ and the left boundary of $\eta'((-\infty,t])$.  Since the quantum length of the part of the intersection which corresponds to before these paths have merged is equal to $0$, it follows that the change is equal to the part of the intersection after the paths have merged.  Since flow lines do not separate after merging \cite[Theorem~1.7]{ms2013imag4}, it follows that the boundary length is removed in order.  A similar statement holds if we run time backward starting from $z$. We conclude that it is a.s.\ the case that any two times $s$ and $t$ such that 
\begin{enumerate} \item $s<t$ (so that in particular $\eta'$ hits a rational point in between times $s$ and $t$), and
\item $L_s = L_t$ and $L_r > L_s$ for each $r \in (s,t)$
\end{enumerate}
the times must correspond to opposite sides of a flow line started at a rational point, and we must have $\eta'(s) = \eta'(t)$. One can make the analogous statement for $R$, and this implies the result.
\end{proof}

\begin{lemma}
\label{lem::equivalenceequivalence}
Consider the correlated Brownian motion pair $(L, R)$ discussed in this section and let $\equiv$ be the smallest equivalence relation on $\R$ such that $s \equiv t$ whenever either
\begin{enumerate}
 \item $L_s = L_t$ and each $r \in (s,t)$ satisfies $L_r > L_s$
 \item $R_s = R_t$ and each $r \in (s,t)$ satisfies $R_r > R_s$
\end{enumerate}
In other words, $s \equiv t$ if and only if vertical line segments between the graphs of $\wt L$ and $C- \wt R$ (i.e.\ analogs of the red line segments from Figure~\ref{fig::lamination} in Section~\ref{subsec::easy}) with $x$-coordinates $s$ and $t$ are equivalent under $\cong$.  Then the map taking an equivalence class of $\R / \equiv$ to the corresponding point of $\C / \cong$ is a topological isomorphism (i.e., a homeomorphism).
\end{lemma}
\begin{proof}
The correspondence between elements of $\R / \equiv$ and elements of $\C / \cong$ is clearly one-to-one. The problem is to show that the correspondence is a homeomorphism (i.e., that a set is open on one side if and only if its image on the other side is open).  To this end, first observe that to any open set $A$ of $\C$, which is a union of equivalence classes of the kind shown in Figure~\ref{fig::lamination}, the corresponding subset $\wt A$ of $\R$ contains the $x$-coordinate associated to each of the red vertical lines contained in $A$. Clearly, if $A$ is open then $\wt A$ must also be open.

Conversely, we will now argue that if $A$ is not open then $\wt A$ cannot be open either. If $A$ is {\em not open} then there must be a sequence of points $(z_n)$ which are not in $A$ that has a limit point $a \in A$.

First, suppose that infinitely many of the $z_n$ lie on vertical red lines (which implies that $a$ also lies on a vertical red line, possibly at an endpoint).  Then it is immediate that the $x$-coordinate of $a$ (which belongs to $\wt A$) is a limit point of the $x$-coordinates of the $z_n$, which are not in $\wt A$, which implies that $\wt A$ cannot be open. On the other hand, suppose that each $z_n$ (except for finitely many) {\em does not} lie on a vertical red line, and hence lies in the middle of a horizontal chord of the graph (a green line segment). Consider the sequence $(\wt z_n)$ obtained by replacing each $z_n$ by the rightmost endpoint of the green segment it belongs to. Any limit point $\wt a$ of these $\wt z_n$ must have the same vertical height as $a$, and one easily checks that it must be $\cong$ equivalent to $a$.
But then the $x$-coordinate of $\wt a$ corresponds to a point in $\wt A$ which is a limit point of the $x$-coordinates of the $\wt z_n$, which are not in $\wt A$, which implies that $\wt A$ cannot be open. We conclude that $A$ is open if and only if the corresponding set $\wt A$ is open. 
\end{proof}

\begin{remark} \label{rem::toptrees}
Suppose $\equiv_1$ and $\equiv_2$ are two equivalence relations on a topological space $S$ and $\equiv_3$ is the smallest equivalence containing both.  Then $\equiv_2$ induces an equivalence in $S /\! \equiv_1$ in the obvious way, and it is not hard to see that $(S /\! \equiv_1) / \!\equiv_2$ is homeomorphic to $S /\! \equiv_3$ and hence also to $(S /\!\equiv_2)/\! \equiv_1$ in the obvious way. To apply this observation, suppose that $S$ consists of two disjoint copies of $\R$, and $\equiv_1$ is the smallest equivalence that makes $s$ and $t$ in the first copy equivalent whenever $L_s = L_t$ and each $r \in (s,t)$ satisfies $L_r > L_s$, and $s$ and $t$ in the second copy equivalent whenever $R_s = R_t$ and each $r \in (s,t)$ satisfies $R_r > R_s$. Let $\equiv_2$ be the equivalence that makes $s$ in one copy of $\R$ equivalent to the corresponding $s$ in the other copy. Then note that $(S /\! \equiv_1) / \! \equiv_2$ is a ``topological mating of a correlated pair of CRTs'' while $(S / \! \equiv_2) / \! \equiv_1$ is obviously equivalent to $\R /\! \equiv$. Thus we are justified in calling $R /\! \equiv$ (and hence also $\C /\! \cong$) a topological mating of trees.
\end{remark}

\begin{lemma} \label{lem::transientbounded}
Suppose that $\kappa'>4$ is fixed, and $\eta'$ is a space-filling $\SLE_{\kappa'}$ curve from $\infty$ to $\infty$ (as discussed in this section). Then $\eta'$ is a.s.\ transient in the sense that the $\eta'$ pre-image of any bounded set is also bounded.  Furthermore, it is a.s.\ the case that any point $z$ that $\eta'$ hits more than once lies on a flow line or dual flow line started at a different point.
\end{lemma}

\begin{proof}
The transience of $\eta'$ is established in \cite{ms2013imag4} where the space-filling SLE curves are defined. If $z$ is hit at distinct times $t_1$ and $t_2$ then in particular for any rational $t \in (t_1, t_2)$ (for which $\eta'(t) \not = z$) the point $z$ must be on the boundary between $\eta' \bigl((-\infty, t] \bigr)$ and $\eta' \bigl( [t, \infty) \bigr)$, and hence by definition it lies on the flow line or dual flow line started at $\eta'(t)$.
\end{proof}

\begin{figure}[ht!]
\begin{center}
\includegraphics[scale=0.85]{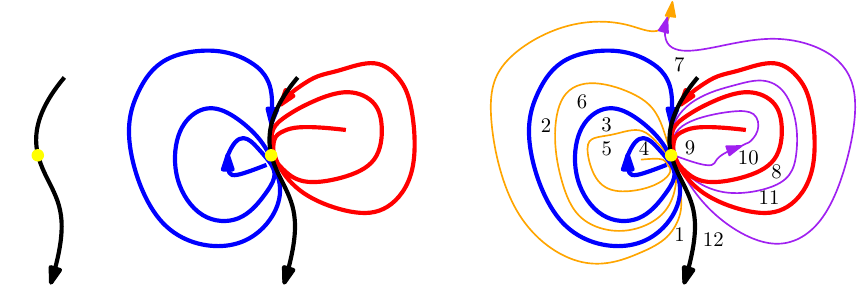}
\end{center}
\caption{\label{fig::doublespiral} {\bf Left:} (sketch of) segment of a flow line passing through point $z$.  {\bf Middle:} flow lines (in blue and red) that spiral around and hit $z$ multiple times before merging into the original flow line. {\bf Right:} dual flow lines (thinner curves) that spiral around (and hit $z$ multiple times, and hit but do not cross the flow lines) before the left and right curves ultimately merging with each other (up top). The numbers indicate the order in which a space-filling curve (one ``tracing the boundary'' of the flow line tree in clockwise order) would visit each of the numbered regions. The red and blue paths may have more intersection points than those shown, but they a.s.\ do not ``cross'' each other.}
\end{figure}

In fact, we will need to go beyond Lemma~\ref{lem::transientbounded} and give a more detailed description of the behaviors that are possible at multiple points of $\eta'$.

Over the next few lemmas, we will aim to show that $s \equiv t$ whenever $\eta'(s) = \eta'(t)$.  The short and informal version of the argument would be ``The results in \cite{ms2013imag4} restrict the ways a point $z$ can be hit multiple times by a space-filling SLE, and these restrictions imply that multiple points (that are not merge points) should look like the point illustrated in Figure~\ref{fig::doublespiral}, and this can be used to show that $s \equiv t$ whenever $\eta'(s) = \eta'(t)$.'' Because this consequence was not explicitly stated in \cite{ms2013imag4}, we will have to remind the reader what was proved in  \cite{ms2013imag4} and work out some simple consequences.

Consider a countable dense set of points in $\C$ --- for example, the set of points with rational coordinates. Each such $z$ is a.s.\ hit by $\eta'$ only once and hence $(\eta')^{-1}(\{z\})$ is a.s.\ a single point $t = t(z)$. From any such point $z$, we can draw a flow line $\eta_z^L$ and a dual flow line $\eta_z^R$, which may hit each other and themselves (recall that interior flow lines are a.s.\ self-intersecting for $\kappa \in (8/3,4)$ and a flow and dual flow starting from the same interior point a.s.\ intersect for $\kappa \in (2,4)$) but which do not cross each other and their liftings to the universal cover of the complement of $z$ are a.s.\ simple curves. These curves form the boundary between $\eta'\bigl( (-\infty, t(z)] \bigr)$ and $\eta' \bigl( [t(z), \infty) \bigr)$.

A {\bf merge point} is a point $z \in \C$ such that, for some two distinct points $x,y \in \C$ with rational coordinates, the curves $\eta_x^L$ and $\eta_y^L$ merge at $z$. 

\begin{lemma} \label{lem::countablemerge}
It is a.s.\ the case that there are only countably many merge points, and that all merge points are triple points of $\eta'$ (i.e., points hit exactly three times) and correspond precisely to the (countably many) local minima and maxima of the Brownian motions $L$ and $R$. Furthermore, for any merge point $z$, the three times in $(\eta')^{-1}(\{z\})$ are equivalent under $\equiv$.
\end{lemma}
\begin{proof}
Since each merge point $z$ corresponds to a merging of a distinct $\eta_x^L$ and $\eta_y^L$ (for $x,y \in \C$ with rational coordinates) it is immediate that there can be only countably many such $z$. Let $t_1$ and $t_2$ be such that $\eta'(t_1) = x$ and $\eta'(t_2) = y$, and assume without loss of generality that $t_1 < t_2$.  (Recall from above that since a.e.\ point in $\C$ is hit by $\eta'$ only once, each rational point a.s.\ $\C$ has a unique $\eta'$-inverse.) If flow lines from $x$ and $y$ merge at $z$ then it is not hard to see that $L$ must obtain a minimum at the time in between $t_1$ and $t_2$ at which $z$ is visited (similarly for $R$ if dual flow lines from $x$ and $y$ merge at $z$), which immediately implies the $\equiv$ equivalence of the three times.

We will now explain why $\eta'$ a.s.\ visits $z$ exactly three times.  First, if we start dual paths $\eta_u^R$, $\eta_v^R$ at points $u,v$ on $\eta_x^L, \eta_y^L$ very close to $z$ then it is very likely that they will merge and form a small diameter pocket $P$ with the merge point $z$ on its boundary.  By the definition of space-filling $\SLE_{\kappa'}$, once $\eta'$ enters this pocket it fills it up entirely before leaving.  Indeed, this follows because two flow lines cannot cross, the same is likewise true for dual flow lines, and flow and dual flow lines cannot cross \cite[Theorem~1.7]{ms2013imag4}.  Moreover, the conditional law of $\eta'$ inside of $P$ is that of a space-filling $\SLE_{\kappa'}(\kappa'/2-4';\kappa'/2-4)$ with force points located at $z$ and the point at which $\eta_u^R$, $\eta_v^R$ merge.  Each boundary point is a.s.\ hit exactly once, and this accounts for one of the three visits.  If we draw in the path $\eta_x^R$, the conditional law of $\eta'$ up until it visits $x$ is an $\SLE_{\kappa'}(\kappa'/2-4;\kappa'/2-4)$ process \cite[Theorem~1.14]{ms2013imag4} in each of the components of $\C \setminus (\eta_x^L \cup \eta_x^R)$ whose boundary consists of part of the right side of $\eta_x^L$ and part of the left side of $\eta_x^R$ (there is in fact only one component if $\kappa' \geq 8$).  If $\kappa' \in (4,8)$, it is a.s.\ the case that $\eta_y^L$ does not merge with $\eta_x^L$ at a point in $\eta_x^L \cap \eta_x^R$, so that $z$ is in the interior of the boundary of one of these pockets.  Therefore $\eta'$ a.s.\ visits $z$ exactly once before hitting $x$.  A similar analysis implies that $\eta'$ a.s.\ visits $z$ exactly once after hitting $x$.
\end{proof}

In addition to these highly special merge points there are uncountably many other points that $\eta'$ hits multiple times, and these other types of multiple points are somewhat more interesting. For example, any point $z$ that lies on a flow line or dual flow line (started from a rational point that is not $z$ itself) must be visited by $\eta'$ both before and after $z$ is visited. But we need a more complete understanding of what can happen at multiple points.

To this end, let us introduce another definition: a {\bf $k$-flow point} is a point $z$ that is not a merge point such that for some rational points $x_1, x_2, \ldots, x_k$ (and a choice of dual flow line or flow line for each $i \in \{1, 2, \ldots, k \}$) the $k$ flow/dual flow lines started from the $x_i$ all reach the point $z$ without merging with each other first.

We recall from \cite{ms2013imag4} that the number of flow and dual flow lines that can meet at a single point (before merging with each other) is at most $n$ if
\begin{equation}
\label{eqn::k_flow_bound}
	\frac{\kappa}{4} \in \Bigl(\frac{n-1}{n}, \frac{n}{n+1}\Bigr].
\end{equation}
For example, when $\kappa$ is between $0$ and $2$, we have $n=1$ (so two flow and dual lines cannot intersect each other).  When $\kappa \in (2, 8/3]$, a flow line and a dual flow line can hit each other but a flow or dual flow line cannot wind around and hit itself (as the $2\pi$ angle gap would be too large to allow such an intersection). In particular, for any $\kappa < 4$, there is a finite upper bound on the number of distinct flow or dual flow lines coinciding at a point, hence a finite upper bound on the $k$ for which $k$-flow points exist.

Suppose $\eta'$ hits $z$ at least $k$ times, and that $z$ is not one of the countably many merge points. Let $t_1, \ldots, t_{k-1}$ be times at which $\eta'$ hits rational points $z_1, \ldots, z_k$, with one such time in each bounded component of $(\eta')^{-1}(z)$. We can draw a flow line and dual flow line from each $z_j$. For any distinct $i, j$, it cannot be the case that both the flow lines from $z_i$ and $z_j$ {\em and} the dual flow lines from $z_i$ and $z_j$ merge {\em before} hitting $z$, since if this happened then $z$ could not be on the boundary of $\eta'([t_i, t_j])$.  In a small neighborhood around $z$, one encounters a finite number of flow line or dual flow line ``strands,'' at most the number $n$ described above.

Now if one follows one of the strands, it either goes off to $\infty$ before returning to $z$ or it merges with one of the other strands after winding around.

\begin{lemma}
\label{lem::multiplepointsandflowpoints}
It is a.s.\ the case that every multiple of point of $\eta'$ is {\em either} a merge point (in which case it is hit by $\eta'$ exactly three times) or a $k$-flow point for some $k \geq 1$ (in which case it is hit by $\eta'$ exactly $k+1$ times).
\end{lemma}

We note that by combining Lemma~\ref{lem::multiplepointsandflowpoints} with~\eqref{eqn::k_flow_bound} gives a deterministic upper bound on the number of times that $\eta'$ can hit any point in $\C$.

\begin{proof}[Proof of Lemma~\ref{lem::multiplepointsandflowpoints}]

We have already established in Lemma~\ref{lem::countablemerge} that merge points are a.s.\ triple points.  We have also shown in Lemma~\ref{lem::transientbounded} that every multiple point is on a flow line or dual flow line path started at another point, which means that it is a $k$-flow point for some $k \geq 1$.  What remains is to show that it is a.s.\ the case that every $k$-flow point for $k \geq 1$ is hit exactly $k+1$ times --- or equivalently, that every point hit exactly $k+1$ times is a $k$-flow point.  We will take mainly the former point of view and proceeding by analyzing a $k$-flow point.

To make the idea of what such points look like more concrete, we sketch a general possible point hit $k+1$ times (with $k = 3$) in Figure~\ref{fig:flower}, and a possible $k$-flow point (with $k = 10$) in Figure~\ref{fig::doublespiral}.  In the latter case, the path $\eta'$ makes $10$ excursions away from $z$ in between the first and last time it hits $a$ (and these excursions are labeled by numbers $2, 3, \ldots, 11$). The initial period (up to first time $z$ is hit) is labeled $1$ and the final period (after $z$ is hit for the last time) labeled $12$.  This implies that the point $z$ is visited $11$ times total (once between filling each pair of consecutively numbered regions).

It is a.s.\ the case that at {\em every} time, the left and right boundary paths, say $\eta^{t,L}$ and $\eta^{t,R}$, of $\eta'((-\infty,0])$ exist starting from $\eta'(t)$.  Any segment of $\eta^{t,L}$ (except possibly at its starting point, $\eta'(t)$) is a segment of a flow line starting from a rational point (since one can take rational points arbitrarily close to $\eta'(t)$ and a flow line starting close to $\eta'(t)$ is very likely to merge with $\eta^{t,L}$ before traveling very far; see \cite[Lemma~3.11]{ms2013imag4} for a similar argument) and the same likewise holds for $\eta^{t,R}$.  This implies that the $\eta^{t,L}$, $\eta^{t,R}$ have all of the properties (at least away from $\eta'(t)$) that a flow or dual flow line starting from a point with rational coordinates has.

Suppose $z$ is a multiple point and let $s$ (resp.\ $t$) be the first (resp.\ last) time that $\eta'$ hits $z$.  Then consider $\eta'(( -\infty, s+\delta]) \cap \eta'([t - \delta, \infty))$ for some very small $\delta > 0$. We claim that (for all sufficiently small $\delta$) this intersection contains exactly one long boundary strand

passing through $z$. To see this, consider a $\delta$ small enough so that $\eta'$ has only hit $z$ once up to time $s+\delta$. (One may assume $\eta'(s+\delta)$ is a rational point, since such points are hit at a dense set of times.) Since $z$ is a multiple point (and must be hit at some subsequent time) so we know that $z$ is on the boundary of $\eta'((-\infty, s+\delta])$ and hence must lie on either (and maybe both) the flow line and dual flow line starting from $\eta'(s+\delta)$, each of which hits $z$ at most once. The time reversal of $\eta'$ swallows points along the flow and dual lines in reverse order, and when it first hits $z$ at time $t$, it must do while either swallowing points along the dual flow line or while swallowing points along the flow line, and in either case one can easily deduce the claim. We refer to the flow line (or dual flow line) strand described in this claim as the {\bf king strand}.  Without loss of generality assume that the king strand is a flow line.  If $z$ is a $k$-flow point, then one can draw other strands which visit $z$, as illustrated in Figure~\ref{fig::doublespiral}.  It is not hard to see that any other such strand must (if one continues to follow it) bounce off~$z$ finitely many times before finally merging with the king strand. Dual strands must do the same except that those that hit the king strand on the left merge with those that hit the king strand on the right at some point {\em away} from~$z$, as illustrated in Figure~\ref{fig::doublespiral}.

Once we have the $k$ strands in place, they divide the region near $z$ into $k+1$ bands, as illustrated in Figure~\ref{fig::doublespiral}.  We need to show that within each such band, $\eta'$ hits $z$ exactly one time.  (In less formal language, if there were multiple hitting times, then their boundaries would be extra ``petals'' that would ``subdivide'' at least one of the bands.)

There are two kinds of bands: the middle bands (which lie between two successive paths) and the two end bands.  The definition of the space-filling ordering implies that, in these middle bands, the bubbles are filled in order (as two flow lines cannot cross and flow and dual flow lines cannot cross \cite[Theorem~1.7]{ms2013imag4}).  Moreover, the conditional law of $\eta'$ in each of these bands (which in fact consist of countably many bubbles if $\kappa' \in (4,8)$) is that of a space-filling $\SLE_{\kappa'}(\kappa'/2-4;\kappa'/2-4)$ process.  In particular, in each of the countably many bubbles the tip is not visited multiple times while the bubble is being filled.

Let us now consider the case of the end bands.  We will without loss of generality consider the leftmost band.  Then the point is to imagine that we stop $\eta'$ at the first time $z$ is visited from the leftmost band --- or rather, at a time $t$ just slightly after that time. Then if the leftmost band has not been completely filled (in a neighborhood of $z$) during that time, then the future/past boundary arcs at time $t$ must go through $z$, and we must either have a merge point at $z$ or a further-left strand, both of which (by assumption) we do not have.
\end{proof}

\begin{lemma} \label{lem::multiplepointsequivalent}
For a space-filling $\SLE_{\kappa'}$ curve $\eta'$ it is furthermore a.s.\ the case that for each point $z$ that the curve hits at multiple times, the times in $(\eta')^{-1}(\{z\})$ are all equivalent under the smallest equivalence relation $\equiv$ such that $s \equiv t$ if one of the following is true:
 \begin{enumerate}
 \item $L_s = L_t$ and each $r \in (s,t)$ satisfies $L_r > L_s$
 \item $R_s = L_t$ and each $r \in (s,t)$ satisfies $R_r > R_s$
 \end{enumerate}
\end{lemma}
\begin{proof}
By Lemma~\ref{lem::multiplepointsandflowpoints}, we may assume that $z$ is a $k$-flow point, and is visited exactly $k+1$ times. As explained in the proof of Lemma~\ref{lem::multiplepointsandflowpoints} (and illustrated in Figure~\ref{fig::doublespiral}), we can divide the region near $z$ into $k+1$ bands and $\eta'$ visits $z$ exactly one time as it is filling each of these bands.  If we order these times from left to right (rather than sequentially) then two adjacent times must be separated by either a flow or a dual flow line and therefore either $L_t$ or $R_t$, respectively, must then take the same value at those two times, and be strictly larger in between those two times.  Thus all of the times are equivalent under $\equiv$ as desired.
\end{proof}

\begin {figure}[htbp]
\begin {center}
\includegraphics [scale=1]{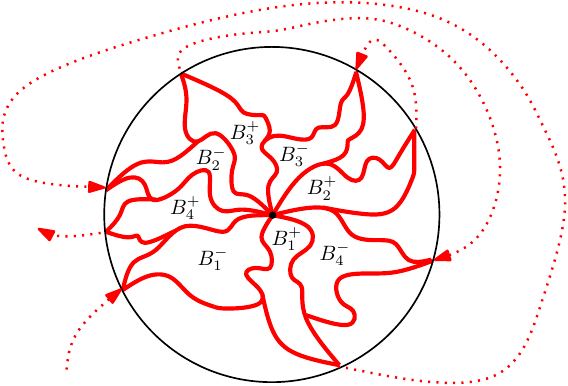}
\caption {\label{fig:flower} Each of the sets $B_j^{\pm}$ represents the region traced by either a segment of $\eta'$ from $\partial B(z,\epsilon)$ to $z$ or a segment of $\eta'$ from $z$ to $\partial B(z,\epsilon)$. Dotted lines are drawn from $B_j^+$ to $B_{j+1}^-$ to illustrate the order in which regions are traversed. Continuity and transitivity of $\eta'$ imply that the union of the $B_j^{\pm}$ regions includes a neighborhood of $z$, in which the eight boundaries between cyclically consecutive regions are given by segments of flow or dual flow lines. Each of these segments (which intersect but not cross one another) is part of the boundary of two neighboring $B_j^{\pm}$ regions; it is ``added'' to the boundary of $\eta' \bigl( (-\infty,t] \bigr)$ as the first of these regions is traversed and ``removed'' from the boundary of $\eta' \bigl( (-\infty,t] \bigr)$ as the second is traversed. }
\end {center}
\end {figure}

\begin{lemma} \label{lem::planeequivalence} If $\eta'$ is the space-filling SLE, then the sets $(\eta')^{-1}(\{z\})$ (for $z$ in the complex plane) are precisely the equivalence classes of $\equiv$.
\end{lemma}
\begin{proof}
This follows from Lemma~\ref{lem::transientbounded} and Lemma~\ref{lem::multiplepointsequivalent}.
\end{proof}

\begin{lemma}
\label{lem::opengivesopen}
Let $\eta'$ be any continuous, transient space-filling continuous function $\R \to \C$.  (By {\em transient} we mean that if $A$ bounded then $(\eta')^{-1}(A)$ must be bounded.)  Then for any $A \subseteq \C$, if $(\eta')^{-1}(A)$ is open, then $A$ is also open.
\end{lemma}
\begin{proof}
We aim to show the converse, i.e., that if $A$ is not open, then $(\eta')^{-1}(A)$ is not open.  Suppose $A$ is not open. Then there must exist $z \in A$ such that there is a sequence $(z_n)$ in $\C \setminus A$ converging to~$z$.   The transience of $\eta'$ implies that the $(\eta')^{-1}(\{z_n\})$ belong to a fixed bounded interval for all large enough~$n$. Thus, by compactness of that interval, we can find a sequence along which the sequence of sets $(\eta')^{-1}(\{z_n\})$ has a limit point $t$. but then by continuity we must have $\eta'(t) = z$. Since $t$ is a point in $(\eta')^{-1}(A)$ that is a limit of points not in $(\eta')^{-1}(A)$, the set $(\eta')^{-1}(A)$ must not be open.
\end{proof}

\begin{lemma}
\label{lem::homeomorphismpath}
Let $\eta'$ be any continuous, transient space-filling continuous function $\R \to \C$. Let $\equiv$ be the equivalence relation on $\R$ such that $s \equiv t$ whenever $\eta'(s) = \eta'(t)$. Let $\wt R$ be the quotient of $\R$ w.r.t.\ $\equiv$.  Then $\eta'$ is a homeomorphism between $\wt R$ and $\C$.
\end{lemma}

\begin{proof}
Given any open $A \subset \C$, the set $(\eta')^{-1}(A)$  (interpreted as a subset of $\wt R$) is clearly open by continuity of $\eta'$ and the definition of the quotient topology.  Conversely, if we have an open set of times that is a union of equivalence classes, then Lemma~\ref{lem::opengivesopen} implies that its image is open as a planar set.
\end{proof}

\begin{theorem}
\label{thm::topologyiscorrect}
The function $\eta'$, interpreted as a map from $\R / \equiv$ to $\C$, is a.s.\ a topological homeomorphism.  The correspondence between $\C / \cong$ and $\C$ is hence a.s.\ also a topological homeomorphism.
\end{theorem}

\begin{proof}
The first assertion is immediate from Lemma~\ref{lem::multiplepointsequivalent} and Lemma~\ref{lem::homeomorphismpath}, and the second follows from Lemma~\ref{lem::equivalenceequivalence}.
\end{proof}

In particular, Theorem~\ref{thm::topologyiscorrect} implies that the ``peanosphere'' given by $\R / \equiv$ (or equivalently the corresponding mating of the correlated pair of CRTs, as given by $\C / \cong$) is a.s.\ homeomorphic to the plane (or sphere if one includes the point at infinity).  Moreover, the map $\eta'$ describes a specific way to map the peanosphere homeomorphically onto the plane. The goal of Section~\ref{sec::trees_determine_embedding} will be to show that (up to rotations, dilations, and translations) this embedding is a.s.\ determined by $L$ and $R$ (or by the corresponding pair of CRTs).

\section{Trees determine embedding}
\label{sec::trees_determine_embedding}

\subsection{Theorem statement and proof using three lemmas}

Let $(\C, h,0,\infty)$ be a $\gamma$-quantum cone, and let~$\eta'$ be an independent space-filling $\SLE_{\kappa'}$ in~$\C$ from~$\infty$ to~$\infty$ with $\kappa' = 16/\gamma^2 \in (4,\infty)$.  We assume that~$\eta'$ is parameterized according to $\gamma$-LQG area so that for each $s < t$ we have $\mu_h(\eta'([s,t])) = t-s$.  It will be convenient for this section to fix the embedding of the quantum cone in $\C$ in a particular way (different from the circle average embedding discussed in Section~\ref{subsec::cones}): namely, by redefining $h$ and $\eta'$ via a quantum coordinate change by a translation and dilation as necessary so that there exists a conformal map
\[ F \colon \C \setminus \ol \D \to \C \setminus \eta'([-1,1]),\]
whose range is the entire unbounded component of $\C \setminus \eta'([-1,1])$ and has the Laurent expansion
\begin{equation}
\label{eqn::Flaurent}
F(z) = z + \sum_{n=1}^\infty \alpha_n z^{-n}.
\end{equation}
Using terms that we will define in Section~\ref{subsec::distortionlemma}, this amounts to scaling the embedding so that $\eta'([-1,1])$ (together with the set of points disconnected from $\infty$ by $\eta'([-1,1])$) has ``harmonic center'' zero and ``outer conformal radius'' one.  In a certain conformal sense (see Section~\ref{subsec::distortionlemma}) this means that $\eta'([-1,1])$ has the same size and center as the unit disk.  Given this embedding in $\C$, we will not expect that the quantum cone marked point $\eta'(0)$ is exactly equal to $0$.  However, the ``infinite'' endpoint of the cone will be $\infty$ as usual.

We have already shown in Section~\ref{sec::brownian_boundary_length} that $(\C, h, \eta')$ determines the Brownian motion pair $(L, R)$ indexed by $t \in \R$.  The aim of this section is to show that $(L, R)$ a.s.\ uniquely determines the triple $(\C, \mu_h, \eta')$ up to a rotation of $\C$.  Combining with \cite{berestycki2014equivalence}, this implies that the path-decorated quantum surface $(\C,h,0,\infty,\eta')$ is a.s.\ determined by $(L,R)$.

Let $\CF$ be the $\sigma$-algebra generated by the quantum surfaces parameterized by $\eta'([1,\infty))$ and by $\eta'((-\infty,-1])$.  Note that these two quantum surfaces are independent quantum wedges of weight $2-\gamma^2/2$.  Now let $\CB$ be the $\sigma$-algebra generated by the restriction of $(L_t, R_t)$ to $[-1,1]$.  We note that $\CF$ and $\CB$ together determine the surface parameterized by the unbounded component of $\C \setminus \eta'([-1,1])$ because $\CB$ determines how the surfaces parameterized by $\eta'((-\infty,-1])$ and $\eta'([1,\infty))$ are glued together.

Proving the following is the main technical challenge of this section.

\begin{theorem}
\label{thm::treesdetermineemebedding}
The conditional law of $F$ (modulo rotation about the origin) given the $\sigma$-algebra generated by {\em both} $\CF$ and $\CB$ is a.s.\ deterministic.
\end{theorem}

Before going on, let us point out that Theorem~\ref{thm::treesdetermineemebedding} has some important consequences.   Obviously, the function $F$ determines the hull of $\eta'([-1,1])$ (i.e., the complement of the unbounded component of $\C \setminus \eta'([-1,1])$) as a set, so Theorem~\ref{thm::treesdetermineemebedding} in some sense implies that the hull of $\eta'([-1,1])$ (modulo rotation about the origin) is determined by the information encoded by $\CF$ and $\CB$. The following is an extension of this statement.

\begin{corollary}
\label{cor::manychunk}
Given the $\sigma$-algebra generated by {\em both} $\CF$ and $\CB$ , the conditional law of the collection of sets $\eta'([j/n, (j+1)/n])$ (modulo rotation about the origin), defined for $j \in \{-n,\ldots,n-1\}$, is a.s.\ deterministic.
\end{corollary}
\begin{proof}
We already know, as in the proof of Lemma~\ref{lem::boundary_length_is_brownian_motion}, that the beaded quantum surfaces which are parameterized by the sets $\eta'([j/n, (j+1)/n])$ are independent of each other, and that these together with the information encoded by $\CF$ determine the embedding (by conformal removability of the collection of boundaries).  By rescaling, Theorem~\ref{thm::treesdetermineemebedding} implies that if we resample {\em one} of these surfaces conditionally on $(L,R)$ then {\em none} of the hulls of $\eta'([j/n, (j+1)/n])$ is changed.  Hence resampling all of them does not change the collection of hulls.  This shows that the collection of hulls of $\eta'([j/n,(j+1)/n])$ is a.s.\ determined by $\CF$ and $\CB$.  We now need to show that the collection of sets $\eta'([j/n, (j+1)/n])$ is determined (i.e., not just the hulls).  Fix $-1 < s < t < 1$ and let $u_0^N = s < u_1^N < \cdots < u_N^N = t$ be a sequence of partitions of $[s,t]$ with mesh size tending to $0$ as $N \to \infty$.  Letting ${\mathrm {hull}}(\eta'([a,b]))$ denote the hull of $\eta'([a,b])$, we then have that
\[ \eta'([s,t]) = \cap_{N=1}^\infty \cup_{k=0}^{N-1} {\mathrm {hull}}(\eta'([u_k^N,u_{k+1}^N])),\]
which completes the proof.
\end{proof}

Theorem~\ref{thm::treesdetermineemebedding} also implies a seemingly stronger statement, Corollary~\ref{cor::bmgivesembedding} below (which is a restatement of Theorem~\ref{thm::trees_determine_embedding}), which states that the parameterized path $\eta'$ and the measure $\mu_h$ are both determined by the Brownian motion $(L, R)$, up to a rotation of $\C$ about $0$.  In other words, the infinite volume topological surface obtained by mating the trees described by $L$ and $R$ a.s.\ has a uniquely defined conformal structure, which determines its embedding in $\C$.  Establishing this has been one of the main goals of this paper.

\begin {figure}[ht!]
\begin {center}
\includegraphics[scale=0.85]{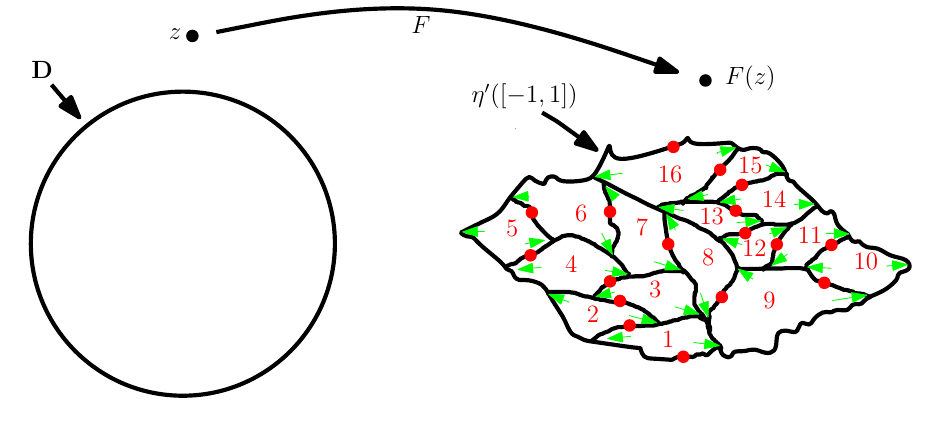}
\caption {\label{fig::necklacemap} The conformal map $F\colon\C \setminus \ol \D \to \C \setminus \eta'([-1,1])$ is normalized to look like the identity near infinity.  Taking $k=3$, we divide $[-1,1]$ into $2^{k+1} = 16$ equal-sized subintervals and illustrate their images, numbered according to the order in which they are traversed by $\eta'$.   The red dots indicate the $\eta'$ images of the endpoints of these intervals.  The green arrow heads point to $\eta'$ images of the times at which $L$ and $R$ were minimal within each given interval.  This figure should correspond to $\kappa' \geq 8$, because the interval images shown do not have cut points.  A $\kappa' \in (4,8)$ image would be similar but somewhat more complicated to illustrate.}
\end{center}
\end{figure}

\begin{corollary}
\label{cor::bmgivesembedding}
In the setting described just above, the pair $(L, R)$ a.s.\ uniquely determines the measure $\mu_h$ and the path $\eta'$ (modulo rotation about the origin).
\end{corollary}
\begin{proof}

Applying Corollary~\ref{cor::manychunk} with $n \to \infty$, and using the fact that the $\mu_h$ measure of each of the sets $\eta'([j/n, (j+1)/n)])$ is $1/n$, we find that the restriction of the $\mu_h$ measure to $\eta'([-1,1])$ {\em and} the restriction of $\eta'$ to $[-1,1]$ are both determined by the information encoded in $\mathcal F$ and $\mathcal B$, as discussed above.  In particular, this means that the beaded surface $\eta'([-1,1])$ is a.s.\ completely determined by the information encoded by $\mathcal B$, i.e., by the restriction of $L$ and $R$ to $[-1,1]$.

By scaling, an analogous statement holds for $[-C,C]$ for any $C$. Taking the $C \to \infty$ limit, we find that the entire conformal structure, together with the measure and path, is determined by the entire process $(L, R)$.
\end{proof}

We will actually establish a stronger quantitative version of Theorem~\ref{thm::treesdetermineemebedding} that implies Theorem~\ref{thm::treesdetermineemebedding}.  To make this stronger statement, let us define $\CB_k$ to be the $\sigma$-algebra generated by the following random variables:
 \begin{enumerate} \item The values $L_t$ and $R_t$ for each $t \in [-1,1]$ that is an integer multiple of $\delta = 2^{-k}$.
 \item The infimum of $L$ and the infimum of $R$ when restricted to each interval of the form $[a 2^{-k}, (a+1)2^{-k}] \subseteq [-1,1]$, for $a \in \Z$.
 \end{enumerate}
Also, let $\CN_1, \CN_2, \ldots, \CN_{2^{k+1}}$ denote the quantum surfaces parameterized by the range of $\eta'$ on these intervals.  As illustrated in Figure~\ref{fig::necklacemap}, each $\CN_j$ comes with four marked points, and the information encoded by the $\CB_k$ determines the lengths of the four boundary segments bounded by the red dots and green arrowheads in Figure~\ref{fig::necklacemap}. We have already established that, given these boundary lengths, the $\CN_j$ are conditionally independent of each other (recall the proof of Lemma~\ref{lem::boundary_length_is_brownian_motion}), and that together the $\CN_j$ determine the conformal structure of the combined quantum surface obtained by gluing them all together (this follows from Theorem~\ref{thm::paths_determined} and Theorem~\ref{thm::zip_up_wedge_rough_statement}).  Intuitively, we might expect that once we know $\CB_k$ for a large $k$ value, the extra information we obtain by sampling the $\CN_j$ would typically ``average out'' and not have a big effect on the function $F$.  This is indeed the case:

\begin{theorem}
\label{thm::quantitativetreesdetermineemebedding}
Write $\delta= 2^{-k}$ and let
$\Delta$ be the positive solution to the KPZ equation
\begin{equation}
\label{eqn::kpz8}
x = \frac{\gamma^2}{4} \Delta^2 + \left(1- \frac{\gamma^2}{4}\right) \Delta,
\end{equation}
with $x=2$.  (Note that when $\gamma \in (0,2)$~\eqref{eqn::kpz8} says that $x=2$ is a convex combination of $\Delta$ and $\Delta^2$, which implies $\Delta \in (\sqrt{2},2)$, so in particular $\Delta - 1 > 0$.)  Then a.s.\ 
\[\liminf_{k \to \infty} \frac{ \log  \E\Bigl[\Var\bigl(F(z) \giv \CB_k, \CF \bigr) \Big| \CF \Bigr]}{\log \delta} \geq \Delta - 1.\]
\end{theorem}

Informally, Theorem~\ref{thm::quantitativetreesdetermineemebedding} states that once the outer surface (i.e., the one parameterized by the unbounded component of $\C \setminus \eta'([-1,1])$) is fixed, the conditional variance of $F(z)$ given $\CB_k$ a.s.\ decays at least as fast as $\delta^{\Delta-1}$ (up to a factor whose log is $o(|\log \delta|)$) where $\delta = 2^{-k}$.
Note that Theorem~\ref{thm::quantitativetreesdetermineemebedding} implies that for each fixed $z \in \C \setminus \ol \D$, the quantity
\[ F_k(z) := \E[ F(z) \giv \CB_k, \CF]\]
converges a.s.\ to $F(z)$ exponentially fast as a function of $k$, which in particular implies that $F(z)$ is $(\CB, \CF)$ measurable (up to redefinition on a set of measure zero) and hence implies Theorem~\ref{thm::treesdetermineemebedding}.  Note that these maps $F_k$ are approximations to $F$, and that for each fixed $z$, $F_k(z)$ is a martingale in $k$.  We remark that it is not hard to see that for each $z$, there are uniform bounds on what the value of $F(z)$ can be (see Section~\ref{subsec::distortionlemma}) --- and thus, since each $F_k(z)$ is an average of analytic, uniformly bounded functions, the maps $F_k$ are themselves a.s.\ analytic on $\C \setminus \ol{\D}$.  However, one would not necessarily expect the maps $F_k$ be one-to-one.

The proof of Theorem~\ref{thm::quantitativetreesdetermineemebedding} will be a consequence of three lemmas, which we state here and then prove in the coming subsections.  The first of these is a simple and very general variance bound known as the Efron-Stein bound, established by Efron and Stein in 1981 \cite{efronstein}:

\begin{lemma}
\label{lem::generalvariancebound}
Let $A = A(X_1, X_2, \ldots, X_n)$ be a function of $n$ independent random variables defined on a common measure space such that $\E[ A^2] < \infty$.  Then
\[ \Var(A) \leq \sum_{i=1}^n \E[ \Var(A \giv X_j: j \not = i) ].\]
\end{lemma}

As stated above, given $\CB_k$ the quantum surfaces $\CN_j$ are conditionally independent of each other, and $F(z)$ is a function of these objects.  That is the context in which Lemma~\ref{lem::generalvariancebound} will be used.  To bound the overall variance of $F(z)$ it is enough to bound the expected amount of variance involved in resampling the quantum surface corresponding to a randomly chosen {\em one} of the pockets from Figure~\ref{fig::necklacemap} conditioned on $\CF$ and on all of the other pockets.  One way to bound this variance would be to give an absolute bound on how much $z$ can move when a pocket is replaced; the variance would then be at most the square of that bound.  Indeed, the second lemma we introduce is a simple complex analysis observation that is relevant to the following question: when one cuts out a small piece of a quantum surface, such as one of the $\CN_j$, and glues another piece in its place, how much can the embedding of the remainder of the quantum surface be distorted?

\begin{lemma}
\label{lem::diametertofourthpower}
There exist universal constants $C_1,C_2 > 0$ such that the following is true.  Let~$K_1$ be a hull of diameter at most~$r$ and~$K_2$ another hull such that there exists a conformal map $G \colon \C \setminus K_1 \to \C \setminus K_2$ of the form 
\[ G(z) = z + \sum_{n=1}^\infty \alpha_n z^{-n}.\]
Then whenever $\dist(z,K_1) \geq C_1 r$ we have that
\begin{equation}
\label{eqn::hullchangebound}
|G(z) - z| \leq C_2 r^2 |z-b_1|^{-1}
\end{equation}
where $b_1$ is the harmonic center of $K_1$ as defined in Section~\ref{subsec::distortionlemma} below.
\end{lemma}

Note that if the diameter of a pocket is $r$, and resampling it affects $F(z)$ by at most $O(r^2)$, then the variance involved in this resampling is $O(r^4)$.  The third lemma we need is a fact, closely related to a special case of the KPZ theorem, that will help us control the diameters of the pockets.

\begin{lemma}
\label{lem::kpzlem}
For a given $k$, suppose that $j$ is chosen uniformly at random from the set $\{1,2,\ldots,2^{k+1} \}$.  Then it is a.s.\ the case that the conditional expected size of $\diam(\CN_j)^4$ given $\CF$ decays almost as fast as $\delta^{\Delta}$ (see the precise statement below) where $\Delta$ solves the KPZ equation~\eqref{eqn::kpz8} with $x=2$.  The precise statement is that a.s.\ 
\begin{equation}
\label{eqn::loglimitdiamtofourth}
\limsup_{k \to \infty} \frac{ \log  \E[\diam(\CN_j)^4 \giv \CF]}{\log \delta^{-1}} \leq -\Delta,
\end{equation}
where here $\diam(\CN_j)$ is the diameter of the embedding of $\CN_j$ into $\C$.
\end{lemma}

We emphasize that in the statement of Lemma~\ref{lem::kpzlem}, the inequality~\eqref{eqn::loglimitdiamtofourth} holds for any embedding of the quantum cone but the rate of convergence depends on the particular choice of embedding.

We next observe that Theorem~\ref{thm::quantitativetreesdetermineemebedding} is a simple consequence of these three lemmas (which we have yet to prove).

\begin{proof}[Proof of Theorem~\ref{thm::quantitativetreesdetermineemebedding}]
Lemma~\ref{lem::diametertofourthpower} (applied with $K_1$ equal to the hull generated by one of the $\SLE$ segments and $K_2$ equal to a different possible realization of the hulls) implies that the conditional variance of $F(z)$ involved when one resamples one of the $\CN_j$ is at most of order $\diam(\CN_j)^4$.  (Note that if $z$ is fixed, then it follows from basic complex analysis --- see the discussion in Section~\ref{subsec::distortionlemma} --- that the distance from $F(z)$ to $\eta'([-1,1])$ is at least some constant that depends only on $z$.)  Together with Lemma~\ref{lem::generalvariancebound}, this implies that the conditional variance of $F(z)$ given $\CF$ and $\CB_k$ is at most (a constant times) the expected sum of the $\diam (\CN_j)^4$ over all $j$ values, which is just $\delta^{-1}$ times the expectation of $\diam (\CN_j)^4$ when $j$ is chosen uniformly from $\{1,2,\ldots,2^{k+1}\}$.  This latter expectation is bounded in Lemma~\ref{lem::kpzlem}, which shows that it decays at least as fast as $\delta^{\Delta}$ (in the sense that the $\liminf$ of the ratio of the logarithms is at most one), and this implies Theorem~\ref{thm::quantitativetreesdetermineemebedding}.
\end{proof}

The proofs of Lemma~\ref{lem::generalvariancebound} and Lemma~\ref{lem::diametertofourthpower} are both fairly straightforward, but the proof of Lemma~\ref{lem::kpzlem} is a bit more interesting.  Essentially, it will be a variant of the KPZ theorem, but we will have to use some tricks to account for the fact that the parameterization of the random quantum surface in question is not fixed as it is in the KPZ formulation of \cite{ds2011kpz} and also the fact that the statement involves random chunks of the surface traced by $\eta'$ over $\delta$ increments of time, instead of the random Euclidean balls of quantum mass $\delta$ that appear in the KPZ formulation of~\cite{ds2011kpz}.

\subsection{Proof of variance bound}
We next include a short proof of Lemma~\ref{lem::generalvariancebound} (originally proved by Efron and Stein in \cite{efronstein}).  We include this for completeness, and to emphasize that this is an easy fact to derive, but the reader who is already familiar with this inequality may skip this subsection.\footnote{We thank Asaf Nachmias for bringing \cite{efronstein} to our attention.}

\begin{proposition}
Let $I$ be an arbitrary index set.  Let $(X_i)$ for $i \in I$ be a collection of real-valued random variables on a common measure space.  Let $\nu$ be a probability measure on $I$ and assume that $\int \E[X_i^2] d\nu(i) < \infty$.  Write $X = \int X_i d\nu(i)$.  In words, $X$ is an average of the $X_i$, where the average is taken with respect to $\nu$.  Then
\begin{equation}
\label{eqn::averagevar}
\Var X \leq \int_I \Var(X_i) d\nu(i).
\end{equation}
\end{proposition}
\begin{proof}
The map $X \to \Var(X)$ is a quadratic, non-negative map from $L^2$ of the measure space to $\R$.  In particular it is convex.   Applying Jensen's inequality to this map yields~\eqref{eqn::averagevar}.
\end{proof}

\begin{proposition}
\label{prop::twomartingalepaths}
Let $Y$ and $Z$ be independent random variables and let $A = A(Y,Z)$ be a random variable that is a function of the pair $(Y,Z)$ with $\E[A^2] < \infty$.  The variance of the upper left increment in Figure~\ref{fig::martingalesquare} is less than or equal to the expected conditional variance of the lower right increment. That is,
\begin{equation}
\label{eqn::twosidesofmartingalesquare}
\Var(\E[A\giv Y]) \leq \E \Var \bigl[A \giv Z \bigr].
\end{equation}
\end{proposition}

\begin {figure}[h!]
\begin {center}
\includegraphics[scale=0.85]{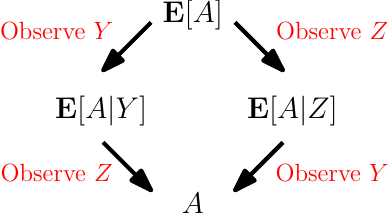}
\caption {\label{fig::martingalesquare} Two martingale pathways describing updated conditional expectations of $A$ as we observe $Y$ and $Z$, in either of the two orders.  The variance of the upper left arrow is at most that of the lower right arrow on average.}
\end{center}
\end{figure}

Informally, Proposition~\ref{prop::twomartingalepaths} states that the expected amount of fluctuation of our expectation of~$A$ when we observe $Y$ if we {\em have not} already observed~$Z$ is less than the expected amount of fluctuation when we observe~$Y$ if we {\em have} already observed~$Z$.  For example, suppose~$Y$ and~$Z$ are both uniformly chosen from $\{-1,1\}$ and $A(Y,Z) = YZ$.  If we do not know~$Z$, then observing $Y$ does not change our conditional expectation of~$A$.  But if we do know~$Z$, then observing $Y$ changes our conditional expectation for~$A$ from~$0$ to an element of $\{ \pm 1 \}$.  Thus, in this particular case, the left hand side of~\eqref{eqn::twosidesofmartingalesquare} evaluates to $0$ while the right hand side evaluates to $1$.

\begin{proof}[Proof of Proposition~\ref{prop::twomartingalepaths}]
The quantity $B = \E[A \giv Y]$ is a random variable that can be written as a function of $Y$.  If $z$ is fixed, then $B_z = \E[A \giv Y, Z=z] = \E[A(Y,z) \giv Y]$ is also a random variable that can be written as a function of $Y$.  Note that $B = \int B_z dz$ where $dz$ is the law of $Z$.  Thus it follows from~\eqref{eqn::averagevar} that $\Var B \leq \E \Var B_z$, where the expectation on the right is taken w.r.t.\ $z$ values chosen from the law of $Z$.  This is another way of writing~\eqref{eqn::twosidesofmartingalesquare}.
\end{proof}

\begin{proof}[Proof of Lemma~\ref{lem::generalvariancebound}]
Consider the martingale $Y_i = \E[A \giv  X_1, X_2, \ldots, X_i]$, indexed by $i = 0,\ldots,n$.  By the orthogonality of martingale increments, we have
\begin{equation}
\label{eqn::martingalevariance}
\Var(A) = \sum_{i=1}^n \E \Var [Y_i \giv  X_1, X_2, \ldots, X_{i-1}].
\end{equation}
We will now apply Proposition~\ref{prop::twomartingalepaths} to argue that the variance (of the change in our conditional expectation for $A$) that comes from observing $X_i$, given that we know only $X_1, X_2, \ldots, X_{i-1}$, is less than the variance that comes from observing $X_i$ when we know {\em all} of $X_j$ with $j \not = i$.  By treating $X_1, X_2, \ldots, X_{i-1}$ as given and applying Proposition~\ref{prop::twomartingalepaths} with $X_i$ playing the role of $Y$ and the variables $X_{i+1}, X_{i+2}, \ldots, X_n$ together playing the role of $Z$, we obtain
\begin{align*}
\E \Var [Y_i \giv  X_1, X_2, \ldots, X_{i-1}]
&=\E \Var \Bigl[\E [A | X_1,X_2,\ldots,X_i] \giv  X_1, X_2, \ldots, X_{i-1}\Bigr]\\
&\leq \E \Var [A \giv  X_j: j \neq i]
\end{align*}
Applying this bound to each term of~\eqref{eqn::martingalevariance} yields the proposition.
\end{proof}

\subsection{Proof of basic distortion lemma}
\label{subsec::distortionlemma}

This subsection recalls a few standard complex analysis facts (see, e.g., the text \cite{rudin1986real}).   We will work in the whole plane setting and use the term {\em hull} to refer to a closed, bounded, connected subset $K$ of $\C$ for which $\C \setminus K$ is also connected.   By the Riemann mapping theorem, and using the standard Laurent expansion, there exists a unique conformal map $F\colon\C  \setminus \ol{\D} \to \C \setminus K$ that has the form
\begin{equation}
\label{eqn::hulllaurent}
F(z) = a_{-1} z + a_0 + \sum_{n=1}^\infty \alpha_n z^{-n},
\end{equation}
where $a_{-1}$ is positive and real.  (This map fixes $\infty$.)  We refer to the $a_0$ of~\eqref{eqn::hulllaurent} as the {\em harmonic center} of $K$.  Computing term by term, we see that the mean value of $F$ on $\partial B(0,r)$, for any $r > 1$, is given by $a_0$.  If $F$ extends continuously to $\partial \D$, then taking $r \to 1$ shows that the mean value of the extension of $F$ to $\partial \D$ is also~$a_0$.  Another way to say this (which makes sense regardless of whether $F$ extends continuously to the boundary) is that $a_0$ is the average value of the points in $\partial K$ sampled according to harmonic measure viewed from infinity.  We say that $K$ is {\em centered} if its harmonic center is $0$.

We refer to the value $a_{-1}$ of~\eqref{eqn::hulllaurent} as the {\em outer conformal radius} of $K$.  (It can also be called ``whole plane capacity''; recall that whole plane Loewner evolution, directed toward $\infty$, describes an increasing sequence of hulls parameterized by the log of this quantity.)  Observe that $r \ol{\D}$ has outer conformal radius~$r$.  The following is known as the area theorem \cite[Theorem~14.13]{rudin1986real}.

\begin{proposition}
\label{prop::area}
Let $K$ be a centered hull of outer conformal radius $1$, so that there exists a conformal map $F\colon\C \setminus \ol \D \to \C \setminus K$ of the form
\begin{equation}
\label{eqn::centeredlaurent}
F(z) = z + \sum_{n=1}^\infty \alpha_n z^{-n}.
\end{equation}
Then $\sum_{n=1}^\infty n|\alpha_n|^2 \leq 1$.
\end{proposition}

In particular, Proposition~\ref{prop::area} implies that $|\alpha_n| \leq n^{-1/2}$ for each $n$, so
\[ |F(z) - z| \leq \sum_{n=1}^\infty n^{-1/2} |z|^{-n}.\]
When $|z| > 2$, the right hand side is dominated by the first term, which implies
\begin{equation}
\label{eqn::distortionbound}
|F(z) - z| \leq 2 |z|^{-1}.
\end{equation}
Similarly $|F'(z) - 1| \leq \sum_{n=1}^\infty n^{1/2} |z|^{-n-1},$
which is dominated by the first term when $|z| > 4$, so that in this case
\begin{equation}
\label{eqn::distortionbound2}
|F'(z) - 1| \leq 2 |z|^{-2}.
\end{equation}

\begin{proposition}
\label{prop::invertedkoebe}
If a hull $K$ has outer conformal radius $r$, then $K$ has diameter at most $4r$.  Also, the smallest closed ball containing $K$ has radius at least $r$ (and diameter at least $2r$).
\end{proposition}

\begin{proof}
By scaling and translating, we may reduce to the case that $K$ has outer conformal radius $1$ and $0 \in K$.  We now let $F$ be as in~\eqref{eqn::hulllaurent} and apply the Koebe-$1/4$ theorem to the function $\Phi(z) = 1/ F(1/z)$ (we take $\Phi(0) = 0$).  Observe that $\Phi$ is a conformal map from $\D$ to a subset of the complex plane; since $\Phi(0) = 0$ and $\Phi'(0) = 1$, the Koebe-$1/4$ theorem states that its image contains $B(0,\tfrac{1}{4})$.  This in turn implies that the image of $\C \setminus \ol \D$ under $F$ contains the set $\C \setminus \ol{B(0,4)}$, so that $K \subseteq \ol{B(0,4)}$.  Since this is true for any translation of $K$ (as long as $0 \in K$), we conclude that $K$ has diameter at most $4$.  The second sentence follows from the Schwarz lemma, which states that if the map $1/F(1/z)$ has derivative $1/z$ at the origin, then it cannot map $\D$ into a strict subset of a radius $1$ disk.
\end{proof}

We remark that both of the bounds mentioned in Proposition~\ref{prop::invertedkoebe} are tight.  For the tightness of the first, one can check explicitly that $K = [0,4] \subseteq \R \subseteq \C$ has outer conformal radius $1$.   This is related to the fact that the inverted set $\{z: 1/z \in K \} = [1/4,\infty]$, is the complement of a domain $\C \setminus [1/4,\infty)$ whose conformal radius (viewed from zero) is one.  (This is the standard example that shows the tightness of the Koebe $1/4$ theorem.)   For the tightness of the second bound in Proposition~\ref{prop::invertedkoebe}, it suffices to consider the trivial case $K = \ol \D$.

\begin{proof}[Proof of Lemma~\ref{lem::diametertofourthpower}]
If $r = 1$ and $b_1=0$, then~\eqref{eqn::hullchangebound} can be seen by applying~\eqref{eqn::distortionbound} and ~\eqref{eqn::distortionbound2} to the pair of maps $F_i\colon\C \setminus \D \to \C \setminus K_i$ and noting that $G = F_2 \circ F_1^{-1}$.  The general case follows by scaling and translation.
\end{proof}

\subsection{Proof of KPZ lemma on diameter fourth powers}
\subsubsection{A ``fixed parameterization'' variant of the fourth power estimate}
As part of the proof of a form of the KPZ scaling dimension relation, as presented in \cite[Section~4]{ds2011kpz},  one lets $h$ denote an instance of the zero-boundary GFF on a bounded domain $D$.  One then fixes a $\delta > 0$ and then seeks to compute the expectation of $\epsilon^{2x}$ where $\epsilon$ is a random number chosen through the following steps:
\begin{enumerate}
\item Sample $\mu_h$ from its law weighted by $\mu_h(S)$, where $S$ is a bounded, open set with $\ol{S} \subseteq D$.
\item Sample a point $z$ from $\mu_h$ restricted to $S$.
\item Let $\epsilon$ be the radius of the ball centered at $z$ with quantum mass exactly $\delta$.
\end{enumerate}
The version of the KPZ theorem\footnote{The presentation in \cite{ds2011kpz} assumes that there is a random fractal set $\CX$, independent of $h$, and that the $x$ is a scaling exponent associated to $\CX$, so that the probability that any given $\epsilon$ ball intersects $\CX$ is approximately $\epsilon^{2x}$.  The quantity $\Delta$ is then interpreted as a quantum scaling exponent associated with the fractal $\CX$, which gives the probability that a $\delta$ quantum area ball centered at a quantum typical point intersects $\CX$.  However, we stress that~\eqref{eqn::limitlogratio} also makes sense without reference to any particular $\CX$.   See also \cite{aru2015kpz} for an interesting example of the KPZ relation failing when $\CX$ is not independent of $h$.}
 established in \cite{ds2011kpz} states that
\begin{equation}
\label{eqn::limitlogratio}
\lim_{\delta \to 0} \frac{\log \E[\epsilon^{2x}]}{\log \delta^\Delta} = 1,
\end{equation}
where $\Delta$ and $x$ are related by~\eqref{eqn::kpz8}.

We would like to do something similar but with a slightly modified third step.  Namely, we will take $\epsilon$ to be the radius of the smallest  ball centered at $z$ that contains $\eta'([b-\delta, b+\delta])$ where $\eta'(b) = z$.  (Note that there is a.s.\ a unique $b$ for which $\eta'(b) = z$.)  The $\epsilon$ we define this way can only be larger than the one defined above.  We are only concerned with bounding $\E[\epsilon^4]$ from above which amounts to considering the case $x=2$.

\begin{proposition}
\label{prop::fourthpowerprop}
Let $S$ be open with $\ol{S} \subseteq \D$.  Let $h$ be an instance of the GFF on $\C$ with the additive constant fixed so that $h_1(0) = 0$ (i.e., the average of~$h$ on $\partial \D$ is equal to $0$) and define $\mu_h$ as usual using a fixed $\gamma \in (0,2)$.  Let~$\eta'$ be an independent whole-plane space-filling $\SLE_{\kappa'}$ from~$\infty$ to~$\infty$, parameterized so that $\mu_h (\eta'([s,t])) = t-s$ for all $s < t$ and $\eta'(0) = 0$.  Now suppose that we
\begin{enumerate}
\item Sample $\mu_h$ from its law weighted by $\mu_h(S)$.
\item Sample a point $z$ from $\mu_h$ restricted to $S$.
\item Let $\epsilon$ be the radius of the smallest closed ball centered at $z$ that contains $\eta'([b-\delta, b+\delta])$ where $b$ is taken so that $\eta'(b) = z$.
\end{enumerate}
Then using this definition of $\epsilon$, we still have that $\E[\epsilon^4]$ is approximately $\delta^\Delta$ in the sense that
\begin{equation}\label{eqn::fourthpowerkpz} \liminf_{\delta \to 0} \frac{ \log \E[\epsilon^4]}{\log \delta^\Delta}= 1 \end{equation} where $\Delta$ is the positive solution to~\eqref{eqn::kpz8} with $x=2$.
\end{proposition}

The proof of~\eqref{eqn::limitlogratio} given in \cite[Section 4]{ds2011kpz} implies that~\eqref{eqn::fourthpowerkpz} would hold if we instead took $\epsilon$ to be the radius of the Euclidean ball centered at $z$ containing $\delta$ units of $\mu_h$ area.\footnote{The setup in \cite[Section 4]{ds2011kpz} assumes fixed boundary conditions on a bounded domain, which is somewhat different from our setup here.  However, this distinction is not relevant to the argument in \cite[Section 4]{ds2011kpz}.}   It turns out that with a couple of supplemental arguments, the same proof also implies Proposition~\ref{prop::fourthpowerprop}.  We will not repeat the entire proof of \cite[Section~4]{ds2011kpz} here, but we will give a short overview and explain the necessary supplemental arguments during the remainder of this subsection.
 
The proof in \cite[Section~4]{ds2011kpz} actually begins with yet another analog of~\eqref{eqn::limitlogratio} which is stated as \cite[Theorem~4.2]{ds2011kpz}.  For this analog, one notes that if one has $h$ and $z$, one can construct $h_r(z)$, which evolves as a Brownian motion in $-\log r$.  One then establishes~\eqref{eqn::limitlogratio} with $\epsilon$ replaced by the smallest $r$ for which
\begin{equation}
\label{eqn::circleaverage}
\gamma h_r(z) + \gamma Q \log r = \log \delta.
\end{equation}
Roughly speaking, this is interpreted as the smallest $r$ such that the ``typical value'' of $\log \mu_h \bigl( B(z,r) \bigr)$, conditioned on $h_r(z)$, is equal to $\log \delta$.

The next step in \cite[Section 4]{ds2011kpz} is to derive~\eqref{eqn::limitlogratio} from \cite[Theorem 4.2]{ds2011kpz}.  This is done by deriving bounds (\cite[Lemmas 4.6--4.7]{ds2011kpz}) on the probability that  
the actual value of $\log \mu_h \bigl( B(z,r) \bigr)$ differs from the LHS of~\eqref{eqn::circleaverage} by a large amount.  In order to establish Proposition~\ref{prop::fourthpowerprop} using the argument of  \cite[Section 4]{ds2011kpz}, it suffices to establish variants of these lemmas using the definition of $\epsilon$ from Proposition~\ref{prop::fourthpowerprop}.

The proofs of these variants are exactly the same as the proofs in \cite[Lemmas 4.6--4.7]{ds2011kpz} except that one requires an analog, appropriate for our setting, of \cite[Lemma 4.5]{ds2011kpz}.  The statement and proof of this analog comprise the remainder of this subsection.  We begin by giving a restatement of \cite[Lemma 4.5]{ds2011kpz}.

\begin{proposition}({\it Restatement of \cite[Lemma 4.5]{ds2011kpz}})
\label{prop::kpzsmallareabound}
Fix $\gamma \in [0,2)$, let $h$ be an instance of the GFF on $\C$ with $h_1(0) = 0$, and let $\mu_h$ be the $\gamma$-LQG area measure associated with $h$.  Then the random variable $A = \log \mu_h (B(0, 1/2))$ satisfies $p_A(t) := \p[A \leq t] \leq e^{-Ct^2}$ for some fixed constant $C > 0$ and all sufficiently negative values of $t$.
\end{proposition}

The KPZ argument in \cite[Section 4]{ds2011kpz} actually only requires the fact that $p_A(t)$ decays superexponentially (i.e., faster than $e^{-Ct}$ for any $C$).  Thus, the fact that $p_A(t)$ decays quadratic exponentially, as stated in Proposition~\ref{prop::kpzsmallareabound}, is somewhat more than is necessary for the purposes of \cite[Section 4]{ds2011kpz}.  Proposition~\ref{prop::kpzsmallareabound} can be interpreted as bounding, within the particular setting of the proposition, the probability that $\epsilon > 1/2$ when $\delta = e^{-t}$ is extremely small and $\epsilon$ is defined to make $\log \mu_h \bigl( B(0,\epsilon) \bigr) = 0$.  To bound the probability that the $\epsilon$ defined as in Proposition~\ref{prop::fourthpowerprop} exceeds $1/2$, we will need the following analog of Proposition~\ref{prop::kpzsmallareabound}.

\begin{proposition}
\label{prop::kpzslesmallareabound}
Fix $\gamma \in [0,2)$, let $h$ be an instance of the GFF on $\C$ with $h_1(0) = 0$, and let $\mu_h$ be the $\gamma$-LQG area measure associated with $h$.   Let $\eta'$ be an independent space-filling~$\SLE_{\kappa'}$ from~$\infty$ to~$\infty$ in~$\C$ parameterized so that $\eta'(0) = 0$ and $\mu_h(\eta'([s,t])) = t-s$ for $s < t$.  Let $a = \sup \{s < 0: |\eta'(s)|=1/2 \}$ and $b = \inf \{s > 0: |\eta'(s)| = 1/2 \}$, and write $A'  = \log \min(-a,b)$.  Then
\begin{equation}
\label{prop::pabound}
p_{A'}(t) := \p[A' \leq t] \leq e^{-C t^2/\log |t|}
\end{equation}
for some fixed constant $C > 0$ and all sufficiently negative values of $t$.
\end{proposition}
\begin{proof}
It is enough to prove that~\eqref{prop::pabound} holds if we replace $A'$ with either $\log(-a)$ or $\log b$ (which are equivalent by symmetry) since combining the two affects the RHS of~\eqref{prop::pabound} by at most a factor of $2$.  For the proof below, we redefine $A'$ to be $\log b$.

The difficulty in proving Proposition~\ref{prop::kpzslesmallareabound} is that there are essentially two things that can make $\eta'$ get to $\partial B(0,1/2)$ in a small amount of time (which would make $b$ small), only the first of which is dealt with in Proposition~\ref{prop::kpzsmallareabound}:
\begin{enumerate}
\item The total mass of $B(0,1/2)$ being small.
\item The path $\eta'$ filling only very small amounts of Euclidean area before reaching $\partial \D$ and/or somehow managing to avoid the parts of $B(0,1/2)$ containing the most mass during its trajectory from $0$ to $\partial B(0,1/2)$.
\end{enumerate}
To proceed with the analysis, let us imagine that we do the sampling in two steps.  First we choose $\eta'$, which determines the set $\eta'([0,b])$, although the value of $b$ is not known yet.  Then we let $r$ denote the radius of the largest ball contained in $\eta'([0,b])$.  Then we single out the value $z$ such that  $B(z,r) \subseteq \eta'([0,b])$.  Then we sample $h$ and check whether $\mu_h \bigl( B(z,r) \bigr) \leq e^t$.  The latter condition is {\em necessary} in order for us to have to have $b < e^t$, so to establish
\eqref{prop::pabound} it suffices to show that
\[ \p \Bigl[ \log \mu_h \bigl( B(z,r) \bigr) \leq t \Bigr] \leq e^{-C t^2 / \log|t|}.\]
First, we observe that the random variable $T = 1/r$ decays exponentially, in that $\p[T \geq s] \leq e^{-Cs}$ for some $C$ and all sufficiently large $s$.  This can be seen by considering $\lfloor 1/2r \rfloor$ disks centered at $0$ with radii $2r, 4r, 6r, \ldots$ and noting that conditioned on $\eta'$ up to the first time it exits one of these disks, there is always a uniformly positive chance that it fills a disk of radius $r$ before it exits the next disk outward.  (See analogous bounds in \cite[Section~4.3.1]{ms2013imag4}.)

Second, observe that given $\eta'$ and $z$, the law of $h_r(z)$ is given by a Gaussian with variance $\log T$ (plus or minus a constant); if we redefine $T$ to be exactly this variance, then we can write $h_r(z) = W_{\log(T)}$ where $W$ is a standard Brownian motion.

Third, observe that once $h_r(z)$ is known, there is a constant $c$ such that the conditional law of the random variable $Z = \log \mu_h( B(z,r) ) - \gamma h_r(z) + c \log r$  has quadratic exponential decay independently of $r$ and $h_r(z)$.  (This is included in \cite[Lemma 4.6]{ds2011kpz}.)  The result now follows immediately from Proposition~\ref{prop::reallyfastdecaybound} as stated below.
\end{proof}

\begin{proposition}
\label{prop::reallyfastdecaybound}
Let $T$ be a random variable with the property that
\begin{equation}
\label{eqn::Tbound}
\p[T \geq s] \leq e^{-Cs}
\end{equation}
for some $C > 0$ and all sufficiently large $s$.  Independently of $T$, sample a standard Brownian motion $W$.  Then define the random variable $X = W_{\log(T)} +c \log T + Y$ where $c$ is a fixed constant and $Y$ is an independent random variable such that 
\begin{equation}
\label{eqn::Ybound}
\p[ Y \geq s] \leq e^{- C s^2}
\end{equation}
for sufficiently large $s$.  Then for all sufficiently large $a$ and some $C'$ we have
\begin{equation}
\label{eqn::Xbound}
\p[X \geq a] \leq e^{-C' a^2/\log (a)}
\end{equation} 
\end{proposition}
\begin{proof}
On can break the probability considered in~\eqref{eqn::Xbound} into two cases: $T > a^2$ and $T \leq a^2$.  In the former case, we get a bound of $e^{-C a^2}$ from~\eqref{eqn::Tbound}.  In the latter case, $W_{\log(T)}$ is a Brownian motion stopped at time at most $2 \log a$.  Thus, for an appropriate $C'$, the probability that  $W_{\log(T)}$ exceeds $a/2$ decays as fast as the RHS of~\eqref{eqn::Xbound}, and the probability that $Y > a/2$ decays faster than the RHS of~\eqref{eqn::Xbound} by~\eqref{eqn::Ybound}.
\end{proof}

\subsubsection{Quantum cones: smooth centering and resampling}

We would now like to use Proposition~\ref{prop::fourthpowerprop} to establish Lemma~\ref{lem::kpzlem}.  The statements of these two results are quite similar in a sense.  Both involve expected fourth powers of $\epsilon$ values, which are defined from $h$, $\eta'$, and $\delta$ in a similar way.  However, the statements are set up rather differently.  The former involves a GFF $h$ on $\C$ and the quantum surface parameterized by a fixed domain $S$.  Although $S$ is fixed, its quantum measure $\mu_h(S)$ is random (and one weights by this quantum measure before sampling the quantum typical point $z$).  The latter involves the random piece of a quantum cone traced by $\eta'$ during the time interval $[-1,1]$.  Although the quantum size of the glued-in quantum surface is fixed, the set $\eta'([-1,1])$ that parameterizes it is random.

Despite these differences between the setups, one might expect that the way that $\E[\epsilon^4]$ scales as $\delta \to 0$ in the two settings should depend somehow on the bulk behavior of the quantum surfaces, and that we could use this intuition to somehow derive Lemma~\ref{lem::kpzlem} from Proposition~\ref{prop::fourthpowerprop}.  That is the purpose of the current section.

To begin with, let $(\C,h,0,\infty)$ be a $\gamma$-quantum cone with the projection of $h$ onto $\CH_1(\C)$ normalized as in Definition~\ref{def::quantum_cone}.  For $t \geq 0$, let $R_t$ be such that $e^{\gamma R_t}$ is the expected $\gamma$-LQG mass in $B(0,e^{-t})$ given the mean value of $h$ on $\partial B(0,e^{-t})$.  Then $R_t = B_t + (\gamma-Q) t$ where $B_t$ for $t > 0$ is a standard Brownian motion; note that $\gamma-Q < 0$.

Now we will consider the quantity $s \mapsto \int_{\R} R_t \phi(t-s) dt$ where $\phi$ is a non-negative $C^\infty$ bump function supported on $(-1/100,0]$ with total integral one.  This is an integral of $h$ against a smooth rotationally invariant bump function.  We apply a coordinate change rescaling to $h$ so that this function achieves the value $0$ for the first time at $0$.  We refer to the $h$ with this type of scaling as the {\em smooth canonical description} for the quantum cone.  Now let $S$ be a closed, simply-connected subset of $\D \setminus \{0\}$ with non-empty interior.  We also fix values $t_2 > t_1 > 0$  and $r > 0$.  Given a smooth canonical description $h$ of a $\gamma$-quantum cone as above, and an independently chosen $\eta'$, we say that the configuration $(h, \eta')$ is {\bf $(S, t_1, t_2, r)$-stable} if the following are true:
\begin{enumerate}
\item\label{it::stable_cond1} $\eta'([t_1,t_2]) \subseteq S$ and $\diam(\eta'([t_1,t_2])) > r$
\item\label{it::stable_cond2} The previous item continues to hold if we cut out the surface parameterized by $\eta'([t_1,t_2])$ and weld in any surface at all of the same $\gamma$-LQG area and $\gamma$-LQG boundary length in its place.  In other words, if we consider any conformal map $\psi$ from the unbounded connected component of  $\C \setminus \eta'([t_1,t_2])$ to $\C \setminus K$, for some $K$, that is normalized so that
\begin{enumerate}[(i)]
\item $\psi$ fixes $\infty$ and $0$,
\item $\psi$ has a positive real derivative at $\infty$,
\item $\psi$ is scaled in such a way that the pushforward of $h$ via the quantum coordinate change described by $\psi$ corresponds to a smooth canonical description (note that on the event $K \subseteq S$ this statement does not depend on how the pushforward is defined on $K$ itself), 
\end{enumerate}
then $K \subseteq S$ and $\diam(K) \geq r$.
\end{enumerate}

\begin{remark}
\label{rem::stable_measurable}
Let $F \colon \C \setminus \ol{\D} \to \C \setminus \eta'([t_1,t_2])$ be the unique conformal transformation which looks like the identity at $\infty$.  The convenient property that this definition possesses is that it is measurable with respect to the $\sigma$-algebra generated by $h \circ F^{-1} + Q\log|(F^{-1})'|$ (this is the $\sigma$-algebra $\CF$ introduced earlier with $[t_1,t_2]$ in place of $[-1,1]$).  That is, it is completely independent of the quantum surface parameterized by $\eta'([t_1,t_2])$ --- i.e., one can cut out this piece and replace it with any quantum surface (with the same boundary lengths and area) without affecting whether the overall configuration is $(S, t_1, t_2, r)$-stable or not.  If we had omitted Condition~\ref{it::stable_cond2} from the definition of $(S,t_1,t_2,r)$-stability (and only included Condition~\ref{it::stable_cond1}), then whether a configuration is $(S,t_1,t_2,r)$-stable would not be $\CF$-measurable.  This point will be very important when we complete the proof of Lemma~\ref{lem::kpzlem} below.  In particular, we will:
\begin{enumerate}
\item Argue (Proposition~\ref{prop::stablerestrictionfourthpower}) that we have the desired fourth moment control on the event that a configuration is $(S,t_1,t_2,r)$-stable.
\item We will then argue that there a.s.\ exists some $(S,t_1,t_2,r)$ such that the configuration is $(S,t_1,t_2,r)$-stable after applying a linear a change of coordinates.  Due to our definition of stability, the values of $t_1$, $t_2$, $r$, and $S$ are determined by $\CF$ hence we get the desired fourth moment control conditional on $\CF$.
\end{enumerate}
\end{remark}

We first observe that stable configurations occur with positive probability:

\begin{proposition}
\label{prop::st1t2rstablehaspositiveprob}
Let $h$ be a smooth canonical description of a $\gamma$-quantum cone, and let $\eta'$ be an independent space-filling $\SLE_{\kappa'}$ in $\C$ from $\infty$ to $\infty$, so that $\mu_h(\eta'([s,t])) = t-s$ for all $s < t$.  Fix $t_2 > t_1 > 0$ and let $\CA(S, t_1, t_2, r)$ denote the event that $(h, \eta')$ is $(S, t_1, t_2, r)$-stable.  Then there exists a closed, simply-connected set $S \subseteq \D \setminus \{0 \}$ with non-empty interior and $r > 0$ for which $\CA(S, t_1, t_2, r)$ has positive probability.
\end{proposition}

This proposition may seem obvious --- of course there is {\em some} probability that $\eta'([t_1,t_2])$ has very small diameter and lies in the middle of $S$, and then one would expect that cutting out the quantum surface traced by $\eta'([t_1,t_2])$ and gluing in a new should roughly preserve the location of $\eta'([t_1,t_2])$ in the center of $S$ while changing its diameter by at most a constant factor.  The trouble is that after replacing $\eta'([t_1,t_2])$ with a new surface, the field $h$ undergoes a quantum coordinate change that changes its values everywhere, and after this modification, the degree of rescaling that is required in order to produce a smooth canonical description could be drastically different (enough so that the point no longer lies in $S$).  To understand how the degree of rescaling required might change after a coordinate change, we will have to consider how much $(h, \phi)$ can change under small perturbations of $\phi$.  Before we prove Proposition~\ref{prop::st1t2rstablehaspositiveprob} we will need some preliminary results along these lines.  We begin by recalling some general facts about families of Gaussian random variables. The following proposition is stated in the introduction of \cite{chatterjee2008chaos}, which contains a short account of its history.

\begin{figure}[ht!]
\begin{center}
\includegraphics[scale=0.85]{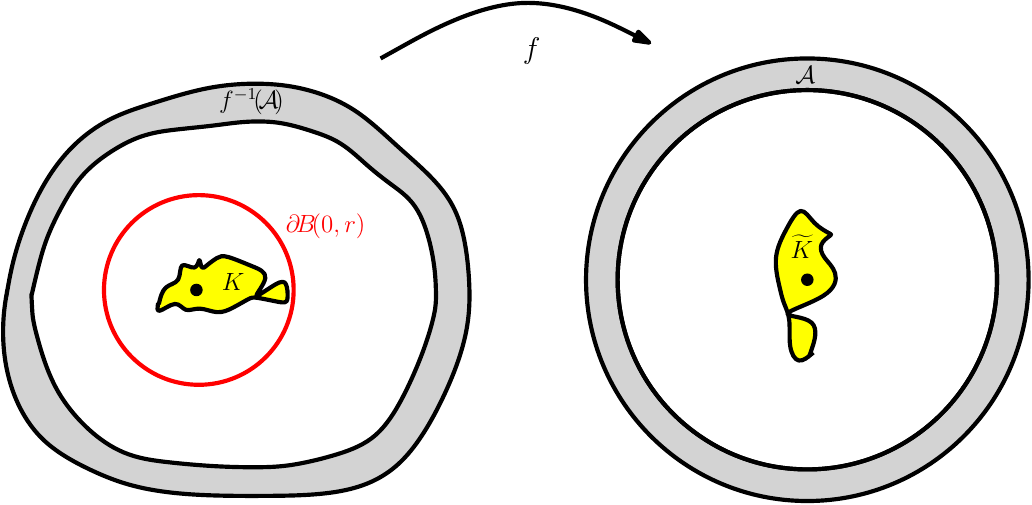}
\caption {\label{distortedbumprange} The set $K$ is a subset of $B(0,r)$ (whose boundary is shown as a red circle on the left) and $f$ is a continuous homeomorphism from $\C$ to $\C$ that maps $\C \setminus K$ conformally to $\C \setminus \wt K$, has derivative $1$ at $\infty$, and fixes the origin.  On the right, a fixed annulus $\CA$ is shown in grey, and on the left we have the $f$ preimage of that set.  Suppose that $h$ is the field on the left and $\wt h = h \circ f^{-1} + Q \log |f'|$ is the field defined on the right.  Let $\rho \colon \CA \to \R$ be a radially symmetric bump function supported in $\CA$.  Then there is a constant $b$ and $\wt \rho$ such that $(\wt h, \rho) = (h, \wt \rho) +b$.  As $K$ and $\wt K$ vary, the supremum and infimum of $(\wt h, \rho)$ are random finite quantities.}
\end{center}
\end{figure}

\begin{proposition}
\label{prop::finitemin}
If $X_1, X_2, \ldots, X_k$ are centered jointly Gaussian random variables (not necessarily independent) each with variance at most $\sigma^2$, then
\[ \Var \bigl( \min_{1 \leq i \leq k} X_i \bigr) = \Var \bigl( \max_{1 \leq i \leq k} X_i \bigr) \leq \sigma^2. \]
\end{proposition}

The following is an easy extension:

\begin{proposition}
\label{prop::countablemin}
If $X_1, X_2, \ldots$ is a countable sequence of jointly Gaussian variables, each with variance at most $\sigma^2$, such that $\min X_i$ and $\max X_i$ are a.s.\ finite, then
\[ \Var \bigl( \min_i X_i \bigr) = \Var \bigl( \max_i X_i \bigr) \leq \sigma^2.\]
\end{proposition}

\begin{proof}
Proposition~\ref{prop::countablemin} can be obtained from Proposition~\ref{prop::finitemin} by considering the monotone sequence of random variables $\min(X_1, X_2, \ldots, X_k)$, for $k = 1, 2, 3, \ldots$, and observing that if the variance of the limit exceeds $\sigma^2$, then for all $k$ sufficiently large, the variance of $\min(X_1, X_2, \ldots, X_k)$ must exceed $\sigma^2$, which would contradict Proposition~\ref{prop::finitemin}.
\end{proof}
\vspace{.1in}

Now let $\Phi$ be the set of ``distorted bump functions'' of the form $|g'|^{2} \phi \circ g$ for which $g^{-1}$ is conformal outside of $K$ for some $K \subseteq B_{1/10}(0)$, and is normalized to look like the identity near infinity, and fixes the origin.  Then $\Phi$ is a sequentially compact subset of the space of test functions w.r.t.\ the topology of uniform convergence of all derivatives on compact sets.  Since this space is a Fr\'echet space (thus metrizable) $\Phi$ is also compact.  In particular, for any distribution $h$, both of the quantities
\[ M_1(h) = \inf_{\wt \phi \in \Phi } (h, \phi - \wt \phi) \quad\text{and}\quad  M_2(h) = \sup_{\wt \phi \in \Phi } (h, \phi - \wt \phi)\]
are finite.  By continuity, they are equal to the infimum and supremum taken over a countable dense subset of $\phi \in \Phi$.  Since the variance of $(h, \phi - \wt \phi)$ is a continuous function of $\wt \phi$, it also has a maximum over $\Phi$, and it follows from Proposition~\ref{prop::countablemin} that the variances of $M_1(h)$ and $M_2(h)$ are finite.

Now, for $j \in \{1,2\}$ we can write $R_t^j = M_j (z \to h(e^t z) + Q t)$.  Roughly speaking, $R_t^1$ and $R_t^2$ describe minimal and maximal averages on the distorted annular bump functions with inner radius about $e^t$ (with the $Qt$ term to account for the coordinate change).  Now we observing the following:

\begin{proposition}
\label{prop::a1a2transience}
It is a.s.\ the case that the processes $R_t^1$ and $R_t^2$ both tend to~$\infty$ as $t \to \infty$ and to $-\infty$ as $t \to -\infty$.
\end{proposition}
\begin{proof} The continuity argument above also implies that the minimum and maximum values obtained by $R_t^1$ and $R_t^2$, as $t$ varies over a compact range of values, are also a.s.\ finite and have bounded variance.  The different between the minima (or maxima) over $[n,n+1]$ and over $[n,n+2]$ thus has finite variance and finite expectation, and the result is easily obtained by the ergodic theorem.
\end{proof}

\begin{proof}[Proof of Proposition~\ref{prop::st1t2rstablehaspositiveprob}]
Proposition~\ref{prop::a1a2transience} implies that there a.s.\ exists $c \in (0,\infty)$ (random) such that we have $R_t^1, R_t^2 > 0$ for $t > c$ and $R_t^1, R_t^2 < 0$ for $t < -c$.  Pick $c_1 \in (0,\infty)$ such that $\p[ c \geq c_1] \leq \tfrac{1}{2}$.  Fix a square $S \subseteq \D \setminus \{0\}$ with side length $\tfrac{1}{2}$ very close to $0$.  It is easy to see that we can find $K \subseteq S$ closed with non-empty interior such that $a K \subseteq S$ for any $a \in (\tfrac{1}{10} e^{-c_1}, 10 e^{c_1})$ and that $\eta'([t_1,t_2])$ lies in $K$ with positive probability.  On this event, it is not hard to check that if one swaps in a new quantum surface for the one traced by $\eta'([t_1,t_2])$, then resulting appropriately scaled surface will still belong to $S$, and also that its diameter is scaled by some factor which is contained in the interval $[\tfrac{1}{10} e^{-c}, 10 e^c]$.  This, in turn, implies that for some $r > 0$, the diameter of $\eta'([t_1,t_2])$ is greater than~$r$ no matter what piece is swapped in.  This implies the proposition statement.
\end{proof}

\begin{proposition}
\label{prop::stablerestrictionfourthpower}
On the event $\CA(S,t_1,t_2,r)$, we have the analog of~\eqref{eqn::loglimitdiamtofourth} in which the $\CN_j$ are obtained by dividing $[t_1,t_2]$, rather than $[-1,1]$, into size $\delta$ increments.
\end{proposition}
\begin{proof}
In this proof, will be considering two different laws:
\begin{enumerate}
\item\label{it:trip_law1} The law on triples $(h,\eta',z)$ where $h$ is a GFF on $\C$ with the additive constant fixed so that $h_1(0) = 0$ (i.e., its average on $\partial \D$ is equal to $0$) weighted by $\mu_h(S)$ and $z$ is picked from $\mu_h$.  We will denote the corresponding probability and expectation by $\p^*$ and $\E^*$.  This is the setting of Proposition~\ref{prop::fourthpowerprop}.
\item\label{it:trip_law2} The law on triples $(h,\eta',z)$ where we have \emph{not} weighted the law of $h$ by $\mu_h(S)$.  We will denote the corresponding probability and expectation by $\p$ and $\E$.
\end{enumerate}

Here, we will use that the smooth canonical description of a $\gamma$-quantum cone restricted to $B(0,1/2)$ looks like a whole-plane GFF plus $-\gamma \log|\cdot|$ inside of $B(0,1/2)$.

Suppose that we have a triple $(h,\eta',z)$ picked from $\p^*$ as in~\ref{it:trip_law1} above.  The bound~\eqref{eqn::fourthpowerkpz} in Proposition~\ref{prop::fourthpowerprop} bounds $\E^*[\epsilon^4]$.  Let $E$ be the event that both $\CA(S,t_1,t_2,r)$ and $z \in \eta'([t_1,t_2])$ hold.  Since $\epsilon^4$ is a.s.\ positive, we also get an upper bound on $\E^*[\epsilon^4 \one_E]$.

On $E$, we have that $\mu_h(\eta'([t_1,t_2])) = t_2 - t_1 \leq \mu_h(S)$ since $\eta'([t_1,t_2]) \subseteq S$.  Consequently, we have that
\[ \frac{d\p^*}{d\p} \one_E = c_1 \mu_h(S) \one_E \geq c_1 (t_2-t_1) \one_E \]
where $c_1 \in (0,\infty)$ is a normalizing constant.  Rearranging, we have that
\begin{equation}
\label{eqn::rn_ubd}
\one_E \leq \left(\frac{1}{c_1(t_2-t_1)}\right) \frac{d\p^*}{d\p} \one_E.
\end{equation}
In particular, when we integrate a non-negative random variable on~$E$ we can replace~$\E$ with~$\E^*$ and only lose a constant factor.

On $E$, let $j$ be such that $\CN_j$ contains $z$.  Note that $\CN_j$ is then uniform among the surfaces parameterized by $\delta$-length intervals of $[t_1,t_2]$.  With $b$ the hitting time of~$z$ by~$\eta'$, note also that $\CN_j \subseteq \eta'([b-\delta,b+\delta])$.  Combining everything, we see that there exists a constant $c_2 > 0$ depending only on $t_1$, $t_2$, and $S$ such that
\begin{align*}
          \E[\diam(\CN_j)^4 \one_E]
&\leq \E[ \diam(\eta'([b-\delta,b+\delta]))^4 \one_E]
 \leq c_2 \E^*[\epsilon^4] \quad\text{(by \eqref{eqn::rn_ubd})}.
\end{align*}
This does not quite complete the proof yet because in~\eqref{eqn::loglimitdiamtofourth}, we have normalized $\eta'$ so that $\eta'([-1,1])$ has outer conformal radius equal to $1$.  In the present context, we note that the scaling factor necessary to normalize $\eta'([t_1,t_2])$ to have outer conformal radius $1$ on $E$ is at most a constant times $1/r$ since $\diam(\eta'([t_1,t_2])) \geq r$ on $E$.  Therefore the result follows from~\eqref{eqn::fourthpowerkpz}.
\end{proof}

\begin{proof}[Proof of Lemma~\ref{lem::kpzlem}]
Fix $t_2 > t_1 > 0$.  Since the pair which consists of the quantum cone and $\eta'$ is invariant under scaling and also invariant under the operation of mapping $\eta'(t)$ to $0$, we may consider the map $\psi$ that sends the pair $(h,\eta')$ to the rescaled and reparameterized configuration obtained by applying the affine map of time that sends the pair of times $(t_1,t_2)$ to $(-1,1)$.  Thus, it follows from Proposition~\ref{prop::stablerestrictionfourthpower} that the bound~\eqref{eqn::loglimitdiamtofourth} holds for the $(h,\eta')$ configuration on the event that its $\psi$ preimage belongs to $\CA(S, t_1, t_2, r)$.  We emphasize that (as mentioned in Remark~\ref{rem::stable_measurable}) this event is $\CF$-measurable.

Fix $p_0 \in (0,1)$.  To finish the proof, it suffices to show that there exists $t_1,t_2,r$ and $S$ with $\p[\CA(S,t_1,t_2,r)] \geq p_0$.  One can see that there has to be such $t_1$, $t_2$, $r$, and $S$ using the following variant of the argument used to prove Proposition~\ref{prop::st1t2rstablehaspositiveprob}.  Using the continuity of $\eta'$ and the argument used to prove Proposition~\ref{prop::st1t2rstablehaspositiveprob} it follows that we can pick $t_1,t_2,r > 0$ and take $S = \{ z \in \C : \epsilon \leq |z| \leq \tfrac{1}{2},\ \arg(z) \in [0,2\pi-\epsilon]\}$ with $\epsilon > 0$ very small such that the following are true:
\begin{enumerate}
\item The probability that $\eta'(t_1) \in S$ and $\dist(\eta'(t_1),\partial S) \geq \epsilon$ is at least $\tfrac{1}{2}+\tfrac{p_0}{2}$.
\item Given that $\eta'(t_1) \in S$ and $\dist(\eta'(t_1), \partial S) \geq \epsilon$, we have that the conditional probability of $\CA(S,t_1,t_2,r)$ is at least $\tfrac{1}{2}+\tfrac{p_0}{2}$.
\end{enumerate}
Therefore $\p[\CA(S,t_1,t_2,r)] \geq p_0$, which proves the claim.
\end{proof}

\section{Weldings of forested wedges and duality}
\label{sec::duality}

This section contains the proofs of our main results about matings of forested wedges and forested lines, namely Theorems~\ref{thm::gluingtwoforestedlines}--\ref{thm::quantum_natural_zip_unzip_rough_statement} (Section~\ref{subsec::bm_sle}) but we begin with some more general discussion about trees of disks and trees of spheres (Section~\ref{subsec::duality}).

\subsection{Bottlenecks, baby universes and Liouville duality}
\label{subsec::duality}

Here we remark that the idea of producing ``forested wedges'' and ``forested lines'' as trees of quantum disks is also not without precedent.  In the physics literature, such objects have been heuristically described in various ways and are known as ``Liouville quantum gravity with parameter $\gamma'>2$'' where $\gamma' = 4/\gamma$.  (This is related to the fact that they encode $\SLE_{\kappa'}$ paths where $\kappa' = 16/\kappa$.)  They are expected to be scaling limits of random simply connected maps with a critical number of domains cut off
by small bottlenecks, as first constructed with discrete matrix models in  \cite{MR1057913,MR1186338,MR1214333,MR1293688,MR1279108}.  (See also \cite{MR1175135}.)

The corresponding Liouville measure was first heuristically considered for $\gamma'>2$ in  \cite{kleb1995touching_surfaces,kleb1995non_perturb,kleb1996wormholes}, in the so-called 
``other gravitational dressing'' of the Liouville potential,  
 or ``dual branch of gravity,'' leading in particular to a {\it dual} version of the KPZ relation, extended to $\gamma'>2$.  However, the limit of the regularized Liouville bulk measure~\eqref{e.mudef} (and of the boundary one~\eqref{e.nudef}) actually {\it vanishes} for $\gamma'\geq 2$:  $\lim_{\eps \to 0} \eps^{\gamma'^{\,2}/2} e^{\gamma' h_\eps(z)}dz=0$. This is a quite general phenomenon, first observed by Kahane  in the eighties \cite{MR829798} in the case of the so-called {\it Gaussian multiplicative chaos}, inspired by  Mandelbrot's cascade model \cite{0227.76081}.

For the self-dual critical value, $\gamma=\gamma'=2$, the Liouville measure was recently constructed and shown to be non-atomic using the so-called derivative martingale \cite{DRSV1,MR3215583}.  For $\gamma' >2$, however, a precise mathematical description of Liouville duality requires additional probabilistic machinery, which we now describe. 
 
 We first recall that a forested wedge may be obtained by beginning with the tree of circles described by Figure~\ref{levytreegluing} (if one cuts out the gray shaded regions) and then gluing the boundary of an independent quantum disk (with the given boundary length) onto each of these circles.  That procedure produces something like a tree of quantum disks.

We next remark that there is a variant of this procedure that produces a ``tree'' of spheres (instead of a tree of disks).  This tree is constructed as a quotient of the tree of circles described by Figure~\ref{levytreegluing} (ignoring the gray fillings).  To explain how this works, let us note that we have already observed that one can put an equivalence relation on a circle in such a way that the quotient of the circle w.r.t.\ this equivalence relation is a topological sphere --- and the map from the original circle to this quotient space is a space-filling path on that sphere.  Precisely, given a parameterized circle of length $T$, one can map it into the peanosphere described by an excursion of a correlated Brownian motion into the interior of a quadrant (starting and ending at the corner of the quadrant), as in Figure~\ref{fig::lamination}, and this sphere may be understood as the quotient of the length $T$ circle described by the equivalences induced by the pair of Brownian motions.  If one applies a similar procedure to {\em each} of the circles in Figure~\ref{levytreegluing} (ignoring the gray fillings) then every one of the circles becomes a sphere, and the quotient of the entire tree of circles is a``tree of quantum spheres''  rather than a``tree of quantum disks.''  The individual branches of (spheres rooted at a given sphere) are sometimes called (in the physics literature) ``baby universes''
of \textit{dual} parameter $\gamma=4/\gamma',\,\gamma <2< \gamma'$.

The paragraphs above (together with related discussion in introduction) comprise the first mathematically complete description of these trees of disks and spheres.  However, some closely related objects have been addressed mathematically (by two of the current authors   \cite{ds2009qg_prl,2008ExactMethodsBD}, among others, e.g., \cite{bjrv2013super_critical}), as we will now explain.  The ``tree of quantum spheres'' discussed above corresponds to what is described in the physics literature as a ``pinched surface.''  It is often natural to single out one of the spheres in this tree and call it ``principal bubble.'' There are a countable number of ``pinch points'' on the principle bubble, i.e., vertices whose removal would disconnect the principal bubble from a ``branch'' of the tree of spheres. One obtains an atomic measure on the principle bubble by assigning to each of these pinch points an atom whose mass is the length of the original L\'evy excursion corresponding to that branch.  (This is one measure of the ``size'' of the tree branch.) If the principal bubble is conformally parameterized by the Riemann sphere $\C \cup \{ \infty \}$ then this procedure induces an atomic measure on $\C$.

Under this perspective, one way to rigorously construct  the singular quantum measures $\mu_{\gamma'}$ with $\gamma' > 2$ (without developing the entire trees of disks or spheres, as we do here) has been presented in  \cite{ds2009qg_prl,2008ExactMethodsBD}.  First, we note that the principal bubble comes endowed with a measure $\mu_{\gamma}$, $\gamma=4/\gamma' <2$, and if the principal bubble is parameterized by $\C \cup \{\infty \}$ we may interpret $\mu_\gamma$ as a measure on $\C$.  Then, given $\mu_\gamma$, we produce an atomic measure $\mu_{\gamma'}$ on $\C$ by choosing the pinch points $\{z_i : i \in \N\}$ at which finite amounts of  quantum area $\mu_{\gamma'}(z_i)$ are to be \textit{localized}. The pair $\{(\mu_{\gamma'}(z_i), z_i) : i \in \N\}$ is distributed as a Poisson point process (\ppp) $\mathcal N_{\gamma'}(du,dz)$ on $\R_+ \times \C$, with intensity measure $\vartheta_\theta \otimes \mu_{\gamma}$, where $\vartheta_\theta(du) := u^{-1-\theta}du$ and $du$ denotes Lebesgue measure on $\R_+$, with $\theta:=4/\gamma'^{\,2} \in (0,1)$. The quantum measure for $\gamma'>2$ is then {\it purely atomic}, 
\begin{equation*}\mu_{\gamma'}(dz):=\int_0^{\infty}u\, \mathcal N_{\gamma'} (du,dz),\,\,\, \gamma' >2.
\end{equation*}
A slightly different, but equivalent construction was proposed  in the context of Gaussian multiplicative chaos  \cite{bjrv2013super_critical}.

Section~\ref{subsec::bm_sle} will add more detail to the explanation, provided in Section~\ref{subsec::matingsandloops}, of the fact that cutting an appropriate quantum wedge with a counterflow line produces a L\'evy tree of quantum disks.  We will see that this tree can be constructed from a totally asymmetric $\tfrac{\kappa'}{4}$-stable process.  (The exponent $\tfrac{\kappa'}{4}$ was mentioned but not derived in Section~\ref{subsec::matingsandloops}.)  This will imply that the lengths of the left and right boundaries of $\eta'$ evolve as independent totally asymmetric $\tfrac{\kappa'}{4}$-stable processes.  
We will use this machinery to prove several of our main results about gluings of forested wedges: namely, Theorems~\ref{thm::gluingtwoforestedlines}--\ref{thm::quantum_natural_zip_unzip_rough_statement}.

As mentioned above, forested wedges (along with the closely related trees of quantum spheres constructed there) correspond to the ``dual quantum gravity'' surfaces that have been constructed non-rigorously in the physics literature.  Section~\ref{subsec::kpz_duality} will explain how these objects are related to certain random atomic measures that have been mathematically constructed elsewhere (see \cite{bjrv2013super_critical}, as well as related work in \cite{ds2009qg_prl,2008ExactMethodsBD}).

\subsection{Brownian motion and $\SLE$}
\label{subsec::bm_sle}

We are now going to recall how to construct a L\'evy tree to encode the genealogy structure of the quantum disks which are cut out by an $\SLE_{\kappa'}$ process for $\kappa' \in (4,8)$.  In order to do so, we first need to derive the form of the process which describes the time-evolution of the left and right boundary lengths of such an object when drawn on top of a certain type of quantum wedge.

\begin{theorem}
\label{thm::sle_exploration_boundary_length}
Fix $\kappa' \in (4,8)$.  Suppose that $\eta'$ is an $\SLE_{\kappa'}$ process drawn on top of an independent quantum wedge~$\CW$ of weight $\tfrac{3\gamma^2}{2} - 2$.  Let $\qnt_u$ (resp.\ $(f_t)$) be the quantum natural time (resp.\ centered Loewner flow) for $\eta'$.  Let $A_t$ (resp.\ $B_t$) be the leftmost (resp.\ rightmost) point on $\R$ hit by $\eta'|_{[0,t]}$ and let $X_u$ (resp.\ $Y_u$) be the difference of the $\gamma$-LQG length of the segment of the outer boundary of $\eta'([0,\qnt_u])$ which is to the left (resp.\ right) of $\eta'(\qnt_u)$ minus the $\gamma$-LQG length of $(A_{\qnt_u},0]$ (resp.\ $[0,B_{\qnt_u})$).  Then $(X_u,Y_u)$ evolves as a pair of independent totally asymmetric $\tfrac{\kappa'}{4}$-stable processes.
\end{theorem}

Before we start to give the proof of Theorem~\ref{thm::sle_exploration_boundary_length}, we first give the following corollary which gives the time-evolution of the boundary length of the unbounded component when one explores a weight $\tfrac{\gamma^2}{2}$ quantum cone with a whole-plane $\SLE_{\kappa'}$ process, parameterized by quantum natural time.

\begin{corollary}
\label{cor::radial_boundary_length}
Fix $\kappa' \in (4,8)$ and suppose that $\CC = (\C,h,0,\infty)$ is a quantum cone of weight~$\tfrac{\gamma^2}{2}$.  Let~$\eta'$ be a whole-plane $\SLE_{\kappa'}$ process from~$0$ to~$\infty$ which is independent of~$h$ and let~$\qnt_u$ be the quantum natural time for~$\eta'$.  Then the $\gamma$-LQG boundary length of the boundary of the unbounded component of $\C\setminus \eta'([0,\qnt_u])$ evolves in $u$ as a totally asymmetric $\tfrac{\kappa'}{4}$-stable process conditioned to be non-negative.
\end{corollary}
\begin{proof}
Theorem~\ref{thm::sle_exploration_boundary_length} implies that if we perform a chordal $\SLE_{\kappa'}$ exploration on top of a weight $\tfrac{3\gamma^2}{2}-2$ quantum wedge with the quantum natural time parameterization then the change in the left and right boundary lengths evolve as independent totally asymmetric $\tfrac{\kappa'}{4}$-stable processes.  By taking the sum, this implies that the change in the total boundary length evolves as a totally asymmetric $\tfrac{\kappa'}{4}$-stable process.  This, in turn, implies that if we perform reverse rather than forward $\SLE_{\kappa'}$ with the quantum natural time parameterization, then the change in the total boundary length evolves as a totally asymmetric $\tfrac{\kappa'}{4}$-stable process (but now with only positive rather than only negative jumps).  In the setting of a weight $\tfrac{3\gamma^2}{2}-2$ quantum wedge, we know that the evolution of the boundary length is determined by the bubbles cut out by the $\SLE_{\kappa'}$ (since the bubbles determine the jumps and the jumps determine the boundary length process since it is a L\'evy process without a Gaussian part).  By the local absolute continuity of the behavior of a whole-plane $\SLE_{\kappa'}$ on top of an independent weight $\tfrac{\gamma^2}{2}$ quantum cone with respect to that of a chordal $\SLE_{\kappa'}$ on top of a weight $\tfrac{3\gamma^2}{2}-2$ quantum wedge, the change in the boundary length of the former is determined from the bubbles in the same manner as the latter.  By Theorem~\ref{thm::sle_kp_quantum_local_typical}, the structure of the bubbles cut out by a whole-plane $\SLE_{\kappa'}$ in the setting described in the statement of the corollary is the same as the structure of the bubbles cut out by a chordal $\SLE_{\kappa'}$.  Therefore, in the setting of the last part of Theorem~\ref{thm::kappa_prime_bubbles} the length of $\partial \D$ evolves as a totally asymmetric $\tfrac{\kappa'}{4}$-stable process when one runs a reverse $\SLE_{\kappa'}$ radial Loewner flow parameterized by a quantum natural time.  The result for whole-plane $\SLE_{\kappa'}$ on a weight $\tfrac{\gamma^2}{2}$ quantum cone follows by combining this with the time-reversal result from \cite[Theorem~18, Chapter~7]{bertoin96levy}. 
\end{proof}

As mentioned in Section~\ref{subsubsec::discrete_intuition}, Theorem~\ref{thm::sle_exploration_boundary_length} and Corollary~\ref{cor::radial_boundary_length} are consistent with other conjectures and results in the literature \cite{angel2003growth,krikun2005uipq,legall2014icm} in the context of random planar maps and the Brownian plane.

We will derive Theorem~\ref{thm::sle_exploration_boundary_length} from the results of Section~\ref{sec::brownian_boundary_length} and a general result about planar Brownian motion (Proposition~\ref{prop::bm_pinch_times} below).

Fix a value of $\theta \in [\tfrac{\pi}{2},\pi)$ and let  $p = -\cos(\theta)/(1-\cos(\theta)) \in [0,\tfrac{1}{2})$ (recall~\eqref{eqn::p_in_terms_of_theta}).  Suppose that $Z =(L,R)$ is a planar Brownian motion with $\var(L_1) = \var(R_1) = \tfrac{1}{2}-\tfrac{p}{2}$ and $\cov(L_1,R_1) = \tfrac{p}{2}$.  We say that a time $s$ is an ancestor of a time $t$ if for all $r \in (s,t]$ we have that $L_r > L_s$ and $R_r > R_s$.  If a time $t$ has an ancestor $s$, then we will also refer to $t$ as being {\bf pinched}.  Note that if a time $t$ is pinched, then the ancestor time $s$ is a $\tfrac{\pi}{2}$-cone time for $(L,R)$ as defined in Section~\ref{sec::brownian_boundary_length} just before Lemma~\ref{lem::brownian_covariance}.  If $t$ is not pinched, then we say that $t$ is {\bf ancestor free}.  Note that the set of times which have an ancestor (resp.\ are ancestor free) is open (resp.\ closed) in $[0,\infty)$.  In particular, we can express the times with an ancestor as a countable disjoint union of open intervals $I_j$; we will refer to such an interval as a pinched interval.

\begin{proposition}
\label{prop::bm_pinch_times}
Suppose that $Z$ is as described just above.  There exists a non-decreasing RCLL process $\qlt$ which is adapted to the filtration generated by $Z$ and is a.s.\ constant on the pinched intervals such that the following is true.  Let $T_u = \inf\{t \geq 0 : \qlt_t > u\}$ be the right-continuous inverse of $\qlt$.  Then $L_{T_u}$ and $R_{T_u}$ are independent, totally asymmetric $\tfrac{\pi}{\theta}$-stable processes.
\end{proposition}

The process $\qlt$ constructed in Proposition~\ref{prop::bm_pinch_times} should be thought of as the local time for $Z$ in the ancestor free times.  We note that Proposition~\ref{prop::bm_pinch_times} implies Proposition~\ref{prop::iidlevy}.

\begin{proof}[Proof of Proposition~\ref{prop::bm_pinch_times}]
For each $\alpha \in [0,2\pi)$, we let $\W_\alpha = \{ z \in \C : \arg(z) \in [0,\alpha]\}$ be the Euclidean wedge in $\C$ of opening angle $\alpha$ centered at~$0$.  Note that a time $t$ is pinched if and only if there exists $s \in (0,t]$ such that $Z_r \in \W_{\pi/2} + Z_s$ for all $r \in (s,t]$.

It will be useful for us now to perform a change of coordinates.  Let
\[  \wt{Z} =  \frac{1}{\cos(\theta/2)} \begin{pmatrix} 1 & \cos(\theta) \\ 0 & \sin(\theta) \end{pmatrix} Z.\]
Then $\wt{Z}$ is a standard two-dimensional Brownian motion.  Moreover, $t$ is a pinched time for $Z$ if and only if there exists $s \in (0,t]$ such that $\wt{Z}_r \in \W_\theta + \wt{Z}_s$ for all $r \in (s,t]$.

We are now going to identify the Poissonian structure of the pinched excursions.  Fix $\epsilon > 0$ and then inductively define times as follows.  We let
\[ \tau_{1,\epsilon} = \inf\left\{s \geq 0 : \exists t \geq s : \wt{Z}_u \in \W_\theta  + \wt{Z}_s \quad \forall u \in [s,t] \quad\text{and}\quad |\wt{Z}_t - \wt{Z}_s|^{\pi/\theta} \geq \epsilon \right\}\]
and
\[ \sigma_{1,\epsilon}= \inf\left\{t \geq \tau_{1,\epsilon} : \wt{Z}_t \notin \W_\theta + \wt{Z}_{\tau_{1,\epsilon}} \right\}.\]
Given that $\tau_{1,\epsilon},\sigma_{1,\epsilon},\ldots,\tau_{k,\epsilon},\sigma_{k,\epsilon}$ have been defined for some $k \in \N$ we let
\[ \tau_{k+1,\epsilon}= \inf\left\{s \geq \sigma_{k,\epsilon} : \exists t \geq s : \wt{Z}_u \in \W_\theta  + \wt{Z}_s \quad \forall u \in [s,t] \quad\text{and}\quad |\wt{Z}_t - \wt{Z}_s|^{\pi/\theta} \geq \epsilon \right\}\]
and
\[ \sigma_{k+1,\epsilon}= \inf \left\{t \geq \tau_{k+1,\epsilon} : \wt{Z}_t \notin \W_\theta + \wt{Z}_{\tau_{k+1,\epsilon}} \right\}.\]
We note that the $\tau_{k,\epsilon}$ are not stopping times but the $\sigma_{k,\epsilon}$ are stopping times.

Let 
\[ \Delta_{j,\epsilon} = |\wt{Z}_{\sigma_{j,\epsilon}} - \wt{Z}_{\tau_{j,\epsilon}}| = \frac{1}{\cos(\theta/2)}|Z_{\sigma_{j,\epsilon}} - Z_{\tau_{j,\epsilon}}|.\]  We also let $\xi_{j,\epsilon} = 1$ if $\wt{Z}_{\sigma_{j,\epsilon}}$ is contained in the horizontal component of $\partial \W_\theta \setminus \{0\} + \wt{Z}_{\tau_{j,\epsilon}}$, otherwise we let $\xi_{j,\epsilon} = -1$.  Equivalently, $\xi_{j,\epsilon} = 1$ (resp.\ $\xi_{j,\epsilon} = -1$) if $Z_{\sigma_{j,\epsilon}} - Z_{\tau_{j,\epsilon}}$ has zero imaginary (resp.\ real) part.  With $\Sigma_{j,\epsilon} = \xi_{j,\epsilon} \Delta_{j,\epsilon}$, the signed jump sizes $(\Sigma_{j,\epsilon})$ are i.i.d.\ by the strong Markov property for $Z$.  (The law of the lengths of the excursions is described, for example, in \cite{spi1958bm}.)

We can describe the law of the $\Sigma_{j,\epsilon}$ as follows.  For each $k \geq 1$ we let
\[ \zeta_{k,\epsilon} = \inf\left\{t \geq \tau_{k,\epsilon} : |\wt{Z}_t - \wt{Z}_{\tau_{k,\epsilon}}|^{\pi/\theta} \geq \epsilon \right\}.\]
Let $\mu$ denote the law of $\arg\big((\wt{Z}_{\zeta_{k,\epsilon}} - \wt{Z}_{\tau_{k,\epsilon}})^{\pi/\theta}\big)$.  Then $\mu$ is supported on $(0,\pi)$.  By the scale-invariance of Brownian motion, we have that $\mu$ does not depend on $\epsilon$.  Recall that the Poisson kernel on $\h$ is given by $P_y(x) = y / (\pi(x^2+y^2))$.  Thus if $\theta = \pi$, the law of $\Sigma_{j,\epsilon}$ has density with respect to Lebesgue measure on $\R$ which is proportional to
\[ \left(\int_0^\pi \frac{\im(e^{i \psi}) }{(u-\epsilon \re(e^{i\psi}))^2 + \epsilon^2 \im(e^{i \psi})^2} d\mu(\psi)\right) du.\] 
For general values of $\theta$, we can use the conformal invariance of Brownian motion and the map $z \mapsto z^{\pi/\theta}$ to determine the law of the $\Sigma_{j,\epsilon}$.  In particular, the law of $\Sigma_{j,\epsilon}$ is proportional to
\begin{equation}
\label{eqn::bm_approx_levy_measure}
\varsigma_{\theta,\epsilon}(du) = \left(\int_0^\theta \frac{u^{\pi/\theta -1} \im(e^{i\psi})}{(u^{\pi/\theta} - \epsilon \re(e^{i \psi}))^2 + \epsilon^2 \im(e^{i \psi})^2} d\mu(\psi)\right) du.
\end{equation}
Let
\begin{equation}
\label{eqn::bm_levy_measure}
\varsigma_{\theta}(du) = u^{-1-\pi/\theta} du.
\end{equation}
For each $u_0 > 0$ we note that both $\varsigma_{\theta,\epsilon}$ and $\varsigma_\theta$ give finite mass to $A_{u_0} = \{u \in \R : |u| \geq u_0\}$.  Since $\mu(\{0\}) = \mu(\{\pi\}) = 0$, it follows that for each $u_0 > 0$ the conditional law of $\varsigma_{\theta,\epsilon}$ given $A_{u_0}$ (i.e., the restriction of $\varsigma_{\theta,\epsilon}$ to $A_{u_0}$ normalized to be a probability measure) converges in total variation as $\epsilon \to 0$ to the conditional law of $\varsigma_\theta$ given $A_{u_0}$.

Let $\Lambda$ be a \ppp\ on $\R_+ \times \R$ with intensity measure $du \otimes \varsigma_\theta$ where $du$ denotes Lebesgue measure on $\R_+$.  It therefore follows that for each $u_0 > 0$, we can couple together the pinched excursions of $Z$ with displacement at least $u_0$ with the elements $(u,\Sigma) \in \Lambda$ with $|\Sigma| \geq u_0$ to be the same.  Since this holds for each $u_0 > 0$, we can identify the pinched excursions of $Z$ with the elements of~$\Lambda$.

If $(u,\Sigma) \in \Lambda$, then $u$ gives the value of $\qlt$ in the pinched interval corresponding to~$\Sigma$ and it is clear that we can extend $\qlt$ to be RCLL on $\R_+$.  It follows from the law of large numbers argument used to prove \cite[Proposition~19.12]{KAL_FOUND} that $\qlt$ is in fact adapted to the filtration generated by $Z$.  In particular, $\qlt$ is construct as an a.s.\ limit.

Recall that $(u,\Sigma) \in \Lambda$ corresponds to a jump of size $|\Sigma|$ in $L$ (resp.\ $R$) if $\Sigma > 0$ (resp.\ $\Sigma < 0$).  Let $T_u$ be the right-continuous inverse of $\qlt$ as in the statement of the proposition.  Then $L_{T_u}$ and $R_{T_u}$ are L\'evy processes with the same law.  They are both totally asymmetric $\tfrac{\pi}{\theta}$-stable since they do not have a Gaussian part, only have negative jumps, and the L\'evy measure for the jump part of each is given by restricting $\varsigma_\theta$ to $\R_-$.  Moreover, $L_{T_u}$ and $R_{T_u}$ are independent since their jump parts are independent.
\end{proof}

Recall from Section~\ref{sec::brownian_boundary_length} that we can think of $(L,R)$ as being the two-dimensional Brownian motion which encodes the change in the boundary lengths of the left and right sides of an independent space-filling $\SLE_{\kappa'}$ process $\eta'$ drawn on top of a $\gamma$-quantum cone with $\gamma=4/\sqrt{\kappa'}$.  In this context, the pinched excursions come with a notion of ``area'' (namely, the length of the excursion) and ``boundary length'' (the displacement of whichever of $L$ and $R$ is not at the same place at the two interval endpoints).  The ``boundary length'' is precisely the size of the jump of the stable process from Proposition~\ref{prop::bm_pinch_times}.  Using the notion of a L\'evy tree, there is also a partial order on these jumps corresponding to the ancestor/descendant relationship.  We want to use this to put a tree structure on the bubbles cut off by an $\SLE_{\kappa'}$ process.  To do so, we first need to relate the range of an $\SLE_{\kappa'}$ process with the ancestor free times of $Z$.

Consider the weight $2-\gamma^2/2$ (thin) quantum wedge which is parameterized by $\eta'([0,\infty))$.  By the definition of space-filling $\SLE_{\kappa'}$, this is the region between the flow lines of the whole-plane GFF used to generate $\eta'$ with angles $-\tfrac{\pi}{2}$ and $\tfrac{\pi}{2}$ starting from $0$.  We then let $\wt{\eta}'$ be the counterflow line from $\infty$.  By \cite[Theorem~1.14]{ms2013imag4}, the conditional law of $\wt{\eta}'$ given $\eta'([0,\infty))$ (i.e., the outer boundary of $\wt{\eta}'$) is independently that of an $\SLE_{\kappa'}(\tfrac{\kappa'}{2}-4;\tfrac{\kappa'}{2}-4)$ in each of the bubbles of $\eta'([0,\infty))$.

\begin{lemma}
\label{lem::exposed_cfl_relationship}
Let $\wt{\eta}'$ be the counterflow line from $\infty$ to~$0$ which travels through the thin wedge which is parameterized by $\eta'([0,\infty))$ as described above.  Then a point of $\eta'([0,\infty))$ is in the range of $\wt{\eta}'$ if and only if it is visited by $\eta'$ at an ancestor free time for $Z$.
\end{lemma}
\begin{proof}
We will prove the lemma by showing that a point $z \in \eta'([0,\infty))$ is not in the range of $\wt{\eta}'$ if and only if $t$ is a pinch time for $Z$.  (See Figure~\ref{fig::cone_time} and Figure~\ref{fig::ancestor_free} for an illustration.)

Suppose that $t$ is a time such that $\eta'(t)$ is not in the range of $\wt{\eta}'$.  Then $\eta'(t)$ is contained in a bubble $B$ which is pinched off by $\wt{\eta}'$.  Let $\sigma$ be the time that $\wt{\eta}'$ pinches off $B$ (i.e., disconnects $B$ from its target point) and let $s$ be the time that $\eta'$ first visits $\wt{\eta}'(\sigma)$.  Then it is easy to see from the definition of space-filling $\SLE_{\kappa'}$ that $L_t > L_s$ and $R_t > R_s$.  In fact, if $u \in (s,t)$ then $L_u > L_s$ and $R_u > R_s$.  That is, $t$ is a pinch time for $(L,R)$.

Conversely, suppose that $t$ is a pinch time for $Z$ and let $s$ be the left endpoint of the corresponding pinched interval.  Then $s$ is a local cut time for $\eta'$.  Consequently, the time-reversal of $\eta'$ will hit $\eta'(s)$ twice hence separate $\eta'(t)$ from $\infty$.  Therefore $\wt{\eta}'$ will not visit $\eta'(t)$.
\end{proof}

\begin{remark}
\label{rem::equivalence_of_disk_measures}
Observe that if draw an independent $\SLE_{\kappa'}(\tfrac{\kappa'}{2}-4;\tfrac{\kappa'}{2}-4)$ process on top of a wedge of weight $2-\tfrac{\gamma^2}{2}$ (more precisely, one independent process for each bubble of the wedge), then we have two equivalent Poissonian descriptions for the bubbles cut out of the wedge by the path.  The first is given in terms of the cone-excursions of a two-dimensional Brownian motion and the second is given in terms of the excursion measure of a Bessel process.  Indeed, this first description follows from Proposition~\ref{prop::bm_pinch_times} above and the second follows by combining Proposition~\ref{prop::slice_wedge_many_times} and Theorem~\ref{thm::sle_kp_quantum_local_typical}.
\end{remark}

Combining Lemma~\ref{lem::exposed_cfl_relationship} with Proposition~\ref{prop::bm_pinch_times} and Proposition~\ref{prop::unit_area_bl_construction} we get that gluing together two forested lines yields a $2-\tfrac{\gamma^2}{2}$ wedge.  We record this result in the following proposition.

\begin{proposition}
\label{prop::glue_forested_lines}
Suppose that~$\CW$ is a wedge of weight $2-\tfrac{\gamma^2}{2}$.  If we cut~$\CW$ with a concatenation~$\eta'$ of independent $\SLE_{\kappa'}(\tfrac{\kappa'}{2}-4;\tfrac{\kappa'}{2}-4)$ processes (one for each bead of~$\CW$), then the bubbles which are to the left (resp.\ right) of the path have the structure of two independent forested lines.  Moreover, the topology of the surface~$\CW$ decorated by~$\eta'$ is equivalent to the topological gluing of independent forested lines illustrated in Figure~\ref{fig::fig8}.  In particular, $\eta'$ is a continuous path when parameterized by the quantum natural time parameterization.
\end{proposition}

\begin{proof}
As explained above, Lemma~\ref{lem::exposed_cfl_relationship} and Proposition~\ref{prop::bm_pinch_times} (together with Theorem~\ref{thm::topologyiscorrect}) implies that the topology of the bubbles cut out by $\eta'$ on its left and right sides are given by independent forested lines.   Note that the right-continuous inverse constructed in Proposition~\ref{prop::bm_pinch_times} is monotone hence has at most countably many jumps.  By the definition of the quantum natural time parameterization, these jumps are the only possible source of discontinuity for $\eta'$ parameterized in this way.  By Theorem~\ref{thm::topologyiscorrect}, these jumps correspond to the macroscopic bubbles cut off by $\eta'$.  Since the beginning and ending point of such a bubble are (by definition) the same, it follows that $\eta'$ is continuous with respect to the quantum natural time parameterization.  It similarly follows from Lemma~\ref{lem::exposed_cfl_relationship},  Proposition~\ref{prop::bm_pinch_times}, and Theorem~\ref{thm::topologyiscorrect} that the topology of $\CW$ decorated by $\eta'$ is equivalent to the topological gluing of independent forested lines.
\end{proof}

Proposition~\ref{prop::glue_forested_lines} gives the first part of Theorem~\ref{thm::gluingtwoforestedlines}.  To finish proving Theorem~\ref{thm::gluingtwoforestedlines}, we need to show that the pair of forested lines a.s.\ determines the wedge of weight $2-\tfrac{\gamma^2}{2}$.  We are going to deduce this result from the main result of Section~\ref{sec::trees_determine_embedding}, which implies that the entire wedge of weight $2-\tfrac{\gamma^2}{2}$ which is parameterized by $\eta'([0,\infty))$ is a.s.\ determined by the pair $(L,R)$ of Brownian motions for $t \geq 0$.  In order to do so, we first need to collect the following observation.

\begin{proposition}
\label{prop::abc}
Let $A$, $B$ and $C$ be random variables defined on a common probability space.  Suppose that $A$ and $B$ together a.s.\ determine $C$.  (That is, suppose there exists a function $f$ such that $C = f(A,B)$ a.s.)  Suppose further that given $A$, the values $B$ and $C$ are conditionally independent.  Then $A$ a.s.\ determines $C$.
\end{proposition}

We are now going to give the proofs of Theorems~\ref{thm::gluingtwoforestedlines}--\ref{thm::quantum_natural_zip_unzip_rough_statement}.

\begin{proof}[Proof of Theorem~\ref{thm::gluingtwoforestedlines}]
As we mentioned earlier, the first part of Theorem~\ref{thm::gluingtwoforestedlines} is given in Proposition~\ref{prop::glue_forested_lines}.  To finish the proof, we need to show that the pair of forested lines a.s.\ determine the zipping.  Let $A$ be the pair of forested lines, which we view as arising by cutting a weight $2-\tfrac{\gamma^2}{2}$ wedge with a concatenation of independent $\SLE_{\kappa'}(\tfrac{\kappa'}{2}-4;\tfrac{\kappa'}{2}-4)$ processes $\wt{\eta}'$.  As described above, we can view $\wt{\eta}'$ as corresponding to the set of ancestor free times of a space-filling $\SLE_{\kappa'}$ process $\eta'$.  Let $B$ be the countable collection of paths obtained by restricting $\eta'$ to each of the countably many disks of $A$.  Let $C$ be the entire quantum surface traced by $\eta'$.  It is not hard to see that $A$ and $B$ together determine the boundary length processes (i.e., the pair of Brownian motions) that determine $C$.  By construction, we have that $B$ and $C$ are conditionally independent given $A$.  The result then follows from Proposition~\ref{prop::abc}.
\end{proof}

\begin{proof}[Proof of Theorem~\ref{thm::sle_kp_on_wedge}]
The theorem in the case that $\kappa' \geq 8$ is a consequence of Proposition~\ref{prop::slice_wedge_many_times}, so we will focus on the case that $\kappa' \in (4,8)$ (hence $\gamma \in (\sqrt{2},2)$).

Theorem~\ref{thm::gluingtwoforestedlines} gives the result for forested lines, i.e.\ when $W_1 = W_2 = 0$.  We are now going to generalize the result to the case that $W_1,W_2 \geq 0$.  Fix $\rho_1,\rho_2 \geq \tfrac{\kappa'}{2}-4$ and let $W_i = \gamma^2 - 2 + \tfrac{\gamma^2}{4} \rho_i \geq 0$.  Suppose that $\CW = (\h,h,0,\infty)$ is a wedge of weight $W = W_1 + W_2 + 2-\tfrac{\gamma^2}{2}$.  Further, suppose that $\eta_1$ (resp.\ $\eta_2$) is the flow line starting from $0$ and targeted at $\infty$ of a GFF on $\h$ with boundary conditions given by~$0$ on~$\R_-$ and $(W-2)\lambda$ on~$\R_+$ with angle $\theta_1 = \tfrac{\lambda}{\chi}(1-W_1)$ (resp.\ $\theta_2 = -\tfrac{\lambda}{\chi}(1 - \tfrac{\gamma^2}{2}+W_1)$).  Proposition~\ref{prop::slice_wedge_many_times} implies that the region which is to the left (resp.\ right) of $\eta_1$ (resp.\ $\eta_2$), viewed as a quantum surface, is a wedge of weight~$W_1$ (resp.\ $W_2$).  Moreover, the region between~$\eta_1$ and~$\eta_2$, viewed as a quantum surface, is a wedge of weight $2-\tfrac{\gamma^2}{2}$ and these three wedges are independent.  The counterflow line $\eta'$ of this GFF plus $\pi(\tfrac{\gamma}{4} - \tfrac{W_1}{\gamma})$ from~$\infty$ is marginally an $\SLE_{\kappa'}(\rho_1;\rho_2)$ process and, by \cite[Proposition~7.30]{ms2012imag1}, the conditional law of $\eta'$ given $\eta_1$ and $\eta_2$ is independently that of an $\SLE_{\kappa'}(\tfrac{\kappa'}{2}-4;\tfrac{\kappa'}{2}-4)$ process in each of the bubbles between~$\eta_1$ and~$\eta_2$.  (The reason that we have to add an angle to the field in the definition of~$\eta'$ is that we have normalized the boundary data to be consistent with Proposition~\ref{prop::slice_wedge_many_times}.)  Theorem~\ref{thm::gluingtwoforestedlines} implies that these bubbles have the structure of independent forested lines and, moreover, these forested lines a.s.\ determine the middle wedge of weight $2-\tfrac{\gamma^2}{2}$.  Combining everything proves the result.
\end{proof}

\begin{proof}[Proof of Theorem~\ref{thm::quantum_cone_sle_kp}]
This follows by applying Proposition~\ref{prop::slice_cone_many_times} with a pair of GFF flow lines with an angle gap of $\pi$ and then applying Theorem~\ref{thm::gluingtwoforestedlines} to the resulting wedge of weight $2-\tfrac{\gamma^2}{2}$.
\end{proof}

\begin{proof}[Proof Theorem~\ref{thm::quantum_natural_zip_unzip_rough_statement}]
The first part of the result is a consequence of Theorem~\ref{thm::sle_kp_on_wedge}.  The second part of the theorem is given in Theorem~\ref{thm::sle_kp_quantum_local_typical}.
\end{proof}

\begin{proof}[Proof of Theorem~\ref{thm::sle_exploration_boundary_length}]
This result follows from Corollary~\ref{cor::stablelength}, whose derivation from Theorem~\ref{thm::quantum_natural_zip_unzip_rough_statement} is explained just after the statement of Theorem~\ref{thm::quantum_natural_zip_unzip_rough_statement}.
\end{proof}

\subsection{Surfaces with bottlenecks and atomic measures}
\label{subsec::kpz_duality}

The rigorous construction of the singular quantum measures $\mu_{\gamma'}$ with $\gamma' > 2$ uses the dual measure $\mu_{\gamma}$, $\gamma:=4/\gamma'<2$, and the additional randomness of a set of point masses $\{z_i : i \in \N\}$ such that $\{(\mu_{\gamma'}(z_i), z_i) : i \in \N\}$ is distributed as a \ppp  \ $\mathcal N_{\gamma'} (du,dz)$  on $\R_+ \times D$ with intensity measure $\vartheta_\theta \otimes \mu_{\gamma}$, where $\vartheta_\theta(du) := u^{-1-\theta}du$ and $du$ denotes Lebesgue measure on $\R_+$ and $\theta:=4/\gamma'^{\,2}= \gamma^2/4\in (0,1)$ \cite{ds2009qg_prl,2008ExactMethodsBD}. Each point $(u,z)$ represents an atom of size $u$ located at $z$, and the quantum measure for $\gamma'>2$ is then {\it purely atomic},
\begin{equation}
\label{e.dualmeasure}
\mu_{\gamma'}(dz):=\int_0^{\infty}u\, \mathcal N_{\gamma'} (du,dz),\quad \gamma' >2.
\end{equation}
In \cite{bjrv2013super_critical},  a slightly different approach was proposed, the so-called ``atomic Gaussian multiplicative chaos''.  After regularization, its limit coincides  in the GFF case with the present construction.

For any subdomain $U \subseteq D$, the  Laplace transforms of the  atomic $\gamma'$-quantum measure~\eqref{e.dualmeasure}
  and the dual $\gamma$-quantum measure~\eqref{e.mudef} then obey the {\it duality identity} \cite[Section~18.6.2]{2008ExactMethodsBD},
\begin{equation}
\label{Laplace}
\E\left[ \exp(-\phi\,\mu_{\gamma'}(U)) \right] = \E\left[\exp(\Gamma(-\theta)\phi^\theta\mu_{\gamma}(U)) \right],\quad \gamma\gamma'=4,\quad \forall \phi\geq 0.
\end{equation}
This is a L\'evy-Khintchine formula for an independently scattered stable process valid for all $\phi \in \R_+$, and where $\Gamma(-\theta)=-\Gamma(1-\theta)/\theta<0$  is the usual Euler $\Gamma$-function. 
It  can also be seen as the probabilistic formulation of the
 \textit{Legendre transform} relating the free energies of dual Liouville theories, as first observed in \cite{kleb1995non_perturb}.

The conjugate variable $\phi$ to the quantum area measure $\mu_{\gamma'}$ is often called ``cosmological constant'' in Liouville quantum gravity \cite{MR623209,Ginsparg-Moore,MR1320471,MR2073993}. The duality identity~\eqref{Laplace} implies that the dual cosmological constant of $\mu_{\gamma}$ is  $\phi':=-\Gamma(-\theta)\phi^\theta$; in other words, given $\mu_{\gamma}$, the ``typical size'' (say, the median size) of $\mu_{\gamma'}(U)$ is of order $\mu_{\gamma}(U)^{\gamma'^{\,2}/4}=\mu_{\gamma}(U)^{4/\gamma^2}$.
Consider indeed the conditional exponential expectation associated with~\eqref{Laplace},
\begin{equation*}
\label{Laplaceconditional}
\E\left[ \exp(-\phi \mu_{\gamma'}(U)) \giv \mu_{\gamma}(U)\right] = \exp\big(\Gamma(-\theta)\phi^\theta \mu_{\gamma}(U)\big).
\end{equation*}
This yields  the moment formula (see also \cite{bjrv2013super_critical}),
\[
 \E\left[\big(\mu_{\gamma'}(U)\big)^p \giv \mu_{\gamma}(U)\right]=\frac{\Gamma(1-\tfrac{p}{\theta}) \big(-\Gamma(-\theta) \big)^{p/\theta}}{\Gamma(1-p)}\big(\mu_{\gamma}(U)\big)^{p/\theta},\,\,\,0\leq p<\theta=\frac{4}{\gamma'^{\,2}} <1,
 \]
which precisely illustrates the announced scaling relation existing between dual Liouville measures.

We are now going to explain how to view these measures for $\gamma'^{\,2} \in (4,8)$ as corresponding to trees of surfaces using L\'evy trees.  We first need to recall a few facts about L\'evy processes from \cite{bertoin96levy}.  Fix $\alpha \in (1,2)$.  Suppose that $X$ is a totally asymmetric $\alpha$-stable process with $\alpha \in (1,2)$ and let $I_t = \inf_{0 \leq s \leq t} X_s$ be the running infimum of $X$.  Then $X_t - I_t$ is a Markov process \cite[Proposition~1, Chapter~6]{bertoin96levy} with a local time $\qlt_t$ at~$0$.  In fact, $\qlt_t = I_0-I_t$.  The right-continuous inverse $T_u = \inf\{t > 0 : \qlt_t > u\}$ of $\qlt_t$ is called the ladder time process.  By \cite[Lemma~1, Chapter~8]{bertoin96levy}, we have that $T_u$ is a stable subordinator of index $1/\alpha$.  This implies that if $X_0 = v$ for $v > 0$, $u < v$, and $\tau = \inf\{t \geq 0 : X_t \leq u\}$, then we can sample the lengths of the excursions of $X-I$ from~$0$ in $[0,\tau]$ as follows.  Let $\Lambda$ be a \ppp\ on $[u,v] \times \R_+$ with intensity measure given by $dt \otimes \vartheta_{\alpha^{-1}}$ where $dt$ denotes Lebesgue measure on $[u,v]$.  Then the elements of $\Lambda$ are in correspondence with the excursions of $X-I$ from~$0$.  In particular, $(w,\Sigma) \in \Lambda$ corresponds to an excursion of $X-I$ from~$0$ where the value of $I$ at the start and end of the excursion is given by $w$ and $\Sigma$ gives its length.

To give the L\'evy tree construction of the surfaces with $\gamma'^{\,2} \in (4,8)$, we now suppose that $X$ is a totally asymmetric $\tfrac{\gamma'^{\,2}}{4}$-stable process with $X_0 = u$ for some $u > 0$ and let $I$ be the running infimum of $X$.  Consider the tree of quantum surfaces which correspond to the L\'evy tree generated by $X$ (recall Figure~\ref{levytreegluing}).  In this construction, we associate with each jump of $X$ a conditionally independent $\gamma$-LQG sphere, $\gamma=4/\gamma'$, as in Definition~\ref{def::finite_volume_surfaces} with $\gamma$-LQG mass given by the size of the jump; we view the vertical segment from $0$ to $X_0$ as the first jump of $X$.  We think of the surface $S$ associated with the first jump as being the principle bubble of the $\gamma'$-LQG surface.  An excursion of $X-I$ from $0$ is naturally associated with a root location on the vertical segment from $0$ to $X_0$ given by the value that $X$ takes on at the start of the excursion (which is also the same as the value taken on by $I$ at this time).  We then define the mass of such a point to be the length of time that it takes for $X-I$ to complete the corresponding excursion.  This defines a purely atomic measure on the vertical segment from $0$ to $X_0$.

We can use this purely atomic measure to define a purely atomic measure on $S$ as follows.  Pick a point $z_0 \in S$ from its $\gamma$-LQG area measure and then let $\eta'$ be a space-filling $\SLE_{\kappa'}$ process on $S$ starting from $z_0$ which is otherwise sampled conditionally independently of $S$.  We then reparameterize $\eta'$ according to $\gamma$-LQG area.  Then the map which takes $t \in [0,X_0]$ to $\eta'$ given by $t \mapsto \eta'(t)$ induces a measure preserving transformation from the vertical segment from $0$ to $X_0$ onto $S$.  It follows from the discussion in the previous paragraph that the induced point measure on $S$ is equal in law to the $\gamma'$-LQG measure, $\gamma'^{\,2} \in (4,8)$, described at the beginning of this subsection.  We also note that the L\'evy tree structure is the same as the L\'evy tree structure we deduced for the surfaces cut out by an $\SLE_{\kappa'}$ process as described in Section~\ref{subsec::bm_sle}.

\section{Open Questions}
\label{sec::questions}

First let us recall a few problems that have already been posed in this paper.  In the beginning of Section~\ref{sec::introduction}, we asked the following.

\begin{problem}
Can one define a metric space structure on the peanosphere directly?
\end{problem}

\begin{problem}
\label{prob::components}
Fix $\kappa' \in (4,8)$.  Consider the random graph whose vertices are the components of the complement of an $\SLE_{\kappa'}$ curve, where two components are considered adjacent if their boundaries intersect.  Is this graph a.s.\ connected?  In other words, is it a.s.\ the case that for any two such components $U$ and $V$ there exists a finite sequence of adjacent components $U_0, U_1,\ldots,U_n$ such that $U_0 = U$ and $U_n = V$?
\end{problem}

\noindent{{\it Update.}  Question~\ref{prob::components} has been solved for $\kappa' \in (4,\kappa_0']$ where $\kappa_0' \approx 5.62$ in \cite{gp2018components}.}

Question~\ref{prob::components} is the $\SLE$ version of a well-known open question posed by Y.\ Peres for the connectivity of the graph structure of the components $\C \setminus B([0,1])$ where $B$ is a standard Brownian motion in $\C$.

One very interesting question is to strengthen the topology of convergence to the peanosphere for certain discrete models, in order to understand the conformal structure of the discrete embedding.  We remark that the number of ways to frame this problem is at least $kmn$ if one has
\begin{enumerate}
\item $k$ random planar map models (triangulations, quadrangulations, planar maps with unrestricted face sizes, etc.)
\item $m$ ways to put a statistical physics structure on the planar map (uniform spanning tree, Ising model, three state Potts model, critical percolation, double dimer model, etc.)
\item $n$ ways to embed a discrete graph ``conformally'' in the plane (circle packing, square tiling, discrete analytic mapping,  Riemann conformal mapping of surface obtained by gluing unit polygons, etc.)
\end{enumerate}
Solving even one of these formulations would be a significant breakthrough, and given the flexibility, one can reasonably argue that the variant that is easiest to solve is the most natural one to consider.  
A variant of that problem would be to show that a simple random walk on the graph of $\delta$-area regions of Figure~\ref{fig::necklacemap} (with two regions adjacent when the boundaries intersect) scales to two dimensional Brownian motion, up to a time change.  The correct time change should be the one which corresponds to Liouville Brownian motion, as constructed in \cite{ber2015lbm,grv2016lbm}.

\noindent{\it Update:} The convergence of the random walk on the graph of $\delta$-area regions to Brownian motion was established in \cite{gms2017tutte}, modulo time parameterization (see further applications to random walks on planar maps in \cite{gm2017spectral}).  The convergence was upgraded to include the time parameterization in \cite{bg2020convergence}.

Additional problems include the following:

\begin{problem}
Use the hamburger cheeseburger machinery to strengthen the topology of convergence in a more modest (but still interesting) way: namely, show that in FK-decorated RPM, the lengths of the $k$ largest loops (together with the adjacency graph on those loops, and the areas surrounded by the loops) scale to the corresponding continuum quantities defined for CLE-decorated LQG.
\end{problem}

\noindent{\it Update:} This question has been solved in \cite{gm2016topology}, building on the works \cite{gms2015cone_times,gs2015finite_volume,gs2015finite}.

\begin{problem}
In the peanosphere construction, compute the relationship between~$\kappa'$ and the Brownian correlation coefficient for $\kappa' > 8$.
\end{problem}

\noindent{\it Update:} This question has been solved in \cite{ghms2015covariance}.

\begin{problem}
Understand near critical FK-decorated RPM models and their scaling limits from the peanosphere perspective.
\end{problem}

We also direct the reader to the open questions sections of \cite{she2010zipper,ms2013qle} for additional questions.

\appendix

\section{Quantum disks and spheres as limits}
\label{app::disks_spheres}

We are now going to show that the unit boundary length quantum disk and the unit area quantum sphere as defined in Definition~\ref{def::finite_volume_surfaces} can be constructed using a limiting procedure.  Throughout, we let $\strip_+ = \R_+ \times [0,\pi]$ and $\cyl_+ = \R_+ \times [0,2\pi]$ with $\R_+ \times \{0\}$ and $\R_+ \times \{2\pi\}$ identified.

\subsection{Unit boundary length quantum disk}
\label{subapp::unit_boundary_length_disk}

We begin with the case of the unit boundary length quantum disk.

\begin{proposition}
\label{prop::unit_boundary_length_quantum_disk}
Fix a simply connected, bounded smooth domain $D \subseteq \h$ and suppose that $L = \partial D \cap \partial \h$ is a non-empty interval and $\wt{L} \subseteq L$ is a non-empty subinterval with $\dist(\partial D \setminus L, \wt{L}) > 0$.  Let $h$ be a GFF on $D$ with free boundary conditions on $L$ and zero boundary conditions on $\partial D \setminus L$.  Fix $C,\epsilon > 0$ and condition on $\{\sqrt{C} \leq \nu_h(\wt{L}) \leq \sqrt{C}(1 + \epsilon)\}$.  Let $\wh{h} = h - \gamma^{-1}\log C$ so that $1 \leq \nu_{\wh{h}}(\wt{L}) \leq 1+\epsilon$.   Then the law of $(D,\wh{h})$ (viewed as a quantum surface) converges weakly to that of the unit boundary length quantum disk as in Definition~\ref{def::finite_volume_surfaces} when we take a limit as $C \to \infty$ and then as $\epsilon \to 0$.  More precisely, suppose that $x \in \wt{L}$ is picked from $\nu_{\wh{h}}$ and let $\varphi \colon D \to \strip_+ - r$ be the unique conformal map which takes $x$ to $+\infty$ and the endpoints of $L$ to $-r$ and $-r+i\pi$ where the value of $r$ is chosen so that the $\gamma$-LQG length assigned by $\wh{h} \circ \varphi^{-1} + Q \log| (\varphi^{-1})'|$ to $\partial \strip_+ \setminus [0,i\pi]$ is equal to $1/2$.  Then the law of $\wh{h} \circ \varphi^{-1} + Q \log| (\varphi^{-1})'|$ converges weakly in the space of distributions as $C \to \infty$ and then as $\epsilon \to 0$ to that of a unit boundary length quantum disk embedded into $\strip$ so that the $\gamma$-LQG length of $\partial \strip_+ \setminus [0,i\pi]$ is $1/2$.
\end{proposition}

We emphasize that the limiting law in Proposition~\ref{prop::unit_boundary_length_quantum_disk} does not depend on $D$, $L$, or $\wt{L}$.

Recall that a quantum surface is defined modulo conformal transformation.  If we parameterize a quantum disk by $\strip$ with the marked points taken to be $\pm \infty$, then there is one extra degree of freedom in choosing the embedding (i.e., the horizontal translation of $\strip$).  In the statement of Proposition~\ref{prop::unit_boundary_length_quantum_disk}, we have chosen the horizontal translation so that the amount of $\gamma$-LQG length assigned to $\partial \strip_+ \setminus [0,i \pi]$ is $1/2$, but there are many other ways of fixing this horizontal translation.

The idea to prove Proposition~\ref{prop::unit_boundary_length_quantum_disk} is to consider the law of $h$ weighted by $\nu_h(\wt{L})$ (i.e., $\nu_h(\wt{L}) dh$).  As we will explain in Lemma~\ref{lem::weighted_quantum_surface_boundary}, sampling from this law is equivalent to first sampling from the (unweighted) law of $h$ and then adding to it a certain $\log$ singularity at a point chosen from a given measure on $\wt{L}$ which has a density with respect to Lebesgue measure.  We will then apply a change of coordinates from $D$ to $\strip_+$ with this extra marked point sent to $+\infty$ and the two endpoints of $L$ sent to $0$ and $i\pi$.  Let $\wt{h}$ be the resulting field on $\strip_+$.  We will argue that (Lemma~\ref{lem::tail_behavior}) conditioning the boundary length to be large is equivalent to conditioning the projection of $\wt{h}$ onto $\CH_1(\strip_+)$ from Lemma~\ref{lem::strip_spaces_orthogonal} to take on a large value and that this large value is realized along a line $u+[0,i\pi]$ with $u$ large.  We then find that (Lemma~\ref{lem::line_average_process_large}) if we horizontally translate by $-u$ so that this line segment is sent to $[0,i\pi]$ then the law of $\wt{h}(\cdot+u)$ which describes the surface converges to a limiting law as we take a limit as the boundary length tends to $\infty$ (while renormalizing).  A similar procedure will lead to the construction of the unit area quantum sphere (Proposition~\ref{prop::unit_area_quantum_sphere}).

We note that the considerations which will play a role in the proof of Proposition~\ref{prop::unit_boundary_length_quantum_disk} are closely related to the tail behavior of the LQG measure which has recently been studied in \cite{rv2019tail,wong2019tailprofile}.

\begin{lemma}
\label{lem::harmonic_strip_constant}
Suppose that $f \colon \strip_+ \to \R$ is a function which is harmonic in $\strip_+$ with Neumann boundary conditions on $\strip_+ \setminus [0,i\pi]$ (as well as at $+\infty$) and Dirichlet boundary conditions on $[0,i\pi]$.  Then the limit
\begin{equation}
\label{eqn::neumann_harmonic_limit}
\lim_{\re(z) \to \infty} f(z)
\end{equation}
exists and is given by the average $\ol{f}$ of $f$ on $[0,i\pi]$.  That is, for every $\epsilon > 0$ there exists $R \geq 0$ such that $\sup_{\re(z) \geq R} |f(z)-\ol{f}| \leq \epsilon$.
\end{lemma}
\begin{proof}
Let $F$ be the harmonic function which is defined on $\{z \in \C : \re(z) \geq 0\} = -i \h$ whose boundary conditions on the imaginary axis are given by first extending the boundary conditions of $f$ from $[0,\pi i]$ to $[0,2\pi i]$ by reflection and then to all of~$i \R$ periodically.  Then the restriction of $F$ to $\strip_+$ is equal to $f$.  From the explicit form of the Poisson kernel on $-i \h$, it is clear that $\lim_{\re(z) \to \infty} F(z)$ exists in the sense described in the lemma statement.
\end{proof}

\begin{lemma}
\label{lem::line_average_process_large}
Suppose that $h$ is a GFF on $\strip_+$ with free boundary conditions on $\partial \strip_+ \setminus [0,i\pi]$ and zero boundary conditions on $[0,i\pi]$.  Let $F$ be a function which is harmonic in $\strip_+$ with Dirichlet boundary conditions on $[0,i\pi]$ and Neumann boundary conditions on $\partial \strip_+ \setminus [0,i\pi] $.  (In particular, $F$ has Neumann boundary conditions at $+\infty$.)  Fix $a > 0$.  For each $u \geq 0$, let $X_u$ be given by the average of $h+F-a \re(z)$ on $u + [0,i\pi]$.  For each $C > 0$, let $\tau_C = \inf\{u \geq 0 : X_u \geq C\}$ and let $E_C = \{ \tau_C < \infty\}$.  On $E_C$, let $\wt{h}$ be the field which is given by precomposing $h+F-a\re(z)-C$ with the horizontal translation $z \mapsto z + \tau_C$.  Then the conditional law of $\wt{h}$ given $E_C$ converges weakly as $C \to \infty$.  With $\CH_1(\strip)$ and $\CH_2(\strip)$ as in Lemma~\ref{lem::strip_spaces_orthogonal}, the limit law can be sampled from by:
\begin{enumerate}[(i)]
\item\label{it::line_average_large_limit_law} Taking its projection $\wt{X}_u$ onto $\CH_1(\strip)$ to be given by $B_{2u} - a u$ for $u \geq 0$ where $B$ is a standard Brownian motion with $B_0 = 0$ and to be given by $\wh{B}_{-2u} + au$ for $u < 0$ where $\wh{B}$ is an independent standard Brownian motion with $\wh{B}_0 = 0$ conditioned so that $\wh{B}_{-2u} + au <0$ for all $u < 0$.  We fix the additive constant so that the projection exactly agrees with this process. 
\item\label{it::line_average_large_limit_law2} Independently sampling its projection onto $\CH_2(\strip)$ from the law of the corresponding projection of a free boundary GFF on $\strip$.  We fix the additive constant so that the average on $[0, \pi i]$ vanishes. 
\end{enumerate}
\end{lemma}
\begin{proof}
Let $\tau_C$ be as in the statement and let $Y$ have the law of $X$ given $E_C$.  To prove the result, it suffices to show that:
\begin{enumerate}[(I)]
\item\label{it::line_average_large_limit_suffice1} For each $N \geq 0$ we have that $\p[\tau_C \geq N \giv E_C] \to 1$ as $C \to \infty$ and
\item\label{it::line_average_large_limit_suffice2} The law of $Y - C$ with time translated by $\tau_C$ (so that it first hits $0$ at time $u=0$) converges to the limit described in~\eqref{it::line_average_large_limit_law} above.
\end{enumerate}

Indeed, suppose that $\wh{h}$ is a GFF on $\strip$ with free boundary conditions.  Let $\wh{\Fh}$ be the function which is harmonic in $\strip_+$ with boundary values given by those of $\wh{h}$ on $[0, i \pi]$ and with Neumann boundary conditions on $\partial \strip_+ \setminus [0,i \pi]$ so that $\wh{h} - \wh{\Fh}$ is a GFF on $\strip_+$ with the same boundary conditions as $h$.  We assume that $\wh{h}$, $h$ are coupled together so that $h = \wh{h} - \wh{\Fh}$.  Suppose that $\phi \in \CH_2(\strip)$.  Lemma~\ref{lem::harmonic_strip_constant} together with~\eqref{it::line_average_large_limit_suffice1} implies we a.s.\ have that $(F(\cdot+\tau_C)-a\re(\cdot + \tau_C)-C ,\phi)_\nabla \to 0$ as $C \to \infty$.  For the same reason, we a.s.\ have that $(\wh{\Fh}(\cdot+\tau_C),\phi) \to 0$ as $C \to \infty$.  Therefore the law of $(h(\cdot+\tau_C),\phi)_\nabla = (\wh{h}(\cdot+\tau_C) - \wh{\Fh}(\cdot+\tau_C),\phi)_\nabla$ converges weakly to the law of $(\wh{h},\phi)_\nabla$ as $C \to \infty$.  Combining gives that the law of the projection onto $\CH_2(\strip)$ of $h+F-a \re(\cdot) - C$ precomposed with $z \mapsto z + \tau_C$ converges weakly to the law of the corresponding projection of a free boundary GFF on $\strip$.

Note that since $F$ is harmonic on $\strip_+$ with Dirichlet boundary conditions on $[0,i\pi]$ and Neumann boundary conditions on $\partial \strip_+ \setminus [0,i\pi]$, it follows that its average on each vertical line $u + [0,i\pi]$ is the same.  It thus follows that $X$ is a Brownian motion with linear drift, so that the law of $Y$ is given by Lemma~\ref{lem::condition_bm_negative_drift_large}.  It is therefore clear from Lemma~\ref{lem::condition_bm_negative_drift_large} that both~\eqref{it::line_average_large_limit_suffice1} and~\eqref{it::line_average_large_limit_suffice2} hold.
\end{proof}

We now turn to prove a lemma which implies that conditioning the boundary length to be large is in a certain sense equivalent to conditioning the projection of the field onto $\CH_1(\strip)$ (with fixed additive constant) to take on a large value.  This will be convenient as it is easier to analyze the effect of the latter.

\begin{lemma}
\label{lem::tail_behavior}
Fix $a \in (0,2/\gamma)$ and suppose that $h = \wt{h} - a \re(z)$ where $\wt{h}$ is a GFF on $\strip_+$ with free boundary conditions on $\partial \strip_+ \setminus [0,i\pi]$ and Dirichlet boundary conditions on $[0,i\pi]$.  Let $X$ be the projection of $h$ onto $\CH_1(\strip_+)$ (with fixed additive constant) from Lemma~\ref{lem::strip_spaces_orthogonal}.  For each $C, \epsilon > 0$ and $\beta > 1$ we let
\begin{equation}
\label{eqn::ec_ecb}
\begin{split}
E_{C,\epsilon} &= \left\{C^{-1/2} \nu_h(\partial \strip_+ \setminus [0,i\pi]) \in [1,1+\epsilon] \right\} \quad\text{and}\\
E_{C,\beta}' &= \left\{ \sup_{u \geq 0} X_u  \geq \gamma^{-1} \log\left(\frac{C}{\beta}\right) \right\}.
\end{split}
\end{equation}
We have both
\begin{align}
 \p[ E_{C,\beta}' \giv E_{C,\epsilon}]& \to 1 \quad\text{as}\quad  \beta \to \infty \quad\text{uniformly in}\quad C \quad\text{and} \label{eqn::line_avg_high_given_avg_high}\\
  \p[E_{C,\epsilon} \giv E_{C,\beta}'] & > 0 \quad\text{uniformly in}\quad C \quad\text{for}\quad  \epsilon, \beta \quad\text{fixed}. \label{eqn::avg_high_given_line_avg_high}
\end{align}
The same likewise holds if we replace $\strip_+$ with $\cyl_+$, take $h$ to be a GFF on $\cyl_+$ with Dirichlet boundary conditions, and let $E_{C,\epsilon} = \{ C^{-1} \mu_h(\cyl_+) \in [1,1+\epsilon]\}$.
\end{lemma}

Before we give the proof of Lemma~\ref{lem::tail_behavior}, we will need to collect the following moment bound for Brownian motion with negative drift as well as a moment bound for the LQG measure.

\begin{lemma}
\label{lem:bm_moment_bound}
Suppose that $B$ is a standard Brownian motion and $a,c > 0$.  Let $M = \sup_{t \geq 0} e^{c(B_t - at)}$.  For every $p > 0$, there exists a constant $c_p > 0$ such that
\begin{equation}
\label{eqn::bm_exp_int_moment_bound}
\E \left[ \left(\int_0^\infty e^{c(B_t - at)} dt \right)^p \giv M \right] \leq c_p M^p.
\end{equation}
Similarly, if for each $j$ we let $j^*$ be the value of $t$ in $[j,j+1]$ at which $B_t - at$ attains its supremum, we have that
\begin{equation}
\label{eqn::bm_exp_sum_moment_bound}
\E \left[ \left(\sum_{j=0}^\infty e^{c(B_{j^*} - aj^*)} \right)^p \giv M \right] \leq c_p M^p.
\end{equation}
\end{lemma}
\begin{proof}
We are going to give the proof in the case of~\eqref{eqn::bm_exp_int_moment_bound}.  The proof of~\eqref{eqn::bm_exp_sum_moment_bound} is analogous.

We note that the conditional law of the process $B_t - a t$ given $M$ can be sampled from as follows (recall Lemma~\ref{lem::condition_bm_negative_drift_large} and Remark~\ref{rem::bessel_ito_excursion}).  Let $\tau$ be the time at which $B_t - at$ hits $c^{-1} \log M$.  In the interval $[0,\tau]$, it evolves as $\wh{B}_t + at$, where $\wh{B}$ is a standard Brownian motion, run until the first time it hits $c^{-1} \log M$.  In the interval $[\tau,\infty)$, it evolves as $c^{-1} \log M + \wt{B}_{t -\tau} - a(t-\tau)$ where~$\wt{B}$ is a standard Brownian motion conditioned so that $\wt{B}_t - at \leq 0$ for all $t$.  Together, this implies that $\int_0^\tau e^{c(B_t - at)} dt$ is stochastically dominated from above by $\int_\tau^\infty e^{c(B_t-a t)}$ since the process $t \mapsto B_{\tau-t} - a(\tau-t)$ is equal in distribution to the process $t \mapsto B_{t+\tau} - a(t+\tau)$ run up until the last time it hits $0$.  From the form of the conditional law, it is clear that it suffices to prove the result in the case that $M=1$.  That is, it suffices to show that
\[ \int_0^\infty e^{c (\wt{B}_t - a t)} dt \]
has finite moments of all orders.

We can bound this integral from above by $\sum_{k=0}^\infty 2^{-k} X_k$ where $X_k$ is the amount of time that $e^{c(\wt{B}_t - at)}$ spends in $(2^{-k-1},2^{-k}]$.  Assume that $p > 1$.  By Jensen's inequality, the $p$th power of the integral is thus at most a constant times $\sum_{k=0}^\infty 2^{-k} X_k^p$.  It therefore suffices to show that $\E[X_k^p]$ is uniformly bounded in $k$.  Note that $X_k$ is at most $\sum_j X_{j,k}$ where we let $X_{j,k}$ be the lengths of the successive excursions that $e^{c(\wt{B}_t - at)}$ makes from $[2^{-k-1},2^{-k}]$ to $\R \setminus [2^{-k-2},2^{-k+1}]$.  The strong Markov property implies that the number of such excursions that $e^{c(\wt{B}_t - a t)}$ makes is stochastically dominated by a geometric distribution.  Since $\wt{B}_t - at$ has a negative drift, the corresponding parameter is positive (uniformly in $k$).  Moreover, if $\sigma$ is a stopping time for $\wt{B}_t - at$ then on the event $e^{c(\wt{B}_\sigma-a \sigma)} \in [2^{-k-1},2^{-k}]$ the conditional probability that $e^{c(\wt{B}_{\sigma+t} - a(\sigma+t))} \geq 2^{-k-2}$ decays exponentially in $t$ (uniformly in $k$) due to the negative drift.  This implies that the length of such an excursion has an exponential tail.   Combining, we have shown that $X_k$ is stochastically dominated by $\sum_{i=1}^N A_i$ where $N$ is a geometric random variable (with parameter $q \in (0,1)$ uniform in $k$) and the $(A_i)$ are i.i.d.\ with an exponential tail (uniformly in $k$).  Fix $\alpha > 1$.  Using that $\sum_{i=1}^N A_i = \sum_{i=1}^\infty A_i \one_{\{N \geq i\}} \alpha^i \alpha^{-i}$, we have that
  \begin{align*}
    \E\left[  \left(\sum_{i=1}^N A_i \right)^p \right]
 &\leq (\alpha-1)^{1-p} \E\left[ \sum_{i=1}^\infty A_i^p \one_{\{N \geq i\}} \alpha^{pi} \alpha^{-i} \right] \quad\text{(Jensen's inequality)}\\
 &= (\alpha-1)^{1-p} \sum_{i=1}^\infty \big(\E[ A_i^{2p}]\big)^{1/2} \big(\p[ N \geq i]\big)^{1/2} \alpha^{(p-1) i} \quad\text{(Cauchy-Schwarz)}\\
 &\leq c_{\alpha,p} \sum_{i=1}^\infty q^{i/2} \alpha^{(p-1) i}.
  \end{align*}
  Here, $c_{\alpha,p} > 0$ is a constant depending only on $\alpha,p$.  This sum is finite provided we take $\alpha > 1$ sufficiently close to $1$ so that $q^{1/2} \alpha^{(p-1)} < 1$, which completes the proof of the result.
\end{proof}

\begin{lemma}
\label{lem::moment_bound}
Suppose that $h$ is a GFF on $\strip_+$ with free boundary conditions on $\partial \strip_+ \setminus [0,i\pi]$ and Dirichlet boundary conditions on $[0,i\pi]$, let $\wt{h}$ be the projection of $h$ onto $\CH_2(\strip_+)$, and let $\nu_{\wt{h}}$ be the $\gamma$-LQG boundary measure associated with $\wt{h}$.  For each $p \in (0,4/\gamma^2)$ there exists a constant $c_p < \infty$ such that
\begin{equation}
\label{eqn::moment_bound}
\E\left[ \big(\nu_{\wt{h}}( [u,u+1] \times \{0,\pi\} )\big)^p \right] < c_p  \quad\text{for all}\quad u \in \R_+.
\end{equation}
Suppose that $h$ is instead a GFF on $\cyl_+$ with Dirichlet boundary conditions, $\wt{h}$ is the projection of $h$ onto $\CH_2(\cyl_+)$, and let $\mu_{\wt{h}}$ be the $\gamma$-LQG area measure associated with $\wt{h}$.  For each $p \in (0,4/\gamma^2)$ there exists a constant $c_p < \infty$ such that
\begin{equation}
\label{eqn::moment_bound_area}
\E\left[ \big(\mu_{\wt{h}}( [u,u+1] \times [0,2\pi] )\big)^p \right] < c_p  \quad\text{for all}\quad u \in \R_+.
\end{equation}

\end{lemma}
\begin{proof}
Fix $R > 1$ and suppose that $h$ is a GFF on $B(0,R) \cap \h$ with Dirichlet boundary conditions on $\h \cap \partial B(0,R)$ and free boundary conditions on $[-R,R]$ and let $\wt{h}$ be the projection of $h$ onto $\CH_2(B(0,R) \cap \h)$.  By applying a change of coordinates, it suffices to show that $\nu_{\wt{h}}([-1,1])$ has a finite $p$th moment for each fixed $p \in (0,4/\gamma^2)$ which is uniform in $R$.  This, in turn, follows from \cite[Proposition~3.5]{rv2010revisited}.  The case of a GFF on $\cyl_+$  with Dirichlet boundary conditions is proved similarly.
\end{proof}

\begin{proof}[Proof of Lemma~\ref{lem::tail_behavior}]
We will first give the proof of~\eqref{eqn::line_avg_high_given_avg_high} and then turn to the proof of~\eqref{eqn::avg_high_given_line_avg_high}.
Note that $X_u = B_{2u} - au$ where $B$ is a standard Brownian motion.  Let
\[ M = \sup_{u \geq 0} e^{\gamma X_u/2}.\]
We claim that there exists constants $c_0, c_1 > 0$ so that
\begin{equation}
\label{eqn::ecbp}
c_0 \leq \frac{\p[M \geq C^{1/2}]}{C^{-a/\gamma}} \leq c_1.
\end{equation}
Indeed, this is a standard calculation for Brownian motion with drift (see, e.g., \cite[Chapter~3.5C]{ks91bm}).  Alternatively, this can be seen by using that $Z_t = e^{\gamma X_t/2}$ reparameterized according to its quadratic variation is a Bessel process of dimension $2-2 a/\gamma$ (Proposition~\ref{prop::bessel_exponential_bm}) and that $Z_t^{2-\delta}$ is a continuous local martingale (recall the proof of Lemma~\ref{lem::condition_bm_negative_drift_large}).  By~\eqref{eqn::avg_high_given_line_avg_high} (proved below) and~\eqref{eqn::ecbp}, it thus follows that for a constant $c_2 > 0$ we have
\begin{equation}
\label{eqn::e_c_eps_lbd}
\p[E_{C,\epsilon}] \geq c_2 C^{-a/\gamma}.
\end{equation}
Thus to prove~\eqref{eqn::line_avg_high_given_avg_high} it suffices to show that for each $\delta > 0$ there exists $\beta_0 > 0$ such that $\beta \geq \beta_0$ implies that $\p[E_{C,\epsilon} \cap (E_{C,\beta}')^c] \leq \delta C^{-a/\gamma}$.

Let $N = \nu_h(\partial \strip_+ \setminus [0, i\pi])$.  For each $j \in \N_0$, we define events
\[ Q_j = \left\{ \frac{N}{M} \geq \beta^{1/2} e^j \right\} \quad\text{and}\quad R_j = \left\{ \frac{C^{1/2}}{\beta^{1/2}} e^{-j-1} \leq M \leq \frac{C^{1/2}}{\beta^{1/2}} e^{-j} \right\}.\]

Let $\wt{h}$ be the projection of $h$ onto $\CH_2(\strip_+)$.  For each $j \geq 0$, we let $\wt{N}_j$ be the $\gamma$-LQG boundary length assigned to $[j,j+1] \times \{0,\pi\}$ by $\wt{h}$ and we let $j^*$ be the value of $u \in [j,j+1]$ which maximizes $B_{2u} - a u$.  We then have that,
\begin{align*}
      \frac{N}{M}
&\leq \frac{1}{M}  \sum_{j=0}^\infty e^{ (\gamma/2) (B_{2j^*} - aj^*)} \wt{N}_j
\end{align*} 
Applying Jensen's inequality, we thus have that if $p \geq 1$ then
\begin{align*}
\left(\frac{N}{M}\right)^p \leq  \left(\frac{1}{M} \sum_{j=0}^\infty e^{ (\gamma/2) (B_{2j^*} - aj^*)} \right)^p \sum_{j=0}^\infty p_j \wt{N_j}^p \quad\text{where}\quad p_j = \frac{e^{(\gamma/2) (B_{2j^*} - aj^*)}}{\sum_k e^{(\gamma/2) (B_{2k^*}-ak^*)}}.
\end{align*}
In particular, the $p_j$ are non-negative, sum to $1$, and are measurable with respect to $B$.
By Lemma~\ref{lem::moment_bound}, we have that $\E[ \wt{N}_j^p] < \infty$ uniformly in $j$ provided $p \in (0,4/\gamma^2)$.  Since the $\wt{N}_j$ are independent of $B$, it follows from Lemma~\ref{eqn::bm_exp_sum_moment_bound} that there exist constants $c_3,c_4 > 0$ so that
\begin{equation}
\label{eqn::b_a_bounds}
 c_3 \leq \E\left[ \left( \frac{N}{M} \right)^p \giv M \right] \leq c_4 \quad\text{and}\quad \E\left[ \left( \frac{N}{M} \right)^p \right] < \infty \quad\text{for all}\quad p \in [1,4/\gamma^2).
 \end{equation}
By another application of Jensen's inequality, the moments in~\eqref{eqn::b_a_bounds} are finite for all $p \in (0,4/\gamma^2)$.  We can thus dispense with the assumption that $p \geq 1$.  Fix $p \in (2 a/\gamma,4/\gamma^2)$.  Using~\eqref{eqn::b_a_bounds}, Markov's inequality for $N/M$ and~\eqref{eqn::ecbp} respectively imply that for a constant $c_5 > 0$ we have 
\[ \p[Q_j] \leq c_5 \beta^{-p/2} e^{-p j}, \quad \p[ Q_j \giv R_j] \leq c_5 \beta^{-p/2} e^{-p j}, \quad\text{and}\quad \p[R_j] \leq c_5 \beta^{a/\gamma} C^{-a/\gamma} e^{2 a j/\gamma}.\]

For constants $c_6, c_7 > 0$ we have that,
\begin{align*}
     \p[ N \geq C^{1/2},\ M \leq (C/\beta)^{1/2}]
&\leq \sum_{j=0}^\infty \p[ Q_j \cap R_j ] 
  \leq \sum_{j=0}^\infty \p[ Q_j \giv R_j] \p[R_j] \\
&\leq c_6 \beta^{a/\gamma} C^{-a/\gamma}\sum_{j=0}^{\infty} e^{2aj/\gamma} \times \beta^{-p/2} e^{-p j}  \\
&\leq c_7 \beta^{a/\gamma-p/2} C^{-a/\gamma} \quad\text{(since $p \in (2 a/\gamma,4/\gamma^2)$)}.
\end{align*}
Combining with~\eqref{eqn::e_c_eps_lbd} proves~\eqref{eqn::line_avg_high_given_avg_high}.

We now turn to prove~\eqref{eqn::avg_high_given_line_avg_high}.  In view of Lemma~\ref{lem::condition_bm_negative_drift_large}, it is obvious that there exists $p_0 > 0$ so that for all $\beta > 1$ and $C > 0$ we have that $\p[ \nu_h(\partial \strip_+ \setminus [0, i \pi]) \geq (C/\beta)^{1/2} \giv E_{C,\beta}'] \geq p_0$.  Moreover, \eqref{eqn::bm_exp_int_moment_bound} from Lemma~\ref{lem:bm_moment_bound} implies that for a constant $c_8 > 0$ (which does not depend on $C$ or $\beta$) we have that $\E[ \nu_h(\partial \strip_+ \setminus [0, i \pi]) \giv E_{C,\beta}'] \leq c_8 (C/\beta)^{1/2}$.  Combining this with Markov's inequality (and the above lower bound) implies that for a constant $c_9 > 1$ (which does not depend on $C$ or $\beta$) we have that $\p[ \nu_h(\partial \strip_+ \setminus [0, i \pi]) \in [(C/\beta)^{1/2}, c_9 (C/\beta)^{1/2}] \giv E_{C,\beta}' ] \geq p_0/2$.  We can cover $[(C/\beta)^{1/2}, c_9 (C/\beta)^{1/2}]$ by $n=\lceil 2(c_9-1) \epsilon^{-1} \rceil$ intervals $I_1,\ldots,I_n$ of length $(C/\beta)^{1/2}\epsilon/2$.  There must be one such interval, say $I_j$, so that $\p[ \nu_h(\partial \strip_+ \setminus [0, i \pi]) \in I_j \giv E_{C,\beta}'] \geq p_0/(4n)$.  Let $\phi$ be a $C_0^\infty(\strip_+)$ function which is equal to $\gamma^{-1} \log (\beta j^{-2})$ on $\re(z) \geq r_0 = \gamma^{-1} \log \beta$.  Note that the conditional probability that $\nu_h([0,r_0] \times \{0, \pi\}) \geq \epsilon (C/\beta)^{1/2}/10$ given $E_{C,\beta}'$ tends to $0$ as $C \to \infty$.  Let $\wh{h} = h + \phi$.  Then on $\nu_h(\partial \strip_+ \setminus [0, i \pi]) \in I_j$ and $\nu_h([0,r_0] \times \{0, \pi\}) \leq \epsilon (C/\beta)^{1/2}/10$, we have that $\nu_{\wh{h}}(\partial \strip_+ \setminus [0,i \pi]) \in [C^{1/2},C^{1/2} (1+\epsilon)]$.  This proves~\eqref{eqn::avg_high_given_line_avg_high} since $h$ and $\wh{h}$ are mutually absolutely continuous with Radon-Nikodym derivative given by $\exp( (h,\phi)_\nabla - \| \phi \|_\nabla^2/2)$ which has finite moments of all orders depending only on $\beta$.

Minor modifications to the above argument also give the result in the case that $h$ is a GFF with Dirichlet boundary conditions on $\cyl_+$.  In particular, in this case $X_u = B_u - au$ and $M = \sup_{u \geq 0} e^{\gamma X_u}$.
\end{proof}

We will now describe the law of a GFF $h$ weighted by its $\nu_h$ mass.  (Versions of this idea are also present in \cite[Section~6]{ds2011kpz}, \cite{she2010zipper}, and \cite[Section~4.4]{ms2013qle}.)

For a simply connected domain $D \subseteq \C$ and $z \in D$, we let $\confrad(z;D)$ denote the conformal radius of $z \in D$.  That is, $\confrad(z;D) = \varphi'(0)$ where $\varphi \colon \D \to D$ is the unique conformal transformation with $\varphi(0) = z$ and $\varphi'(0) > 0$.

\begin{lemma}
\label{lem::weighted_quantum_surface_boundary}
Suppose that $h$ is a GFF on a simply connected domain $D \subseteq \h$ and assume that $L = \partial D \cap \partial \h$  is a non-empty interval.  Let $dh$ denote the law of a GFF on $D$ with free boundary conditions on $L$ and Dirichlet boundary conditions on $\partial D \setminus L$.  Let $D^\dagger$ be the union of $D$, $\{\ol{z} : z \in D\}$, and the interior of $L$.  Suppose that $\wt{L} \subseteq L$ is a non-empty interval. A sample from the law $\nu_h(\wt{L}) dh$ can be produced by:
\begin{enumerate}
\item Sampling $h$ according to its (unweighted) law.
\item Picking $z_0 \in \wt{L}$ according to the measure $\confrad(z;D^{\dagger})^{\gamma^2/4}dz$ where $dz$ is Lebesgue measure on $\wt{L}$, independently of $h$ and then adding to $h$ the function $z \mapsto \tfrac{\gamma}{2} G(z_0,z)$ where $G$ is the Green's function on $D$ with Neumann (resp.\ Dirichlet) boundary conditions on $L$ (resp.\ $\partial D \setminus L$).
\end{enumerate}
\end{lemma}
\begin{proof}
We will prove the result using a standard Cameron-Martin (or Girsanov) style argument.  For each $\epsilon > 0$ and $z \in \ol{D}$, we let $h_\epsilon(z)$ be the average of $h$ on $D \cap \partial B(z,\epsilon)$.  Consider the law $\epsilon^{\gamma^2/4}\exp(\tfrac{\gamma}{2} h_\epsilon(u)) dh du$ where $du$ is Lebesgue measure on $\wt{L}$.  As explained in \cite[Section~6.2]{ds2011kpz}, we note that we can write $h$ as the even part of a GFF $h^\dagger$ on $D^\dagger$ on with Dirichlet boundary conditions.  This in particular means that for $u \in \wt{L}$ and $h_\epsilon^\dagger(u)$ the average of $h^\dagger$ on $\partial B(u,\epsilon)$ with $B(u,\epsilon) \subseteq D^\dagger$ we have that
\[ h_\epsilon(u) = \frac{1}{\sqrt{2}} \cdot 2 h_\epsilon^\dagger(u) = \sqrt{2} h_\epsilon^\dagger(u)\]
(see also \cite[Section~3.2]{she2010zipper}).  Consequently, using \cite[Proposition~3.2]{ds2011kpz} the marginal law of $u$ has density with respect to Lebesgue measure given by
\[ \epsilon^{\gamma^2/4} \int e^{\gamma h_\epsilon^\dagger(u) / \sqrt{2}} dh^\dagger = \confrad(u;D^\dagger)^{\gamma^2/4}.\]
Under this law, the conditional law of $h$ given $u$ is given by $\exp(\tfrac{\gamma}{2} h_\epsilon(u)) dh$.  For $z \in \ol{D}$ we let $\rho_\epsilon^z$ be the uniform measure on $D \cap \partial B(z,\epsilon)$ and let
\[ \xi_\epsilon^z(y) = -2\log \max(|z-y|,\epsilon) - \wt{G}_{z,\epsilon}(y)\]
where $\wt{G}_{z,\epsilon}$ is given by the function which is harmonic in $D$ with Neumann boundary conditions on $L$ and with Dirichlet boundary values given by $y \mapsto -2\log \max(|z-y|,\epsilon)$ on $\partial D \setminus L$.  For $u \in \wt{L}$, we note that $\Delta \xi_\epsilon^z(u) = -2\pi \rho_\epsilon^z(u)$ for all $\epsilon > 0$ sufficiently small so that $\partial B(u,\epsilon) \cap \partial D \subseteq L$.  Moreover, we have that
\begin{align*}
      h_\epsilon(u)
&= \int_D h(y) \rho_\epsilon^u(y) dy
 =  -\frac{1}{2\pi} \int_D h(y) \Delta \xi_\epsilon^u(y) dy\\
&= \frac{1}{2\pi} \int_D \nabla h(y) \cdot \nabla \xi_\epsilon^u(y) dy
 = (h, \xi_\epsilon^u)_\nabla
\end{align*}
for all $\epsilon > 0$ sufficiently small.  Now, the marginal law of $h$ under the weighted law $\exp(\tfrac{\gamma}{2} (h,\xi_\epsilon^u)_\nabla) dh$ can be sampled from by sampling $h$ according to its unweighted law and then adding to it $\tfrac{\gamma}{2} \xi_\epsilon^u$.  The result follows because $\xi_\epsilon^u \to G(u,\cdot)$ for $u \in \wt{L}$ and $\epsilon^{\gamma^2/4} \exp( \tfrac{\gamma}{2} h_\epsilon(u))du \to \nu_h$ as $\epsilon \to 0$.
\end{proof}

\begin{proof}[Proof of Proposition~\ref{prop::unit_boundary_length_quantum_disk}]
Let $G$ be the Green's function on $D$ with Neumann (resp.\ Dirichlet) boundary conditions on $L$ (resp.\ $\partial D \setminus L$).    Let $x_0$ (resp.\ $y_0$) be the two endpoints of $L$ and fix $u_0 \in (x_0,y_0)$.  Let $\psi \colon D \to \strip_+$ be the unique conformal transformation with $\psi(u_0) = +\infty$, $\psi(x_0) = 0$, and $\psi(y_0) = \pi i$.  Then $h \circ \psi^{-1} + Q\log|(\psi^{-1})'|$ can be written as $\wt{h} + F_{u_0}  +(\gamma- Q) \re(z)$ where $\wt{h}$ and $F_{u_0}$ are as in the statement of Lemma~\ref{lem::line_average_process_large}.  Indeed, the terms $F_{u_0}$ and $-Q \re(z)$ come from the coordinate change term $Q\log|(\psi^{-1})'|$ and the term $\gamma \re(z)$ comes because $G(u_0,\cdot)$ behaves like $-2\log|u_0-\cdot|$ near $u_0$.

We are going to prove that the quantum surface described by the field $\wt{h} + F_{u_0}  +(\gamma- Q) \re(z)$ conditioned on $E_{C,\epsilon}$ converges as $C \to \infty$ and then $\epsilon \to 0$ to the unit boundary length quantum disk.  We will then explain at the end of the proof how to deduce the statement of the proposition from this result.

Let $\wt{X}_u$ be the projection of $\wt{h} + F_{u_0} + (\gamma- Q) \re(z)$ onto $\CH_1(\strip_+)$.  For $C,\epsilon > 0$ and $\beta>1$, we let $E_{C,\epsilon}$ and $E_{C,\beta}'$ be as in Lemma~\ref{lem::tail_behavior} (note that $Q-\gamma = 2/\gamma-\gamma/2 \in (0,2/\gamma)$).  Note that $\wt{h} + F_{u_0}$ is a GFF on $\strip_+$ with Dirichlet boundary conditions on $[0, i \pi]$ and free boundary conditions on $\partial \strip_+ \setminus [0,i \pi]$.  Therefore we can apply Lemma~\ref{lem::tail_behavior} to the field $(\wt{h} + F_{u_0}) + (\gamma-Q) \re(z)$.  Note that~\eqref{eqn::line_avg_high_given_avg_high} implies that for each $\delta > 0$ there exists $\beta_0$ such that $\beta \geq \beta_0$ implies that the total variation distance between the law of $\wt{h} - \gamma^{-1} \log C$ given $E_{C,\epsilon}$ and $\wt{h} - \gamma^{-1} \log C$ given $E_{C,\beta}' \cap E_{C,\epsilon}$ is at most $\delta > 0$.  Thus by~\eqref{eqn::avg_high_given_line_avg_high}, to prove the existence of the limit of the law of $\wt{h} - \gamma^{-1} \log C$ given $E_{C,\epsilon}$ as $C \to \infty$ it suffices to prove the existence of the limit of the law of $\wt{h} - \gamma^{-1}\log C$ given $E_{C,\beta}'$ as $C \to \infty$ for each $\beta > 1$.  Lemma~\ref{lem::line_average_process_large} implies that the law of $\wt{h} - \gamma^{-1} \log C$ given $E_{C,\beta}'$ with the horizontal translation fixed so that $\wt{X}$ first hits $\gamma^{-1}\log(C/\beta)$ at $u=0$ converges to a limit as $C \to \infty$ which does not depend on our initial choices.

Lemma~\ref{lem::line_average_process_large} implies that the limiting process is given by $B_{2u} + (\gamma-Q) u$ for $u \geq 0$ and by $\wh{B}_{-2u} - (\gamma-Q) u$ for $u \leq 0$ where $B,\wh{B}$ are independent standard Brownian motions and $\wh{B}$ is conditioned so that $\wh{B}_{-2u} - (\gamma-Q) u < 0$ for all $u < 0$.  We will now describe this process in terms of a Bessel process.  Note that the process $\exp( (\gamma/2)(B_{2u} + (\gamma-Q)u))$ is equal in distribution to $\exp( B_{\gamma^2 u/2} + (\gamma/2)(\gamma-Q)u)$.  If we then perform a time-change by replacing $u$ with $2u/\gamma^2$, this process becomes $\exp(B_u + (1/\gamma)(\gamma-Q) u)$.  Note that $(1/\gamma)(\gamma-Q) = 1/2 - 2/\gamma^2$.  Therefore Proposition~\ref{prop::bessel_exponential_bm} implies that it corresponds to a Bessel process of dimension $3-4/\gamma^2$ in exactly the same way as the projection of a sample from $\diskmeasure$ (recall that it is defined in Definition~\ref{def::finite_volume_surfaces}) onto $\CH_1(\strip)$ (with additive constant fixed) corresponds to a Bessel process of dimension $3-4/\gamma^2$.  The same also applies to $\wh{B}_{-2u} - (\gamma - Q) u$ for $u \leq 0$ (recall Lemma~\ref{lem::condition_bm_negative_drift_large} for further explanation of this point).
  
Altogether, we have shown that follows that the limiting law as $C \to \infty$ for $\wt{h} - \gamma^{-1} \log C$ conditioned on $E_{C,\epsilon} \cap E_{C,\beta}'$ is given by $\diskmeasure$ conditioned on both
\begin{enumerate}
\item The supremum of the projection onto $\CH_1(\strip)$ (with fixed additive constant) is at least $-\gamma^{-1} \log \beta$ and
\item The $\gamma$-LQG boundary length is between $1$ and $1+\epsilon$.
\end{enumerate}
(Note that this describes a probability measure since this set has positive and finite $\diskmeasure$-mass.)  The result thus follows by sending $\beta \to \infty$ and $\epsilon \to 0$.  (Recall the discussion just after Definition~\ref{def::finite_volume_surfaces} where it is explained why the limit as $\epsilon \to 0$ exists.)

We are now going to explain how to deduce the statement of the proposition from the above.  It suffices to prove the result with $h$ sampled from the law $\nu_h(\wt{L}) dh$ as in Lemma~\ref{lem::weighted_quantum_surface_boundary} in place of the unweighted law $dh$.  By Lemma~\ref{lem::weighted_quantum_surface_boundary}, we know that we can produce a sample of $h$ from $\nu_h(\wt{L}) dh$ by first sampling $u_0 \in \wt{L}$ from the measure $\confrad(z;D^\dagger)^{\gamma^2/4} dz$ where $dz$ is Lebesgue measure on $\wt{L}$ (as in the statement of Lemma~\ref{lem::weighted_quantum_surface_boundary}).  We then sample $\wh{h}$ according to the law $dh$ independently of $u_0$ and then take $h = \wh{h} + \tfrac{\gamma}{2} G(u_0,\cdot)$.  If we want to produce a sample from $\nu_h(\wt{L}) dh$ conditioned on $C^{-1/2} \nu_h(\wt{L}) \in [1,1+\epsilon]$, then we sample $u_0$ from its new marginal law and then take $h = \wh{h} + \tfrac{\gamma}{2} G(u_0,\cdot)$ where $\wt{h}$ has law $dh$ and we have conditioned on $C^{-1/2} \nu_h(\wt{L}) \in [1,1+\epsilon]$.  The result then follows since we have assumed that $\dist(\wt{L}, \partial D \setminus \partial \h) > 0$ so that $\dist(u_0, \partial D \setminus \partial \h)$ is bounded from below (i.e., $u_0$ is never arbitrarily close to where the boundary conditions for $h$ are not free).
\end{proof}

We are now going to deduce from the proof of Proposition~\ref{prop::unit_boundary_length_quantum_disk} a certain symmetry property for the unit boundary length quantum disk which is not immediately obvious from Definition~\ref{def::finite_volume_surfaces}.

\begin{proposition}
\label{prop::unit_boundary_length_disk_points_uniformly_random}
Suppose that $(\strip,h)$ has the law of a unit boundary length quantum disk with the embedding into $\strip$ as described in Definition~\ref{def::finite_volume_surfaces}.  Conditional on the quantum surface $(\strip,h)$, the points which correspond to $-\infty$ and $+\infty$ are uniformly and independently distributed according to the quantum boundary measure.  That is, if $x,y \in \partial \strip$ are chosen independently from $\nu_h$ and $\varphi \colon \strip \to \strip$ is a conformal map with $\varphi(x) = +\infty$ and $\varphi(y) = -\infty$, then $h \circ \varphi^{-1} + Q\log|(\varphi^{-1})'|$ has the same law as $h$ (modulo a horizontal translation of~$\strip$).
\end{proposition}
\begin{proof}
Recall from the proof of Proposition~\ref{prop::unit_boundary_length_quantum_disk} completed just above that the point at $+\infty$ in the limiting construction of the unit boundary length quantum disk was given by the image of a point chosen from the quantum boundary measure.  This implies that, for the unit boundary length quantum disk, the conditional law of the point which corresponds to $+\infty$ is uniformly distributed according to the quantum boundary measure.  That is, if we pick $x \in \partial \strip$ according to $\nu_h$ and let $\varphi \colon \strip \to \strip$ be a conformal transformation which fixes $-\infty$ and with $\varphi(x) = +\infty$, then we have that $h \circ \varphi^{-1} + Q\log|(\varphi^{-1})'|$ has the same law as $h$ (modulo a horizontal translation of~$\strip$).  The proposition follows as the law of $h$ is clearly also invariant (modulo a horizontal translation of~$\strip$) under the operation of swapping $\pm \infty$.
\end{proof}

We are now going to explain a variant of the proof of Proposition~\ref{prop::unit_boundary_length_quantum_disk} which yields the following useful resampling property for the unit boundary length quantum disk.

\begin{figure}
\begin{center}
\includegraphics[scale=0.85]{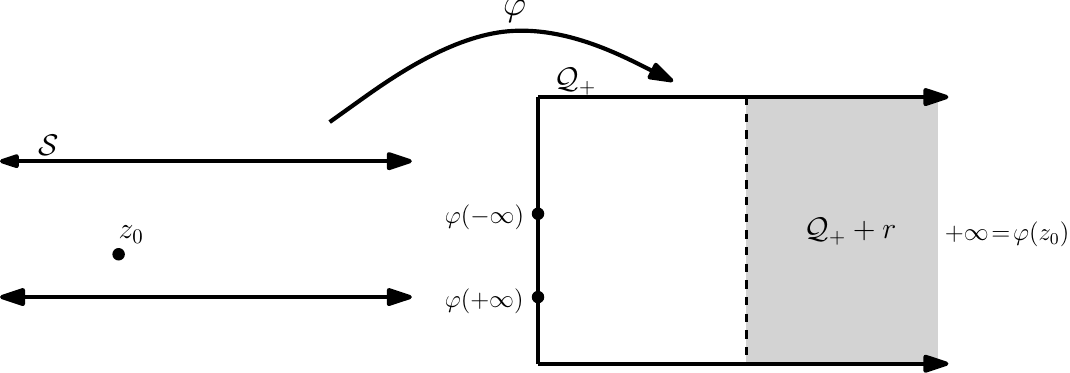}
\end{center}
\caption{\label{fig::unit_boundary_length_disk_resampling} Shown on the left is a unit boundary length quantum disk parameterized by $\strip$, $z_0$ a point picked in $\strip$ uniformly from the quantum measure, and $\varphi \colon \strip \to \cyl_+$ the unique conformal map with $\varphi(z_0) = +\infty$ and $\varphi'(z_0) > 0$.  Proposition~\ref{prop::unit_boundary_length_disk_resampling} implies that the conditional law of the surface on the right given the field values in $\cyl_+ + r$, $r > 0$ fixed, is given by that of a GFF in $[0,r] \times [0,2\pi] \subseteq \cyl_+$ with Dirichlet (resp.\ free) boundary conditions on $\partial \cyl_+ + r$ (resp.\ $\partial \cyl_+$) conditioned on $\partial \cyl_+$ having quantum boundary length equal to $1$.}
\end{figure}

\begin{proposition}
\label{prop::unit_boundary_length_disk_resampling}
Suppose that $D \subseteq \C$ is a simply connected domain and $h$ is such that $(D,h)$ has the law of a unit boundary length quantum disk weighted by its quantum area.  Pick $z \in D$ uniformly at random from the quantum measure associated with $h$, let $\varphi \colon D \to \cyl_+$ be the unique conformal transformation with $\varphi(z) = +\infty$ and $\varphi'(z) > 0$, and let $\wt{h} = h \circ \varphi^{-1} + Q \log|(\varphi^{-1})'|$.  For each $r > 0$, let $\CF_r$ be the $\sigma$-algebra generated by the restriction of $h$ to $\cyl_+ + r$.  Then the conditional law of $h$ given $\CF_r$ is that of a GFF on $[0,r] \times [0,2\pi] \subseteq \cyl_+$ with free boundary conditions on $\partial \cyl_+$ and Dirichlet boundary conditions given by those of $h$ on $\partial (\cyl_+ + r)$ conditioned on the quantum boundary length of $\partial \cyl_+$ being equal to $1$.
\end{proposition}

See Figure~\ref{fig::unit_boundary_length_disk_resampling} for an illustration of the setup of Proposition~\ref{prop::unit_boundary_length_disk_resampling}.  To give the proof of Proposition~\ref{prop::unit_boundary_length_disk_resampling}, we need the following analog of Lemma~\ref{lem::weighted_quantum_surface_boundary} (see also \cite[Section~3.3]{ds2011kpz}).  This lemma will also be useful to us later when we construct the unit area quantum sphere as a limit in Appendix~\ref{subapp::unit_area_sphere} below.

\begin{lemma}
\label{lem::weighted_quantum_surface}
Suppose that $h$ is a GFF on a bounded, simply connected domain $D$ with Dirichlet boundary conditions.  Fix $U \subseteq D$ open.  Let $dh$ denote the law of $h$ and consider the law $\mu_h(U) dh$.  A sample from the law $\mu_h(D) dh$ can be produced by:
\begin{enumerate}
\item Sampling $h$ according to its (unweighted) law.
\item Picking $z_0 \in U$ according to the measure $\confrad(z;D)^{\gamma^2/2} dz$ where $dz$ denotes Lebesgue measure on $U$ independently of $h$ and then adding to $h$ the function $z \mapsto \gamma G(z_0,z)$ where $G$ is the Green's function on $D$ with Dirichlet boundary conditions.
\end{enumerate}
The same likewise holds if $\partial D = \partial^{\mathrm F} \cup \partial^{\mathrm D}$ is a disjoint union with $\partial^{\mathrm F} , \partial^{\mathrm D} \neq \emptyset$ and $h$ is a GFF on $D$ with Dirichlet (resp.\ free) boundary conditions on $\partial^{\mathrm D}$ (resp.\ $\partial^{\mathrm F}$) where we take $G$ to be the Green's function on $D$ with Dirichlet (resp.\ Neumann) boundary conditions on $\partial^{\mathrm D}$ (resp.\ $\partial^{\mathrm F}$) and $C(z;D)$ is the density for the law of $z_0 \in D$ is replaced by the function $\lim_{y \to z} \exp(G(y,z) + \log|y-z|)$.  (In this case, $\mu_h(D) dh$ may be an infinite measure).

Finally, if $h$ is a whole-plane GFF with the additive constant fixed so that its average on $\partial \D$ is equal to $0$ then a sample from the law of $\mu_h(\D) dh$ can be produced by sampling $h$ according to its unweighted law, then picking $z_0 \in \D$ according to Lebesgue measure, then adding to $h$ the function $z \mapsto -\gamma\log|z-z_0| + \gamma \log\max(|z|,1)$, and then finally taking the additive constant so that the average of the resulting field on $\partial \D$ is equal to $0$.
\end{lemma}
\begin{proof}
As in the case of Lemma~\ref{lem::weighted_quantum_surface_boundary}, we will prove the result using a standard Cameron-Martin (or Girsanov) style argument.  We are going to give the proof in the case of Dirichlet boundary conditions; the proof in the case of mixed boundary conditions or the whole-plane GFF follows from the same argument.

For each $\epsilon > 0$, we let $h_\epsilon$ be the circle average process associated with~$h$.   Consider the measure $\epsilon^{\gamma^2/2}\exp(\gamma h_\epsilon(u)) dh du$ where $du$ is Lebesgue measure on $U$.  By \cite[Proposition~3.2]{ds2011kpz}, we note that the density of the marginal law of $u$ with respect to Lebesgue measure is given by 
\[ \epsilon^{\gamma^2/2} \int e^{\gamma h_\epsilon(u)} dh = \confrad(u;D)^{\gamma^2/2}\]
provided $B(u,\epsilon) \subseteq D$.  Moreover, the law of $h$ under this weighting converges as $\epsilon \to 0$ to the same law as $\mu_h(U) dh$.  Using the former law, the conditional law of $h$ given $u$ is given by $\exp(\gamma h_\epsilon(u)) dh$.  Let $\rho_\epsilon^z$ be the uniform measure on $\partial B(z,\epsilon)$ and let
\[ \xi_\epsilon^z(y) = -\log \max(|z-y|,\epsilon) - \wt{G}_{z,\epsilon}(y)\]
where $\wt{G}_{z,\epsilon}$ is given by the harmonic extension of $y \mapsto -\log \max(|z-y|,\epsilon)$ from $\partial D$ to $D$.  Note that $\Delta \xi_\epsilon^z(y) = -2\pi \rho_\epsilon^z(y)$.  Moreover, we have that
\begin{align*}
      h_\epsilon(u)
&= \int_D h(y) \rho_\epsilon^u(y) dy
 =  -\frac{1}{2\pi} \int_D h(y) \Delta \xi_\epsilon^u(y) dy\\
&= \frac{1}{2\pi} \int_D \nabla h(y) \cdot \nabla \xi_\epsilon^u(y) dy
 = (h, \xi_\epsilon^u)_\nabla.
\end{align*}
Now, the marginal law of $h$ under the weighting $\exp(\gamma (h,\xi_\epsilon^u)_\nabla) dh$ is given by taking the law of $h$ and then adding to it $\gamma \xi_\epsilon^u$.  The result follows because $\xi_\epsilon^u \to G(u,\cdot)$ as $\epsilon \to 0$.
\end{proof}

\begin{proof}[Proof of Proposition~\ref{prop::unit_boundary_length_disk_resampling}]
We are going to deduce the result from the proof of Proposition~\ref{prop::unit_boundary_length_quantum_disk}.  In particular, in the proof of Proposition~\ref{prop::unit_boundary_length_quantum_disk} we showed that the unit boundary length quantum disk arises by starting with a GFF with mixed boundary conditions, conditioning the boundary length to be large, and then mapping to $\strip_+$ with a typical point chosen from the boundary measure taken to $+\infty$.  We will use that this GFF satisfies its usual Markov property to deduce the claimed resampling property for the limiting unit boundary length quantum disk.

We first suppose that $(D,h)$ is as in the statement of Proposition~\ref{prop::unit_boundary_length_quantum_disk}.  Let $G$ be the Green's function for $\Delta $ on $D$ with Dirichlet (resp.\ Neumann) boundary conditions on $\partial D \setminus L$ (resp.\ $L$).  By Lemma~\ref{lem::weighted_quantum_surface_boundary}, we know that we can produce a sample from the law of $\nu_h(\wt{L}) dh$ by first producing a sample of $\wh{h}$ from the law $dh$ and then taking $\wt{h} = \wh{h} + \tfrac{\gamma}{2} G(u_0,\cdot)$ where $u_0$ is picked from the measure $\confrad(z;D^\dagger)^{\gamma^2/4}dz$, $dz$ Lebesgue measure on $\wt{L}$, and $D^\dagger$ is the union of $D$, $\{\ol{z} : z \in D\}$, and the interior of $\wt{L}$.  Fix a neighborhood $U \subseteq \wt{L}$ of $u_0$.  Lemma~\ref{lem::weighted_quantum_surface} implies that we can produce a sample from the conditional law of $\wt{h}$ given its values on $U$ and weighted by $\mu_{\wt{h}}(\wt{U})$ by first picking $z_0 \in \wt{U}$ according to the measure $(f(z))^{\gamma^2/2} dz$ and then adding $\gamma \wt{G}(z_0,z)$ to the field where $\wt{G}$ is the Green's function on $D$ with Dirichlet boundary values on $(\partial D \setminus L) \cup U$ and Neumann boundary conditions on $L \setminus U$ and $f(z) = \lim_{y \to z} \exp(\wt{G}(y,z) + \log|y-z|)$.

Let $x,y$ be the two endpoints of $L$ and let $\varphi$ be the unique conformal map $D \to \strip_+$ which takes $x$ to $0$, $y$ to $i \pi$, and $u_0$ to $+\infty$.  We then let $r_0$ be such that the real part of $\varphi(z_0)+r_0$ is equal to $0$ and then let $\wt{\varphi} = \varphi+r_0$ and $\breve{h} = \wt{h} \circ \wt{\varphi}^{-1} + Q \log|(\wt{\varphi}^{-1})'|$.

It follows from the proof of Proposition~\ref{prop::unit_boundary_length_quantum_disk} that the limit of the law of $(\strip_+ - r_0,\wt{h} - (\log C)/\gamma)$ conditioned on $1 \leq \nu_{\wt{h}}(\partial \strip_+ - r_0) \leq (1+\epsilon)$ converges when we take a limit as $C \to \infty$ and then $\epsilon \to 0$ to the law of the unit boundary length quantum disk.  Thus the above tells us how to resample the law of the unit boundary length quantum disk weighted by its quantum area once we have picked a uniformly random point from the quantum area measure and we have fixed its values in a neighborhood $U$ of $\pm \infty$ in $\partial \strip$.  Namely, the conditional law is given by that of a GFF in $\strip$ with free boundary conditions along $\partial \strip \setminus U$ and Dirichlet boundary conditions along $\partial U$ plus $\gamma G(w_0,\cdot)$ where $G$ is the Green's function on $\strip$ with Neumann (resp.\ Dirichlet) boundary conditions on $\partial \strip \setminus U$ (resp.\ $\partial U \cap \strip$) conditioned on the total boundary length being equal to $1$.

Suppose that $(\strip,h)$ has the law of a unit boundary length quantum disk weighted by its quantum area, $w_0$ is chosen uniformly from the quantum area measure, and that $\varphi \colon \strip \to \cyl_+$ is the unique conformal map with $\varphi(w_0) = +\infty$ and $\varphi'(w_0) > 0$.  Then with $\wt{h} = h \circ \varphi^{-1} + Q \log|(\varphi^{-1})'|$ we have that $(\cyl_+,\wt{h})$ has the law of a unit boundary length quantum disk weighted by its quantum area.  Suppose that $a_0,b_0 \in \partial \cyl_+$ are picked uniformly and independently from the quantum boundary measure.  Proposition~\ref{prop::unit_boundary_length_disk_points_uniformly_random} combined with the above implies that if $r > 0$ is fixed, then the conditional law of $\wt{h}$ in $[0,r] \times [0,2\pi]$ given its values in $\cyl_+ + r$ and in a neighborhood $U$ of $a_0,b_0$ in $\partial \cyl_+$ is that of a GFF with free (resp.\ Dirichlet) boundary conditions on $\partial \cyl_+ \setminus U$ (resp.\ $U \cup \partial (\cyl_+ +r)$) conditioned on the boundary length of $\partial \cyl_+$ being equal to $1$.  The result follows by taking a limit as the size of $U$ tends to $0$.
\end{proof}

\subsection{Unit area quantum sphere}
\label{subapp::unit_area_sphere}

We now turn to the case of the unit area quantum sphere.

\begin{proposition}
\label{prop::unit_area_quantum_sphere}
Fix a smooth, bounded domain $D$.  Let $h$ be a GFF on $D$ with zero boundary conditions.  Let $U \subseteq D$ be non-empty, open, and satisfy $\dist(\partial U, \partial D) > 0$.  Fix $C,\epsilon > 0$ and condition on $\{C \leq \mu_h(U) \leq C(1+\epsilon)\}$.  Let $\wh{h} = h - \gamma^{-1} \log C$ so that $1 \leq \mu_{\wh{h}}(U) \leq 1+\epsilon$.  Then the law of $(D,\wh{h})$ (viewed as a quantum surface) converges weakly to that of the unit area quantum sphere as in Definition~\ref{def::finite_volume_surfaces} when we take a limit as $C \to \infty$ and then $\epsilon \to 0$.  More precisely, suppose that $z \in U$ is picked from $\mu_{\wh{h}}$ and let $\varphi \colon D \to \cyl_+ - r$ be the unique conformal map which takes $z$ to $+\infty$ with $\varphi'(z) > 0$ where the value of $r$ is chosen so that the $\gamma$-LQG mass assigned by $\wh{h} \circ \varphi^{-1} + Q \log| (\varphi^{-1})'|$ to $\cyl_+$ is equal to $1/2$.  Then the law of $\wh{h} \circ \varphi^{-1} + Q \log| (\varphi^{-1})'|$ converges weakly in the space of distributions as $C \to \infty$ and $\epsilon \to 0$ to that of a unit area quantum sphere embedded into $\cyl$ so that the $\gamma$-LQG mass of $\cyl_+$ is $1/2$.
\end{proposition}

We emphasize that the limit in Proposition~\ref{prop::unit_area_quantum_sphere} does not depend on the choice of $D$ or $U$.

Before proving Proposition~\ref{prop::unit_area_quantum_sphere}, we need to collect the following analog of Lemma~\ref{lem::line_average_process_large}.

\begin{lemma}
\label{lem::circle_average_process_large}
Suppose that $h$ is a GFF on $\cyl_+$ with zero boundary conditions and let $F$ be a function which is harmonic on $\cyl_+$.  Fix $a > 0$ and for each $u > 0$ we let $X_u$ be the average of $h+F-a\re(z)$ on $u + [0,2\pi i]$.  For each $C > 0$, let $\tau_C = \inf\{u \geq 0 : X_u \geq C\}$ and let $E_C = \{\tau_C < \infty\}$.  On $E_C$, we let $\wt{h}$ be given by precomposing $h+F - a \re(z)-C$ with the horizontal translation $z \mapsto z+\tau_C$.  Then the conditional law of $\wt{h}$ given $E_C$ converges as $C \to \infty$.  Moreover, the limit law can be sampled from by:
\begin{enumerate}[(i)]
\item \label{it::circle_average_large_limit_law}
Taking its projection $\wt{X}_u$ onto $\CH_1(\cyl)$ to be given by $B_{u} - a u$ for $u \geq 0$ where $B$ is a standard Brownian motion with $B_0 = 0$ and to be given by $\wh{B}_{-u} + au$ for $u < 0$ where $\wh{B}$ is a standard Brownian motion with $\wh{B}_0 = 0$ conditioned so that $\wh{B}_{-u} + au <0$ for all $u < 0$.  We fix the additive constant so that its average on $[0, 2\pi i]$ vanishes.
\item \label{it::circle_average_large_limit_law2}
Independently sampling its projection onto $\CH_2(\cyl)$ according to the law of the corresponding projection of the GFF on $\cyl$.  We fix the additive constant so that its average on $[0, 2\pi i]$ vanishes.
\end{enumerate}
\end{lemma}
\begin{proof}
We omit the details since it is very similar to the proof of Lemma~\ref{lem::line_average_process_large}.
\end{proof}

\begin{proof}[Proof of Proposition~\ref{prop::unit_area_quantum_sphere}]
The proof is similar to the proof of Proposition~\ref{prop::unit_boundary_length_quantum_disk}.  We consider the law $\mu_h(U) dh$ in place of the law $h$.

By Lemma~\ref{lem::weighted_quantum_surface}, we can produce a sample of $h$ from $\mu_h(U) dh$ by first sampling $\wh{h}$ according to the law of $dh$ and then taking $h = \wh{h} + \gamma G(z_0,\cdot)$ where $z_0 \in U$ is chosen the measure with density with respect to Lebesgue measure as described in Lemma~\ref{lem::weighted_quantum_surface} independently of $\wh{h}$.  We let $\psi \colon D \to \cyl_+$ be the unique conformal transformation with $\psi(z_0) = +\infty$ and $\psi'(z_0) > 0$.  Consider the field $\wt{h} = h \circ \psi^{-1} + Q\log|(\psi^{-1})'|$ on $\cyl_+$.  Note that $\wt{h}$ is given by a GFF on $\cyl_+$ with Dirichlet boundary conditions plus the function $(\gamma-Q)\re(z)$.  We let $\wt{X}_u$ be the projection of $\wt{h}$ onto $\CH_1(\cyl_+)$ (with fixed additive constant) as in Lemma~\ref{lem::strip_spaces_orthogonal}.  As in the proof of Proposition~\ref{prop::unit_boundary_length_quantum_disk}, for $C,\epsilon > 0$ and $\beta > 1$ we define events
\begin{equation}
\begin{split}
 E_{C,\epsilon} &= \left\{C^{-1}  \mu_{\wt{h}}(\cyl_+) \in [1,1+\epsilon] \right\} \quad\text{and}\\
  E_{C,\beta}' &= \left\{ \sup_{u \geq 0} \wt{X}_u \geq \gamma^{-1} \log\left(\frac{C}{\beta}\right) \right\}.
 \end{split}
\end{equation}
Therefore the result follows by arguing as in Proposition~\ref{prop::unit_boundary_length_quantum_disk} using Lemma~\ref{lem::circle_average_process_large} in place of Lemma~\ref{lem::line_average_process_large}.
\end{proof}

We finish this section with the following analog of Proposition~\ref{prop::unit_boundary_length_disk_points_uniformly_random} for the unit area sphere.

\begin{proposition}
\label{prop::unit_area_sphere_points_random}
Suppose that $(\cyl,h)$ is a unit area quantum sphere embedded into $\cyl$ as in Definition~\ref{def::finite_volume_surfaces}.  Conditional on the quantum surface $(\cyl,h)$, the points which correspond to $-\infty$ and $+\infty$ are distributed independently and uniformly from the quantum area measure.  That is, if $x,y \in \cyl$ are chosen independently from $\mu_h$ and $\varphi \colon \cyl \to \cyl$ is a conformal map with $\varphi(x) = +\infty$ and $\varphi(y) = -\infty$, then $h \circ \varphi^{-1} + Q\log|(\varphi^{-1})'|$ has the same law as $h$ modulo a horizontal translation of $\cyl$.
\end{proposition}
\begin{proof}
This follows from the same argument used to prove Proposition~\ref{prop::unit_boundary_length_disk_points_uniformly_random}.
\end{proof}

\begin{proposition}
\label{prop::unit_area_sphere_resampling}
Suppose that $(\cyl,h)$ has the law of the unit area sphere embedded into $\cyl$ as in Definition~\ref{def::finite_volume_surfaces}.  Assume that $z \in \cyl$ is picked uniformly from the quantum measure and let $\wt{h} = h(\cdot+z)$ (with translation in the cylinder defined by taking the imaginary coordinate modulo $2\pi$).  Then $(\cyl,\wt{h})$ has the law of the unit area sphere with a fixed embedding because we have chosen the locations of three points.  Fix a bounded domain $U \subseteq \cyl$ containing $0$.  Let $G$ be the Green's function on $U$ with Dirichlet boundary conditions.  Then the conditional law of $h$ in $U$ given its values in $\cyl \setminus U$ is that of a GFF in $U$ with the given boundary conditions on $\partial U$ plus $\gamma G(0,\cdot)$ conditioned so that the total area is equal to $1$.
\end{proposition}
\begin{proof}
This follows from the same argument used to prove Proposition~\ref{prop::unit_boundary_length_disk_resampling}.
\end{proof}

We remark that a statement similar to Proposition~\ref{prop::unit_area_sphere_resampling} also appears in \cite{twoperspectives}, which proves the equivalence of the unit area quantum sphere with another construction given in \cite{lqg_sphere}. The follow up paper \cite{quantum_spheres} to the current paper also contains more detail about obtaining unit area spheres as matings of finite trees.

\section{KPZ interpretation}
\label{app::kpz}

\subsection{Scaling exponents and KPZ}
\label{subapp::kpz}

Historically, one of the major goals of the conformal field theory and quantum gravity literature has been an understanding of the scaling dimensions and exponents associated to random fractal sets.  For example, one might ask for the dimension of the set of double points of an $\SLE_{\kappa'}$ process or the dimension of the intersection of two GFF flow lines of different angles.  Similarly, one might ask for the scaling exponent associated to the probability that $n$ independent planar Brownian motions reach a given ball of radius $\epsilon$ without intersecting each other.

All of these problems now have {\em rigorous} solutions in the mathematics literature.  Additionally, {\em heuristic} solutions to most of these problems have appeared (often decades earlier) in the physics literature.  We will present citations and review both literatures in this section.

Despite the existence of rigorous solutions, the physics arguments still have a good deal of appeal.  In many cases they provide the quickest and most satisfying explanation of {\em why} certain dimensions and exponents are what they are.  This is in particular true of the quantum gravity arguments developed by the first co-author  that make use of constructions that are similar to the ``wedge weldings'' that appear in the current paper, along with the so-called KPZ formula (as briefly explained in Section~\ref{subsec::welding}).

In this section
we will revisit some of these heuristic arguments (specifically, the ones that make use of KPZ and welding ideas --- we will not address the heuristic arguments based on so-called Coulomb gas techniques) and explain how they fit into the context of the current paper.  We will explain how (in light of the results of the current paper) several of these arguments are now quite close to being proofs (although turning them into {\em actual} proofs would still require some elaboration).

The key idea is to show that the local behavior at a ``quantum typical'' point in a random fractal $X$ looks like a particular type of $\alpha$-quantum cone (bulk case) or $\alpha$-quantum wedge (boundary case).  By quantum typical, we mean a point chosen from the measure on $X$ constructed by restricting the $\gamma$-LQG measure to the union of the set of balls of quantum mass $\delta$ that intersect $X$, followed by appropriate $\delta \to 0$ rescaling so that the limit is non-trivial.  (We will not actually construct these measures in every case here; see \cite{DS3} for one approach to constructing $\SLE$ quantum measures.) 
 As explained in \cite[Equation~(63)]{ds2011kpz},  a quantum typical point  in a fractal $X$  of quantum scaling exponent $\Delta$ is a point near which the field is {\em $\alpha$-thick} where  $\alpha = \gamma(1 - \Delta)$.  This means that the quantum surface near that point looks locally like an $\alpha$-quantum cone (bulk case) or wedge (boundary case) \cite{MR2642894}.  If we have some means of computing $\alpha$, then we obtain $\Delta = 1-\alpha/\gamma$.

The results of this paper allow us to obtain $\alpha$ once we know the corresponding weight $W$.  In some cases, one can argue that near a quantum typical point, one sees a welding of a number of different types of wedges, and by summing these weights, we then obtain $W$, which allows us to compute $\alpha$ and then $\Delta$.  The Euclidean counterparts of the quantum exponents $\Delta$ are then obtained from the celebrated Knizhnik-Polyakov-Zamolodchikov (KPZ) relation \cite{MR947880,MR1005268,MR981529,ds2011kpz,PSS:8474530,MR3215583}.  It relates the quantum and Euclidean (expectation) scaling exponents~$\Delta$ and~$x$ of a random fractal~$X$, sampled independently of a $\gamma$-LQG surface, as
\begin{equation*}
 x = \frac{\gamma^2}{4} \Delta^2 + \left(1 - \frac{\gamma^2}{4} \right) \Delta.
\end{equation*}
(See \cite[Theorem~1.4]{ds2011kpz} for the bulk version and \cite[Theorem~6.1]{ds2011kpz} for the boundary version.) This, in turn, should give the Euclidean dimension of $X$ as $2-2x$ (resp.\ $1-x$) in the  bulk  (resp. on the boundary).

The original (heuristic) formulation of the KPZ relation dealt with the so-called gravitational dressing of primary operators coupled to gravity in conformal field theory, whereas the recent probabilist approach showed it to hold for any random fractal sampled independently of the GFF. A construction of quantum measures on fractals in the spirit of \cite{DS3} could help bridge the gap between these two approaches.

\subsection{KPZ overview}

Suppose that $D \subseteq \C$ is a domain with smooth boundary and that $X \subseteq \ol{D}$ is a random fractal which is independent of a $\gamma$-LQG with $\gamma \in (0,2)$ surface described by $h$.  Let $B_\epsilon(X)$ (resp.\ $B^\delta(X)$) be the $\epsilon$-Euclidean neighborhood (resp.\ $\delta$-quantum neighborhood) of $X$ as defined in \cite[Section~1.3]{ds2011kpz}.  Recall that the quantum and Euclidean expectation scaling exponents $\Delta$ and $x$ are respectively given by:
\[\Delta = \lim_{\delta \to 0} \frac{\log \E[\mu_h(B^\delta(X))]}{\log \delta} \quad\text{and}\quad  x = \lim_{\epsilon \to 0} \frac{\log \E[\mu_0(B_\epsilon(X))]}{\log \epsilon^2}\]
where $\mu_0$ denotes Lebesgue measure.  The boundary exponents are defined analogously if $X \subseteq \partial D$ except with $\nu_h$ in place of $\mu_h$ and $\epsilon$ in place of $\epsilon^2$ (see \cite[Section~6.3]{ds2011kpz}).  If $X \subseteq D$ (resp.\ $X \subseteq \partial D$), then $2-2x$ (resp.\ $1-x$) gives the Euclidean expectation dimension of~$X$.  The KPZ formula, which relates $\Delta$ and $x$, is given by:
\begin{equation}
\label{eqn::kpz}
 x = \frac{\gamma^2}{4} \Delta^2 + \left(1 - \frac{\gamma^2}{4} \right) \Delta.
\end{equation}
(See \cite[Theorem~1.4]{ds2011kpz} for the bulk version and \cite[Theorem~6.1]{ds2011kpz} for the boundary version.  See also \cite{PSS:8474530,aru2015kpz,bgr2016kpz,ghm2015kpz} for other versions of the KPZ relation.)

Following the discussion begun in Appendix~\ref{subapp::kpz}, we are now going to use~\eqref{eqn::kpz} and the quantum wedge and cone theory developed earlier in this article to give a heuristic derivation of both the quantum and Euclidean dimensions of a number of random fractals which arise in $\SLE$ and related discrete models.  As we mentioned in Appendix~\ref{subapp::kpz}, we believe that with some additional work these arguments can be converted into rigorous proofs.

\subsection{General principles}

\begin{figure}[ht!]
\begin{center}
\subfloat[\label{subfig::sle_kp_wedges}]{\includegraphics[scale=0.81,page=4]{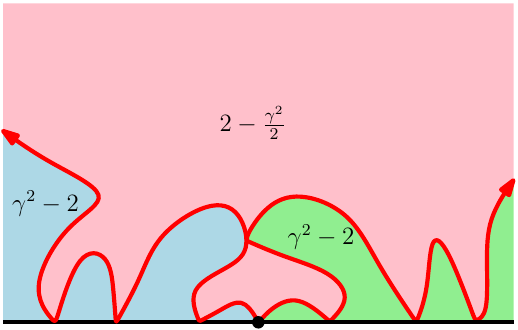}}
\hspace{0.00\textwidth}
\subfloat[\label{subfig::sle_kp_non_intersection_wedges}]{\includegraphics[scale=0.81,page=2]{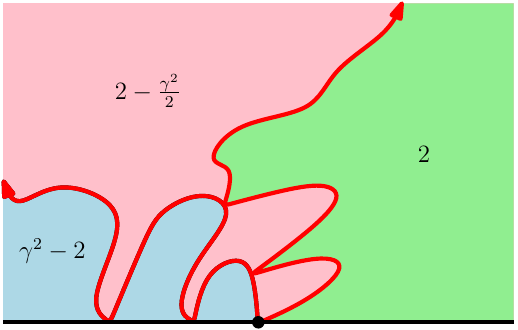}}
\end{center}
\caption{\label{fig::sle_kp_wedges} Fix $\kappa' \in (4,8)$.  {\bf Left:} an $\SLE_{\kappa'}$ process (red) divides a quantum wedge of weight $\tfrac{3\gamma^2}{2} - 2$ into three independent wedges of respective weights $\gamma^2-2$ (blue), $2-\tfrac{\gamma^2}{2}$ (red), and $\gamma^2-2$ (green) corresponding to those regions which are to the left of the path, surrounded by the path, and to the right of the path.  {\bf Right:} an $\SLE_{\kappa'}$ process (red) \emph{conditioned} not to hit $(0,\infty)$ divides a quantum wedge of weight $\tfrac{\gamma^2}{2}+2$ into three independent wedges of weight $\gamma^2-2$ (blue), $2-\tfrac{\gamma^2}{2}$ (red), and $2$ (green) corresponding to those regions which are to the left of the path, surrounded by the path, and to the right of the path.}
\end{figure}

The general strategy is the following:
\begin{itemize}
\item We will (using certain heuristics) argue that the local behavior at a ``quantum typical'' point in $X$ looks like a particular type of $\alpha$-quantum cone (bulk case) or $\alpha$-quantum wedge (boundary case).  By quantum typical, we mean a point chosen from the measure on $X$ constructed by restricting the $\gamma$-LQG measure to $B^\delta(X)$ and then taking a limit of these measures as $\delta \to 0$ with appropriate rescaling so that the limit is non-trivial.  (We will not actually construct this measure in each case; see \cite{DS3} for $\SLE$ quantum measures, in particular the $\SLE$ quantum length or natural parameterization.)
\item It is explained in \cite[Equation~(63)]{ds2011kpz} that a quantum typical point $z_0$ in $X$ is $\gamma - \gamma \Delta$ thick for $h$; this means that the circle average process of $h$ centered at $z_0$ has growth $(\gamma-\gamma \Delta)\log \epsilon^{-1}$ as $\epsilon \to 0$.  The circle average process associated with an $\alpha$-quantum cone or an $\alpha$-quantum wedge centered at the origin grows like $\alpha \log \epsilon^{-1}$ as $\epsilon \to 0$ and this allows us to match $\alpha$ and~$\Delta$.  That is, $\alpha = \gamma (1 - \Delta)$, hence $\Delta = 1-\tfrac{\alpha}{\gamma}$. 
\item We can determine $x$ from $\Delta$ using~\eqref{eqn::kpz}.  This, in turn, should give that the Euclidean dimension of $X$ is $2-2x$ (resp.\ $1-x$) if $X \subseteq D$ (resp.\ $X \subseteq \partial D$).
\end{itemize}

We emphasize again that the presentation in this section will be heuristic, and not fully detailed.  In this section, we are going to consider two different settings depending on whether $\kappa=\gamma^2 \in (0,4)$ or $\kappa' =\frac{16}{\gamma^2} \in (4,8)$.

In the former setting, there are a few wedge types which occur naturally: 

\begin{itemize}
\item Wedge of weight $2$, which is the weight of an $\gamma$-quantum wedge  obtained by zooming in at a ``typical boundary measure'' point, as mentioned in Section~\ref{subsec::surfaces}. In the simplest setting of Theorem~\ref{thm::welding}, gluing two weight $2$ wedges yields a wedge of weight $4$ decorated by an independent chordal $\SLE_\kappa$ process from $0$ to $\infty$ \cite{she2010zipper}. 
 \item Wedges of weight $W_i=2 + \rho_i, i=1,2$, obtained by cutting out from a quantum wedge of weight $W=W_1+W_2$, an independent $\SLE_\kappa(\rho_1;\rho_2)$ process from $0$ to $\infty$ with force points located at $0^-$ and $0^+$. (See Theorem~\ref{thm::welding}.)  
  \item Wedge of weight $W=2 + \rho$, obtained by cutting out from a quantum cone $\CC = (\C,h,0,\infty)$ of same weight $W$, an independent  whole-plane $\SLE_\kappa(\rho)$ process $\eta$  starting from~$0$. (See Theorem~\ref{thm::zip_up_wedge_rough_statement}.)
   \item Wedge of weight $W=2 + \rho$,  corresponding to the region cut off by a boundary-intersecting $\SLE_{\kappa}(\rho)$ process before or after hitting a typical boundary intersection point.
  \end{itemize}

In the latter setting, we fix $\kappa' \in (4,8)$ and suppose that $\eta'$ is an $\SLE_{\kappa'}$ process in $\h$ from $0$ to $\infty$.  Fix a quantum wedge $\CW$ parameterized by $\h$ which is independent of $\eta'$.  In the computations that follow, there will be three different types of wedges that will be important for us to consider.  These correspond to:
\begin{itemize}
\item The region surrounded by $\eta'$,
\item The region which is to the left or to the right of $\eta'$,
\item The region which is to the right of $\eta'$ when it is conditioned not to hit $(0,\infty)$.
\end{itemize}
In view of Theorem~\ref{thm::sle_kp_on_wedge}, it is natural in the first two cases to take $\CW$ to have weight $\tfrac{3\gamma^2}{2}-2$.  Theorem~\ref{thm::sle_kp_on_wedge} implies that the regions which are to the left and to the right of $\eta'$ correspond to wedges of weight $\gamma^2-2$ and the set of points surrounded by $\eta'$ corresponds to a wedge of weight $2-\tfrac{\gamma^2}{2}$.  Moreover, these three wedges are independent.  (See left side of Figure~\ref{fig::sle_kp_wedges}.)

Recall that the law of $\eta'$ conditioned not to hit $(0,\infty)$ is that of an $\SLE_{\kappa'}(\kappa'-4)$ process \cite{dub2005mg_duality,ms2012imag2}.  By Theorem~\ref{thm::sle_kp_on_wedge}, it is natural to take $\CW$ to be a wedge of weight $\tfrac{\gamma^2}{2}+2$ and it follows that the region of~$\h$ which is to the right of $\eta'$ corresponds to a wedge of weight $2$.\footnote{This is also natural because the dimension of the Bessel process which corresponds to a wedge of weight $\gamma^2-2$ is $3-\tfrac{4}{\gamma^2}$ (recall Table~\ref{tab::wedge_parameterization}).  Conditioning this Bessel process not to hit~$0$ corresponds to conditioning the wedge (together with its prime-end boundary) to be homeomorphic to $\ol{\h}$.  The conditioned Bessel process has dimension $4-(3-\tfrac{4}{\gamma^2}) = 1+\tfrac{4}{\gamma^2}$ and this is the Bessel dimension which corresponds to a weight $2$ wedge.}  As in the unconditioned case, the region surrounded by $\eta'$ itself is a wedge of weight $2-\tfrac{\gamma^2}{2}$ and the region which is to the left of $\eta'$ to a wedge of weight $\gamma^2-2$.  Moreover, these three wedges are independent.  (See right side of Figure~\ref{fig::sle_kp_wedges}.)

We can now glue together wedges of these three types in various ways to form other types of surfaces.  In particular,
\begin{itemize}
\item A wedge of weight $2-\tfrac{\gamma^2}{2}$ corresponds to an $\SLE_{\kappa'}$ process,
\item Gluing a wedge of weight $2$ between two wedges of weight $2-\tfrac{\gamma^2}{2}$ corresponds to conditioning the two $\SLE_{\kappa'}$ processes not to intersect, and
\item Gluing a wedge of weight $\gamma^2-2$ in between two wedges of weight $2-\tfrac{\gamma^2}{2}$ corresponds to not performing any conditioning (i.e., interpreting the paths as though they were $\SLE_{\kappa'}$ interfaces arising from the same statistical physics model, which reflect off one another but do not overlap).
\end{itemize}

For example, consider the surface which results when gluing together independent wedges $\CW_1,\ldots, \CW_5$ with respective weights $\gamma^2-2$, $2-\tfrac{\gamma^2}{2}$, $2$, $2-\tfrac{\gamma^2}{2}$, and $\gamma^2-2$.  This corresponds to having two $\SLE_{\kappa'}$ paths $\eta_1'$ and $\eta_2'$ (represented respectively by $\CW_2$ and $\CW_4$) in $\h$ from $0$ to $\infty$ with $\eta_1'$ to the left of and conditioned not to hit $\eta_2'$.
\begin{remark}\label{rk:dualexponents}
Besides the quantum scaling exponents $\Delta$, the Liouville quantum gravity literature discusses so-called {\it dual} quantum scaling exponents $\wt\Delta$ \cite{kleb1995touching_surfaces,kleb1995non_perturb,kleb1996wormholes,dup2004qg,LH2005,ds2009qg_prl,2008ExactMethodsBD,bjrv2013super_critical}, defined via the duality relation 
\begin{equation}\label{def::dualdim}
\alpha=\gamma(1-\Delta)=\gamma'(1-\wt \Delta),
\end{equation}
where $\gamma':=\sqrt{\kappa'}=4/\sqrt{\kappa}=4/\gamma$ is the dual $\gamma'>2$ of the usual LQG parameter $\gamma<2$. 
(We then have $\wt\Delta=1-\frac{\gamma^2}{4}+\frac{\gamma^2}{4}\Delta$.)  This remark will provide an extremely rough sense of the meaning of these exponents, and related formulas that have been derived in the physics literature. As argued above, the local behavior at a ``quantum typical'' point in a random fractal $X$ of exponent $\Delta$, where the measure on $X$ is constructed by using the $\gamma$-LQG measure, looks like a particular type of $\alpha$-quantum cone (bulk case) or $\alpha$-quantum wedge (boundary case).  In a certain sense, the dual scaling exponents $\wt\Delta$ should correspond to the same value of $\alpha$, but should arise when one chooses a typical point (conditioned to belong to $X$) using the atomic quantum measure $\mu_{\gamma'}$ introduced in Section~\ref{subsec::duality}, 
 instead of the standard Liouville measure $\mu=\mu_\gamma$ \eqref{e.mudef} (or the boundary atomic measure corresponding to the quantum natural time, instead of the standard boundary LQG measure $\nu=\nu_\gamma$ \eqref{e.nudef}), as discussed in~\eqref{e.dualmeasure} in Section~\ref{subsec::kpz_duality}. The dual quantum exponents obey a dual version of the KPZ formula, which relates the triple $\gamma'$, $\wt\Delta$ and $x$ \cite{kleb1995touching_surfaces,dup2004qg,ds2009qg_prl,2008ExactMethodsBD}, and was proved rigorously in the context of Gaussian multiplicative chaos \cite{bjrv2013super_critical},
\begin{equation}
\label{eqn::kpzdual}
 x = \frac{\gamma'^{\,2}}{4} \wt\Delta^2 + \left(1 - \frac{\gamma'^{\,2}}{4} \right) \wt\Delta;
\end{equation} 
standard and dual quantum exponents are then such that the identity $x=\Delta\wt\Delta$ holds.
In the wedge parameterization, Table \ref{tab::wedge_parameterization}, the relation $W=2+\gamma^2\Delta$ becomes in terms 
of the dual quantum exponent $\wt\Delta$,
\begin{equation}\label{weightdual}
W=\gamma^2-2+4\wt\Delta.
\end{equation}

We can then generalize to $\SLE_{\kappa'}$ processes, the Remark \ref{rem::linearity_of_wedges}, made in the case 
of $\SLE_\kappa$, that the additivity of wedge weights under gluing should correspond to the additivity of certain non-intersection (boundary) quantum exponents. Indeed, we have seen above that gluing a wedge of weight $\gamma^2-2$ in between the boundary and a wedge (of weight $2-\tfrac{\gamma^2}{2}$) corresponding to an $\SLE_{\kappa'}$ process (see Fig. \ref{fig::sle_kp_wedges}), or between two wedges  corresponding to two $\SLE_{\kappa'}$ processes, corresponds to not performing any conditioning between the boundary and the  $\SLE_{\kappa'}$ path, or between the two paths. The addititivity of wedge weights in Theorem~\ref{thm::sle_kp_on_wedge} therefore implies that $W-(\gamma^2-2)$ should be a linear function of the number of paths.   The relation \eqref{weightdual} then implies a similar additivity of dual quantum boundary exponents with respect to the set of paths \cite{dup2004qg,LH2005}. An illustration of this fact will appear in Appendix~\ref{subapp::SLEdoublepointdimension} when computing the $\SLE_{\kappa'}$ double-point dimension and the generalization thereof. 
\end{remark}
  
\subsection{Non-intersecting paths and cut point dimension}
\begin{figure}[ht!]
\begin{center}
\label{fig::nsle_k_wedges}
\includegraphics[scale=0.85]{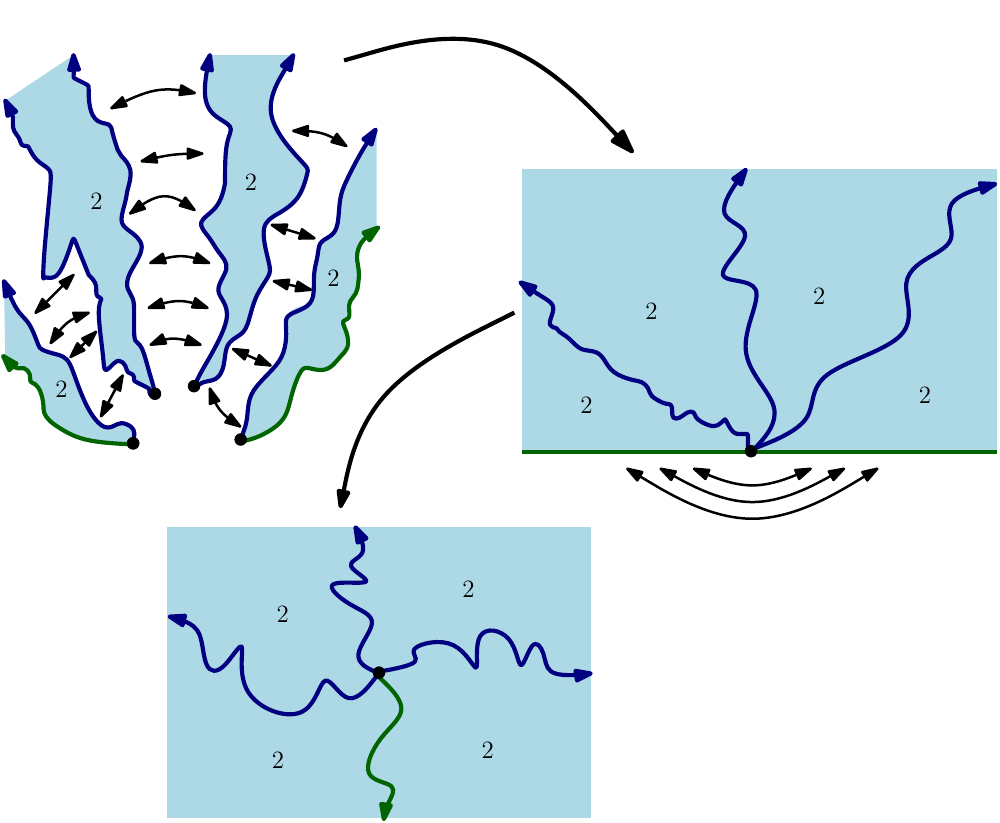}
\end{center}
\caption{\label{fig::nsle_k_wedges} {\bf Top:} Conformal welding according to their $\gamma$-LQG boundary length of $n(=4)$ independent identical quantum wedges of weight $2$, followed by a conformal mapping to $\H$. The resulting $n-1$ interfaces, $\eta_i$, $i=1,\cdots, n-1$, are coupled $\SLE_\kappa(\rho_i,\wt{\rho}_i)$ processes, with $\kappa=\gamma^2\in (0,4)$, and $\rho_i=2i-2,\wt{\rho}_i=2(n-i)-2$.  {\bf Bottom:} Conformally welding the left and right boundaries ($\R_-$ and $\R_+$) of the resulting wedge $\qwedge$ according to quantum length yields a quantum cone decorated with an independent whole-plane $\SLE_\kappa(\rho)$, with $\rho=2(n-1)$. Its conditional law, given the other $n-1$ paths, is that of a chordal $\SLE_\kappa$ process in the sector between $\eta_{n-1}$ and $\eta_1$. By symmetry, this also holds  for any path given the others, resulting in a non-intersecting $n$-star configuration.}
\end{figure}
Consider first the simplest case of a quantum wedge $\qwedge$ of weight $W=2n$, made by gluing together $n$ independent quantum wedges $\qwedge_i$, $i=1,\cdots,n$, each of weight~$2$. By the remarks above, and by Theorem~\ref{thm::welding} and Proposition~\ref{prop::slice_wedge_many_times}
, the resulting $n-1$ interfaces are represented by a collection of paths $\eta_1,\ldots,\eta_{n-1}$ where the marginal law of $\eta_i$ for each $i = 1,\ldots,n-1$ is that of an $\SLE_\kappa(2i-2; 2(n-i)-2)$ process and the conditional law of $\eta_i$ given $\eta_j$ for $j \neq i$ is that of an $\SLE_\kappa$ process, $\kappa=\gamma^2 \in (0,4)$, in the component of $\h \setminus \cup_{j \neq i} \eta_j$ between $\eta_{i-1}$ and $\eta_{i+1}$ with the convention that $\eta_0 = \R_-$ and $\eta_n = \R_+$.  See Figure~\ref{fig::nsle_k_wedges} for an illustration.  Then $\qwedge$ is an $\alpha$-quantum wedge with $\alpha=\gamma-\frac{2}{\gamma}(n-1)\leq Q$, corresponding to a boundary quantum scaling exponent 
\begin{equation}
\label{DeltanB}
\Delta=1-\frac{\alpha}{\gamma}=\frac{2}{\gamma^2}(n-1)=\frac{2}{\kappa}(n-1).
\end{equation}
Observe here the additivity of boundary quantum exponents, as advocated in \cite{MR1964687,dup2004qg}, with a contribution $2/\kappa$ for each of the $n-1$ mutually-avoiding paths. 

By Theorem~\ref{thm::zip_up_wedge_rough_statement}, zipping-up the left and right sides of the thick wedge $\qwedge$ gives an $\alpha'$-quantum cone decorated with an independent whole-plane $\SLE_\kappa(\rho)$, with from~\eqref{eqn::cone_wedge_value}, $\alpha'=\frac{\alpha}{2}+\frac{1}{\gamma}$, and from~\eqref{eqn::cone_rho} $\rho=2+\gamma^2-2\alpha'\gamma=\gamma^2-\alpha\gamma=2(n-1)$.  The conditional law of the path which corresponds to the zipped up boundary given the other paths is that of a chordal $\SLE_\kappa$ process (see Proposition~\ref{prop::slice_cone_many_times}).  In particular, this path does not intersect the other $n-1$ paths originally drawn on the quantum wedge $\qwedge$ (see also \cite[Section~10.5]{LH2005}); this yields a collection of $n$ non-intersecting whole-plane simple $\SLE_\kappa(\rho)$ processes, with $\kappa \in (0,4)$.  The corresponding bulk quantum scaling exponent is
\begin{equation}
\label{Deltan}
\Delta'=1-\frac{\alpha'}{\gamma}=\frac{1}{2\kappa}\left(2n+\kappa-4\right).
\end{equation}
Both $\Delta$~\eqref{DeltanB} and $\Delta'$~\eqref{Deltan} coincide (see, e.g., \cite[Equation (11.36)]{dup2004qg}) with the quantum exponents obtained by random matrix techniques for multiple paths in the $O(N)$ critical model on a random lattice \cite{1988PhRvL..61.1433D,1989MPLA....4..217K,1990NuPhB.340..491D,MR1964687,dup2004qg,LH2005}, with the conjectural identification $N=-2\cos (4\pi/\kappa), N\in [-2,2], \kappa\in [2,4]$ \cite{MR1964687,kg2004guide_to_sle}.

The corresponding Euclidean scaling exponents, i.e., conformal weights, are obtained by using the KPZ relation~\eqref{eqn::kpz}, to get 
\begin{equation}
\label{xn}
x=\frac{n}{2\kappa}(2n+4-\kappa) \quad\text{and}\quad x'=\frac{1}{16\kappa}\big( 4n^2-(4-\kappa )^2\big),
\end{equation} for, respectively,   the boundary and bulk exponents of $n$ non-intersecting  $\SLE_\kappa$ paths with $\kappa < 4$.  They agree with the corresponding exponents predicted for the critical $O(N)$ model or the tricritical Potts model; see \cite{1984NuPhB.240..514C,PhysRevLett.57.3179,0305-4470-26-16-004} and \cite{PhysRevLett.49.1062,Nien84,DG1987,0305-4470-19-13-009,dupJSP97,PhysRevLett.61.138}.
\begin{figure}[ht!]
\begin{center}
\subfloat[\label{subfig::one_path_non_intersection}]{\includegraphics[scale=0.81,page=2]{figures/sle_kp_non_intersecting}}
\hspace{0.00\textwidth}
\subfloat[\label{subfig::two_path_non_intersection}]{\includegraphics[scale=0.81,page=1]{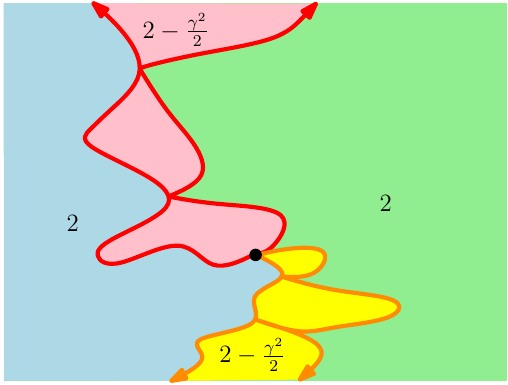}}
\end{center}
\caption{{\bf Left:} quantum cone formed by gluing a weight $2-\tfrac{\gamma^2}{2}$ wedge with a weight $2$ wedge.  The wedge of weight $2-\tfrac{\gamma^2}{2}$ corresponds to the range of a whole-plane $\SLE_{\kappa'}$ process $\eta'$ conditioned on the event that it does not separate the origin from $\infty$.  (The conditional law of $\eta'$ is that of a whole-plane $\SLE_{\kappa'}(\kappa'-4)$ process.)  {\bf Right:} quantum cone formed by gluing two weight $2-\tfrac{\gamma^2}{2}$ wedges with two weight $2$ wedges, alternating between the two types.  The two wedges of weight $2-\tfrac{\gamma^2}{2}$ correspond to two whole-plane $\SLE_{\kappa'}$ processes starting from the origin conditioned not to intersect each other.  This describes the local behavior of an $\SLE_{\kappa'}$ process near a quantum typical cut point.}
\end{figure}

We are now going to explain how to derive the Brownian intersection exponents and the cut point dimension for $\SLE_{\kappa'}, \kappa'\in (4,8)$.

Consider a wedge of weight $\CW$ which is given by gluing together independent wedges with weights $2-\tfrac{\gamma^2}{2}$ and $2$.  Then the weight of $\CW$ is $4-\tfrac{\gamma^2}{2}$.  That is, $\CW$ is an $\alpha$-quantum wedge with $\alpha = \tfrac{3 \gamma}{2} - \tfrac{2}{\gamma}$.  Let
\[ \alpha' = \frac{\alpha}{2} + \frac{1}{\gamma} = \frac{3 \gamma}{4}.\]
Theorem~\ref{thm::zip_up_wedge_rough_statement} (recall also~\eqref{eqn::cone_rho} and~\eqref{eqn::cone_wedge_value}) implies that if we zip up the left and right sides of $\CW$, then we get an $\alpha'$-quantum cone $\CC = (\C,h,0,\infty)$ decorated with an independent whole-plane $\SLE_{\kappa'}(\kappa'-4)$ process $\eta'$ from $0$ to $\infty$ (see also Theorem~\ref{thm::quantum_cone_sle_kp}; equivalently, $\eta'$ is a whole-plane $\SLE_{\kappa'}$ process conditioned not to disconnect $0$ from $\infty$).  See Figure~\ref{subfig::one_path_non_intersection} for an illustration.

Suppose that $\kappa'=6$ and imagine that $B$ is a standard $2$-dimensional Brownian motion drawn on top of an independent $\sqrt{8/3}$-LQG surface described by some $h$.  Then we can view $(\CC,\eta')$ as describing the local picture of the set $X$ consisting of those points $z$ where $B$ gets very close to and then travels macroscopically far away from without separating $z$ from $\infty$.  Let $\Delta$ denote the quantum scaling associated with this set.  As explained earlier, a quantum typical point in $X$ should be $\gamma-\gamma \Delta$ thick for $h$.  That is, we should have that $\alpha' = \gamma-\gamma \Delta$.  This gives that $\Delta = \tfrac{1}{4}$.  Plugging this into~\eqref{eqn::kpz} gives that the Euclidean scaling of $X$ is $x = \tfrac{1}{8}$ so that the Euclidean dimension of $X$ should be $\tfrac{7}{4}$, in agreement with that of pioneer points of Brownian motion \cite{zbMATH01690755}.

We can generalize this further by considering additional paths.  Suppose now that we have $n$ paths $\eta_1',\ldots,\eta_n'$ in $\h$ from $0$ to $\infty$ where each $\eta_j'$ lies to the right of $\eta_{j-1}'$ and to the left of $\eta_{j+1}'$ and that the conditional law of $\eta_j'$ given $\eta_{j-1}'$ and~$\eta_{j+1}'$ is that of an $\SLE_{\kappa'}$ conditioned not to hit~$\eta_{j-1}'$ and~$\eta_{j+1}'$.  (Equivalently, the conditional law is that of an $\SLE_{\kappa'}(\kappa'-4;\kappa'-4)$ process \cite{dub2005mg_duality,ms2012imag2}.)  If we draw these paths on top of an independent wedge $\CW$ of weight $W=(4-\tfrac{\gamma^2}{2})n$, then their left and right boundaries divide $\CW$ into independent wedges $\CW_1,\CV_1,\ldots,\CW_n,\CV_n$ where each~$\CW_i$ has weight $2-\tfrac{\gamma^2}{2}$ and each~$\CV_j$ has weight~$2$.

Welding together the left and right sides of $\CW$ yields an $\alpha$-quantum cone, where by~\eqref{eqn::quantum_cone_weight} 
\[ \alpha =Q-\frac{W}{2\gamma}= Q-\left(\frac{2}{\gamma}-\frac{\gamma}{4}\right)n
,\]
decorated by paths $\wt{\eta}_1',\ldots,\wt{\eta}_n'$ which can be viewed as $n$ whole-plane $\SLE_{\kappa'}$ processes starting from the origin conditioned not to intersect.  See Figure~\ref{subfig::two_path_non_intersection} for an illustration.  
Let $\Delta$ be so that $\alpha = \gamma - \gamma \Delta$.  That is, 
\[ \Delta = 1-\frac{\alpha}{\gamma} = \frac{1}{2}-\frac{2}{\gamma^2}+\left(\frac{2}{\gamma^2}-\frac{1}{4}\right)n
.\]
This exponent matches the multiple disconnection exponent of $n$ $\SLE_{\kappa'}$ paths,  obtained by using the quantum gravity methods of \cite[Section 12.3]{dup2004qg}.
In the case that $\gamma=\sqrt{8/3}$, this gives that
\[ \Delta = \frac{1}{2}\left(n - \frac{1}{2} \right)\]
so that by~\eqref{eqn::kpz} we have that
\[ x = \frac{1}{24}\left(4 n^2 - 1\right).\]
This case has the interpretation of describing $n$ whole-plane $\SLE_6$ processes conditioned not to intersect (or $n$ non-intersecting  clusters or $2n$ path crossings in percolation \cite{PhysRevLett.58.2325,PhysRevLett.82.3940,PhysRevLett.83.1359}).  By the relationship between $\SLE_6$ and Brownian motion, it in turn has the interpretation as describing $n$ Brownian motions starting near the origin conditioned not to intersect.  In particular, the probability that $n$ Brownian motions starting from distance $\epsilon$ from the origin make it macroscopically far from the origin without intersecting should be proportional to $\epsilon^{2x}$.  These are the Brownian intersection exponents derived in \cite{dk1998rwks} and later in \cite{dup1998rwk} using quantum gravity methods.  The values of the Brownian intersection exponents were determined rigorously by Lawler, Schramm, and Werner in \cite{lsw2001intersection1,lsw2001intersection2,lsw2002intersection3}.

These exponents for other values of $\gamma \in (\sqrt{2},2)$ also have an interesting interpretation.  In particular, for $n=2$ and $\gamma = \sqrt{16/\kappa'}$, $\Delta$ and $x$ have the interpretation of giving the quantum and Euclidean scaling exponents for the cut points of an $\SLE_{\kappa'}$ process (Figure~\ref{subfig::two_path_non_intersection}).  This translates into a prediction of the dimension of the cut points given by
\begin{equation}
\label{eqn::cut_point_dimension}
3-\frac{3\kappa'}{8}.
\end{equation}
This matches the prediction for the $n=2$ $\SLE$ disconnection exponent  \cite[Equations~(12.48),~(12.50)]{dup2004qg} as well as the  rigorous derivation of this dimension given in \cite[Theorem~1.2]{mw2013intersections}.

\subsection{$\SLE_{\kappa'}$ double point dimension}
\label{subapp::SLEdoublepointdimension}

\begin{figure}
\begin{center}
\includegraphics[scale=0.85]{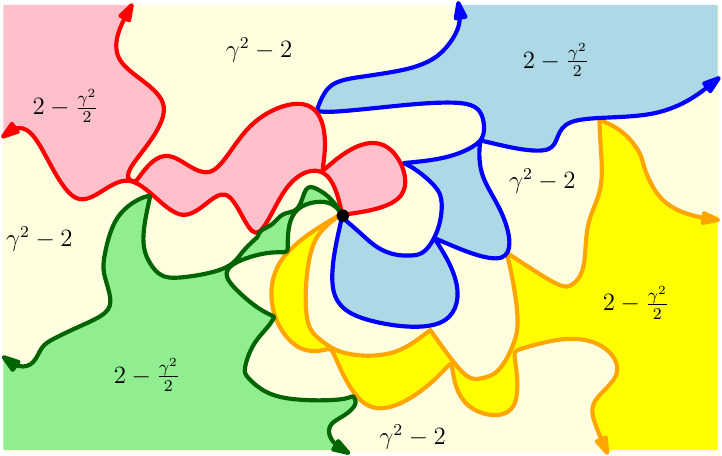}
\end{center}
\caption{\label{fig::sle_double_point} The configuration of paths which at a quantum typical double point $z_0$ for an $\SLE_{\kappa'}$ process $\eta'$ with $\kappa' \in (4,8)$ can be described by the quantum cone which arises by welding together four wedges of weight $2-\tfrac{\gamma^2}{2}$ and four wedges of weight $\gamma^2-2$ where the wedges of different types alternate as shown.  The boundaries of the $2-\tfrac{\gamma^2}{2}$ wedges are drawn by $\eta'$ either in counterclockwise or clockwise order depending on whether $z_0$ is on the left or right side of $\eta'$, respectively.}
\end{figure}

In order for a point $z_0$ to a be a double point for an $\SLE_{\kappa'}$ process $\eta'$, it has to be that $\eta'$ hits $z_0$ and then gets macroscopically far away from $z_0$ before returning.  Therefore we can represent the local behavior of $\eta'$ near a quantum typical double point by gluing together eight wedges of alternating weight $2-\tfrac{\gamma^2}{2}$ and $\gamma^2-2$.  The former correspond to the two segments of $\eta'$ that approach $z_0$ and the two segments of $\eta'$ that leave from $z_0$ (for a total of four).  The latter correspond to the regions between these four segments.  The reason that it is natural to expect these latter regions to be $\gamma^2-2$ wedges is that if one conditions on three of the four segments of $\eta'$, then the conditional law of the fourth should be an $\SLE_{\kappa'}$ (no extra $\rho$ values).  See Figure~\ref{fig::sle_double_point} for an illustration.

By Proposition~\ref{prop::slice_cone_many_times}, gluing together these eight wedges yields a $\big(\tfrac{2}{\gamma} - \tfrac{\gamma}{2}\big)$-quantum cone.  This leads to the following predictions of the quantum and Euclidean scaling exponents
\begin{equation}
\label{eqn::double_point_exponents}
\Delta = \frac{3}{2} - \frac{\kappa'}{8} \quad\text{and}\quad x = \frac{1}{2} + \frac{3}{\kappa'} - \frac{\kappa'}{16}.
\end{equation}
This, in turn, leads to the prediction that the Euclidean dimension of the double points should be
\begin{equation}
\label{eqn::double_point_dimension}
2 - \frac{(12-\kappa')(4+\kappa')}{8\kappa'}.
\end{equation}
This matches the heuristic derivation given by in \cite{PhysRevLett.58.2325,dupJSP97} in the context of contours of the FK model and the rigorous derivation given in \cite[Theorem~1.1]{mw2013intersections}.

This can be generalized to  the local behavior of $\eta'$ near a quantum typical multiple point of order $n$, by gluing together $2n$ wedges of alternating weight $2-\frac{\gamma^2}{2}$ and $\gamma^2-2$, yielding a quantum wedge of total weight $W=\tfrac{n\gamma^2}{2}$. Zipping up its two sides, one gets by Proposition~\ref{prop::slice_cone_many_times} an  $\alpha$-quantum cone  with $\alpha= Q-\frac{W}{2\gamma}$, and a quantum scaling dimension
\begin{equation}
\label{DeltanD}
\Delta=1-\frac{\alpha}{\gamma}=\frac{1}{2}+\frac{n}{4}-\frac{\kappa'}{8}.
\end{equation}
For describing the local behavior near a boundary quantum typical multiple point of order $n$, one glues together in an alternating way $n+1$ quantum wedges of weight $\gamma^2-2$ with the $n$ quantum wedges of weight $2-\tfrac{\gamma^2}{2}$ drawn by $\eta'$, yielding an  $\alpha'$-quantum wedge of total weight $W'=(n+2)\frac{\gamma^2}{2}-2$, with $\alpha'=\frac{4}{\gamma}-n\frac{\gamma}{2}$, and  a boundary  quantum exponent
 \begin{equation}
 \label{DeltanBD}
 \Delta'=1-\frac{\alpha'}{\gamma}=1+\frac{n}{2}-\frac{\kappa'}{4}.
 \end{equation} 
 Equations~\eqref{DeltanD} and~\eqref{DeltanBD} match the  exponents of multiple non-simple $\SLE_{\kappa'}$ paths   obtained by quantum gravity methods in \cite[Equation (11.37)]{dup2004qg}, together with the boundary-bulk quantum rule for non-simple paths, $\Delta'=2\Delta$. Notice that the {\it dual} dimension \eqref{def::dualdim} associated with $\Delta'$ \eqref{DeltanBD}, is $\wt\Delta'=\frac{2}{\kappa'}n$, hence linear in $n$, in agreement with Remark \ref{rk:dualexponents} above, and is the analog of \eqref{DeltanB} with $\kappa$ simply replaced by $\kappa'$ \cite{dup2004qg}. These results also  match the multiple path exponents obtained by random matrix techniques for the so-called dense phase of the $O(N)$ model on a random lattice  \cite{1988PhRvL..61.1433D,1989MPLA....4..217K,1990NuPhB.340..491D,Kazakov1992520,dup2004qg}.

The corresponding boundary and bulk Euclidean exponents of $n$ non-intersecting  $\SLE_{\kappa'}$ paths, with $\kappa' \geq 4$, are obtained by using the KPZ relation~\eqref{eqn::kpz}, which yields the same results as in~\eqref{xn}, with $\kappa$ there simply replaced by $\kappa'$.  They agree with the corresponding exponents predicted for the dense phase critical $O(N)$ model or, equivalently,  for the critical Potts model \cite{PhysRevB.27.1674,Nien84,DG1987,0305-4470-19-16-011,Duplantier1987291,PhysRevLett.58.2325,dupJSP97,PhysRevLett.61.138,0305-4470-25-4-009,0305-4470-26-16-004}.
 
\subsection{Boundary intersection dimension for $\SLE_\kappa(\rho)$}

\begin{figure}
\begin{center}
\subfloat[\label{subfig::sle_kr_boundary}]{\includegraphics[scale=0.81,page=3]{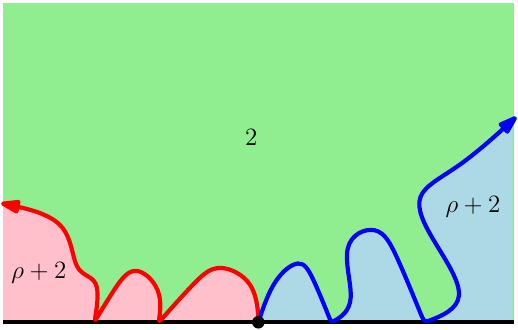}}
\hspace{0.00\textwidth}
\subfloat[\label{subfig::gff_flow_line_intersection}]{\includegraphics[scale=0.81,page=5]{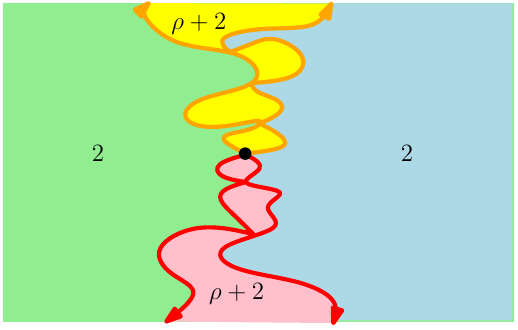}}
\end{center}
\caption{{\bf Left:} a typical boundary intersection point $z_0$ for a boundary-intersecting $\SLE_{\kappa}(\rho)$ process looks like the welding of wedges of weight~$\rho+2$,~$2$, and~$\rho+2$.  The first and third wedges correspond to the region cut off by the path before and after hitting $z_0$ and the weight $2$ wedge corresponds to the region which is not cut off by the path.  {\bf Right:} a typical intersection point for two GFF flow lines with relative angle $\theta$ looks like the quantum cone which arises by welding together wedges of weights $\rho+2$, $2$, $\rho+2$, and $2$ where $\rho = \theta \chi/\lambda -2$.}
\end{figure}

Suppose that $\eta$ is a boundary intersecting $\SLE_\kappa(\rho)$ process with its force point immediately to the right of its seed.  We know from Theorem~\ref{thm::welding} that if we draw~$\eta$ on top of an independent wedge $\CW$ of weight $\rho+4$ then $\eta$ divides $\CW$ into independent wedges of weight $2$ and $\rho+2$.  From Section~\ref{subsec::quantum_typical_zip_unzip}, we also know that the joint law of $(\CW,\eta)$ is invariant under the operation of cutting along $\eta$ until reaching a ``quantum typical'' intersection point of $\eta$ with $\partial \h$.  Hence, the law of the part of $\CW$ which is to the right of $\eta$ at a quantum typical intersection point should be a wedge of weight $\rho+2$.  By symmetry, such a wedge should also describe the local behavior of the surface to the left of $\eta$ before a quantum typical intersection point.  Therefore, gluing together three independent wedges with respective weights $\rho+2$, $2$, and $\rho+2$ should describe the local behavior of $\CW$ near a quantum typical intersection point of $\eta$ with $\partial \h$.  (See Figure~\ref{subfig::sle_kr_boundary} for an illustration.)

This leads to the prediction that the (boundary) quantum and Euclidean scaling exponents for $\eta \cap \partial \h$ are respectively given by:
\begin{equation}
\label{eqn::slekr_boundary_intersection_exponents}
\Delta = \frac{2(2+\rho)}{\kappa} \quad\text{and}\quad x = \frac{(2+\rho)(2(4+\rho) - \kappa)}{2\kappa}.
\end{equation}
Comparing this expression for $\Delta$ to~\eqref{DeltanB} shows that  $\rho$ can be interpreted as a number of equivalent  $\SLE_\kappa$ paths, in agreement with the quantum gravity equivalence predicted in~\cite[Section~10.4]{LH2005}. This, in turn, leads to a prediction of the dimension of $\eta \cap \partial \h$ given by
\begin{equation}
\label{eqn::slekr_boundary_intersection}
\frac{(4+\rho)(\kappa-2(2+\rho))}{2\kappa}.
\end{equation}
This prediction matches the a.s.\ Hausdorff dimension for $\eta \cap \partial \h$ which was rigorously derived in \cite[Theorem~1.6]{mw2013intersections}.  By $\SLE$ duality, for $\kappa \in (2,4)$ and $\rho = \kappa-4$ this gives the dimension $2-\frac{8}{\kappa'}$ of the intersection of an $\SLE_{\kappa'}$ process, $\kappa'=\tfrac{16}{\kappa}$, with $\partial \h$, predicted in \cite{PhysRevLett.57.3179,Duplantier198971}, and  first derived in \cite{she_alberts2008boundary}.

We note that~\eqref{eqn::slekr_boundary_intersection} (while implicit in \cite[Section~10.4]{LH2005}) was derived previously in the physics literature.  See, for example, \cite{djs2009on} as well as the longer version \cite{djs2010on}.

\subsection{Intersection dimension of GFF flow lines}

Suppose that we have $\SLE_\kappa$ flow lines $\eta_1$, $\eta_2$ of a GFF on $\h$ with angles $\theta_1 < \theta_2$ with angle difference $\theta = \theta_2 - \theta_1 \in (0,\pi \kappa/(4-\kappa))$; recall that this is the range of angle differences that allows two such paths to intersect and bounce off each other \cite[Theorem~1.5]{ms2012imag1}.  If we conformally map away the region which is to the right of $\eta_1$, then a quantum typical intersection point $z_0$ of $\eta_1$ and $\eta_2$ should be mapped to a quantum typical intersection point of the image of $\eta_2$ with the domain boundary.  Let $\rho = \theta \chi/\lambda-2$ where $\lambda = \tfrac{\pi}{\sqrt{\kappa}}$ and $\chi = \tfrac{2}{\sqrt{\kappa}} - \tfrac{\sqrt{\kappa}}{2}$.  By our heuristic for the intersection of $\SLE_\kappa(\rho)$ with~$\partial \h$, this suggests that the local behavior of the surface near~$z_0$ and to the left of~$\eta_1$ should look like a gluing of three independent wedges with respective weights $\rho+2$, $2$, and $\rho+2$.  Moreover, these wedges should be independent of the local behavior of the surface to the right of $\eta_1$.  By symmetry, this region should look like another independent wedge of weight $2$.  Therefore the whole picture should look like two weight $\rho+2$ wedges glued with two weight $2$ wedges where the different types alternate.  (See Figure~\ref{subfig::gff_flow_line_intersection} for an illustration.)

This leads to the prediction that the quantum and Euclidean scaling exponents are given by
\begin{equation}
\label{eqn::gff_flow_line_exponents}
 \Delta = \frac{4+\kappa+2\rho}{2\kappa} \quad\text{and}\quad x = \frac{1}{4\kappa}\left(\rho + \frac{\kappa}{2} +2\right)\left(\rho - \frac{\kappa}{2}+6\right).
\end{equation}
This leads to the prediction that the dimension of $\eta_1 \cap \eta_2 \cap \h$ is given by
\begin{equation}
\label{eqn::gff_flow_line_dimension}
2 - \frac{1}{2\kappa}\left(\rho + \frac{\kappa}{2} +2\right)\left(\rho - \frac{\kappa}{2}+6\right).
\end{equation}
This matches the rigorous derivation of the a.s.\ Hausdorff dimension of $\eta_1 \cap \eta_2 \cap \h$ given in \cite[Theorem~1.5]{mw2013intersections}.

\bibliographystyle{abbrv}
\addcontentsline{toc}{section}{References}
\bibliography{welding}

\begin{thebibliography}{100}

\bibitem{PhysRevLett.83.1359}
M.~Aizenman, B.~Duplantier, and A.~Aharony.
\newblock Path-crossing exponents and the external perimeter in 2d percolation.
\newblock {\em Phys. Rev. Lett.}, 83:1359--1362, 1999.
\newblock \arxiv{cond-mat/9901018}.

\bibitem{she_alberts2008boundary}
T.~Alberts and S.~Sheffield.
\newblock Hausdorff dimension of the {SLE} curve intersected with the real
  line.
\newblock {\em Electron. J. Probab.}, 13:no. 40, 1166--1188, 2008.
\newblock \arxiv{0711.4070}.

\bibitem{ald1991crt1}
D.~Aldous.
\newblock The continuum random tree. {I}.
\newblock {\em Ann. Probab.}, 19(1):1--28, 1991.

\bibitem{ald1991crt2}
D.~Aldous.
\newblock The continuum random tree. {II}. {A}n overview.
\newblock In {\em Stochastic analysis ({D}urham, 1990)}, volume 167 of {\em
  London Math. Soc. Lecture Note Ser.}, pages 23--70. Cambridge Univ. Press,
  Cambridge, 1991.

\bibitem{ald1993crt3}
D.~Aldous.
\newblock The continuum random tree. {III}.
\newblock {\em Ann. Probab.}, 21(1):248--289, 1993.

\bibitem{MR1214333}
L.~Alvarez-Gaum{\'e}, J.~L.~F. Barb{\'o}n, and {\v{C}}.~Crnkovi{\'c}.
\newblock A proposal for strings at {$D>1$}.
\newblock {\em Nuclear Phys. B}, 394(2):383--422, 1993.
\newblock \arxiv{hep-th/9208026}.

\bibitem{MR1279108}
J.~Ambj{\o}rn, B.~Durhuus, and T.~J{\'o}nsson.
\newblock A solvable {$2$}{D} gravity model with {$\gamma>0$}.
\newblock {\em Modern Phys. Lett. A}, 9(13):1221--1228, 1994.
\newblock \arxiv{hep-th/9401137}.

\bibitem{angel2003growth}
O.~Angel.
\newblock Growth and percolation on the uniform infinite planar triangulation.
\newblock {\em Geom. Funct. Anal.}, 13(5):935--974, 2003.
\newblock \arxiv{math/0208123}.

\bibitem{ac2013percolation}
O.~Angel and N.~Curien.
\newblock Percolations on random maps {I}: {H}alf-plane models.
\newblock {\em Ann. Inst. Henri Poincar\'e Probab. Stat.}, 51(2):405--431,
  2015.
\newblock \arxiv{1301.5311}.

\bibitem{ar2013classification}
O.~Angel and G.~Ray.
\newblock Classification of half-planar maps.
\newblock {\em Ann. Probab.}, 43(3):1315--1349, 2015.
\newblock \arxiv{1303.6582}.

\bibitem{as2003uipt}
O.~Angel and O.~Schramm.
\newblock Uniform infinite planar triangulations.
\newblock {\em Comm. Math. Phys.}, 241(2-3):191--213, 2003.
\newblock \arxiv{math/0207153}.

\bibitem{aru2015kpz}
J.~Aru.
\newblock K{PZ} relation does not hold for the level lines and {SLE$_\kappa$}
  flow lines of the {G}aussian free field.
\newblock {\em Probab. Theory Related Fields}, 163(3-4):465--526, 2015.
\newblock \arxiv{1312.1324}.

\bibitem{twoperspectives}
J.~Aru, Y.~Huang, and X.~Sun.
\newblock Two {P}erspectives of the 2{D} {U}nit {A}rea {Q}uantum {S}phere and
  {T}heir {E}quivalence.
\newblock {\em Comm. Math. Phys.}, 356(1):261--283, 2017.
\newblock \arxiv{1512.06190}.

\bibitem{aspenberg2009mating}
M.~Aspenberg and M.~Yampolsky.
\newblock Mating non-renormalizable quadratic polynomials.
\newblock {\em Comm. Math. Phys.}, 287(1):1--40, 2009.
\newblock \arxiv{math/0610343}.

\bibitem{astala2011random}
K.~Astala, A.~Kupiainen, E.~Saksman, and P.~Jones.
\newblock Random conformal weldings.
\newblock {\em Acta mathematica}, 207(2):203--254, 2011.

\bibitem{bjrv2013super_critical}
J.~Barral, X.~Jin, R.~Rhodes, and V.~Vargas.
\newblock Gaussian multiplicative chaos and {KPZ} duality.
\newblock {\em Comm. Math. Phys.}, 323(2):451--485, 2013.
\newblock \arxiv{1202.5296}.

\bibitem{PhysRevLett.61.138}
M.~T. Batchelor and H.~W.~J. Bl{\"o}te.
\newblock Conformal anomaly and scaling dimensions of the {${\rm O}(n)$} model
  from an exact solution on the honeycomb lattice.
\newblock {\em Phys. Rev. Lett.}, 61(2):138--140, 1988.

\bibitem{0305-4470-26-16-004}
M.~T. Batchelor and J.~Suzuki.
\newblock Exact solution and surface critical behaviour of an {${\rm O}(n)$}
  model on the honeycomb lattice.
\newblock {\em J. Phys. A}, 26(16):L729--L735, 1993.

\bibitem{b2017natural}
S.~Benoist.
\newblock Natural parametrization of {SLE}: the {G}aussian free field point of
  view.
\newblock {\em Electron. J. Probab.}, 23:Paper No. 103, 16, 2018.
\newblock \arxiv{1708.03801}.

\bibitem{ber2015lbm}
N.~Berestycki.
\newblock Diffusion in planar {L}iouville quantum gravity.
\newblock {\em Ann. Inst. Henri Poincar\'e Probab. Stat.}, 51(3):947--964,
  2015.
\newblock \arxiv{1301.3356}.

\bibitem{bgr2016kpz}
N.~Berestycki, C.~Garban, R.~Rhodes, and V.~Vargas.
\newblock K{PZ} formula derived from {L}iouville heat kernel.
\newblock {\em J. Lond. Math. Soc. (2)}, 94(1):186--208, 2016.
\newblock \arxiv{1406.7280}.

\bibitem{bg2020convergence}
N.~{Berestycki} and E.~{Gwynne}.
\newblock {Random walks on mated-CRT planar maps and Liouville Brownian
  motion}.
\newblock {\em arXiv e-prints}, page arXiv:2003.10320, Mar. 2020.

\bibitem{berestycki2014equivalence}
N.~{Berestycki}, S.~{Sheffield}, and X.~{Sun}.
\newblock {Equivalence of Liouville measure and Gaussian free field}.
\newblock {\em ArXiv e-prints}, Oct. 2014.

\bibitem{bertoin96levy}
J.~Bertoin.
\newblock {\em L\'evy processes}, volume 121 of {\em Cambridge Tracts in
  Mathematics}.
\newblock Cambridge University Press, Cambridge, 1996.

\bibitem{bertoin1999subordinators}
J.~Bertoin.
\newblock Subordinators: examples and applications.
\newblock In {\em Lectures on probability theory and statistics}, pages 1--91.
  Springer, 1999.

\bibitem{bishop2007conformal}
C.~J. Bishop.
\newblock Conformal welding and {K}oebe's theorem.
\newblock {\em Ann. of Math. (2)}, 166(3):613--656, 2007.

\bibitem{boyd2012medusa}
S.~H. Boyd and C.~Henriksen.
\newblock The {M}edusa algorithm for polynomial matings.
\newblock {\em Conform. Geom. Dyn.}, 16:161--183, 2012.
\newblock \arxiv{1102.5047}.

\bibitem{1984NuPhB.240..514C}
J.~L. {Cardy}.
\newblock {Conformal invariance and surface critical behavior}.
\newblock {\em Nuclear Phys. B}, 240:514--532, 1984.

\bibitem{0305-4470-25-4-009}
J.~L. Cardy.
\newblock Critical percolation in finite geometries.
\newblock {\em J. Phys. A}, 25(4):L201--L206, 1992.
\newblock \arxiv{hep-th/9111026}.

\bibitem{chassaing2004random}
P.~Chassaing and G.~Schaeffer.
\newblock Random planar lattices and integrated superbrownian excursion.
\newblock {\em Probability Theory and Related Fields}, 128(2):161--212, 2004.

\bibitem{chatterjee2008chaos}
S.~Chatterjee.
\newblock Chaos, concentration, and multiple valleys.
\newblock {\em ArXiv e-prints}, Dec. 2008.

\bibitem{cori1981planar}
R.~Cori and B.~Vauquelin.
\newblock Planar maps are well labeled trees.
\newblock {\em Canadian Journal of Mathematics}, 33(5):1023--1042, 1981.

\bibitem{cur2013glimpse}
N.~Curien.
\newblock A glimpse of the conformal structure of random planar maps.
\newblock {\em Comm. Math. Phys.}, 333(3):1417--1463, 2015.
\newblock \arxiv{1308.1807}.

\bibitem{ck2013looptrees}
N.~Curien and I.~Kortchemski.
\newblock Random stable looptrees.
\newblock {\em Electron. J. Probab.}, 19:no. 108, 35, 2014.
\newblock \arxiv{1304.1044}.

\bibitem{cl2012brownianplane}
N.~Curien and J.-F. Le~Gall.
\newblock The {B}rownian plane.
\newblock {\em J. Theoret. Probab.}, 27(4):1249--1291, 2014.
\newblock \arxiv{1204.5921}.

\bibitem{cl2014hull}
N.~Curien and J.-F. Le~Gall.
\newblock The hull process of the {B}rownian plane.
\newblock {\em Probab. Theory Related Fields}, 166(1-2):187--231, 2016.
\newblock \arxiv{1409.4026}.

\bibitem{MR1057913}
S.~R. Das, A.~Dhar, A.~M. Sengupta, and S.~R. Wadia.
\newblock New critical behavior in {$d=0$} large-{$N$} matrix models.
\newblock {\em Modern Phys. Lett. A}, 5(13):1041--1056, 1990.

\bibitem{MR981529}
F.~David.
\newblock Conformal field theories coupled to {$2$}-{D} gravity in the
  conformal gauge.
\newblock {\em Modern Phys. Lett. A}, 3(17):1651--1656, 1988.

\bibitem{lqg_sphere}
F.~David, A.~Kupiainen, R.~Rhodes, and V.~Vargas.
\newblock Liouville quantum gravity on the {R}iemann sphere.
\newblock {\em Comm. Math. Phys.}, 342(3):869--907, 2016.
\newblock \arxiv{1410.7318}.

\bibitem{PhysRevB.27.1674}
M.~den Nijs.
\newblock Extended scaling relations for the magnetic critical exponents of the
  {P}otts model.
\newblock {\em Phys. Rev. B}, 27(3):1674--1679, 1983.

\bibitem{MR1320471}
P.~Di~Francesco, P.~Ginsparg, and J.~Zinn-Justin.
\newblock {$2$}{D} gravity and random matrices.
\newblock {\em Phys. Rep.}, 254:1--133, 1995.

\bibitem{dddf2019tightness}
J.~{Ding}, J.~{Dub{\'e}dat}, A.~{Dunlap}, and H.~{Falconet}.
\newblock {Tightness of Liouville first passage percolation for $\gamma \in
  (0,2)$}.
\newblock {\em arXiv e-prints}, page arXiv:1904.08021, Apr 2019.

\bibitem{MR1005268}
J.~Distler and H.~Kawai.
\newblock Conformal field theory and {$2$}{D} quantum gravity.
\newblock {\em Nuclear Phys. B}, 321(2):509--527, 1989.

\bibitem{djs2009on}
J.~Dubail, J.~L. Jacobsen, and H.~Saleur.
\newblock Exact solution of the anisotropic special transition in the $o(n)$
  model in two dimensions.
\newblock {\em Phys. Rev. Lett.}, 103:145701, Oct 2009.

\bibitem{djs2010on}
J.~Dubail, J.~L. Jacobsen, and H.~Saleur.
\newblock Conformal boundary conditions in the critical {$O(n)$} model and
  dilute loop models.
\newblock {\em Nuclear Phys. B}, 827(3):457--502, 2010.

\bibitem{dub2005mg_duality}
J.~Dub{\'e}dat.
\newblock {${\rm SLE}(\kappa,\rho)$} martingales and duality.
\newblock {\em Ann. Probab.}, 33(1):223--243, 2005.
\newblock \arxiv{math/0303128}.

\bibitem{dub_dual}
J.~Dub{\'e}dat.
\newblock Duality of {S}chramm-{L}oewner evolutions.
\newblock {\em Ann. Sci. \'Ec. Norm. Sup\'er. (4)}, 42(5):697--724, 2009.
\newblock \arxiv{0711.1884}.

\bibitem{dub2009part}
J.~Dub{\'e}dat.
\newblock S{LE} and the free field: partition functions and couplings.
\newblock {\em J. Amer. Math. Soc.}, 22(4):995--1054, 2009.

\bibitem{0305-4470-19-16-011}
B.~Duplantier.
\newblock Exact critical exponents for two-dimensional dense polymers.
\newblock {\em J. Phys. A}, 19(16):L1009--L1014, 1986.

\bibitem{dupJSP97}
B.~Duplantier.
\newblock Critical exponents of {M}anhattan {H}amiltonian walks in two
  dimensions, from {P}otts and {${\rm O}(n)$} models.
\newblock {\em J. Statist. Phys.}, 49(3-4):411--431, 1987.

\bibitem{Duplantier198971}
B.~Duplantier.
\newblock Fractals in two dimensions and conformal invariance.
\newblock {\em Physica D}, 38(1-3):71--87, 1989.
\newblock Fractals in physics (Vence, 1989).

\bibitem{dup1998rwk}
B.~Duplantier.
\newblock Random walks and quantum gravity in two dimensions.
\newblock {\em Phys. Rev. Lett.}, 81(25):5489--5492, 1998.

\bibitem{PhysRevLett.82.3940}
B.~Duplantier.
\newblock Harmonic measure exponents for two-dimensional percolation.
\newblock {\em Phys. Rev. Lett.}, 82(20):3940--3943, 1999.

\bibitem{dup2000fractals}
B.~Duplantier.
\newblock Conformally invariant fractals and potential theory.
\newblock {\em Phys. Rev. Lett.}, 84(7):1363--1367, 2000.

\bibitem{MR1964687}
B.~Duplantier.
\newblock Higher conformal multifractality.
\newblock {\em J. Stat. Phys.}, 110:691--738, 2003.
\newblock Special issue in honor of Michael E.\ Fisher's 70th birthday
  (Piscataway, NJ, 2001).

\bibitem{dup2004qg}
B.~Duplantier.
\newblock Conformal fractal geometry \& boundary quantum gravity.
\newblock In M.~L. Lapidus and M.~van Frankenhuysen, editors, {\em Fractal
  geometry and applications: a jubilee of {B}eno\^\i t {M}andelbrot, {P}art 2},
  volume~72 of {\em Proc. Sympos. Pure Math.}, pages 365--482. Amer. Math.
  Soc., Providence, RI, 2004.

\bibitem{LH2005}
B.~Duplantier.
\newblock Conformal random geometry.
\newblock In A.~Bovier, F.~Dunlop, F.~den Hollander, A.~van Enter, and
  J.~Dalibard, editors, {\em Mathematical statistical physics (Les Houches
  Summer School, Session LXXXIII, 2005)}, pages 101--217. Elsevier B. V.,
  Amsterdam, 2006.

\bibitem{2008ExactMethodsBD}
B.~Duplantier.
\newblock A rigorous perspective on {L}iouville quantum gravity and the {KPZ}
  relation.
\newblock In S.~Ouvry, J.~Jacobsen, V.~Pasquier, D.~Serban, and L.~Cugliandolo,
  editors, {\em Exact methods in low-dimensional statistical physics and
  quantum theory (Les Houches Summer School LXXXIX, 2008)}, pages 529--561.
  Oxford Univ. Press, Oxford, 2010.

\bibitem{1988PhRvL..61.1433D}
B.~Duplantier and I.~Kostov.
\newblock Conformal spectra of polymers on a random surface.
\newblock {\em Phys. Rev. Lett.}, 61(13):1433--1437, 1988.

\bibitem{1990NuPhB.340..491D}
B.~Duplantier and I.~K. Kostov.
\newblock Geometrical critical phenomena on a random surface of arbitrary
  genus.
\newblock {\em Nuclear Phys. B}, 340(2-3):491--541, 1990.

\bibitem{dk1998rwks}
B.~Duplantier and K.-H. Kwon.
\newblock Conformal invariance and intersections of random walks.
\newblock {\em Phys. Rev. Lett.}, 61:2514--2517, 1988.

\bibitem{DRSV1}
B.~Duplantier, R.~Rhodes, S.~Sheffield, and V.~Vargas.
\newblock Critical {G}aussian multiplicative chaos: convergence of the
  derivative martingale.
\newblock {\em Ann. Probab.}, 42(5):1769--1808, 2014.
\newblock \arxiv{1206.1671}.

\bibitem{MR3215583}
B.~Duplantier, R.~Rhodes, S.~Sheffield, and V.~Vargas.
\newblock Renormalization of {C}ritical {G}aussian {M}ultiplicative {C}haos and
  {KPZ} {R}elation.
\newblock {\em Comm. Math. Phys.}, 330(1):283--330, 2014.

\bibitem{PhysRevLett.57.3179}
B.~Duplantier and H.~Saleur.
\newblock Exact surface and wedge exponents for polymers in two dimensions.
\newblock {\em Phys. Rev. Lett.}, 57(25):3179--3182, 1986.

\bibitem{Duplantier1987291}
B.~Duplantier and H.~Saleur.
\newblock Exact critical properties of two-dimensional dense self-avoiding
  walks.
\newblock {\em Nuclear Phys. B}, 290(3):291--326, 1987.

\bibitem{MR947310}
B.~Duplantier and H.~Saleur.
\newblock Winding-angle distributions of two-dimensional self-avoiding walks
  from conformal invariance.
\newblock {\em Phys. Rev. Lett.}, 60(23):2343--2346, 1988.

\bibitem{ds2009qg_prl}
B.~Duplantier and S.~Sheffield.
\newblock Duality and the {K}nizhnik-{P}olyakov-{Z}amolodchikov relation in
  {L}iouville quantum gravity.
\newblock {\em Phys. Rev. Lett.}, 102(15):150603, 4, 2009.
\newblock \arxiv{0901.0277}.

\bibitem{ds2011kpz}
B.~Duplantier and S.~Sheffield.
\newblock Liouville quantum gravity and {KPZ}.
\newblock {\em Invent. Math.}, 185(2):333--393, 2011.
\newblock \arxiv{0808.1560}.

\bibitem{DS3}
B.~Duplantier and S.~Sheffield.
\newblock {S}chramm-{L}oewner evolution and {L}iouville quantum gravity.
\newblock {\em Phys. Rev. Lett.}, 107:131305, 2011.
\newblock \arxiv{1012.4800}.

\bibitem{dlg2002trees_levy}
T.~Duquesne and J.-F. Le~Gall.
\newblock Random trees, {L}\'evy processes and spatial branching processes.
\newblock {\em Ast\'erisque}, 281:vi+147, 2002.
\newblock \arxiv{math/0509558}.

\bibitem{MR1293688}
B.~Durhuus.
\newblock Multi-spin systems on a randomly triangulated surface.
\newblock {\em Nuclear Phys. B}, 426(1):203--222, 1994.
\newblock \arxiv{hep-th/9402052}.

\bibitem{efronstein}
B.~Efron and C.~Stein.
\newblock The jackknife estimate of variance.
\newblock {\em Ann. Statist.}, 9(3):586--596, 1981.

\bibitem{EVANS_CONE_TIMES}
S.~N. Evans.
\newblock On the {H}ausdorff dimension of {B}rownian cone points.
\newblock {\em Math. Proc. Cambridge Philos. Soc.}, 98(2):343--353, 1985.

\bibitem{gps2013pivotal}
C.~Garban, G.~Pete, and O.~Schramm.
\newblock Pivotal, cluster, and interface measures for critical planar
  percolation.
\newblock {\em J. Amer. Math. Soc.}, 26(4):939--1024, 2013.
\newblock \arxiv{1008.1378}.

\bibitem{grv2016lbm}
C.~Garban, R.~Rhodes, and V.~Vargas.
\newblock Liouville {B}rownian motion.
\newblock {\em Ann. Probab.}, 44(4):3076--3110, 2016.
\newblock \arxiv{1301.2876}.

\bibitem{Ginsparg-Moore}
P.~{Ginsparg} and G.~{Moore}.
\newblock Lectures on 2{D} gravity and 2{D} string theory ({TASI} 1992).
\newblock In J.~{Harvey} and J.~{Polchinski}, editors, {\em Recent direction in
  particle theory, Proceedings of the 1992 TASI}. World Scientific, Singapore,
  1993.

\bibitem{gwynne2019random}
E.~Gwynne.
\newblock Random surfaces and {L}iouville quantum gravity.
\newblock {\em arXiv preprint arXiv:1908.05573}, 2019.

\bibitem{ghm2015kpz}
E.~Gwynne, N.~Holden, and J.~Miller.
\newblock An almost sure {KPZ} relation for {SLE} and {B}rownian motion.
\newblock {\em Ann. Probab.}, 48(2):527--573, 2020.
\newblock \arxiv{1512.01223}.

\bibitem{ghm2016trans}
E.~Gwynne, N.~Holden, and J.~Miller.
\newblock Dimension transformation formula for conformal maps into the
  complement of an {SLE} curve.
\newblock {\em Probab. Theory Related Fields}, 176(1-2):649--667, 2020.
\newblock \arxiv{1603.05161}.

\bibitem{ghms2015covariance}
E.~Gwynne, N.~Holden, J.~Miller, and X.~Sun.
\newblock Brownian motion correlation in the peanosphere for {$\kappa>8$}.
\newblock {\em Ann. Inst. Henri Poincar\'e Probab. Stat.}, 53(4):1866--1889,
  2017.
\newblock \arxiv{1510.04687}.

\bibitem{gwynne2019mating}
E.~Gwynne, N.~Holden, and X.~Sun.
\newblock Mating of trees for random planar maps and {L}iouville quantum
  gravity: a survey.
\newblock {\em arXiv preprint arXiv:1910.04713}, 2019.

\bibitem{ghs2017distances}
E.~Gwynne, N.~Holden, and X.~Sun.
\newblock A mating-of-trees approach for graph distances in random planar maps.
\newblock {\em Probab. Theory Related Fields}, 177(3-4):1043--1102, 2020.
\newblock \arxiv{1711.00723}.

\bibitem{GKMW:active-tree-map}
E.~Gwynne, A.~Kassel, J.~Miller, and D.~B. Wilson.
\newblock Active {S}panning {T}rees with {B}ending {E}nergy on {P}lanar {M}aps
  and {SLE}-{D}ecorated {L}iouville {Q}uantum {G}ravity for {$\kappa > 8$}.
\newblock {\em Comm. Math. Phys.}, 358(3):1065--1115, 2018.
\newblock \arxiv{1603.09722}.

\bibitem{gms2015cone_times}
E.~Gwynne, C.~Mao, and X.~Sun.
\newblock Scaling limits for the critical {F}ortuin-{K}asteleyn model on a
  random planar map {I}: {C}one times.
\newblock {\em Ann. Inst. Henri Poincar\'{e} Probab. Stat.}, 55(1):1--60, 2019.
\newblock \arXiv{1502.00546}.

\bibitem{gm2016topology}
E.~Gwynne and J.~Miller.
\newblock Convergence of the topology of critical {F}ortuin-{K}asteleyn planar
  maps to that of {CLE}$_\kappa$ on a {L}iouville quantum surface.
\newblock 2016.
\newblock In preparation.

\bibitem{gm2017spectral}
E.~{Gwynne} and J.~{Miller}.
\newblock {Random walk on random planar maps: spectral dimension, resistance,
  and displacement}.
\newblock {\em ArXiv e-prints}, Nov. 2017.

\bibitem{gm2019conformalcov}
E.~{Gwynne} and J.~{Miller}.
\newblock {Conformal covariance of the Liouville quantum gravity metric for
  $\gamma \in (0,2)$}.
\newblock {\em arXiv e-prints}, page arXiv:1905.00384, May 2019.

\bibitem{gm2019uniqueness}
E.~{Gwynne} and J.~{Miller}.
\newblock {Existence and uniqueness of the Liouville quantum gravity metric for
  $\gamma \in (0,2)$}.
\newblock {\em arXiv e-prints}, page arXiv:1905.00383, May 2019.
\newblock To appear in Inventiones.

\bibitem{gm2019local}
E.~{Gwynne} and J.~{Miller}.
\newblock {Local metrics of the Gaussian free field}.
\newblock {\em arXiv e-prints}, page arXiv:1905.00379, May 2019.
\newblock To appear in Ann. Fourier.

\bibitem{gm2019confluence}
E.~Gwynne and J.~Miller.
\newblock Confluence of geodesics in {L}iouville quantum gravity for {$\gamma
  \in (0,2)$}.
\newblock {\em Ann. Probab.}, 48(4):1861--1901, 2020.
\newblock \arxiv{1905.00381}.

\bibitem{gms2017tutte}
E.~{Gwynne}, J.~{Miller}, and S.~{Sheffield}.
\newblock {The Tutte embedding of the mated-CRT map converges to Liouville
  quantum gravity}.
\newblock {\em ArXiv e-prints}, May 2017.

\bibitem{gp2018components}
E.~Gwynne and J.~Pfeffer.
\newblock Connectivity properties of the adjacency graph of {${\rm
  SLE}_{\kappa}$} bubbles for {$\kappa\in(4,8)$}.
\newblock {\em Ann. Probab.}, 48(3):1495--1519, 2020.
\newblock \arxiv{1803.04923}.

\bibitem{gs2015finite}
E.~{Gwynne} and X.~{Sun}.
\newblock {Scaling limits for the critical Fortuin-Kastelyn model on a random
  planar map III: finite volume case}.
\newblock {\em ArXiv e-prints}, Oct. 2015.

\bibitem{gs2015finite_volume}
E.~Gwynne and X.~Sun.
\newblock Scaling limits for the critical {F}ortuin-{K}asteleyn model on a
  random planar map {II}: local estimates and empty reduced word exponent.
\newblock {\em Electron. J. Probab.}, 22:Paper No. 45, 56, 2017.
\newblock \arXiv{1505.03375}.

\bibitem{hls2018natural}
N.~{Holden}, X.~{Li}, and X.~{Sun}.
\newblock {Natural parametrization of percolation interface and pivotal
  points}.
\newblock {\em ArXiv e-prints}, Apr. 2018.

\bibitem{euclideanmatingoftrees}
N.~Holden and X.~Sun.
\newblock S{LE} as a mating of trees in {E}uclidean geometry.
\newblock {\em Comm. Math. Phys.}, 364(1):171--201, 2018.
\newblock \arxiv{1610.05272}.

\bibitem{MR2642894}
X.~Hu, J.~Miller, and Y.~Peres.
\newblock Thick points of the {G}aussian free field.
\newblock {\em Ann. Probab.}, 38(2):896--926, 2010.
\newblock \arxiv{0902.3842}.

\bibitem{MR1175135}
S.~Jain and S.~D. Mathur.
\newblock World-sheet geometry and baby universes in {$2$}{D} quantum gravity.
\newblock {\em Phys. Lett. B}, 286(3-4):239--246, 1992.
\newblock \arxiv{hep-th/9204017}.

\bibitem{MR1785402}
P.~W. Jones and S.~K. Smirnov.
\newblock Removability theorems for {S}obolev functions and quasiconformal
  maps.
\newblock {\em Ark. Mat.}, 38(2):263--279, 2000.

\bibitem{kg2004guide_to_sle}
W.~Kager and B.~Nienhuis.
\newblock A guide to stochastic {L}\"owner evolution and its applications.
\newblock {\em J. Statist. Phys.}, 115(5-6):1149--1229, 2004.
\newblock \arxiv{math-ph/0312056}.

\bibitem{MR829798}
J.-P. Kahane.
\newblock Sur le chaos multiplicatif.
\newblock {\em Ann. Sci. Math. Qu\'ebec}, 9(2):105--150, 1985.

\bibitem{KAL_FOUND}
O.~Kallenberg.
\newblock {\em Foundations of modern probability}.
\newblock Probability and its Applications (New York). Springer-Verlag, New
  York, second edition, 2002.

\bibitem{ks91bm}
I.~Karatzas and S.~E. Shreve.
\newblock {\em Brownian motion and stochastic calculus}, volume 113 of {\em
  Graduate Texts in Mathematics}.
\newblock Springer-Verlag, New York, second edition, 1991.

\bibitem{Kazakov1992520}
V.~A. Kazakov and I.~K. Kostov.
\newblock Loop gas model for open strings.
\newblock {\em Nuclear Phys. B}, 386(3):520--557, 1992.
\newblock \arxiv{hep-th/9205059}.

\bibitem{kmsw2015bipolar}
R.~Kenyon, J.~Miller, S.~Sheffield, and D.~B. Wilson.
\newblock Bipolar orientations on planar maps and {${\rm SLE}_{12}$}.
\newblock {\em Ann. Probab.}, 47(3):1240--1269, 2019.
\newblock \arxiv{1511.04068}.

\bibitem{kleb1995touching_surfaces}
I.~R. Klebanov.
\newblock Touching random surfaces and {L}iouville gravity.
\newblock {\em Phys. Rev. D (3)}, 51(4):1836--1841, 1995.
\newblock \arxiv{hep-th/9407167}.

\bibitem{kleb1995non_perturb}
I.~R. Klebanov and A.~Hashimoto.
\newblock Non-perturbative solution of matrix models modified by trace-squared
  terms.
\newblock {\em Nuclear Phys. B}, 434(1-2):264--282, 1995.
\newblock \arxiv{hep-th/9409064}.

\bibitem{kleb1996wormholes}
I.~R. Klebanov and A.~Hashimoto.
\newblock Wormholes, matrix models, and {L}iouville gravity.
\newblock {\em Nuclear Phys. B Proc. Suppl.}, 45BC:135--148, 1996.
\newblock String theory, gauge theory and quantum gravity (Trieste, 1995).

\bibitem{MR947880}
V.~G. Knizhnik, A.~M. Polyakov, and A.~B. Zamolodchikov.
\newblock Fractal structure of {$2$}{D}-quantum gravity.
\newblock {\em Modern Phys. Lett. A}, 3(8):819--826, 1988.

\bibitem{MR1186338}
G.~P. Korchemsky.
\newblock Matrix model perturbed by higher order curvature terms.
\newblock {\em Modern Phys. Lett. A}, 7(33):3081--3100, 1992.
\newblock \arxiv{hep-th/9205014}.

\bibitem{1989MPLA....4..217K}
I.~K. Kostov.
\newblock {${\rm O}(n)$} vector model on a planar random lattice: spectrum of
  anomalous dimensions.
\newblock {\em Modern Phys. Lett. A}, 4(3):217--226, 1989.

\bibitem{krikun2005uipq}
M.~{Krikun}.
\newblock {Local structure of random quadrangulations}.
\newblock {\em ArXiv e-prints}, Dec. 2005.

\bibitem{lamp1967csbp}
J.~Lamperti.
\newblock Continuous state branching processes.
\newblock {\em Bull. Amer. Math. Soc.}, 73:382--386, 1967.

\bibitem{law2013transience}
G.~F. Lawler.
\newblock Continuity of radial and two-sided radial {$SLE$} at the terminal
  point.
\newblock In {\em In the tradition of {A}hlfors-{B}ers. {VI}}, volume 590 of
  {\em Contemp. Math.}, pages 101--124. Amer. Math. Soc., Providence, RI, 2013.
\newblock \arxiv{1104.1620}.

\bibitem{law2015minkowski}
G.~F. Lawler.
\newblock Minkowski content of the intersection of a {S}chramm-{L}oewner
  evolution ({SLE}) curve with the real line.
\newblock {\em J. Math. Soc. Japan}, 67(4):1631--1669, 2015.

\bibitem{lr2012minkowski}
G.~F. Lawler and M.~A. Rezaei.
\newblock Minkowski content and natural parameterization for the
  {S}chramm-{L}oewner evolution.
\newblock {\em Ann. Probab.}, 43(3):1082--1120, 2015.
\newblock \arxiv{1211.4146}.

\bibitem{zbMATH01690755}
G.~F. Lawler, O.~Schramm, and W.~Werner.
\newblock The dimension of the planar {B}rownian frontier is {$4/3$}.
\newblock {\em Math. Res. Lett.}, 8(4):401--411, 2001.
\newblock \arxiv{math/0010165}.

\bibitem{lsw2001intersection1}
G.~F. Lawler, O.~Schramm, and W.~Werner.
\newblock Values of {B}rownian intersection exponents. {I}. {H}alf-plane
  exponents.
\newblock {\em Acta Math.}, 187(2):237--273, 2001.
\newblock \arxiv{math/9911084}.

\bibitem{lsw2001intersection2}
G.~F. Lawler, O.~Schramm, and W.~Werner.
\newblock Values of {B}rownian intersection exponents. {II}. {P}lane exponents.
\newblock {\em Acta Math.}, 187(2):275--308, 2001.
\newblock \arxiv{math/0003156}.

\bibitem{lsw2002intersection3}
G.~F. Lawler, O.~Schramm, and W.~Werner.
\newblock Values of {B}rownian intersection exponents. {III}. {T}wo-sided
  exponents.
\newblock {\em Ann. Inst. H. Poincar\'e Probab. Statist.}, 38(1):109--123,
  2002.
\newblock \arxiv{math/0005294}.

\bibitem{lsw2004ust}
G.~F. Lawler, O.~Schramm, and W.~Werner.
\newblock Conformal invariance of planar loop-erased random walks and uniform
  spanning trees.
\newblock {\em Ann. Probab.}, 32(1B):939--995, 2004.
\newblock \arxiv{math/0112234}.

\bibitem{ls2011natural_param}
G.~F. Lawler and S.~Sheffield.
\newblock A natural parametrization for the {S}chramm-{L}oewner evolution.
\newblock {\em Ann. Probab.}, 39(5):1896--1937, 2011.
\newblock \arxiv{0906.3804}.

\bibitem{lv2016lerw}
G.~F. {Lawler} and F.~{Viklund}.
\newblock {Convergence of loop-erased random walk in the natural
  parametrization}.
\newblock {\em ArXiv e-prints}, Mar. 2016.

\bibitem{MR1742883}
G.~F. Lawler and W.~Werner.
\newblock Intersection exponents for planar {B}rownian motion.
\newblock {\em Ann. Probab.}, 27(4):1601--1642, 1999.

\bibitem{MR1796962}
G.~F. Lawler and W.~Werner.
\newblock Universality for conformally invariant intersection exponents.
\newblock {\em J. Eur. Math. Soc. (JEMS)}, 2(4):291--328, 2000.

\bibitem{lz2013natural_param}
G.~F. Lawler and W.~Zhou.
\newblock {\it {SLE}} curves and natural parametrization.
\newblock {\em Ann. Probab.}, 41(3A):1556--1584, 2013.
\newblock \arxiv{1006.4936}.

\bibitem{topological2007LeGall}
J.-F. Le~Gall.
\newblock The topological structure of scaling limits of large planar maps.
\newblock {\em Invent. Math.}, 169(3):621--670, 2007.

\bibitem{legall2014icm}
J.-F. Le~Gall.
\newblock Random geometry on the sphere.
\newblock In {\em Proceedings of the {I}nternational {C}ongress of
  {M}athematicians---{S}eoul 2014. {V}ol. 1}, pages 421--442. Kyung Moon Sa,
  Seoul, 2014.
\newblock \arxiv{1403.7943}.

\bibitem{le2013uniqueness}
J.-F. Le~Gall et~al.
\newblock Uniqueness and universality of the brownian map.
\newblock {\em The Annals of Probability}, 41(4):2880--2960, 2013.

\bibitem{homeomorphic2008legallpaulin}
J.-F. Le~Gall and F.~Paulin.
\newblock Scaling limits of bipartite planar maps are homeomorphic to the
  2-sphere.
\newblock {\em Geom. Funct. Anal.}, 18(3):893--918, 2008.

\bibitem{lsw2017schnyder}
Y.~{Li}, X.~{Sun}, and S.~S. {Watson}.
\newblock {Schnyder woods, SLE(16), and Liouville quantum gravity}.
\newblock {\em ArXiv e-prints}, May 2017.

\bibitem{LIG_PARTICLES}
T.~M. Liggett.
\newblock {\em Interacting particle systems}.
\newblock Classics in Mathematics. Springer-Verlag, Berlin, 2005.
\newblock Reprint of the 1985 original.

\bibitem{0227.76081}
B.~B. Mandelbrot.
\newblock Possible refinement of the lognormal hypothesis concerning the
  distribution of energy dissipation in intermittent turbulence.
\newblock In {\em Statistical models and turbulence}, pages 333--351. Springer,
  1972.

\bibitem{marckert2006limit}
J.-F. Marckert, A.~Mokkadem, et~al.
\newblock Limit of normalized quadrangulations: the brownian map.
\newblock {\em The Annals of Probability}, 34(6):2144--2202, 2006.

\bibitem{mmq2018uniqueness}
O.~{McEnteggart}, J.~{Miller}, and W.~{Qian}.
\newblock {Uniqueness of the welding problem for SLE and Liouville quantum
  gravity}.
\newblock {\em arXiv e-prints}, page arXiv:1809.02092, Sept. 2018.

\bibitem{miermont2013brownian}
G.~Miermont et~al.
\newblock The brownian map is the scaling limit of uniform random plane
  quadrangulations.
\newblock {\em Acta mathematica}, 210(2):319--401, 2013.

\bibitem{m2016lightcone_dim}
J.~Miller.
\newblock Dimension of the {SLE} {L}ight {C}one, the {SLE} {F}an, and {${\rm
  SLE}_\kappa(\rho)$} for {$\kappa \in (0,4)$} and {$\rho \in$}
  {$\big[{\tfrac{\kappa}{2}}-4,-2\big)$}.
\newblock {\em Comm. Math. Phys.}, 360(3):1083--1119, 2018.
\newblock \arxiv{1606.07055}.

\bibitem{map_making}
J.~{Miller} and S.~{Sheffield}.
\newblock {An axiomatic characterization of the Brownian map}.
\newblock {\em ArXiv e-prints}, June 2015.
\newblock To appear in Journal Ecole Polytechnique.

\bibitem{ms2012imag1}
J.~Miller and S.~Sheffield.
\newblock Imaginary geometry {I}: interacting {SLE}s.
\newblock {\em Probab. Theory Related Fields}, 164(3-4):553--705, 2016.
\newblock \arxiv{1201.1496}.

\bibitem{ms2012imag2}
J.~Miller and S.~Sheffield.
\newblock {Imaginary geometry II: reversibility of SLE$_\kappa(\rho_1;\rho_2)$
  for $\kappa \in (0,4)$}.
\newblock {\em Ann. Probab.}, 44(3):1647--1722, 2016.
\newblock \arxiv{1201.1497}.

\bibitem{ms2012imag3}
J.~{Miller} and S.~{Sheffield}.
\newblock {Imaginary geometry III: reversibility of SLE$_\kappa$ for $\kappa
  \in (4,8)$}.
\newblock {\em Ann. of Math. (2)}, 184(2):455--486, 2016.
\newblock \arxiv{1201.1498}.

\bibitem{qle_continuity}
J.~{Miller} and S.~{Sheffield}.
\newblock {Liouville quantum gravity and the Brownian map II: geodesics and
  continuity of the embedding}.
\newblock {\em ArXiv e-prints}, May 2016.

\bibitem{qle_determined}
J.~{Miller} and S.~{Sheffield}.
\newblock {Liouville quantum gravity and the Brownian map III: the conformal
  structure is determined}.
\newblock {\em ArXiv e-prints}, Aug. 2016.

\bibitem{ms2013qle}
J.~Miller and S.~Sheffield.
\newblock Quantum {L}oewner evolution.
\newblock {\em Duke Math. J.}, 165(17):3241--3378, 2016.
\newblock \arxiv{1312.5745}.

\bibitem{ms2013imag4}
J.~Miller and S.~Sheffield.
\newblock Imaginary geometry {IV}: interior rays, whole-plane reversibility,
  and space-filling trees.
\newblock {\em Probab. Theory Related Fields}, 169(3-4):729--869, 2017.
\newblock \arxiv{1302.4738}.

\bibitem{quantum_spheres}
J.~Miller and S.~Sheffield.
\newblock Liouville quantum gravity spheres as matings of finite-diameter
  trees.
\newblock {\em Ann. Inst. Henri Poincar\'{e} Probab. Stat.}, 55(3):1712--1750,
  2019.
\newblock \arxiv{1506.03804}.

\bibitem{qlebm}
J.~Miller and S.~Sheffield.
\newblock Liouville quantum gravity and the {B}rownian map {I}: the {${\rm
  QLE}(8/3,0)$} metric.
\newblock {\em Invent. Math.}, 219(1):75--152, 2020.
\newblock \arxiv{1507.00719}.

\bibitem{mw2013intersections}
J.~Miller and H.~Wu.
\newblock Intersections of {SLE} paths: the double and cut point dimension of
  {SLE}.
\newblock {\em Probab. Theory Related Fields}, 167(1-2):45--105, 2017.
\newblock \arxiv{1303.4725}.

\bibitem{milnor2004pasting}
J.~Milnor.
\newblock Pasting together {J}ulia sets: a worked out example of mating.
\newblock {\em Experiment. Math.}, 13(1):55--92, 2004.

\bibitem{milnor2006dynamics}
J.~Milnor.
\newblock {\em Dynamics in one complex variable}, volume 160 of {\em Annals of
  Mathematics Studies}.
\newblock Princeton University Press, Princeton, NJ, third edition, 2006.

\bibitem{moore1925concerning}
R.~L. Moore.
\newblock Concerning upper semi-continuous collections of continua.
\newblock {\em Trans. Amer. Math. Soc.}, 27(4):416--428, 1925.

\bibitem{MR2073993}
Y.~Nakayama.
\newblock Liouville field theory: a decade after the revolution.
\newblock {\em Internat. J. Modern Phys. A}, 19(17-18):2771--2930, 2004.

\bibitem{PhysRevLett.49.1062}
B.~Nienhuis.
\newblock Exact critical point and critical exponents of {${\rm O}(n)$} models
  in two dimensions.
\newblock {\em Phys. Rev. Lett.}, 49(15):1062--1065, 1982.

\bibitem{Nien84}
B.~Nienhuis.
\newblock Critical behavior of two-dimensional spin models and charge asymmetry
  in the {C}oulomb gas.
\newblock {\em J. Statist. Phys.}, 34(5-6):731--761, 1984.

\bibitem{DG1987}
B.~Nienhuis.
\newblock Coulomb gas formulation of two-dimensional phase transitions.
\newblock In C.~Domb, M.~Green, and J.~Lebowitz, editors, {\em Phase
  transitions and critical phenomena, {V}ol.\ 11}, pages 1--53. Academic Press,
  London, 1987.

\bibitem{nolan2003stable}
J.~Nolan.
\newblock {\em Stable distributions: models for heavy-tailed data}.
\newblock Birkh{\"a}user, 2003.

\bibitem{py82bessel}
J.~Pitman and M.~Yor.
\newblock A decomposition of {B}essel bridges.
\newblock {\em Z. Wahrsch. Verw. Gebiete}, 59(4):425--457, 1982.

\bibitem{py96maximum}
J.~Pitman and M.~Yor.
\newblock Decomposition at the maximum for excursions and bridges of
  one-dimensional diffusions.
\newblock In {\em It\^o's stochastic calculus and probability theory}, pages
  293--310. Springer, Tokyo, 1996.

\bibitem{MR623209}
A.~M. Polyakov.
\newblock Quantum geometry of bosonic strings.
\newblock {\em Phys. Lett. B}, 103(3):207--210, 1981.

\bibitem{ry99martingales}
D.~Revuz and M.~Yor.
\newblock {\em Continuous martingales and {B}rownian motion}, volume 293 of
  {\em Grundlehren der Mathematischen Wissenschaften [Fundamental Principles of
  Mathematical Sciences]}.
\newblock Springer-Verlag, Berlin, third edition, 1999.

\bibitem{PSS:8474530}
R.~Rhodes and V.~Vargas.
\newblock K{PZ} formula for log-infinitely divisible multifractal random
  measures.
\newblock {\em ESAIM Probab. Stat.}, 15:358--371, 2011.
\newblock \arxiv{0807.1036}.

\bibitem{rv2019tail}
R.~Rhodes and V.~Vargas.
\newblock The tail expansion of {G}aussian multiplicative chaos and the
  {L}iouville reflection coefficient.
\newblock {\em Ann. Probab.}, 47(5):3082--3107, 2019.

\bibitem{rv2010revisited}
R.~Robert and V.~Vargas.
\newblock Gaussian multiplicative chaos revisited.
\newblock {\em Ann. Probab.}, 38(2):605--631, 2010.
\newblock \arxiv{0807.1030}.

\bibitem{rs2005sle}
S.~Rohde and O.~Schramm.
\newblock Basic properties of {SLE}.
\newblock {\em Ann. of Math. (2)}, 161(2):883--924, 2005.
\newblock \arxiv{math/0106036}.

\bibitem{rudin1986real}
W.~Rudin.
\newblock {\em Real and complex analysis}.
\newblock McGraw-Hill Book Co., New York, third edition, 1987.

\bibitem{0305-4470-19-13-009}
H.~Saleur.
\newblock New exact exponents for two-dimensional self-avoiding walks.
\newblock {\em J. Phys. A}, 19(13):L807--L810, 1986.

\bibitem{PhysRevLett.58.2325}
H.~Saleur and B.~Duplantier.
\newblock Exact determination of the percolation hull exponent in two
  dimensions.
\newblock {\em Phys. Rev. Lett.}, 58(22):2325--2328, 1987.

\bibitem{S0}
O.~Schramm.
\newblock Scaling limits of loop-erased random walks and uniform spanning
  trees.
\newblock {\em Israel J. Math.}, 118:221--288, 2000.
\newblock \arxiv{math/9904022}.

\bibitem{SchrammSheffieldGFF2}
O.~Schramm and S.~Sheffield.
\newblock A contour line of the continuum {G}aussian free field.
\newblock {\em Probab. Theory Related Fields}, 157(1-2):47--80, 2013.
\newblock \arxiv{1008.2447}.

\bibitem{sw2005sle_coordinate_changes}
O.~Schramm and D.~B. Wilson.
\newblock S{LE} coordinate changes.
\newblock {\em New York J. Math.}, 11:659--669 (electronic), 2005.
\newblock \arxiv{math/0505368}.

\bibitem{MR1182173}
N.~Seiberg.
\newblock Notes on quantum {L}iouville theory and quantum gravity.
\newblock {\em Progr. Theoret. Phys. Suppl.}, (102):319--349 (1991), 1990.
\newblock Common trends in mathematics and quantum field theories (Kyoto,
  1990).

\bibitem{Sh}
S.~Sheffield.
\newblock Gaussian free fields for mathematicians.
\newblock {\em Probab. Theory Related Fields}, 139(3-4):521--541, 2007.
\newblock \arxiv{math/0312099}.

\bibitem{she2009cle}
S.~Sheffield.
\newblock Exploration trees and conformal loop ensembles.
\newblock {\em Duke Math. J.}, 147(1):79--129, 2009.
\newblock \arxiv{math/0609167}.

\bibitem{she2010zipper}
S.~Sheffield.
\newblock Conformal weldings of random surfaces: {SLE} and the quantum gravity
  zipper.
\newblock {\em Ann. Probab.}, 44(5):3474--3545, 2016.
\newblock \arxiv{1012.4797}.

\bibitem{sheffield2011qg_inventory}
S.~Sheffield.
\newblock Quantum gravity and inventory accumulation.
\newblock {\em Ann. Probab.}, 44(6):3804--3848, 2016.
\newblock \arxiv{1108.2241}.

\bibitem{sheffieldwang}
S.~{Sheffield} and M.~{Wang}.
\newblock {Field-measure correspondence in Liouville quantum gravity almost
  surely commutes with all conformal maps simultaneously}.
\newblock {\em ArXiv e-prints}, May 2016.

\bibitem{spi1958bm}
F.~Spitzer.
\newblock Some theorems concerning {$2$}-dimensional {B}rownian motion.
\newblock {\em Trans. Amer. Math. Soc.}, 87:187--197, 1958.

\bibitem{stephensonCP}
K.~Stephenson.
\newblock {CirclePack}.
\newblock \url{http://www.math.utk.edu/~kens/CirclePack/}, 1992--2010.

\bibitem{timorin2010topological}
V.~Timorin.
\newblock Topological regluing of rational functions.
\newblock {\em Invent. Math.}, 179(3):461--506, 2010.
\newblock \arxiv{0810.5559}.

\bibitem{w1974pathdecomp}
D.~Williams.
\newblock Path decomposition and continuity of local time for one-dimensional
  diffusions. {I}.
\newblock {\em Proc. London Math. Soc. (3)}, 28:738--768, 1974.

\bibitem{wong2019tailprofile}
M.~D. Wong.
\newblock Universal tail profile of {G}aussian multiplicative chaos.
\newblock {\em Probab. Theory Related Fields}, 177(3-4):711--746, 2020.
\newblock \arxiv{1902.04054}.

\bibitem{yampolsky2001mating}
M.~Yampolsky and S.~Zakeri.
\newblock Mating {S}iegel quadratic polynomials.
\newblock {\em J. Amer. Math. Soc.}, 14(1):25--78 (electronic), 2001.
\newblock \arxiv{math/9808009}.

\bibitem{MR2435856}
D.~Zhan.
\newblock Reversibility of chordal {SLE}.
\newblock {\em Ann. Probab.}, 36(4):1472--1494, 2008.
\newblock \arxiv{0808.3649}.

\end{thebibliography}

\bigskip

\filbreak
\begingroup
\small
\parindent=0pt

\bigskip
\vtop{
\hsize=5.3in
Institut de Physique Th\'eorique\\
Universit\'e Paris-Saclay, CEA, CNRS\\
CEA/Saclay\\
91191 Gif-sur-Yvette Cedex\\
FRANCE}

\bigskip
\vtop{
\hsize=5.3in
Statistical Laboratory\\
Center for Mathematical Sciences\\
University of Cambridge\\
Cambridge CB3 0WB\\
UK}

\bigskip
\vtop{
\hsize=5.3in
Department of Mathematics\\
Massachusetts Institute of Technology\\
Cambridge, MA\\
USA } \endgroup \filbreak

\end{document}